\documentclass[fleqn,11pt]{article}

\usepackage{amsmath, amsthm, amssymb}
\usepackage[backref]{hyperref}
\usepackage[T1]{fontenc}
\usepackage[margin=1.6in]{geometry}
\usepackage[title]{appendix}
\usepackage{color,xcolor}
\usepackage{enumitem}
\usepackage{listings}
\usepackage{float}

\definecolor{darkblue}{RGB}{0,96,168}

\numberwithin{equation}{section}
\allowdisplaybreaks[4]
\hypersetup{
    colorlinks=true,
    linkcolor=darkblue,
    urlcolor=gray,
    citecolor=darkblue,
}

\lstdefinelanguage{VB}{
  keywords={Dim, As, Integer, Double, Long, String, Boolean, Private, Sub, Function, If, Then, Else, ElseIf, End, For, To, Step, Next, While, Wend, Do, Loop, Until, Select, Case, Call, Exit, Option, Explicit, Open, Close, Redim, Input, Output, Append, GoTo, Print},
  keywordstyle=\color{darkblue}\bfseries,
  comment=[l]{\'},
  commentstyle=\color{gray},
  alsoother={0,1,2,3,4,5,6,7,8,9},
  basicstyle=\ttfamily\color{black},
  numbers=left,
  numberstyle=\tiny\color{gray}
}

\newtheorem{conjecture}{Conjecture}

\newtheorem{problem}{Problem}
\newtheorem{definition}{Definition}
\newtheorem{example}{Example}[section]
\newtheorem{algorithm}{Algorithm}[section]
\newtheorem{identity}{Identity}
\newtheorem{corollary}{Corollary}[section]
\newtheorem*{notation}{Notation}
\newtheorem*{idealsolu}{Ideal solutions}
\newtheorem*{idealchain}{Ideal chains}
\newtheorem*{primesolu}{Ideal prime solutions}

\newtheorem*{relatedtype}{Related types}
\newtheorem*{mirrortype}{Mirror type}
\newcommand{\refIdentity}[1]{Identity\:\ref{#1}}

\begin{document}
\author{Chen Shuwen}
\title{A survey of The Prouhet-Tarry-Escott Problem \\and Its Generalizations}
\date{}
\maketitle

\begin{abstract}
This paper explores the Prouhet-Tarry-Escott problem (PTE), the Generalized PTE problem  (GPTE), and the Fermat form of Generalized PTE problem (FPTE). The GPTE problem extends the PTE problem by allowing different sets of exponents, while the FPTE problem considers cases where the number of integers in the two sets differs by one. Multigrade chains are also investigated, involving multiple sets of integers satisfying the GPTE system. The study of PTE and GPTE problems is further extended from integers to trigonometric functions.

Three novel generalizations of the Girard-Newton Identities are introduced to solve the PTE and GPTE problems: the first extends the domain of exponents to all integers; the second further generalizes to a broader form; and the third focuses on odd integer exponents. The constant \( C \) in the PTE and GPTE problems is investigated, and a novel approach is proposed by introducing the normalized GPTE problem with six conjectures to determine bounds for each variable. This enhances the efficiency of computer searches for ideal non-negative integer solutions that satisfy \( \sum_{i=1}^{n+1} a_i^k = \sum_{i=1}^{n+1} b_i^k \) for \( k = k_1, k_2, \dots, k_n \).

A general process of computer searches for GPTE solutions is discussed, and reference code for the search program is provided. Several parametric solutions for GPTE are presented, along with parametric methods for ideal prime solutions. Six open problems related to the PTE and GPTE problems are proposed, and three approaches to address these problems are suggested. In the appendix, an overview of the research status is provided for the 296 types of GPTE and FPTE that have ideal solutions.
\end{abstract}

\medskip


\setcounter{secnumdepth}{3}
\setcounter{tocdepth}{3}
\tableofcontents

\clearpage

\section{The PTE Problem and Its Generalizations}\label{Intro}

\subsection{The Prouhet-Tarry-Escott Problem (PTE)}
The study of the Prouhet-Tarry-Escott problem can be traced back to the period of 1750-1751, when Euler and Goldbach observed the case for the second power \cite{Dickson52}. In 1851, Prouhet demonstrated that integer solutions exist for any power, provided that the number of variables in the arrays is sufficiently large \cite{Wright1959}. In the 1910s, E. B. Escott and G. Tarry found partial solutions for powers up to the seventh \cite{Gloden44}. Consequently, the problem is named after these three individuals \cite{Borwein1994}. For further references on the Prouhet-Tarry-Escott problem, see \cite{Guy04, Hardy38, Hua, Raghavendran2019}, and other relevant works.
\begin{problem} \textnormal{\textbf{(The Prouhet-Tarry-Escott Problem)}}\\[1mm]
The Prouhet-Tarry-Escott problem, often abbreviated as the PTE problem, can be stated as follows:
Given a positive integer \( n \), find two distinct sets of integer solutions \( \{a_1, a_2, \dots, a_m\} \) and \( \{b_1, b_2, \dots, b_m\} \) that satisfy the Diophantine system:
\begin{align}
\label{PTEmk}
\sum_{i=1}^m a_i^k = \sum_{i=1}^m b_i^k, \quad \text{for} \;\; k = 1, 2, \dots, n.
\end{align}
\end{problem}
Here, \( m \) is referred to as the size of the system \eqref{PTEmk}, and \( n \) is the degree.
We define:
\begin{equation}
\begin{aligned}
\label{notationmk}
\left[ a_{1}, a_{2}, \dots, a_{m} \right]^{k} := a_{1}^{k} + a_{2}^{k} + \cdots + a_{m}^{k}, \quad \text{for} \;\; k \in \mathbb{N},
\end{aligned}
\end{equation}
so that the system \eqref{PTEmk} can be expressed as:
\begin{equation}
\begin{aligned}
\label{PTEmk2}
\left[ a_{1}, a_{2}, \dots, a_{m} \right]^{k} = \left[ b_{1}, b_{2}, \dots, b_{m} \right]^{k}, \quad (k = 1, 2, \dots, n).
\end{aligned}
\end{equation}
Given that \eqref{PTEmk2} holds, it is simple to prove that for any integers \( S \) and \( T \), the following also holds:
\begin{equation}
\begin{aligned}
\left[ S a_{1} + T, S a_{2} + T, \dots, S a_{m} + T \right]^{k} = \left[ S b_{1} + T, S b_{2} + T, \dots, S b_{m} + T \right]^{k}, & \\
(k = 1, 2, \dots, n). &
\end{aligned}
\end{equation}
Therefore, without loss of generality, for the PTE problem, we generally consider integer solutions \( \{a_1, a_2, \dots, a_m\} \) and \( \{b_1, b_2, \dots, b_m\} \) satisfying the following conditions simultaneously:
\begin{itemize}[itemsep=-1mm]
    \item All elements are non-negative integers,
    \item The greatest common divisor (gcd) of all elements is 1,
    \item The smallest element is 0.
\end{itemize}
It has been proved that \cite{Dickson52} there are no non-trivial solutions to the system \eqref{PTEmk} when \( m \leq n \). Consequently, we have the following definition:
\begin{definition}
If \( m = n + 1 \), a solution of the system \eqref{PTEmk} is called an \textbf{ideal solution of the PTE problem}. \textup{\cite{Borwein1994}}
\end{definition}
In this case, the system \eqref{PTEmk} takes the following form:
\begin{align}
\label{PTEideal}
\left[ a_{1}, a_{2}, \dots, a_{n+1} \right]^{k} = \left[ b_{1}, b_{2}, \dots, b_{n+1} \right]^{k}, \quad (k = 1, 2, \dots, n).
\end{align}
In this paper, we generally focus on the ideal solutions of the PTE problem. To date, ideal solutions of the PTE problem have been discovered for degrees \( n \leq 9 \) and \( n = 11 \).

Ideal solutions of the PTE problem can be further classified into two types: ideal symmetric solutions and ideal non-symmetric solutions.

\subsubsection{Ideal Symmetric Solution of PTE}
\begin{definition}
An ideal solution of the system \eqref{PTEideal} is called an \textbf{ideal symmetric solution of the PTE problem} if it satisfies:
\begin{equation}
\begin{aligned}
\label{defsym}
\begin{cases}
a_1+a_{n+1}=a_2+a_n=\dots=b_1+b_{n+1}=b_2+b_n=\dots, &\text{if}\:\;n\;\text{is odd},\\
a_1+b_{n+1}=a_2+b_n=\dots=a_{n+1}+b_1, &\text{if}\:\;n\; \text{is even}.
\end{cases}
\end{aligned}
\end{equation}
\end{definition}
As of the present, ideal symmetric solutions have been found for the PTE problem of degrees \( n \leq 9 \) and \( n = 11 \).
\begin{example}
Ideal symmetric solutions for degrees \emph{6,7,8,9,11} are as follows:
\end{example}

The ideal symmetric solution of degree 6 was first discovered by E.B. Escott in 1910:
\begin{align}
[0, 18, 27, 58, 64, 89, 101]^k = [1, 13, 38, 44, 75, 84, 102]^k, \nonumber \\ 
(k = 1, 2, 3, 4, 5, 6)
\tag*{\eqref{k123456s102}}
\end{align}

The ideal symmetric solution of degree 7 was first discovered by G. Tarry in 1913:
\begin{align}
[0, 4, 9, 23, 27, 41, 46, 50]^k = [1, 2, 11, 20, 30, 39, 48, 49]^k, \nonumber \\ 
(k = 1, 2, 3, 4, 5, 6, 7)
\tag*{\eqref{k1234567s50}}
\end{align}

The ideal symmetric solutions of degree 8 and degree 9 were first discovered by A. Letac in 1942:
\begin{align}
& [0, 24, 30, 83, 86, 133, 157, 181, 197]^k \nonumber \\
& \quad = [1, 17, 41, 65, 112, 115, 168, 174, 198]^k, \nonumber \\
& \qquad\qquad\qquad\qquad (k = 1, 2, 3, 4, 5, 6, 7, 8) 
\tag*{\eqref{k12345678s198}} \\
& [0, 3083, 3301, 11893, 23314, 24186, 35607, 44199, 44417, 47500]^k \nonumber \\
& \quad = [12, 2865, 3519, 11869, 23738, 23762, 35631, 43981, 44635, 47488]^k, \nonumber \\ 
& \qquad\qquad\qquad\qquad\qquad\qquad\qquad\qquad\quad (k = 1, 2, 3, 4, 5, 6, 7, 8, 9)
\tag*{\eqref{k123456789s47500}}
\end{align}

The ideal symmetric solution of degree 11 was first discovered by Nuutti Kuosa, Jean-Charles Meyrignac, and Chen Shuwen in 1999:
\begin{align}
& [0, 11, 24, 65, 90, 129, 173, 212, 237, 278, 291, 302]^k \nonumber \\
& \quad = [3, 5, 30, 57, 104, 116, 186, 198, 245, 272, 297, 299]^k, \nonumber \\
& \qquad\qquad\qquad\qquad\qquad (k = 1, 2, 3, 4, 5, 6, 7, 8, 9, 10, 11)
\tag*{\eqref{k1234567891011s302}}
\end{align}
\subsubsection{Ideal Non-Symmetric Solution of PTE}
\begin{definition}
An ideal solution of the system \eqref{PTEideal} is called an \textbf{ideal non-symmetric solution of the PTE problem} if it does not satisfy the condition \eqref{defsym}.
\end{definition}
Finding ideal non-symmetric solutions to the PTE problem is more challenging than finding ideal symmetric solutions. To date, ideal non-symmetric solutions have only been discovered for PTE of degrees $n\leq 7$. 

\begin{example}
Ideal non-symmetric solutions for degrees \emph{5,6,7} are as follows:
\end{example}

The ideal non-symmetric solution of degree 5 was first discovered by A. Gloden in 1944:
\begin{align}
& [0, 19, 25, 57, 62, 86]^{k} = [2, 11, 40, 42, 69, 85]^{k}, \;\; (k = 1, 2, 3, 4, 5)
\tag*{\eqref{k12345s86}}
\end{align}

The ideal non-symmetric solutions of degree 6 and degree 7 were first discovered by Chen Shuwen in 1997:
\begin{align}
& [0, 18, 19, 50, 56, 79, 81]^k = [1, 11, 30, 39, 68, 70, 84]^k, \nonumber \\
& \qquad\qquad\qquad\qquad\qquad\qquad\qquad (k = 1, 2, 3, 4, 5, 6) \tag*{\eqref{k123456s84}} \\
& [0, 7, 23, 50, 53, 81, 82, 96]^k = [1, 5, 26, 42, 63, 72, 88, 95]^k, \nonumber \\ 
& \qquad\qquad\qquad\qquad\qquad\qquad\qquad (k = 1, 2, 3, 4, 5, 6, 7) \tag*{\eqref{k1234567s96}}
\end{align}

We note that each non-symmetric solution has an \textbf{equivalent solution} that can be obtained through a simple transformation: by subtracting each element from the largest element \cite{Chen22}. For example, the following non-symmetric solution is equivalent to \eqref{k123456s84}:
\begin{align}
\label{k123456s3}
& [0, 14, 16, 45, 54, 73, 83]^{k} = [3, 5, 28, 34, 65, 66, 84]^{k}, \quad (k = 1, 2, 3, 4, 5, 6) \nonumber 
\end{align}

\subsection{The Generalized Prouhet-Tarry-Escott Problem (GPTE)}
The PTE problem discussed in the previous section concerns exponent sets consisting of consecutive natural numbers \(\{1, 2, 3, \ldots, n\}\). Beginning in the 1930s, researchers including A.~Moessner, A.~Gloden, A.~Letac, T.~N.~Sinha, Ajai Choudhry et al.\ began investigating cases with non-consecutive positive integer exponents. Since 1995, the author has systematically studied the generalized case where the exponents are arbitrary integer sequences  \cite{Chen23}, including negative integers \cite{Chenkminus23}. In particular, the case where \(k = 0\) is defined as the equal products case \cite{ChenkProducts23}. To formalize this, we first generalize the definition in \eqref{notationmk} as follows:
\begin{definition}
For integers \( a_1, a_2, \dots, a_m \) and \( k \in \mathbb{Z} \), we define
\begin{equation}
\label{EPES}
\left[ a_{1}, a_{2}, \dots, a_{m} \right]^{k} :=
\begin{cases}
a_{1}^{k} + a_{2}^{k} + \cdots + a_{m}^{k}, & \text{if } k \neq 0, \\
a_{1} a_{2} \cdots a_{m}, & \text{if } k = 0.
\end{cases}
\end{equation}
\end{definition}
Based on the definition provided in \eqref{EPES}, we study the following more general problem.

\begin{problem}\textnormal{\textbf{(The Generalized Prouhet-Tarry-Escott Problem)}}\\[1mm]
The generalized Prouhet-Tarry-Escott problem, abbreviated as the GPTE problem, can be stated as:
Given integers \( \{k_1, k_2, \dots, k_n\} \), find two \mbox{distinct} sets of integer solutions \( \{a_1, a_2, \dots, a_m\} \) and \( \{b_1, b_2, \dots, b_m\} \) that satisfy the Diophantine system:
\begin{align}
\label{GPTE}
\left[ a_{1}, a_{2}, \dots, a_{m} \right]^{k} = \left[ b_{1}, b_{2}, \dots, b_{m} \right]^{k}, \quad (k = k_1, k_2, \dots, k_n).
\end{align}
\end{problem}
Here, \( m \) is referred to as the size of the system \eqref{GPTE}.

\subsubsection{Ideal Non-Negative Integer Solution of GPTE}
When \( m = n \), no non-negative integer solutions are found for any exponent set \( (k = k_1, k_2, \dots, k_n) \) of the system \eqref{GPTE}. However, when \( m = n + 1 \), non-negative integer solutions exist for the system \eqref{GPTE}. In this case, system \eqref{GPTE} takes the following form:
\begin{align}
\label{GPTEk}
& \left[ a_{1}, a_{2}, \dots, a_{n+1} \right]^{k} = \left[ b_{1}, b_{2}, \dots, b_{n+1} \right]^{k}, \quad (k = k_1, k_2, \dots, k_n).
\end{align}
\begin{definition}
A non-negative integer solution of the system \eqref{GPTEk} is called an \textbf{ideal non-negative integer solution of the GPTE problem}.
\end{definition}
So far, a total of 164 distinct types of ideal non-negative integer solutions have been discovered for various combinations of \( (k = k_1, k_2, \dots, k_n) \). These solution types can be categorized into five groups with counts of 42, 42, 24, 23, and 33, respectively, summing to the total of 164. For a comprehensive listing of these 164 types, see Appendix A. Notably, beyond these 164 documented types, there could exist infinitely many other solution types awaiting discovery.

\begin{example}
\label{examplek1}
GPTE with all \( k > 0 \)
\end{example}
Ideal non-negative integer solutions of the GPTE problem have been identified for 42 types with all \( k > 0 \), including 10 known types of PTE. For example,
\begin{align}
\label{k14}
& [3, 25, 38]^k = [7, 20, 39]^k, \quad (k = 1, 4) \tag*{\eqref{k14s39}} \\
& [1741, 2435, 3004, 3476]^k = [1937, 2111, 3280, 3328]^k, \;\; (k = 1, 3, 7) \tag*{\eqref{k137s3476}} \\
& [975, 224368, 300495, 366448]^k \nonumber \\
& \quad = [37648, 202575, 337168, 344655]^k, \quad (k = 2, 3, 4) \tag*{\eqref{k234s366448}} \\
& [34, 133, 165, 299, 332, 366]^k = [35, 124, 177, 286, 353, 354]^k, \nonumber \\
& \qquad\qquad\qquad\qquad\qquad\qquad\qquad\qquad\qquad (k = 1, 2, 3, 4, 7) \tag*{\eqref{k12347s366}} \\
& [269, 397, 409, 683, 743, 901, 923]^k \nonumber \\ 
& \quad = [299, 313, 493, 613, 827, 839, 941]^k, \quad (k = 1, 2, 3, 5, 7, 9) \tag*{\eqref{k123579s941}} \\
& [77, 159, 169, 283, 321, 443, 447, 501]^k \nonumber \\ 
& \quad = [79, 137, 213, 237, 363, 399, 481, 491]^k, \;\; (k = 1, 2, 3, 4, 5, 6, 8) \tag*{\eqref{k1234568s501}}
\end{align}

\begin{example}
\label{examplek2}
GPTE with all \( k < 0 \)
\end{example}
Since 2001, we have systematically investigated this set of categories \cite{Chenkminus23}. By applying the following simple arithmetic transformations, which we first \mbox{discovered} in 2001:\\[1mm]
\indent Let \( \{a_1, a_2, \dots, a_m\} \) and \( \{b_1, b_2, \dots, b_m\} \) be two distinct sets of non-zero integers (\( a_i \neq 0, \, b_i \neq 0 \)) satisfying
\[
[a_1, a_2, \dots, a_m]^k = [b_1, b_2, \dots, b_m]^k \quad \text{for} \;\; k = k_1, k_2, \dots, k_n.
\]
Let \( C \) be the least common multiple (LCM) of all \( \{a_i, b_i\} \). Then
\begin{align}
\label{LCM}
\left[\frac{C}{a_1}, \frac{C}{a_2}, \dots, \frac{C}{a_m}\right]^k = \left[\frac{C}{b_1}, \frac{C}{b_2}, \dots, \frac{C}{b_m}\right]^k, \;\; (k = -k_1, -k_2, \dots, -k_n).
\end{align}
Using these transformations on solutions for GPTE with all \( k > 0 \), we have identified 42 distinct types of ideal non-negative integer solutions of GPTE with all \( k < 0 \). For example,
\begin{align}
& [188678700, 224349300, 835077950, 901884186]^k \nonumber \\ 
& \quad = [178945275, 302645700, 322101495, 15031403100]^k, \nonumber \\
& \qquad\qquad\qquad\qquad\qquad\qquad\qquad\qquad (k = -4, -3, -1) \tag*{\eqref{kn134s15031403100}} \\
& [13104, 13650, 20475, 23400, 54600, 65520]^k \nonumber \\
& \quad = [12600, 15600, 16380, 32760, 36400, 81900]^k, \nonumber \\
& \qquad\qquad\qquad\qquad\qquad (k = -5, -4, -3, -2, -1) \tag*{\eqref{kn12345s81900}} \\
& [1081404, 1113210, 1401820, 1992060, 9462285, 12616380]^k \nonumber \\
& \quad = [1051365, 1220940, 1261638, 2226420, 5407020, 37849140]^k, \nonumber \\
& \qquad\qquad\qquad\qquad\qquad\qquad\qquad\qquad (k = -8, -6, -4, -2, -1) \tag*{\eqref{kn12468s37849140}}
\end{align}
We refer to the type \( (k = -k_1, -k_2, \dots, -k_n) \) obtained through \eqref{LCM} as the \textbf{mirror type} of the original type \( (k = k_1, k_2, \dots, k_n) \). For example, the mirror type of \( (k = 1, 3, 4) \) is \( (k = -4, -3, -1) \).
\begin{example}
\label{examplek3}
GPTE with \( k_1 = 0 \) and all others \( k > 0 \)
\end{example}
According to the definition in \eqref{EPES}, \( k_1 = 0 \) implies that two distinct sets of integer solutions \( \{a_1, a_2, \dots, a_m\} \) and \( \{b_1, b_2, \dots, b_m\} \) of \eqref{GPTE} must satisfy an equality of products:
\[
a_1 a_2 \cdots a_m = b_1 b_2 \cdots b_m.
\]
Since 2001, we have systematically investigated this category of solutions \cite{ChenkProducts23}. To date, ideal non-negative integer solutions of the GPTE problem have been identified for 24 distinct types with \( k_1 = 0 \) and all others \( k > 0 \). For example,
\begin{align}
& [2, 2, 11, 21]^k = [1, 6, 7, 22]^k, \quad (k = 0, 1, 2) \tag*{\eqref{k012s22}} \\
& [14, 33, 37, 108, 112, 192, 221]^k = [17, 21, 52, 88, 128, 189, 222]^k, \nonumber \\ 
& \qquad\qquad\qquad\qquad\qquad\qquad\qquad\qquad\qquad (k = 0, 1, 2, 3, 4, 6) \tag*{\eqref{k012346s222}} \\
& [423, 548, 621, 902, 923, 1159, 1218]^k \nonumber \\
& \quad = [437, 497, 732, 754, 1034, 1107, 1233]^k, \quad (k = 0, 1, 2, 4, 6, 8) \tag*{\eqref{k012468s1233}} \\
& [387, 388, 416, 447, 494, 536, 573, 589, 610]^k \nonumber \\
& \quad = [382, 402, 403, 456, 485, 549, 559, 596, 608]^k, \nonumber \\
& \qquad\qquad\qquad\qquad\qquad (k = 0, 1, 2, 3, 4, 5, 6, 7) \tag*{\eqref{k01234567s610}}
\end{align}

\begin{example}
\label{examplek4}
GPTE with \( k_n = 0 \) and all others \( k < 0 \)
\end{example}
By applying simple arithmetic transformations using \eqref{LCM}, based on solutions for GPTE with \( k_1 = 0 \) and all other exponents \( k > 0 \), we have identified 23 distinct types of ideal non-negative integer solutions of GPTE where \( k_n = 0 \) and all other exponents \( k < 0 \). For example,
\begin{align}
& [11712265, 15189721, 28371045, 142587381]^k \nonumber \\
& \quad = [11991885, 13979155, 35936657, 119465103]^k, \;\; (k = -6, -2, 0) \tag*{\eqref{kn026s142587381}} \\
& [32760, 40248, 48762, 97524, 234780, 528255]^k \nonumber \\
& \quad = [34830, 35217, 58968, 78260, 325080, 422604]^k, \nonumber \\
& \qquad\qquad\qquad\qquad\qquad\qquad (k = -5, -3, -2, -1, 0) \tag*{\eqref{k0n1235s528255}} \\
& [324720, 364320, 373428, 466785, 663872, 829840, 905280]^k \nonumber \\
& \quad = [331936, 339480, 414920, 432960, 728640, 746856, 933570]^k, \nonumber \\
& \qquad\qquad\qquad\qquad\qquad\qquad\qquad\quad (k = -6, -4, -3, -2, -1, 0) \tag*{\eqref{k0n12346s933570}}
\end{align}

\begin{example}
\label{examplek5}
GPTE with \( k_1 < 0 \) and \( k_n > 0 \)
\end{example}
To date, 33 distinct types of ideal non-negative integer solutions of GPTE have been identified with \( k_1 < 0 \) and \( k_n > 0 \). For example,
\begin{align}
& [77, 1057, 1661]^k = [91, 143, 1963]^k, \quad (k = -2, 2) \tag*{\eqref{k2n2s1963}} \\
& [85, 286, 702, 858]^k = [81, 374, 585, 891]^k, \quad (k = -1, 1, 5) \tag*{\eqref{k15n1s891}} \\
& [266, 494, 494, 1463, 1547]^k = [287, 374, 611, 1394, 1598]^k, \nonumber \\
& \qquad\qquad\qquad\qquad\qquad\qquad\qquad\qquad\quad (k = -1, 1, 2, 3) \tag*{\eqref{k123n1s1598}} \\
& [9, 17, 21, 51, 99, 143, 143]^k = [11, 11, 33, 39, 117, 119, 153]^k, \nonumber \\
& \qquad\qquad\qquad\qquad\qquad\qquad\qquad\qquad (k = -1, 0, 1, 2, 3, 4) \tag*{\eqref{k01234n1s153}} \\
\label{k0123n123}
& [22, 22, 33, 33, 56, 56, 84, 84]^k = [21, 24, 28, 42, 44, 66, 77, 88]^k, \nonumber \\
& \qquad\qquad\qquad\qquad\qquad\qquad\qquad (k = -3, -2, -1, 0, 1, 2, 3) \tag*{\eqref{k0123n123s88}}
\end{align}

For the above Example\,\ref{examplek1} to Example\,\ref{examplek5}, the types \( (k = 1, 3, 7) \) and \( (k = -2, 2) \) were first solved by Ajai Choudhry in 1999 and 2011, respectively. The type \( (k = 0, 1, 2) \) was first solved by A. Gloden in 1944. The remaining examples were discovered by Chen Shuwen between 1995 and 2023. For more examples, please refer to Appendix A.\\[1mm]
\indent
Although ideal non-negative integer solutions have been identified for 164 distinct types of \( (k = k_1, k_2, \dots, k_n) \) in GPTE, countless other types remain unsolved. However, it has not been proven that any particular type of \( (k = k_1, k_2, \dots, k_n) \) lacks an ideal non-negative integer solution. For instance, it remains unknown whether the following types have non-negative integer solutions:
\begin{align}
\label{typek5}
& a_1^k + a_2^k = b_1^k + b_2^k, \qquad\qquad \text{for} \;\; k \geq 5, \; k \in \mathbb{N}, \\
\label{typek34}
& a_1^k + a_2^k + a_3^k = b_1^k + b_2^k + b_3^k, \qquad \text{for} \;\; k = 3, 4.
\end{align}
\subsubsection{Ideal Integer Solution of GPTE}
When \( m = n \), for easier classification, we use \( h \) instead of \( k \) to denote system \eqref{GPTE}. That is,
\begin{align}
\label{GPTEh}
&\left[ a_{1}, a_{2}, \cdots, a_{n} \right]^{h} = \left[ b_{1}, b_{2}, \cdots, b_{n} \right]^{h}, \quad \left( h = h_1, h_2, \dots, h_n \right).
\end{align}

As mentioned in Section 1.2.1, no non-negative integer solution has been \mbox{discovered} for any type of \( (h = h_1, h_2, \cdots, h_n) \) in system \eqref{GPTEh} so far. \mbox{However}, integer solutions do exist for system \eqref{GPTEh}.

\begin{definition}
We define an integer solution of system \eqref{GPTEh} as an \textbf{ideal integer solution of GPTE}.
\end{definition}

Until now, a total of 65 distinct types of ideal integer solutions of GPTE have been discovered, specifically categorized as \(15 + 15 + 6 + 6 + 23\), which correspond to a wide array of configurations for \( (h = h_1, h_2, \dots, h_n) \). However, as research progresses, it is expected that the number of such solution types will continue to increase. For a detailed overview of these 65 types, please refer to Appendix B.

\begin{example}
\label{exampleh1}
GPTE with all \( h > 0 \)
\end{example}
Ideal integer solutions of GPTE have been found for 15 types with all \( h > 0 \). For example,
\begin{align}
\label{h134}
& [ -3254, 5583, 5658 ]^h = [ -1329, 2578, 6738 ]^h, \quad (h = 1, 3, 4)  \tag*{\eqref{h134s6738}} \\
& [-815, 358, 1224]^h = [ -776, -410, 1233]^h, \quad (h = 2, 3, 4)  \tag*{\eqref{h234s1233}} \\
& [-59, -5, -1, 33, 57]^h = [ -55, -23, 13, 39, 51]^h, \;\; (h = 1, 2, 3, 5, 7) \tag*{\eqref{h12357s57}} \\
& [-285, -187, -173, 30, 93, 226, 296]^h \nonumber \\
& \quad = [-264, -250, -89, -47, 170, 177, 303]^h, \;\; (h = 1, 2, 3, 4, 5, 6, 8)  \tag*{\eqref{h1234568s303}} \\
& [-48, -44, -23, -7, 14, 23, 39, 46]^h \nonumber \\
& \quad = [-49, -42, -26, 1, 4, 32, 33, 47]^h, \quad (h = 1, 2, 3, 4, 5, 6, 7, 9) \tag*{\eqref{h12345679s47}}
\end{align}

\begin{example}
\label{exampleh2}
GPTE with all \( h < 0 \)
\end{example}
By applying simple arithmetic transformations using \eqref{LCM}, based on
solutions for GPTE with all \( h > 0 \), so far, ideal integer solutions of GPTE have been found for 15 types with all \( h < 0 \). For example,
\begin{align}
& [-661172085, 662954220, 5719907340]^h \nonumber \\
&\quad= [-607298804, 801159660, 1756828683]^h, \;\; (h = -6, -2, -1)      \tag*{\eqref{hn126s5719907340}} \\
& [-15935205, -188035419, -940177095, 28490215, 16494335]^h \nonumber \\
&\quad=[-17094129, -40877265, 72321315, 24107105, 18434845]^h, \nonumber \\
&\qquad\qquad\qquad\qquad\qquad\qquad\qquad\qquad (h = -7, -5, -3, -2, -1)  \tag*{\eqref{hn12357s940177095}}
\end{align}

\begin{example}
\label{exampleh3}
GPTE with \( h_1 = 0 \) and all others \( h > 0 \)
\end{example}
According to the definition in \eqref{EPES}, \( h_1 = 0 \) means that two distinct sets of integer solutions \( \{ a_1, a_2, \dots, a_n \} \) and \( \{ b_1, b_2, \dots, b_n \} \) of \eqref{GPTEh} have an equal product. Until now, ideal integer solutions of GPTE have been discovered for 6 types with \( h_1 = 0 \) and all others \( h > 0 \). For example,
\begin{align}
& [-138, 9, 62, 67]^h = [ -93, -23, -18, 134]^h, \quad (h = 0, 1, 2, 4)  \tag*{\eqref{h0124s134}} \\
& [ -1518, 14, 1363, 1581]^h = [ -561, 138, 406, 1457]^h, \quad (h = 0, 1, 3, 5) \tag*{\eqref{h0135s1581}} \\
& [ -11716, -6437, -1460, 1175, 7897, 10541]^h \nonumber \\
& \quad = [-10865, -8383, -596, 3683, 4700, 11461]^h, \nonumber \\
& \qquad\qquad\qquad\qquad\qquad\qquad (h = 0, 1, 2, 3, 4, 6) \tag*{\eqref{h012346s11716}}
\end{align}

\begin{example}
\label{exampleh4}
GPTE with \( h_n = 0 \) and all others \( h < 0 \)
\end{example}
By applying simple arithmetic transformations using \eqref{LCM} to the solutions for GPTE with \( h_1 = 0 \) and all others \( h > 0 \), ideal integer solutions of GPTE have been found for 6 types with \( h_n = 0 \) and all others \( h < 0 \). For example,
\begin{align}
& [ -604656, 406980, 1113840, 1867320 ]^h \nonumber \\
& \quad= [-488376, 393120, 671840, 3968055 ]^h, \;\; (h = -5, -3, -1, 0)  \tag*{\eqref{h0n135s3968055}} \\
& [-3600, -150, 160, 288, 450 ]^h = [ -800, -180, 144, 225, 2400 ]^h, \nonumber \\
& \qquad\qquad\qquad\qquad\qquad\qquad\qquad\qquad (h = -6, -4, -2, -1, 0)  \tag*{\eqref{h0n1246s2400}}
\end{align}

\begin{example}
\label{exampleh5}
GPTE with \( h_1 < 0 \) and \( h_n > 0 \)
\end{example}
Until now, ideal integer solutions of GPTE have been found for 23 types with \( h_1 < 0 \) and \( h_n > 0 \). For example,
\begin{align}
& [-891, 85, 286, 702]^h = [ -858, 81, 374, 585]^h, \quad (h = -1, 0, 1, 5) \tag*{\eqref{h015n1s702}} \\
& [-34775, 2247, 21828, 36594 ]^h \nonumber \\
& \quad= [-15246, 2299, 12100, 26741 ]^h, \quad (h = -2, -1, 1, 3) \tag*{\eqref{h13n12s36594}} \\
& [-156, -130, 13, 35, 42, 196 ]^h = [ -147, -140, 14, 26, 52, 195]^h, \nonumber \\
& \qquad\qquad\qquad\qquad\qquad\qquad\qquad\qquad\quad (h = -1, 0, 1, 2, 3, 5) \tag*{\eqref{h01235n1s196}} \\
\label{h12345n1}
&[-13, 26, 52, 130, 156, 195]^h = [-14, 35, 42, 140, 147, 196]^h, \nonumber \\
& \qquad\qquad\qquad\qquad\qquad\qquad\qquad\quad\; (h = -1, 1, 2, 3, 4, 5) \tag*{\eqref{h12345n1s196}}
\end{align}

For Examples\,\ref{exampleh1} to \ref{exampleh5}, the types \( (h = 1, 3, 4) \), \( (h = 2, 3, 4) \), \( (h = 0, 1, 2, 4) \), \( (h = 0, 1, 2, 3, 4, 6) \), and \( (h = -1, 1, 2, 3, 4, 5) \) were first discovered by Ajai Choudhry between 1991 and 2011. The type \( (h = 1, 2, 3, 5, 7) \) was first discovered by G. Palama in 1953. The remaining examples were first discovered by Chen Shuwen between 1997 and 2023. For more examples, please refer to Appendix B.

As mentioned in Section 1.2.1, it has not been proven that any specific type of \( (k = k_1, k_2, \dots, k_n) \) does not have an ideal non-negative integer solution. However, it has been proven that certain types of \( (h = h_1, h_2, \dots, h_n) \) do not have ideal integer solutions \cite{Hua,Hardy38,Choudhry04}. For example,
\begin{align}
\label{h123n}
& [a_{1}, a_{2}, \dots, a_{n}]^{h} = [b_{1}, b_{2}, \dots, b_{n}]^{h}, \quad (h = 1, 2, \dots, n). \\
\label{h2462n}
& [a_{1}, a_{2}, \dots, a_{n}]^{h} = [b_{1}, b_{2}, \dots, b_{n}]^{h}, \quad (h = 2, 4, \dots, 2n). \\
\label{h1251236}
& [a_{1}, a_{2}, \dots, a_{n}]^{h} = [b_{1}, b_{2}, \dots, b_{n}]^{h}, \quad (h = 1, 2, \dots, n-1, n+2).
\end{align}

\subsubsection{Ideal Prime Solution of GPTE}
As early as 1938, Hua Loo-Keng observed that the integer solutions to the PTE problem could all be prime numbers \cite{Hua38}. In recent years, prime number solutions for PTE and GPTE have attracted increasing attention and research \cite{Rivera65,Chen23,JCM}. In this subsection, we restrict the range of ideal solutions discussed in Subsection 1.2.1 from non-negative integers to prime numbers.

\begin{definition}
If a positive solution of system \eqref{GPTEk} exists where all \( \{a_1, a_2, \dots, a_{n+1}\} \) and \( \{b_1, b_2, \dots, b_{n+1}\} \) are prime numbers, we call it an \textbf{ideal prime solution of GPTE}.
\end{definition}

It can be easily shown that prime number solutions for system \eqref{GPTEk} can only exist when all \( k \) are positive integers. As summarized in Example \ref{examplek1}, ideal non-negative integer solutions have been found for 42 types of GPTE where all \( k > 0 \). Among these 42 types, ideal prime solutions have been identified for 27 types, including 10 known types of PTE.

\begin{example}
\label{examplep1}
Ideal prime solutions of GPTE with all \( k > 0 \)
\end{example}

The following are some examples:
\begin{align}
& [89, 277, 337]^k = [139, 233, 353]^k, \quad (k = 2, 4) \tag*{\eqref{k24s353}} \\
& [1777, 5003, 6089]^k = [2657, 3833, 6379]^k, \quad (k = 1, 5) \tag*{\eqref{k15s6379}} \\
& [8417237, 104616559, 111462317]^k \nonumber \\
& \quad = [47946583, 69380393, 127776401]^k, \quad (k = 2, 3) \tag*{\eqref{k23s127776401}} \\
& [896501990958793919143, 2056330598071774290263, \nonumber \\
& \qquad\qquad 3997663854855273138397, 5094457378727364512429]^k \nonumber \\
& \quad = [994237422911295892921, 1908177853245929320403, \nonumber \\
& \qquad\qquad 4082781516440229111169, 5059757030015751535739]^k, \nonumber \\
& \qquad\qquad\qquad\qquad\qquad\qquad\qquad\qquad\qquad\qquad\qquad (k = 1, 3, 4) \tag*{\eqref{k134s5094457378727364512429}} \\
& [83, 149, 337, 439, 503]^k = [71, 173, 313, 463, 491]^k, \;\; (k = 1, 2, 4, 6) \tag*{\eqref{k1246s503}} \\
& [12251, 34511, 42461, 80621, 102881, 141041, 148991, 171251]^k \nonumber \\
& \quad = [13841, 26561, 59951, 63131, 120371, 123551, 156941, 169661]^k, \nonumber \\
& \qquad\qquad\qquad\qquad\qquad\qquad\qquad\qquad\qquad (k = 1, 2, 3, 4, 5, 6, 7) \tag*{\eqref{k1234567s171251}} \\
& [3522263, 4441103, 5006543, 7904423, 9388703, 11897843, \nonumber \\
& \qquad\qquad 13876883, 15361163, 15643883]^k \nonumber \\
& \quad = [3698963, 3981683, 5465963, 7445003, 9954143, 11438423, \nonumber \\
& \qquad\qquad 14336303, 14901743, 15820583]^k, \;\; (k = 1, 2, 3, 4, 5, 6, 7, 8) \tag*{\eqref{k12345678s15820583}} \\
& [32058169621, 32367046651, 32732083141, 33883352071, \nonumber \\
& \qquad\quad 34585345321, 35680454791, 36915962911, 38011072381, \nonumber \\
& \qquad\quad 38713065631, 39864334561, 40229371051, 40538248081]^k \nonumber \\
& \quad = [32142408811, 32198568271, 32900561521, 33658714231, \nonumber \\
& \qquad\quad 34978461541, 35315418301, 37280999401, 37617956161, \nonumber \\
& \qquad\quad 38937703471, 39695856181, 40397849431, 40454008891]^k, \nonumber \\
& \qquad\quad \qquad\qquad\qquad\qquad\qquad (k = 1, 2, 3, 4, 5, 6, 7, 8, 9, 10, 11) \tag*{\eqref{k1234567891011s40538248081}}
\end{align}

\indent
For Example\,\ref{examplep1} above, type \( (k = 1, 2, 3, 4, 5, 6, 7) \) was first found by T.W.A. Baumann in 1999, types \( (k = 1, 5) \) and \( (k = 1, 2, 3, 4, 5, 6, 7, 8, 9, 10, 11) \) were first found by Jarosław Wróblewski in 2002 and 2023, respectively. The remaining types were first found by Chen Shuwen between 2011 and 2023.\\[1mm]
\indent
The following ideal prime solution of GPTE, discovered by Chen Shuwen in 2023 based on \eqref{k15rst}, is currently the numerically largest known one.
\begin{align}
& [68049074651809716616587682328308420187688753216224749729, \nonumber \\
& \qquad\quad 148818734733829951795084131100190065917882950867486851411, \nonumber \\
& \qquad\quad 159627898794439357258113507241692664029601335467631785279]^k \nonumber \\
& \quad = [68691516188504321631164213329991813284950709035020363731, \nonumber \\
& \qquad\quad 146607580632312953765823649961087271271521143632892464049, \nonumber \\
& \qquad\quad 161196611359261750272797457379112065578701186883430558639]^k, \nonumber \\
& \qquad\qquad\qquad\qquad\qquad\qquad\qquad\qquad\qquad\qquad\qquad (k = 1, 5) \tag*{\eqref{k15p2}}
\end{align}

For more examples of these 27 types of ideal prime solutions, see Appendix A.1. Currently, it remains unproven whether for various types \( (k = k_1, k_2, \dots, k_n) \) of GPTE with all \( k > 0 \), if there exist ideal non-negative integer solutions, there must also exist ideal prime solutions.
\subsection{The Generalized PTE Problem of Fermat Form (FPTE)}
In 2023, we noticed a new series where the number of integers in the two sets differs by one \cite{Chen2125}. For instance, Fermat's last theorem belongs to this series.

\begin{problem}\textnormal{\textbf{(The Generalized PTE Problem of Fermat Form)}}\\[1mm]
The generalized Prouhet-Tarry-Escott problem of Fermat form, abbreviated as the FPTE problem, can be stated as:
Given integers \( \{k_1, k_2, \dots, k_n\} \), find two distinct sets of integer solutions \( \{a_1, a_2, \dots, a_{m-1}\} \) and \( \{b_1, b_2, \dots, b_m\} \) that satisfy the Diophantine system:
\begin{equation}
\begin{aligned}
\label{FPTE}
\left[ a_{1}, a_{2}, \dots, a_{m-1} \right]^{k} = \left[ b_{1}, b_{2}, \dots, b_{m} \right]^{k}, \quad (k = k_1, k_2, \dots, k_n).
\end{aligned}
\end{equation}
\end{problem}

It should be noted that, when \( k \) includes zero or negative integers, \eqref{FPTE} cannot be regarded as a special case of \eqref{GPTE}.

\subsubsection{Ideal Positive Integer Solution of FPTE}
We systematically studied the integer solutions of FPTE and noticed that, when \( m \leq n \), only one unique set of positive integer solutions for \eqref{FPTE} has been discovered so far:
\begin{align}
& [4]^k = [2, 2]^k, \quad (k = 0, 1) \nonumber
\end{align}

When \( m = n + 1 \), there exist numerous positive integer solutions for system \eqref{FPTE}. In this case, for easy classification, we use \( r \) instead of \( k \) to denote system \eqref{FPTE}, that is:
\begin{equation}
\begin{aligned}
\label{FPTEr}
\left[ a_{1}, a_{2}, \dots, a_{n} \right]^{r} = \left[ b_{1}, b_{2}, \dots, b_{n+1} \right]^{r}, \quad (r = r_1, r_2, \dots, r_n).
\end{aligned}
\end{equation}

\begin{definition}
We call a positive integer solution of system \eqref{FPTEr} an \textbf{ideal \\positive integer solution of FPTE}.
\end{definition}
So far, ideal positive integer solutions of FPTE have been found for 54 types of \( (r = r_1, r_2, \dots, r_n) \). These 54 types can be grouped into five categories with counts of 22, 22, 5, 3, and 2, respectively. For a detailed overview of these 54 types, please refer to Appendix C.
\begin{example}
\label{exampler1}
FPTE with all \( r > 0 \)
\end{example}
Ideal positive integer solutions of FPTE have been discovered for 22 types with all \( r > 0 \), including 10 known types of PTE. For example:
\begin{align}
& [5]^r = [3, 4]^r, \quad (r = 2) \tag*{\eqref{r2s5}} \\
& [349, 557]^r = [95, 238, 573]^r, \quad (r = 1, 4) \tag*{\eqref{r14s573}} \\
& [223, 642, 770]^r = [42, 160, 698, 735]^r, \quad (r = 1, 2, 5) \tag*{\eqref{r125s770}} \\
& [32, 51, 106, 115]^r = [7, 16, 66, 95, 120]^r, \quad (r = 1, 2, 3, 5) \tag*{\eqref{r1235s120}} \\
& [19, 44, 95, 102, 136]^r = [4, 11, 52, 84, 110, 135]^r, \;\; (r = 1, 2, 3, 4, 6) \tag*{\eqref{r12346s136}}
\end{align}

\begin{example}
FPTE with all \( r < 0 \)
\end{example}
By applying simple arithmetic transformations using \eqref{LCM}, based on
solutions for FPTE with all \( r > 0 \), so far, ideal positive integer solutions of FPTE have been found for 22 types with all \( r < 0 \). For example:
\begin{align}
& [4521479970, 7216230210]^r \nonumber \\
& \quad = [4395225730, 10581782955, 26510150982]^r, \quad (r = -4, -1) \tag*{\eqref{rn14s26510150982}}
\end{align}

\begin{example}
FPTE with \( r_1 = 0 \) and all others \( r > 0 \)
\end{example}
Ideal positive integer solutions of FPTE have been identified for 5 types with \( r_1 = 0 \) and all others \( r > 0 \). For example:
\begin{align}
& [6]^r = [2, 3]^r, \quad (r = 0) \tag*{\eqref{r0s6}} \\
& [2, 6]^r = [1, 3, 4]^r, \quad (r = 0, 1) \tag*{\eqref{r01s6}} \\
& [432, 650]^r = [18, 20, 780]^r, \quad (r = 0, 2) \tag*{\eqref{r02s780}} \\
& [5, 18, 28]^r = [1, 6, 14, 30]^r, \quad (r = 0, 1, 2) \tag*{\eqref{r012s30}} \\
& [2401, 3216, 3690, 4300]^r = [1, 2460, 3010, 4020, 4116]^r, \nonumber \\
& \qquad\quad \qquad\qquad\qquad\qquad\qquad\qquad\qquad (r = 0, 1, 2, 3) \tag*{\eqref{r0123s4300}}
\end{align}

\begin{example}
FPTE with \( r_n = 0 \) and all others \( r < 0 \)
\end{example}
Ideal positive integer solutions of FPTE have been identified for 3 types with \( r_n = 0 \) and all others \( r < 0 \). For example:
\begin{align}
& [1, 30]^r = [2, 3, 5]^r, \quad (r = -1, 0) \tag*{\eqref{r0n1s30}} \\
& [4, 120]^r = [5, 8, 12]^r, \quad (r = -2, 0) \tag*{\eqref{r0n2s120}} \\
& [20, 27, 83538]^r = [17, 52, 189, 270]^r, \quad (r = -2, -1, 0) \tag*{\eqref{r0n12s83538}}
\end{align}

\begin{example}
\label{exampler5}
FPTE with \( r_1 < 0 \) and \( r_n > 0 \)
\end{example}
Ideal positive integer solutions of FPTE have been identified for 2 types with \( r_1 < 0 \) and \( r_n > 0 \). For example:
\begin{align}
& [5, 45]^r = [8, 18, 24]^r, \quad (r = -1, 1) \tag*{\eqref{r1n1s45}} \\
\label{r01n1}
& [33, 70200, 157872]^r = [54, 143, 208, 227700]^r, \quad (r = -1, 0, 1) \tag*{\eqref{r01n1s227700}}
\end{align}

For Example\,\ref{exampler1} to Example\,\ref{exampler5} mentioned above, the types \( (r = 2) \) and \( (r = 0) \) were already well known, whereas the remaining examples were discovered by Chen Shuwen in 2023 \cite{Chen2125}. For more examples, please refer to Appendix C.\\[1mm]
\indent According to Fermat's Last Theorem, the following two series of FPTE do not have positive integer solutions:
\begin{align}
& a_1^r = b_1^r + b_2^r, \qquad (r \geq 3, \; r \in \mathbb{Z}), \\
& a_1^r = b_1^r + b_2^r, \qquad (r \leq -3, \; r \in \mathbb{Z}).
\end{align}

\subsubsection{Ideal Non-Zero Integer Solution of FPTE}
When \( m = n \), there are integer solutions for system \eqref{FPTE}. In this case, for easy classification, we use \( s \) instead of \( k \) to denote system \eqref{FPTE}, that is
\begin{equation}
\begin{aligned}
\label{FPTEs}
\left[ a_{1}, a_{2}, \dots, a_{n-1} \right]^{s} = \left[ b_{1}, b_{2}, \dots, b_{n} \right]^{s}, \quad (s = s_1, s_2, \dots, s_n).
\end{aligned}
\end{equation}
\begin{definition}
We call a non-zero integer solution of system \eqref{FPTEs} an \textbf{ideal non-zero integer solution of FPTE}.
\end{definition}
So far, ideal non-zero integer solutions of FPTE have been discovered for 13 distinct types of \( (s = s_1, s_2, \dots, s_n) \). These 13 types can be categorized into five groups, containing 4, 4, 1, 2, and 2 types, respectively. For a detailed overview of these 13 types, please refer to Appendix D.

\begin{example}
\label{examples1}
FPTE with all \( s > 0 \)
\end{example}
Ideal non-zero integer solutions of FPTE have been found for 4 types with all \( s > 0 \). For example:
\begin{align}
& [ -7, 7 ]^s = [ -8, 3, 5 ]^s, \quad (s = 1, 2, 4) \tag*{\eqref{s124s7}} \\
& [ -38, -13, 51 ]^s = [ -33, -24, 7, 50 ]^s, \quad (s = 1, 2, 3, 5) \tag*{\eqref{s1235s51}} \\
& [ -47, -46, 37, 56 ]^s = [ -54, -35, -7, 44, 52 ]^s, \quad (s = 1, 2, 3, 4, 6) \tag*{\eqref{s12346s56}} \\
& [ -99, -13, 34, 98 ]^s = [ -82, -58, 16, 69, 75 ]^s, \quad (s = 1, 2, 3, 5, 7) \tag*{\eqref{s12357s98}}
\end{align}

\begin{example}
FPTE with all \( s < 0 \)
\end{example}
By applying simple arithmetic transformations using \eqref{LCM}, based on
solutions for FPTE with all \( s > 0 \), so far, ideal non-zero integer solutions of FPTE have been found for 4 types with all \( s < 0 \). For example:
\begin{align}
& [ -120, 120 ]^s = [ -105, 168, 280 ]^s, \quad (s = -4, -2, -1) \tag*{\eqref{sn124s280}} \\
& [ -3803800, 5105100, 14922600 ]^s \nonumber \\
& \quad = [ -27713400, -3879876, 5878600, 8083075 ]^s, \nonumber \\
& \qquad\qquad\qquad\qquad\qquad\quad (s = -5, -3, -2, -1) \tag*{\eqref{sn1235s14922600}}
\end{align}

\begin{example}
FPTE with \( s_1 = 0 \) and all others \( s > 0 \)
\end{example}
So far, for the types where \( s_1 = 0 \) and all other \( s > 0 \), only one ideal non-zero integer solution has been found, and it corresponds to one particular type. That is,
\begin{align}
& [4]^s = [2, 2]^s, \quad (s = 0, 1) \tag*{\eqref{s01s4}}
\end{align}

It is easy to prove that for the type \( (s = 0, 1) \), there is only one such ideal non-zero integer solution listed above, and this solution only contains positive integers.

\begin{example}
FPTE with \( s_n = 0 \) and all others \( s < 0 \)
\end{example}
Ideal non-zero integer solutions of FPTE have been found for 2 types where \( s_n = 0 \) and all other \( s < 0 \). For example,
\begin{align}
& [ -6 ]^s = [ -2, 3 ]^s, \quad (s = -1, 0) \tag*{\eqref{s0n1s3}} \\
& [ -120, 4 ]^s = [ -12, 5, 8 ]^s, \quad (s = -2, -1, 0) \tag*{\eqref{s0n12s8}}
\end{align}

\begin{example}
\label{examples5}
FPTE with \( s_1 < 0 \) and \( s_n > 0 \)
\end{example}
Ideal non-zero integer solutions of FPTE have been found for 2 types with \( s_1 < 0 \) and \( s_n > 0 \). For example,
\begin{align}
& [ -60, 112 ]^s = [ -10, 14, 48 ]^s, \quad (s = -1, 0, 1) \tag*{\eqref{s01n1s112}} \\
& [ -84, 8, 112 ]^s = [ -16, 14, 14, 24 ]^s, \quad (s = -2, -1, 0, 1) \tag*{\eqref{s01n12s112}}
\end{align}
\indent
For the above Example\,\ref{examples1} to Example\,\ref{examples5}, type \( (s = 0, 1) \) is well known. The remaining examples were first discovered by Chen Shuwen in 2023 \cite{Chen2125}. For more examples, please refer to Appendix D.
\subsection{Multigrade Chains}
During the research of PTE and GPTE, multigrade chains have also attracted attention \cite{Chen17,Gloden44,Hardy38,Hua}.

\begin{problem}\textnormal{\textbf{(The Problem of Multigrade Chains)}}\\[1mm]
The problem of multigrade chains can be stated as: Given \mbox{integers} \( \{k_1, k_2, \dots, k_n\} \), find \( j \) distinct sets of integer solutions \( \{a_{11}, a_{12}, \dots, a_{1m}\} \), \( \{a_{21}, a_{22}, \dots, a_{2m}\} \), \(\dots\), \( \{a_{j1}, a_{j2}, \dots, a_{jm}\} \) for a Diophantine system of the form
\begin{align}
\label{MG}
& \left[ a_{11}, a_{12}, \dots, a_{1m} \right]^{k} = \left[ a_{21}, a_{22}, \dots, a_{2m} \right]^{k} = \cdots 
= \left[ a_{j1}, a_{j2}, \dots, a_{jm} \right]^{k}, \nonumber
\\ 
& \qquad\qquad\qquad\qquad\qquad\qquad\qquad\qquad\qquad \left( k = k_1, k_2, \dots, k_n \right).
\end{align}
\end{problem}
We call \( j \) the length of system \eqref{MG}.

\subsubsection{Ideal Non-Negative Integer Chains of GPTE}
When \( m = n + 1 \), there are non-negative integer chains for system \eqref{MG}. In this case, system \eqref{MG} becomes
\begin{flalign}
\label{MGk}
& \quad [a_{11}, a_{12}, \dots, a_{1(n+1)}]^{k} = [a_{21}, a_{22}, \dots, a_{2(n+1)}]^{k} = \dots = [a_{j1}, a_{j2}, \dots, a_{j(n+1)}]^{k}, \nonumber & \\ 
& \qquad\qquad\qquad\qquad\qquad\qquad\qquad\qquad (k = k_1, k_2, \dots, k_n).&
\end{flalign}
\begin{definition}
We call a non-negative integer chain of system \eqref{MGk} an \textbf{ideal non-negative integer chain of GPTE}.
\end{definition}
So far, ideal non-negative integer chains have been discovered for 30 types of \( (k = k_1, k_2, \dots, k_n) \), and these can be categorized into two groups, with 15 types in each. For a detailed overview of these 30 types, please refer to Appendix A.
\begin{example}
\label{examplec1}
Ideal non-negative integer chains of GPTE for arbitrary \( j \)
\end{example}
Ideal non-negative integer chains of GPTE for arbitrary \( j \) have been solved for 15 types of \( (k = k_1, k_2, \dots, k_n) \). Among these 15 types, 4 types are PTE types (degree 1, 2, 3, and 5). For example:
\begin{align}
\label{k1c}
& [0, 9]^k = [1, 8]^k = [2, 7]^k = [3, 6]^k = [4, 5]^k, \quad (k = 1) \tag*{\eqref{k1s9}} \\
& [13, 91]^k = [23, 89]^k = [35, 85]^k = [47, 79]^k = [65, 65]^k, \quad (k = 2) \tag*{\eqref{k2s91}} \\
& [2421, 19083]^k = [5436, 18948]^k = [10200, 18072]^k = [13322, 16630]^k, \nonumber \\
& \qquad\qquad\qquad\qquad\qquad\qquad\qquad\qquad\qquad\qquad\qquad\qquad (k = 3) \tag*{\eqref{k3s19083}} \\
& [0, 16, 17]^k = [1, 12, 20]^k = [2, 10, 21]^k = [5, 6, 22]^k, \quad (k = 1, 2) \tag*{\eqref{k12s22}} \\
& [23, 25, 48]^k = [15, 32, 47]^k = [8, 37, 45]^k = [3, 40, 43]^k, \;\; (k = 2, 4) \tag*{\eqref{k24s48}} \\
& [0, 28, 29, 57]^k = [1, 21, 36, 56]^k = [2, 18, 39, 55]^k = [6, 11, 46, 51]^k, \nonumber \\
& \qquad\qquad\qquad\qquad\qquad\qquad\qquad\qquad\qquad\qquad\qquad\;\; (k = 1, 2, 3) \tag*{\eqref{k123s57}} \\
& [0, 23, 25, 71, 73, 96]^k = [1, 16, 33, 63, 80, 95]^k \nonumber \\
& \quad = [3, 11, 40, 56, 85, 93]^k = [5, 8, 45, 51, 88, 91]^k, \nonumber \\
& \qquad\qquad\qquad\qquad\qquad\qquad\qquad (k = 1, 2, 3, 4, 5) \tag*{\eqref{k12345s96}} \\
& [1, 24]^k = [2, 12]^k = [3, 8]^k = [4, 6]^k, \quad (k = 0) \tag*{\eqref{k0s24}}
\end{align}
The remaining 7 solvable types for arbitrary \( j \) are \( (k = -1) \), \( (k = -2) \), \( (k = -3) \), \( (k = -2, -1) \), \( (k = -4, -2) \), \( (k = -3, -2, -1) \), and \( (k = -5, -4, -3, -2, -1) \).
\begin{example}
\label{examplec2}
Ideal non-negative integer chains of GPTE for a certain \( j \geq 3 \)
\end{example}
Ideal non-negative integer chains of GPTE for a certain \( j \) have been found for 15 types of \( (k = k_1, k_2, \dots, k_n) \). Among these 15 types, for \( (k = 1, 3) \), \( (k = 0, 1) \), \( (k = -3, -1) \), and \( (k = -1, 0) \), ideal non-negative integer chains have been found where \( j = 65 \). For \( (k = 3) \) and \( (k = -3) \), \( j = 6 \). For the remaining 9 types, \( j = 3 \).
\begin{align}
& [2, 52, 58]^k = [4, 46, 62]^k = [13, 32, 67]^k = [22, 22, 68]^k, \;\; (k = 1, 3) \tag*{\eqref{k13s68}} \\
& [24, 201, 216]^k = [66, 132, 243]^k = [73, 124, 244]^k, \quad (k = 1, 4) \tag*{\eqref{k14s244}} \\
& [14, 37, 39, 64]^k = [16, 29, 46, 63]^k = [19, 24, 49, 62]^k, \;\; (k = 1, 2, 4) \tag*{\eqref{k124s64}} \\
& [21, 169, 183, 273]^k = [43, 113, 229, 261]^k = [53, 99, 241, 253]^k, \nonumber \\
& \qquad\qquad\qquad\qquad\qquad\qquad\qquad\qquad\qquad\qquad (k = 1, 3, 5) \tag*{\eqref{k135s273}} \\
& [14, 50, 54]^k = [15, 40, 63]^k = [18, 30, 70]^k = [21, 25, 72]^k, \nonumber \\
& \qquad\qquad\qquad\qquad\qquad\qquad\qquad\qquad\qquad\quad (k = 0, 1) \tag*{\eqref{k01s72}} \\
& [5, 140, 210]^k = [10, 60, 245]^k = [21, 28, 250]^k, \quad (k = 0, 2) \tag*{\eqref{k02s250}} \\
& [9, 28, 30, 65]^k = [10, 20, 39, 63]^k = [13, 14, 45, 60]^k, \nonumber \;\; (k = 0, 1, 2) \tag*{\eqref{k012s65}} \\
\label{k1n1c}
& [45, 165, 198]^k = [48, 120, 240]^k = [65, 70, 273]^k, \quad (k = -1, 1) \tag*{\eqref{k1n1s273}}
\end{align}

For the above Example \,\ref{examplec1} to \,\ref{examplec2}, types \( (k = 1) \) and \( (k = 0) \) are obviously easy to obtain. Methods for types \( (k = 2) \), \( (k = 1, 2) \), \( (k = 1, 2, 3) \), \( (k = 2, 4) \), and \( (k = 1, 2, 3, 4, 5) \) can be found in A.~Gloden's book \cite{Gloden44} or H.~Hardy's book \cite{Hardy38}. Chains for \( (k = 3) \) with \( j = 4 \) were found by E.~Rosenstiel et al. in 1991. Chains for \( (k = 0, 1) \) with \( j = 4 \) were obtained by J.~G.~Mauldon in 1981. Chains for \( (k = 1, 3) \) with \( j = 65 \) were obtained by Jarosław Wróblewski in 2001. The remaining examples were first found by Chen Shuwen between 1997 and 2023. For more examples, please refer to Appendix A.\\[1mm]
\indent For the type of Example\,\ref{examplec2}, it remains unknown whether the length of the chain is finite or infinite. Additionally, for other types where no chain solution has been found, the existence of such solutions is still unproven.
\subsubsection{Ideal Integer Chains of GPTE}
When \( m = n \), there are integer chains for system \eqref{MG}. For easy classification, in this case, we use \( h \) instead of \( k \) to denote system \eqref{MG}, that is
\begin{align}
\label{MGh}
& [ a_{11}, a_{12}, \dots, a_{1n} ]^{h} = [ a_{21}, a_{22}, \dots, a_{2n} ]^{h} = \dots = [ a_{j1}, a_{j2}, \dots, a_{jn} ]^{h}, \nonumber \\
& \qquad\qquad\qquad\qquad\qquad\qquad\qquad\qquad\quad ( h = h_1, h_2, \dots, h_n ) 
\end{align}
\begin{definition}
We call an integer chain of system \eqref{MGh} an \textbf{ideal integer chain of GPTE}.
\end{definition}
So far, ideal integer chains have been found for 12 types of \( (h = h_1, h_2, \dots, h_n) \), and these can be categorized into two groups: one with 8 types and the other with 4 types. For a detailed overview of these 12 types, please refer to Appendix B.
\begin{example}
\label{examplec3}
Ideal integer chains of GPTE for arbitrary $j$
\end{example}
Ideal integer chains of GPTE for arbitrary $j$ have been solved for 8 types of $(h = h_1, h_2, \dots, h_n)$. For example:
\begin{align}
& [-48, 23, 25]^h = [-47, 15, 32]^h = [-45, 8, 37]^h = [-43, 3, 40]^h, \nonumber \\
& \qquad\qquad\qquad\qquad\qquad\qquad\qquad\qquad\qquad\qquad\;\; (h = 1, 2, 4) \tag*{\eqref{h124s40}} \\
& [-30, 55, 66]^h = [-21, 28, 84]^h = [-26, 39, 78]^h = [-9, 10, 90]^h, \nonumber \\
& \qquad\qquad\qquad\qquad\qquad\qquad\qquad\qquad\qquad\qquad\; (h = -1, 1, 2) \tag*{\eqref{h12n1s90}} \\
& [-7, 9, 56, 72]^h = [-10, 15, 50, 75]^h \nonumber \\
& \quad = [-12, 21, 44, 77]^h = [-13, 26, 39, 78]^h, \quad (h = -1, 1, 2, 3) \tag*{\eqref{h123n1s78}} \\
& [-14, 35, 42, 140, 147, 196]^h = [-13, 26, 52, 130, 156, 195]^h \nonumber \\
& \quad = [-8, 11, 72, 110, 171, 190]^h = [-5, 6, 80, 102, 176, 187]^h, \nonumber \\
& \qquad\qquad\qquad\qquad\qquad\qquad\qquad\qquad\quad (h = -1, 1, 2, 3, 4, 5) \tag*{\eqref{h12345n1s196}}
\end{align}
The remaining 4 solvable types for arbitrary $j$ are $(h = -4, -2, -1)$, $(h = -2, -1, 1)$, $(h = -3, -2, -1, 1)$, and $(h = -5, -4, -3, -2, -1, 1)$.
\begin{example}
\label{examplec4}
Ideal integer chains of GPTE for a certain $j \geq 3$
\end{example}
Ideal integer chains of GPTE for a certain $j \geq 3$ have been found for 4 types of $(h = h_1, h_2, \dots, h_n)$. For example:
\begin{align}
& [-1632, 22, 1610]^h = [-1564, 24, 1540]^h = [-1012, 60, 952]^h \nonumber \\
& \quad = [-840, 92, 748]^h = [-782, 110, 672]^h = [-759, 119, 640]^h \nonumber \\
& \quad = [-660, 184, 476]^h = [-644, 204, 440]^h = [-616, 276, 340]^h, \nonumber \\
& \qquad\qquad\qquad\qquad\qquad\qquad\qquad\qquad\qquad\qquad\quad (h = 0, 1, 3) \tag*{\eqref{h013s1610}} \\
& [-231, 11, 23, 197]^h = [-223, -49, 93, 179]^h \nonumber \\
& \quad = [-217, -69, 137, 149]^h, \quad\; (h = 1, 2, 3, 5) \tag*{\eqref{h1235s197}}
\end{align}
The remaining other 2 solvable types for a certain $j \geq 3$ are $(h = -3, -1, 0)$ and $(h = -5, -3, -2, -1)$.\\[1mm]
\indent
For the above Example \,\ref{examplec3} to \,\ref{examplec4}, the method for type $(h = 1, 2, 4)$ can be found in A. Gloden's book \cite{Gloden44}. The types $(h = -1, 1, 2)$, $(h = -1, 1, 2, 3)$, $(h = -1, 1, 2, 3, 4, 5)$, and $(h = 0, 1, 3)$ were first solved by Ajai Choudhry between 2001 and 2011. The type $(h = 1, 2, 3, 5)$ was first discovered by Chen Shuwen in 1997.\\[1mm]
\indent
We observe that for system \eqref{MGh}, depending on different \( (h = h_1, h_2, \dots, h_n) \), the maximum value of \( j \), denoted as \( j_{\text{max}} \), exhibits the following four possibilities:
\begin{enumerate}
    \item \( j_{\text{max}} = 1 \): Neither system \eqref{GPTEh} nor system \eqref{MGh} has ideal integer \\ solutions. For example, see \eqref{h123n}, \eqref{h2462n}, \eqref{h1251236}, \eqref{equationh012345}, \eqref{equationh01234n1} and \eqref{equationh014}.
    \item \( j_{\text{max}} = 2 \): System \eqref{GPTEh} has ideal integer solutions, but system \eqref{MGh} does not have ideal integer chains. For example, see \eqref{equationh134}, \eqref{equationh126}, \eqref{equationh02n1}.
    \item \( j_{\text{max}} \geq 3 \) but finite: System \eqref{MGh} has ideal integer chains of finite length. For example, see \eqref{equationh234}.
    \item \( j_{\text{max}} \) is unbounded: System \eqref{MGh} has ideal integer chains of arbitrary length. See Example \ref{examplec3}.
\end{enumerate}

\subsubsection{Ideal Prime Chains of GPTE}
\begin{definition}
If a positive solution of system \eqref{MGk} with all \( \{a_{11}, a_{12}, \dots, a_{1(n+1)}\} \), \( \{a_{21}, a_{22}, \dots, a_{2(n+1)}\} \), \(\dots\), \( \{a_{j1}, a_{j2}, \dots, a_{j(n+1)}\} \) are primes, we call it an \textbf{ideal prime chain of GPTE}.
\end{definition}

So far, ideal prime chains have been found for 6 types of \( (k = k_1, k_2, \dots, k_n) \) for a certain \( j \geq 3 \).

\begin{example}
\label{examplec5}
Ideal prime solutions of GPTE for a certain \( j \geq 3 \)
\end{example}

As summarized in Example\,\ref{examplep1}, ideal non-negative prime solutions have been found for 27 types of GPTE. Among these 27 types, ideal prime chains have been found for 6 types, including 4 types of PTE. For example,
\begin{align}
& [7, 53]^k = [13, 47]^k = [17, 43]^k = [19, 41]^k = [23, 37]^k = [29, 31]^k, \nonumber \\
& \qquad\qquad\qquad\qquad\qquad\qquad\qquad\qquad\qquad\qquad\qquad\qquad (k = 1) \tag*{\eqref{k1s53}} \\
& [53, 281]^k = [71, 277]^k = [97, 269]^k = [137, 251]^k \nonumber \\
& \quad = [157, 239]^k = [193, 211]^k, \qquad\qquad (k = 2) \tag*{\eqref{k2s281}} \\
& [11, 107, 113]^k = [17, 83, 131]^k = [23, 71, 137]^k, \quad (k = 1, 2) \tag*{\eqref{k12s137}} \\
& [83, 757, 827]^k = [107, 677, 883]^k \nonumber \\
& \quad = [197, 523, 947]^k = [281, 419, 967]^k, \quad (k = 1, 3) \tag*{\eqref{k13s967}} \\
& [59, 137, 163, 241]^k = [61, 127, 173, 239]^k = [71, 103, 197, 229]^k, \nonumber \\
& \qquad\qquad\qquad\qquad\qquad\qquad\qquad\qquad\qquad\qquad\quad (k = 1, 2, 3) \tag*{\eqref{k123s241}} \\
& [1459, 3943, 4159, 9127, 9343, 11827]^k \nonumber \\
& \quad = [1567, 3187, 5023, 8263, 10099, 11719]^k \nonumber \\ 
& \quad = [1783, 2647, 5779, 7507, 10639, 11503]^k, \quad (k = 1, 2, 3, 4, 5) \tag*{\eqref{k12345s11827}}
\end{align}

\indent
In Example \ref{examplec5}, type \( (k = 1) \) is easy to obtain. Type \( (k = 1, 2, 3) \) was first discovered by Carlos Rivera in 1999, while the remaining four types were first discovered by Chen Shuwen between 2016 and 2023.
\clearpage

\section{The Generalized Girard-Newton Identities}
Over the centuries, various methods have been employed in the study of PTE and GPTE types of Diophantine equations. Since 1993, we have recognized the Girard-Newton Identities as one of the more useful methods. Furthermore, we have developed extensive generalizations of the original identities in several ways \cite{Chen21}. 
\subsection{The Girard-Newton Identities}
\subsubsection{Classical Form of the Girard-Newton Identities}
The classical form of the Girard-Newton Identities can be expressed as follows \cite{NewtonIdentities}:
\begin{identity}
\label{identity1}
Let \( n \) and \( k \) be positive integers, and define \( P_k \) as the sum of the \( k \)-th powers of \( n \) numbers \( (a_1, a_2, \dots, a_n) \):
\begin{equation}
\begin{aligned}
&P_{k}= a_{1}^{k}+a_{2}^{k}+a_{3}^{k}+\cdots+a_{n}^{k}.\\
\end{aligned}
\end{equation}
Additionally, define the elementary symmetric sums \( S_1, S_2, \dots, S_n \) as:
\begin{equation}
\begin{aligned}
&S_1=-(a_1+a_2+a_3+\cdots+a_{n}),\\
&S_2=a_1 a_2+a_1 a_3+a_2 a_3+\cdots+a_{n-1}a_{n},\\
&S_3=-(a_1 a_2 a_3 +a_1 a_2 a_4+\cdots+a_{n-2}a_{n-1}a_{n}),\\
&\qquad\vdots\\
&S_n=(-1)^n (a_1 a_2 a_3\cdots a_n).
\end{aligned}
\end{equation}
Then, the Girard-Newton Identities are given by:
\begin{equation}
\begin{aligned}
&P_1+S_1=0,\\
&P_2+S_1 P_1+2 S_2 =0,\\
&P_3+S_1 P_2+S_2 P_1+3 S_3 =0,\\
&\qquad\vdots\\
&P_n+S_1 P_{n-1}+S_2 P_{n-2}+\cdots +S_{n-1}P_1+n S_n =0,\\
&P_k+S_1 P_{k-1}+S_2 P_{k-2}+\cdots +S_{n-1}P_{k-n+1}+S_n P_{k-n}=0, \quad (k>n).\\
\end{aligned}
\end{equation}
\end{identity}
\refIdentity{identity1} is useful in solving PTE problems, as demonstrated in references \cite{Caley12} and \cite{Caley13}.

\subsubsection{Equivalent Form of the Girard-Newton Identities}
From 1999 to 2021, we have discovered several equivalent forms of \refIdentity{identity1}, which are summarized below as \refIdentity{identity2}. These formulas are significantly more useful for investigating PTE or GPTE problems.
\begin{identity}
\label{identity2}
Let $n$ and $k$ be positive integers. Define $P_k$ as the sum of the $k$-th powers of $n$ numbers $(a_1, a_2, \dots, a_n)$:
\begin{equation}
\begin{aligned}
&P_{k}= a_{1}^{k}+a_{2}^{k}+a_{3}^{k}+\cdots+a_{n}^{k}\\
\end{aligned}
\end{equation}
In the following, we provide six equivalent definitions for $S_k$:\\
1. Recursive Form
\begin{equation}
\label{identityRF} 
\begin{aligned}
\begin{cases} 
S_1=-(P_1)/1\\
S_2=-(P_2+S_1 P_1)/2\\
S_3=-(P_3+S_1 P_2+S_2 P_1)/3\\
S_4=-(P_4+S_1 P_3+S_2 P_2+S_3 P_1)/4\\
\qquad\vdots\\
S_k=-(P_k+S_1 P_{k-1}+S_2 P_{k-2}+\cdots +S_{k-1} P_1)/k
\end{cases} \\ 
\end{aligned}
\end{equation}
2. Determinant Form
\label{identityDF} 
\begin{equation}
\begin{aligned}
S_{k}=\frac{(-1)^k}{k!} \begin{vmatrix}
P_1 & 1 & 0 & \cdots & 0 & 0 & 0 \\ 
P_2 & P_1 & 2 & \cdots & 0 & 0 & 0 \\ 
P_3 & P_2 & P_1 & \cdots & 0 & 0 & 0 \\ 
 \vdots & \vdots & \vdots & & \vdots & \vdots & \vdots\\
P_{k-2} & P_{k-3} & P_{k-4} & \cdots & P_1 & {k-2} & 0 \\ 
P_{k-1} & P_{k-2} & P_{k-3} & \cdots & P_2 & P_1 & {k-1} \\ 
P_{k} & P_{k-1} & P_{k-2} & \cdots & P_3 & P_2 & P_1 \\ 
\end{vmatrix}
\end{aligned}
\end{equation}
3. Polynomial Expansion Form
\begin{equation}
\label{identityPEF} 
\begin{aligned}
S_{k}=\sum_{\substack{0\leq i_1,i_2,i_3,\dots,i_k\leq k\\i_1+2 i_2+\cdots +k\cdot i_{k}=k }}\frac{(-P_1/1)^{i_1} (-P_2/2)^{i_2}(-P_3/3)^{i_3} \cdots (-P_k/k)^{i_k}}{i_1 !\; i_2 ! \cdots i_k !} 
\end{aligned}
\end{equation}
4. Simplified Polynomial Expansion Form
\begin{equation}
\label{identitySPEF} 
\begin{aligned}
& S_{k}=\sum_{\substack{0\leq i_1,i_2,i_3,\dots,i_k\leq k\\i_1+2 i_2+\cdots +k\cdot i_{k}=k }}\frac{ Q_1^{i_1} Q_2^{i_2}Q_3^{i_3} \cdots Q_k^{i_k}}{i_1 !\; i_2 ! \cdots i_k !} 
\end{aligned}
\end{equation}
\quad\: where
\begin{equation}
\begin{aligned}
\qquad Q_k=-\frac{P_k}{k}
\end{aligned}
\end{equation}
5. Factorial Polynomial Form
\begin{equation}
\label{identityFPF} 
\begin{aligned}
& S_{k}=\sum_{\substack{0\leq i_1 \leq 1\\ 0\leq i_2,i_3,\dots,i_k\leq k\\i_1+2 i_2+\cdots +k\cdot i_{k}=k }}\frac{ F_1^{i_1} F_2^{i_2} F_3^{i_3} \cdots F_k^{i_k}}{i_1 !\; i_2 ! \cdots i_k !}
\end{aligned}
\end{equation}
\quad\: where
\begin{equation}
\label{identityFKP}
\qquad F_k = 
\begin{cases} 
-P_1, & k = 1, \\[1mm]
\displaystyle\frac{P_1^k - P_k}{k}, & k > 1.
\end{cases}
\end{equation}

6. Inclusion-Exclusion Factorial Form
\begin{equation}
\label{identitySD} 
\begin{aligned}
& S_{k}=\sum_{\substack{0\leq i_1,i_2,i_3,\dots,i_k\leq 1 \\ i_1+2 i_2+\cdots+k\cdot i_{k}=k \\ }}{ D_1^{i_1} D_2^{i_2}D_3^{i_3} \cdots D_k^{i_k}}
\end{aligned}
\end{equation}
\quad\: where
\begin{equation}
\label{identityDP}
\quad D_k = 
\begin{cases} 
-P_1, & k=1, \\[1mm]
\displaystyle\frac{(P_1^k-P_k)}{k}+\sum_{\substack{f_j \mid k,\: \forall j \\ (\text{non-trivial factors})}} (-1)^{(f_j)} \frac{1}{f_j} D_{k/f_j}^{f_j} , & k>1.
\end{cases}
\end{equation}
Using any of these equivalent definitions of $S_{k}$, we obtain:
\begin{equation}
\label{identitySA} 
\begin{aligned}
\begin{cases}  
S_1=-(a_1+a_2+a_3+\cdots+a_{n})\\
S_2=a_1 a_2+a_1 a_3+a_2 a_3+\cdots+a_{n-1}a_{n}\\
S_3=-(a_1 a_2 a_3 +a_1 a_2 a_4+\cdots+a_{n-2}a_{n-1}a_{n})\\
\qquad\vdots\\
S_n=(-1)^n (a_1 a_2 a_3\cdots a_n)\\
S_k=0, \;(k>n)
\end{cases} \\ 
\end{aligned}
\end{equation}
Additionally,
\begin{equation}
\begin{aligned}
\begin{cases}  
P_1+S_1=0\\
P_2+S_1 P_1+2 S_2 =0\\
P_3+S_1 P_2+S_2 P_1+3 S_3 =0\\
\qquad\vdots\\
P_n+S_1 P_{n-1}+S_2 P_{n-2}+\cdots +S_{n-1}P_1+n S_n =0\\
P_k+S_1 P_{k-1}+S_2 P_{k-2}+\cdots +S_n P_{k-n}=0, \;(k>n)
\end{cases} \\ 
\end{aligned}
\end{equation}
\end{identity}
While the author independently discovered these six definitions \cite{Chen21}, subsequent literature review revealed that the first four forms have been previously established by other researchers \cite{Garrett2010,Kostrikin1996,NewtonIdentities,Xi2016}.

Below is an example of \refIdentity{identity2} for \(n=6\). Since this example is for analyzing types $(k=1,2,3,4,5,6)$ and $(k=1,2,3,4,5,7)$, we only need to consider the definitions of $P_k$ for $1\leq k \leq 7$ and $S_7$.

\begin{example}
Denote 
\begin{equation}
\begin{aligned}
P_k=a_1^k+ a_2^k+ a_3^k+ a_4^k+ a_5^k+ a_6^k,\qquad (k=1,2,3,4,5,6,7)
\end{aligned}
\end{equation}
According to the definition of $S_k$ in \refIdentity{identity2}, we have the following six equivalent expressions for $S_7$:
\begin{equation}
\label{identityS7R} 
\begin{aligned}
& S_7=-(P_7+S_1 P_6+S_2 P_5+S_3 P_4+S_4 P_3+S_5 P_2+S_6 P_1)/7,\\
& \qquad \text{where} \,\,
\begin{cases} 
S_1=-(P_1)/1,\\
S_2=-(P_2+S_1 P_1)/2,\\
S_3=-(P_3+S_1 P_2+S_2 P_1)/3,\\
S_4=-(P_4+S_1 P_3+S_2 P_2+S_3 P_1)/4,\\
S_5=-(P_5+S_1 P_4+S_2 P_3+S_3 P_2+S_4 P_1)/5,\\
S_6=-(P_6+S_1 P_5+S_2 P_4+S_3 P_3+S_4 P_2+S_5 P_1)/6.
\end{cases} 
\end{aligned}
\end{equation}

\begin{equation}
\label{identityS7H} 
\begin{aligned}
S_{7}=-\frac{1}{7!} \begin{vmatrix}
 P_1 & 1 & 0 & 0 & 0 & 0 & 0 \\
 P_2 & P_1 & 2 & 0 & 0 & 0 & 0 \\
 P_3 & P_2 & P_1 & 3 & 0 & 0 & 0 \\
 P_4 & P_3 & P_2 & P_1 & 4 & 0 & 0 \\
 P_5 & P_4 & P_3 & P_2 & P_1 & 5 & 0 \\
 P_6 & P_5 & P_4 & P_3 & P_2 & P_1 & 6 \\
 P_7 & P_6 & P_5 & P_4 & P_3 & P_2 & P_1 \\
\end{vmatrix}
\end{aligned}
\end{equation}
\begin{equation}
\label{identityS7P} 
\begin{aligned}
& S_7=\frac{\left(-P_1\right){}^7}{7!}+\frac{\left(-P_1\right){}^5}{5!}\frac{(-P_2)}{2}+\frac{\left(-P_1\right){}^3}{3!}\frac{\left(\frac{-P_2}{2}\right){}^2}{2!}+\left(-P_1\right)\frac{\left(-\frac{P_2}{2}\right){}^3}{3!}\\
& \qquad +\cdots+\frac{(-P_2)}{2} \frac{(-P_5)}{5}+(-P_1)\frac{(-P_6)}{6}+\frac{(-P_7)}{7}.
\end{aligned}
\end{equation}
\begin{equation}
\label{identityS7Q} 
\begin{aligned}
& S_7=\frac{Q_1^7}{7!}+\frac{Q_1^5}{5!} Q_2+\frac{Q_1^3}{ 3!}\frac{Q_2^2}{2!}+Q_1 \frac{Q_2^3}{3!}+\cdots+Q_2 Q_5+Q_1 Q_6+Q_7,\\
&\qquad \text{where}\quad Q_k=-\frac{P_k}{k},\,\, (k \geq 2).
\end{aligned}
\end{equation}
\begin{equation}
\label{identityS7F} 
\begin{aligned}
& S_7=\frac{F_1 F_2^3}{6} +\frac{F_2^2 F_3}{2} +\frac{F_1 F_3^2}{2} +F_1 F_2 F_4+F_3 F_4 +F_2 F_5+F_1 F_6+F_7,\\
&\qquad \text{where} \quad
\begin{cases}
F_1=-P_1,\\
F_k=(P_1^k-P_k)/k,\,\, (k \geq 2).
\end{cases}
\end{aligned}
\end{equation}
\begin{equation}
\label{identityS7D} 
\begin{aligned}
&S_7=D_1 D_2 D_4+D_3 D_4+D_2 D_5+D_1 D_6+D_7, \\
&\qquad \text{where} \quad
\begin{cases}
D_1=-P_1,\\
D_k=(P_1^k-P_k)/k,\,\,\, (k=2,3,5,7),\\
D_4=(P_1^4-P_4)/4+D_2^2/2,\\
D_6=(P_1^6-P_6)/6-D_2^3/3+D_3^2/2.
\end{cases}
\end{aligned}
\end{equation}
According to \eqref{identitySA}, we have
\begin{equation}
\label{identityS70} 
\begin{aligned}
& S_7=0.
\end{aligned}
\end{equation}
By expanding any one of the formulas labeled \eqref{identityS7R}, \eqref{identityS7H}, \eqref{identityS7P}, \eqref{identityS7Q}, \eqref{identityS7F}, or \eqref{identityS7D} and combining it with \eqref{identityS70}, we can obtain the following same result, which indicates the relationship between $P_1,P_2,\cdots,P_7$:
\begin{equation}
\label{identityS7P0} 
\begin{aligned}
& \frac{-P_1^7}{5040}+\frac{P_1^5 P_2}{240}-\frac{P_1^3 P_2^2}{48}+\frac{P_1 P_2^3}{48} -\frac{P_1^4 P_3}{72}+ \frac{P_1^2 P_2 P_3 }{12}-\frac{P_2^2 P_3}{24}-\frac{P_1 P_3^2}{18} \\
&\qquad +\frac{P_1^3 P_4 }{24}-\frac{P_1 P_2 P_4}{8} +\frac{P_3 P_4}{12}-\frac{P_1^2P_5 }{10} +\frac{P_2 P_5}{10} -\frac{P_1 P_6}{6}-\frac{P_7}{7}=0.
\end{aligned}
\end{equation}
By combining \eqref{identityS7Q} with \eqref{identityS70}, we can obtain:
\begin{equation}
\label{identityS7Q0} 
\begin{aligned}
& \frac{Q_1^7}{5040}+\frac{Q_1^5 Q_2}{120} +\frac{Q_1^3  Q_2^2 }{12}+\frac{ Q_1 Q_2^3}{6}+\frac{ Q_1^4 Q_3}{24}+\frac{Q_1^2 Q_2 Q_3 }{2}+\frac{Q_2^2 Q_3}{2}+\frac{Q_1 Q_3^2 }{2} \\
&\qquad   +\frac{ Q_1^3 Q_4}{6}+ Q_1 Q_2 Q_4 +Q_3 Q_4+\frac{ Q_1^2 Q_5}{2} +Q_2 Q_5+Q_1 Q_6 +Q_7=0.
\end{aligned}
\end{equation}
By combining \eqref{identityS7F} with \eqref{identityS70}, we can obtain:
\begin{equation}
\label{identityS7F0} 
\begin{aligned}
& \frac{F_1 F_2^3}{6} +\frac{F_2^2 F_3}{2} +\frac{F_1 F_3^2}{2}+F_1 F_2 F_4 +F_3 F_4 +F_2 F_5+F_1 F_6+F_7=0.
\end{aligned}
\end{equation}
By combining \eqref{identityS7D} with \eqref{identityS70}, we can obtain:
\begin{equation}
\label{identityS7D0} 
\begin{aligned}
& D_1 D_2 D_4+D_3 D_4+D_2 D_5+D_1 D_6+D_7=0.
\end{aligned}
\end{equation}
\end{example}

Among the four equivalent relationships presented in \eqref{identityS7P0}, \eqref{identityS7Q0}, \eqref{identityS7F0}, and \eqref{identityS7D0}, \eqref{identityS7P0} most directly indicates the relationship between \( P_1, P_2, \dots, P_7 \); \eqref{identityS7Q0} has simpler polynomial coefficients compared to \eqref{identityS7P0}, and \( Q_1, Q_2, \dots, Q_7 \) are completely independent variables; \eqref{identityS7F0} is further simplified compared to \eqref{identityS7Q0}, but it should be noted that \( F_2, F_3, \dots, F_7 \) are all related to \( F_1 \); \eqref{identityS7D0} is the simplest, but it should be noted that \( D_2, D_3, \dots, D_7 \) are not completely independent variables. From \eqref{identityS7P0}, \eqref{identityS7Q0}, or \eqref{identityS7F0}, we can directly derive the following two corollaries:

\begin{corollary}\label{corollaryh123456}
There is no non-trivial real number solution for
\begin{equation}
\begin{aligned}
\label{h123456}
a_1^h+a_2^h+a_3^h+a_4^h+a_5^h+a_6^h =  b_1^h+b_2^h+b_3^h+b_4^h+b_5^h+b_6^h,\\
(h=1,2,3,4,5,6)
\end{aligned}
\end{equation}
\end{corollary}
\begin{proof}[Proof Idea] 
Let
\begin{equation}
\begin{aligned}
\label{P6P6}
& P_h = a_1^h + a_2^h + a_3^h + a_4^h + a_5^h + a_6^h, \\
& P'_h = b_1^h + b_2^h + b_3^h + b_4^h + b_5^h + b_6^h.
\end{aligned}
\end{equation}
If \eqref{h123456} admits a non-trivial real solution, i.e., two distinct sets \(\{a_1, a_2, \dots, a_6\}\) and \(\{b_1, b_2, \dots, b_6\}\) satisfy \(P_h = P'_h\) for \(h = 1, 2, \dots, 6\),
then from \eqref{identityS7P0}, we can deduce \(P_7 = P'_7\). Similarly, with the help of \eqref{identityDF}, we can further derive \(P_h = P'_h\) for all \(h \geq 7\). This implies that \(\{a_1, a_2, \dots, a_6\}\) and \(\{b_1, b_2, \dots, b_6\}\) are identical sets. Therefore, \eqref{h123456} admits only trivial real number solutions.
\end{proof}

\begin{corollary}\label{corollaryh123457}
If 
\begin{equation}
\label{equationh123457}
\begin{aligned}
a_1^h + a_2^h + a_3^h + a_4^h + a_5^h + a_6^h = b_1^h + b_2^h + b_3^h + b_4^h + b_5^h + b_6^h,\\ (h=1,2,3,4,5,7)
\end{aligned}
\end{equation}
has a non-trivial real solution, then
\begin{equation}
\begin{aligned}
a_1 + a_2 + a_3 + a_4 + a_5 + a_6 = b_1 + b_2 + b_3 + b_4 + b_5 + b_6 = 0.
\end{aligned}
\end{equation}
\end{corollary}

\begin{proof}[Proof Idea]  
If \eqref{equationh123457} admits a non-trivial real solution, then from \eqref{P6P6}, it follows that \( P_h = P'_h \) for \( h = 1, 2, 3, 4, 5, 7 \).  
Furthermore, from \eqref{identityS7P0}, we can deduce that \( P_1 P_6 = P'_1 P'_6 \). To satisfy \( P_6 \neq P'_6 \), it is necessary that \( P_1 = P'_1 = 0 \).
\end{proof}
Here is an example of a numerical solution for system \eqref{equationh123457}:
\begin{align}
& [ -71, -44, -20, 31, 37, 67 ]^h= [ -68, -53, 1, 4, 55, 61 ]^h \nonumber \\
& \qquad \qquad\qquad\qquad\qquad\qquad \qquad (h=1,2,3,4,5,7)  \tag*{\eqref{h123457s71}}
\end{align}
We can verify that:
\begin{align}
-71 -44 -20+ 31+ 37+ 67=-68 -53+ 1+ 4+ 55+ 61=0. \nonumber
\end{align}
Below is an example of \refIdentity{identity2} for \(n=3\).

\begin{example}
Denote 
\begin{equation}
\begin{aligned}
P_k=a_1^k+ a_2^k+ a_3^k,\qquad (k=1,2,3,4,5,6)
\end{aligned}
\end{equation}
According to the definition of $S_k$ in \refIdentity{identity2}, we have the following  expressions for $S_1, S_2, S_3, \dots, S_6$:
\begin{align}
\label{identityS1R} 
& S_1=-(P_1)/1,\\
\label{identityS2R} 
& S_2=-(P_2+S_1 P_1)/2,\\
\label{identityS3R} 
& S_3=-(P_3+S_1 P_2+S_2 P_1)/3,\\
\label{identityS4R} 
& S_4=-(P_4+S_1 P_3+S_2 P_2+S_3 P_1)/4,\\
\label{identityS5R} 
& S_5=-(P_5+S_1 P_4+S_2 P_3+S_3 P_2+S_4 P_1)/5,\\
\label{identityS6R} 
& S_6=-(P_6+S_1 P_5+S_2 P_4+S_3 P_3+S_4 P_2+S_5 P_1)/6.
\end{align}
According to \eqref{identitySA}, we have
\begin{align}
\label{identityS40} 
& S_4=0,\\
\label{identityS50} 
& S_5=0,\\
\label{identityS60} 
& S_6=0.
\end{align}
By combining \eqref{identityS4R} with \eqref{identityS40}, \eqref{identityS5R} with \eqref{identityS50}, and \eqref{identityS6R} with \eqref{identityS60}, \mbox{accordingly}, we can obtain:
\begin{align}
\label{identityP40} 
& \frac{P_1^4}{24} - \frac{P_1^2 P_2}{4} + \frac{P_1 P_3}{3} + \frac{P_2^2}{8} - \frac{P_4}{4} = 0, \\[1mm]
\label{identityP50} 
& -\frac{P_1^5}{120} + \frac{P_1^3 P_2}{12} - \frac{P_1^2 P_3}{6} - \frac{P_1 P_2^2}{8} + \frac{P_1 P_4}{4} + \frac{P_2 P_3}{6} - \frac{P_5}{5} = 0, \\[1mm]
\label{identityP60} 
& \frac{P_1^6}{720} - \frac{P_1^4 P_2}{48} + \frac{P_1^3 P_3}{18} + \frac{P_1^2 P_2^2}{16} - \frac{P_1^2 P_4}{8} - \frac{P_1 P_2 P_3}{6} + \frac{P_1 P_5}{5} - \frac{P_2^3}{48}\nonumber \\
& \quad  + \frac{P_3^2}{18} + \frac{P_2 P_4}{8} - \frac{P_6}{6} = 0.
\end{align}
\end{example}

Using \eqref{identityP40}, \eqref{identityP50}, and \eqref{identityP60} above, we analyze the types $(h=1,2,4)$, $(h=1,2,5)$, $(h=1,2,6)$, $(h=1,3,4)$, and $(h=2,3,4)$ in the following several corollaries.

\begin{corollary}\label{corollaryh124}
If
\begin{equation}
\label{equationh124}
\begin{aligned}
a_{11}^h + a_{12}^h + a_{13}^h = a_{21}^h + a_{22}^h + a_{23}^h = \dots = a_{j1}^h + a_{j2}^h + a_{j3}^h, \\
(h = 1, 2, 4)
\end{aligned}
\end{equation}
has a non-trivial real solution, then
\begin{align}
\label{h124a1a2a3}
a_{11} + a_{12} + a_{13} = a_{21} + a_{22} + a_{23} = \dots = a_{j1} + a_{j2} + a_{j3} = 0.
\end{align}
\end{corollary}

\begin{proof}[Proof Idea]
Let
\begin{equation}
\begin{aligned}
\label{P1P3P4}
& P_1 = a_{11}^1 + a_{12}^1 + a_{13}^1 = a_{21}^1 + a_{22}^1 + a_{23}^1 = \dots = a_{j1}^1 + a_{j2}^1 + a_{j3}^1, \\
& P_2 = a_{11}^2 + a_{12}^2 + a_{13}^2 = a_{21}^2 + a_{22}^2 + a_{23}^2 = \dots = a_{j1}^2 + a_{j2}^2 + a_{j3}^2, \\
& P_4 = a_{11}^4 + a_{12}^4 + a_{13}^4 = a_{21}^4 + a_{22}^4 + a_{23}^4 = \dots = a_{j1}^4 + a_{j2}^4 + a_{j3}^4.
\end{aligned}
\end{equation}
Based on \eqref{identityP40}, when $P_1$, $P_2$, and $P_4$ take certain values, $P_3$ can take at least two distinct real values only if $P_1 = 0$.
\end{proof}

\indent Here is an example of a numerical solution for system \eqref{equationh124}:\begin{align}
& [-48, 23, 25 ]^h = [ -47, 15, 32 ]^h = [ -45, 8, 37 ] ^h= [ -43, 3, 40 ] ^h,\nonumber \\
& \qquad\qquad\qquad\qquad\qquad\qquad\qquad\qquad\qquad\qquad\: (h=1,2,4).  \tag*{\eqref{h124s40}}
\end{align}
We can verify that: $-48+23+25=-47+15+32=-45+8+37=-43+3+40=0$. 

\begin{corollary}\label{corollaryh134}
If
\begin{equation}
\label{equationh134}
\begin{aligned}
a_{11}^h + a_{12}^h + a_{13}^h = a_{21}^h + a_{22}^h + a_{23}^h = \dots = a_{j1}^h + a_{j2}^h + a_{j3}^h, \\
(h = 1, 3, 4)
\end{aligned}
\end{equation}
has a non-trivial real solution, then $j = 2$.
\end{corollary}
\begin{proof}[Proof Idea]
Based on \eqref{identityP40}, when $P_1$, $P_3$, and $P_4$ are certain values, $P_2$ can take at most two distinct real values. Therefore, $j = 2$. Consequently, there is no ideal solution chain of length $j \geq 3$ for the type $(h = 1, 3, 4)$.
\end{proof}

\begin{corollary}\label{corollaryh234}
If
\begin{equation}
\label{equationh234}
\begin{aligned}
a_{11}^h + a_{12}^h + a_{13}^h &= a_{21}^h + a_{22}^h + a_{23}^h = \dots = a_{j1}^h + a_{j2}^h + a_{j3}^h, \\
(h&=2,3,4)
\end{aligned}
\end{equation}
has a non-trivial real solution, then $2 \leq j \leq 4$.
\end{corollary}
\begin{proof}[Proof Idea]
Based on \eqref{identityP40}, when $P_2$, $P_3$, $P_4$ are certain values, $P_1$ can take at most four distinct real values. Hence, $2 \leq j \leq 4$. Therefore, the maximum length of the ideal solution chain for type $(h=2,3,4)$ is 4.
\end{proof}

\begin{corollary}\label{corollaryh125}
There is no non-trivial real number solution for
\begin{equation}
\label{equationh125}
\begin{aligned}
a_1^h + a_2^h + a_3^h = b_1^h + b_2^h + b_3^h, \quad (h=1,2,5)
\end{aligned}
\end{equation}
\end{corollary}

\begin{proof}
To prove Corollary\,\ref{corollaryh125}, we first combine \eqref{identityP40} and \eqref{identityP50} to obtain the following equation:
\begin{align}
\label{equationP1235}
P_1^5 - 5 P_1^3 P_2 + 5 P_1^2 P_3 + 5 P_2 P_3 - 6 P_5 = 0
\end{align}
Based on \eqref{equationP1235}, when $P_1$, $P_2$, and $P_5$ take certain values, $P_3$ can take at least two distinct real values only if $P_1^2 + P_2 = 0$. However, this is impossible because $P_1^2 + P_2 = 0$ implies that both $P_1$ and $P_2$ must be zero, which contradicts the non-trivial nature of the solution. Therefore, \eqref{equationh125} has no non-trivial real number solution.
\end{proof}

\begin{corollary}\label{corollaryh126}
If
\begin{equation}
\label{equationh126}
\begin{aligned}
a_{11}^h + a_{12}^h + a_{13}^h = a_{21}^h + a_{22}^h + a_{23}^h = \dots = a_{j1}^h + a_{j2}^h + a_{j3}^h, \\
(h = 1, 2, 6)
\end{aligned}
\end{equation}
has a non-trivial real solution, then $j = 2$.
\end{corollary}
\begin{proof}
To prove Corollary\,\ref{corollaryh126}, we first combine \eqref{equationP1235} and \eqref{identityP60} to obtain the following equation:
\begin{align}
\label{equationP1236}
P_3^2 + P_1 \left(P_1^2 + 3 P_2\right) P_3 + \frac{1}{4} \left(P_1^6 - 3 P_2 P_1^4 - 9 P_2^2 P_1^2 + 3 P_2^3 - 12 P_6\right) = 0.
\end{align}
Based on \eqref{equationP1236}, when $P_1$, $P_2$, and $P_6$ take certain values, $P_3$ can take at most two distinct real values. Therefore, $j = 2$. Consequently, there is no ideal solution chain of length $j \geq 3$ for the type $(h = 1, 2, 6)$.
\end{proof}

\begin{example}
Let $n$ be a positive integer. For for all positive integers $k$, define $P_k$ as the sum of the k-th powers of $n$ numbers $(a_1, a_2, ..., a_n):$
\begin{equation}
\begin{aligned}
&P_{k}= a_{1}^{k}+a_{2}^{k}+a_{3}^{k}+\cdots+a_{n}^{k}
\end{aligned}
\end{equation}
According to the definition \eqref{identityFKP}:
\begin{equation}
F_k = 
\begin{cases} 
-P_1, & k = 1, \\[1mm]
\displaystyle\frac{P_1^k - P_k}{k}, & k > 1.
\end{cases}
\tag*{\eqref{identityFKP}} 
\end{equation}
\noindent Based on the definition \eqref{identityFPF},we have the specific expression of $S_k$ as follows:
\begin{align}
& S_1=F_1\\
& S_2=F_2\\
& S_3=F_1 F_2+F_3\\
& S_4=\frac{F_2^2}{2}+F_1 F_3+F_4\\
& S_5=\frac{1}{2} F_1 F_2^2+F_2 F_3+F_1 F_4+F_5\\
& S_6=\frac{F_2^3}{6}+F_1 F_2 F_3+F_2 F_4+\frac{F_3^2}{2}+F_1 F_5+F_6\\
& S_7=\frac{1}{6} F_1 F_2^3+\frac{1}{2} F_2^2 F_3+F_1 F_2 F_4+F_2 F_5+\frac{1}{2} F_1 F_3^2+F_3 F_4\nonumber\\
&\qquad +F_1 F_6+F_7\\
& S_8=\frac{F_2^4}{24}+\frac{1}{2} F_1 F_2^2 F_3+\frac{1}{2} F_2^2 F_4+\frac{1}{2}  F_2 F_3^2+F_1 F_2 F_5+F_2 F_6\nonumber\\
&\qquad +\frac{F_4^2}{2}+F_1 F_3 F_4+F_3 F_5+F_1 F_7+F_8\\
& \qquad\vdots \nonumber
\end{align}
Then according to \refIdentity{identity2}, we have
\begin{equation}
\begin{aligned}
\begin{cases}  
S_1=-(a_1+a_2+a_3+\cdots+a_{n})\\
S_2=a_1 a_2+a_1 a_3+a_2 a_3+\cdots+a_{n-1}a_{n}\\
S_3=-(a_1 a_2 a_3 +a_1 a_2 a_4+\cdots+a_{n-2}a_{n-1}a_{n})\\
\qquad\vdots\\
S_n=(-1)^n (a_1 a_2 a_3\cdots a_n)\\
S_k=0, \;(k>n)
\end{cases} \\
\end{aligned}
\tag*{\eqref{identitySA}} 
\end{equation}
\end{example}

As a specific example, when $n=3$ and applying \eqref{identitySA}, we obtain
\begin{align}
& F_1=-(a_1+a_2+a_3)\\
& F_2=a_1 a_2+a_1 a_3+a_2 a_3\\
& F_1 F_2+F_3=-a_1 a_2 a_3\\
\label{identityF40}
& \frac{1}{2}F_2^2+F_1 F_3+F_4=0\\
\label{identityF50}
& \frac{1}{2}F_1 F_2^2+F_2 F_3+F_1 F_4+F_5=0\\
\label{identityF60}
& \frac{1}{6}F_2^3+F_1 F_2 F_3+F_2 F_4+\frac{1}{2}F_3^2+F_1 F_5+F_6=0\\
& \qquad\vdots \nonumber
\end{align}
In the example above, after adopting the Factorial Polynomial Form of \eqref{identityFPF} for the definition of \( S_k \), the expression for \( S_k \) becomes significantly more concise. Consequently, \eqref{identityP40}, \eqref{identityP50}, and \eqref{identityP60} are simplified into \eqref{identityF40}, \eqref{identityF50}, and \eqref{identityF60}, respectively. Besides being more intuitive, another important significance of the simplified formulas is that they avoid redundant computations during computer searches, thereby improving computational efficiency.
\begin{example}
Let $n$ be a positive integer. For for all positive integers $k$, define $P_k$ as the sum of the k-th powers of $n$ numbers $(a_1, a_2, ..., a_n):$
\begin{equation}
\begin{aligned}
&P_{k}= a_{1}^{k}+a_{2}^{k}+a_{3}^{k}+\cdots+a_{n}^{k}
\end{aligned}
\end{equation}
According to the definition \eqref{identityDP}, we have the specific expression of $D_i$ as follows:
\begin{align}
& D_1 = -P_1 \\
& D_2 = \frac{1}{2} (P_1^2 - P_2)\\
& D_3 = \frac{1}{3} (P_1^3 - P_3)\\
& D_4 = \frac{1}{4} (P_1^4 - P_4) + \frac{D_2^2}{2} \\
& D_5 = \frac{1}{5} (P_1^5 - P_5)\\
& D_6 = \frac{1}{6}(P_1^6 - P_6) - \frac{D_2^3}{3} + \frac{D_3^2}{2} \\
& D_7 = \frac{1}{7}(P_1^7 - P_7) \\
& D_8 = \frac{1}{8}(P_1^8 - P_8) + \frac{D_2^4}{4}+ \frac{D_4^2}{2} \\
& D_9 = \frac{1}{9}(P_1^9 - P_9) - \frac{D_3^3}{3} \\
& D_{10} = \frac{1}{10}(P_1^{10} - P_{10}) - \frac{D_2^5}{5} + \frac{D_5^2}{2} \\
& \qquad\vdots \nonumber
\end{align}
Based on the definition \eqref{identitySD}, we define $S_i$ as:
\begin{align}
& S_1=D_1\\
& S_2=D_2\\
& S_3=D_1 D_2 + D_3\\
& S_4=D_1 D_3 + D_4\\
& S_5=D_2 D_3 + D_1 D_4+D_5\\
& S_6=D_1 D_2 D_3 + D_2 D_4 + D_1 D_5+D_6\\
& S_7=D_1 D_2 D_4+D_3 D_4+D_2 D_5+D_1 D_6+D_7\\
& S_8=D_1 D_3 D_4+D_1 D_2 D_5+D_3 D_5+D_2 D_6+D_1 D_7+D_8\\
& S_9=D_2 D_3 D_4+D_1 D_3 D_5+D_4 D_5+D_1 D_2 D_6+D_3 D_6+D_2 D_7 \nonumber\\
& \qquad +D_1 D_8+D_9\\
& S_{10}=D_1 D_2 D_3 D_4+D_1 D_4 D_5+D_2 D_3 D_5+D_1 D_3 D_6+D_4 D_6 \nonumber\\
& \qquad +D_1 D_2 D_7+D_3 D_7+D_2 D_8+D_1 D_9+D_{10}\\
& \qquad\vdots \nonumber
\end{align}
Then according to \refIdentity{identity2}, we have
\begin{equation}
\begin{aligned}
\begin{cases}  
S_1=-(a_1+a_2+a_3+\cdots+a_{n})\\
S_2=a_1 a_2+a_1 a_3+a_2 a_3+\cdots+a_{n-1}a_{n}\\
S_3=-(a_1 a_2 a_3 +a_1 a_2 a_4+\cdots+a_{n-2}a_{n-1}a_{n})\\
\qquad\vdots\\
S_n=(-1)^n (a_1 a_2 a_3\cdots a_n)\\
S_k=0, \;(k>n)
\end{cases} \\
\end{aligned}
\tag*{\eqref{identitySA}} 
\end{equation}
\end{example}
As a specific example, when $n=6$ and applying \eqref{identitySA}, we obtain
\begin{equation}
\begin{aligned}
\begin{cases}  
D_1=-(a_1+a_2+a_3+a_4+a_5+a_6)\\
D_2=a_1 a_2+a_1 a_3+a_2 a_3+\cdots+a_5 a_6\\
D_1 D_2 + D_3=-(a_1 a_2 a_3 +a_1 a_2 a_4+\cdots+a_4 a_5 a_6)\\
D_1 D_3 + D_4=a_1 a_2 a_3 a_4 +a_1 a_2 a_3 a_5+\cdots+a_3 a_4 a_5 a_6\\
D_2 D_3 + D_1 D_4+D_5=-(a_1 a_2 a_3 a_4 a_5 +\cdots+a_2 a_3 a_4 a_5 a_6)\\
D_1 D_2 D_3 + D_2 D_4 + D_1 D_5+D_6=a_1 a_2 a_3 a_4 a_5 a_6\\
D_1 D_2 D_4+D_3 D_4+D_2 D_5+D_1 D_6+D_7=0\\
D_1 D_3 D_4+D_1 D_2 D_5+D_3 D_5+D_2 D_6+D_1 D_7+D_8=0\\
\qquad\vdots\\
\end{cases} \\ 
\end{aligned}
\end{equation}
In the example above, after adopting the Inclusion-Exclusion Factorial Form of \eqref{identitySD} for the definition of \( S_k \), all formulas become significantly more concise: all coefficients are reduced to 1, all exponents are reduced to 1, and the number of polynomial terms is minimized. However, it should be noted that \( D_2, D_3, D_4,\dots \) are not completely independent variables.

\subsection{Generalization of the Girard-Newton Identities}
In this subsection, we present three distinct generalizations of the Girard-Newton Identities. All of these results were independently derived by the author between 1993 and 2021 \cite{Chen21}, with the primary objective of finding ideal solutions for PTE and GPTE.
\subsubsection{First Generalization of the Girard-Newton Identities}
In the classical Girard-Newton Identities, all power exponents are restricted to positive integers. For the first generalization of the Girard-Newton Identities (referred to as \refIdentity{identity_GNI1} below), we extend the domain of the power exponents from positive integers to all integers \cite{Chen21}. 
\begin{identity}  
\label{identity_GNI1}   
Let $n$ be a positive integer and $k$ be an integer. We denote by    
\begin{equation} 
\label{identity3P} 
\begin{aligned}  
P_{k}=  
\begin{cases}  
a_{1}^{k}+a_{2}^{k}+\cdots+a_{n}^{k}, & k \neq 0 \\  
a_{1}a_{2}\cdots a_{n}, & k=0  
\end{cases}  
\end{aligned}  
\end{equation} 
the power sums and product, respectively. Furthermore, define  
\begin{equation} 
\label{identity3S}  
\begin{aligned}  
&\begin{cases}  
S_{1} = -P_1, \\  
S_k=-(P_k+S_1 P_{k-1}+S_2 P_{k-2}+\cdots+S_{k-1} P_1)/k, & k > 1 \\  
S_{0} = -P_0, \\  
S_{-1} = P_0P_{-1}, \\  
S_k =(-P_0 P_k+S_{-1} P_{k+1}+S_{-2} P_{k+2}+\cdots+S_{k+1}P_{-1})/k, &k < -1 
\end{cases} \\  
\end{aligned}  
\end{equation}  
and  
\begin{equation}  
\label{identity3T} 
\begin{aligned}  
T_{k}=  
\begin{cases}  
(-1)^n S_{k-n}, & k \leq -1, \\  
1 + (-1)^n S_{-n}, & k = 0, \\  
S_k + (-1)^n S_{k-n}, & 1 \leq k \leq n, \\  
S_k, & k \geq n+1.  
\end{cases}  
\end{aligned}  
\end{equation}    
Then, it follows that    
\begin{equation}  
\label{identity3T0} 
\begin{aligned}  
T_{k} = 0, \quad \text{for all } k.  
\end{aligned}  
\end{equation}    
\end{identity}
In fact, similar to Identity \ref{identity2}, there are also six equivalent forms for the \mbox{definition} of $S_k$ in Identity \ref{identity_GNI1}. For brevity, this paper only presents one of these forms, \mbox{specifically} the Recursive Form, for the definition of $S_k$ in Identity \ref{identity_GNI1}.\\
\indent At present, we are unable to provide a complete mathematical proof for Identity \ref{identity_GNI1}. However, when considering certain specific values of $n$ and $k$, the proof becomes considerably simpler, especially in cases where the absolute values of $n$ and $k$ are small integers. Below, we present some illustrative examples of Identity \ref{identity_GNI1}, each of which can be easily verified using Mathematica.
\begin{example}
$Let$
\begin{equation}
\begin{aligned}  
& P_{k}=  
\begin{cases}  
a_1^k+a_2^k+a_3^k+a_4^k+a_5^k+a_6^k, \quad & (k=-1,1,2,3,4,5) \\ 
a_1\cdot a_2\cdot a_3\cdot a_4\cdot a_5\cdot a_6, & (k=0)  \\
\end{cases}  
\end{aligned}  
\end{equation}
based on the definition \eqref{identity3S},we have the specific expression of $S_k$ as follows:
\begin{equation}
\begin{aligned}
& S_1=-(P_1)/1\\
& S_2=-(P_2+S_1 P_1)/2\\
& S_3=-(P_3+S_1 P_2+S_2 P_1)/3\\
& S_4=-(P_4+S_1 P_3+S_2 P_2+S_3 P_1)/4\\
& S_5=-(P_5+S_1 P_4+S_2 P_3+S_3 P_2+S_4 P_1)/5\\
& S_0=-P_0\\
& S_{-1}=P_0 P_{-1}\\
\end{aligned}
\end{equation}
then according to \eqref{identity3T} and \eqref{identity3T0}, we have
\begin{equation}
\label{identity3T5} 
\begin{aligned}
& T_5=S_5+(-1)^6 S_{-1}=0
\end{aligned}
\end{equation}
by expanding \eqref{identity3T5}, it follows that
\begin{equation}
\begin{aligned}
\label{T012345n1}
& P_{-1} P_0-\frac{P_1^5}{120}+\frac{P_1^3 P_2}{12} -\frac{P_1 P_2^2}{8}  -\frac{P_1^2 P_3}{6}  +\frac{P_2 P_3}{6}+\frac{P_1 P_4}{4}-\frac{P_5}{5}=0
\end{aligned}
\end{equation}
\end{example}
The above \eqref{T012345n1} is the specific form of \refIdentity{identity_GNI1} when $n=6$ and  $-1 \leq k \leq 5$. From \eqref{T012345n1}, we can directly obtain the following four corollaries:

\begin{corollary}\label{corollaryh012345}
There is no non-trivial ideal integer solution for the type \((h = 0, 1, 2, 3, 4, 5)\). That is, the system
\begin{equation}
\label{equationh012345}
\begin{cases}  
a_1^h + a_2^h + \cdots + a_6^h = b_1^h + b_2^h + \cdots + b_6^h, \quad & \text{for }\, h = 1, 2, 3, 4, 5, \\ 
a_1 a_2 a_3 a_4 a_5 a_6 = b_1 b_2 b_3 b_4 b_5 b_6, \\
\end{cases}  
\end{equation}
has no non-trivial integer solution.
\end{corollary}

\begin{proof}[Proof Idea]
Based on \eqref{T012345n1}, when \( P_0 \), \( P_1 \), \( P_2 \), \( P_3 \), \( P_4 \), and \( P_5 \) are fixed, \( P_{-1} \) can take at least two distinct real values only if \( P_0 = 0 \). However, this is impossible because \( P_0 = 0 \) implies that at least one \( a_i \) and one \( b_i \) must be zero, which contradicts the non-trivial nature of the solution. Therefore, \eqref{equationh012345} has no non-trivial real number solution, and consequently, no non-trivial integer solution exists.
\end{proof}

\begin{corollary}
There is no non-trivial real number solution for the type \((h = -1, 0, 1, 2, 3, 4)\). That is, the system
\begin{equation}
\label{equationh01234n1}
\begin{cases}  
a_1^h + a_2^h + \cdots + a_6^h = b_1^h + b_2^h + \cdots + b_6^h, \quad & \text{for }\, h = -1, 1, 2, 3, 4, \\ 
a_1 a_2 a_3 a_4 a_5 a_6 = b_1 b_2 b_3 b_4 b_5 b_6, \\
\end{cases}  
\end{equation}
has no non-trivial integer solution.
\end{corollary}

\begin{proof}[Proof Idea]
The proof idea can be derived by following a similar approach to that of Corollary\,\ref{corollaryh012345}.
\end{proof}

\begin{corollary}\label{corollaryh12345n1}
If ${\{a_1, a_2, \cdots, a_6\}}$ and ${\{b_1, b_2, \cdots, b_6\}}$ are real solutions to the following system
\begin{equation}
\label{equationh12345n1}
\begin{aligned}
[a_1,a_2,a_3,a_4,a_5,a_6]^h=[b_1,b_2,b_3,b_4,b_5,b_6]^h,\\ (h=-1,1,2,3,4,5)
\end{aligned}
\end{equation}
then it follows that
\begin{equation}
\begin{aligned}
\label{fomulah12345n1zero}
\frac{1}{a_1}+\frac{1}{a_2}+\frac{1}{a_3}+\cdots+\frac{1}{a_6}=\frac{1}{b_1}+\frac{1}{b_2}+\frac{1}{b_3}+\cdots+\frac{1}{b_6}=0\\
\end{aligned}
\end{equation}
\end{corollary}
\begin{proof}[Proof Idea]
The proof is based on \eqref{T012345n1} and follows a similar approach to the proof of  Corollary \ref{corollaryh123457}.
\end{proof}
Here is an example of a numerical integer solution for system \eqref{equationh12345n1}:
\begin{equation}
\begin{aligned}
& [-7, 14, 28, 70, 84, 105]^h=[-6, 10, 33, 65, 88, 104]^h,\\ 
& \qquad\qquad\qquad\qquad\qquad\qquad\quad (h=-1,1,2,3,4,5) 
\end{aligned}
\tag*{\eqref{h12345n1s105}}
\end{equation}
We can verify that:
\begin{align}
\frac{1}{-7}+\frac{1}{14}+\frac{1}{28}+\frac{1}{70}+\frac{1}{84}+\frac{1}{105}=\frac{1}{-6}+\frac{1}{10}+\frac{1}{33}+\frac{1}{65}+\frac{1}{88}+\frac{1}{104}=0
\end{align}

\begin{corollary}\label{corollaryh12345n1}
If ${\{a_1, a_2, \cdots, a_6\}}$ and ${\{b_1, b_2, \cdots, b_6\}}$ are real solutions to the following system
\begin{equation}
\label{equationh01235n1}
\begin{aligned}
[a_1,a_2,a_3,a_4,a_5,a_6]^h=[b_1,b_2,b_3,b_4,b_5,b_6]^h,\\ (h=-1,0,1,2,3,5)
\end{aligned}
\end{equation}
then it follows that
\begin{equation}
\begin{aligned}
\label{fomulah01235n1zero}
a_1+a_2+a_3+a_4+a_5+a_6=b_1+b_2+b_3+b_4+b_5+b_6=0\\
\end{aligned}
\end{equation}
\end{corollary}
\begin{proof}[Proof Idea]
The proof is also based on \eqref{T012345n1} and follows a similar approach to the proof of Corollary \ref{corollaryh12345n1}.
\end{proof}
Here is an example of a numerical integer solution for system \eqref{equationh01235n1}:
\begin{equation}
\begin{aligned}
& [ -156, -130, 13, 35, 42, 196 ]^h = [ -147, -140, 14, 26, 52, 195 ]^h,\\ 
& \qquad\qquad\qquad\qquad\qquad\qquad\qquad\qquad\quad (h=-1,0,1,2,3,5) 
\end{aligned}
\tag*{\eqref{h01235n1s196}}
\end{equation}
We can verify that:
\begin{align}
-156 -130+ 13+ 35+ 42+ 196=-147 -140+ 14+ 26+ 52+ 195=0 \nonumber
\end{align}
By utilizing \eqref{fomulah12345n1zero} and \eqref{fomulah01235n1zero}, it can be effectively applied to solving \mbox{problems} of the types \( (h = -1, 1, 2, 3, 4, 5) \) and \( (h = -1, 0, 1, 2, 3, 5) \), whether for \mbox{deriving} parametric solutions or conducting computer searches. (See Type\,\ref{Typeh12345n1}, \ref{Typeh01235n1}, and Example\,\ref{Parah12345n1}.)
\\[1mm]
\indent Below is an example of \refIdentity{identity_GNI1} for \(n=3\).
\begin{example}
$Let$
\begin{equation}
\begin{aligned}  
& P_{k}=  
\begin{cases}  
a_1^k+a_2^k+a_3^k, \quad & (-1,1,2,3,4) \\ 
a_1\cdot a_2\cdot a_3, & (k=0)  \\
\end{cases}  
\end{aligned}  
\end{equation}
based on the definition \eqref{identity3S},we have the specific expression of $S_k$ as follows:
\begin{equation}
\begin{aligned}
& S_1=-(P_1)/1\\
& S_2=-(P_2+S_1 P_1)/2\\
& S_3=-(P_3+S_1 P_2+S_2 P_1)/3\\
& S_4=-(P_4+S_1 P_3+S_2 P_2+S_3 P_1)/4\\
& S_0=-P_0\\
& S_{-1}=P_0 P_{-1}
\end{aligned}
\end{equation}
then according to \eqref{identity3T} and \eqref{identity3T0}, we have
\begin{equation}
\label{identity3T2T3T4} 
\begin{aligned}
& T_2=S_2+(-1)^3 S_{-1}=0\\
& T_3=S_3+(-1)^3 S_0=0\\
& T_4=S_4=0
\end{aligned}
\end{equation}
by expanding \eqref{identity3T2T3T4}, it follows that
\begin{align}
\label{m3T012n1}
& P_{-1} P_0-\frac{P_1^2}{2}+\frac{P_2}{2}=0\\[1mm]
\label{m3T0123}
& P_0-\frac{P_1^3}{6}+\frac{P_1 P_2}{2}-\frac{P_3}{3}=0\\[1mm]
\label{m3T1234}
& \frac{P_1^4}{24}-\frac{P_1^2 P_2}{4}  +\frac{P_2^2}{8}+\frac{P_1 P_3}{3}-\frac{P_4}{4}=0
\end{align}
\end{example}
Using \eqref{m3T012n1}, \eqref{m3T0123}, and \eqref{m3T1234} above, we analyze the types $(h=-1,0,2)$, and $(h=0,1,4)$ in the following corollaries.

\begin{corollary}\label{corollaryh02n1}
If the type \((h = -1, 0, 2)\), namely, the system
\begin{equation}
\label{equationh02n1}
\begin{cases}
a_{11}^2 + a_{12}^2 + a_{13}^2 & = a_{21}^2 + a_{22}^2 + a_{23}^2 \quad = \dots = a_{j1}^2 + a_{j2}^2 + a_{j3}^2, \\
a_{11} \cdot a_{12} \cdot a_{13} &  = a_{21} \cdot a_{22} \cdot a_{23} \qquad = \dots = a_{j1} \cdot a_{j2} \cdot a_{j3}, \\[1mm]
\displaystyle \frac{1}{a_{11}} + \frac{1}{a_{12}} + \frac{1}{a_{13}} & =\displaystyle \frac{1}{a_{21}} + \frac{1}{a_{22}} + \frac{1}{a_{23}}\,\, = \dots = \frac{1}{a_{j1}} + \frac{1}{a_{j2}} + \frac{1}{a_{j3}}.
\end{cases}
\end{equation}
has a non-trivial real solution, then \( j = 2 \).
\end{corollary}

\begin{proof}[Proof Idea]
Based on \eqref{m3T012n1}, when \( P_{-1} \), \( P_0 \), and \( P_2 \) are fixed, \( P_1 \) can take at most two distinct real values. Therefore, \( j = 2 \). Consequently, there exists no ideal solution chain of length \( j \geq 3 \) for the type \( (h = -1, 0, 2) \).
\end{proof}

\begin{corollary}\label{corollaryh014}
There is no non-trivial real number solution for the type \((h = 0, 1, 4)\), namely, the system
\begin{equation}
\label{equationh014}
\begin{cases}
a_1^h + a_2^h + a_3^h = b_1^h + b_2^h + b_3^h, \quad & \text{for }\, h=1,4,\\
a_1\cdot a_2\cdot a_3 =b_1\cdot b_2\cdot b_3.
\end{cases}
\end{equation}
\end{corollary}

\begin{proof}
To prove Corollary\,\ref{corollaryh014}, we first combine \eqref{m3T0123} and \eqref{m3T1234} to obtain the following equation:
\begin{align}
\label{equationP0124}
P_0 P_1-\frac{P_1^4}{8}+\frac{1}{4} P_2 P_1^2+\frac{P_2^2}{8}-\frac{P_4}{4} = 0
\end{align}
Based on \eqref{equationP0124}, when $P_{0}$, $P_1$, and $P_4$ take certain values, $P_2$ has two roots. However, one of the roots, \mbox{$-P_1^2 - \sqrt{2} \sqrt{P_1^4 - 4P_0P_1 + P_4}$}, is never a positive real number, which contradicts the definition of $P_2$. Therefore, \eqref{equationh014} has no non-trivial real number solution.
\end{proof}

\subsubsection{Second Generalization of the Girard-Newton Identities}
By further expanding the defined formula \eqref{identity3P}, we have extended the First Generalization of the Girard-Newton Identities (Identity \ref{identity_GNI1}) to a more general form, called the Second Generalization of the Girard-Newton Identities (hereinafter referred to as Identity \ref{identity_GNI2}). \cite{Chen21}
\begin{identity}
\label{identity_GNI2}
Let $n$ and $m$ be positive integers, and let $k$ be an integer. We introduce the following definitions:
\begin{equation}
\begin{aligned}
P_{k}=\begin{cases}  a_{1}^{k}+a_{2}^{k}+\cdots+a_{n}^{k}-(b_{1}^{k}+b_{2}^{k}+\cdots+b_{m}^{k}),\ \ &k\neq 0 \\[2mm]
\displaystyle -\frac{b_{1}b_{2}\cdots b_{m}}{a_{1}a_{2}\cdots a_{n}}, &k=0
\end{cases}\\
\end{aligned}
\end{equation}
We further define $T_k$ based on the following cases:
\begin{equation}
\begin{aligned}
T_k=\begin{cases} (P_k+T_1P_{k-1}+T_2 P_{k-2}+\cdots+T_{k-1}P_1)/k, &k\geq1\\P_0, &k=0\\(P_0 P_k+T_{-1}P_{k+1}+T_{-2} P_{k+2}+\cdots+T_{k+1}P_{-1})/(-k), &k\leq{-1}\end{cases}\\
\end{aligned}
\end{equation}
Furthermore, we define $S_k$ as:
\begin{equation}
\begin{aligned}
&S_k=\gamma_{k}+\delta_{k},\quad \text{where} \\&\qquad\gamma_{k}=\begin{cases}0, & k<0 \\
1, & k=0 \\T_{k}, & k>0\end{cases}, \quad\delta_{k}=\begin{cases}(-1)^{m-n}T_{k-m+n}, 
& k \leq m-n \\0, & k > m-n
\end{cases}
\end{aligned}
\end{equation}
For ${0}\leq{r}<{n}$, we define the determinant $W_{n k+r}$ as:
\begin{equation}
\begin{aligned}
W_{n k+r}=\begin{vmatrix}
 S_{k} & S_{k-1}  & \cdots & S_{k-n+2}& S_{k-n+1}\\
 S_{k+1} & S_{k}  & \cdots & S_{k-n+3}& S_{k-n+2}\\
 \vdots & \vdots  & \vdots & \vdots & \vdots\\
 S_{k+n-r-1} & S_{k+n-r-2} &  \cdots & S_{k-r+1}& S_{k-r}\\
 S_{k+n-r+1} & S_{k+n-r} &  \cdots & S_{k-r+3}& S_{k-r+2}\\
 \vdots & \vdots &  \vdots & \vdots & \vdots\\
 S_{k+n-1} & S_{k+n-2}  & \cdots & S_{k+1}& S_{k}\\
 S_{k+n} & S_{k+n-1} &  \cdots & S_{k+2}& S_{k+1}\\
\end{vmatrix}.
\end{aligned}
\end{equation}
With these definitions, the following identities hold:
\begin{equation}
\begin{aligned}
& W_{nm}=\prod _{i=1}^{n} \prod _{j=1}^{m} (a_i-b_j)\:,\\
& W_{nk+1}=(a_1+a_2+\cdots+a_n) \cdot W_{nk} \:,\\
& W_{nk+2}=(a_1a_2+a_1a_3+\cdots+a_{n-1}a_n) \cdot W_{nk}\:, \\
& W_{nk+3}=(a_1a_2a_3+a_1a_2a_4+\cdots+a_{n-2}a_{n-1}a_n) \cdot W_{nk}\:, \\
& \quad\quad\vdots \\
& W_{nk+n}=(a_1a_2\cdots a_n) \cdot W_{nk}\:.
\end{aligned}
\end{equation}
\end{identity}
These identities provide relationships between the determinants $W_{n k+r}$ and the products or sums of the elements $a_i$ and $b_i$. It is challenging for us to fully prove Identity 4. Nonetheless, for certain values of $n$, $m$ and $k$, the proof becomes easier, particularly when $n$, $m$ and the absolute values of $k$ are smaller integers. Below, we present several useful examples of Identity 4, all of which have been verified using Mathematica.
\begin{example}
Let $k$ be an integer. We define:
\begin{equation}
\begin{aligned}
P_{k}=\begin{cases}  a_1^k+a_2^k+a_3^k-b_1^k-b_2^k-b_3^k-b_4^k-b_5^k, \ &k\ne0\\-(b_1 b_2 b_3 b_4 b_5)/(a_1 a_2 a_3), &k=0\end{cases}\\
\end{aligned}
\end{equation}
Next, define:
\begin{equation}
\begin{aligned}
T_k=\begin{cases} (P_k+T_1P_{k-1}+T_2 P_{k-2}+\cdots+T_{k-1}P_1)/k, &k\geq1\\P_0, &k=0\\(P_0 P_k+T_{-1}P_{k+1}+T_{-2} P_{k+2}+\cdots+T_{k+1}P_{-1})/(-k), &k\leq{-1}\end{cases}\;\;\\
\end{aligned}
\end{equation}
Furthermore, we define $S_k$ as:
\begin{equation}
\begin{aligned}
S_k=\begin{cases} 
T_k, &k\geq3\\
T_k+T_{k-2}, &1\leq k \leq 2\\
1+T_{-2}, &k=0\\
T_{k-2}, &k\leq{-1}
\end{cases}\;\;\\
\end{aligned}
\end{equation}
With these definitions, we introduce the following determinants:
\begin{equation}
\begin{aligned}
&W_{3k}=
\begin{vmatrix}
 S_k & S_{k-1} & S_{k-2} \\
 S_{k+1} & S_k & S_{k-1} \\
 S_{k+2} & S_{k+1} & S_k \\
\end{vmatrix},\quad
&W_{3k+1}=
\begin{vmatrix}
 S_k & S_{k-1} & S_{k-2} \\
 S_{k+1} & S_k & S_{k-1} \\
 S_{k+3} & S_{k+2} & S_{k+1} \\
\end{vmatrix},\\
&W_{3k+2}=
\begin{vmatrix}
 S_k & S_{k-1} & S_{k-2} \\
 S_{k+2} & S_{k+1} & S_{k} \\
 S_{k+3} & S_{k+2} & S_{k+1} \\
\end{vmatrix}. \\
\end{aligned}
\end{equation}
Then it follows that:
\begin{equation}
\begin{aligned}
&W_{15}=\begin{vmatrix}
 S_{5} & S_{4} & S_{3} \\
 S_{6} & S_{5} & S_{4} \\
 S_{7} & S_{6} & S_{5} \\
\end{vmatrix}
=\prod _{i=1}^3 \prod _{j=1}^5 \left(a_i-b_j\right)\;,\\
&W_{3k+1}=(a_1+a_2+a_3)\cdot W_{3k}\;,\\
&W_{3k+2}=(a_1 a_2+a_1 a_3+a_2 a_3)\cdot W_{3k}\;,\\
&W_{3k+3}=(a_1 a_2 a_3)\cdot W_{3k}\;.
\end{aligned}
\end{equation}
\end{example}
Example 2.5 above illustrates the case of Identity 4 when $n=3$, $m=5$. Example 2.6 below presents the specific formulas of Identity 4 when $n=3$, $m=4$ and $-5 \leq k \leq 8$.
\begin{example}
Let $k$ be an integer. We define:
\begin{equation}
\begin{aligned}
P_{k}=\begin{cases}  a_1^k+a_2^k+a_3^k-b_1^k-b_2^k-b_3^k-b_4^k, \ &k\ne0\\-(b_1 b_2 b_3 b_4)/(a_1 a_2 a_3), &k=0\end{cases}\\
\end{aligned}
\end{equation}
Next, define:
\begin{equation}
\begin{aligned}
& T_1 =(P_1)/1\;,\\
& T_2 =(P_2+T_1 P_1)/2\;,\\
& T_3 =(P_3+T_1 P_2+T_2 P_1)/3\;,\\
& T_4 =(P_4+T_1 P_3+T_2 P_2+T_3 P_1)/4\;,\\
& T_5 =(P_5+T_1 P_4+T_2 P_3+T_3 P_2+T_4P_1)/5\;,\\
& T_6=(P_6+T_1 P_5+T_2 P_4+T_3 P_3+T_4 P_2+T_5 P_1)/6\;,\\
& T_7=(P_7+T_1 P_6+T_2 P_5+T_3 P_4+T_4 P_3+T_5 P_2+T_6 P_1)/7\;,\\
& T_8=(P_8+T_1 P_7+T_2 P_6+T_3 P_5+T_4 P_4+T_5 P_3+T_6 P_2+T_7 P_1)/8\;,\\
& T_0 = P_0\;,\\
& T_{-1}= (P_0 P_{-1})/1\;,\\
& T_{-2}= (P_0 P_{-2}+T_{-1} P_{-1})/2\;,\\
& T_{-3}= (P_0 P_{-3}+T_{-1} P_{-2}+T_{-2} P_{-1})/3\;,\\
& T_{-4}= (P_0 P_{-4}+T_{-1} P_{-3}+T_{-2} P_{-2}+T_{-3} P_{-1})/4\;,\\
& T_{-5}= (P_0 P_{-5}+T_{-1} P_{-4}+T_{-2} P_{-3}+T_{-3} P_{-2}+T_{-4} P_{-1})/5\;.
\end{aligned}
\end{equation}
Furthermore, define:
\begin{equation}
\begin{aligned}
S_k=\begin{cases} 
T_k, &k\geq2\\
T_1-T_0, &k=1\\
1-T_{-1}, &k=0\\
T_{k-1}, &k\leq{-1}
\end{cases}\;\;\\
\end{aligned}
\end{equation}
Then it follows that:
\begin{align}
&W_{12}=\begin{vmatrix}
 S_{4} & S_{3} & S_{2} \\
 S_{5} & S_{4} & S_{3} \\
 S_{6} & S_{5} & S_{4} \\
\end{vmatrix}
=\prod _{i=1}^3 \prod _{j=1}^4 \left(a_i-b_j\right)\;,\\
&W_{-3}=
\begin{vmatrix}
 S_{-1} & S_{-2} & S_{-3} \\
 S_{0} & S_{-1} & S_{-2} \\
 S_{1} & S_{0} & S_{-1} \\
\end{vmatrix}
=\frac{W_{12}}{a_1^5 a_2^5 a_3^5}\; ,\\
&W_{-2}=
\begin{vmatrix}
 S_{-1} & S_{-2} & S_{-3} \\
 S_{0} & S_{-1} & S_{-2} \\
 S_{2} & S_{1} & S_{0} \\
\end{vmatrix}
=(a_1+a_2+a_3)\cdot W_{-3}\; ,\\
&W_{-1}=
\begin{vmatrix}
 S_{-1} & S_{-2} & S_{-3} \\
 S_{1} & S_{0} & S_{-1} \\
 S_{2} & S_{1} & S_{0} \\
\end{vmatrix}
=(a_1 a_2+a_1 a_3+a_2 a_3)\cdot W_{-3}\; ,\\
&W_{0}=
\begin{vmatrix}
 S_{0} & S_{-1} & S_{-2} \\
 S_{1} & S_{0} & S_{-1} \\
 S_{2} & S_{1} & S_{0} \\
\end{vmatrix}
=\frac{W_{12}}{a_1^4 a_2^4 a_3^4}\; ,\\
&W_{1}=
\begin{vmatrix}
 S_{0} & S_{-1} & S_{-2} \\
 S_{1} & S_{0} & S_{-1} \\
 S_{3} & S_{2} & S_{1} \\
\end{vmatrix}
=(a_1+a_2+a_3)\cdot W_{0}\; ,\\
&W_{2}=
\begin{vmatrix}
 S_{0} & S_{-1} & S_{-2} \\
 S_{2} & S_{1} & S_{0} \\
 S_{3} & S_{2} & S_{1} \\
\end{vmatrix}
=(a_1 a_2+a_1 a_3+a_2 a_3)\cdot W_{0}\; ,\\
&W_{3}=
\begin{vmatrix}
 S_{1} & S_{0} & S_{-1} \\
 S_{2} & S_{1} & S_{0} \\
 S_{3} & S_{2} & S_{1} \\
\end{vmatrix}
=\frac{W_{12}}{a_1^3 a_2^3 a_3^3}\; ,\\
& \quad\quad\vdots \nonumber\\
&W_{13}=
\begin{vmatrix}
 S_{4} & S_{3} & S_{2} \\
 S_{5} & S_{4} & S_{3} \\
 S_{7} & S_{6} & S_{5} \\
\end{vmatrix}
=(a_1+a_2+a_3)\cdot W_{12}\; ,\\
&W_{14}=
\begin{vmatrix}
 S_{4} & S_{3} & S_{2} \\
 S_{6} & S_{5} & S_{4} \\
 S_{7} & S_{6} & S_{5} \\
\end{vmatrix}
=(a_1 a_2+a_1 a_3+a_2 a_3)\cdot W_{12}\; ,\\
&W_{15}=
\begin{vmatrix}
 S_{5} & S_{4} & S_{3} \\
 S_{6} & S_{5} & S_{4} \\
 S_{7} & S_{6} & S_{5} \\
\end{vmatrix}
=a_1 a_2 a_3\cdot W_{12}\; ,\\
&W_{16}=
\begin{vmatrix}
 S_{5} & S_{4} & S_{3} \\
 S_{6} & S_{5} & S_{4} \\
 S_{8} & S_{7} & S_{6} \\
\end{vmatrix}
=a_1 a_2 a_3(a_1+a_2+a_3)\cdot W_{12}\; ,\\
&W_{17}=
\begin{vmatrix}
 S_{5} & S_{4} & S_{3} \\
 S_{7} & S_{6} & S_{5} \\
 S_{8} & S_{7} & S_{6} \\
\end{vmatrix}
=a_1 a_2 a_3(a_1 a_2+a_1 a_3+a_2 a_3)\cdot  W_{12}\; ,\\
&W_{18}=
\begin{vmatrix}
 S_{6} & S_{5} & S_{4} \\
 S_{7} & S_{6} & S_{5} \\
 S_{8} & S_{7} & S_{6} \\
\end{vmatrix}
=a_1^2 a_2^2 a_3^2\cdot  W_{12}\; .
\end{align}
\end{example}

Example 2.7 below presents the specific formulas of Identity 4 when $n=3$, $m=4$ and $1 \leq k \leq 7$.
\begin{example}
\label{example_k1234567}
Let
\begin{equation}
\begin{aligned}
P_k=a_1^k+a_2^k+a_3^k-b_1^k-b_2^k-b_3^k-b_4^k,\\(k=1,2,3,4,5,6,7)
\end{aligned}
\end{equation}
Define
\begin{equation}
\begin{aligned}
& S_1=(P_1)/1\\
& S_2=(P_2+S_1 P_1)/2\\
& S_3=(P_3+S_1 P_2+S_2 P_1)/3\\
& S_4=(P_4+S_1 P_3+S_2 P_2+S_3 P_1)/4\\
& S_5=(P_5+S_1 P_4+S_2 P_3+S_3 P_2+S_4 P_1)/5\\
& S_6=(P_6+S_1 P_5+S_2 P_4+S_3 P_3+S_4 P_2+S_5 P_1)/6\\
& S_7=(P_7+S_1 P_6+S_2 P_5+S_3 P_4+S_4 P_3+S_5 P_2+S_6 P_1)/7\\
\end{aligned}
\end{equation}
Furthermore, define
\begin{equation}
\begin{aligned}
&W_{12}=
\begin{vmatrix}
 S_4 & S_3 & S_2 \\
 S_5 & S_4 & S_3 \\
 S_6 & S_5 & S_4 \\
\end{vmatrix},\quad
&W_{13}=
\begin{vmatrix}
 S_4 & S_3 & S_2 \\
 S_5 & S_4 & S_3 \\
 S_7 & S_6 & S_5 \\
\end{vmatrix},\\
&W_{14}=
\begin{vmatrix}
 S_4 & S_3 & S_2 \\
 S_6 & S_5 & S_4 \\
 S_7 & S_6 & S_5 \\
\end{vmatrix},\quad
&W_{15}=
\begin{vmatrix}
 S_5 & S_4 & S_3 \\
 S_6 & S_5 & S_4 \\
 S_7 & S_6 & S_5 \\
\end{vmatrix}.\\
\end{aligned}
\end{equation}
Then it follows that
\begin{equation}
\begin{aligned}
\label{n3m4k1to7}
&W_{12}=\prod _{i=1}^3 \prod _{j=1}^4 \left(a_i-b_j\right)\;,\\
&W_{13}=(a_1+a_2+a_3)\cdot W_{12}\;,\\
&W_{14}=(a_1 a_2+a_1 a_3+a_2 a_3)\cdot W_{12}\;,\\
&W_{15}=(a_1 a_2 a_3)\cdot W_{12}\;.\\
\end{aligned}
\end{equation}
\end{example}
The above \eqref{n3m4k1to7} is very efficient for searching the ideal non-symmetric \mbox{solutions} of PTE of degree 7, that is
\begin{equation}
\begin{aligned}
[0,a_1,a_2,a_3,a_4,a_5,a_6,a_7]^k=[b_1,b_2,b_3,b_4,b_5,b_6,b_7,b_8]^k,\\ (k=1,2,3,4,5,6,7)
\end{aligned}
\end{equation}
We notice that
\begin{equation}
\begin{aligned}
& P_k=a_1^k+a_2^k+a_3^k-b_1^k-b_2^k-b_3^k-b_4^k\\ 
&\quad =b_5^k+b_6^k+b_7^k+b_8^k-a_4^k-a_5^k-a_6^k-a_7^k \; ,(k=1,2,3,4,5,6,7)
\end{aligned}
\end{equation}
For instance, when ${\{b_5,b_6,b_7,b_8,a_4,a_5,a_6,a_7\}=\{63, 72, 88, 95, 54, 81, 82, 96\}}$, we find that $W_{13}/W_{12}=815091/11000$, which is not a positive integer. At this point, we stop the current iteration and proceed to the next one. \\
When ${\{b_5,b_6,b_7,b_8,a_4,a_5,a_6,a_7\}=\{63, 72, 88, 95, 53, 81, 82, 96\}}$, we obtain the results $W_{13}/W_{12}=80$, $W_{14}/W_{12}=1661$, and $W_{15}/W_{12}=8050$, all of which are positive integers. We then continue with the subsequent steps and ultimately derive the smallest ideal non-symmetric solution for the PTE of degree 7:
\begin{align}
& [0, 7, 23, 50, 53, 81, 82, 96]^k = [ 1, 5, 26, 42, 63, 72, 88, 95 ] ^k\;,\nonumber\\
& \qquad\qquad\qquad\qquad\qquad\qquad\qquad\quad (k=1,2,3,4,5,6,7) \tag*{\eqref{k1234567s96}}
\end{align}
The above method was first discovered by the author around 1995 and has been successfully applied to finding numerical solutions for a large number of GPTE problems. 
\begin{example}
Let
\begin{equation}
\begin{aligned}
P_k=a_1^k+a_2^k-b_1^k-b_2^k-b_3^k-b_4^k-b_5^k,\\(k=1,2,3,4,5,6,7)
\end{aligned}
\end{equation}
Define
\begin{equation}
\begin{aligned}
& S_1=(P_1)/1\\
& S_2=(P_2+S_1 P_1)/2\\
& S_3=(P_3+S_1 P_2+S_2 P_1)/3\\
& S_4=(P_4+S_1 P_3+S_2 P_2+S_3 P_1)/4\\
& S_5=(P_5+S_1 P_4+S_2 P_3+S_3 P_2+S_4 P_1)/5\\
& S_6=(P_6+S_1 P_5+S_2 P_4+S_3 P_3+S_4 P_2+S_5 P_1)/6\\
& S_7=(P_7+S_1 P_6+S_2 P_5+S_3 P_4+S_4 P_3+S_5 P_2+S_6 P_1)/7\\
\end{aligned}
\end{equation}
Furthermore, define
\begin{equation}
\begin{aligned}
&W_{10}=
\begin{vmatrix}
 S_5 & S_4 \\
 S_6 & S_5 \\
\end{vmatrix},\quad
&W_{11}=
\begin{vmatrix}
 S_5 & S_4 \\
 S_7 & S_6 \\
\end{vmatrix},\quad
&W_{12}=
\begin{vmatrix}
 S_6 & S_5 \\
 S_7 & S_6 \\
\end{vmatrix}.\\
\end{aligned}
\end{equation}
Then it follows that
\begin{equation}
\begin{aligned}
&W_{10}=\prod _{i=1}^2 \prod _{j=1}^5 \left(a_i-b_j\right)\:,\\
&W_{11}=(a_1+a_2)\cdot W_{10}\:,\\
&W_{12}=(a_1 a_2)\cdot W_{12}\:.\\
\end{aligned}
\end{equation}
\end{example}
Example 2.8 above presents the specific formulas of Identity 4 when $n=2$ and $m=5$. Meanwhile, Example 2.9 below presents the specific formulas of Identity 4 when $n=1$ and $m=6$.

\begin{example}
Let
\begin{equation}
\begin{aligned}
P_k=a_1^k-b_1^k-b_2^k-b_3^k-b_4^k-b_5^k-b_6^k,\\(k=1,2,3,4,5,6,7)
\end{aligned}
\end{equation}
Define
\begin{equation}
\begin{aligned}
& S_1=(P_1)/1\\
& S_2=(P_2+S_1 P_1)/2\\
& S_3=(P_3+S_1 P_2+S_2 P_1)/3\\
& S_4=(P_4+S_1 P_3+S_2 P_2+S_3 P_1)/4\\
& S_5=(P_5+S_1 P_4+S_2 P_3+S_3 P_2+S_4 P_1)/5\\
& S_6=(P_6+S_1 P_5+S_2 P_4+S_3 P_3+S_4 P_2+S_5 P_1)/6\\
& S_7=(P_7+S_1 P_6+S_2 P_5+S_3 P_4+S_4 P_3+S_5 P_2+S_6 P_1)/7\\
\end{aligned}
\end{equation}
Furthermore, define
\begin{equation}
\begin{aligned}
&W_6= S_6,\qquad &W_7= S_7 \:.\\
\end{aligned}
\end{equation}
Then it follows that
\begin{equation}
\begin{aligned}
&W_{6}=(a_1-b_1) (a_1-b_2) (a_1-b_3) (a_1-b_4) (a_1-b_5) (a_1-b_6) \:,\\
&W_{7}=a_1\cdot W_{6}\:.\\
\end{aligned}
\end{equation}
\end{example}

\begin{example}
Let
\begin{equation}
\begin{aligned}
&P_k=a_1^k+a_2^k-b_1^k-b_2^k,\quad (k=2,4,6,8)\\
\end{aligned}
\end{equation}
Define
\begin{equation}
\begin{aligned}
&S_2=(P_2)/1\\
&S_4=(P_4+S_2 P_2)/2\\
&S_6=(P_6+S_2 P_4+S_4 P_2)/3\\
&S_8=(P_8+S_2 P_6+S_4 P_4+S_6 P_2)/4\\
\end{aligned}
\end{equation}
Furthermore, define
\begin{equation}
\begin{aligned}
&W_{8}=
\begin{vmatrix}
 S_4 & S_2 \\
 S_6 & S_4 \\
\end{vmatrix}
,\quad
&W_{10}=
\begin{vmatrix}
 S_4 & S_2 \\
 S_8 & S_6 \\
\end{vmatrix}
,\quad
&W_{12}=
\begin{vmatrix}
 S_6 & S_4 \\
 S_8 & S_6 \\
\end{vmatrix}
\end{aligned}
\end{equation}
Then it follows that
\begin{equation}
\begin{aligned}
\label{n2m2k2468}
&W_{8}=(a_1^2-b_1^2)(a_2^2-b_1^2)(a_1^2-b_2^2)(a_2^2-b_2^2) \:,\\
&W_{10}=(a_1^2+a_2^2) \cdot W_{8}\:,\\
&W_{12}=(a_1^2\cdot a_2^2) \cdot W_{8}\:.\\
\end{aligned}
\end{equation}
\end{example}
By applying \eqref{n2m2k2468}, we found the below ideal non-negative solution of GPTE in 2023.
\begin{equation}
\begin{aligned}
\left[ 498, 3773, 3783, 4567, 4787 \right]^{k}&=\left[ 517, 3598, 4017, 4463, 4827 \right]^{k}\\
& (k=2,4,6,8)
\end{aligned}
\tag*{\eqref{k2468s4827}}
\end{equation}

\begin{example}
Let
\begin{equation}
\begin{aligned}
P_k=a_1^k+a_2^k+a_3^k+a_4^k-b_1^k-b_2^k-b_3^k-b_4^k-b_5^k,\\(k=1,2,3,4,5,6,7,8,9)
\end{aligned}
\end{equation}
and
\begin{equation}
\begin{aligned}
S_k=(P_{k}+S_1 P_{k-1}+S_2 P_{k-2}+...+ S_{k-1}P_1)/k,\\(k=1,2,3,4,5,6,7,8,9)
\end{aligned}
\end{equation}
and
\begin{equation}
\begin{aligned}
&W_{20}=
\begin{vmatrix}
 S_5 & S_4 & S_3 & S_2 \\
 S_6 & S_5 & S_4 & S_3 \\
 S_7 & S_6 & S_5 & S_4 \\
 S_8 & S_7 & S_6 & S_5 \\
\end{vmatrix},\quad
&W_{21}=
\begin{vmatrix}
 S_5 & S_4 & S_3 & S_2 \\
 S_6 & S_5 & S_4 & S_3 \\
 S_7 & S_6 & S_5 & S_4 \\
 S_9 & S_8 & S_7 & S_6 \\
\end{vmatrix},\\
&W_{22}=
\begin{vmatrix}
 S_5 & S_4 & S_3 & S_2 \\
 S_6 & S_5 & S_4 & S_3 \\
 S_8 & S_7 & S_6 & S_5 \\
 S_9 & S_8 & S_7 & S_6 \\
\end{vmatrix},\quad
&W_{23}=
\begin{vmatrix}
 S_5 & S_4 & S_3 & S_2 \\
 S_7 & S_6 & S_5 & S_4 \\
 S_8 & S_7 & S_6 & S_5 \\
 S_9 & S_8 & S_7 & S_6 \\
\end{vmatrix},\\
&W_{24}=
\begin{vmatrix}
 S_6 & S_5 & S_4 & S_3 \\
 S_7 & S_6 & S_5 & S_4 \\
 S_8 & S_7 & S_6 & S_5 \\
 S_9 & S_8 & S_7 & S_6 \\
\end{vmatrix}\\
\end{aligned}
\label{solTEP3}
\end{equation}
then
\begin{equation}
\begin{aligned}
&W_{20}=\prod _{i=1}^4 \prod _{j=1}^5 \left(a_i-b_j\right)\;,\\
&W_{21}=(a_1+a_2+a_3+a_4)\cdot W_{20}\;,\\
&W_{22}=(a_1 a_2+a_1 a_3+a_1 a_4+a_2 a_3+a_2 a_4+a_3 a_4)\cdot W_{20}\;,\\
&W_{23}=(a_1 a_2 a_3+a_1 a_2 a_4+a_1 a_3 a_4+a_2 a_3 a_4)\cdot W_{20}\;,\\
&W_{24}=(a_1 a_2 a_3 a_4)\cdot W_{20}\\
\end{aligned}
\label{solTEP3}
\end{equation}
\end{example}

\begin{example}
$Let$
\begin{equation}
\begin{aligned}
P_k=a_1^k+a_2^k+a_3^k+a_4^k-b_1^k-b_2^k-b_3^k-b_4^k-b_5^k-b_6^k,\\(k=1,2,3,4,5,6,7,8,9,10)
\end{aligned}
\end{equation}
$and$
\begin{equation}
\begin{aligned}
& S_1=(P_1)/1\\
& S_2=(P_2+S_1 P_1)/2\\
& S_3=(P_3+S_1 P_2+S_2 P_1)/3\\
& \quad\cdots\\
& S_{10}=(P_{10}+S_1 P_9+S_2 P_8+\cdots+S_9 P_1)/10
\end{aligned}
\end{equation}
$and$
\begin{equation}
\begin{aligned}
&W_{24}=
\begin{vmatrix}
 S_6 & S_5 & S_4 & S_3 \\
 S_7 & S_6 & S_5 & S_4 \\
 S_8 & S_7 & S_6 & S_5 \\
 S_9 & S_8 & S_7 & S_6 \\
\end{vmatrix},\quad
&W_{25}=
\begin{vmatrix}
 S_6 & S_5 & S_4 & S_3 \\
 S_7 & S_6 & S_5 & S_4 \\
 S_8 & S_7 & S_6 & S_5 \\
 S_{10} & S_9 & S_8 & S_7 \\
\end{vmatrix},\\
&W_{26}=
\begin{vmatrix}
 S_6 & S_5 & S_4 & S_3 \\
 S_7 & S_6 & S_5 & S_4 \\
 S_9 & S_8 & S_7 & S_6 \\
 S_{10} & S_9 & S_8 & S_7 \\
\end{vmatrix},\quad
&W_{27}=
\begin{vmatrix}
 S_6 & S_5 & S_4 & S_3 \\
 S_8 & S_7 & S_6 & S_5 \\
 S_9 & S_8 & S_7 & S_6 \\
 S_{10} & S_9 & S_8 & S_7 \\
\end{vmatrix},\\
&W_{28}=
\begin{vmatrix}
 S_7 & S_6 & S_5 & S_4 \\
 S_8 & S_7 & S_6 & S_5 \\
 S_9 & S_8 & S_7 & S_6 \\
 S_{10} & S_9 & S_8 & S_7 \\
\end{vmatrix},\quad\\
\end{aligned}
\end{equation}
$then$
\begin{equation}
\begin{aligned}
\label{n6m4k1to10}
&W_{24}=\prod _{i=1}^4 \prod _{j=1}^6 \left(a_i-b_j\right)\;,\\
&W_{25}=(a_1+a_2+a_3+a_4)\cdot W_{24} \;,\\
&W_{26}=(a_1 a_2+a_1 a_3+a_1 a_4+a_2 a_3+a_2 a_4+a_3 a_4)\cdot W_{24} \;,\\
&W_{27}=(a_1 a_2 a_3+a_1 a_2 a_4+a_1 a_3 a_4+a_2 a_3 a_4)\cdot W_{24} \;,\\
&W_{28}=(a_1 a_2 a_3 a_4)\cdot W_{24}
\end{aligned}
\end{equation}
\end{example}
The formula \eqref{n6m4k1to10} is useful for searching for ideal non-symmetric solutions of PTE of degree 10. Based on \eqref{n6m4k1to10}, we have searched for more than 6 years with several computers, but still have not found such a solution within the range of 1600.

\begin{example}
Let
\begin{equation}
\begin{aligned}
& P_0=-\frac{b_1 b_2 b_3 b_4}{a_1 a_2 a_3 a_4}\;,\\
& P_k=a_1^k+a_2^k+a_3^k+a_4^k-b_1^k-b_2^k-b_3^k-b_4^k,\ (k=1,2,3,4,5,6,7)\;.
\end{aligned}
\end{equation}
Define
\begin{equation}
\begin{aligned}
& S_0=P_0+1\;,\\
& S_1=(P_1)/1\;,\\
& S_2=(P_2+S_1 P_1)/2\;,\\
& S_3=(P_3+S_1 P_2+S_2 P_1)/3\;,\\
& S_4=(P_4+S_1 P_3+S_2 P_2+S_3 P_1)/4\;,\\
& S_5=(P_5+S_1 P_4+S_2 P_3+S_3 P_2+S_4 P_1)/5\;,\\
& S_6=(P_6+S_1 P_5+S_2 P_4+S_3 P_3+S_4 P_2+S_5 P_1)/6\;,\\
& S_7=(P_7+S_1 P_6+S_2 P_5+S_3 P_4+S_4 P_3+S_5 P_2+S_6 P_1)/7\;.
\end{aligned}
\end{equation}
Furthermore, define
\begin{equation}
\begin{aligned}
&W_{12}=
\begin{vmatrix}
 S_3 & S_2 & S_1 & S_0 \\
 S_4 & S_3 & S_2 & S_1 \\
 S_5 & S_4 & S_3 & S_2 \\
 S_6 & S_5 & S_4 & S_3 \\
\end{vmatrix},\quad
&W_{13}=
\begin{vmatrix}
 S_3 & S_2 & S_1 & S_0 \\
 S_4 & S_3 & S_2 & S_1 \\
 S_5 & S_4 & S_3 & S_2 \\
 S_7 & S_6 & S_5 & S_4 \\
\end{vmatrix},\\
&W_{14}=
\begin{vmatrix}
 S_3 & S_2 & S_1 & S_0 \\
 S_4 & S_3 & S_2 & S_1 \\
 S_6 & S_5 & S_4 & S_3 \\
 S_7 & S_6 & S_5 & S_4 \\
\end{vmatrix},\quad
&W_{15}=
\begin{vmatrix}
 S_3 & S_2 & S_1 & S_0 \\
 S_5 & S_4 & S_3 & S_2 \\
 S_6 & S_5 & S_4 & S_3 \\
 S_7 & S_6 & S_5 & S_4 \\
\end{vmatrix},\\
&W_{16}=
\begin{vmatrix}
 S_4 & S_3 & S_2 & S_1 \\
 S_5 & S_4 & S_3 & S_2 \\
 S_6 & S_5 & S_4 & S_3 \\
 S_7 & S_6 & S_5 & S_4 \\
\end{vmatrix}.\quad\\
\end{aligned}
\end{equation}
Then it follows that
\begin{equation}
\begin{aligned}
\label{n4m4k0to7}
&W_{12}&=&\;\frac{1}{a_1 a_2 a_3 a_4} \prod_{i=1}^4 \prod _{j=1}^4 \left(a_i-b_j\right)\;,\\
&W_{13}&=&\;(a_1+a_2+a_3+a_4)\cdot W_{12}\;,\\
&W_{14}&=&\;(a_1 a_2+a_1 a_3+a_1 a_4+a_2 a_3+a_2 a_4+a_3 a_4)\cdot W_{12}\;,\\
&W_{15}&=&\;(a_1 a_2 a_3+a_1 a_2 a_4+a_1 a_3 a_4+a_2 a_3 a_4)\cdot W_{12}\;,\\
&W_{16}&=&\;(a_1 a_2 a_3 a_4)\cdot W_{12} \;.
\end{aligned}
\end{equation}
\end{example}
By applying \eqref{n4m4k0to7}, we found the below ideal non-negative solution of GPTE in 2023.
\begin{align}
&[ 387, 388, 416, 447, 494, 536, 573, 589, 610 ]^k \nonumber \\
& \quad = [ 382, 402, 403, 456, 485, 549, 559, 596, 608 ]^k\;,\nonumber\\
&\qquad\qquad\qquad\qquad\qquad\quad (k=0,1,2,3,4,5,6,7) \tag*{\eqref{k01234567s610}}
\end{align}

\begin{example}
Let
\begin{equation}
\begin{aligned}
&P_0=\;-\frac{b_1 b_2 b_3}{a_1 a_2 a_3}\:,\\
&P_k=\;a_1^k+a_2^k+a_3^k-b_1^k-b_2^k-b_3^k\;,\quad (k=-1,1,2,3,4)\:.
\end{aligned}
\end{equation}
Define
\begin{equation}
\begin{aligned}
& S_{-1}=\; P_0 P_{-1}\:,\\
& S_0 =\;1+P_0\:,\\
& S_1 =\;(P_1)/1\:,\\
& S_2 =\;(P_2+S_1 P_1)/2\:,\\
& S_3 =\;(P_3+S_1 P_2+S_2 P_1)/3\:,\\
& S_4 =\;(P_4+S_1 P_3+S_2 P_2+S_3 P_1)/4\:.
\end{aligned}
\end{equation}
Furthermore, define
\begin{equation}
\begin{aligned}
&W_{3}=
\begin{vmatrix}
 S_1 & S_0 & S_{-1}\\
 S_2 & S_1 & S_0 \\
 S_3 & S_2 & S_1 \\
\end{vmatrix},\quad
&W_{4}=
\begin{vmatrix}
 S_1 & S_0 & S_{-1}\\
 S_2 & S_1 & S_0 \\
 S_4 & S_3 & S_2 \\
\end{vmatrix},\\
&W_{5}=
\begin{vmatrix}
 S_1 & S_0 & S_{-1}\\
 S_3 & S_2 & S_1 \\
 S_4 & S_3 & S_2 \\
\end{vmatrix},\quad
&W_{6}=
\begin{vmatrix}
 S_2 & S_1 & S_0 \\
 S_3 & S_2 & S_1 \\
 S_4 & S_3 & S_2 \\
\end{vmatrix}.\\
\end{aligned}
\end{equation}
Then it follows that
\begin{equation}
\begin{aligned}
\label{n3m3kn1to4}
&W_{3}&=&\;\frac{1}{a_1^2 a_2^2 a_3^2 } \prod_{i=1}^3 \prod _{j=1}^3 \left(a_i-b_j\right)\;,\\
&W_{4}&=&\;(a_1+a_2+a_3)\cdot W_{3}\;,\\
&W_{5}&=&\;(a_1 a_2+a_1 a_3+a_2 a_3)\cdot W_{3}\;,\\
&W_{6}&=&\;(a_1 a_2 a_3)\cdot W_{3}\;.\\
\end{aligned}
\end{equation}
\end{example}
By applying \eqref{n3m3kn1to4}, we found the following ideal non-negative solutions of GPTE in 2017.
\begin{align}
& [9, 17, 21, 51, 99, 143, 143 ]^k = [  11, 11, 33, 39, 117, 119, 153]^k,\nonumber\\
& \qquad\qquad\qquad\qquad\qquad\qquad\qquad\qquad (k=-1,0,1,2,3,4) \tag*{\eqref{k01234n1s153}}\\
& [56, 77, 99, 152, 174, 228, 261]^k = [57, 72, 116, 126, 203, 209, 264]^k,\nonumber\\
& \qquad\qquad\qquad\qquad\qquad\qquad\qquad\qquad\qquad(k=-1,0,1,2,3,4)  \tag*{\eqref{k01234n1s264}}
\end{align}

\begin{example}
Let
\begin{equation}
\begin{aligned}
P_k=a_1^k+a_2^k+a_3^k-b_1^k-b_2^k-b_3^k-b_4^k,\\(k=1,2,3,4,5,6,8)
\end{aligned}
\end{equation}
Define
\begin{equation}
\begin{aligned}
& S_1=(P_1)/1\;,\\
& S_2=(P_2+S_1 P_1)/2\;,\\
& S_3=(P_3+S_1 P_2+S_2 P_1)/3\;,\\
& S_4=(P_4+S_1 P_3+S_2 P_2+S_3 P_1)/4\;,\\
& S_5=(P_5+S_1 P_4+S_2 P_3+S_3 P_2+S_4 P_1)/5\;,\\
& S_6=(P_6+S_1 P_5+S_2 P_4+S_3 P_3+S_4 P_2+S_5 P_1)/6\;,\\
& S_7=(P_7+S_1 P_6+S_2 P_5+S_3 P_4+S_4 P_3+S_5 P_2+S_6 P_1)/7\;,\\
& S_8=(P_8+S_1 P_7+S_2 P_6+S_3 P_5+S_4 P_4+S_5 P_3+S_6 P_2+S_7 P_1)/8\;.\\
\end{aligned}
\end{equation}
Furthermore, define
\begin{equation}
\begin{aligned}
&W_{12}=
\begin{vmatrix}
 S_4 & S_3 & S_2 \\
 S_5 & S_4 & S_3 \\
 S_6 & S_5 & S_4 \\
\end{vmatrix},\quad
&W_{14}=
\begin{vmatrix}
 S_5-S_4 P_1 & S_3 & S_2 \\
 S_6-S_5 P_1 & S_4 & S_3 \\
 S_8-S_7 P_1 & S_6 & S_5 \\
\end{vmatrix}.\\
\end{aligned}
\end{equation}
Then it follows that
\begin{equation}
\begin{aligned}
\label{n3m4k1234568}
&W_{12}=\prod _{i=1}^3 \prod _{j=1}^4 (a_i-b_j)\;,\\
&W_{14}=(a_1+a_2+a_3)(b_1+b_2+b_3+b_4)\cdot W_{12} \;,\\
& P_1^2+\frac{4 W_{14}}{W_{12}}=(a_1+a_2+a_3+b_1+b_2+b_3+b_4)^2 \:.
\end{aligned}
\end{equation}
\end{example}
It should be noted that $P_7$ is unknown, but it will be eliminated in the calculation of $W_{14}$. Therefore, $P_7$ can be set to 0 without affecting the final result. By applying \eqref{n3m4k1234568}, we found the following ideal non-negative solution for GPTE.
\begin{align}
&[ 77, 159, 169, 283, 321, 443, 447, 501 ]^k \nonumber \\
& \quad = [ 79, 137, 213, 237, 363, 399, 481, 491 ]^k\;,\, (k=1,2,3,4,5,6,8) \tag*{\eqref{k1234568s501}}
\end{align}

\begin{example}
Let
\begin{equation}
\begin{aligned}
&P_0=-\frac{b_1 b_2 b_3}{a_1 a_2}\;,\\
&P_k=a_1^k+a_2^k-b_1^k-b_2^k-b_3^k\;,\quad (k=-1,1,2,4)\;.
\end{aligned}
\end{equation}
Define
\begin{equation}
\begin{aligned}
& T_{-1}= P_0 P_{-1}\;,\\
& T_0 = P_0\;,\\
& T_1 =(P_1)/1\;,\\
& T_2 =(P_2+T_1 P_1)/2\;,\\
& T_3 =(P_3+T_1 P_2+T_2 P_1)/3\;,\\
& T_4 =(P_4+T_1 P_3+T_2 P_2+T_3 P_1)/4\;.\\
\end{aligned}
\end{equation}
And let
\begin{equation}
\begin{aligned}
& S_0 = 1-T_{-1}\;,\\
& S_1 = T_1-T_{0}\;,\\
& S_2 = T_2\;,\\
& S_3 = T_3\;,\\
& S_4 = T_4\;.\\
\end{aligned}
\end{equation}
Furthermore, denote
\begin{equation}
\begin{aligned}
&W_{2}=
\begin{vmatrix}
 S_1 & S_0 \\
 S_2 & S_1 \\
\end{vmatrix},\quad
&W_{4}=
\begin{vmatrix}
 S_2 - T_1 S_1 & S_0 \\
 S_4 - T_1 S_3 & S_2 \\
\end{vmatrix}.\\
\end{aligned}
\end{equation}
Then it follows that
\begin{equation}
\begin{aligned}
\label{n2m3k0124n1}
&W_{2}=\frac{1}{a_1^2 a_2^2 } \prod_{i=1}^2 \prod _{j=1}^3 \left(a_i-b_j\right)\;,\\
&W_{4}=(a_1+a_2)(b_1+b_2+b_3) \cdot W_{2}\;,\\
&T_1^2+ \frac{4 W_4}{W_2}= (a_1+a_2+b_1+b_2+b_3)^2 \:.\\
\end{aligned}
\end{equation}
\end{example}
It should be noted that $P_3$ is unknown, but it will be eliminated during the calculation of $W_{4}$. Therefore, $P_3$ can be set to 0 without affecting the result. By applying \eqref{n2m3k0124n1}, we found the following ideal non-negative solutions of GPTE.
\begin{align}
& [10, 14, 24, 65, 117, 132  ]^k = [ 11, 12, 26, 63, 120, 130 ]^k\;, \nonumber\\
& \qquad\qquad\qquad\qquad\qquad\qquad\qquad\qquad (k=-1,0,1,2,4) \tag*{\eqref{k0124n1s132}}\\
& [36, 51, 72, 130, 195, 238 ]^k = [40, 42, 85, 117, 204, 234 ]^k\;, \nonumber\\
& \qquad\qquad\qquad\qquad\qquad\qquad\qquad\qquad (k=-1,0,1,2,4) \tag*{\eqref{k0124n1s238}}
\end{align}

\begin{example}
For positive integer $k$, we denote
\begin{equation}
\begin{aligned}
&P_k=a_1^k+a_2^k+a_3^k-b_1^k-b_2^k-b_3^k
\end{aligned}
\end{equation}
and denote
\begin{align}
\label{SkPositive}
& S_{k}=(P_{k}+S_1 P_{k-1}+S_2 P_{k-2}+\cdots+S_{k-1} P_1)/k
\end{align}
For $k \ge 3$, we define:
\begin{equation}
\label{L3R3W3k}
\begin{aligned}
&W_{3k}=
\begin{vmatrix}
 S_k & S_{k-1} & S_{k-2} \\
 S_{k+1} & S_k & S_{k-1} \\
 S_{k+2} & S_{k+1} & S_k \\
\end{vmatrix},\quad
&W_{3k+1}=
\begin{vmatrix}
 S_k & S_{k-1} & S_{k-2} \\
 S_{k+1} & S_k & S_{k-1} \\
 S_{k+3} & S_{k+2} & S_{k+1} \\
\end{vmatrix},\\
&W_{3k+2}=
\begin{vmatrix}
 S_k & S_{k-1} & S_{k-2} \\
 S_{k+2} & S_{k+1} & S_{k} \\
 S_{k+3} & S_{k+2} & S_{k+1} \\
\end{vmatrix}. \\
\end{aligned}
\end{equation}
Then it follows that:
\begin{equation}
\label{L3R3W9}
\begin{aligned}
&W_{9}=\begin{vmatrix}
 S_{3} & S_{2} & S_{1} \\
 S_{4} & S_{3} & S_{2} \\
 S_{5} & S_{4} & S_{3} \\
\end{vmatrix}
=\prod _{i=1}^3 \prod _{j=1}^3 \left(a_i-b_j\right)\;,\\
&W_{3k+1}=(a_1+a_2+a_3)\cdot W_{3k}\;,\\
&W_{3k+2}=(a_1 a_2+a_1 a_3+a_2 a_3)\cdot W_{3k}\;,\\
&W_{3k+3}=(a_1 a_2 a_3)\cdot W_{3k}\;.
\end{aligned}
\end{equation}
\end{example}
\begin{corollary}\label{corollary6108}
For positive integer $k$, we denote
\begin{equation}
\begin{aligned}
&P_k=a_1^k+a_2^k+a_3^k-b_1^k-b_2^k-b_3^k
\end{aligned}
\end{equation}
If $P_1=P_2=P_4=0$ and $P_3 \neq 0$, then it follows that
\begin{equation}
\label{L3R3-6-10-8}
\begin{aligned}
64 P_6 P_{10}=45 P_8^2\\
\end{aligned}
\end{equation}
\end{corollary}
The proof process of the above corollary is briefly as follows:\\[1mm]
Step 1: Since $P_1=P_2=P_4=0$, according to \eqref{h124a1a2a3}, we can derive that
\begin{align}
& a_1+a_2+a_3=0
\end{align}
Step 2: From \eqref{SkPositive}, we obtain 
\begin{align}
& S_1=S_2=0, S_3=\frac{P_3}{3}, \cdots, S_{10}=\frac{P_5^2}{50}+\frac{P_3 P_7}{21}+\frac{P_{10}}{10}.
\end{align}
Step 3: By \eqref{L3R3W3k} and \eqref{L3R3W9}, from $W_{13}=(a_1+a_2+a_3)\cdot W_{12}=0$, we get
\begin{align}
& P_7=\frac{21 P_5^2}{25 P_3}
\end{align}
Step 4: From $W_{16}=(a_1+a_2+a_3)\cdot W_{15}=0$, we get
\begin{align}
& P_5=\frac{5 P_3 P_8}{8 P_6}
\end{align}
Step 5: From $W_{19}=(a_1+a_2+a_3)\cdot W_{18}=0$, we get
\begin{align}
& P_9=\frac{3456 P_6^5+128 P_3^4 P_6^3+729 P_3^2 P_8^3}{4608 P_3 P_6^3}
\end{align}
Step 6: From $ W_{23}/W_{21}=W_{11}/W_{9}$, we get \eqref{L3R3-6-10-8}.\\

In fact, \eqref{L3R3-6-10-8} is the Ramanujan 6-10-8 formula, and its previous proofs can be found in \cite{Chamberland09}. In the steps from 1 to 6 above, we have presented a novel proof utilizing \refIdentity{identity_GNI2}. By employing the same method and with the assistance of \refIdentity{identity_GNI2}, we can derive numerous similar results. For example,

\begin{corollary}
For non-zero integer $k$, we denote
\begin{equation}
\begin{aligned}
&Q_k=\frac{a_1^k+a_2^k+a_3^k+a_4^k-b_1^k-b_2^k-b_3^k-b_4^k}{k}
\end{aligned}
\end{equation}
If $Q_{-2}=Q_{-1}=Q_1=Q_2=0$ and $Q_4 \neq 0$, then it follows that
\begin{equation}
\begin{aligned}
\frac{Q_{-3} Q_{-5}}{Q_{-4}^2}=\frac{Q_3 Q_5}{Q_4^2}\\
\end{aligned}
\end{equation}
\end{corollary}
The above fomula was discovered by us in 2019 \cite{Chen19}. It can be derived from \refIdentity{identity_GNI2}, and it can also be verified using the following example:
\begin{align}
& [-15,-9,45,45]^h = [-11,-11,33,55]^h,\quad (h=-2,-1,1,2)  \tag*{\eqref{h12n12s55}}
\end{align}
with
\begin{align}
& \{Q_{-5},Q_{-4},Q_{-3}\}=\{\frac{351232}{300186253125},-\frac{175616}{20012416875},\frac{7168}{121287375}\}\\
& \{Q_3,Q_4,Q_5\}=\{-7168,-526848,-34771968\}
\end{align}

\begin{corollary}
Denote
\begin{equation}
\begin{aligned}
Q_{k}=
\begin{cases}  
\displaystyle\frac{a_1^k+a_2^k+a_3^k+a_4^k-b_1^k-b_2^k-b_3^k-b_4^k}{k}, \ &k\ne0\\[3mm]
\displaystyle\frac{2(a_1 a_2 a_3 a_4-b_1 b_2 b_3 b_4)}{a_1 a_2 a_3 a_4+b_1 b_2 b_3 b_4}, &k=0
\end{cases}\\
\end{aligned}
\end{equation}
We have\\
\indent If \  $Q_{1\phantom{-}}=Q_{2\phantom{-}}=Q_{3\phantom{-}}=Q_5=0$, then $Q_4 Q_9=2 Q_6 Q_7$.\\
\indent If \ $Q_{0\phantom{-}}=Q_{1\phantom{-}}=Q_{2\phantom{-}}=Q_4=0$, then $Q_3 Q_8=2 Q_5 Q_6$.\\
\indent If \ $Q_{-1}=Q_{0\phantom{-}}=Q_{1\phantom{-}}=Q_3=0$, then $Q_2 Q_7=2 Q_4 Q_5$.\\
\indent If \ $Q_{-2}=Q_{-1}=Q_{0\phantom{-}}=Q_2=0$, then $Q_1 Q_6=2 Q_3 Q_4+ Q_1^3 Q_4/12$.\\
\indent If \ $Q_{-3}=Q_{-2}=Q_{-1}=Q_1=0$, then $Q_0 Q_5=2 Q_2 Q_3$.\\
and\\
\indent If \  $Q_{1\phantom{-}}=Q_{2\phantom{-}}=Q_{3\phantom{-}}=Q_5=0$, then $Q_7 Q_{0\phantom{-}}^2=-Q_4^2 Q_{-1}(1-Q_0^2/4)$.\\
\indent If \ $Q_{0\phantom{-}}=Q_{1\phantom{-}}=Q_{2\phantom{-}}=Q_4=0$, then $Q_6 Q_{-1}^2=-Q_3^2 Q_{-2}$.\\
\indent If \ $Q_{-1}=Q_{0\phantom{-}}=Q_{1\phantom{-}}=Q_3=0$, then $Q_5 Q_{-2}^2=-Q_2^2 Q_{-3}$.\\
\indent If \ $Q_{-2}=Q_{-1}=Q_{0\phantom{-}}=Q_2=0$, then $Q_4 Q_{-3}^2=-Q_1^2 Q_{-4}$.\\
\indent If \ $Q_{-3}=Q_{-2}=Q_{-1}=Q_1=0$, then $Q_3 Q_{-4}^2=-Q_0^2 Q_{-5}(1-Q_0^2/4)$.
\end{corollary}
The author discovered the above formulas in 2019 \cite{Chen19}. These formulas can be derived from \refIdentity{identity_GNI2} and can also be verified by using the ideal integer solutions for the following types:  $(h=1,2,3,5), (h=0,1,2,4), (h=-1,0,1,3), (h=-2,-1,0,2)$, and $(h=-3,-2,-1,1)$. (See \ref{h1235}, \ref{h0124}, \ref{h013n1}, \ref{h02n12}, \ref{h1n123})

\subsubsection{Third Generalization of the Girard-Newton Identities}
For the Third Generalization of the Girard-Newton Identities, we specifically consider cases that are derived from the First Generalization, where all exponential powers are odd integers \cite{Chen21}. Here, exponential power can refer to both positive and negative odd integers. However, when both positive and negative odd integers are involved, the general identity might become quite complex. Therefore, in the Identity 5 presented below, we only address cases where all exponential powers are positive odd integers. We will provide examples later for cases where exponential powers include both positive and negative odd integers.

\begin{identity}\label{identity_GNI3}
Let $n$ be a positive integer, and let $t$ be a non-negative integer. We define:
\begin{align}
\label{defineP2tp1} 
P_{2t+1}= a_{1}^{2t+1}+a_{2}^{2t+1}+\cdots+a_{n}^{2t+1}
\end{align}
and introduce $G_{2t+1}$ as follows:
\begin{align}
\label{defineG} 
G_{2t+1}=\frac{P_1^{2t+1}-P_{2t+1}}{2t+1} -\frac{\sum_{j=1}^{t-3} (2j+1) G_{2j+1} Z_{2t-2j}}{2t+1}
\end{align}
where\\
\begin{align}
Z_{2t}=\sum_{\substack{i_3,i_5,...,i_{2t-3}\geq0\\3 i_3 + 5 i_5+\cdots +(2t-3) i_{2t-3}=2t }}\frac{(i_3+i_5+\cdots+i_{2t-3})!}{i_3 !\; i_5 !\; \cdots\; i_{2t-3} !} G_3^{i_3} G_5^{i_5}\cdots G_{2t-3}^{i_{2t-3}}
\end{align}
When $n$ is even, we define the following determinants:
\begin{equation}\nonumber
\begin{aligned}
U_{n,0}=\begin{vmatrix}
 G_{n-1} & G_{n-3} & \cdots & G_5 & G_3 & V_1\\
 G_{n+1} & G_{n-1} & \cdots & G_7 & G_5 & V_3\\
 G_{n+3} & G_{n+1} & \cdots & G_9 & G_7 & V_5\\
 \vdots & \vdots & \vdots & \vdots & \vdots & \vdots\\
 G_{2n-5} & G_{2n-7}  & \cdots & G_{n+1}& G_{n-1}& V_{n-3}\\
 G_{2n-3} & G_{2n-5}  & \cdots & G_{n+3}& G_{n+1}& V_{n-1}\\
\end{vmatrix}\;,
\end{aligned}
\end{equation}
\begin{equation}
\begin{aligned}
U_{n,2}=\begin{vmatrix}
 G_{n+1} & G_{n-3} & \cdots & G_5 & G_3 & V_1\\
 G_{n+3} & G_{n-1} & \cdots & G_7 & G_5 & V_3\\
 G_{n+5} & G_{n+1} & \cdots & G_9 & G_7 & V_5\\
 \vdots & \vdots & \vdots & \vdots & \vdots & \vdots\\
 G_{2n-3} & G_{2n-7}  & \cdots & G_{n+1}& G_{n-1}& V_{n-3}\\
 G_{2n-1} & G_{2n-5}  & \cdots & G_{n+3}& G_{n+1}& V_{n-1}\\
\end{vmatrix}\;,
\end{aligned}
\end{equation}
\begin{equation}\nonumber
\begin{aligned}
U_{n,4}=\begin{vmatrix}
 G_{n+1} & G_{n-3} & \cdots & G_5 & G_3 & V_1\\
 G_{n+3} & G_{n-1} & \cdots & G_7 & G_5 & V_3\\
 G_{n+5} & G_{n+1} & \cdots & G_9 & G_7 & V_5\\
 \vdots & \vdots & \vdots & \vdots & \vdots & \vdots\\
 G_{2n-3} & G_{2n-7}  & \cdots & G_{n+1}& G_{n-1}& V_{n-3}\\
 G_{2n+1} & G_{2n-3}  & \cdots & G_{n+5}& G_{n+3}& V_{n+1}\\
\end{vmatrix}\;,
\end{aligned}
\end{equation}
\begin{equation}\nonumber
\begin{aligned}
\qquad\quad where \quad 
\begin{cases}
 V_1=G_1\\
 V_3=G_1^3+G_3\\
 V_5=G_1^5+G_3 G_1^2+G_5\\
 \qquad \cdots \\
 V_{n+1}=G_1^{n+1}+ G_3 G_1^{n-2} + \cdots + G_{n-1} G_1^2 + G_{n+1}\\
\end{cases}
\end{aligned}
\end{equation}
When $n$ is odd, we define the following determinants:
\begin{equation}\nonumber
\begin{aligned}
U_{n,0}=\begin{vmatrix}
 G_{n-2} & G_{n-4} & \cdots & G_3 & 0 & 1\\
 G_{n} & G_{n-2} & \cdots & G_5 & G_3 & V_2\\
 G_{n+2} & G_{n} & \cdots & G_7 & G_5 & V_4\\
 \vdots & \vdots & \vdots & \vdots & \vdots & \vdots\\
 G_{2n-5} & G_{2n-7}  & \cdots & G_{n}& G_{n-2}& V_{n-3}\\
 G_{2n-3} & G_{2n-5}  & \cdots & G_{n+2}& G_{n}& V_{n-1}\\
\end{vmatrix}\;,
\end{aligned}
\end{equation}
\begin{equation}
\begin{aligned}
U_{n,2}=\begin{vmatrix}
 G_{n} & G_{n-4} & \cdots & G_3 & 0 & 1\\
 G_{n+2} & G_{n-2} & \cdots & G_5 & G_3 & V_2\\
 G_{n+4} & G_{n} & \cdots & G_7 & G_5 & V_4\\
 \vdots & \vdots & \vdots & \vdots & \vdots & \vdots\\
 G_{2n-3} & G_{2n-7}  & \cdots & G_{n}& G_{n-2}& V_{n-3}\\
 G_{2n-1} & G_{2n-5}  & \cdots & G_{n+2}& G_{n}& V_{n-1}\\
\end{vmatrix}\;,
\end{aligned}
\end{equation}
\begin{equation}\nonumber
\begin{aligned}
U_{n,4}=\begin{vmatrix}
 G_{n} & G_{n-4} & \cdots & G_3 & 0 & 1\\
 G_{n+2} & G_{n-2} & \cdots & G_5 & G_3 & V_2\\
 G_{n+4} & G_{n} & \cdots & G_7 & G_5 & V_4\\
 \vdots & \vdots & \vdots & \vdots & \vdots & \vdots\\
 G_{2n-3} & G_{2n-7}  & \cdots & G_{n}& G_{n-2}& V_{n-3}\\
 G_{2n+1} & G_{2n-3}  & \cdots & G_{n+4}& G_{n+2}& V_{n+1}\\
\end{vmatrix}\;,
\end{aligned}
\end{equation}
\begin{equation}\nonumber
\begin{aligned}
\qquad\quad where \quad 
\begin{cases}
 V_2=G_1^2\\
 V_4=G_1^4+G_3 G_1\\
 V_6=G_1^6+G_3 G_1^3+G_5 G_1 \\
 \qquad \cdots \\
 V_{n+1}=G_1^{n+1}+ G_3 G_1^{n-2} + \cdots + G_{n-2} G_1^3 + G_n G_1\\
\end{cases}
\end{aligned}
\end{equation}
With these definitions, the identity states that:
\begin{equation}
\begin{aligned}
& G_1^2+ \frac {2 U_{n,2}}{U_{n,0}}=a_1^2+a_2^2+\cdots+a_n^2\\ 
& G_1^4+ \frac {4 U_{n,4}}{U_{n,0}}=(a_1^2+a_2^2+\cdots+a_n^2)^2\\
\end{aligned}
\end{equation}
\end{identity}
Since all $k$ are odd integers, for the definition of $P_{2t+1}$, we only need to consider the case given by equation \eqref{defineP2tp1}. We do not need to consider the case where $P_{2t+1}=\sum_{i=1}^{n}a^{2t+1}-\sum_{j=1}^{m}b^{2t+1}$ because it can be regarded as $P_{2t+1}=\sum_{i=1}^{n} a^{2t+1}+\sum_{j=1}^{m}(-b)^{2t+1}$. Identity 5 has not been proved completely, it can be verified by using Mathematica for a certain smaller $n$.  \\
\indent
The following Example 2.20 to Example 2.32 are specific examples originating from Identity 5, and they can be easily proved through algebraic methods.

\begin{example}
According to the definition \eqref{defineG}, for all integers t where $0 \leq t \leq 8$, we can obtain the specific expression of $G_{2t + 1}$ as follows:
\begin{align}
\label{G1}
& G_1=-P_1\\
\label{G3}
& G_3=\frac{1}{3} \left(P_1^3-P_3\right)\\
\label{G5}
& G_5=\frac{1}{5} \left(P_1^5-P_5\right)\\
\label{G7}
& G_7=\frac{1}{7} \left(P_1^7-P_7\right)\\
\label{G9}
& G_9=\frac{1}{9} \left(P_1^9-P_9\right)-\frac{G_3^3}{3}\\
\label{G11}
& G_{11}=\frac{1}{11} \left(P_1^{11}-P_{11}\right)-G_3^2 G_5\\
\label{G13}
& G_{13}=\frac{1}{13} \left(P_1^{13}-P_{13}\right)-\left(G_7 G_3^2+G_5^2 G_3\right)\\
\label{G15}
& G_{15}=\frac{1}{15} (P_1^{15}-P_{15}) -\left(\frac{G_3^5}{5}+G_9 G_3^2+2 G_5 G_7 G_3+\frac{G_5^3}{3}\right)\\
\label{G17}
& G_{17}=\frac{1}{17} \left(P_1^{17}-P_{17}\right) \nonumber\\
& \qquad\qquad -(G_5 G_3^4+G_{11} G_3^2+G_7^2 G_3+2 G_5 G_9 G_3+G_5^2 G_7)
\end{align}
\end{example}

\begin{example}
Let
\begin{equation}
\begin{aligned}
P_k=a_1^k+a_2^k,\quad (k=1,3)
\end{aligned}
\end{equation}
Define
\begin{align}
& G_1=-P_1 \tag*{\eqref{G1}}\\ 
& G_3=(P_1^3-P_3)/3 \tag*{\eqref{G3}}
\end{align}
Then it follows that
\begin{equation}
\begin{aligned}
\label{identity13} 
& G_1^2+ 2 \frac {\begin{vmatrix}
 G_3 \\
\end{vmatrix}}{\begin{vmatrix}
 G_1 \\
\end{vmatrix}}=a_1^2+a_2^2\\
\end{aligned}
\end{equation}
\end{example}
The basic application of Identity \eqref{identity13} is to find ideal non-negative integer solutions for $(k=1,3)$, including ideal prime solutions. For example, the following ideal prime solution chain of length 4 is obtained by using Identity \eqref{identity13}:
\begin{align}
&[83,757,827]^k = [107,677,883]^k= [197,523,947]^k = [281,419,967]^k, \nonumber \\
& \qquad\qquad\qquad\qquad\qquad\qquad\qquad\qquad\qquad\qquad\qquad (k=1,3)  \tag*{\eqref{k13s967}} 
\end{align}
\begin{example}
Let
\begin{equation}
\begin{aligned}
P_k=a_1^k+a_2^k,\quad (k=1,5)
\end{aligned}
\end{equation}
Define
\begin{align}
& G_1=-P_1 \tag*{\eqref{G1}}\\ 
& G_5=(P_1^5-P_5)/5 \tag*{\eqref{G5}} 
\end{align}
Then it follows that
\begin{equation}
\begin{aligned}
\label{identity15} 
& G_1^4+ 4 \frac {\begin{vmatrix}
 G_5 \\
\end{vmatrix}}{\begin{vmatrix}
 G_1 \\
\end{vmatrix}}=(a_1^2+a_2^2)^2\\
\end{aligned}
\end{equation}
\end{example}

\begin{example}
Let
\begin{equation}
\begin{aligned}
P_k=a_1^k+a_2^k+a_3^k,\quad (k=1,3,5)
\end{aligned}
\end{equation}
Define
\begin{align}
& G_1=-P_1 \tag*{\eqref{G1}}\\ 
& G_3=(P_1^3-P_3)/3 \tag*{\eqref{G3}}\\ 
& G_5=(P_1^5-P_5)/5 \tag*{\eqref{G5}} 
\end{align}
Then it follows that
\begin{equation}
\begin{aligned}
\label{identity135} 
& G_1^2+ 2 \frac {\begin{vmatrix}
 G_3 & 1 \\
 G_5 & G_1^2 \\
\end{vmatrix}}{\begin{vmatrix}
 0 & 1 \\
 G_3 & G_1^2 \\
\end{vmatrix}}=a_1^2+a_2^2+a_3^2\\
\end{aligned}
\end{equation}
\end{example}

\begin{example}
Let
\begin{equation}
\begin{aligned}
P_k=a_1^k+a_2^k+a_3^k,\quad (k=1,3,7)
\end{aligned}
\end{equation}
Define
\begin{align}
& G_1=-P_1 \tag*{\eqref{G1}}\\ 
& G_3=(P_1^3-P_3)/3 \tag*{\eqref{G3}}\\ 
& G_7=(P_1^7-P_7)/7 \tag*{\eqref{G7}}
\end{align}
Then it follows that
\begin{equation}
\begin{aligned}
\label{identity137} 
& G_1^4+ 4 \frac {\begin{vmatrix}
 G_3 & 1 \\
 G_7 & G_1^4+G_3 G_1 \\
\end{vmatrix}}{\begin{vmatrix}
 0 & 1 \\
 G_3 & G_1^2 \\
\end{vmatrix}}=(a_1^2+a_2^2+a_3^2)^2\\
\end{aligned}
\end{equation}
\end{example}
By applying \eqref{identity137}, Chen Shuwen found the following solution in 2000.
\begin{align}
& [ 184, 443, 556, 698 ]^k = [ 230, 353, 625, 673 ]^k, \quad (k=1,3,7)  \tag*{\eqref{k137s698}}
\end{align}
\begin{example}
Let
\begin{equation}
\begin{aligned}
P_k=a_1^k+a_2^k+a_3^k+a_4^k,\quad (k=1,3,5,7)
\end{aligned}
\end{equation}
and
\begin{align}
& G_1=-P_1 \tag*{\eqref{G1}}\\ 
& G_3=(P_1^3-P_3)/3 \tag*{\eqref{G3}}\\ 
& G_5=(P_1^5-P_5)/5 \tag*{\eqref{G5}} \\
& G_7=(P_1^7-P_7)/7 \tag*{\eqref{G7}}
\end{align}
Then the following identity holds: 
\begin{align}
\label{identity1357} 
& G_1^2+ 2 \frac {\begin{vmatrix}
 G_5 & G_1 \\
 G_7 & G_1^3+G_3\\
\end{vmatrix}}{\begin{vmatrix}
 G_3 & G_1 \\
 G_5 & G_1^3+G_3\\
\end{vmatrix}}=a_1^2+a_2^2+a_3^2+a_4^2
\end{align}
Furthermore, define:
\begin{align}
 U_{4,0}={\begin{vmatrix}
 G_3 & G_1 \\
 G_5 & G_1^3+G_3\\
\end{vmatrix}}
\end{align}
Then it follows that
\begin{align}
& U_{4,0}=-(a_1+a_2)(a_1+a_3)(a_1+a_4)(a_2+a_3)(a_2+a_4)(a_3+a_4)\\
& {\begin{vmatrix}
 G_5 & G_1 \\
 G_7 & G_1^3+G_3\\
\end{vmatrix}}\cdot \frac{1}{U_{4,0}}=-(a_1 a_2 +a_1 a_3+\cdots+a_3 a_4)\\
& {\begin{vmatrix}
 G_5 & G_1^2 \\
 G_7 & G_1^4+G_3 G_1 \\
\end{vmatrix}}\cdot \frac{1}{U_{4,0}}-G_3=a_1 a_2 a_3+a_1 a_2 a_4+\cdots+a_2 a_3 a_4\\
& {\begin{vmatrix}
 G_3 & 0 & 1 \\
 G_5 & G_3 & G_1^2 \\
 G_7 & G_5 & G_1^4+G_3 G_1 \\
\end{vmatrix}}\cdot \frac{1}{U_{4,0}}=a_1 a_2 a_3 a_4\\
& {\begin{vmatrix}
 G_5 & G_3 \\
 G_7 & G_5 \\
\end{vmatrix}}\cdot \frac{1}{U_{4,0}}=(P_1-a_1)(P_1-a_2)(P_1-a_3)(P_1- a_4)
\end{align}
\end{example}

\begin{example}
Let
\begin{align}
P_k=a_1^k+a_2^k+a_3^k+a_4^k,\quad (k=1,3,5,9)
\end{align}
and
\begin{align}
& G_1=-P_1 \tag*{\eqref{G1}}\\ 
& G_3=(P_1^3-P_3)/3 \tag*{\eqref{G3}}\\ 
& G_5=(P_1^5-P_5)/5 \tag*{\eqref{G5}} \\
& G_9=(P_1^9-P_9)/9-(G_3^3/3) \tag*{\eqref{G9}}
\end{align}
then
\begin{equation}
\begin{aligned}
\label{identity1359} 
& G_1^4+ 4 \frac {\begin{vmatrix}
 G_5 & G_1 \\
 G_9 & G_1^5+G_3 G_1^2+G_5 \\
\end{vmatrix}}{\begin{vmatrix}
 G_3 & G_1 \\
 G_5 & G_1^3+G_3 \\
\end{vmatrix}}=(a_1^2+a_2^2+a_3^2+a_4^2)^2\\
\end{aligned}
\end{equation}
\end{example}

\begin{align}
a_1^k+a_2^k+a_3^k+a_4^k+a_5^k=b_1^k+b_2^k+b_3^k+b_4^k+b_5^k,\,\, (k=1,3,5,9)
\end{align}

\begin{example}
Let
\begin{equation}
\begin{aligned}
P_k=a_1^k+a_2^k+a_3^k+a_4^k+a_5^k,\quad (k=1,3,5,7,9)
\end{aligned}
\end{equation}
and
\begin{align}
& G_1=-P_1 \tag*{\eqref{G1}}\\ 
& G_3=(P_1^3-P_3)/3 \tag*{\eqref{G3}}\\ 
& G_5=(P_1^5-P_5)/5 \tag*{\eqref{G5}} \\
& G_7=(P_1^7-P_7)/7 \tag*{\eqref{G7}}\\
& G_9=(P_1^9-P_9)/9-(G_3^3/3) \tag*{\eqref{G9}}
\end{align}
Then the following identity holds:
\begin{align}
\label{identity13579} 
& G_1^2+ 2 \frac {\begin{vmatrix}
 G_5 & 0 & 1 \\
 G_7 & G_3 & G_1^2 \\
 G_9 & G_5 & G_1^4+G_3 G_1 \\
\end{vmatrix}}{\begin{vmatrix}
 G_3 & 0 & 1 \\
 G_5 & G_3 & G_1^2 \\
 G_7 & G_5 & G_1^4+G_3 G_1 \\
\end{vmatrix}}=a_1^2+a_2^2+a_3^2+a_4^2+a_5^2
\end{align}
Furthermore, define:
\begin{align}
 U_{5,0}={\begin{vmatrix}
 G_3 & 0 & 1 \\
 G_5 & G_3 & G_1^2 \\
 G_7 & G_5 & G_1^4+G_3 G_1 \\
\end{vmatrix}}
\end{align}
Then it follows that
\begin{align}
& U_{5,0}=(a_1+a_2)(a_1+a_3)\cdots (a_4+a_5)
\end{align}
and
\begin{align}
& {\begin{vmatrix}
 G_5 & 0 & 1 \\
 G_7 & G_3 & G_1^2 \\
 G_9 & G_5 & G_1^4+G_3 G_1 \\
\end{vmatrix}}\cdot \frac{1}{U_{5,0}}=-(a_1 a_2+a_1 a_3+\cdots+a_4 a_5)\\[1mm]
& {\begin{vmatrix}
 G_5 & 0 & G_1 \\
 G_7 & G_3 & G_1^3+G_3 \\
 G_9 & G_5 & G_1^5+G_3 G_1^2+G_5 \\
\end{vmatrix}}\cdot \frac{1}{U_{5,0}}-G_3\nonumber\\
& \qquad =a_1 a_2 a_3+a_1 a_2 a_4 +\cdots+a_3 a_4 a_5\\[1mm]
& {\begin{vmatrix}
 G_5 & G_3 & 1 \\
 G_7 & G_5 & G_1^2 \\
 G_9 & G_7 & G_1^4+G_3 G_1 \\
\end{vmatrix}}\cdot \frac{1}{U_{5,0}}+G_1 G_3\nonumber\\
& \qquad =a_1 a_2 a_3 a_4+a_1 a_2 a_3 a_5 +\cdots+ a_2 a_3 a_4 a_5\\[1mm]
& {\begin{vmatrix}
 G_5 & G_3 & G_1 \\
 G_7 & G_5 & G_1^3+G_3 \\
 G_9 & G_7 & G_1^5+G_3 G_1^2+G_5 \\
\end{vmatrix}}\cdot \frac{1}{U_{5,0}}=-a_1 a_2 a_3 a_4 a_5\\[1mm]
& {\begin{vmatrix}
 G_5 & G_3 & 0 \\
 G_7 & G_5 & G_3 \\
 G_9 & G_7 & G_5 \\
\end{vmatrix}}\cdot \frac{1}{U_{5,0}}=(P_1-a_1)(P_1-a_2)\cdots(P_1-a_5)
\end{align}
\end{example}

The most important use of Identity \eqref{identity13579} is in finding the ideal symmetric solutions of PTE of degree 10, specifically, the ideal non-negative integer solutions of type $(k=1,3,5,7,9)$ where one item is 0. However, despite several years of computer searching, we have still not been able to find such a solution.\\
\indent
Nevertheless, by utilizing Identity \eqref{identity13579} and with the assistance of computer searching, we successfully found the first set of positive integer solutions for $(k=1,3,5,7,9)$ in 2000.
\begin{align}
[7, 91, 173, 269, 289, 323 ]^k = [ 29, 59, 193, 247, 311, 313]^k\;,\nonumber\\(k=1,3,5,7,9) 
\tag*{\eqref{k13579s323}}
\end{align}
The following four other known solutions of $(k=1,3,5,7,9)$, first found by Jarosław Wróblewski in 2009,  using his method, may also be found by applying Identity \eqref{identity13579}.
\begin{align}
& [ 23, 163, 181, 341, 347, 407 ]^k = [ 37, 119, 221, 311, 371, 403 ]^k\;,\tag*{\eqref{k13579s407}}\\
& [ 43, 161, 217, 335, 391, 463 ]^k = [ 85, 91, 283, 287, 403, 461 ]^k\;,
\tag*{\eqref{k13579s463}}\\
& [ 57, 399, 679, 995, 1167, 1293 ]^k = [ 115, 299, 767, 925, 1205, 1279 ]^k\;,\tag*{\eqref{k13579s1293}}\\
& [ -13, 365, 689, 1111, 1115, 1325 ]^k= [ 23, 305, 731, 1037, 1177, 1319 ]^k\;,\nonumber\\
& \qquad\qquad\qquad\qquad\qquad\qquad\qquad\qquad\qquad\qquad (k=1,3,5,7,9)
\end{align}
Another attempt to apply Identity \eqref{identity13579} is to find an ideal solution of $(k=1,2,3,5,7,9)$, which will require more computer time. We note that all five known solutions of $(k=1,3,5,7,9)$ and most of the known solutions of $(k=1,2,3,5,7)$ consist only of odd integers. Therefore, we try to find the solution of $(k=1,2,3,5,7,9)$ only in the range of odd integers to speed up the search. After running the algorithm on an i7 laptop for 30 minutes, we found the first ideal solution in 2023.
\begin{align}
& [ 7, 89, 91, 251, 253, 341, 373 ]^k = [29, 31, 151, 193, 311, 313, 377 ]^k\;,\nonumber\\
& \qquad\qquad\qquad\qquad\qquad\qquad\qquad\qquad\qquad\quad (k=1,2,3,5,7,9) 
\tag*{\eqref{k123579s377}} 
\end{align}
By improving the algorithm, we found the second ideal solutions of $(k=1,2,3,5,7,9)$ after the computer searched for approximately 36 hours in 2023.
\begin{align}
& [ 269, 397, 409, 683, 743, 901, 923 ]^k\nonumber\\ 
& \quad= [ 299, 313, 493, 613, 827, 839, 941 ]^k,\quad(k=1,2,3,5,7,9)  \tag*{\eqref{k123579s941}} 
\end{align}
We noticed that \eqref{k13579s323} and \eqref{k123579s377} have six of the same numbers, and then we obtained the chain below, although it is not an ideal chain.
\begin{equation}
\begin{aligned}
& [-91,-7,29,193,311,313]^k= [-247,-59,173,269,289,323]^k\\ & \quad  = [-377,-151,-31,89,251,253,341,373]^k\;,\quad (k=1,3,5,7,9) 
\end{aligned}
\end{equation}

\begin{example}
Let
\begin{equation}
\begin{aligned}
P_k=a_1^k+a_2^k+a_3^k+a_4^k+a_5^k,\quad (k=1,3,5,7,11)
\end{aligned}
\end{equation}
Define
\begin{align}
& G_1=-P_1 \tag*{\eqref{G1}}\\ 
& G_3=(P_1^3-P_3)/3 \tag*{\eqref{G3}}\\ 
& G_5=(P_1^5-P_5)/5 \tag*{\eqref{G5}} \\
& G_7=(P_1^7-P_7)/7 \tag*{\eqref{G7}}\\
& G_{11}=(P_1^{11}-P_{11})/11-G_3^2 G_5 \tag*{\eqref{G11}}
\end{align}
Then it follows that
\begin{align}
\label{identity135711} 
& G_1^4+ 4 \frac {\begin{vmatrix}
 G_5 & 0 & 1 \\
 G_7 & G_3 & G_1^2 \\
 G_{11} & G_7 & G_1^6+G_3 G_1^3+G_5 G_1 \\
\end{vmatrix}}{\begin{vmatrix}
 G_3 & 0 & 1 \\
 G_5 & G_3 & G_1^2 \\
 G_7 & G_5 & G_1^4+G_3 G_1 \\
\end{vmatrix}}=(a_1^2+a_2^2+a_3^2+a_4^2+a_5^2)^2
\end{align}
\end{example}

\begin{example}
Let
\begin{align}
& P_k=a_1^k+a_2^k+a_3^k+a_4^k+a_5^k+a_6^k,\quad (k=1,3,5,7,9,11)
\end{align}
and
\begin{align}
& G_1=-P_1 \tag*{\eqref{G1}}\\ 
& G_3=(P_1^3-P_3)/3 \tag*{\eqref{G3}}\\ 
& G_5=(P_1^5-P_5)/5 \tag*{\eqref{G5}} \\
& G_7=(P_1^7-P_7)/7 \tag*{\eqref{G7}}\\
& G_9=(P_1^9-P_9)/9-(G_3^3/3) \tag*{\eqref{G9}}\\
& G_{11}=(P_1^{11}-P_{11})/11-(G_3^2 G_5) \tag*{\eqref{G11}}
\end{align}
Then the following identity holds:
\begin{align}
\label{identity1357911} 
& G_1^2+ 2 \frac {\begin{vmatrix}
 G_7 & G_3 & G_1 \\
 G_9 & G_5 & G_1^3+G_3 \\
 G_{11} & G_7 & G_1^5+G_3 G_1^2+G_5 \\
\end{vmatrix}}{\begin{vmatrix}
 G_5 & G_3 & G_1 \\
 G_7 & G_5 & G_1^3+G_3 \\
 G_9 & G_7 & G_1^5+G_3 G_1^2+G_5 \\
\end{vmatrix}}=a_1^2+a_2^2+\cdots+a_6^2
\end{align}
Furthermore, define:
\begin{align}
 U_{6,0}={\begin{vmatrix}
 G_5 & G_3 & G_1 \\
 G_7 & G_5 & G_1^3+G_3 \\
 G_9 & G_7 & G_1^5+G_3 G_1^2+G_5  \\
\end{vmatrix}}
\end{align}
Then it follows that
\begin{align}
& U_{6,0}=(a_1+a_2)(a_1+a_3)\cdots (a_5+a_6)
\end{align}
and
\begin{align}
& \begin{vmatrix}
 G_7 & G_3 & G_1 \\
 G_9 & G_5 & G_1^3+G_3 \\
 G_{11} & G_7 & G_1^5+G_3 G_1^2+G_5 \\
\end{vmatrix} \cdot \frac{1}{U_{6,0}} \nonumber\\
& \qquad =-(a_1 a_2+a_1 a_3+\cdots+a_5 a_6)\\
& \begin{vmatrix}
 G_7 & G_3 & G_1^2 \\
 G_9 & G_5 & G_1^4+G_3 G_1 \\
 G_{11} & G_7 & G_1^6+G_3 G_1^3+G_5 G_1 \\
\end{vmatrix} \cdot \frac{1}{U_{6,0}}-G_3  \nonumber\\
& \qquad =a_1 a_2 a_3+a_1 a_2 a_4 +\cdots+ a_4 a_5 a_6\\
& \begin{vmatrix}
 G_7 & G_5 & G_1 \\
 G_9 & G_7 & G_1^3+G_3 \\
 G_{11} & G_9 & G_1^5+G_3 G_1^2+G_5 \\
\end{vmatrix}\cdot \frac{1}{U_{6,0}}+G_1 G_3  \nonumber\\
& \qquad =a_1 a_2 a_3 a_4+a_1 a_2 a_3 a_5 +\cdots+ a_3 a_4 a_5 a_6\\
& \begin{vmatrix}
 G_5 & G_3 & G_1 & 0 \\
 G_7 & G_5 & G_3 & G_1 \\
 G_9 & G_7 & G_5 & G_1^3+G_3 \\
 G_{11} & G_9 & G_7 & G_1^5+G_3 G_1^2+G_5 \\
\end{vmatrix}\cdot \frac{1}{U_{6,0}} \nonumber\\
& \qquad =-(a_1 a_2 a_3 a_4 a_5+a_1 a_2 a_3 a_4 a_6 +\cdots+ a_2 a_3 a_4 a_5 a_6)\\
& \begin{vmatrix}
 G_5 & G_3 & 0 & 1 \\
 G_7 & G_5 & G_3 & G_1^2 \\
 G_9 & G_7 & G_5 & G_1^4+G_3 G_1 \\
 G_{11} & G_9 & G_7 & G_1^6+G_3 G_1^3+G_5 G_1 \\
\end{vmatrix}\cdot \frac{1}{U_{6,0}} =(a_1 a_2 a_3 a_4 a_5 a_6)\\
& \begin{vmatrix}
 G_7 & G_5 & G_3 \\
 G_9 & G_7 & G_5 \\
 G_{11} & G_9 & G_7 \\
\end{vmatrix}\cdot \frac{1}{U_{6,0}} =-(P_1-a_1)(P_1-a_2)\cdots(P_1-a_6)
\end{align}
\end{example}

\begin{example}
Let
\begin{align}
& P_k=a_1^k+a_2^k+a_3^k+a_4^k+a_5^k+a_6^k+a_7^k,\nonumber \\
& \qquad\qquad\qquad\quad (k=1,3,5,7,9,11,13)
\end{align}
and
\begin{align}
& G_1=-P_1 \tag*{\eqref{G1}}\\ 
& G_3=(P_1^3-P_3)/3 \tag*{\eqref{G3}}\\ 
& G_5=(P_1^5-P_5)/5 \tag*{\eqref{G5}} \\
& G_7=(P_1^7-P_7)/7 \tag*{\eqref{G7}}\\
& G_9=(P_1^9-P_9)/9-(G_3^3/3) \tag*{\eqref{G9}}\\
& G_{11}=(P_1^{11}-P_{11})/11-(G_3^2 G_5) \tag*{\eqref{G11}}\\
& G_{13}=(P_1^{13}-P_{13})/13-(G_3 G_5^2+G_3^2 G_7) \tag*{\eqref{G13}}
\end{align}
then
\begin{equation}
\begin{aligned}
\label{identity135791113} 
& G_1^2+ 2 \frac {\begin{vmatrix}
 G_7 & G_3 & 0 & 1 \\
 G_9 & G_5 & G_3 & G_1^2 \\
 G_{11} & G_7 & G_5 & G_1^4+G_3 G_1 \\
 G_{13} & G_9 & G_7 & G_1^6+G_3 G_1^3+G_5 G_1 \\
\end{vmatrix}}{\begin{vmatrix}
 G_5 & G_3 & 0 & 1 \\
 G_7 & G_5 & G_3 & G_1^2 \\
 G_9 & G_7 & G_5 & G_1^4+G_3 G_1 \\
 G_{11} & G_9 & G_7 & G_1^6+G_3 G_1^3+G_5 G_1 \\
\end{vmatrix}}\\
& \qquad =a_1^2+a_2^2+a_3^2+a_4^2+a_5^2+a_6^2+a_7^2\\
\end{aligned}
\end{equation}
\end{example}

\begin{example}
Let
\begin{align}
& P_k=a_1^k+a_2^k+a_3^k+a_4^k+a_5^k+a_6^k+a_7^k+a_8^k,\nonumber \\
& \qquad\qquad\qquad\qquad (k=1,3,5,7,9,11,13,15)
\end{align}
and
\begin{align}
& G_1=-P_1 \tag*{\eqref{G1}}\\ 
& G_3=(P_1^3-P_3)/3 \tag*{\eqref{G3}}\\ 
& G_5=(P_1^5-P_5)/5 \tag*{\eqref{G5}} \\
& G_7=(P_1^7-P_7)/7 \tag*{\eqref{G7}}\\
& G_9=(P_1^9-P_9)/9-(G_3^3/3) \tag*{\eqref{G9}}\\
& G_{11}=(P_1^{11}-P_{11})/11-(G_3^2 G_5) \tag*{\eqref{G11}}\\
& G_{13}=(P_1^{13}-P_{13})/13-(G_3 G_5^2+G_3^2 G_7) \tag*{\eqref{G13}}\\
& G_{15}=(P_1^{15}-P_{15})/15-(G_3^5/3+ G_5^3/5+2 G_3 G_5 G_7+G_3^2 G_9) \tag*{\eqref{G15}}
\end{align}
then
\begin{equation}
\begin{aligned}
\label{identity13579111315} 
& G_1^2+ 2 \frac {\begin{vmatrix}
 G_9 & G_5 & G_3 & G_1 \\
 G_{11} & G_7 & G_5 & G_1^3+G_3 \\
 G_{13} & G_9 & G_7 & G_1^5+G_3 G_1^2+G_5 \\
 G_{15} & G_{11} & G_9 & G_1^7+G_3 G_1^4+G_5 G_1^2+G_7 \\
\end{vmatrix}}{\begin{vmatrix}
 G_7 & G_5 & G_3 & G_1 \\
 G_9 & G_7 & G_5 & G_1^3+G_3 \\
 G_{11} & G_9 & G_7 & G_1^5+G_3 G_1^2+G_5 \\
 G_{13} & G_{11} & G_9 & G_1^7+G_3 G_1^4+G_5 G_1^2+G_7 \\
\end{vmatrix}}\\
& \qquad =a_1^2+a_2^2+a_3^2+a_4^2+a_5^2+a_6^2+a_7^2+a_8^2\\
\end{aligned}
\end{equation}
\end{example}

\begin{example}
Let
\begin{align}
& P_k=a_1^k+a_2^k+a_3^k+a_4^k+a_5^k+a_6^k+a_7^k+a_8^k+a_9^k,\nonumber \\
& \qquad\qquad\qquad\qquad\quad (k=1,3,5,7,9,11,13,15,17)
\end{align}
and
\begin{align}
& G_1=-P_1 \tag*{\eqref{G1}}\\ 
& G_3=(P_1^3-P_3)/3 \tag*{\eqref{G3}}\\ 
& G_5=(P_1^5-P_5)/5 \tag*{\eqref{G5}} \\
& G_7=(P_1^7-P_7)/7 \tag*{\eqref{G7}}\\
& G_9=(P_1^9-P_9)/9-(G_3^3/3) \tag*{\eqref{G9}}\\
& G_{11}=(P_1^{11}-P_{11})/11-(G_3^2 G_5) \tag*{\eqref{G11}}\\
& G_{13}=(P_1^{13}-P_{13})/13-(G_3 G_5^2+G_3^2 G_7) \tag*{\eqref{G13}}\\
& G_{15}=(P_1^{15}-P_{15})/15-(G_3^5/3+ G_5^3/5+2 G_3 G_5 G_7+G_3^2 G_9) \tag*{\eqref{G15}}\\
& G_{17}=(P_1^{17}-P_{17})/17 \nonumber\\
& \qquad\qquad -(G_5 G_3^4+G_{11} G_3^2+G_7^2 G_3+2 G_5 G_9 G_3+G_5^2 G_7) \tag*{\eqref{G17}}
\end{align}
then
\begin{equation}
\begin{aligned}
\label{identity1357911131517} 
& G_1^2+ 2 \frac {\begin{vmatrix}
 G_9 & G_5 & G_3 & 0 & 1 \\
 G_{11} & G_7 & G_5 & G_3 & G_1^2 \\
 G_{13} & G_9 & G_7 & G_5 & G_1^4+G_3 G_1 \\
 G_{15} & G_{11} & G_9 & G_7 & G_1^6+G_3 G_1^3+G_5 G_1 \\
 G_{17} & G_{13} & G_{11} & G_9 & G_1^8+G_3 G_1^5+G_5 G_1^3+G_7 G_1 \\
\end{vmatrix}}{\begin{vmatrix}
 G_7 & G_5 & G_3 & 0 & 1 \\
 G_9 & G_7 & G_5 & G_3 & G_1^2 \\
 G_{11} & G_9 & G_7 & G_5 & G_1^4+G_3 G_1 \\
 G_{13} & G_{11} & G_9 & G_7 & G_1^6+G_3 G_1^3+G_5 G_1 \\
 G_{15} & G_{13} & G_{11} & G_9 & G_1^8+G_3 G_1^5+G_5 G_1^3+G_7 G_1 \\
\end{vmatrix}}\\
& \qquad =a_1^2+a_2^2+a_3^2+a_4^2+a_5^2+a_6^2+a_7^2+a_8^2+a_9^2\\
\end{aligned}
\end{equation}
\end{example}
For cases where the exponential powers are negative odd integers, the \mbox{definitions} of $P_{2t+1}$ and $G_{2t+1}$ are similar to those in \eqref{defineP2tp1} and \eqref{defineG}. Below are examples  of $G_{2t+1}$ for integers $t$ where $-7 \leq t \leq -1$:

\begin{example}
\begin{align}
\label{Gm1}
& G_{-1}=P_{-1}\\
\label{Gm3}
& G_{-3}=-\frac{1}{3} \left(P_{-1}^3-P_{-3}\right)\\
\label{Gm5}
& G_{-5}=-\frac{1}{5} \left(P_{-1}^5-P_{-5}\right)\\
\label{Gm7}
& G_{-7}=-\frac{1}{7} \left(P_{-1}^7-P_{-7}\right)\\
\label{Gm9}
& G_{-9}=-\frac{1}{9} \left(P_{-1}^9-P_{-9}\right)-\frac{G_{-3}^3}{3}\\
\label{Gm11}
& G_{-11}=-\frac{1}{11} \left(P_{-1}^{11}-P_{-11}\right)-G_{-3}^2 G_{-5}\\
\label{Gm13}
& G_{-13}=-\frac{1}{13} \left(P_{-1}^{13}-P_{-13}\right)-\left(G_{-3} G_{-5}^2 + G_{-3}^2 G_{-7} \right)
\end{align}
\end{example}

The following are examples of the Third generalization of the Girard-Newton Identities, where the exponential powers include both positive and negative odd integers.

\begin{example}
Let
\begin{align}
P_k=a_1^k+a_2^k+a_3^k,\quad (k=-1,1,3)
\end{align}
and
\begin{align}
& G_{-1}=P_{-1} \tag*{\eqref{Gm1}}\\ 
& G_1=-P_1 \tag*{\eqref{G1}}\\ 
& G_3=(P_1^3-P_3)/3 \tag*{\eqref{G3}}
\end{align}
then
\begin{equation}
\begin{aligned}
\label{identity13n1} 
& G_1^2+ 2 \frac {\begin{vmatrix}
 0 & -G_{-1} \\
 G_3 & 1 \\
\end{vmatrix}}{\begin{vmatrix}
 1+G_{-1} G_1 & -G_{-1} \\
 0 & 1 \\
\end{vmatrix}}=a_1^2+a_2^2+a_3^2\\
\end{aligned}
\end{equation}
\end{example}

\begin{example}
Let
\begin{align}
P_k=a_1^k+a_2^k+a_3^k,\quad (k=-1,1,5)
\end{align}
and
\begin{align}
& G_{-1}=P_{-1} \tag*{\eqref{Gm1}}\\ 
& G_1=-P_1 \tag*{\eqref{G1}}\\ 
& G_5=(P_1^5-P_5)/5 \tag*{\eqref{G5}}
\end{align}
then
\begin{align}
\label{identity15n1} 
& G_1^4+ 4 \frac {\begin{vmatrix}
 0 & -G_{-1} \\
 G_5 & G_1^2 \\
\end{vmatrix}}{\begin{vmatrix}
 1+G_{-1} G_1 & -G_{-1} \\
 0 & 1 \\
\end{vmatrix}}=(a_1^2+a_2^2+a_3^2)^2
\end{align}
\end{example}

\begin{example}
Let
\begin{align}
P_k=a_1^k+a_2^k+a_3^k+a_4^k,\quad (k=-1,1,3,5)
\end{align}
and
\begin{align}
& G_{-1}=P_{-1} \tag*{\eqref{Gm1}}\\ 
& G_1=-P_1 \tag*{\eqref{G1}}\\ 
& G_3=(P_1^3-P_3)/3 \tag*{\eqref{G3}}\\
& G_5=(P_1^5-P_5)/5 \tag*{\eqref{G5}}
\end{align}
then
\begin{align}
\label{identity135n1} 
 G_1^2+ 2 \frac {\begin{vmatrix}
 G_3 (1+G_{-1} G_1) & G_{-1} \\
 G_5 & G_1 \\
\end{vmatrix}}{\begin{vmatrix}
 G_1 & G_{-1} \\
 G_3 & G_1 \\
\end{vmatrix}}=a_1^2+a_2^2+a_3^2+a_4^2
\end{align}
\end{example}

\begin{example}
Let
\begin{align}
P_k=a_1^k+a_2^k+a_3^k+a_4^k,\quad (k=-1,1,3,7)
\end{align}
and
\begin{align}
& G_{-1}=P_{-1} \tag*{\eqref{Gm1}}\\ 
& G_1=-P_1 \tag*{\eqref{G1}}\\ 
& G_3=(P_1^3-P_3)/3 \tag*{\eqref{G3}}\\
& G_7=(P_1^7-P_7)/7 \tag*{\eqref{G7}}
\end{align}
then
\begin{align}
\label{identity137n1} 
 G_1^4+ 4 \frac {\begin{vmatrix}
 G_3 (1+G_{-1} G_1) & G_{-1} \\
 G_7 & G_1^3+G_3 \\
\end{vmatrix}}{\begin{vmatrix}
 G_1 & G_{-1} \\
 G_3 & G_1 \\
\end{vmatrix}}=(a_1^2+a_2^2+a_3^2+a_4^2)^2
\end{align}
\end{example}

\begin{example}
Let
\begin{align}
P_k=a_1^k+a_2^k+a_3^k+a_4^k+a_5^k,\quad (k=-1,1,3,5,7)
\end{align}
and
\begin{align}
& G_{-1}=P_{-1} \tag*{\eqref{Gm1}}\\ 
& G_1=-P_1 \tag*{\eqref{G1}}\\ 
& G_3=(P_1^3-P_3)/3 \tag*{\eqref{G3}}\\
& G_5=(P_1^5-P_5)/5 \tag*{\eqref{G5}}\\
& G_7=(P_1^7-P_7)/7 \tag*{\eqref{G7}}
\end{align}
then
\begin{align}
\label{identity1357n1} 
 G_1^2+ 2 \frac {\begin{vmatrix}
 G_1 G_3 & 1+G_{-1} G_1 & -G_{-1} \\
 G_5 & 0 & 1 \\
 G_7 & G_3 & G_1^2 \\
\end{vmatrix}}{\begin{vmatrix}
 0 & 1+G_{-1} G_1 & -G_{-1} \\
 G_3 & 0 & 1 \\
 G_5 & G_3 & G_1^2 \\
\end{vmatrix}}=a_1^2+a_2^2+a_3^2+a_4^2+a_5^2
\end{align}
\end{example}

\begin{example}
Let
\begin{align}
P_k=a_1^k+a_2^k+a_3^k+a_4^k+a_5^k,\quad (k=-1,1,3,5,9)
\end{align}
and
\begin{align}
& G_{-1}=P_{-1} \tag*{\eqref{Gm1}}\\ 
& G_1=-P_1 \tag*{\eqref{G1}}\\ 
& G_3=(P_1^3-P_3)/3 \tag*{\eqref{G3}}\\
& G_5=(P_1^5-P_5)/5 \tag*{\eqref{G5}}\\
& G_9=(P_1^9-P_9)/9-G_3^3/3 \tag*{\eqref{G9}}
\end{align}
then
\begin{align}
\label{identity1359n1} 
 G_1^4+ 4 \frac {\begin{vmatrix}
 G_1 G_3 & 1+G_{-1} G_1 & -G_{-1} \\
 G_5 & 0 & 1 \\
 G_9 & G_5 & G_1^4+G_3 G_1 \\
\end{vmatrix}}{\begin{vmatrix}
 0 & 1+G_{-1} G_1 & -G_{-1} \\
 G_3 & 0 & 1 \\
 G_5 & G_3 & G_1^2 \\
\end{vmatrix}}=(a_1^2+a_2^2+a_3^2+a_4^2+a_5^2)^2
\end{align}
\end{example}

\begin{example}
Let
\begin{align}
P_k=a_1^k+a_2^k+a_3^k+a_4^k,\quad (k=-3,-1,1,3)
\end{align}
and
\begin{align}
& G_{-1}=P_{-1} \tag*{\eqref{Gm1}}\\ 
& G_{-3}=-(P_{-1}^3-P_{-3})/3 \tag*{\eqref{Gm3}}\\
& G_1=-P_1 \tag*{\eqref{G1}}\\ 
& G_3=(P_1^3-P_3)/3 \tag*{\eqref{G3}}
\end{align}
then
\begin{align}
\label{identity13n13} 
 G_1^2+ 2 \frac {\begin{vmatrix}
 G_1 & G_{-3} \\
 G_3 & G_{-1} \\
\end{vmatrix}}{\begin{vmatrix}
 G_{-1} & G_{-3} \\
 G_1 & G_{-1} \\
\end{vmatrix}}=a_1^2+a_2^2+a_3^2+a_4^2
\end{align}
\end{example}

\begin{example}
Let
\begin{align}
P_k=a_1^k+a_2^k+a_3^k+a_4^k,\quad (k=-3,-1,1,5)
\end{align}
and
\begin{align}
& G_{-1}=P_{-1} \tag*{\eqref{Gm1}}\\ 
& G_{-3}=-(P_{-1}^3-P_{-3})/3 \tag*{\eqref{Gm3}}\\
& G_1=-P_1 \tag*{\eqref{G1}}\\ 
& G_5=(P_1^5-P_5)/5 \tag*{\eqref{G5}}
\end{align}
then
\begin{align}
\label{identity13n15} 
 G_1^4+ 4 \frac {\begin{vmatrix}
 G_1 (1+G_{-1} G_1) & G_{-3} \\
 G_5 & G_1 \\
\end{vmatrix}}{\begin{vmatrix}
 G_{-1} & G_{-3} \\
 G_1 & G_{-1} \\
\end{vmatrix}}=(a_1^2+a_2^2+a_3^2+a_4^2)^2
\end{align}
\end{example}

\begin{example}
Let
\begin{align}
P_k=a_1^k+a_2^k+a_3^k+a_4^k+a_5^k,\quad (k=-3,-1,1,3,5)
\end{align}
and
\begin{align}
& G_{-1}=P_{-1} \tag*{\eqref{Gm1}}\\ 
& G_{-3}=-(P_{-1}^3-P_{-3})/3 \tag*{\eqref{Gm3}}\\
& G_1=-P_1 \tag*{\eqref{G1}}\\ 
& G_3=(P_1^3-P_3)/3 \tag*{\eqref{G3}}\\
& G_5=(P_1^5-P_5)/5 \tag*{\eqref{G5}}
\end{align}
then
\begin{align}
\label{identity135n13} 
 G_1^2+ 2 \frac {\begin{vmatrix}
 G_{-1} G_3 & G_1 G_{-3} & -G_{-3} \\
 G_1 G_3 &  1+G_{-1} G_1 & -G_{-1} \\
 G_5 & 0 & 1 \\
\end{vmatrix}}{\begin{vmatrix}
 1+G_{-1} G_1 & G_1 G_{-3} & -G_{-3} \\
 0 &  1+G_{-1} G_1 & -G_{-1} \\
 G_3 & 0 & 1 \\
\end{vmatrix}}=a_1^2+a_2^2+a_3^2+a_4^2+a_5^2 \nonumber\\
\end{align}
\end{example}

\begin{example}
Let
\begin{align}
P_k=a_1^k+a_2^k+a_3^k+a_4^k+a_5^k,\quad (k=-3,-1,1,3,7)
\end{align}
and
\begin{align}
& G_{-1}=P_{-1} \tag*{\eqref{Gm1}}\\ 
& G_{-3}=-(P_{-1}^3-P_{-3})/3 \tag*{\eqref{Gm3}}\\
& G_1=-P_1 \tag*{\eqref{G1}}\\ 
& G_3=(P_1^3-P_3)/3 \tag*{\eqref{G3}}\\
& G_7=(P_1^7-P_7)/7 \tag*{\eqref{G7}}
\end{align}
then
\begin{align}
\label{identity137n13} 
 G_1^4+ 4 \frac {\begin{vmatrix}
 G_{-1} G_3 & G_1 G_{-3} & -G_{-3} \\
 G_1 G_3 &  1+G_{-1} G_1 & -G_{-1} \\
 G_7 & G_3 & G_1^2 \\
\end{vmatrix}}{\begin{vmatrix}
 1+G_{-1} G_1 & G_1 G_{-3} & -G_{-3} \\
 0 &  1+G_{-1} G_1 & -G_{-1} \\
 G_3 & 0 & 1 \\
\end{vmatrix}}=(a_1^2+a_2^2+a_3^2+a_4^2+a_5^2)^2 \nonumber\\
\end{align}
\end{example}

\begin{example}
Let
\begin{align}
P_k=a_1^k+a_2^k+a_3^k+a_4^k+a_5^k+a_6^k,\quad (k=-3,-1,1,3,5,7)
\end{align}
and
\begin{align}
& G_{-1}=P_{-1} \tag*{\eqref{Gm1}}\\ 
& G_{-3}=-(P_{-1}^3-P_{-3})/3 \tag*{\eqref{Gm3}}\\
& G_1=-P_1 \tag*{\eqref{G1}}\\ 
& G_3=(P_1^3-P_3)/3 \tag*{\eqref{G3}}\\
& G_5=(P_1^5-P_5)/5 \tag*{\eqref{G5}}\\
& G_7=(P_1^7-P_7)/7 \tag*{\eqref{G7}}
\end{align}
then
\begin{align}
\label{identity1357n13} 
 G_1^2+ 2 \frac {\begin{vmatrix}
 G_3 (1+G_{-1} G_1) & G_{-1} & G_{-3} \\
 G_5 & G_1 & G_{-1} \\
 G_7-G_1^2 G_5 & G_3 & G_1 \\
\end{vmatrix}}{\begin{vmatrix}
 G_1 & G_{-1} & G_{-3} \\
 G_3 & G_1 & G_{-1} \\
 G_5-G_1^2 G_3 & G_3 & G_1 \\
\end{vmatrix}}=a_1^2+a_2^2+a_3^2+a_4^2+a_5^2+a_6^2 \nonumber\\
\end{align}
\end{example}

\begin{example}
Let
\begin{align}
P_k=a_1^k+a_2^k+a_3^k+a_4^k+a_5^k+a_6^k,\quad (k=-3,-1,1,3,5,9)
\end{align}
and
\begin{align}
& G_{-1}=P_{-1} \tag*{\eqref{Gm1}}\\ 
& G_{-3}=-(P_{-1}^3-P_{-3})/3 \tag*{\eqref{Gm3}}\\
& G_1=-P_1 \tag*{\eqref{G1}}\\ 
& G_3=(P_1^3-P_3)/3 \tag*{\eqref{G3}}\\
& G_5=(P_1^5-P_5)/5 \tag*{\eqref{G5}}\\
& G_9=(P_1^9-P_9)/9 -G_3^3/3\tag*{\eqref{G9}}
\end{align}
then
\begin{align}
\label{identity1359n13} 
& G_1^4+ 4 \frac {\begin{vmatrix}
 G_3 (1+G_{-1} G_1) & G_{-1} & G_{-3} \\
 G_5 & G_1 & G_{-1} \\
 G_9-G_1^4 G_5-G_1 G_3 G_5 & G_5 & G_1^3+G_3 \\
\end{vmatrix}}{\begin{vmatrix}
 G_1 & G_{-1} & G_{-3} \\
 G_3 & G_1 & G_{-1} \\
 G_5-G_1^2 G_3 & G_3 & G_1 \\
\end{vmatrix}} \nonumber\\
& \qquad  =(a_1^2+a_2^2+a_3^2+a_4^2+a_5^2+a_6^2)^2 
\end{align}
\end{example}

\begin{example}
Let
\begin{align}
P_k=a_1^k+a_2^k+a_3^k+a_4^k+a_5^k+a_6^k,\quad (k=-5,-3,-1,1,3,5)
\end{align}
and
\begin{align}
& G_{-1}=P_{-1} \tag*{\eqref{Gm1}}\\ 
& G_{-3}=-(P_{-1}^3-P_{-3})/3 \tag*{\eqref{Gm3}}\\
& G_{-5}=-(P_{-1}^5-P_{-5})/5 \tag*{\eqref{Gm5}}\\
& G_1=-P_1 \tag*{\eqref{G1}}\\ 
& G_3=(P_1^3-P_3)/3 \tag*{\eqref{G3}}\\
& G_5=(P_1^5-P_5)/5 \tag*{\eqref{G5}}
\end{align}
then
\begin{align}
\label{identity135n135} 
 G_1^2+ 2 \frac {\begin{vmatrix}
 G_1 & G_{-3} & G_{-5}-G_{-1}^2 G_{-3} \\
 G_3 & G_{-1} & G_{-3} \\
 G_5-G_1^2 G_3 & G_1 & G_{-1} \\
\end{vmatrix}}{\begin{vmatrix}
 G_{-1} & G_{-3} & G_{-5}-G_{-1}^2 G_{-3} \\
 G_1 & G_{-1} & G_{-3} \\
 G_3 & G_1 & G_{-1} \\
\end{vmatrix}} =a_1^2+a_2^2+a_3^2+a_4^2+a_5^2+a_6^2\nonumber\\
\end{align}
\end{example}

\begin{example}
Let
\begin{align}
P_k=a_1^k+a_2^k+a_3^k+a_4^k+a_5^k+a_6^k,\quad (k=-5,-3,-1,1,3,7)
\end{align}
and
\begin{align}
& G_{-1}=P_{-1} \tag*{\eqref{Gm1}}\\ 
& G_{-3}=-(P_{-1}^3-P_{-3})/3 \tag*{\eqref{Gm3}}\\
& G_{-5}=-(P_{-1}^5-P_{-5})/5 \tag*{\eqref{Gm5}}\\
& G_1=-P_1 \tag*{\eqref{G1}}\\ 
& G_3=(P_1^3-P_3)/3 \tag*{\eqref{G3}}\\
& G_7=(P_1^7-P_7)/7 \tag*{\eqref{G7}}
\end{align}
then
\begin{align}
\label{identity137n135} 
&  G_1^4+ 4 \frac {\begin{vmatrix}
 G_1 & G_{-3} & G_{-5}-G_{-1}^2 G_{-3} \\
 G_3 & G_{-1} & G_{-3} \\
 G_7-G_1^4 G_3-G_1 G_3^2 & G_1^3+G_3 & G_1 (1+G_{-1}G_1) \\
\end{vmatrix}}{\begin{vmatrix}
 G_{-1} & G_{-3} & G_{-5}-G_{-1}^2 G_{-3} \\
 G_1 & G_{-1} & G_{-3} \\
 G_3 & G_1 & G_{-1} \\
\end{vmatrix}} \nonumber\\
& \qquad =(a_1^2+a_2^2+a_3^2+a_4^2+a_5^2+a_6^2)^2
\end{align}
\end{example}

In this chapter, we have discussed in detail the classical form of the Girard-Newton Identities (\refIdentity{identity1}), the equivalent form of the Girard-Newton Identities (\refIdentity{identity2}), and the three generalized forms of the Girard-Newton Identities (Identity \ref{identity_GNI1}, Identity \ref{identity_GNI2}, and Identity \ref{identity_GNI3}). From the analysis presented in this chapter, it can be observed that \refIdentity{identity1}, \refIdentity{identity2}, and \refIdentity{identity_GNI1} are primarily applied to the following system, with a focus on analyzing the possibility of the existence of ideal integer solutions of the GPTE problem:
\begin{align}
&\left[ a_{1}, a_{2}, \cdots, a_{n} \right]^{h} = \left[ b_{1}, b_{2}, \cdots, b_{n} \right]^{h}, \quad \left( h = h_1, h_2, \dots, h_n \right).  \tag*{\eqref{GPTEh}}
\end{align}
On the other hand, Identity \ref{identity_GNI2} and Identity \ref{identity_GNI3} are primarily utilized in the following system, with an emphasis on searching for ideal non-negative integer solutions of the GPTE problem:
\begin{align}
\left[ a_{1}, a_{2}, \dots, a_{n+1} \right]^{k} = \left[ b_{1}, b_{2}, \dots, b_{n+1} \right]^{k}, \quad (k = k_1, k_2, \dots, k_n).  \tag*{\eqref{GPTEk}}
\end{align}
In Chapter 5, we will also utilize Identities \ref{identity_GNI2} and \ref{identity_GNI3} to determine the exact upper and lower bounds for each variable of \eqref{GPTEk}. This will enhance the efficiency of computer searches for non-negative integer solutions to the GPTE problem.
\clearpage

\section{The Constant in The Generalized PTE Problem}
In the study of the PTE problem, the constant \( C \) has always been a focus of attention \cite{Borwein2003, ReesSmyth}. In this chapter, we will also explore the constant \( C \) for certain types of GPTE.

\subsection{The Constant in The PTE Problem}
For the Diophantine system studied in the PTE problem:
\begin{align}
\sum_{i=1}^m a_i^k = \sum_{i=1}^m b_i^k, \quad \text{for } k = 1, 2, \dots, n, \tag*{\eqref{PTEmk}}
\end{align}
one equivalent form is \cite{Borwein1994, Borwein2003, Caley12}:
\begin{align}
\deg\left(\prod_{i=1}^{m} (x - a_i) - \prod_{i=1}^{m} (x - b_i)\right) \leq m - n - 1.
\end{align}

When \( m = n + 1 \), \(\{a_i\}\) and \(\{b_i\}\) are ideal solutions to the system \eqref{PTEmk}. In this case, we have:
\begin{identity}\label{identityC_PTE}
Let \( n \) be a positive integer. For the type \((k = 1, 2, \dots, n)\), if there exist two sets of integers \(\{a_1, a_2, \dots, a_{n+1}\}\) and \(\{b_1, b_2, \dots, b_{n+1}\}\) satisfying:
\begin{align}
\left[ a_1, a_2, \dots, a_{n+1} \right]^{k} = \left[ b_1, b_2, \dots, b_{n+1} \right]^{k}, \quad (k = 1, 2, \dots, n), \tag*{\eqref{PTEideal}}
\end{align}
then the following identity holds for any real number \( x \):
\begin{align}
\label{equationC1}
\prod_{i=1}^{n+1} (x - a_i) - \prod_{i=1}^{n+1} (x - b_i) = C,
\end{align}
where \( C \) is a non-zero constant.
\end{identity}
The proof of \refIdentity{identityC_PTE} can be derived using the Girard-Newton identities. When we substitute $x=0$, $x=a_i$, and $x=b_i$ into \eqref{equationC1} accordingly, we obtain:

\begin{corollary}\label{corollaryC1}
In \refIdentity{identityC_PTE}, the following relationships hold:
\begin{align}
\label{identityC1Eq1}
C &=(-1)^{n+1}(a_1 a_2 \dots a_{n+1}-b_1 b_2 \dots b_{n+1}), \\[1mm]
\label{identityC1Eq2}
-C &=\prod _{i=1}^{n+1} (a_1-b_i)=\prod _{i=1}^{n+1} (a_2-b_i)=\dots=\prod _{i=1}^{n+1} (a_{n+1}-b_i),\\[1mm]
\label{identityC1Eq3}
C &=\prod_{i=1}^{n+1} (b_1 - a_i) = \prod_{i=1}^{n+1} (b_2 - a_i) = \dots = \prod_{i=1}^{n+1} (b_{n+1} - a_i).
\end{align}
\end{corollary}
The formulas \eqref{identityC1Eq2} and \eqref{identityC1Eq3} in the above corollary indicate that the constant \( C \) can be divisible by any \( a_i - b_j \). This is very useful for enhancing the computer search for ideal solutions of the PTE problem.\\[1mm]
\indent For various numerical examples, in addition to applying \eqref{identityC1Eq1}, we can also directly use \eqref{equationC1} to compute the values of \( C \). For instance:
\begin{align}
& [0, 18, 19, 50, 56, 79, 81]^k = [1, 11, 30, 39, 68, 70, 84]^k, \nonumber \\
& \qquad\qquad\qquad\qquad\qquad\qquad\qquad (k = 1, 2, 3, 4, 5, 6) \tag*{\eqref{k123456s84}} \\
& \qquad \textup{it follows that} \quad C =5145940800=2^6 \cdot 3^3 \cdot 5^2 \cdot 7^2 \cdot 11 \cdot 13 \cdot 17,  \nonumber\\[1mm]
& [0, 7, 23, 50, 53, 81, 82, 96]^k = [1, 5, 26, 42, 63, 72, 88, 95]^k, \nonumber \\ 
& \qquad\qquad\qquad\qquad\qquad\qquad\qquad (k = 1, 2, 3, 4, 5, 6, 7) \tag*{\eqref{k1234567s96}}\\
& \qquad \textup{it follows that} \quad C =207048441600=2^8 \cdot 3^5 \cdot 5^2 \cdot 7^2 \cdot 11 \cdot 13 \cdot 19.  \nonumber
\end{align}

\subsection{The Constant in The GPTE Problem}
During research on the GPTE problem, we discovered that for certain series of GPTE types, there exists a constant \( C \) analogous to that in the PTE problem. We will specify this by \refIdentity{identityC2}, \refIdentity{identityC3}, and \refIdentity{identityC4} as follows.
\begin{identity}\label{identityC2}
Let \( n \) be a positive integer. For the type \((k = k_1, k_1+1, k_1+2, \dots, k_1+n-1)\), where \(k_1\) can be one of the \(n+1\) numbers: \(-n+1, -n+2, \dots, 0, 1\), if there exist two sets of integers \(\{a_1, a_2, \dots, a_{n+1}\}\) and \(\{b_1, b_2, \dots, b_{n+1}\}\) satisfying:
\begin{equation}
\begin{aligned}
[a_1, a_2, \dots, a_{n+1}]^k &= [b_1, b_2, \dots, b_{n+1}]^k, \\
&\quad (k = k_1, k_1+1, k_1+2, \dots, k_1+n-1),
\end{aligned}
\end{equation}
then the following identity holds for any non-zero real number \( x \):
\begin{align}
\label{equationC2}
\frac{1}{x^m} \prod_{i=1}^{n+1} (x - a_i) - \frac{1}{x^m} \prod_{i=1}^{n+1} (x - b_i) = C,
\end{align}
where \( C \) is a non-zero constant, and \(m\) is determined by:
\begin{equation}
m =
\begin{cases}
0, & \textup{for } (k = 1, 2, 3, \dots, n), \\
1, & \textup{for } (k = 0, 1, 2, \dots, n-1), \\
2, & \textup{for } (k = -1, 0, 1, \dots, n-2), \\
\vdots \\
n, & \textup{for } (k = -n+1, -n+2, -n+3, \dots, 0).
\end{cases}
\end{equation}
\end{identity}
The proof of \refIdentity{identityC2} can be derived using the First Generalization of the Girard-Newton Identities (\refIdentity{identity_GNI1}).\\[1mm]
\indent When we substitute \( x = a_i \neq 0 \) and \( x = b_i \neq 0 \) into \eqref{equationC2} accordingly, we obtain:
\begin{corollary}\label{corollaryC2}
In \refIdentity{identityC2}, when \( a_i \neq 0 \) and \( b_i \neq 0 \), the following relationships hold:
\begin{align}
-C = & \frac{\prod _{i=1}^{n+1} (a_1-b_i)}{a_1^m}=\frac{\prod _{i=1}^{n+1} (a_2-b_i)}{a_2^m}=\dots=\frac{\prod _{i=1}^{n+1} (a_{n+1}-b_i)}{a_{n+1}^m}, \\[1mm]
C = & \frac{\prod_{i=1}^{n+1} (b_1 - a_i)}{b_1^m}  = \frac{\prod_{i=1}^{n+1} (b_2 - a_i)}{b_2^m} = \dots = \frac{\prod_{i=1}^{n+1} (b_{n+1} - a_i)}{b_{n+1}^m}. 
\end{align}
\end{corollary}
Using \refIdentity{identity_GNI1}, we can derive the formula for the constant \( C \) of the GPTE types related to \refIdentity{identityC2}. The following example shows the result for \( n = 4 \).

\begin{example}
\label{examplek1234}
For the type \((k = k_1, k_1+1, k_1+2, k_1+3)\), where \(k_1\) can be one of the numbers: \(1, 0, -1, -2, -3 \), if there exist two sets of integers \(\{a_1, a_2, a_3, a_4, a_5\}\) and \(\{b_1, b_2, b_3, b_4, b_5\}\) satisfying:
\begin{align}
[a_1, a_2, a_3, a_4, a_5]^k = [b_1, b_2, b_3, b_4, b_5]^k, \quad (k=k_1, k_1+1, k_1+2, k_1+3). \nonumber
\end{align}
then the constant \( C \) is given by:
\begin{align}
& C = -a_1 a_2 a_3 a_4 a_5 + b_1 b_2 b_3 b_4 b_5, \qquad\qquad \textup{for}\,\, (k=1,2,3,4)\\
& C = a_1 a_2 a_3 a_4 + a_1 a_2 a_3 a_5 + \dots + a_2 a_3 a_4 a_5 \nonumber \\
& \qquad - b_1 b_2 b_3 b_4 - b_1 b_2 b_3 b_5 - \dots - b_2 b_3 b_4 b_5, \quad  \textup{for}\,\, (k=0,1,2,3)\\
& C = - a_1 a_2 a_3 - a_1 a_2 a_4 - \dots - a_3 a_4 a_5 \nonumber \\
& \qquad + b_1 b_2 b_3 + b_1 b_2 b_4  + \dots +  b_3 b_4 b_5, \quad\:\: \textup{for}\,\, (k=-1,0,1,2)\\
& C = a_1 a_2 + a_1 a_3 + \dots + a_4 a_5 \nonumber \\
& \qquad -b_1 b_2 - b_1 b_3  - \dots - b_4 b_5,  \qquad\qquad \textup{for}\,\, (k=-2,-1,0,1)\\
& C = - a_1 - a_2 - a_3 - a_4- a_5 \nonumber \\
& \qquad  +b_1 + b_2 + b_3 + b_4 + b_5.  \quad\qquad\qquad \textup{for}\,\, (k=-3,-2,-1,0)
\end{align}
\end{example}
For various numerical examples, in addition to applying formulas similar to that in Example\,\ref{examplek1234}, we can also directly use \eqref{equationC2} to compute the values of \( C \). For instance:
\begin{align}
& [ 0, 3, 5, 11, 13, 16 ]^k = [ 1, 1, 8, 8, 15, 15 ]^k, \quad (k=1,2,3,4,5),  \tag*{\eqref{k12345s16}}\\
& \qquad \textup{it follows that} \quad C =14400=2^6 \cdot 3^2 \cdot 5^2,  \nonumber\\[1mm]
& [13, 20, 22, 45, 45, 57]^k = [15, 15, 27, 38, 52, 55]^k, \;\; (k=0,1,2,3,4),  \tag*{\eqref{k01234s57}}\\
& \qquad \textup{it follows that} \quad C =176400=2^4 \cdot 3^2 \cdot 5^2 \cdot 7^2,  \nonumber\\[1mm]
& [11, 18, 35, 84, 90, 132]^k = [12, 15, 44, 63, 110, 126]^k,  \nonumber\\
& \qquad\qquad\qquad\qquad\qquad\qquad (k=-1,0,1,2,3),  \tag*{\eqref{k0123n1s132}}\\
& \qquad \textup{it follows that} \quad C =645840=2^4 \cdot 3^3 \cdot 5 \cdot 13 \cdot 23 ,  \nonumber\\[1mm]
& [20, 24, 24, 35, 35, 42]^k = [21, 21, 28, 30, 40, 40]^k, \nonumber\\
& \qquad\qquad\qquad\qquad\qquad\qquad (k=-2,-1,0,1,2),  \tag*{\eqref{k012n12s42}}\\
& \qquad \textup{it follows that} \quad C =4=2^2,  \nonumber\\[1mm]
& [105, 154, 165, 396, 770, 1260]^k = [110, 126, 220, 315, 924, 1155]^k, \nonumber\\
& \qquad\qquad\qquad\qquad\qquad\qquad (k=-3,-2,-1,0,1),  \tag*{\eqref{k01n123s1260}}\\
& \qquad \textup{it follows that} \quad C =17940=2^2 \cdot 3 \cdot 5 \cdot 13 \cdot 23,  \nonumber\\[1mm]
& [3510, 3672, 5967, 8840, 26520, 39780]^k \nonumber\\
& \quad = [3315, 4420, 4680, 11934, 18360, 47736]^k, \nonumber\\
& \qquad\qquad\qquad\qquad\qquad\qquad (k=-4,-3,-2,-1,0),  \tag*{\eqref{k0n1234s47736}}\\
& \qquad \textup{it follows that} \quad C =2156=2^2 \cdot 7^2 \cdot 11.  \nonumber
\end{align}

\begin{identity}\label{identityC3}
Let \( n \) be a positive integer greater than \(1\). For the type \((k = k_1, k_1+1, k_1+2, \dots, k_1+n-2, k_1+n)\), where \(k_1\) can be one of the \(n\) \mbox{numbers}: \(-n+2, -n+3, \dots, 0, 1\), if there exist two sets of integers \(\{a_1, a_2, \dots, a_{n+1}\}\) and \(\{b_1, b_2, \dots, b_{n+1}\}\) satisfying:
\begin{equation}
\begin{aligned}
[a_1, a_2, \dots, a_{n+1}]^k &= [b_1, b_2, \dots, b_{n+1}]^k, \\
&\quad (k = k_1, k_1+1, k_1+2, \dots, k_1+n-2, k_1+n),
\end{aligned}
\end{equation}
then the following identity holds:
\begin{align}
\label{equationC3}
\frac{1}{x^m (s - x)} \prod_{i=1}^{n+1} (x - a_i) - \frac{1}{x^m (s - x)} \prod_{i=1}^{n+1} (x - b_i) = C,
\end{align}
where 
\begin{align}
s &= \frac{1}{2} \sum_{i=1}^{n+1} (a_i + b_i),
\end{align}
\( C \) is a non-zero constant, and \(m\) is determined by:
\begin{equation}
m =
\begin{cases}
0, & \textup{for } (k = 1, 2, 3, \dots, n-1, n+1), \\
1, & \textup{for } (k = 0, 1, 2, \dots, n-2, n), \\
2, & \textup{for } (k = -1, 0, 1, \dots, n-3, n-1),\\
\vdots \\
n-1, & \textup{for } (k = -n+2, -n+3, -n+4, \dots,0,2).
\end{cases}
\end{equation}
\end{identity}

The proof of \refIdentity{identityC3} can be derived using the first generalization of the Girard-Newton Identities (Identity \ref{identity_GNI1}). \\[1mm]
\indent
When we substitute \( x = a_i \neq 0 \) and \( x = b_i \neq 0 \) into \eqref{equationC3} accordingly, we obtain:
\begin{corollary}\label{corollaryC3}
In \refIdentity{identityC3}, when \( a_i \neq 0 \) and \( b_i \neq 0 \), the following relationships hold:
\begin{align}
\label{GPTEc1x}
-C & =\frac{\prod_{i=1}^{n+1} (a_1 - b_i)}{a_1^m (s - a_1)} = \frac{\prod_{i=1}^{n+1} (a_2 - b_i)}{a_2^m (s - a_2)} = \dots = \frac{\prod_{i=1}^{n+1} (a_{n+1} - b_i)}{a_{n+1}^m (s - a_{n+1})},\\
C &= \frac{\prod_{i=1}^{n+1} (b_1 - a_i)}{b_1^m (s - b_1)} = \frac{\prod_{i=1}^{n+1} (b_2 - a_i)}{b_2^m (s - b_2)} = \dots = \frac{\prod_{i=1}^{n+1} (b_{n+1} - a_i)}{b_{n+1}^m (s - b_{n+1})}.
\end{align}
\end{corollary}
Using \refIdentity{identity_GNI1}, we can derive the formula for the constant \( C \) of the GPTE types related to \refIdentity{identityC3}. The following example shows the result for \( n = 4 \).

\begin{example}
\label{examplek1235}
For the type \((k = k_1, k_1+1, k_1+2, k_1+4)\), where \(k_1\) can be one of the numbers: \(1, 0, -1, -2 \), if there exist two sets of integers \(\{a_1, a_2, a_3, a_4, a_5\}\) and \(\{b_1, b_2, b_3, b_4, b_5\}\) satisfying:
\begin{align}
[a_1, a_2, a_3, a_4, a_5]^k = [b_1, b_2, b_3, b_4, b_5]^k, \quad (k=k_1, k_1+1, k_1+2, k_1+4). \nonumber
\end{align}
then the constant \( C \) is given by:
\begin{align}
& C = -a_1 a_2 a_3 a_4 - a_1 a_2 a_3 a_5 - \dots - a_2 a_3 a_4 a_5 \nonumber \\
& \qquad + b_1 b_2 b_3 b_4 + b_1 b_2 b_3 b_5 + \dots + b_2 b_3 b_4 b_5, \quad  \textup{for}\,\, (k=1,2,3,5)\\
& C = a_1 a_2 a_3 + a_1 a_2 a_4 + \dots + a_3 a_4 a_5 \nonumber \\
& \qquad - b_1 b_2 b_3 - b_1 b_2 b_4  - \dots -  b_3 b_4 b_5, \quad\:\: \textup{for}\,\, (k=0,1,2,4)\\
& C = -a_1 a_2 - a_1 a_3 - \dots - a_4 a_5 \nonumber \\
& \qquad +b_1 b_2 + b_1 b_3  + \dots + b_4 b_5,  \qquad\qquad \textup{for}\,\, (k=-1,0,1,3)\\
& C = a_1 + a_2 + a_3 + a_4 + a_5 \nonumber \\
& \qquad  -b_1 - b_2 - b_3 - b_4 - b_5.  \quad\qquad\qquad \textup{for}\,\, (k=-2,-1,0,2)
\end{align}
\end{example}
For various numerical examples, in addition to applying formulas similar to that in Example\,\ref{examplek1235}, we can also directly use \eqref{equationC3} to compute the values of \( C \). For instance:
\begin{align}
& [ 1, 14, 17, 46, 48, 60 ]^k = [ 4, 6, 25, 38, 56, 57 ]^k, \quad (k=1,2,3,4,6),  \tag*{\eqref{k12346s60}}\\
& \qquad \textup{it follows that} \quad C =221760=2^6 \cdot 3^2 \cdot 5 \cdot 7 \cdot 11,  \nonumber\\[1mm]
& [ 28, 57, 100, 174, 200, 245]^k = [ 30, 49, 125, 140, 228, 232 ]^k, \nonumber\\
& \qquad\qquad\qquad\qquad\qquad\qquad (k=0,1,2,3,5),  \tag*{\eqref{k01235s245}}\\
& \qquad \textup{it follows that} \quad C =856800=2^5 \cdot 3^2 \cdot 5^2 \cdot 7 \cdot 17,  \nonumber\\[1mm]
& [ 10, 14, 24, 65, 117, 132]^k = [ 11, 12, 26, 63, 120, 130 ]^k, \nonumber\\
& \qquad\qquad\qquad\qquad\qquad\qquad (k=-1,0,1,2,4),  \tag*{\eqref{k0124n1s132}}\\
& \qquad \textup{it follows that} \quad C =636=2^2 \cdot 3 \cdot 53,  \nonumber\\[1mm]
& [2, 31, 33, 84, 89, 136, 151]^k = [4, 19, 51, 66, 101, 133, 152]^k, \nonumber\\
& \qquad\qquad\qquad\qquad\qquad\qquad (k=1,2,3,4,5,7),  \tag*{\eqref{k123457s152}}\\
& \qquad \textup{it follows that} \quad C =395841600=2^6 \cdot 3^3 \cdot 5^2 \cdot 7^2 \cdot 11 \cdot 17,  \nonumber\\[1mm]
& [32,40,41,69,72,88,90]^k = [33,36,45,64,80,82,92]^k, \nonumber\\
& \qquad\qquad\qquad\qquad\qquad\qquad (k=0,1,2,3,4,6),  \tag*{\eqref{k012346s92}}\\
& \qquad \textup{it follows that} \quad C =18720=2^5 \cdot 3^2 \cdot 5 \cdot 13.  \nonumber
\end{align}

\begin{identity}\label{identityC4}
Let \( n \) be a positive integer greater than \(1\). For the type \((k = k_1, k_1+2, k_1+4, \dots, k_1+2n-2)\), where \(k_1\) can be one of the \(n+1\) \mbox{numbers}: \(-2n+1, -2n+3, \dots, -1, 1\), if there exist two sets of integers \(\{a_1, a_2, \dots, a_{n+1}\}\) and \(\{b_1, b_2, \dots, b_{n+1}\}\) satisfying:
\begin{equation}
\begin{aligned}
[a_1, a_2, \dots, a_{n+1}]^k &= [b_1, b_2, \dots, b_{n+1}]^k, \\
&\quad (k = k_1, k_1+2, k_1+4, \dots, k_1+2n-2),
\end{aligned}
\end{equation}
then the following identity holds:
\begin{align}
\label{equationC4}
\frac{1}{2 x^m } \prod_{i=1}^{n+1} (x + a_i)(x - b_i) - \frac{1}{2 x^m } \prod_{i=1}^{n+1}(x - a_i)(x + b_i) = C,
\end{align}
where \( C \) is a non-zero constant, and \(m\) is determined by:
\begin{equation}
m =
\begin{cases}
1, & \textup{for } (k = 1, 3, 5, \dots, 2n-1), \\
3, & \textup{for } (k = -1, 1, 3, \dots,2n-3), \\
5, & \textup{for } (k = -3, -1, 1, \dots, 2n-5),\\
\vdots \\
2n+1, & \textup{for } (k = -2n+1, -2n+3, -2n+5, \dots,-1).
\end{cases}
\end{equation}
\end{identity}

The proof of \refIdentity{identityC4} can be derived using the Third Generalization of the Girard-Newton Identities (Identity \ref{identity_GNI3}). \\[1mm]
\indent When we substitute \( x = a_j \neq 0 \) and \( x = b_j \neq 0 \) into \eqref{equationC4} accordingly, we obtain:
\begin{corollary}\label{corollaryC4}
In \refIdentity{identityC4}, when \( a_j \neq 0 \) and \( b_j \neq 0 \), the following relationships hold:
\begin{align}
\label{corollaryC4eq}
& C = \frac{1}{2 a_j^m} \prod _{i=1}^{n+1} (a_j+a_i)(a_j-b_i)=-\frac{1}{2 b_j^m} \prod_{i=1}^{n+1}(b_j - a_i) (b_j + b_i), \nonumber \\
& \qquad\qquad\qquad\qquad\qquad\qquad\qquad\qquad \textup{for}\; j=1,2,\dots,n+1. 
\end{align}
\end{corollary}
When $a_i=a_j$ (or $b_i=b_j$), it follows that $a_i+a_j=2a_j$ (or $b_i+b_j=2b_j $), so the above \eqref{corollaryC4eq} can be further simplified. For example:
\begin{example}
If two sets of integers \(\{a_1, a_2, a_3, a_4, a_5\}\) and \(\{b_1, b_2, b_3, b_4, b_5\}\) \mbox{satisfy}:
\begin{align}
[a_1, a_2, a_3, a_4, a_5]^k &= [b_1, b_2, b_3, b_4, b_5]^k, \quad (k = 1, 3, 5, 7),
\end{align}
then the constant \( C \) is equal to:
{\small
\begin{equation}
\begin{aligned}
\label{C1357eq}
C & = (a_1+a_2) (a_1+a_3) (a_1+a_4) (a_1+a_5) (a_1-b_1) (a_1-b_2) (a_1-b_3) (a_1-b_4) (a_1-b_5)\\
& =(a_1+a_2) (a_2+a_3) (a_2+a_4) (a_2+a_5) (a_2-b_1) (a_2-b_2) (a_2-b_3) (a_2-b_4) (a_2-b_5) \\
& =(a_1+a_3) (a_2+a_3) (a_3+a_4) (a_3+a_5) (a_3-b_1) (a_3-b_2) (a_3-b_3) (a_3-b_4) (a_3-b_5)\\
& =(a_1+a_4) (a_2+a_4) (a_3+a_4) (a_4+a_5) (a_4-b_1) (a_4-b_2) (a_4-b_3) (a_4-b_4) (a_4-b_5)\\
&= (a_1+a_5) (a_2+a_5) (a_3+a_5) (a_4+a_5) (a_5-b_1) (a_5-b_2) (a_5-b_3) (a_5-b_4) (a_5-b_5)\\
& =-(b_1-a_1) (b_1-a_2) (b_1-a_3) (b_1-a_4) (b_1-a_5)(b_1+b_2) (b_1+b_3) (b_1+b_4) (b_1+b_5)\\  
& =-(b_2-a_1) (b_2-a_2) (b_2-a_3) (b_2-a_4) (b_2-a_5)(b_1+b_2) (b_2+b_3) (b_2+b_4) (b_2+b_5) \\
& =-(b_3-a_1) (b_3-a_2) (b_3-a_3) (b_3-a_4) (b_3-a_5)(b_1+b_3) (b_2+b_3) (b_3+b_4) (b_3+b_5)\\ 
& =-(b_4-a_1) (b_4-a_2) (b_4-a_3) (b_4-a_4) (b_4-a_5)(b_1+b_4) (b_2+b_4) (b_3+b_4) (b_4+b_5)\\ 
& =-(b_5-a_1) (b_5-a_2) (b_5-a_3) (b_5-a_4) (b_5-a_5)(b_1+b_5) (b_2+b_5) (b_3+b_5) (b_4+b_5) 
\end{aligned}
\end{equation}
}
\end{example}
The above relationship \eqref{C1357eq} for the constant $C$ is highly useful for \mbox{computer} search of numerical solutions for type $(k=1,3,5,7)$, as it can significantly \mbox{accelerate} the search process.\\[1mm]
\indent Using \refIdentity{identity_GNI1}, we can derive the formula for the constant \( C \) of the GPTE types related to \refIdentity{identityC4}. The following example shows the result for \( n = 4 \).
\begin{example}
\label{examplek1357}
For the type \((k = k_1, k_1+2, k_1+4, k_1+6)\), where \(k_1\) can be one of the numbers: \(1, -1, -3, -5, -7 \), if there exist two sets of integers \(\{a_1, a_2, a_3, a_4, a_5\}\) and \(\{b_1, b_2, b_3, b_4, b_5\}\) satisfying:
\begin{align}
[a_1, a_2, a_3, a_4, a_5]^k = [b_1, b_2, b_3, b_4, b_5]^k, \quad (k=k_1, k_1+2, k_1+4, k_1+6). \nonumber
\end{align}
then the constant \( C \) is given by:
\begin{flalign}
&\quad C = U_5 V_4-V_5 U_4 , \qquad\qquad\qquad\qquad\qquad\quad \textup{for}\,\, (k=1,3,5,7) &\\
&\quad C = U_5 V_2-V_5 U_2 +U_3 V_4-V_3 U_4 , \qquad\qquad\,\, \textup{for}\,\, (k=-1,1,3,5) &\\
&\quad C = U_5-V_5+U_3 V_2-V_3 U_2 +U_1 V_4-V_1 U_4, \quad \textup{for}\,\, (k=-3,-1,1,3) &\\
&\quad C = U_3-V_3 +U_1 V_2 -V_1 U_2, \qquad\qquad\qquad\,\, \textup{for}\,\, (k=-5,-3,-1,1) &\\
&\quad C = U_1-V_1, \qquad\qquad\qquad\qquad\qquad\qquad\quad \textup{for}\,\, (k=-7,-5,-3,-1) &
\end{flalign}
where
\begin{align}
& U_1=a_1+a_2+a_3+a_4+a_5,\\
& V_1\,=b_1+b_2\,+b_3+b_4\,+b_5,\\ 
& U_2=a_1 a_2+a_1 a_3+\dots+a_4 a_5,\\
& V_2\,=b_1 b_2\,+b_1 b_3\,+\dots+b_4 b_5,\\
& U_3=a_1 a_2 a_3+a_1 a_2 a_4+\dots+a_3 a_4 a_5,\\
& V_3\,=b_1 b_2 b_3\:+b_1 b_2 b_4\:\,+\dots+b_3 b_4 b_5,\\
& U_4=a_1 a_2 a_3 a_4+a_1 a_2 a_3 a_5+\dots+a_2 a_3 a_4 a_5,\\
& V_4\,=b_1 b_2 b_3 b_4\:\,+b_1 b_2 b_3 b_5\:\,+\dots+b_2 b_3 b_4 b_5,\\
& U_5=a_1 a_2 a_3 a_4 a_5,\\
& V_5=b_1 b_2 b_3 b_4 b_5.
\end{align}
\end{example}

For various numerical examples, in addition to applying formulas similar to that in Example\,\ref{examplek1357}, we can also directly use \eqref{equationC4} to compute the values of \( C \). For instance:
\begin{align}
& [1, 13, 17, 23 ]^k = [ 3, 9, 21, 21 ]^k, \quad (k=1,3,5),  \tag*{\eqref{k135s23}}\\
& \qquad \textup{it follows that} \quad C =38707200=2^{13} \cdot 3^3 \cdot 5 \cdot 2 \cdot 7.  \nonumber\\[1mm]
& [3, 10, 15, 30]^k = [4, 5, 21, 28 ]^k, \quad (k=-1,1,3),  \tag*{\eqref{k13n1s30}}\\
& \qquad \textup{it follows that} \quad C =772200=2^3 \cdot 3^3 \cdot 5^2 \cdot 11 \cdot 13.  \nonumber\\[1mm]
& [14, 28, 42, 140]^k = [15, 20, 84, 105 ]^k, \qquad (k=-3,-1,1),  \tag*{\eqref{k1n13s140}}\\
& \qquad \textup{it follows that} \quad C =360360=2^3 \cdot 3^2 \cdot 5 \cdot 7 \cdot 11 \cdot 13.  \nonumber\\[1mm]
& [12558, 16744, 18837, 50232]^k = [13104, 14352, 23184, 43056]^k, \nonumber\\
& \qquad\qquad\qquad\qquad\qquad\qquad (k=-5,-3,-1),  \tag*{\eqref{kn135s50232}}\\
& \qquad \textup{it follows that} \quad C =4675=5^2 \cdot 11 \cdot 17.  \nonumber \\
& [3, 19, 37, 51, 53]^k = [9, 11, 43, 45, 55]^k, \quad (k=1,3,5,7),  \tag*{\eqref{k1357s55}}\\
& \qquad \textup{it follows that} \quad C =22317642547200=2^{19} \cdot 3^5 \cdot 5^2 \cdot 7^2 \cdot 11 \cdot 13.  \nonumber\\[1mm]
& [245126961, 299599619, 313534485, 1225634805, 1497998095]^k \nonumber\\
&\quad  = [254377035, 264352605, 364377915, 709578045, 4493994285]^k \nonumber\\
& \qquad\qquad\qquad\qquad\qquad\qquad (k=-7,-5,-3,-1), \tag*{\eqref{kn1357s4493994285}}\\
& \qquad \textup{it follows that} \quad C =2504785920=2^{18} \cdot 3 \cdot 5 \cdot 7^2 \cdot 13.  \nonumber\\[1mm]
& [7, 91, 173, 269, 289, 323 ]^k = [ 29, 59, 193, 247, 311, 313]^k,\nonumber\\ 
& \qquad\qquad\qquad\qquad\qquad\qquad (k=1,3,5,7,9),  \tag*{\eqref{k13579s323}}\\
& \qquad \textup{it follows that} \quad C =2^{22} \cdot 3^8 \cdot 5^3 \cdot 7^2 \cdot 11^2 \cdot 13 \cdot 17\cdot 19\cdot 23\cdot 31\cdot 37.  \nonumber
\end{align}

In this chapter, through the analysis of the constant \( C \) for three series of GPTE types in \refIdentity{identityC2} to \refIdentity{identityC4},  we observe the significant links between the ideal non-negative integer solutions of each series of the GPTE problem. \mbox{Meanwhile}, Corollary \ref{corollaryC1} to Corollary \ref{corollaryC4} demonstrate the connections between the constant \( C \) and \( a_i - b_j \), which are also very useful for improving the efficiency of computer search for ideal solutions. (See Chapter 6) \\[1mm]
\indent Furthermore, in \refIdentity{identityC2} and \refIdentity{identityC3}, through the analysis of the constant \( C \) for two series of GPTE types, the rationality of the definition of the GPTE problem proposed in Chapter 1 is further verified. Specifically, when \( k = 0 \), the problem reduces to equal products, as shown in the following definition:
\begin{equation}
\left[ a_{1}, a_{2}, \dots, a_{m} \right]^{k} :=
\begin{cases}
a_{1}^{k} + a_{2}^{k} + \cdots + a_{m}^{k}, & \text{if } k \neq 0, \\
a_{1} a_{2} \cdots a_{m}, & \text{if } k = 0.
\end{cases}
\tag*{\eqref{EPES}}
\end{equation}
In addition, \refIdentity{identity_GNI1} and \refIdentity{identity_GNI2} in Chapter 2 also fully validate the \mbox{rationality} of the definition \eqref{EPES}. For further discussions on why \( k = 0 \) is used for the equal products equation, see the author's webpage \cite{ChenkProducts23}.

\clearpage

\section{The Generalized PTE Problem of Trigonometric Form}
\subsection{Trigonometric Solutions of PTE and GPTE}
Research on the PTE Problem has traditionally been conducted within the scope of integers. In 2021, the author discovered a series of identities \cite{Chen21}, thereby extending the exploration of ideal solutions for the PTE Problem and the GPTE Problem into the realm of trigonometric functions.

\subsubsection{Ideal Trigonometric Solution of PTE}
In 2021, we discovered for the first time a trigonometric function identity pertaining to the PTE problem, which is as follows \cite{Chen21}.

\begin{identity}
\label{identityT1}
If two sets of non-negative real numbers $\{a_{1}, a_{2}, \cdots, a_{n+1}\}$ and \\ $\{b_{1}, b_{2}, \cdots, b_{n+1}\}$ satisfy:
\begin{align}
\label{GPTEreal}
\left[ a_{1}, a_{2}, \cdots, a_{n+1} \right]^{k} = \left[ b_{1}, b_{2}, \cdots, b_{n+1} \right]^{k}, \quad (k = 1, 2, \dots, n),
\end{align}
without loss of generality, let $a_{n+1}$ be the maximum element in $\{a_{1}, a_{2}, \cdots, a_{n+1}\}$, $b_{n+1}$ be the maximum element in $\{b_{1}, b_{2}, \cdots, b_{n+1}\}$, with $a_{n+1} > b_{n+1}$, and denote the minimum value of $b_{n+1}$ as $b_{n+1, \min}$. Then we can conclude that:
\begin{align}
\label{PTE_bmin}
\frac{b_{n+1, \min}}{a_{n+1}} = \sin^{2}\left(\frac{n \pi}{2n + 2}\right).
\end{align}
Alternatively stated,
\begin{align}
\label{bnp1}
\frac{b_{n+1}}{a_{n+1}} \geq \sin^{2}\left(\frac{n\pi}{2n + 2}\right).
\end{align}
\end{identity}

For the above conclusion, we have not yet provided a complete proof, but it has been verified through millions of numerical solutions without finding any counterexamples. For example, in the following case:
\begin{align}
& [0, 664, 701, 1787, 1788]^k = [188, 189, 1275, 1312, 1976]^k, \nonumber \\
& \qquad \qquad \qquad \qquad \qquad (k = 1, 2, 3, 4),
\end{align}
it follows that
\begin{align}
\frac{1788}{1976} = \sin^{2}\left(\frac{4\pi}{10}\right) \times 1.0003867\ldots.
\end{align}
It is evident that, under the assumption that \eqref{PTE_bmin} or \eqref{bnp1} holds, the efficiency of searching for ideal non-negative integer solutions to the PTE problem using computers can be effectively improved. We further obtain more precise identities as follows.

\begin{identity}
\label{identityT2}
For a positive odd integer \( n \), we have
\begin{align}
\label{t1357}
& 2 \sum _{i=1}^{\frac{n-1}{2}} \sin ^{2 k}\left(\frac{2i}{2 n+2} \pi\right) + 1 = 2 \sum _{i=1}^{\frac{n+1}{2}} \sin ^{2 k}\left(\frac{2 i-1}{2 n+2}\pi \right) \nonumber \\
& \qquad\qquad\qquad\qquad\qquad (k = 1, 2, 3, \dots, n).
\end{align}
For a positive even integer \( n \), we have
\begin{align}
\label{t2468}
& 2 \sum _{i=1}^{\frac{n}{2}} \sin ^{2 k}\left(\frac{2i}{2 n+2} \pi\right) = 2 \sum _{i=1}^{\frac{n}{2}} \sin ^{2 k}\left(\frac{2 i-1}{2 n+2}\pi\right) + 1 \nonumber \\
& \qquad\qquad\qquad\qquad\qquad (k = 1, 2, 3, \dots, n).
\end{align}
\end{identity}

The author is currently unable to provide a complete proof of these identities. However, for small values of \( n \), such as \( n \leq 20 \), the proof can be verified using the mathematical software Mathematica. Below are some specific examples. It is worth noting that the solutions obtained are all ideal trigonometric solutions, meaning that in the cases where \( (k = 1, 2, \dots, n) \) or \( (k = k_1, k_2, \dots, k_n) \), the number of elements in each group of arrays is \( n + 1 \). For consistency in notation, we extend the definition of \( \{a_i\} \) in \eqref{EPES} from integers to real numbers.

\begin{example}
Applying identities \eqref{t1357} and \eqref{t2468}, when $n$ takes the values of $1,2,3,4,5,6$, we can obtain:
\begin{align}
& \big[\sin ^2(\frac{0 \pi }{4}),\sin ^2(\frac{2 \pi }{4})\big]^k =\big[\sin ^2(\frac{\pi }{4}),\sin ^2(\frac{\pi }{4})\big]^k, \qquad (k=1)\\[2mm]
& \big[\sin ^2(\frac{0 \pi }{6}),\sin ^2(\frac{2 \pi }{6}),\sin ^2(\frac{2 \pi }{6})\big]^k =\big[\sin ^2(\frac{\pi }{6}),\sin ^2(\frac{\pi }{6}),\sin ^2(\frac{3 \pi }{6})\big]^k, \nonumber\\
& \qquad\qquad\qquad\qquad\qquad\qquad\qquad\qquad\qquad\qquad\quad\:\: (k=1,2)\\[2mm]
\label{t123}
& \big[\sin ^2(\frac{0 \pi }{8}),\sin ^2(\frac{2 \pi }{8}),\sin ^2(\frac{2 \pi }{8}),\sin ^2(\frac{4 \pi }{8})\big]^k \nonumber \\
& \quad =\big[\sin ^2(\frac{\pi }{8}),\sin ^2(\frac{\pi }{8}),\sin ^2(\frac{3 \pi }{8}),\sin ^2(\frac{3 \pi }{8})\big]^k, \qquad (k=1,2,3)\\[2mm]
\label{t1234}
& \big[\sin ^2(\frac{0 \pi }{10}),\sin ^2(\frac{2 \pi }{10}),\sin ^2(\frac{2 \pi }{10}),\sin ^2(\frac{4 \pi }{10}),\sin ^2(\frac{4 \pi }{10})\big]^k \nonumber \\
& \quad =\big[\sin ^2(\frac{\pi }{10}),\sin ^2(\frac{\pi }{10}),\sin ^2(\frac{3 \pi }{10}),\sin ^2(\frac{3 \pi }{10}),\sin ^2(\frac{5 \pi }{10})\big]^k, \nonumber\\
& \qquad\qquad\qquad\qquad\qquad\qquad\qquad (k=1,2,3,4)\\[2mm]
\label{t12345}
& \big[\sin ^2(\frac{0 \pi }{12}),\sin ^2(\frac{2 \pi }{12}),\sin ^2(\frac{2 \pi }{12}),\sin ^2(\frac{4 \pi }{12}),\sin ^2(\frac{4 \pi }{12}),\sin ^2(\frac{6 \pi }{12})\big]^k \nonumber \\
& \quad =\big[\sin ^2(\frac{\pi }{12}),\sin ^2(\frac{\pi }{12}),\sin ^2(\frac{3 \pi }{12}),\sin ^2(\frac{3 \pi }{12}),\sin ^2(\frac{5 \pi }{12}),\sin ^2(\frac{5 \pi }{12})\big]^k, \nonumber\\
& \qquad\qquad\qquad\qquad\qquad\qquad\qquad (k=1,2,3,4,5)\\[2mm]
\label{t123456}
& \big[\sin ^2(\frac{0 \pi }{14}),\sin ^2(\frac{2 \pi }{14}),\sin ^2(\frac{2 \pi }{14}),\sin ^2(\frac{4 \pi }{14}),\sin ^2(\frac{4 \pi }{14}),\sin ^2(\frac{6 \pi }{14}),\sin ^2(\frac{6 \pi }{14})\big]^k \nonumber\\
& \quad =\big[\sin ^2(\frac{\pi }{14}),\sin ^2(\frac{\pi }{14}),\sin ^2(\frac{3 \pi }{14}),\sin ^2(\frac{3 \pi }{14}),\sin ^2(\frac{5 \pi }{14}),\sin ^2(\frac{5 \pi }{14}),\sin ^2(\frac{7 \pi }{14})\big]^k, \nonumber\\
& \qquad\qquad\qquad\qquad\qquad\qquad\qquad (k=1,2,3,4,5,6)
\end{align}
\end{example}
Using the fundamental properties of trigonometric functions, it is easy to prove that \eqref{t123456} can also be written in the following two forms:
\begin{align}
& \big[\cos ^2(\frac{0 \pi }{14}),\cos ^2(\frac{2 \pi }{14}),\cos ^2(\frac{2 \pi }{14}),\cos ^2(\frac{4 \pi }{14}),\cos ^2(\frac{4 \pi }{14}),\cos ^2(\frac{6 \pi }{14}),\cos ^2(\frac{6 \pi }{14})\big]^k \nonumber \\
& \: =\big[\cos ^2(\frac{\pi }{14}),\cos ^2(\frac{\pi }{14}),\cos ^2(\frac{3 \pi }{14}),\cos ^2(\frac{3 \pi }{14}),\cos ^2(\frac{5 \pi }{14}),\cos ^2(\frac{5 \pi }{14}),\cos ^2(\frac{7 \pi }{14})\big]^k, \nonumber\\
& \qquad\qquad\qquad\qquad\qquad\qquad\qquad (k=1,2,3,4,5,6)\\[2mm]
& \big[\cos (\frac{0 \pi }{7}),\cos (\frac{2 \pi }{7}),\cos (\frac{2 \pi }{7}),\cos (\frac{4 \pi }{7}),\cos (\frac{4 \pi }{7}),\cos (\frac{6 \pi }{7}),\cos (\frac{6 \pi }{7})\big]^k \nonumber \\
& \quad =\big[\cos (\frac{\pi }{7}),\cos (\frac{\pi }{7}),\cos (\frac{3 \pi }{7}),\cos (\frac{3 \pi }{7}),\cos (\frac{5 \pi }{7}),\cos (\frac{5 \pi }{7}),\cos (\frac{7 \pi }{7})\big]^k, \nonumber\\
& \qquad\qquad\qquad\qquad\qquad\qquad\qquad (k=1,2,3,4,5,6)
\end{align}

\begin{example}
By applying the identity \eqref{t2468} with $n=10$, and denote
\begin{align}
g(x)=\sin^2\left(\frac{x}{22}\pi\right)
\end{align}
then we can obtain:
\begin{align}
\label{t12345678910}
& \big[\ g(0),\ g(2),\ g(2),\ g(4),\ g(4),\ g(6),\ g(6),\ g(8),\ g(8),\ g(10),\ g(10) \big]^{k} & \nonumber\\
& \quad =\big[\ g(1),\ g(1),\ g(3),\ g(3),\ g(5),\ g(5),\ g(7),\ g(7),\ g(9),\ g(9),\ g(11) \big]^{k} & \nonumber\\
& \qquad\qquad\qquad\qquad\qquad (k=1,2,3,4,5,6,7,8,9,10)
\end{align}
\end{example}
Using the properties of trigonometric functions, the following two formulas can be easily derived from \eqref{t12345678910}:
\begin{align}
& \big[\cos (\frac{0 \pi }{11}),\ \cos (\frac{2 \pi }{11}),\ \cos (\frac{2 \pi }{11}),\ \cos (\frac{4 \pi }{11}),\ \cos (\frac{4 \pi }{11}),\ 0\big]^k \nonumber \\
& \quad =\big[\cos (\frac{1\pi }{11}),\ \cos (\frac{1\pi }{11}),\ \cos (\frac{3 \pi }{11}),\ \cos (\frac{3 \pi }{11}),\ \cos (\frac{5 \pi }{11}),\ \cos (\frac{5 \pi }{11})\big]^k, \nonumber\\
& \qquad\qquad\qquad\qquad\qquad\qquad\qquad (k=1,3,5,7,9) \\[2mm]
& \big[\cos (\frac{0 \pi }{11}),\ \cos (\frac{2 \pi }{11}),\ \cos (\frac{4 \pi }{11}),\ \cos (\frac{6 \pi }{11}),\ \cos (\frac{8 \pi }{11}),\ \cos (\frac{10\pi }{11})\big]^k \nonumber \\
& \quad =\big[\cos (\frac{1 \pi }{11}),\ \cos (\frac{3\pi }{11}),\ \cos (\frac{5 \pi }{11}),\ \cos (\frac{7 \pi }{11}),\ \cos (\frac{9 \pi }{11}),\ 0\big]^k, \nonumber\\
& \qquad\qquad\qquad\qquad\qquad\qquad\qquad (k=1,3,5,7,9)
\end{align}

In the following four examples, we will demonstrate the connection between the trigonometric solutions and integer solutions of PTE.
\begin{example}
Starting from \eqref{t123}, we derive the following values:
\begin{equation}
\begin{aligned}
a_1=\phantom{0000000000000000000000000000000000000}0, \\
a_2=\phantom{0}8127503623352124799931806454985399361,\\
a_3=\phantom{0}8127503623352124799931806454985399362,\\
a_4=16255007246704249599863612909970798723,\\
b_1=\phantom{0}2380490697161601722646526306111946640, \\
b_2=\phantom{0}2380490697161601722646526306111946641, \\
b_3=13874516549542647877217086603858852082, \\
b_4=13874516549542647877217086603858852083.
\end{aligned}
\end{equation}
These values satisfy the relationship:
\begin{equation}
\label{k123tr}
\begin{aligned}
[a_1,a_2,a_3,a_4]^k=[b_1,b_2,b_3,b_4]^k,\quad (k=1,2,3)
\end{aligned}
\end{equation}
Furthermore, the following approximate relationships hold:
\begin{align}
\frac{a_2}{a_4}=\sin ^2(\frac{2\pi }{8}) \times 0.9999999999999999999999999999999999999384...\nonumber\\[1mm]
\frac{a_3}{a_4}=\sin ^2(\frac{2\pi }{8}) \times 1.0000000000000000000000000000000000000615...\nonumber\\[1mm]
\frac{b_1}{a_4}=\sin ^2(\frac{1\pi }{8}) \times 0.9999999999999999999999999999999999997899...\nonumber\\[1mm]
\frac{b_2}{a_4}=\sin ^2(\frac{1\pi }{8}) \times 1.0000000000000000000000000000000000002100...\nonumber\\[1mm]
\frac{b_3}{a_4}=\sin ^2(\frac{3 \pi }{8}) \times 0.9999999999999999999999999999999999999639...\nonumber\\[1mm]
\frac{b_4}{a_4}=\sin ^2(\frac{3 \pi }{8}) \times 1.0000000000000000000000000000000000000360...\nonumber
\end{align}
\end{example}
Since $a_1+a_4=a_2+a_3=b_1+b_4=b_2+b_3$, \eqref{k123tr} is a symmetric solution. Similarly, \eqref{t123} is also a symmetric solution.
\begin{example}
Starting from \eqref{t123}, we derive the following values:
\begin{equation}
\begin{aligned}
c_1=\phantom{000000000000000000000000000000000}0, \\
c_2=c_3=2490042132448610714806132969640263,\\
c_4=4980084264897221429612265939280525,\\
d_1=\phantom{0}729316455153986719204462547867488, \\
d_2=\phantom{0}729316455153986778546280472407413, \\
d_3=d_4=4250767809743234680736894429143075.
\end{aligned}
\end{equation}
These values satisfy the relationship:
\begin{equation}
\label{k123tr2}
\begin{aligned}
[c_1,c_2,c_3,c_4]^k=[d_1,d_2,d_3,d_4]^k,\quad (k=1,2,3)
\end{aligned}
\end{equation}
Furthermore, the following approximate relationships hold:
\begin{equation}
\begin{aligned}
\frac{c_2}{c_4}=\frac{c_3}{c_4}=& \sin ^2(\frac{2\pi }{8}) \times 1.0000000000000000000000000000000002008...\nonumber\\[1mm]
\frac{d_1}{c_4}=& \sin ^2(\frac{1\pi }{8}) \times 0.99999999999999995932...\nonumber\\[1mm]
\frac{d_2}{c_4}=& \sin ^2(\frac{1\pi }{8}) \times 1.00000000000000004068...\nonumber\\[1mm]
\frac{d_3}{c_4}=\frac{d_4}{c_4}=& \sin ^2(\frac{3 \pi }{8}) \times 1.0000000000000000000000000000000000294...\nonumber
\end{aligned}
\end{equation}
\end{example}
Since $ d_1+d_4 \neq d_2+d_3$, \eqref{k123tr2} is a non-symmetric solution. This contrasts with the previous example of a symmetric solution, \eqref{k123tr}, even though both originate from the symmetric trigonometric solution \eqref{t123}.\\[1mm]
\indent Additionally, we let
\begin{equation}
\begin{aligned}
\label{h1235tr2}
& [c_1 + x, c_2 + x, c_3 + x, c_4 + x]^h = [d_1 + x, d_2 + x, d_3 + x, d_4 + x]^h, \\
& \qquad\qquad\qquad\qquad\qquad\qquad\qquad\qquad\qquad\qquad (h = 1, 2, 3, 5)
\end{aligned}
\end{equation}
By solving for \( x \), we obtain \( x = -(9960168529794442859224531878561051/4) \). Substituting this value into \eqref{h1235tr2}, we derive the following values:
\begin{equation}
\begin{aligned}
a_1 &= -9960168529794442859224531878561051, \\
a_2 = a_3 & = \phantom{-000000000000000000000000000000000}1, \\
a_4 &= \phantom{-}9960168529794442859224531878561049, \\
b_1 &= -7042902709178495982406681687091099, \\
b_2 &= -7042902709178495745039409988931399, \\
b_3 = b_4 & = \phantom{-}7042902709178495863723045838011249.
\end{aligned}
\end{equation}
These values satisfy the relationship:
\begin{equation}
\label{h1235tr2b}
\begin{aligned}
[a_1, a_2, a_3, a_4]^h = [b_1, b_2, b_3, b_4]^h, \quad (h = 1, 2, 3, 5)
\end{aligned}
\end{equation}
Also the following approximate relationships hold:
\begin{equation}
\begin{aligned}
\frac{a_1}{a_4} &= -1.000000000000000000000000000000000200799815\ldots \\
\frac{a_2}{a_4} = \frac{a_3}{a_4} & = \phantom{-}0.000000000000000000000000000000000100003999\ldots \\
\frac{b_1}{a_4} &= -\sin\left(\frac{\pi}{4}\right) \times 1.00000000000000001685\ldots \\
\frac{b_2}{a_4} &= -\sin\left(\frac{\pi}{4}\right) \times 0.99999999999999998314\ldots \\
\frac{b_3}{a_4} = \frac{b_4}{a_4} & = \phantom{-} \sin\left(\frac{\pi}{4}\right) \times 1.0000000000000000000000000000000000294\ldots
\end{aligned}
\end{equation}

\begin{example}
Starting from \eqref{t1234}, we derive the following values:
\begin{equation}
\begin{aligned}
a_1=\phantom{0}86701260586398906579842518642185971022150,\\
a_2=\phantom{0}86701260586398906579842518642185971022151,\\
a_3=594259280661563023216102935903929957828588,\\
a_4=594259280661563023216102935903929957828589,\\
a_5=907947388330615906394593939394821238467652,\\
b_1=\phantom{00000000000000000000000000000000000000000}0,\\
b_2=313688107669052883178491003490891280639063,\\
b_3=313688107669052883178491003490891280639064,\\
b_4=821246127744216999814751420752635267445501,\\
b_5=821246127744216999814751420752635267445502.
\end{aligned}
\end{equation}
These values satisfy the relationship:
\begin{equation}
\label{k1234tr} 
\begin{aligned}
[a_1,a_2,a_3,a_4,a_5]^k=[b_1,b_2,b_3,b_4,b_5]^k,\quad (k=1,2,3,4)
\end{aligned}
\end{equation}
Furthermore, the following approximate relationships hold:
\begin{align}
\frac{a_1}{a_5}=\sin ^2(\frac{\pi }{10}) \times  0.999999999999999999999999999999999999999994233...\nonumber\\[1mm]
\frac{a_2}{a_5}=\sin ^2(\frac{\pi }{10}) \times  1.000000000000000000000000000000000000000005766...\nonumber\\[1mm]
\frac{a_3}{a_5}=\sin ^2(\frac{3 \pi }{10})\times 0.999999999999999999999999999999999999999999158...\nonumber\\[1mm]
\frac{a_4}{a_5}=\sin ^2(\frac{3 \pi }{10})\times 1.000000000000000000000000000000000000000000841...\nonumber\\[1mm]
\frac{b_2}{a_5}=\sin ^2(\frac{2\pi }{10}) \times 0.999999999999999999999999999999999999999998406...\nonumber\\[1mm]
\frac{b_3}{a_5}=\sin ^2(\frac{2\pi }{10}) \times 1.000000000000000000000000000000000000000001593...\nonumber\\[1mm]
\frac{b_4}{a_5}=\sin ^2(\frac{4 \pi }{10})\times 0.999999999999999999999999999999999999999999391...\nonumber\\[1mm]
\frac{b_5}{a_5}=\sin ^2(\frac{4 \pi }{10})\times 1.000000000000000000000000000000000000000000608...\nonumber
\end{align}
\end{example}
Since $a_1+b_5=a_2+b_4=a_3+b_3=a_4+b_2=a_5+b_1$, \eqref{k1234tr} is a symmetric solution. Similarly, \eqref{t1234} is also a symmetric solution.

\begin{example}
Starting from \eqref{t12345}, we derive the following values:
\begin{equation}
\label{equationk12345} 
\begin{aligned}
a_1=\phantom{000000000000000000000000000000000000000000000}0, \\
a_2=\phantom{0}602177338768045259223669380026575662658682576,\\
a_3=\phantom{0}602177338768045259223669380026575662658682577,\\
a_4=1806532016304135777671008140079726987976047729,\\
a_5=1806532016304135777671008140079726987976047730,\\
a_6=2408709355072181036894677520106302650634730306,\\
b_1=\phantom{0}161352931623220326694532240637339153544496145, \\
b_2=\phantom{0}161352931623220326694532240637339153544496146, \\
b_3=1204354677536090518447338760053151325317365152, \\
b_4=1204354677536090518447338760053151325317365154, \\
b_5=2247356423448960710200145279468963497090234160, \\
b_6=2247356423448960710200145279468963497090234161.
\end{aligned}
\end{equation}
These values satisfy the relationship:
\begin{equation}
\label{k12345tr} 
\begin{aligned}
[a_1,a_2,a_3,a_4,a_5,a_6]^k=[b_1,b_2,b_3,b_4,b_5,b_6]^k,\quad (k=1,2,3,4,5)
\end{aligned}
\end{equation}
Furthermore, the following approximate relationships hold:
\begin{align}
\frac{a_2}{a_6}=\sin ^2(\frac{2\pi }{12}) \times 0.999999999999999999999999999999999999999999999169...\nonumber\\[1mm]
\frac{a_3}{a_6}=\sin ^2(\frac{2\pi }{12}) \times 1.000000000000000000000000000000000000000000000830...\nonumber\\[1mm]
\frac{a_4}{a_6}=\sin ^2(\frac{4 \pi }{12}) \times 0.999999999999999999999999999999999999999999999723...\nonumber\\[1mm]
\frac{a_5}{a_6}=\sin ^2(\frac{4 \pi }{12}) \times 1.000000000000000000000000000000000000000000000276...\nonumber\\[1mm]
\frac{b_1}{a_6}=\sin ^2(\frac{\pi }{12}) \times 0.999999999999999999999999999999999999999999996901...\nonumber\\[1mm]
\frac{b_2}{a_6}=\sin ^2(\frac{\pi }{12}) \times 1.000000000000000000000000000000000000000000003098...\nonumber\\[1mm]
\frac{b_3}{a_6}=\sin ^2(\frac{3 \pi }{12}) \times 0.999999999999999999999999999999999999999999999169...\nonumber\\[1mm]
\frac{b_4}{a_6}=\sin ^2(\frac{3 \pi }{12}) \times 1.000000000000000000000000000000000000000000000830...\nonumber\\[1mm]
\frac{b_5}{a_6}=\sin ^2(\frac{5 \pi }{12}) \times 0.999999999999999999999999999999999999999999999777...\nonumber\\[1mm]
\frac{b_6}{a_6}=\sin ^2(\frac{5 \pi }{12}) \times 1.000000000000000000000000000000000000000000000222...\nonumber
\end{align}
\end{example}
The solution method for this example can be found in \eqref{Parak12345}. Since $a_1+a_6=a_2+a_5=a_3+a_4=b_1+b_6=b_2+b_5=b_3+b_4$, \eqref{k12345tr} is a symmetric solution. Similarly, \eqref{t12345} is also a symmetric solution. 

\subsubsection{Ideal Trigonometric Solutions of GPTE/PTE}
In 1748, Leonhard Euler presented the following identities in his work \cite{Euler1988}:
\begin{flalign}
& {\small \quad\: \cos (z)-\cos (z+\frac{\pi }{3})+\cos (z+\frac{2 \pi }{3})=0},&\\[1mm]
& {\small \quad\: \cos (z)-\cos (z+\frac{\pi }{5})+\cos (z+\frac{2 \pi }{5})-\cos (z+\frac{3 \pi }{5})+\cos (z+\frac{4 \pi }{5})=0},&\\[1mm]
& {\small \quad\: \cos (z)-\cos (z+\frac{\pi }{7})+\cos (z+\frac{2 \pi }{7})-\cos (z+\frac{3 \pi }{7})+\cos (z+\frac{4 \pi }{7})}&\nonumber\\
& {\small \qquad\qquad -\cos (z+\frac{5 \pi }{7})+\cos (z+\frac{6 \pi }{7})=0}.&
\end{flalign}
where $z$ is any real number. For any odd number $n$ greater than $1$, the more general identity is as follows:
\begin{align}
\label{euler135}
& \sum _{j=0}^{n-1} (-1)^j \cos \left(\frac{j \pi}{n}+z\right)=0.
\end{align}
Inspired by Euler's above identity \eqref{euler135}, we discovered an even more remarkable result in 2021 \cite{Chen21}:
\begin{align}
\label{eulerk135}
& \sum _{j=0}^{n-1} (-1)^j \cos ^k\left(\frac{j \pi}{n}+z\right)=0, \qquad (k=1,3,5,\cdots,n-2).
\end{align}
Furthermore, we have the following identities:
\begin{identity}
\label{identityT3}
Let $z$ be any real number. Then,\\[1mm]
for any odd integer $n$ greater than $1$, the following hold:
\begin{align}
\label{sin135n}
& \sum _{j=0}^{n-1} (-1)^j \sin ^k\left(\frac{j \pi}{n}+z\right)=0, \qquad (k=1,3,5,\cdots,n-2)\\
\label{cos135n}
& \sum _{j=0}^{n-1} (-1)^j \cos ^k\left(\frac{j \pi}{n}+z\right)=0, \qquad (k=1,3,5,\cdots,n-2) \tag*{\eqref{eulerk135}}
\end{align}
for any even integer $n$ greater than $2$, the following hold:
\begin{align}
\label{sin246n}
& \sum _{j=0}^{n-1} (-1)^j \sin ^k\left(\frac{j \pi}{n}+z\right)=0, \qquad (k=2,4,6,\cdots,n-2)\\
\label{cos246n}
& \sum _{j=0}^{n-1} (-1)^j \cos ^k\left(\frac{j \pi}{n}+z\right)=0, \qquad (k=2,4,6,\cdots,n-2)
\end{align}
\end{identity}
The above identities were discovered by us in 2021, and a complete proof has not yet been provided. However, it is easy to prove for smaller integer values of $n$. Below, we present some specific examples.
\begin{example}
Applying identities \eqref{sin135n} and \eqref{sin246n}, when $n$ takes the values of 7, 8, 9, and 10, we can obtain: 
\begin{align}
\label{z135}
& \sin^k(\frac{\pi}{7}+z)+\sin^k(\frac{3\pi}{7}+z)+\sin^k(\frac{5\pi}{7}+z), \nonumber\\
& \quad=\sin^k(z)+\sin^k(\frac{2\pi}{7}+z)+\sin^k(\frac{4\pi}{7}+z)+\sin^k(\frac{6\pi}{7}+z)\nonumber\\
& \qquad\qquad\qquad\qquad\qquad (k=1,3,5)\\[3mm]
\label{z246}
& \sin^k(\frac{\pi}{8}+z)+\sin^k(\frac{3\pi}{8}+z)+\sin^k(\frac{5\pi}{8}+z)+\sin^k(\frac{7\pi}{8}+z),\nonumber\\
& \quad=\sin^k(z)+\sin^k(\frac{2\pi}{8}+z)+\sin^k(\frac{4\pi}{8}+z)+\sin^k(\frac{6\pi}{8}+z)\nonumber\\
& \qquad\qquad\qquad\qquad\qquad (k=2,4,6)\\[3mm]
\label{z1357}
& \sin^k(\frac{\pi}{9}+z)+\sin^k(\frac{3\pi}{9}+z)+\sin^k(\frac{5\pi}{9}+z)+\sin^k(\frac{7\pi}{9}+z), \nonumber\\
& \quad=\sin^k(z)+\sin^k(\frac{2\pi}{9}+z)+\sin^k(\frac{4\pi}{9}+z)+\sin^k(\frac{6\pi}{9}+z)+\sin^k(\frac{8\pi}{9}+z)\nonumber\\
& \qquad\qquad\qquad\qquad\qquad (k=1,3,5,7)\\[3mm]
& \sin^k(\frac{\pi}{10}+z)+\sin^k(\frac{3\pi}{10}+z)+\sin^k(\frac{5\pi}{10}+z)+\sin^k(\frac{7\pi}{10}+z)+\sin^k(\frac{9\pi}{10}+z),\nonumber\\
& \quad=\sin^k(z)+\sin^k(\frac{2\pi}{10}+z)+\sin^k(\frac{4\pi}{10}+z)+\sin^k(\frac{6\pi}{10}+z)+\sin^k(\frac{8\pi}{10}+z)\nonumber\\
& \qquad\qquad\qquad\qquad\qquad (k=2,4,6,8)
\end{align}
\end{example}

\begin{example}
By applying the identity \eqref{z135} with $z=\displaystyle\frac{\pi}{28}$, we can obtain: 
\begin{align}
& \big[0, \sin (\frac{5 \pi }{28}), \sin (\frac{7 \pi }{28}), \sin (\frac{13 \pi }{28})\big] ^k \nonumber\\
& \quad =\big[\sin(\frac{\pi }{28}),\sin (\frac{3 \pi }{28}),\sin (\frac{9 \pi }{28}),\sin (\frac{11 \pi }{28})\big]^k, \quad (k=1,3,5)
\end{align}
By applying the identity \eqref{z246} with $z=\displaystyle\frac{\pi}{32}$, we can obtain: 
\begin{align}
& \big[ \sin (\frac{\pi }{32}), \sin (\frac{7 \pi }{32}), \sin (\frac{9 \pi }{32}), \sin (\frac{15 \pi }{32})\big] ^k \nonumber\\
& \quad =\big[ \sin (\frac{3\pi }{32}), \sin (\frac{5 \pi }{32}), \sin (\frac{11 \pi }{32}), \sin (\frac{13 \pi }{32})\big] ^k,\quad (k=2,4,6)
\end{align}
By applying the identity \eqref{z1357} with $z=\displaystyle\frac{\pi}{27}$ and $z=\displaystyle\frac{\pi}{36}$  respectively, we can obtain: 
\begin{align}
& [0,\sin(\frac{4\pi}{27}),\sin(\frac{5\pi}{27}),\sin(\frac{10\pi}{27}),\sin(\frac{11\pi}{27})]^k \nonumber\\
& \quad=[\sin(\frac{\pi}{27}),\sin(\frac{2\pi}{27}),\sin(\frac{7\pi}{27}),\sin(\frac{8\pi}{27}),\sin(\frac{13\pi}{27})]^k,\nonumber\\
& \qquad\qquad\qquad\qquad\qquad\qquad\qquad\qquad (k=1,3,5,7) \\[2mm]
& \big[ 0, \sin (\frac{5\pi }{36}), \sin (\frac{7 \pi }{36}), \sin (\frac{13 \pi }{36}), \sin (\frac{15 \pi }{36})\big] ^k \nonumber\\
& \quad =\big[ \sin (\frac{\pi }{36}), \sin (\frac{3\pi }{36}),\sin (\frac{9 \pi }{36}), \sin (\frac{11 \pi }{36}), \sin (\frac{17 \pi }{36})\big] ^k ,\nonumber\\
& \qquad\qquad\qquad\qquad\qquad\qquad\qquad\qquad (k=1,3,5,7)
\end{align}
\end{example}
\begin{example}
In identity \eqref{z135}, by setting $z=\displaystyle\frac{121 \pi }{2800}$, we can obtain:
\begin{align}
\label{k135c2800}
& \{c_1,c_2,c_3,c_4\}=\left\{\sin\big(\frac{121 \pi }{2800}\big),\sin\big(\frac{279 \pi }{2800}\big),\sin\big(\frac{921 \pi }{2800}\big),\sin\big(\frac{1079 \pi }{2800}\big)\right\}\\
\label{k135d2800}
& \{d_1,d_2,d_3,d_4\}=\left\{0,\sin \big(\frac{521 \pi }{2800}\big),\sin\big (\frac{679 \pi }{2800}\big),\sin \big(\frac{1321 \pi }{2800}\big)\right\}
\end{align}
These sets satisfy:
\begin{align}
& [c_1,c_2,c_3,c_4]^k=[d_1,d_2,d_3,d_4]^k,\quad (k=1,3,5)
\end{align}
Further, starting from equations \eqref{k135c2800} and \eqref{k135d2800}, and with the help of computer search, we  obtain the following result in 2025:
\begin{align}
& \{a_1,a_2,a_3,a_4\}=\{1075, 2413, 6812, 7412\}\\
& \{b_1,b_2,b_3,b_4\}=\{0 , 4305, 5520 , 7887\}
\end{align}
These sets also satisfy:
\begin{align}
& [a_1,a_2,a_3,a_4]^k=[b_1,b_2,b_3,b_4]^k,\quad (k=1,3,5)
\end{align}
Additionally, the following relationships hold:
\begin{align}
& a_1=c_1 \times R_4 \times 1.00310\ldots\nonumber\\
& a_2=c_2 \times R_4 \times 0.98959\ldots\nonumber\\
& a_3=c_3 \times R_4 \times 1.00149\ldots\nonumber\\
& a_4=c_4 \times R_4 \times 1.00026\ldots\nonumber\\
& b_2=d_2 \times R_4 \times 0.98524\ldots\nonumber\\
& b_3=d_3 \times R_4 \times 1.00997\ldots\nonumber
\end{align}
where
\begin{align}
& R_4=\frac{b_4}{d_4}=\frac{7887}{\displaystyle\sin \big(\frac{1321 \pi }{2800}\big)}
\end{align}

\end{example}

\begin{example}
By applying the identities \eqref{sin246n} and \eqref{cos246n}, with $n=22$ and $z=\displaystyle\frac{\pi}{88}$, and denote
\begin{align}
p(x)=\sin^2\left(\frac{x\pi}{88}\right), \qquad q(x)=\cos^2\left(\frac{x\pi}{88}\right)
\end{align}
then we can obtain: 
\begin{flalign}
&\qquad\: \big[\: p(1),\ p(7),\ p(9),p(15),p(17),p(23),p(25),p(31),p(33),p(39),p(41) \:\big]^{k} & \nonumber\\[1mm]
&\quad =\big[\: p(3),p(5),p(11),p(13),p(19),p(21),p(27),p(29),p(35),p(37),p(43) \:\big]^{k} & \nonumber\\[2mm]
&\qquad\qquad\qquad (k=1,2,3,4,5,6,7,8,9,10) &
\end{flalign}
and
\begin{flalign}
&\qquad\: \big[\: p(1),p(7),p(9),p(15),p(17),p(23),p(25),p(31),p(33),p(39),p(41) \:\big]^{k} & \nonumber\\[1mm]
&\quad =\big[\: q(1),q(7),q(9),q(15),q(17),q(23),q(25),q(31),q(33),q(39),q(41) \:\big]^{k} & \nonumber\\[2mm]
&\qquad\qquad\qquad (k=1,2,3,4,5,6,7,8,9,10) &
\end{flalign}
\end{example}
In the above two solutions, due to the fact that $p(x)+p(44-x)=1$ or $p(x)+q(x)=1$, both are symmetric solutions.
\begin{example}
By applying the identities \eqref{sin246n} and \eqref{cos246n}, with $n=28$ and $z=\displaystyle\frac{\pi}{49}$, and denote
\begin{align}
p(x)=\sin^2\left(\frac{x\pi}{196}\right), \qquad q(x)=\cos^2\left(\frac{x\pi}{196}\right) 
\end{align}
then we can obtain:
\begin{flalign}
&\qquad\qquad\: \big[\: p(3),\ p(11),\ p(17),\ p(25),\ p(31),\ p(39),\ p(45),\ &\nonumber\\
&\qquad\qquad\qquad q(3),\ q(11),\ q(17),\ q(25),\ q(31),\ q(39),\ q(45) \:\big]^{k} & \nonumber\\[2mm]
&\qquad\quad =\big[\: p(4),\ p(10),\ p(18),\ p(24),\ p(32),\ p(38),\ p(46),\ &\nonumber\\
&\qquad\qquad\qquad q(4),\ q(10),\ q(18),\ q(24),\ q(32),\ q(38),\ q(46) \:\big]^{k} & \nonumber\\[2mm]
&\qquad\qquad\qquad\qquad (k=1,2,3,4,5,6,7,8,9,10,11,12,13) &
\end{flalign}
\end{example}
Since $p(x)+q(x)=1$, the above solution is a symmetric solution.

\begin{example}
By applying the identities \eqref{sin246n} and \eqref{cos246n}, with $n=38$ and $z=\displaystyle\frac{\pi}{1729}$, and denote
\begin{align}
p(x)=\sin^2\left(\frac{x\pi}{1729}\right), \qquad q(x)=\cos^2\left(\frac{x\pi}{1729}\right)
\end{align}
then we can obtain:
\begin{flalign}
&\qquad\: \big[\: p(1),p(90),p(92),p(181),p(183),p(272),p(274),p(363),p(365),p(454),&\nonumber\\
&\qquad\qquad p(456),p(545),p(547),p(636),p(638),p(727),p(729),p(818),p(820) \:\big]^{k} & \nonumber\\[2mm]
&\quad =\big[\: q(1),q(90),q(92),q(181),q(183),q(272),q(274),q(363),q(365),q(454),&\nonumber\\
&\qquad\qquad q(456),q(545),q(547),q(636),q(638),q(727),q(729),q(818),q(820) \:\big]^{k},& \nonumber\\[2mm]
&\qquad\qquad\qquad (k=1,2,3,4,5,6,7,8,9,10,11,12,13,14,15,16,17,18) &
\end{flalign}
\end{example}
Since $p(x)+q(x)=1$, the above solution is a symmetric solution.

\subsection{Trigonometric Chains of PTE and GPTE}
\subsubsection{Ideal Trigonometric Chains of PTE}
Based on Identity \ref{identityT3}, we further generalize it to obtain the following new Identity \cite{Chen21}.
\begin{identity}
\label{identityT4}
Denote
\begin{align}
\Theta_{n,k}=\frac{n (2 k-1)!}{2^{2 k-1} k! (k-1)!}
\end{align}
For positive even integer $n$ and any real number $z$, we have
\begin{align}
\label{evensin}
& \sum _{j=1}^{\frac{n}{2}} \left(\sin ^{2 k}\Big(\frac{(2 j-1)\pi}{2 n}-z\Big)+\sin ^{2 k}\Big(\frac{(2 j-1)\pi}{2 n}+z\Big)\right)=\Theta_{n,k}\\
\label{evencos}
&\sum _{j=1}^{\frac{n}{2}} \left(\cos ^{2 k}\Big(\frac{(2 j-1)\pi}{2 n}-z\Big)+\cos ^{2 k}\Big(\frac{(2 j-1)\pi}{2 n}+z\Big)\right)=\Theta_{n,k}\\
& \qquad\qquad\qquad\qquad\qquad (k=1,2,3,\cdots,n-1) \nonumber
\end{align}
For positive odd integer $n$ and any real number $z$, we have
\begin{align}
\label{oddsin}
&\sin^{2 k}(z)+\sum _{j=1}^{\frac{n-1}{2}} \left(\sin^{2 k}\Big(\frac{j\pi}{n}-z\Big)+\sin^{2 k}\Big(\frac{j\pi}{n}+z\Big)\right)=\Theta_{n,k} \\
\label{oddcos}
&\cos^{2 k}(z)+\sum _{j=1}^{\frac{n-1}{2}} \left(\cos^{2 k}\Big(\frac{j\pi}{n}-z\Big)+\cos^{2 k}\Big(\frac{j\pi}{n}+z\Big)\right)=\Theta_{n,k} \\
& \qquad\qquad\qquad\qquad\qquad (k=1,2,3,\cdots,n-1) \nonumber
\end{align}
\end{identity}
In the above four formulas, the left-hand side of each contains an arbitrary real number $z$, while the right-hand side is independent of $z$. This enables us to utilize these formulas to obtain ideal trigonometric chains for the PTE problem.
\begin{example}
Denote
\begin{align}
g(x)=\cos^2\left(\frac{x\pi}{100}\right)
\end{align}
By utilizing the identity \eqref{evencos} with $n=10$ and setting $z$ to $\displaystyle\frac{5\pi}{100}$,$\displaystyle\frac{4\pi}{100}$,$\displaystyle\frac{3\pi}{100}$,$\displaystyle\frac{2\pi}{100}$,$\displaystyle\frac{1\pi}{100}$ and $\displaystyle\frac{0\pi}{100}$ respectively, we derive the following trigonometric chain: 
\begin{align}
&\ \ \;\big[\ g(0),\ g(10),\ g(10),\ g(20),\ g(20),\ g(30),\ g(30),\ g(40),\ g(40),\ g(50)\ \big]^{k} & \nonumber\\
&=\big[ \ g(1),\ g(9),\ g(11),\ g(19),\ g(21),\ g(29),\ g(31),\ g(39),\ g(41),\ g(49)\ \big]^{k} & \nonumber\\
&=\big[ \ g(2),\ g(8),\ g(12),\ g(18),\ g(22),\ g(28),\ g(32),\ g(38),\ g(42),\ g(48)\ \big]^{k} & \nonumber\\
&=\big[ \ g(3),\ g(7),\ g(13),\ g(17),\ g(23),\ g(27),\ g(33),\ g(37),\ g(43),\ g(47)\ \big]^{k} & \nonumber\\
&=\big[ \ g(4),\ g(6),\ g(14),\ g(16),\ g(24),\ g(26),\ g(34),\ g(36),\ g(44),\ g(46)\ \big]^{k} & \nonumber\\
&=\big[ \ g(5),\ g(5),\ g(15),\ g(15),\ g(25),\ g(25),\ g(35),\ g(35),\ g(45),\ g(45)\ \big]^{k},& \nonumber\\[2mm]
&\qquad\qquad\qquad(k=1,2,3,4,5,6,7,8,9)
\end{align}
\end{example}
Since $g(x)+g(50-x)=1$, the above chain is a symmetric chain.
\begin{example}
Denote
\begin{align}
p(x)=\sin^{2}\left(x\pi\right)
\end{align}
By utilizing the identity \eqref{oddsin} with $n=11$ and setting $z$ to $\displaystyle\frac{0\pi}{11}$,$\displaystyle\frac{1\pi}{22}$,$\displaystyle\frac{1\pi}{33}$,$\cdots$, and $\displaystyle\frac{4\pi}{99}$ respectively, we derive the following trigonometric chain: 
\begin{flalign}
&\quad\, \big[\: p(\frac{0}{11}),p(\frac{1}{11}),p(\frac{1}{11}),p(\frac{2}{11}),p(\frac{2}{11}),p(\frac{3}{11}),p(\frac{3}{11}),p(\frac{4}{11}),p(\frac{4}{11}),p(\frac{5}{11}),p(\frac{5}{11})\:\big]^{k} & \nonumber\\[2mm]
&=\big[\: p(\frac{1}{22}),p(\frac{1}{22}),p(\frac{3}{22}),p(\frac{3}{22}),p(\frac{5}{22}),p(\frac{5}{22}),p(\frac{7}{22}),p(\frac{7}{22}),p(\frac{9}{22}),p(\frac{9}{22}),p(\frac{11}{22})\:\big]^{k}& \nonumber\\[2mm]
&=\big[\: p(\frac{1}{33}),p(\frac{2}{33}),p(\frac{4}{33}),p(\frac{5}{33}),p(\frac{7}{33}),p(\frac{8}{33}),p(\frac{10}{33}),p(\frac{11}{33}),p(\frac{13}{33}),p(\frac{14}{33}),p(\frac{16}{33})\:\big]^{k}& \nonumber\\[2mm]
&=\big[\: p(\frac{1}{44}),p(\frac{3}{44}),p(\frac{5}{44}),p(\frac{7}{44}),p(\frac{9}{44}),p(\frac{11}{44}),p(\frac{13}{44}),p(\frac{15}{44}),p(\frac{17}{44}),p(\frac{19}{44}),p(\frac{21}{44})\:\big]^{k}& \nonumber\\[2mm]
&=\big[\: p(\frac{1}{55}),p(\frac{4}{55}),p(\frac{6}{55}),p(\frac{9}{55}),p(\frac{11}{55}),p(\frac{14}{55}),p(\frac{16}{55}),p(\frac{19}{55}),p(\frac{21}{55}),p(\frac{24}{55}),p(\frac{26}{55})\:\big]^{k}& \nonumber\\[2mm]
&=\big[\: p(\frac{2}{55}),p(\frac{3}{55}),p(\frac{7}{55}),p(\frac{8}{55}),p(\frac{12}{55}),p(\frac{13}{55}),p(\frac{17}{55}),p(\frac{18}{55}),p(\frac{22}{55}),p(\frac{23}{55}),p(\frac{27}{55})\:\big]^{k}& \nonumber\\[2mm]
&=\big[\: p(\frac{1}{66}),p(\frac{5}{66}),p(\frac{7}{66}),p(\frac{11}{66}),p(\frac{13}{66}),p(\frac{17}{66}),p(\frac{19}{66}),p(\frac{23}{66}),p(\frac{25}{66}),p(\frac{29}{66}),p(\frac{31}{66})\:\big]^{k}& \nonumber\\[2mm]
&=\big[\: p(\frac{1}{77}),p(\frac{6}{77}),p(\frac{8}{77}),p(\frac{13}{77}),p(\frac{15}{77}),p(\frac{20}{77}),p(\frac{22}{77}),p(\frac{27}{77}),p(\frac{29}{77}),p(\frac{34}{77}),p(\frac{36}{77})\:\big]^{k}& \nonumber\\[2mm]
&=\big[\: p(\frac{2}{77}),p(\frac{5}{77}),p(\frac{9}{77}),p(\frac{12}{77}),p(\frac{16}{77}),p(\frac{19}{77}),p(\frac{23}{77}),p(\frac{26}{77}),p(\frac{30}{77}),p(\frac{33}{77}),p(\frac{37}{77})\:\big]^{k}& \nonumber\\[2mm]
&=\big[\: p(\frac{3}{77}),p(\frac{4}{77}),p(\frac{10}{77}),p(\frac{11}{77}),p(\frac{17}{77}),p(\frac{18}{77}),p(\frac{24}{77}),p(\frac{25}{77}),p(\frac{31}{77}),p(\frac{32}{77}),p(\frac{38}{77})\:\big]^{k}& \nonumber\\[2mm]
&=\big[\: p(\frac{1}{88}),p(\frac{7}{88}),p(\frac{9}{88}),p(\frac{15}{88}),p(\frac{17}{88}),p(\frac{23}{88}),p(\frac{25}{88}),p(\frac{31}{88}),p(\frac{33}{88}),p(\frac{39}{88}),p(\frac{41}{88})\:\big]^{k}& \nonumber\\[2mm]
&=\big[\: p(\frac{3}{88}),p(\frac{5}{88}),p(\frac{11}{88}),p(\frac{13}{88}),p(\frac{19}{88}),p(\frac{21}{88}),p(\frac{27}{88}),p(\frac{29}{88}),p(\frac{35}{88}),p(\frac{37}{88}),p(\frac{43}{88})\:\big]^{k}& \nonumber\\[2mm]
&=\big[\: p(\frac{1}{99}),p(\frac{8}{99}),p(\frac{10}{99}),p(\frac{17}{99}),p(\frac{19}{99}),p(\frac{26}{99}),p(\frac{28}{99}),p(\frac{35}{99}),p(\frac{37}{99}),p(\frac{44}{99}),p(\frac{46}{99})\:\big]^{k}& \nonumber\\[2mm]
&=\big[\: p(\frac{2}{99}),p(\frac{7}{99}),p(\frac{11}{99}),p(\frac{16}{99}),p(\frac{20}{99}),p(\frac{25}{99}),p(\frac{29}{99}),p(\frac{34}{99}),p(\frac{38}{99}),p(\frac{43}{99}),p(\frac{47}{99})\:\big]^{k}& \nonumber\\[2mm]
&=\big[\: p(\frac{4}{99}),p(\frac{5}{99}),p(\frac{13}{99}),p(\frac{14}{99}),p(\frac{22}{99}),p(\frac{23}{99}),p(\frac{31}{99}),p(\frac{32}{99}),p(\frac{40}{99}),p(\frac{41}{99}),p(\frac{49}{99})\:\big]^{k},& \nonumber\\[2mm]
&\qquad\qquad\qquad(k=1,2,3,4,5,6,7,8,9,10)&
\end{flalign}
\end{example}
In the above chain composed of 15 sets of trigonometric function arrays, each pair of sets constitutes an ideal trigonometric solution of PTE. These trigonometric solutions include symmetric solutions, for example: 
\begin{flalign}
&\quad\, \big[\: p(\frac{0}{11}),p(\frac{1}{11}),p(\frac{1}{11}),p(\frac{2}{11}),p(\frac{2}{11}),p(\frac{3}{11}),p(\frac{3}{11}),p(\frac{4}{11}),p(\frac{4}{11}),p(\frac{5}{11}),p(\frac{5}{11})\:\big]^{k} & \nonumber\\[2mm]
&=\big[\: p(\frac{1}{22}),p(\frac{1}{22}),p(\frac{3}{22}),p(\frac{3}{22}),p(\frac{5}{22}),p(\frac{5}{22}),p(\frac{7}{22}),p(\frac{7}{22}),p(\frac{9}{22}),p(\frac{9}{22}),p(\frac{11}{22})\:\big]^{k}& \nonumber\\
&\qquad\qquad\qquad(k=1,2,3,4,5,6,7,8,9,10)&
\end{flalign}
Most of them, however, are non-symmetric solutions, for example: 
\begin{flalign}
&=\big[\: p(\frac{1}{22}),p(\frac{1}{22}),p(\frac{3}{22}),p(\frac{3}{22}),p(\frac{5}{22}),p(\frac{5}{22}),p(\frac{7}{22}),p(\frac{7}{22}),p(\frac{9}{22}),p(\frac{9}{22}),p(\frac{11}{22})\:\big]^{k}& \nonumber\\[2mm]
&=\big[\: p(\frac{1}{33}),p(\frac{2}{33}),p(\frac{4}{33}),p(\frac{5}{33}),p(\frac{7}{33}),p(\frac{8}{33}),p(\frac{10}{33}),p(\frac{11}{33}),p(\frac{13}{33}),p(\frac{14}{33}),p(\frac{16}{33})\:\big]^{k}& \nonumber\\
&\qquad\qquad\qquad(k=1,2,3,4,5,6,7,8,9,10)&
\end{flalign}

\subsubsection{Ideal Trigonometric Chains of GPTE}
Based on Identity \ref{identityT4}, we further derived the following identity \cite{Chen21}.
\begin{identity}
\label{identityT5}
For positive odd integer $n$ and any real number $z$, we have
\begin{equation}
\begin{aligned}
\label{chainGPTE}
&\cos ^k(z)+\sum _{j=1}^{\frac{n-1}{2}} \left(\cos ^k\Big(\frac{2 \pi j}{n}-z\Big)+\cos ^k\Big(\frac{2 \pi j}{n}+z\Big)\right)=2 \sum _{j=1}^{\frac{n-1}{2}} \cos ^k\left(\frac{2 \pi j}{n}\right)+1\\
&\qquad\qquad (k=1,2,3,4,\cdots,n-2,n-1,n+1,n+3,n+5,\cdots,2n-2)
\end{aligned}
\end{equation}
\end{identity}
The right side of the above identity \eqref{chainGPTE} is independent of the variable z and can be equivalently expressed as
\begin{equation}
\begin{aligned}
& 2 \sum _{j=1}^{\frac{n-1}{2}} \cos ^k\left(\frac{2 \pi j}{n}\right)+1=0,\quad (k=1,3,5,\cdots,n-2)\\
& 2 \sum _{j=1}^{\frac{n-1}{2}} \cos ^k\left(\frac{2 \pi j}{n}\right)+1=\frac{n(k-1)!}{2^{k-1}(k/2-1)! (k/2)!}\: ,\\
& \qquad\qquad\qquad\qquad\qquad\qquad\:\: (k=2,4,6,\cdots,2n-2)
\end{aligned}
\end{equation}
Applying the identity \eqref{chainGPTE}, we can obtain ideal trigonometric chains of GPTE. The following are examples:
\begin{example}
By utilizing the identity \eqref{chainGPTE} with $n=5$ and setting $z$ to $\displaystyle\frac{\pi}{10}$,$\displaystyle\frac{\pi}{15}$,$\displaystyle\frac{2\pi}{25}$ and $\displaystyle\frac{4\pi}{25}$ respectively, we derive the following trigonometric chains: 
\begin{align}
\label{h123468Tri}
&\qquad\: \big[ \cos (\frac{\pi}{10}),\cos (\frac{3\pi}{10}),\cos (\frac{5\pi}{10}),\cos (\frac{7\pi}{10}),\cos(\frac{9\pi}{10}) \big]^{h} \nonumber\\[2mm]
&\quad = \big[ \cos (\frac{\pi}{15}),\cos (\frac{5\pi}{15}),\cos (\frac{7\pi}{15}),\cos (\frac{11\pi}{15}),\cos (\frac{13\pi}{15}) \big]^{h} \nonumber\\[2mm]
&\quad = \big[ \cos (\frac{2\pi}{25}),\cos (\frac{8\pi}{25}),\cos (\frac{12\pi}{25}),\cos (\frac{18\pi}{25}),\cos (\frac{22\pi}{25}) \big]^{h} \nonumber\\[2mm]
&\quad = \big[ \cos (\frac{4\pi}{25}),\cos (\frac{6\pi}{25}),\cos (\frac{14\pi}{25}),\cos (\frac{16\pi}{25}),\cos (\frac{24\pi}{25}) \big]^{h} \nonumber\\[2mm]
&\qquad\qquad\qquad(h=1,2,3,4,6,8)
\end{align}
\end{example}
It is particularly noteworthy that in the above four sets of trigonometric function arrays, each set comprises five elements, yet the equality holds true for six distinct powers. In contrast, for instance, in the following example, each set consists of five integers, and the equality only holds for five different powers.
\begin{align}
& [-180, -131, 71, 307, 308]^h =[-313, 99, 100, 188, 301]^h,\nonumber\\
& \qquad \qquad\qquad\qquad\qquad\qquad\qquad (h=1,2,4,6,8) \tag*{\eqref{h12468s313}}
\end{align}
\begin{example}
Denote
\begin{align}
g(x)=\cos\left(x\pi\right)
\end{align}
By utilizing the identity \eqref{chainGPTE} with $n=11$ and setting $z$ to $\displaystyle\frac{0\pi}{11}$,$\displaystyle\frac{1\pi}{11}$,$\displaystyle\frac{1\pi}{22}$,$\cdots$, and $\displaystyle\frac{5\pi}{66}$ respectively, we derive the following trigonometric chain: 
\begin{flalign}
\label{h1to20Tri}
&\quad\, \big[ g(\frac{0}{11}),g(\frac{2}{11}),g(\frac{2}{11}),g(\frac{4}{11}),g(\frac{4}{11}),g(\frac{6}{11}),g(\frac{6}{11}),g(\frac{8}{11}),g(\frac{8}{11}),g(\frac{10}{11}),g(\frac{10}{11})\big]^{h} & \nonumber\\[2mm]
&=\big[ g(\frac{1}{11}),g(\frac{1}{11}),g(\frac{3}{11}),g(\frac{3}{11}),g(\frac{5}{11}),g(\frac{5}{11}),g(\frac{7}{11}),g(\frac{7}{11}),g(\frac{9}{11}),g(\frac{9}{11}),g(\frac{11}{11})\big]^{h} & \nonumber\\[2mm]
&=\big[ g(\frac{1}{22}),g(\frac{3}{22}),g(\frac{5}{22}),g(\frac{7}{22}),g(\frac{9}{22}),g(\frac{11}{22}),g(\frac{13}{22}),g(\frac{15}{22}),g(\frac{17}{22}),g(\frac{19}{22}),g(\frac{21}{22})\big]^{h}& \nonumber\\[2mm]
&=\big[ g(\frac{1}{33}),g(\frac{5}{33}),g(\frac{7}{33}),g(\frac{11}{33}),g(\frac{13}{33}),g(\frac{17}{33}),g(\frac{19}{33}),g(\frac{23}{33}),g(\frac{25}{33}),g(\frac{29}{33}),g(\frac{31}{33})\big]^{h}& \nonumber\\[2mm]
&=\big[ g(\frac{2}{33}),g(\frac{4}{33}),g(\frac{8}{33}),g(\frac{10}{33}),g(\frac{14}{33}),g(\frac{16}{33}),g(\frac{20}{33}),g(\frac{22}{33}),g(\frac{26}{33}),g(\frac{28}{33}),g(\frac{32}{33})\big]^{h}& \nonumber\\[2mm]
&=\big[ g(\frac{1}{44}),g(\frac{7}{44}),g(\frac{9}{44}),g(\frac{15}{44}),g(\frac{17}{44}),g(\frac{23}{44}),g(\frac{25}{44}),g(\frac{31}{44}),g(\frac{33}{44}),g(\frac{39}{44}),g(\frac{41}{44})\big]^{h}& \nonumber\\[2mm]
&=\big[ g(\frac{3}{44}),g(\frac{5}{44}),g(\frac{11}{44}),g(\frac{13}{44}),g(\frac{19}{44}),g(\frac{21}{44}),g(\frac{27}{44}),g(\frac{29}{44}),g(\frac{35}{44}),g(\frac{37}{44}),g(\frac{43}{44})\big]^{h}& \nonumber\\[2mm]
&=\big[ g(\frac{1}{55}),g(\frac{9}{55}),g(\frac{11}{55}),g(\frac{19}{55}),g(\frac{21}{55}),g(\frac{29}{55}),g(\frac{31}{55}),g(\frac{39}{55}),g(\frac{41}{55}),g(\frac{49}{55}),g(\frac{51}{55})\big]^{h}& \nonumber\\[2mm]
&=\big[ g(\frac{2}{55}),g(\frac{8}{55}),g(\frac{12}{55}),g(\frac{18}{55}),g(\frac{22}{55}),g(\frac{28}{55}),g(\frac{32}{55}),g(\frac{38}{55}),g(\frac{42}{55}),g(\frac{48}{55}),g(\frac{52}{55})\big]^{h}& \nonumber\\[2mm]
&=\big[ g(\frac{3}{55}),g(\frac{7}{55}),g(\frac{13}{55}),g(\frac{17}{55}),g(\frac{23}{55}),g(\frac{27}{55}),g(\frac{33}{55}),g(\frac{37}{55}),g(\frac{43}{55}),g(\frac{47}{55}),g(\frac{53}{55})\big]^{h}& \nonumber\\[2mm]
&=\big[ g(\frac{4}{55}),g(\frac{6}{55}),g(\frac{14}{55}),g(\frac{16}{55}),g(\frac{24}{55}),g(\frac{26}{55}),g(\frac{34}{55}),g(\frac{36}{55}),g(\frac{44}{55}),g(\frac{46}{55}),g(\frac{54}{55})\big]^{h}& \nonumber\\[2mm]
&=\big[ g(\frac{1}{66}),g(\frac{11}{66}),g(\frac{13}{66}),g(\frac{23}{66}),g(\frac{25}{66}),g(\frac{35}{66}),g(\frac{37}{55}),g(\frac{47}{66}),g(\frac{49}{66}),g(\frac{59}{66}),g(\frac{61}{66})\big]^{h}& \nonumber\\[2mm]
&=\big[ g(\frac{5}{66}),g(\frac{7}{66}),g(\frac{17}{66}),g(\frac{19}{66}),g(\frac{29}{66}),g(\frac{31}{66}),g(\frac{41}{55}),g(\frac{43}{66}),g(\frac{53}{66}),g(\frac{55}{66}),g(\frac{65}{66})\big]^{h}& \nonumber\\[2mm]
&\qquad\qquad\qquad(h=1,2,3,4,5,6,7,8,9,10,12,14,16,18,20)&
\end{flalign}
\end{example}
In the 13 sets of trigonometric function arrays presented above, each set \\comprises 11 elements, and remarkably, the equality stands valid for 15 distinct powers. While the example provided illustrates chains of length 13 (corresponding to the 13 sets), it is noteworthy that in \eqref{chainGPTE}, the selection of $z$ is unlimited, thus allowing the length of the chains to be potentially infinite.
\subsection{Trigonometric Solutions of Fermat's Last Theorem}
Ramanujan once presented the following two formulas \cite{Ramanujan1988,Ramanujan1994}:
\begin{align} 
& \sqrt[3]{\cos\left(\frac{2\pi}{7}\right)}+\sqrt[3]{\cos\left(\frac{4\pi}{7}\right)}+\sqrt[3]{\cos\left(\frac{6\pi}{7}\right)}=\sqrt[3]{\frac{5-3\sqrt[3]{7}}{2}}\\
& \sqrt[3]{\cos\left(\frac{2\pi}{9}\right)}+\sqrt[3]{\cos\left(\frac{4\pi}{9}\right)}+\sqrt[3]{\cos\left(\frac{8\pi}{9}\right)}=\sqrt[3]{\frac{3\sqrt[3]{9}-6}{2}}
\end{align}
Inspired by these two formulas of Ramanujan, we derived the following results in 2021 \cite{Chen2125}:
\begin{align} 
\label{cos7}
& \sqrt[3]{\cos \left(\frac{2 \pi }{7}\right)+\frac{2}{9}}+\sqrt[3]{\cos \left(\frac{4 \pi }{7}\right)+\frac{2}{9}}+\sqrt[3]{ \cos \left(\frac{6 \pi }{7}\right)+\frac{2}{9}}={0} \\
& \sqrt[3]{\cos \left(\frac{2 \pi }{9}\right)-\frac{1}{6}}+\sqrt[3]{\cos \left(\frac{4 \pi }{9}\right)-\frac{1}{6}}+\sqrt[3]{\cos \left(\frac{8 \pi }{9}\right)-\frac{1}{6}}={0}
\end{align}
Building upon these previous results, we obtained a series of more general new findings in 2021, namely trigonometric function solutions that satisfy the equation:
\begin{align} 
\label{TriFLT}
\sqrt[k]{X} + \sqrt[k]{Y} = \sqrt[k]{Z}
\end{align}
where $k$ is an integer greater than $1$. We refer to these as Trigonometric Solutions of Fermat's Last Theorem.\\
For Equation \eqref{TriFLT}, the trigonometric function solutions we have found so far can be classified into three types based on the denominator $d$ in the corresponding trigonometric function $\cos\left(\frac{c}{d} \pi \right)$: $d$ is a prime number of the form $6n+1$, $d$ is a prime number of the form $4n+1$, and $d$ is a composite number of the form $3\times$prime. Below, we will provide examples to illustrate each type.
\subsubsection{Prime Numbers of the Form 6n+1}
So far, for primes of the form $6n+1$, we have only found trigonometric solutions to \eqref{TriFLT} when $k=3$ \cite{Chen2125}. The smallest three prime numbers of the form $6n+1$ are $7$, $13$, and $19$. Their corresponding trigonometric function solutions include \eqref{cos7} and the following two formulas:
\begin{example}
\begin{flalign} 
\label{cos13}
& \sqrt[3]{\cos \left(\frac{4 \pi }{13}\right)+\cos \left(\frac{6 \pi }{13}\right)-\frac{1}{9}}+\sqrt[3]{\cos \left(\frac{2 \pi }{13}\right)+\cos \left(\frac{10 \pi }{13}\right)-\frac{1}{9}}\nonumber &\\
& \quad +\sqrt[3]{\cos \left(\frac{8 \pi }{13}\right)+\cos \left(\frac{12 \pi }{13}\right)-\frac{1}{9}}=0 &\\[2mm]
& \sqrt[3]{\cos (\frac{4 \pi }{19})+\cos (\frac{6 \pi }{19})+\cos (\frac{10 \pi }{19})+\frac{5}{9}} +\sqrt[3]{\cos (\frac{2 \pi }{19})+\cos (\frac{14 \pi }{19})+\cos(\frac{16 \pi }{19})+\frac{5}{9}}\nonumber &\\[1mm]
& \quad +\sqrt[3]{\cos (\frac{8 \pi }{19})+\cos (\frac{12 \pi }{19})+\cos (\frac{18 \pi }{19})+\frac{5}{9}}=0 &
\end{flalign}
\end{example}
 The above two formulas can be written in the following equivalent forms:
\begin{flalign} 
& {\small\sqrt[3]{\sum _{k=1}^2 \cos \left(\frac{2 \pi}{13} 2^{3 k}\right)-\frac{1}{9}}+\sqrt[3]{\sum _{k=1}^2 \cos \left(\frac{4 \pi}{13} 2^{3 k}\right)-\frac{1}{9}}+\sqrt[3]{\sum _{k=1}^2 \cos \left(\frac{8 \pi}{13} 2^{3 k}\right)-\frac{1}{9}}=0} &\\
& {\small \sqrt[3]{\sum _{k=1}^3 \cos \left(\frac{2 \pi}{19}  2^{3 k}\right)+\frac{5}{9}}+\sqrt[3]{\sum _{k=1}^3 \cos \left(\frac{10 \pi}{19}  2^{3 k}\right)+\frac{5}{9}}+\sqrt[3]{\sum _{k=1}^3 \cos \left(\frac{50 \pi}{19}  2^{3 k}\right)+\frac{5}{9}}=0} &
\end{flalign}
Here are examples corresponding to other primes of the form $6n+1$ within $100$:
\begin{example}
\begin{flalign} 
& {\small \sqrt[3]{\sum _{k=1}^5 \cos \left(\frac{2 \pi}{31}  2^{3 k}\right)+\frac{7}{18}}+\sqrt[3]{\sum _{k=1}^5 \cos \left(\frac{10 \pi}{31}  2^{3 k}\right)+\frac{7}{18}}+\sqrt[3]{\sum _{k=1}^5 \cos \left(\frac{50 \pi}{31}  2^{3 k}\right)+\frac{7}{18}}=0} &\\
& {\small\sqrt[3]{\sum _{k=1}^6 \cos \left(\frac{2 \pi}{37}  2^{3 k}\right)-\frac{4}{9}}+\sqrt[3]{\sum _{k=1}^6 \cos \left(\frac{4 \pi}{37}  2^{3 k}\right)-\frac{4}{9}} +\sqrt[3]{\sum _{k=1}^6 \cos \left(\frac{8 \pi}{37}  2^{3 k}\right)-\frac{4}{9}}=0} & \\
& {\small\sqrt[3]{\sum _{k=1}^7 \cos \left(\frac{2 \pi}{43}  2^{3 k}\right)-\frac{5}{18}}+\sqrt[3]{\sum _{k=1}^7 \cos \left(\frac{10 \pi}{43}  2^{3 k}\right)-\frac{5}{18}}+\sqrt[3]{\sum _{k=1}^7 \cos \left(\frac{50 \pi}{43}  2^{3 k}\right)-\frac{5}{18}}=0} & \\
& {\small \sqrt[3]{\sum _{k=1}^{10} \cos \left(\frac{2 \pi}{61} 2^{3 k}\right)+\frac{2}{9}}+\sqrt[3]{\sum _{k=1}^{10} \cos \left(\frac{4 \pi}{61} 2^{3 k}\right)+\frac{2}{9}}+\sqrt[3]{\sum _{k=1}^{10} \cos \left(\frac{8 \pi}{61} 2^{3 k}\right)+\frac{2}{9}}=0 } & \\
& {\small \sqrt[3]{\sum _{k=1}^{11} \cos \left(\frac{2 \pi}{67}  2^{3 k}\right)-\frac{1}{9}}+\sqrt[3]{\sum _{k=1}^{11} \cos \left(\frac{4 \pi}{67}  2^{3 k}\right)-\frac{1}{9}}+\sqrt[3]{\sum _{k=1}^{11} \cos \left(\frac{8 \pi}{67}  2^{3 k}\right)-\frac{1}{9}}=0 } & \\
& {\small \sqrt[3]{\sum _{k=1}^{12} \cos \left(\frac{2 \pi}{73} 5^{3 k}\right)+\frac{5}{9}}+\sqrt[3]{\sum _{k=1}^{12} \cos \left(\frac{10 \pi}{73} 5^{3 k}\right)+\frac{5}{9}}+\sqrt[3]{\sum _{k=1}^{12} \cos \left(\frac{50 \pi}{73} 5^{3 k}\right)+\frac{5}{9}}=0 } & \\
& {\small \sqrt[3]{\sum _{k=1}^{13} \cos \left(\frac{2 \pi}{79} 2^{3 k}\right)-\frac{7}{9}}+\sqrt[3]{\sum _{k=1}^{13} \cos \left(\frac{4 \pi}{79} 2^{3 k}\right)-\frac{7}{9}}+\sqrt[3]{\sum _{k=1}^{13} \cos \left(\frac{8 \pi}{79} 2^{3 k}\right)-\frac{7}{9}}=0 } & \\
& {\small \sqrt[3]{\sum _{k=1}^{16} \cos \left(\frac{2 \pi}{97}  5^{3 k}\right)+\frac{11}{9}}+\sqrt[3]{\sum _{k=1}^{16} \cos \left(\frac{10 \pi}{97}  5^{3 k}\right)+\frac{11}{9}}+\sqrt[3]{\sum _{k=1}^{16} \cos \left(\frac{50 \pi}{97}  5^{3 k}\right)+\frac{11}{9}}=0 } &
\end{flalign}
\end{example}

\subsubsection{Prime Numbers of the Form 4n+1}
For primes of the form $4n+1$, we can find trigonometric solutions to $\eqref{TriFLT}$ when $k$ is any integer greater than $1$ \cite{Chen2125}. For example, below is an example for the prime number $5$.
\begin{example}
\begin{align}
& \sqrt[2]{2 \cos \left(\frac{2 \pi }{5}\right)+2}-\sqrt[2]{2 \cos \left(\frac{4 \pi }{5}\right)+2}=1\\
& \sqrt[3]{4 \cos \left(\frac{2 \pi }{5}\right)+3}+\sqrt[3]{4 \cos \left(\frac{4 \pi }{5}\right)+3}=1\\
& \sqrt[4]{6 \cos \left(\frac{2 \pi }{5}\right)+5}-\sqrt[4]{6 \cos \left(\frac{4 \pi }{5}\right)+5}=1\\
& \sqrt[5]{10 \cos \left(\frac{2 \pi }{5}\right)+8}+\sqrt[5]{10 \cos \left(\frac{4 \pi }{5}\right)+8}=1\\
& \qquad\qquad \vdots \nonumber \\
& {\small \sqrt[99]{437845991669110338052 \cos \left(\frac{2 \pi }{5}\right)+354224848179261915075} \nonumber }\\
& {\small \: +\sqrt[99]{437845991669110338052 \cos \left(\frac{4 \pi }{5}\right)+354224848179261915075}=1 }\\
& \qquad\qquad \vdots \nonumber
\end{align}
\end{example}
For the prime number $13$, it can be regarded as both of the form $6n+1$, with its corresponding formula being \eqref{cos13}, and of the form $4n+1$, with its corresponding formulas as shown in the following example:
\begin{example}
Denote
\begin{align} 
& A_{13}=\cos \left(\frac{2 \pi }{13}\right)+\cos \left(\frac{6 \pi }{13}\right)+\cos \left(\frac{8 \pi }{13}\right)\\
& B_{13}=\cos \left(\frac{4 \pi }{13}\right)+\cos \left(\frac{10 \pi }{13}\right)+\cos \left(\frac{12 \pi }{13}\right)
\end{align}
then we have
\begin{align}
& \sqrt[2]{2 A_{13}+4} -\sqrt[2]{2 B_{13}+4}=1\\
& \sqrt[3]{8 A_{13}+7} +\sqrt[3]{8 B_{13}+7}=1\\
& \sqrt[4]{14 A_{13}+19} -\sqrt[4]{14 B_{13}+19}=1\\
& \sqrt[5]{38 A_{13}+40} +\sqrt[5]{38 B_{13}+40}=1\\
& \qquad\qquad \vdots \nonumber \\
& \sqrt[17]{798662 A_{13}+919480} +\sqrt[17]{798662 B_{13}+919480}=1\\
& \sqrt[18]{1838960 A_{13}+2117473} -\sqrt[18]{1838960 B_{13}+2117473}=1\\
& \qquad\qquad \vdots \nonumber
\end{align}
\end{example}
The following is an example corresponding to the prime number $17$:
\begin{example}
Denote
\begin{align}
& A_{17}=\cos \left(\frac{2 \pi }{17}\right)+\cos \left(\frac{4 \pi }{17}\right)+\cos \left(\frac{8 \pi }{17}\right)+\cos \left(\frac{16 \pi }{17}\right)\\
& B_{17}=\cos \left(\frac{6 \pi }{17}\right)+\cos \left(\frac{10 \pi }{17}\right)+\cos \left(\frac{12 \pi }{17}\right)+\cos \left(\frac{14 \pi }{17}\right)
\end{align}
then we have
\begin{align}
& \sqrt[2]{2 A_{17}+5}-\sqrt[2]{2 B_{17}+5}=1\\
& \sqrt[3]{10 A_{17}+9}+\sqrt[3]{10 B_{17}+9}=1\\
& \sqrt[4]{18 A_{17}+29}-\sqrt[4]{18 B_{17}+29}=1\\
& \sqrt[5]{58 A_{17}+65}+\sqrt[5]{58 B_{17}+65}=1 \\
& \qquad\qquad \vdots \nonumber \\
& \sqrt[16]{1666098 A_{17}+2135149} -\sqrt[16]{1666098 B_{17}+2135149}=1\\
& \sqrt[17]{4270298 A_{17}+5467345} +\sqrt[17]{4270298 B_{17}+5467345}=1\\
& \qquad\qquad \vdots \nonumber
\end{align}
\end{example}
Below are examples corresponding to several other prime numbers of the form $4n+1$, with only the formulas for their cube roots listed for simplicity:
\begin{example}
\begin{align}
& \sqrt[3]{16 \sum _{k=1}^7 \cos \left(\frac{2 \pi}{29}4^k\right)+15}+\sqrt[3]{16 \sum _{k=1}^7 \cos \left(\frac{4 \pi }{29} 4^k\right)+15}=1\\[1mm]
& \sqrt[3]{20 \sum _{k=1}^9 \cos \left(\frac{2 \pi}{37}4^k\right)+19}+\sqrt[3]{20 \sum _{k=1}^9 \cos \left(\frac{4 \pi }{37}4^k\right)+19}=1\\[1mm]
& \sqrt[3]{22 \sum _{k=1}^{10} \cos \left(\frac{2 \pi}{41} 2^k\right)+21}+\sqrt[3]{22 \sum _{k=1}^{10} \cos \left(\frac{6 \pi}{41}2^k\right)+21}=1\\[1mm]
& \sqrt[3]{28 \sum _{k=1}^{13} \cos \left(\frac{2 \pi}{53} 4^k\right)+27}+\sqrt[3]{28 \sum _{k=1}^{13} \cos \left(\frac{4 \pi}{53} 4^k\right)+27}=1\\[1mm]
& \sqrt[3]{322 \sum _{k=1}^{160} \cos \left(\frac{2 \pi}{641}7^k\right)+321}+\sqrt[3]{322 \sum _{k=1}^{160} \cos \left(\frac{6 \pi }{641}7^k\right)+321}=1
\end{align}
\end{example}

\subsubsection{Composite Numbers of the Form 3p}
For composite numbers of the form $3\times$prime, which are similar to prime numbers of the form $4n+1$, we can find corresponding trigonometric solutions for any integer $k$ greater than $1$ \cite{Chen2125}. Below are examples corresponding to the composite numbers \(15\), \(21\), and \(33\).
\begin{example}
Denote
\begin{align}
& A_{15}=\cos \left(\frac{2 \pi }{15}\right)+\cos \left(\frac{8 \pi }{15}\right)\\
& B_{15}=\cos \left(\frac{4 \pi }{15}\right)+\cos \left(\frac{14 \pi }{15}\right)
\end{align}
then we have
\begin{align}
& \sqrt[2]{2 A_{15}+1} -\sqrt[2]{2 B_{15}+1}=1\\
& \sqrt[3]{4 A_{15}+1} +\sqrt[3]{4 B_{15}+1}=1\\
& \sqrt[4]{6 A_{15}+2} -\sqrt[4]{6 B_{15}+2}=1\\
& \sqrt[5]{10 A_{15}+3} +\sqrt[5]{10 B_{15}+3}=1\\
& \qquad\qquad \vdots \nonumber \\
& \sqrt[99]{437845991669110338052 A_{15}+135301852344706746049}\nonumber \\
& \:\: +\sqrt[99]{437845991669110338052 B_{15}+135301852344706746049}=1 \\
& \qquad\qquad \vdots \nonumber 
\end{align}
\end{example}
\begin{example}
Denote
\begin{align}
& A_{21}=\cos \left(\frac{2 \pi }{21}\right)+\cos \left(\frac{8 \pi }{21}\right)+\cos \left(\frac{10 \pi }{21}\right)\\
& B_{21}=\cos \left(\frac{4 \pi }{21}\right)+\cos \left(\frac{16 \pi }{21}\right)+\cos \left(\frac{20 \pi }{21}\right)
\end{align}
then we have
\begin{align}
& \sqrt[2]{2 A_{21}+5} -\sqrt[2]{2 B_{21}+5}=1\\
& \sqrt[3]{12 A_{21}+5} +\sqrt[3]{12 B_{21}+5}=1\\
& \sqrt[4]{22 A_{21}+30} -\sqrt[4]{22 B_{21}+30}=1\\
& \sqrt[5]{82 A_{21}+55} +\sqrt[5]{82 B_{21}+55}=1\\
& \sqrt[6]{192 A_{21}+205} -\sqrt[6]{192 B_{21}+205}=1\\
& \sqrt[7]{602 A_{21}+480} +\sqrt[7]{602 B_{21}+480}=1\\
& \qquad\qquad \vdots \nonumber \\
& {\small \sqrt[19]{128914922 A_{21}+ 115397280 }+\sqrt[19]{128914922 B_{21}+ 115397280}=1} \\
& {\small \sqrt[20]{359709482 A_{21}+ 322287305 }-\sqrt[20]{359709482 B_{21}+ 322287305}=1} \\
& \qquad\qquad \vdots \nonumber 
\end{align}
\end{example}
\begin{example}
Denote
\begin{align}
& A_{33}=\sum _{k=1}^5 \cos \left(\frac{2 \pi}{33}2^k\right)\\
& B_{33}=\sum _{k=1}^5 \cos \left(\frac{10 \pi}{33}2^k\right)
\end{align}
then we have
\begin{align}
& \sqrt[2]{2 A_{33}+8} -\sqrt[2]{2 B_{33}+8}=1\\
& \sqrt[3]{18 A_{33}+8} +\sqrt[3]{18 B_{33}+8}=1\\
& \sqrt[4]{34 A_{33}+72} -\sqrt[4]{34 B_{33}+72}=1\\
& \sqrt[5]{178 A_{33}+136} +\sqrt[5]{178 B_{33}+136}=1\\
& \qquad\qquad \vdots \nonumber \\
& \sqrt[33]{91888738166330226 A_{33}+108990567176139656}\nonumber \\
& \:\: +\sqrt[33]{91888738166330226 B_{33}+108990567176139656}=1 \\
& \sqrt[34]{309869872518609538 A_{33}+367554952665320904}\nonumber \\
& \:\: -\sqrt[34]{309869872518609538 B_{33}+367554952665320904}=1 \\
& \qquad\qquad \vdots \nonumber 
\end{align}
\end{example}
Below are examples corresponding to several other composite numbers of the form $3\times$prime, with only the formulas for their cube roots listed for simplicity:
\begin{example}
\begin{align} 
& \sqrt[3]{8 \sum _{k=1}^6 \cos \left(\frac{8 \pi }{39} 4^k\right)+3}+\sqrt[3]{8 \sum _{k=1}^6 \cos \left(\frac{10 \pi  }{39}4^k\right)+3}=1\\[1mm]
& \sqrt[3]{10 \sum _{k=1}^8 \cos \left(\frac{2 \pi  }{51}2^k\right)+4}+\sqrt[3]{10 \sum _{k=1}^8 \cos \left(\frac{10 \pi  }{51}2^k\right)+4}=1\\[1mm]
& \sqrt[3]{30 \sum _{k=1}^9 \cos \left(\frac{2 \pi}{57}2^k\right)+14}+\sqrt[3]{30 \sum _{k=1}^9 \cos \left(\frac{10 \pi}{57}2^k\right)+14}=1\\[1mm]
& \sqrt[3]{36 \sum _{k=1}^{11} \cos \left(\frac{2 \pi}{69}4^k\right)+17}+\sqrt[3]{36 \sum _{k=1}^{11} \cos \left(\frac{14\pi}{69}4^k\right)+17}=1\\[1mm]
& \sqrt[3]{16 \sum _{k=1}^{14} \cos \left(\frac{2 \pi}{87}4^k\right)+7}+ \sqrt[3]{16 \sum _{k=1}^{14} \cos \left(\frac{4 \pi}{87}4^k\right)+7}=1\\[1mm]
& \sqrt[3]{48 \sum _{k=1}^{15} \cos \left(\frac{2 \pi }{93}7^k\right)+23}+\sqrt[3]{48 \sum _{k=1}^{15} \cos \left(\frac{4 \pi }{93}7^k\right)+23}=1
\end{align}
\end{example}

\clearpage

\section{The Normalized GPTE Problem}
In this chapter, we investigate the normalized version of the Generalized Prouhet-Tarry-Escott (GPTE) problem. This approach introduces a novel method for studying the GPTE problem, which has not been previously explored in the literature. The normalized version transforms the original GPTE problem, which is defined over the integers, into an equivalent problem over the real numbers within the interval $[0,1]$. The primary objective of this research is to determine the exact upper and lower bounds for each variable, thereby enhancing the efficiency of computer searches for integer solutions to the GPTE problem.

\subsection{The Normalized GPTE System}
\begin{definition} \textnormal{\textbf{(The Normalized GPTE System)}}\\[1mm]
\label{definition_Normalized_GPTE}
Let \(\{a_{1}, a_{2}, \dots, a_{n+1}\}\) and \(\{b_{1}, b_{2}, \dots, b_{n+1}\}\) be two sets of non-negative real numbers, ordered as follows:
\begin{align}
& 0 \leq a_1 \leq a_2 \leq \cdots \leq a_{n+1}, \\
& 0 \leq b_1 \leq b_2 \leq \cdots \leq b_{n+1}.
\end{align}

Without loss of generality, assume that \(a_{n+1} > b_{n+1}\). For given distinct integers \(\{k_1, k_2, \dots, k_n\}\), if the following GPTE system holds:
\begin{align}
\left[ a_{1}, a_{2}, \dots, a_{n+1} \right]^{k} = \left[ b_{1}, b_{2}, \dots, b_{n+1} \right]^{k}, \quad (k = k_1, k_2, \dots, k_n),
\end{align}
then the \textbf{normalized GPTE system} can be defined as:
\begin{align}
\label{equation_normalized_GPTE}
\left[ \alpha_{1}, \alpha_{2}, \dots, \alpha_{n+1} \right]^{k} = \left[ \beta_{1}, \beta_{2}, \dots, \beta_{n+1} \right]^{k}, \quad (k = k_1, k_2, \dots, k_n).
\end{align}
where for \( i = 1, 2, 3, \dots, n+1 \),
\begin{align}
& \alpha_i=\frac{a_i}{a_{n+1}},\\
& \beta_i=\frac{b_i}{a_{n+1}}.
\end{align}
\end{definition}
According to Definition \ref{definition_Normalized_GPTE}, the normalized GPTE system \eqref{equation_normalized_GPTE} exhibits these fundamental properties:
\begin{enumerate}
    \item The distinct integer sequence \(\{k_1, k_2, \dots, k_n\}\) contains \(n\) elements, while each set of non-negative real numbers contains \(n+1\) elements.
    \item The \(n+1\) elements in each set are non-negative real numbers bounded above by 1, ordered as follows:
    \begin{align}
        & 0 \leq \alpha_1 \leq \alpha_2 \leq \cdots \leq \alpha_{n+1} = 1, \\
        & 0 \leq \beta_1 \leq \beta_2 \leq \cdots \leq \beta_{n+1} < 1.
    \end{align}
    \item \(\alpha_{n+1}\) is normalized to unity.
\end{enumerate}
Throughout this chapter, all normalized GPTE systems maintain these three fundamental properties.

\begin{problem} \textnormal{\textbf{(The Normalized GPTE Problem)}}\\[1mm]
\label{problem_Normalized_GPTE}
The \textbf{normalized GPTE problem} can be stated as follows: Given a sequence of distinct integers \(\{k_1, k_2, \dots, k_n\}\), for the normalized GPTE system \eqref{equation_normalized_GPTE}, how can the upper and lower bounds of each \(\alpha_i\) and \(\beta_i\) (for \(i = 1, 2, \dots, n+1\)) be determined?
\end{problem}
The purpose of our research on Problem \ref{problem_Normalized_GPTE} is to develop a general and efficient algorithm to facilitate the search for ideal non-negative integer solutions to the GPTE problem with the assistance of computer search. Building on years of research, we have proposed six conjectures related to Problem \ref{problem_Normalized_GPTE}, which determine the upper and lower bounds for each variable from different perspectives.

To date, we have not been able to provide complete proofs for these six conjectures, nor have we found any counterexamples. In fact, based on the assumption that these conjectures hold, we have already efficiently identified a large number of ideal solutions to the GPTE problem using computational methods.

If these conjectures are proven correct, the derived boundary conditions would enable us to obtain all ideal non-negative integer solutions for the relevant GPTE problems. Conversely, even if the conjectures are only partially valid, the resulting boundary conditions would still yield a subset of ideal solutions, namely those solutions that satisfy the conjectures. Therefore, applying these conjectures in the search for integer solutions to the GPTE problem remains highly practical.

The first conjecture related to the normalized GPTE problem is as follows:
\begin{conjecture}\textnormal{\textbf{(Interlacing Conjecture)}}\\[1mm]
\label{conjecture_Interlacing}
Let \(\{\alpha_{1}, \alpha_{2}, \dots, \alpha_{n+1}\}\) and \(\{\beta_{1}, \beta_{2}, \dots, \beta_{n+1}\}\) be two sets of non-negative real numbers bounded above by 1, ordered as follows:
\begin{align}
& 0 \leq \alpha_1 \leq \alpha_2 \leq \cdots \leq \alpha_{n+1}=1, \\
& 0 \leq \beta_1 \leq \beta_2 \leq \cdots \leq \beta_{n+1}<1.
\end{align}
For given distinct integers \(\{k_1, k_2, \dots, k_n\}\), if the following system of equations holds:
\begin{align}
\left[ \alpha_{1}, \alpha_{2}, \dots, \alpha_{n+1} \right]^{k} = \left[ \beta_{1}, \beta_{2}, \dots, \beta_{n+1} \right]^{k}, \quad (k = k_1, k_2, \dots, k_n).
\end{align}
then the elements of the two sets interlace in the following manner:
\begin{flalign}
& \;\;  \small{0 \leq \alpha_1 < \beta_1 \leq \beta_2 < \alpha_2 \leq \alpha_3 <...< \beta_n \leq \beta_{n+1} < \alpha_{n+1}=1} \; \textup{(\(n\) is odd)}, & \\
& \;\;  \small{0 \leq \beta_1 < \alpha_1 \leq \alpha_2 < \beta_2 \leq \beta_3 <...< \beta_n \leq \beta_{n+1} < \alpha_{n+1} =1} \; \textup{(\(n\) is even)}. &
\end{flalign}
\end{conjecture}
Based on the assumption that Conjecture \ref{conjecture_Interlacing} is true, we can derive the following corollary:
\begin{corollary}
\label{corollary__Interlacing}
Let \(\{a_{1}, a_{2}, \dots, a_{n+1}\}\) and \(\{b_{1}, b_{2}, \dots, b_{n+1}\}\) be two sets of non-negative integer numbers, ordered as follows:
\begin{align}
& 0 \leq a_1 \leq a_2 \leq \cdots \leq a_{n+1}, \\
& 0 \leq b_1 \leq b_2 \leq \cdots \leq b_{n+1}.
\end{align}
Without loss of generality, assume that \(a_{n+1} > b_{n+1}\). For given distinct integers \(\{k_1, k_2, \dots, k_n\}\), if the following system of equations holds:
\begin{align}
\left[ a_{1}, a_{2}, \dots, a_{n+1} \right]^{k} = \left[ b_{1}, b_{2}, \dots, b_{n+1} \right]^{k}, \quad \text{for } k = k_1, k_2, \dots, k_n,
\end{align}
then the elements of the two sets interlace in the following manner:
\begin{flalign}
& \quad 0 \leq a_1 < b_1 \leq b_2 < a_2 \leq a_3 < \dots < b_n \leq b_{n+1} < a_{n+1} \quad \textup{(if \(n\) is odd)}, & \\
& \quad 0 \leq b_1 < a_1 \leq a_2 < b_2 \leq b_3 < \dots < b_n \leq b_{n+1} < a_{n+1} \quad \textup{(if \(n\) is even)}. &
\end{flalign}
\end{corollary}
The above Corollary \ref{corollary__Interlacing} demonstrates the interlacing pattern between the elements \(\{a_i\}\) and \(\{b_i\}\) of the ideal non-negative integer solutions to the GPTE problem. As a special case of the GPTE problem, the corresponding results for the ideal solutions of the PTE problem (where \(k = 1, 2, 3, \dots, n\)) have been verified to hold \cite{Borwein1994, Caley12}, and are named as the \textbf{Interlacing Theorem} for PTE.

\subsection{Exact Bounds of the Normalized GPTE System}
Based on extensive numerical investigations, we have discovered that for any given distinct integers \(\{k_1, k_2, \dots, k_n\}\), the upper and lower bounds of each element in the normalized GPTE system \eqref{equation_normalized_GPTE} can be exactly determined in descending order. Specifically, setting \(\alpha_{n+1} = 1\) allows the exact determination of the bounds of \(\beta_{n+1}\). Given a value of \(\beta_{n+1}\) within its bounds, the bounds of \(\beta_{n}\) can then be exactly determined based on \(\beta_{n+1}\). Similarly, with \(\beta_{n+1}\) and \(\beta_{n}\) specified within their respective bounds, the bounds of \(\alpha_{n}\) can be exactly determined based on \(\beta_{n+1}\) and \(\beta_{n}\), and so on.

At present, we are unable to provide strict proofs for these bounds. Thus, the following section proposes these bounds as conjectures.

\subsubsection{Upper and Lower Bounds of \(\beta_{n+1}\)}
\begin{conjecture}
\label{conjecture_beta_n+1}\textnormal{\textbf{(Exact Bounds of \(\boldsymbol{\beta_{n+1}}\))}}\\[1mm]
Let \(\{\alpha_{1}, \alpha_{2}, \dots, \alpha_{n+1}\}\) and \(\{\beta_{1}, \beta_{2}, \dots, \beta_{n+1}\}\) be two sets of non-negative real numbers bounded above by 1, ordered as follows:
\begin{align}
& 0 \leq \alpha_1 \leq \alpha_2 \leq \cdots \leq \alpha_{n+1}=1, \\
& 0 \leq \beta_1 \leq \beta_2 \leq \cdots \leq \beta_{n+1}<1.
\end{align}
For given distinct integers \(\{k_1, k_2, \dots, k_n\}\), if the following system of equations holds:
\begin{align}
\left[ \alpha_{1}, \alpha_{2}, \dots, \alpha_{n+1} \right]^{k} = \left[ \beta_{1}, \beta_{2}, \dots, \beta_{n+1} \right]^{k}, \quad (k = k_1, k_2, \dots, k_n).
\end{align}
we conjecture that:
\begin{enumerate}
\item Let \( \beta_{n+1, \min}\) denote the lower bound of \( \beta_{n+1} \). Then, \( \beta_{n+1} \) achieves its minimum value \( \beta_{n+1, \min}\) under the following conditions:
\begin{equation}
\label{conjecture_beta_n+1_condition}
\begin{cases}
& \beta_{n+1} =  \beta_{n},\, \quad  \beta_{n-1} =  \beta_{n-2},\, \quad \dots \\
& \alpha_{n} = \alpha_{n-1}, \quad \alpha_{n-2} = \alpha_{n-3}, \quad \dots \\
& \alpha_1 = \varepsilon \quad \textup{(if \(n\) is odd)} \quad \textup{or} \quad \beta_1 = \varepsilon \quad \textup{(if \(n\) is even)}.
\end{cases}
\end{equation}
where \(\varepsilon\) is a positive infinitesimal real number.
\item Let \( \beta_{n+1, \max} \) denote the upper bound of \( \beta_{n+1} \). Then, \( \beta_{n+1, \max} = 1 - \delta \), where \(\delta\) is a positive infinitesimal real number.
\end{enumerate}

\end{conjecture}

For different sequences \(\{k_1, k_2, \dots, k_n\}\), the value of \(\beta_{n+1, \min}\) varies. When \(\{k_1, k_2, \dots, k_n\}\) are all positive integers, \(\varepsilon\) in equation \eqref{conjecture_beta_n+1_condition} can be directly set to 0. 

\begin{example} \textbf{Lower Bound of \( \boldsymbol{\beta_{n+1}} \) for \( \boldsymbol{(k=1,2,3,\dots, n)} \)} \\[1mm]
When \( k = 1, 2, 3, \dots, n \), this corresponds to the case of Identity \ref{identityT1}. According to equation \eqref{PTE_bmin}, we have:
\begin{align}
\beta_{n+1, \min} = \sin^2\left( \frac{n \pi}{2n + 2} \right)
\end{align}
\end{example}

\begin{example} \textbf{Lower Bound of \( \boldsymbol{\beta_{n+1}} \) for \( \boldsymbol{(k=1,2,3,4,6)} \)} \\[2mm]
For the given normalized GPTE system:
\begin{flalign}
&\quad \left[ \alpha_{1}, \alpha_{2}, \alpha_{3}, \alpha_{4}, \alpha_{5},\alpha_{6} \right]^{k} = \left[ \beta_{1}, \beta_{2}, \beta_{3}, \beta_{4}, \beta_{5},\beta_{6} \right]^{k}, \quad (k = 1,2,3,4,6).&
\end{flalign}
\noindent
Based on \eqref{conjecture_beta_n+1_condition}, \( \beta_{6} =  \beta_{6, \min}\) under the following conditions:
\begin{equation}
\begin{cases}
& \beta_6 =  \beta_5, \:\quad\beta_4 =  \beta_3, \:\quad \beta_2 =  \beta_1, \\
& \alpha_5 = \alpha_4, \quad \alpha_3 = \alpha_2, \\
& \alpha_1 = 0.
\end{cases}
\end{equation}
This leads to the following system of five equations in five variables:
\begin{align}
\left[ 0, x_{3}, x_{3}, x_{5}, x_{5}, 1 \right]^{k} = \left[ y_{2}, y_{2}, y_{4}, y_{4},y_{6}, y_6 \right]^{k}, \quad (k = 1,2,3,4,6).
\end{align}
Specifically,
\begin{equation}
\begin{cases}
\label{exampleD1n5}
& 2 x_3+2 x_5+1=2 y_2+2 y_4+2 y_6,\\
& 2 x_3^2+2 x_5^2+1=2 y_2^2+2 y_4^2+2 y_6^2,\\
& 2 x_3^3+2 x_5^3+1=2 y_2^3+2 y_4^3+2 y_6^3,\\
& 2 x_3^4+2 x_5^4+1=2 y_2^4+2 y_4^4+2 y_6^4,\\
& 2 x_3^6+2 x_5^6+1=2 y_2^6+2 y_4^6+2 y_6^6.
\end{cases}
\end{equation}
By solving \eqref{exampleD1n5}, we obtain \(\{x_3, x_5, y_2, y_4, y_6\}\) satisfying the conditions \( x_3 < x_5\) and \(0<y_2<y_4 < y_6\). Then, we have $ \beta_{6, \min}= y_6 $.
\end{example}
To solve the system of equations \eqref{exampleD1n5}, there are at least two approaches, both of which can be facilitated with the assistance of mathematical software such as Mathematica. \\[1mm]
\indent
$\bullet$  The first approach utilizes Mathematica's \texttt{NSolve} function to obtain numerical solutions as follows:
\begin{equation}
\begin{cases}
& x_3=0.262071224731443764797745233458\dots,\\ 
& x_5=0.762683653666752442709310793796\dots,\\
& y_2=0.070942067935781368407977852387\dots,\\
& y_4=0.516491266199592107558386699449\dots,\\
& y_6=0.937321544262822731540691475418\dots.
\end{cases}
\end{equation}
This yields the theoretical lower bound:
\begin{align}
\label{k12346b6min}
& \beta_{6, \min}= y_6 = 0.937321544262822731540691475418\dots
\end{align}
Then, by applying Mathematica's \texttt{RootApproximant} and \texttt{MinimalPolynomial} function, we determine that \(\beta_{6, \min}\) is a real root of the following equation:
\begin{align}
\label{k12346y6}
& 512 y_6^6-256 y_6^5-1504 y_6^4+1968 y_6^3-830 y_6^2+120 y_6-5=0
\end{align}

\indent
$\bullet$ The second approach for solving the system \eqref{exampleD1n5} is to derive the polynomial equation \eqref{k12346y6} directly by applying Identity \ref{identity_GNI2}, namely the Second Generalization of the Girard-Newton Identities. By transforming \eqref{exampleD1n5}, we obtain:
\begin{align}
\label{examplek12346e4}
(1-2 y_6^k)/2 =y_2^k+y_4^k-x_3^k-x_5^k,\quad (k = 1,2,3,4,6).
\end{align}
Let
\begin{align}
\label{examplek12346e5}
P_k=y_2^k+y_4^k-x_3^k-x_5^k,\quad (k = 1,2,3,4,5,6).
\end{align}
and define
\begin{equation}
\begin{aligned}
\label{examplek12346e6}
& S_1=(P_1)/1,\\
& S_2=(P_2 + S_1 P_1)/2,\\
& S_3=(P_3 + S_1 P_2 + S_2 P_1)/3,\\
& S_4=(P_4 + S_1 P_3 + S_2 P_2+ S_3 P_1)/4,\\
& S_5=(P_5 + S_1 P_4 + S_2 P_3+ S_3 P_2+ S_4 P_1)/5,\\
& S_6=(P_6 + S_1 P_5 + S_2 P_4+ S_3 P_3+ S_4 P_2+ S_5 P_1)/6.
\end{aligned}
\end{equation}
Then, according to Identity \ref{identity_GNI2}, it follows that:
\begin{equation}
\begin{aligned}
\label{examplek12346e7a}
\begin{vmatrix}
 S_1 & S_2 & S_3 \\
 S_2 & S_3 & S_4 \\
 S_3 & S_4 & S_5 \\
\end{vmatrix}=0
\end{aligned}
\end{equation}
and
\begin{equation}
\begin{aligned}
\label{examplek12346e7b}
\begin{vmatrix}
 S_2 & S_3 & S_4 \\
 S_3 & S_4 & S_5 \\
 S_4 & S_5 & S_6 \\
\end{vmatrix}=0
\end{aligned}
\end{equation}
Based on \eqref{examplek12346e4}, we replace \eqref{examplek12346e5} with the following expression:
\begin{align}
\label{examplek12346e8}
P_k=(1-2 y_6^k)/2,\quad (k = 1,2,3,4,6).
\end{align}
Substituting \eqref{examplek12346e8} into \eqref{examplek12346e6} and \eqref{examplek12346e7a}, we obtain:
\begin{align}
\label{examplek12346p5}
P_5=\frac{-4096 y_6^7+4096 y_6^6-512 y_6^5+160 y_6^3+1808 y_6^2-1958 y_6+251}{512 \left(8 y_6^2-8 y_6+1\right)}
\end{align}
Substituting \eqref{examplek12346p5} into \eqref{examplek12346e7b}, we obtain a result that is consistent with the one derived by the first approach, as shown in \eqref{k12346y6}:
\begin{align}
512 y_6^6 - 256 y_6^5 - 1504 y_6^4 + 1968 y_6^3 - 830 y_6^2 + 120 y_6 - 5 = 0 \tag*{\eqref{k12346y6}}
\end{align}
By adding the condition \(1/2 < y_6^6 < 1\) (obtained from Corollary \ref{corollary_Conservative_Bounds}) to equation \eqref{k12346y6}, and utilizing Mathematica's \texttt{NSolve} function, we arrive at the same result in \eqref{k12346b6min}:
\begin{align}
& \beta_{6, \min}= y_6 = 0.937321544262822731540691475418\dots  \tag*{\eqref{k12346b6min}}
\end{align}

\indent
$\bullet$ As a validation of the result in \eqref{k12346b6min}, the ratio \( b_6 / a_6 = 18364/19534 = 0.940104 \dots \) in the numerical example given below is the closest to \( \beta_{6, \min} \) among tens of thousands of ideal non-negative integer solutions of type \((k=1,2,3,4,6)\) that we have discovered through various methods so far.
\begin{align}
& [139,4144,6439,14584,15394,19534]^k \nonumber\\
& \quad =[1219,1714,10084,10534,18319,18364]^k,\quad (k=1,2,3,4,6).
\end{align}

\begin{example}
\textbf{Lower Bound of \( \boldsymbol{\beta_{n+1}} \) for \( \boldsymbol{(k=k_1, k_2, k_3, k_4)} \), \( \boldsymbol{k \in \mathbb{Z}^+} \)} \\[1mm]
For the distinct positive integer sequence \(\{k_1, k_2, k_3, k_4\}\), let \(\{\alpha_{1}, \alpha_{2},\\ \alpha_{3}, \alpha_{4}, \alpha_{5}\}\) and \(\{\beta_{1}, \beta_{2},  \beta_{3}, \beta_{4}, \beta_{5} \}\) be two sets of non-negative real numbers bounded above by 1, ordered as follows:
\begin{align}
& 0 \leq \alpha_1 \leq \alpha_2 \leq \alpha_3 \leq \alpha_4 \leq \alpha_5=1, \\
& 0 \leq \beta_1 \leq \beta_2 \leq \beta_3 \leq \beta_4 \leq \beta_5<1.
\end{align}
These sets satisfy:
\begin{align}
\left[ \alpha_{1}, \alpha_{2}, \alpha_{3}, \alpha_{4}, \alpha_{5} \right]^{k} = \left[ \beta_{1}, \beta_{2}, \beta_{3}, \beta_{4}, \beta_{5} \right]^{k}, \quad (k = k_1, k_2, k_3, k_4).
\end{align}
Then, based on \eqref{conjecture_beta_n+1_condition}, we obtain the following system of four equations in four variables:
\begin{align}
\left[ x_{2}, x_{2}, x_{4}, x_{4}, 1 \right]^{k} = \left[ 0, y_{3}, y_{3}, y_{5}, y_{5} \right]^{k}, \quad (k = k_1, k_2, k_3, k_4).
\end{align}
namely
\begin{equation}
\begin{cases}
\label{exampleD1n4}
& 2 x_2^{k_1}+2 x_4^{k_1}+1=2 y_3^{k_1}+2 y_5^{k_1},\\
& 2 x_2^{k_2}+2 x_4^{k_2}+1=2 y_3^{k_2}+2 y_5^{k_2},\\
& 2 x_2^{k_3}+2 x_4^{k_3}+1=2 y_3^{k_3}+2 y_5^{k_3},\\
& 2 x_2^{k_4}+2 x_4^{k_4}+1=2 y_3^{k_4}+2 y_5^{k_4}.
\end{cases}
\end{equation}
For a given \(\{k_1, k_2, k_3, k_4\}\), by solving \eqref{exampleD1n4}, we obtain \(\{x_2, x_4, y_3, y_5\}\) satisfying the conditions \(0 < x_2 < x_4\) and \(y_3 < y_5\). Then, we have
\begin{align}
& \qquad \beta_{5, \min}= y_5.
\end{align}
\end{example}
As described in the previous example, the solution of \eqref{exampleD1n4} can be obtained with the assistance of mathematical software such as Mathematica. Below are the calculated results for \(\beta_{5, \min}\) in some specific cases of \(\{k_1, k_2, k_3, k_4\}\). \\[1mm]
\indent
$\bullet$ For type \( (k=1,2,3,4) \), using Mathematica's \texttt{NSolve} function, we obtain
\begin{align}
\label{k1234b5min}
\beta_{5, \min} &= 0.904508497187473712051146708591\dots. 
\end{align}
Furthermore, using Mathematica's \texttt{RootApproximant} function, we derive:
\begin{align}
& \beta_{5, \min} = \left( \frac{5}{8} + \frac{\sqrt{5}}{8} \right) = \sin^2\left(\frac{2 \pi}{5}\right)
\end{align}
Then, by applying Mathematica's \texttt{MinimalPolynomial} function, we determine that \(\beta_{5, \min}\) is a real root of the following equation
\begin{align}
\label{k1234y5}
& 16 y_5^2 - 20 y_5 + 5 = 0
\end{align}
By adding the condition \((1/2)^{1/4} < y_5 < 1\) (obtained from Corollary \ref{corollary_Conservative_Bounds}) to equation \eqref{k1234y5}, we arrive at the result in \eqref{k1234b5min}.\\[1mm]
\indent Equation \eqref{k1234tr} in Chapter~4 presents a numerical example of this type, where \(b_5/a_5\) is extremely close to \(\beta_{5, \min}\).\\[1mm]
\indent
$\bullet$ For type $ (k=1,2,3,5) $, using Mathematica's \texttt{NSolve} and \texttt{RootApproximant} function, we have
\begin{align}
& \beta_{5, \min}= 0.913362095721694748879941086068\dots\nonumber \\
& \qquad \;\;  = \frac{1}{8} \left(1+\sqrt{17}+\sqrt{2 \left(23-5 \sqrt{17}\right)}\right)
\end{align}
\indent
Then, by applying Mathematica's \texttt{MinimalPolynomial} function, we determine that \(\beta_{5, \min}\) is a real root of the following equation
\begin{align}
\label{k1235y5}
32 y_5^4-16 y_5^3-60 y_5^2+58 y_5-13 &= 0
\end{align}
\indent The ratio \( b_5 / a_5 = 36932185/40351003=0.915273 \dots \) in the numerical example given below is the closest to \( \beta_{5, \min} \) among hundreds of thousands of ideal non-negative integer solutions of type \((k=1,2,3,5)\) we have discovered through various methods so far.
\begin{align}
& [552455,14578375,15805643,36932185,36932185]^k \nonumber\\
& \quad =[3795949,5636851,26325083,28691957,40351003]^k.
\end{align}
\indent
$\bullet$ For type \( (k=1,2,3,6) \), we have
\begin{align}
\beta_{5, \min} &= 0.920915720822930625728771732305\dots
\end{align}
which is a real root of the following equation:
\begin{align}
\label{k1236y5}
1024 y_5^6 - 1280 y_5^5 + 192 y_5^4 - 1728 y_5^3 + 3528 y_5^2 - 2120 y_5 + 395 = 0
\end{align}
\indent
$\bullet$ For type $ (k=1,2,3,7) $, we have
\begin{align}
\beta_{5, \min}= 0.927408240581916072725667749532\dots
\end{align}
which is a real root of the following equation:
\begin{align}
\label{k1237y5}
& 4096 y_5^8-8192 y_5^7+4608 y_5^6-512 y_5^5-6688 y_5^4\nonumber \\
& \qquad +19088 y_5^3-19572 y_5^2+8510 y_5-1325 = 0
\end{align}
\indent
$\bullet$ For type $ (k=1,3,5,7) $, we have
\begin{align}
& \beta_{5, \min}= 0.939692620785908384054109277324\dots =\cos \left(\frac{\pi }{9}\right)
\end{align}
which is a real root of the following equation:
\begin{align}
\label{k1357y5}
& 8 y_5^3-6 y_5+1 = 0
\end{align}

\indent
$\bullet$ For type $ (k=1,3,5,9) $, we have
\begin{align}
& \beta_{5, \min}= 0.946756706993111318358209122631\dots
\end{align}
which is a real root of the following equation:
\begin{align}
\label{k1359y5}
& 256 y_5^8-256 y_5^7+128 y_5^6-128 y_5^5-240 y_5^4 \nonumber \\
& \qquad +240 y_5^3+52 y_5^2-38 y_5-7 = 0
\end{align}

\indent
$\bullet$ For type $ (k=1,3,7,9) $, we have
\begin{align}
& \beta_{5, \min}= 0.951855934042876264740475099068\dots
\end{align}
which is a real root of the following equation:
\begin{align}
\label{k1379y5}
& 4096 y_5^{12}+6144 y_5^{11}+2048 y_5^{10}+512 y_5^9+256 y_5^8-1536 y_5^7-2496 y_5^6 \nonumber \\
& \qquad -1344 y_5^5-208 y_5^4+88 y_5^3+344 y_5^2+106 y_5-35 = 0
\end{align}

\indent
$\bullet$ For type $ (k=2,4,6,8) $, we have
\begin{align}
& \beta_{5, \min}= 0.951056516295153572116439333379\dots = \cos \left(\frac{\pi }{10}\right)
\end{align}
which is a real root of the following equation:
\begin{align}
\label{k2468y5}
& 16 y_5^4 - 20 y_5^2 + 5 = 0
\end{align}

\indent
$\bullet$ For type $ (k=1,2,4,5) $, we have
\begin{align}
& \beta_{5, \min}= 0.920185510514071805157487190697\dots
\end{align}
which is a real root of the following equation:
\begin{align}
\label{k1245y5}
& 8192 y_5^6-11264 y_5^5+768 y_5^4+1792 y_5^3+656 y_5^2+216 y_5-243 = 0
\end{align}

\indent
$\bullet$ For type $ (k=1,3,4,5) $, we have
\begin{align}
& \beta_{5, \min}= 0.925684823157598579141673321256\dots
\end{align}
which is a real root of the following equation:
\begin{align}
\label{k1345y5}
& 9728 y_5^8-26112 y_5^7+28288 y_5^6-20096 y_5^5+10000 y_5^4 \nonumber \\
& \qquad +752 y_5^3-3500 y_5^2+924 y_5+29 = 0
\end{align}

\indent
$\bullet$ For type $ (k=2,3,4,5) $, we have
\begin{align}
& \beta_{5, \min}= 0.930495269093988042489388711803\dots
\end{align}
which is a real root of the following equation:
\begin{align}
\label{k2345y5}
& 5120000 y_5^{18}-18432000 y_5^{17}+20352000 y_5^{16}-1208320 y_5^{15}-9255936 y_5^{14} \nonumber \\
& \qquad +1296384 y_5^{13}+2943488 y_5^{12}-2365440 y_5^{11}+3712896 y_5^{10}\nonumber \\
& \qquad -2169600 y_5^9-772560 y_5^8+924480 y_5^7+66000 y_5^6-375840 y_5^5 \nonumber \\
& \qquad +197100 y_5^4-25200 y_5^3-8100 y_5^2+675 = 0
\end{align}

\begin{example}
\textbf{Lower Bound of \( \boldsymbol{\beta_{n+1}} \) for \( \boldsymbol{(k=0,1,2,3,4)} \)} \\[1mm]
Let \(\{\alpha_{1}, \alpha_{2}, \dots,  \alpha_{6}\}\) and \(\{\beta_{1}, \beta_{2}, \dots, \beta_{6}\}\) be two sets of non-negative real numbers bounded above by 1, ordered as follows:
\begin{align}
& 0 \leq \alpha_1 \leq \alpha_2 \leq \dots \leq \alpha_6=1, \\
& 0 \leq \beta_1 \leq \beta_2 \leq \dots \leq \beta_6<1.
\end{align}
These sets satisfy:
\begin{flalign}
& \quad \left[ \alpha_{1}, \alpha_{2}, \alpha_{3}, \alpha_{4}, \alpha_{5}, \alpha_{6}  \right]^{k} = \left[ \beta_{1}, \beta_{2}, \beta_{3}, \beta_{4}, \beta_{5}, \beta_{6} \right]^{k}, \quad (k = 0,1,2,3,4).&
\end{flalign}
Then, according to \eqref{conjecture_beta_n+1_condition}, \( \beta_{6, \min}\) is determined by the following system of five equations in five variables:
\begin{align}
\left[\varepsilon_1, x_{3}, x_{3}, x_{5}, x_{5},1  \right]^{k} = \left[ y_{2}, y_{2}, y_{4}, y_{4}, y_{6}, y_{6} \right]^{k}, \quad (k =0,1,2,3,4).
\end{align}
where \(\varepsilon_1\) is a positive infinitesimal real number. Namely,
\begin{equation}
\begin{cases}
\label{exampleD2n5}
& \qquad\quad\;\; \varepsilon_1 \cdot x_3^2 \cdot x_5^2= y_2^2 \cdot y_4^2 \cdot y_6^2,\\
& \varepsilon_1+2 x_3+2 x_5+1=2 y_2+2 y_4+2 y_6,\\
& \varepsilon_1^2+2 x_3^2+2 x_5^2+1=2 y_2^2+2 y_4^2+2 y_6^2,\\
& \varepsilon_1^3+2 x_3^3+2 x_5^3+1=2 y_2^3+2 y_4^3+2 y_6^3,\\
& \varepsilon_1^4+2 x_3^4+2 x_5^4+1=2 y_2^4+2 y_4^4+2 y_6^4.
\end{cases}
\end{equation}
\end{example}

By solving \eqref{exampleD2n5} with \(\varepsilon_1 = 10^{-50}\), we obtain \(\{x_3, x_5, y_2, y_4, y_6\}\) satisfying the conditions \( x_3 < x_5\) and \(0<y_2<y_4 < y_6\). 
\begin{equation}
\begin{cases}
& x_3=0.095491502812526287948853327588\dots,\\ 
& x_5=0.654508497187473712051146722411\dots,\\
& y_2=2.000000000000000000000000640000\dots \times 10^{-26},\\
& y_4=0.345491502812526287948853317588\dots,\\
& y_6=0.904508497187473712051146712411\dots.
\end{cases}
\end{equation}
Then, we have
\begin{align}
& \beta_{6, \min}=y_6= 0.904508497187473712051146712411\dots
\end{align}

\indent In numerical example given below, \( b_6 / a_6 = 921294/1012966=0.909501 \dots \), which is the closest to \( \beta_{6, \min}\) among the ideal non-negative integer solutions of type \((k=0,1,2,3,4)\) we have found so far.
\begin{align}
& [2937, 102794, 125712, 669196, 690477, 1012966]^k \nonumber  \\
& \qquad  =[11122, 14396, 297597, 440016, 919657, 921294]^k, \nonumber  \\
& \qquad\qquad\qquad\qquad\qquad\qquad\qquad\qquad (k=0,1,2,3,4).
\end{align}

\begin{example} \textbf{Lower Bound of \( \boldsymbol{\beta_{n+1}} \) for \( \boldsymbol{(k=-3,-2,-1,0,1,2,3)} \)} \\[1mm]
Let \(\{\alpha_{1}, \alpha_{2}, \dots,  \alpha_{8}\}\) and \(\{\beta_{1}, \beta_{2}, \dots, \beta_{8}\}\) be two sets of non-negative real numbers bounded above by 1, ordered as follows:
\begin{align}
& 0 \leq \alpha_1 \leq \alpha_2 \leq \dots \leq \alpha_8=1, \\
& 0 \leq \beta_1 \leq \beta_2 \leq \dots \leq \beta_8<1.
\end{align}
These sets satisfy:
\begin{align}
& \left[ \alpha_{1}, \alpha_{2}, \alpha_{3}, \alpha_{4}, \alpha_{5}, \alpha_{6}, \alpha_{7}, \alpha_{8} \right]^{k} = \left[ \beta_{1}, \beta_{2}, \beta_{3}, \beta_{4}, \beta_{5}, \beta_{6},\beta_{7}, \beta_{8}  \right]^{k}, \nonumber \\ 
& \qquad\qquad\qquad\qquad\qquad\qquad\qquad\quad (k = -3,-2,-1,0,1,2,3).
\end{align}
Based on \eqref{conjecture_beta_n+1_condition}, we obtain the following system of seven equations in seven variables:
\begin{align}
& \left[\varepsilon_1, x_{3}, x_{3}, x_{5}, x_{5},x_{7}, x_{7},1  \right]^{k} = \left[ y_{2}, y_{2}, y_{4}, y_{4}, y_{6}, y_{6}, y_{8}, y_{8} \right]^{k}, \nonumber \\ 
& \qquad\qquad\qquad\qquad\qquad\qquad\quad (k = -3,-2,-1,0,1,2,3).
\end{align}
where \(\varepsilon_1\) is a positive infinitesimal real number. Namely,
\begin{equation}
\begin{cases}
\label{examplek0123n123e3}
\qquad\qquad\;\;\; \varepsilon_1 \cdot x_3^2 \cdot x_5^2 \cdot x_7^2= y_2^2 \cdot y_4^2 \cdot y_6^2 \cdot y_8^2,\\
\varepsilon_1^k+2 x_3^k+2 x_5^k+2 x_7^k+1=2 y_2^k+2 y_4^k+2 y_6^k+2 y_8^k, \;\; (k = -3,-2,-1,1,2,3).
\end{cases}
\end{equation}
By transforming equation \eqref{examplek0123n123e3}, we obtain:
\begin{equation}
\begin{cases}
\label{examplek0123n123e4}
\qquad\qquad\quad\;\, \displaystyle\frac{y_8}{\sqrt{\varepsilon _1}}=\frac{x_3 x_5 x_7}{y_2 y_4 y_6},\\
\displaystyle\frac{1}{2} \left(1+\varepsilon_1^k-2 y_8^k\right)=y_2^k+y_4^k+y_6^k-x_3^k-x_5^k-x_7^k, \;\; (k = -3,-2,-1,1,2,3).
\end{cases}
\end{equation}
According to Identity \ref{identity_GNI2}, let
\begin{equation}
\label{examplek0123n123e5}
\begin{cases}
P_0=-\displaystyle\frac{x_3 x_5 x_7}{y_2 y_4 y_6},\\
P_k=y_2^k+y_4^k+y_6^k-x_3^k-x_5^k-x_7^k, \quad (k = -3,-2,-1,1,2,3).
\end{cases}
\end{equation}
and define
\begin{equation}
\begin{aligned}
\label{examplek0123n123e6}
& S_1=(P_1)/1,\\
& S_2=(P_2 + S_1 P_1)/2,\\
& S_3=(P_3 + S_1 P_2 + S_2 P_1)/3,\\
& S_0=1+P_0,\\
& S_{-1}=(P_0 P_{-1})/1,\\
& S_{-2}=(P_0 P_{-2} + S_{-1} P_{-1})/2,\\
& S_{-3}=(P_0 P_{-3} + S_{-1} P_{-2} + S_{-2} P_{-1})/3.
\end{aligned}
\end{equation}
Then, it follows that:
\begin{equation}
\begin{aligned}
\label{examplek0123n123e7}
\begin{vmatrix}
 S_{-3} & S_{-2} & S_{-1} & S_0 \\
 S_{-2} & S_{-1} & S_0 & S_1 \\
 S_{-1} & S_0 & S_1 & S_2 \\
 S_0 & S_1 & S_2 & S_3 \\
\end{vmatrix}=0
\end{aligned}
\end{equation}
Based on \eqref{examplek0123n123e4}, we replace \eqref{examplek0123n123e5} with the following expression:
\begin{equation}
\label{examplek0123n123e8}
\begin{cases}
P_0=-\displaystyle\frac{y_8}{\sqrt{\varepsilon _1}},\\
P_k=\displaystyle\frac{1}{2} \left(1+\varepsilon_1^k-2 y_8^k\right), \quad (k = -3,-2,-1,1,2,3).
\end{cases}
\end{equation}
Substituting \eqref{examplek0123n123e8} into \eqref{examplek0123n123e6} and \eqref{examplek0123n123e7}, then we obtain:
\begin{align}
\label{examplek0123n123e9}
8 y_8^4-y_8 \left(1+\sqrt{\epsilon _1}\right){}^2 \left(8 y_8^2-y_8-y_8 \epsilon _1-10 y_8 \sqrt{\epsilon _1}+8 \epsilon _1\right)+8 \epsilon _1^2=0
\end{align}
By solving \eqref{examplek0123n123e9} with \(\varepsilon_1 = 10^{-50}\), we obtain \(\{y_8, y_6, y_4, y_2\}\) satisfying the conditions \(0<y_2<y_4 < y_6<y_8\):
\begin{flalign}
& \quad y_8=0.8535533905932737622004222103417464009876655287597976\dots \nonumber & \\
& \quad y_6=0.1464466094067262377995779896582535990123344712401323\dots \nonumber & \\ 
& \quad y_4=6.8284271247461900976033694886214215124732071060378626\dots \times 10^{-50},\nonumber & \\
& \quad y_2=1.1715728752538099023966225113785784875267928939745373\dots \times 10^{-50}.\nonumber &
\end{flalign}
By solving \eqref{examplek0123n123e3}, we further obtain:
\begin{flalign}
& \quad x_7=0.5000000000000000000000000999999999999999999999999850\dots, \nonumber & \\
& \quad x_5=1.0000000000000000000000000000000000000000000000000000\dots \times 10^{-25},\nonumber & \\
& \quad x_3=1.9999999999999999999999996000000000000000000000001400\dots \times 10^{-50}.\nonumber & 
\end{flalign}
Then, we have:
\begin{flalign}
& \quad \beta_{8, \min}=y_8=0.8535533905932737622004222103417464009876655287597976\dots \nonumber & \\
\end{flalign}
Alternatively, we can proceed more directly. When \(\varepsilon_1 \to 0\), it follows from \eqref{examplek0123n123e9} that:
\begin{align}
\label{examplek0123n123e10}
8 y_8^2 - 8 y_8 + 1 = 0
\end{align}
From equation \eqref{examplek0123n123e10}, we can also obtain:
\begin{align}
\beta_{8, \min} = y_8 = \frac{1}{4} \left(2 + \sqrt{2}\right) = 0.85355339059327376220\dots
\end{align}
\end{example}
In this example, directly solving \eqref{examplek0123n123e3} using Mathematica is quite challenging. Therefore, we first derive \eqref{examplek0123n123e9} by leveraging the Second Generalization of the Girard-Newton Identities, namely Identity \ref{identity_GNI2}, and finally obtain \(\beta_{8, \min}\).
\begin{example}
\textbf{Lower Bound of \( \boldsymbol{\beta_{n+1}} \) for \( \boldsymbol{(k=k_1, k_2, k_3, k_4)} \), \( \boldsymbol{k_1 \notin \mathbb{Z}^+ }\)} \\ [1mm]
For the distinct integer sequence \(\{k_1, k_2, k_3, k_4\}\), where at least one element is non-positive, let \(\{\alpha_{1}, \alpha_{2}, \alpha_{3}, \alpha_{4}, \alpha_{5}\}\) and \(\{\beta_{1}, \beta_{2},  \beta_{3}, \beta_{4}, \beta_{5} \}\) be two sets of non-negative real numbers bounded above by 1, ordered as follows:
\begin{align}
& 0 \leq \alpha_1 \leq \alpha_2 \leq \alpha_3 \leq \alpha_4 \leq \alpha_5=1, \\
& 0 \leq \beta_1 \leq \beta_2 \leq \beta_3 \leq \beta_4 \leq \beta_5<1.
\end{align}
These sets satisfy:
\begin{align}
\left[ \alpha_{1}, \alpha_{2}, \alpha_{3}, \alpha_{4}, \alpha_{5} \right]^{k} = \left[ \beta_{1}, \beta_{2}, \beta_{3}, \beta_{4}, \beta_{5} \right]^{k}, \quad (k = k_1, k_2, k_3, k_4).
\end{align}
Then, based on \eqref{conjecture_beta_n+1_condition}, we obtain the following system of four equations in four variables:
\begin{align}
\left[ x_{2}, x_{2}, x_{4}, x_{4}, 1 \right]^{k} = \left[ \varepsilon_1, y_{3}, y_{3}, y_{5}, y_{5} \right]^{k}, \quad (k = k_1, k_2, k_3, k_4).
\end{align}
where \(\varepsilon_1\) is a positive infinitesimal real number. Namely, if $k_1, k_2, k_3, k_4$ are non-zero integers,
\begin{equation}
\begin{cases}
\label{exampleD2n4a}
& 2 x_2^{k_1}+2 x_4^{k_1}+1=\varepsilon_1 ^{k_1}+ 2 y_3^{k_1}+2 y_5^{k_1},\\
& 2 x_2^{k_2}+2 x_4^{k_2}+1=\varepsilon_1^{k_2}+2 y_3^{k_2}+2 y_5^{k_2},\\
& 2 x_2^{k_3}+2 x_4^{k_3}+1=\varepsilon_1^{k_3}+2 y_3^{k_3}+2 y_5^{k_3},\\
& 2 x_2^{k_4}+2 x_4^{k_4}+1=\varepsilon_1^{k_4}+2 y_3^{k_4}+2 y_5^{k_4}.
\end{cases}
\end{equation}
or, if $k_1$ is zero,
\begin{equation}
\begin{cases}
\label{exampleD2n4b}
& \qquad\qquad x_2^2 \cdot x_4^2=\varepsilon_1 \cdot y_3^2 \cdot y_5^2,\\
& 2 x_2^{k_2}+2 x_4^{k_2}+1=\varepsilon_1^{k_2}+2 y_3^{k_2}+2 y_5^{k_2},\\
& 2 x_2^{k_3}+2 x_4^{k_3}+1=\varepsilon_1^{k_3}+2 y_3^{k_3}+2 y_5^{k_3},\\
& 2 x_2^{k_4}+2 x_4^{k_4}+1=\varepsilon_1^{k_4}+2 y_3^{k_4}+2 y_5^{k_4}.
\end{cases}
\end{equation}
For a given \(\{k_1, k_2, k_3, k_4\}\), by solving \eqref{exampleD2n4a} or \eqref{exampleD2n4b}, we obtain \(\{x_2, x_4, y_3, y_5\}\) satisfying the conditions \(0 < x_2 < x_4\) and \(y_3 < y_5\). Then, we have
\begin{align}
& \qquad \beta_{5, \min}=y_5.
\end{align}
\end{example}
Similar to the methods described in previous examples of this chapter, the solutions of \eqref{exampleD2n4a} and \eqref{exampleD2n4b} can be obtained through two approaches: directly using Mathematica's \texttt{NSolve} function for numerical computation, or by applying the Second Generalization of the Girard-Newton Identities (Identity \ref{identity_GNI2}). A slight difference in this example is that, since at least one element of the power exponents is a non-positive integer, we need to introduce a positive infinitesimal real number \(\varepsilon_1\) during the process. In our calculations, we typically chose \(\varepsilon_1 = 10^{-50}\), which provides sufficient precision. In the final results, \(\varepsilon_1\) will be eliminated. Below are the calculated results of \(\beta_{5, \min}\) for some specific cases of \(\{k_1, k_2, k_3, k_4\}\).\\[1mm]
\indent
$\bullet$ For type $ (k=0,1,2,3) $, we have
\begin{align}
\label{k0123b5min}
& \beta_{5, \min}= 0.853553390593273762200422181052\dots =\frac{1}{4} \left(2+\sqrt{2}\right)
\end{align}
which is a real root of the following equation:
\begin{align}
\label{k0123y5}
& 8 y_5^2-8 y_5+1 = 0
\end{align}
The ratio \( b_5 / a_5 = 2493680/2906467=0.857976 \dots \) in the numerical example given below is the closest to \( \beta_{5, \min} \) among hundreds of thousands of ideal non-negative integer solutions of type \((k=0,1,2,3)\) that we have discovered through various methods so far.
\begin{align}
& [3424,395910,551551,2486913,2493680]^k \nonumber\\
& \quad =[23725,30492,1451562,1519232,2906467]^k, \quad (k=0,1,2,3).
\end{align}
\indent
$\bullet$ For types \( (k=-1,1,2,3) \) and \( (k=-2,1,2,3) \), we obtain results identical to those for type \( (k=0,1,2,3) \):
\begin{align}
& \beta_{5, \min}= 0.853553390593273762200422181052\dots =\frac{1}{4} \left(2+\sqrt{2}\right)  \tag*{\eqref{k0123b5min}}
\end{align}
which is also a real root of the equation \eqref{k0123y5}.\\[1mm]
\indent
$\bullet$ For types $ (k=0,1,2,4) $, $ (k=-1,1,2,4) $, and $ (k=-2,1,2,4) $, we obtain the same results:
\begin{align}
& \beta_{5, \min}= 0.874419525576564851701646263977\dots
\end{align}
which is a real root of the following equation:
\begin{align}
& 32 y_5^4-16 y_5^3-34 y_5^2+24 y_5-3 = 0
\end{align}

\indent
$\bullet$ For types \( (k=0,1,2,5) \), \( (k=-1,1,2,5) \), and \( (k=-2,1,2,5) \), we obtain the same results:
\begin{align}
\beta_{5, \min} = 0.890598214435422423254169001764\dots
\end{align}
which is a real root of the following equation:
\begin{align}
512 y_5^6 - 512 y_5^5 - 320 y_5^3 + 560 y_5^2 - 240 y_5 + 27 = 0
\end{align}

\indent
$\bullet$ For types \( (k=0,1,2,6) \), \( (k=-1,1,2,6) \), and \( (k=-2,1,2,6) \), we obtain the same results:
\begin{align}
& \beta_{5, \min}= 0.903424979534106647631655995194\dots
\end{align}
which is a real root of the following equation:
\begin{align}
& 2048 y_5^8-3072 y_5^7+896 y_5^6+384 y_5^5-1376 y_5^4 +2528 y_5^3\nonumber\\
& \qquad -1858 y_5^2+560 y_5-55 = 0
\end{align}

\indent
$\bullet$ For types \( (k=0,1,3,5) \), \( (k=-1,1,3,5) \), and \( (k=-2,1,3,5) \), we obtain the same results:
\begin{align}
& \beta_{5, \min}= 0.900968867902419126236102319507\dots 
\end{align}
which is a real root of the following equation:
\begin{align}
& 8 y_5^3-4 y_5^2-4 y_5+1 = 0
\end{align}

\indent
$\bullet$ For types \( (k=0,2,4,6) \) and \( (k=-2,2,4,6) \), we obtain the same results:
\begin{align}
& \beta_{5, \min}= 0.923879532511286756128183189396\dots
\end{align}
which is a real root of the following equation:
\begin{align}
8 y_5^4-8 y_5^2+1 = 0
\end{align}

\indent
$\bullet$ For type $ (k=-3,-1,1,3) $, we have
\begin{align}
& \beta_{5, \min}= 0.809016994374947424102293417182\dots 
\end{align}
which is a real root of the following equation:
\begin{align}
& 4 y_5^2-2 y_5-1 = 0
\end{align}

\subsubsection{Upper and Lower Bounds of \(\beta_{n}\)}
Through extensive numerical computations, we have discovered that given a value of \(\beta_{n+1}\) within its bounds, the bounds of \(\beta_{n}\) can be exactly determined based on \(\beta_{n+1}\). Specifically, we propose the following conjecture:
\begin{conjecture}
\label{conjecture_beta_n}\textnormal{\textbf{(Exact Bounds of \(\boldsymbol{\beta_n}\))}}\\[1mm]
Let \(\{\alpha_{1}, \alpha_{2}, \dots, \alpha_{n+1}\}\) and \(\{\beta_{1}, \beta_{2}, \dots, \beta_{n+1}\}\) be two sets of non-negative real numbers bounded above by 1, ordered as follows:
\begin{align}
& 0 \leq \alpha_1 \leq \alpha_2 \leq \cdots \leq \alpha_{n} \leq \alpha_{n+1} = 1, \\
& 0 \leq \beta_1 \leq \beta_2 \leq \cdots \leq \beta_{n} \leq \beta_{n+1} < 1.
\end{align}
For given distinct integers \(\{k_1, k_2, \dots, k_n\}\), if the following system of equations holds:
\begin{align}
\left[ \alpha_{1}, \alpha_{2}, \dots, \alpha_{n+1} \right]^{k} = \left[ \beta_{1}, \beta_{2}, \dots, \beta_{n+1} \right]^{k}, \quad (k = k_1, k_2, \dots, k_n),
\end{align}
we conjecture that:
\begin{enumerate}
\item Let \( \beta_{n, \min}(\beta_{n+1}) \) denote the lower bound of \( \beta_{n} \) when \( \beta_{n+1} \) takes a specific value. Then, \( \beta_{n} \) achieves its minimum value \( \beta_{n, \min}(\beta_{n+1}) \) under the following conditions:
\begin{equation}
\label{Conjecture1.3}
\begin{cases}
\beta_{n+1,\min} \leq \beta_{n+1} < \alpha_{n+1} = 1, \\
\beta_{n-1} = \beta_{n-2}, \quad \beta_{n-3} = \beta_{n-4}, \quad \dots \\
\alpha_{n} = \alpha_{n-1}, \quad \alpha_{n-2} = \alpha_{n-3}, \quad \dots \\
\alpha_1 = \varepsilon \quad \textup{(if \(n\) is odd)} \quad \textup{or} \quad \beta_1 = \varepsilon \quad \textup{(if \(n\) is even)},
\end{cases}
\end{equation}
where \(\varepsilon\) is a positive infinitesimal real number. By default, \( \beta_{n, \min}(\beta_{n+1}) \) can be abbreviated as \( \beta_{n, \min} \).
\item Let \( \beta_{n, \max}(\beta_{n+1}) \) denote the upper bound of \( \beta_{n} \) when \( \beta_{n+1} \) takes a specific value. Then, \( \beta_{n, \max}(\beta_{n+1}) = \beta_{n+1} \). By default, \( \beta_{n, \max}(\beta_{n+1}) \) can be abbreviated as \( \beta_{n, \max} \).
\end{enumerate}
\end{conjecture}

\begin{example}\textbf{Lower Bound of \( \boldsymbol{\beta_{n}} \) for \( \boldsymbol{(k=1,2,3,4,5)} \)} \\ [2mm]
For the given normalized GPTE system:
\begin{flalign}
&\quad \left[ \alpha_{1}, \alpha_{2}, \alpha_{3}, \alpha_{4}, \alpha_{5},\alpha_{6} \right]^{k} = \left[ \beta_{1}, \beta_{2}, \beta_{3}, \beta_{4}, \beta_{5},\beta_{6} \right]^{k}, \quad (k = 1,2,3,4,5).&
\end{flalign}
\noindent
Based on Conjecture \textup{\ref{conjecture_beta_n}}, when \( \beta_{6} \) takes a specific value, \( \beta_{5} \) achieves its minimum value \( \beta_{5, \min}(\beta_{6}) \) under the following conditions:
\begin{equation}
\begin{cases}
& \beta_{6,min} \leq \beta_{6} < \alpha_{6}=1,\\
& \beta_4 =  \beta_3, \:\quad \beta_2 =  \beta_1, \\
& \alpha_5 = \alpha_4, \quad \alpha_3 = \alpha_2, \\
& \alpha_1 = 0.
\end{cases}
\end{equation}
This leads to the following system of five equations in five variables:
\begin{align}
\left[ 0, x_{3}, x_{3}, x_{5}, x_{5}, 1 \right]^{k} = \left[ y_{2}, y_{2}, y_{4}, y_{4},y_{5}, \beta_{6} \right]^{k}, \quad (k = 1,2,3,4,5).
\end{align}
Specifically,
\begin{align}
\label{k12345b5e1}
2 x_3^k+2 x_5^k+1=2 y_2^k+ 2 y_4^k+y_5^k+ \beta_6^k,\quad (k = 1,2,3,4,5).
\end{align}
\end{example}

To solve the system of equations \eqref{k12345b5e1}, there are two approaches.\\[1mm]
\indent
$\bullet$ The first approach utilizes numerical computation methods to obtain numerical solutions of \eqref{k12345b5e1}. For instance, when \(\beta_6 = 0.9801\), using mathematical software such as Mathematica and its \texttt{NSolve} function, we obtain:
\begin{equation}
\label{k12345b5e2}
\begin{cases}
x_3 &= 0.23946954408154330620\cdots, \\
x_5 &= 0.71132433880795564470\cdots, \\
y_2 &= 0.06421954150228593742\cdots, \\
y_4 &= 0.47786558905385802094\cdots, \\
y_5 &= 0.83731750466670998506\cdots.
\end{cases}
\end{equation}
This yields the lower bound of \(\beta_{5, \min}\) when \(\beta_6 = 0.9801\):
\begin{align}
\label{k12345b5e3}
\beta_{5, \min} (\beta_6) &= \beta_{5, \min} (0.9801) = y_5 = 0.83731750466670998506\cdots.
\end{align}
\indent
$\bullet$ The second approach for solving the system \eqref{k12345b5e1} involves first deriving a polynomial in terms of \(\beta_6\) and \(y_5\) by applying the Second Generalization of the Girard-Newton identities (Identity \ref{identity_GNI2}), and then computing the value of \(y_5\) using numerical methods. The procedure is as follows:\\[1mm]
\indent By transforming \eqref{k12345b5e1}, we obtain:
\begin{align}
\label{k12345b5e4}
(1-y_5^k-\beta _6^k)/2=y_2^k+y_4^k-x_3^k-x_5^k,\quad (k = 1,2,3,4,5).
\end{align}
Let
\begin{align}
\label{k12345b5e5}
P_k=y_2^k+y_4^k-x_3^k-x_5^k,\quad (k = 1,2,3,4,5).
\end{align}
and define
\begin{equation}
\begin{aligned}
\label{k12345b5e6}
& S_1=(P_1)/1,\\
& S_2=(P_2 + S_1 P_1)/2,\\
& S_3=(P_3 + S_1 P_2 + S_2 P_1)/3,\\
& S_4=(P_4 + S_1 P_3 + S_2 P_2+ S_3 P_1)/4,\\
& S_5=(P_5 + S_1 P_4 + S_2 P_3+ S_3 P_2+ S_4 P_1)/5.\\
\end{aligned}
\end{equation}
Then, according to Identity \ref{identity_GNI2}, it follows that:
\begin{equation}
\begin{aligned}
\label{k12345b5e7}
\begin{vmatrix}
 S_1 & S_2 & S_3 \\
 S_2 & S_3 & S_4 \\
 S_3 & S_4 & S_5 \\
\end{vmatrix}=0
\end{aligned}
\end{equation}
Based on \eqref{k12345b5e4}, we replace \eqref{k12345b5e5} with the following expression:
\begin{align}
\label{k12345b5e8}
P_k=(1-y_5^k-\beta _6^k)/2,\quad (k = 1,2,3,4,5).
\end{align}
Substituting \eqref{k12345b5e8} into \eqref{k12345b5e6} and \eqref{k12345b5e7}, then we obtain:
\begin{align}
\label{k12345b5e9}
& y_5^8+8 \left(\beta _6-1\right) y_5^7-4 \left(\beta _6-1\right) \left(17 \beta _6+7\right) y_5^6 \nonumber\\
& \quad +8 \left(\beta _6-1\right) \left(23 \beta _6^2+14 \beta _6+7\right) y_5^5 \nonumber\\
& \quad -2 \left(\beta _6-1\right) \left(125 \beta _6^3+105 \beta _6^2-9 \beta _6+35\right) y_5^4 \nonumber\\
& \quad +8 \left(\beta _6-1\right) \left(23 \beta _6^4+28 \beta _6^3-22 \beta _6^2-4 \beta _6+7\right) y_5^3 \nonumber\\
& \quad -4 \left(\beta _6-1\right)^2 \left(17 \beta _6^4+52 \beta _6^3+30 \beta _6^2-28 \beta _6-7\right) y_5^2 \nonumber\\
& \quad +8 \left(\beta _6-1\right)^3 \left(\beta _6+1\right)^2 \left(\beta _6^2+6 \beta _6+1\right) y_5+\left(\beta _6-1\right)^8=0
\end{align}
By solving \eqref{k12345b5e9}, we obtain \(\beta_{5,\min}(\beta_6) = y_5\). For instance, when \(\beta_6 = 0.9801\),\\ using Mathematica's \texttt{NSolve} function and incorporating the constraints from Corollary \ref{corollary_Conservative_Bounds}, namely \((1 - \beta_6^5)^{1/5}\)< \(y_5 \leq \beta_6\) , we obtain:
\begin{align}
y_5 &= 0.83731750466670998506\cdots.
\end{align}
This result is identical to \eqref{k12345b5e2}, which was obtained previously using direct numerical computation methods. Thus, for the lower bound of \(\beta_{5, \min}\) when \(\beta_6 = 0.9801\), we arrive at the same value as in \eqref{k12345b5e3}:
\begin{align}
\beta_{5, \min} (\beta_6) &= \beta_{5, \min} (0.9801) = y_5 = 0.83731750466670998506\cdots. \tag*{\eqref{k12345b5e3}}
\end{align}
\indent
$\bullet$ 
When utilizing \(\beta_{n, \min}(\beta_{n+1})\) for the computer search of non-negative integer solutions to the GPTE problem, an efficient approach is to first precompute \(\beta_{n, \min}(\beta_{n+1})\) for each \(\beta_{n+1}\) within the range \(\beta_{n+1,\min} \leq \beta_{n+1} < 1\). This results in a one-dimensional lookup table that can be used during the search process, thereby avoiding the complex and repetitive calculations of \(\beta_{n, \min}(\beta_{n+1})\) during the search. Below is an example of the lookup table for \(\beta_{5, \min}(\beta_6)\) corresponding to the type \((k=1,2,3,4,5)\):
\begin{equation}
\begin{aligned}
& \beta_{5,\min}\left(0.93302\right) = 0.9330054032\cdots, \\
& \beta_{5,\min}\left(0.93303\right) = 0.9329954007\cdots, \\
& \beta_{5,\min}\left(0.93304\right) = 0.9329853960\cdots, \\
& \qquad \vdots \\
& \beta_{5,\min}\left(0.99997\right) = 0.3262410418\cdots, \\
& \beta_{5,\min}\left(0.99998\right) = 0.3041776685\cdots, \\
& \beta_{5,\min}\left(0.99999\right) = 0.2693594803\cdots.
\end{aligned}
\end{equation}
In the table, we list only the first three and the last three entries. The precision of the data in this table is sufficient for searching non-negative integer solutions for the type \((k=1,2,3,4,5)\) within the range of 10000.

\begin{example}\textbf{Lower Bound of \( \boldsymbol{\beta_{n}} \) for \( \boldsymbol{(k=1,2,3,4,6)} \)} \\ [2mm]
For the given normalized GPTE system:
\begin{flalign}
&\quad \left[ \alpha_{1}, \alpha_{2}, \alpha_{3}, \alpha_{4}, \alpha_{5},\alpha_{6} \right]^{k} = \left[ \beta_{1}, \beta_{2}, \beta_{3}, \beta_{4}, \beta_{5},\beta_{6} \right]^{k}, \quad (k = 1,2,3,4,6).&
\end{flalign}
\noindent
Based on Conjecture \textup{\ref{conjecture_beta_n}},  when \( \beta_{6} \) takes a specific value, \( \beta_{5} \) achieves its minimum value \( \beta_{5, \min}(\beta_{6}) \) under the following conditions:
\begin{equation}
\begin{cases}
& \beta_{6,min} \leq \beta_{6} < \alpha_{6}=1,\\
& \beta_4 =  \beta_3, \:\quad \beta_2 =  \beta_1, \\
& \alpha_5 = \alpha_4, \quad \alpha_3 = \alpha_2, \\
& \alpha_1 = 0.
\end{cases}
\end{equation}
This leads to the following system of five equations in five variables:
\begin{align}
\label{k12346b5e1}
2 x_3^k+2 x_5^k+1=2 y_2^k+ 2 y_4^k+y_5^k+ \beta_6^k,\quad (k = 1,2,3,4,6).
\end{align}
\end{example}
Similar to the previous example, to solve the system of equations \eqref{k12346b5e1}, there are two approaches: direct numerical computation and applying the Second Generalization of the Girard-Newton Identities. The former is straightforward and can be accomplished using Mathematica's \texttt{NSolve} function. The latter, while more complex in derivation, provides a polynomial that precisely describes the relationship between \(\beta_6\) and \(y_5\). The main derivation process is as follows:

By transforming \eqref{k12346b5e1}, we obtain:
\begin{align}
\label{k12346b5e4}
(1 - y_5^k - \beta_6^k)/2 = y_2^k + y_4^k - x_3^k - x_5^k, \quad (k = 1,2,3,4,6).
\end{align}
Let
\begin{align}
\label{k12346b5e5}
P_k = y_2^k + y_4^k - x_3^k - x_5^k, \quad (k = 1,2,3,4,5,6).
\end{align}
and define
\begin{equation}
\label{k12346b5e6}
\begin{cases}
& S_1 = (P_1)/1, \\
& S_2 = (P_2 + S_1 P_1)/2, \\
& S_3 = (P_3 + S_1 P_2 + S_2 P_1)/3, \\
& S_4 = (P_4 + S_1 P_3 + S_2 P_2 + S_3 P_1)/4, \\
& S_5 = (P_5 + S_1 P_4 + S_2 P_3 + S_3 P_2 + S_4 P_1)/5, \\
& S_6 = (P_6 + S_1 P_5 + S_2 P_4 + S_3 P_3 + S_4 P_2 + S_5 P_1)/6.
\end{cases}
\end{equation}
Then, according to Identity \ref{identity_GNI2}, it follows that:
\begin{equation}
\begin{aligned}
\label{k12346b5e7}
\begin{vmatrix}
S_1 & S_2 & S_3 \\
S_2 & S_3 & S_4 \\
S_3 & S_4 & S_5 \\
\end{vmatrix} = 0
\end{aligned}
\end{equation}
and
\begin{equation}
\begin{aligned}
\label{k12346b5e8}
\begin{vmatrix}
S_2 & S_3 & S_4 \\
S_3 & S_4 & S_5 \\
S_4 & S_5 & S_6 \\
\end{vmatrix} = 0
\end{aligned}
\end{equation}
From \eqref{k12346b5e7}, we have:
\begin{equation}
\label{k12346b5e9}
\begin{aligned}
S_5 &= -\frac{\begin{vmatrix}
S_1 & S_2 & S_3 \\
S_2 & S_3 & S_4 \\
S_3 & S_4 & 0 \\
\end{vmatrix}}{\begin{vmatrix}
S_1 & S_2 \\
S_2 & S_3 \\
\end{vmatrix}} = \frac{S_3^3 - 2 S_2 S_4 S_3 + S_1 S_4^2}{S_1 S_3 - S_2^2}
\end{aligned}
\end{equation}
From \eqref{k12346b5e6}, we have:
\begin{align}
\label{k12346b5e10}
P_5 &= 5 S_5 - (P_4 S_1 + P_3 S_2 + P_2 S_3 + P_1 S_4)
\end{align}
Based on \eqref{k12346b5e4} and \eqref{k12346b5e10}, we replace \eqref{k12346b5e5} with the following expression:
\begin{equation}
\label{k12346b5e11}
\begin{cases}
& P_k= (1 - y_5^k - \beta_6^k)/2, \quad (k = 1,2,3,4,6),\\
& P_5= 5 S_5 - (P_4 S_1 + P_3 S_2 + P_2 S_3 + P_1 S_4).
\end{cases}
\end{equation}
Based on \eqref{k12346b5e9}, we replace \eqref{k12346b5e6} with the following expression:
\begin{equation}
\label{k12346b5e12}
\begin{cases}
& S_1 = (P_1)/1, \\
& S_2 = (P_2 + S_1 P_1)/2, \\
& S_3 = (P_3 + S_1 P_2 + S_2 P_1)/3, \\
& S_4 = (P_4 + S_1 P_3 + S_2 P_2 + S_3 P_1)/4, \\
& S_5 = (S_3^3 - 2 S_2 S_4 S_3 + S_1 S_4^2)/(S_1 S_3 - S_2^2), \\
& S_6 = (P_6 + S_1 P_5 + S_2 P_4 + S_3 P_3 + S_4 P_2 + S_5 P_1)/6.
\end{cases}
\end{equation}
Substituting \eqref{k12346b5e11} and \eqref{k12346b5e12} into \eqref{k12346b5e8}, we obtain:
\begin{align}
\label{k12346b5e13}
& 5 y_5^{14}+60 \left(\beta _6-1\right) y_5^{13}-5 \left(\beta _6-1\right) \left(29 \beta _6+65\right) y_5^{12} \nonumber \\
& \quad -8 \left(\beta _6-1\right) \left(99 \beta _6^2-137 \beta _6-130\right) y_5^{11} \nonumber \\
& \quad+ \left(\beta _6-1\right) \left(3221 \beta _6^3+85 \beta _6^2-2537 \beta _6-2145\right) y_5^{10}\nonumber\\
& \quad -4 \left(\beta _6-1\right){}^2 \left(879 \beta _6^3+2293 \beta _6^2+1185 \beta _6+715\right) y_5^9 \nonumber\\
& \quad -\left(\beta _6-1\right) \left(1033 \beta _6^5-9283 \beta _6^4+5554 \beta _6^3+1186 \beta _6^2-1147 \beta _6+2145\right) y_5^8  \nonumber\\
& \quad +16 \left(\beta _6-1\right){}^2 \beta _6 \left(275 \beta _6^4+132 \beta _6^3+2 \beta _6^2+740 \beta _6+131\right) y_5^7  \nonumber\\
& \quad -\left(\beta _6-1\right){}^2 \left(1033 \beta _6^6+8754 \beta _6^5+2687 \beta _6^4+\cdots+2194 \beta _6+2145\right) y_5^6\nonumber\\
& \quad +4 \left(\beta _6-1\right){}^2 \left(-879 \beta _6^7+821 \beta _6^6+2573 \beta _6^5+\cdots+607 \beta _6+715\right) y_5^5 \nonumber\\
& \quad +\left(\beta _6-1\right){}^5 \left(3221 \beta _6^5+13965 \beta _6^4+22778 \beta _6^3+\cdots+9745 \beta _6+2145\right)y_5^4  \nonumber\\
& \quad -8 \left(\beta _6-1\right){}^5 \left(99 \beta _6^6+887 \beta _6^5+2184 \beta _6^4+\cdots+699 \beta _6+130\right) y_5^3 \nonumber\\
& \quad +\left(\beta _6-1\right){}^5 \left(-145 \beta _6^7+1163 \beta _6^6+4643 \beta _6^5+\cdots+1681 \beta _6+325\right) y_5^2\nonumber\\
& \quad +4 \left(\beta _6-1\right){}^8 \left(\beta _6+1\right) \left(15 \beta _6^4+60 \beta _6^3+106 \beta _6^2+60 \beta _6+15\right) y_5 \nonumber\\
& \quad +5 \left(\beta _6-1\right){}^{13} \left(\beta _6+1\right)=0
\end{align}
By solving \eqref{k12346b5e13}, with the additional constraints \((1 - \beta_6^6)^{1/6}< y_5 \leq \beta_6 \), we obtain the lookup table for \(\beta_{5, \min}(\beta_6)\) corresponding to the type \((k=1,2,3,4,6)\) as follows:
\begin{equation}
\begin{aligned}
& \beta_{5,\min}\left(0.93733\right) = 0.9373130877\cdots, \\
& \beta_{5,\min}\left(0.93734\right) = 0.9373030847\cdots, \\
& \beta_{5,\min}\left(0.93735\right) = 0.9372930795\cdots, \\
& \qquad \vdots \\
& \beta_{5,\min}\left(0.99997\right) = 0.3633027280\cdots, \\
& \beta_{5,\min}\left(0.99998\right) = 0.3413846089\cdots, \\
& \beta_{5,\min}\left(0.99999\right) = 0.3065466858\cdots.
\end{aligned}
\end{equation}

\begin{example}\textbf{Lower Bound of \( \boldsymbol{\beta_{n}} \) for \( \boldsymbol{(k=1,2,3,4,7)} \)} \\ [2mm]
For the given normalized GPTE system:
\begin{flalign}
&\quad \left[ \alpha_{1}, \alpha_{2}, \alpha_{3}, \alpha_{4}, \alpha_{5},\alpha_{6} \right]^{k} = \left[ \beta_{1}, \beta_{2}, \beta_{3}, \beta_{4}, \beta_{5},\beta_{6} \right]^{k}, \quad (k = 1,2,3,4,7).&
\end{flalign}
\noindent
Based on Conjecture \textup{\ref{conjecture_beta_n}}, the minimum value of \(\beta_5\) can be derived from the following system of five equations in five variables:
\begin{align}
\label{k12347b5e1}
(1 - y_5^k - \beta_6^k)/2 = y_2^k + y_4^k - x_3^k - x_5^k, \quad (k = 1,2,3,4,7).
\end{align}
\end{example}
To solve the system of equations \eqref{k12347b5e1} using the Second Generalization of the Girard-Newton Identities, the main derivation process follows a similar approach to that of previous example. Let
\begin{equation}
\label{k12347b5e11}
\begin{aligned}
& P_k = (1 - y_5^k - \beta_6^k)/2, \quad (k = 1,2,3,4,7),\\
& P_5=5 S_5-\left(S_1 P_4 + S_2 P_3+ S_3 P_2+ S_4 P_1\right),\\
& P_6=6 S_6-\left(S_1 P_5 + S_2 P_4+ S_3 P_3+ S_4 P_2+ S_5 P_1\right).
\end{aligned}
\end{equation} 
and 
\begin{equation}
\label{k12347b5e12}
\begin{aligned}
& S_1=(P_1)/1,\\
& S_2=(P_2 + S_1 P_1)/2,\\
& S_3=(P_3 + S_1 P_2 + S_2 P_1)/3,\\
& S_4=(P_4 + S_1 P_3 + S_2 P_2+ S_3 P_1)/4,\\
& S_5=-\frac{\begin{vmatrix}
 S_1 & S_2 & S_3 \\
 S_2 & S_3 & S_4 \\
 S_3 & S_4 & 0 \\
\end{vmatrix}}{\begin{vmatrix}
 S_1 & S_2  \\
 S_2 & S_3  \\
\end{vmatrix} },\qquad
 S_6=-\frac{\begin{vmatrix}
 S_2 & S_3 & S_4 \\
 S_3 & S_4 & S_5 \\
 S_4 & S_5 & 0 \\
\end{vmatrix}}{\begin{vmatrix}
 S_2 & S_3  \\
 S_3 & S_4  \\
\end{vmatrix} },\\
& S_7=(P_7 + S_1 P_6 + S_2 P_5+ S_3 P_4+ S_4 P_3+ S_5 P_2+ S_6 P_1)/7. 
\end{aligned}  
\end{equation} 
Substituting \eqref{k12347b5e11} and \eqref{k12347b5e12} into the following expression:
\begin{equation}
\begin{aligned}
\label{k12347b5e13}
\begin{vmatrix}
S_3 & S_4 & S_5 \\
S_4 & S_5 & S_6 \\
S_5 & S_6 & S_7 \\
\end{vmatrix} = 0
\end{aligned}
\end{equation}
we have:
\begin{align}
\label{k12347b5e14}
& 65 y_5^{19}+995 \left(\beta _6-1\right) y_5^{18}+\left(635 \beta _6^2-7670 \beta _6+7035\right) y_5^{17}\nonumber\\ 
& \qquad + \left(-21879 \beta _6^3+28501 \beta _6^2+23723 \beta _6-30345\right)y_5^{16}\nonumber\\ 
& \qquad +\cdots + 5 \left(\beta _6-1\right){}^{17} \left(13 \beta _6^2+22 \beta _6+13\right)=0
\end{align} 
Then we obtain the lookup table of \(\beta_{5, \min}(\beta_6)\) for type \((k=1,2,3,4,7)\) as\\ follows:
\begin{equation}
\begin{aligned}
& \beta_{5,\min}\left(0.94120\right) = 0.9411912823\cdots, \\
& \beta_{5,\min}\left(0.94121\right) = 0.9411812801\cdots, \\
& \beta_{5,\min}\left(0.94122\right) = 0.9411712755\cdots, \\
& \qquad \vdots \\
& \beta_{5,\min}\left(0.99997\right) = 0.4003318291\cdots, \\
& \beta_{5,\min}\left(0.99998\right) = 0.3788322292\cdots, \\
& \beta_{5,\min}\left(0.99999\right) = 0.3444407532\cdots.
\end{aligned}
\end{equation}

\begin{example}\textbf{Lower Bound of \( \boldsymbol{\beta_{n}} \) for \( \boldsymbol{(k=1,2,4,5)} \)} \\ [2mm]
For the given normalized GPTE system:
\begin{flalign}
&\quad \left[ \alpha_{1}, \alpha_{2}, \alpha_{3}, \alpha_{4}, \alpha_{5} \right]^{k} = \left[ \beta_{1}, \beta_{2}, \beta_{3}, \beta_{4}, \beta_{5} \right]^{k}, \quad (k = 1,2,4,5).&
\end{flalign}
\noindent
Based on Conjecture \textup{\ref{conjecture_beta_n}}, the minimum value of \(\beta_4\) can be derived from the following system of four equations in four variables:
\begin{align}
\label{k1245b4e1}
(1-y_4^k-\beta _5^k)/2=y_3^k-x_2^k-x_4^k, \quad (k = 1,2,4,5).
\end{align}
\end{example}
To solve the system of equations \eqref{k1245b4e1} using the Second Generalization of the Girard-Newton Identities, the main derivation process is as follows: Let
\begin{align}
\label{k1245b4e5}
P_k = y_3^k-x_2^k-x_4^k, \quad (k = 1,2,3,4,5).
\end{align}
and define
\begin{equation}
\label{k1245b4e6}
\begin{aligned}
& S_1 = (P_1)/1, \\
& S_2 = (P_2 + S_1 P_1)/2, \\
& S_3 = (P_3 + S_1 P_2 + S_2 P_1)/3, \\
& S_4 = (P_4 + S_1 P_3 + S_2 P_2 + S_3 P_1)/4, \\
& S_5 = (P_5 + S_1 P_4 + S_2 P_3 + S_3 P_2 + S_4 P_1)/5.
\end{aligned}
\end{equation}
Then, according to Identity \ref{identity_GNI2}, it follows that:
\begin{equation}
\begin{aligned}
\label{k1245b4e7}
\begin{vmatrix}
S_2 & S_3 \\
S_3 & S_4 \\
\end{vmatrix} = 0
\end{aligned}
\end{equation}
and
\begin{equation}
\begin{aligned}
\label{k1245b4e8}
\begin{vmatrix}
S_3 & S_4 \\
S_4 & S_5 \\
\end{vmatrix} = 0
\end{aligned}
\end{equation}
Substituting \eqref{k1245b4e6} into \eqref{k1245b4e7} and \eqref{k1245b4e8}, we obtain:
\begin{align}
\label{k1245b4e9}
& P_1^6+3 P_2 P_1^4-8 P_3 P_1^3+9 P_2^2 P_1^2-18 P_4 P_1^2+24 P_2 P_3 P_1-9 P_2^3 \nonumber \\
& \quad +16 P_3^2-18 P_2 P_4=0\\
\label{k1245b4e10}
& P_1^8+8 P_2 P_1^6-8 P_3 P_1^5+30 P_2^2 P_1^4-60 P_4 P_1^4+80 P_2 P_3 P_1^3-96 P_5 P_1^3 \nonumber \\
& \quad +160 P_3^2 P_1^2-120 P_2^2 P_3 P_1+240 P_3 P_4 P_1-288 P_2 P_5 P_1 \nonumber \\
& \quad +45 P_2^4-160 P_2 P_3^2+180 P_4^2+180 P_2^2 P_4-192 P_3 P_5=0
\end{align}
From \eqref{k1245b4e9} and \eqref{k1245b4e10}, we have:
\begin{align}
\label{k1245b4e11}
& P_3=\frac{15 \left(P_1^4-P_2^2\right){}^2-60 \left(P_1^2 P_2-P_4\right){}^2}{8 \left(3 P_1^5-10 P_2 P_1^3+5 P_2^2 P_1+10 P_4 P_1-8 P_5\right)}-\frac{1}{2} P_1 \left(P_1^2+3 P_2\right)
\end{align}
Substituting \eqref{k1245b4e11} into \eqref{k1245b4e9} or \eqref{k1245b4e10}, we have:
\begin{align}
\label{k1245b4e12}
& P_1^{12}+6 P_2 P_1^{10}-15 P_2^2 P_1^8-30 P_4 P_1^8-32 P_5 P_1^7-20 P_2^3 P_1^6+120 P_2 P_4 P_1^6 \nonumber \\
& \quad +288 P_2 P_5 P_1^5-25 P_2^4 P_1^4-100 P_4^2 P_1^4-300 P_2^2 P_4 P_1^4-160 P_2^2 P_5 P_1^3 \nonumber \\
& \quad -320 P_4 P_5 P_1^3+150 P_2^5 P_1^2+600 P_2 P_4^2 P_1^2+256 P_5^2 P_1^2+200 P_2^3 P_4 P_1^2 \nonumber \\
& \quad -480 P_2^3 P_5 P_1-320 P_2 P_4 P_5 P_1-25 P_2^6-200 P_4^3+100 P_2^2 P_4^2\nonumber \\
& \quad +256 P_2 P_5^2+50 P_2^4 P_4=0
\end{align}
Based on \eqref{k1245b4e1}, we replace \eqref{k1245b4e5} with the following expression:
\begin{align}
\label{k1245b4e13}
P_k=(1-y_4^k-\beta _5^k)/2,\quad (k = 1,2,4,5).
\end{align}
Substituting \eqref{k1245b4e13} into \eqref{k1245b4e12}, we have:
\begin{align}
& 351 y_4^{12}-1620 \left(\beta _5-1\right) y_4^{11} +6 \left(309 \beta _5^2+514 \beta _5-823\right) y_4^{10}  \nonumber \\
& \quad +12 \left(\beta _5-1\right){}^2 \left(213 \beta _5+95\right) y_4^9 + \cdots \nonumber \\
& \quad + 3 \left(\beta _5-1\right){}^5 \left(117 \beta _5^7+1125 \beta _5^6+\cdots+1107 \beta _5+243\right)=0
\end{align}
Then we obtain the lookup table of \(\beta_{4, \min}(\beta_5)\) for type \((k=1,2,4,5)\) as\\ follows:
\begin{equation}
\begin{aligned}
& \beta_{4,\min}\left(0.92019\right) = 0.9201810209\cdots, \\
& \beta_{4,\min}\left(0.92020\right) = 0.9201710192\cdots, \\
& \beta_{4,\min}\left(0.92021\right) = 0.9201610159\cdots, \\
& \qquad \vdots \\
& \beta_{4,\min}\left(0.99997\right) = 0.2725109594\cdots, \\
& \beta_{4,\min}\left(0.99998\right) = 0.2527883379\cdots, \\
& \beta_{4,\min}\left(0.99999\right) = 0.2220625368\cdots.
\end{aligned}
\end{equation}

\begin{example}\textbf{Lower Bound of \( \boldsymbol{\beta_{n}} \) for \( \boldsymbol{(k=2,3,4)} \)} \\ [2mm]
For the given normalized GPTE system:
\begin{flalign}
&\quad \left[ \alpha_{1}, \alpha_{2}, \alpha_{3}, \alpha_{4} \right]^{k} = \left[ \beta_{1}, \beta_{2}, \beta_{3}, \beta_{4} \right]^{k}, \quad (k = 2,3,4).&
\end{flalign}
\noindent
Based on Conjecture \textup{\ref{conjecture_beta_n}}, the minimum value of \(\beta_3\) can be derived from the following system of three equations in three variables:
\begin{align}
\label{k234b3e1}
(1-y_3^k-\beta _4^k)/2=y_2^k-x_3^k, \quad (k = 2,3,4).
\end{align}
\end{example}
To solve the system of equations \eqref{k234b3e1} using the Second Generalization of the Girard-Newton Identities, the main derivation process is as follows: Let
\begin{align}
\label{k234b3e5}
P_k = y_2^k-x_3^k, \quad (k = 1,2,3,4).
\end{align}
and define
\begin{equation}
\label{k234b3e6}
\begin{aligned}
& S_1 = (P_1)/1, \\
& S_2 = (P_2 + S_1 P_1)/2, \\
& S_3 = (P_3 + S_1 P_2 + S_2 P_1)/3, \\
& S_4 = (P_4 + S_1 P_3 + S_2 P_2 + S_3 P_1)/4.
\end{aligned}
\end{equation}
Then, according to Identity \ref{identity_GNI2}, it follows that:
\begin{equation}
\begin{aligned}
\label{k234b3e7}
\begin{vmatrix}
S_1 & S_2 \\
S_2 & S_3 \\
\end{vmatrix} = 0
\end{aligned}
\end{equation}
and
\begin{equation}
\begin{aligned}
\label{k234b3e8}
\begin{vmatrix}
S_2 & S_3 \\
S_3 & S_4 \\
\end{vmatrix} = 0
\end{aligned}
\end{equation}
Substituting \eqref{k234b3e6} into \eqref{k234b3e7} and \eqref{k234b3e8}, we obtain:
\begin{align}
\label{k234b3e9}
& P_1^4-4 P_3 P_1+3 P_2^2=0\\
\label{k234b3e10}
& P_1^6+3 P_2 P_1^4-8 P_3 P_1^3+9 P_2^2 P_1^2-18 P_4 P_1^2+24 P_2 P_3 P_1 \nonumber\\
& \quad -9 P_2^3+16 P_3^2-18 P_2 P_4=0
\end{align}
From \eqref{k234b3e9} and \eqref{k234b3e10}, we have:
\begin{align}
\label{k234b3e11}
& P_1=\frac{23 P_2^8-80 P_4 P_2^6+160 P_3^2 P_2^5+2 P_4^2 P_2^4-16 P_3^4 P_2^2-80 P_4^3 P_2^2-9 P_4^4}{2 P_3 \left(3 P_2^6-11 P_4 P_2^4+20 P_3^2 P_2^3+P_4^2 P_2^2-4 P_3^2 P_4 P_2-9 P_4^3\right)} \nonumber\\
& \qquad -\frac{7 P_2^2+P_4}{2 P_3}
\end{align}
Substituting \eqref{k234b3e11} into \eqref{k234b3e9} or \eqref{k234b3e10}, we have:
\begin{align}
\label{k234b3e12}
& P_2^9+6 P_4^2 P_2^5-24 P_3^2 P_4 P_2^4+16 P_3^4 P_2^3+9 P_4^4 P_2-8 P_3^2 P_4^3=0
\end{align}
Based on \eqref{k234b3e1}, we replace \eqref{k234b3e5} with the following expression:
\begin{align}
\label{k234b3e13}
P_k=(1-y_3^k-\beta _4^k)/2,\quad (k = 2,3,4).
\end{align}
Substituting \eqref{k234b3e13} into \eqref{k234b3e12}, we have:
\begin{align}
& 9 y_3^{18}+ 81 \left(\beta _4^2-1\right) y_3^{16}-192 \left(\beta _4^3-1\right) y_3^{15}+36 \left(\beta _4^4+6 \beta _4^2-7\right) y_3^{14}  \nonumber \\
& \quad +\cdots +3 \left(\beta _4-1\right){}^4 \left(\beta _4+1\right){}^2 \left(27 \beta _4^{10}+54 \beta _4^9+\cdots+54 \beta _4+27\right) y_3^2 \nonumber \\
& \quad +3 \left(\beta _4-1\right){}^7 \left(\beta _4+1\right){}^3 \left(3 \beta _4^8+12 \beta _4^7+\cdots+12 \beta _4+3\right)=0
\end{align}
Then we obtain the lookup table of \(\beta_{3, \min}(\beta_4)\) for type \((k=2,3,4)\) as\\ follows:
\begin{equation}
\begin{aligned}
& \beta_{3,\min}\left(0.89926\right) = 0.8992450495\cdots, \\
& \beta_{3,\min}\left(0.89927\right) = 0.8992350479\cdots, \\
& \beta_{3,\min}\left(0.89928\right) = 0.8992250449\cdots, \\
& \qquad \vdots \\
& \beta_{3,\min}\left(0.99997\right) = 0.1810811930\cdots, \\
& \beta_{3,\min}\left(0.99998\right) = 0.1646084470\cdots, \\
& \beta_{3,\min}\left(0.99999\right) = 0.1396615035\cdots.
\end{aligned}
\end{equation}

A key step in the above process involves deriving equation \eqref{k234b3e11} from equations \eqref{k234b3e9} and \eqref{k234b3e10}. Specifically, we first multiply equation \eqref{k234b3e9} by \(P_1^2\) and then subtract equation \eqref{k234b3e10} to obtain:
\begin{flalign}
\label{k234b3e14}
& \quad 3 P_2 P_1^4 - 4 P_3 P_1^3 + \left(6 P_2^2 - 18 P_4\right) P_1^2 + 24 P_2 P_3 P_1 - 9 P_2^3 + 16 P_3^2 - 18 P_2 P_4 = 0.&
\end{flalign}
Next, multiplying equation \eqref{k234b3e9} by \(3 P_2\) and subtracting equation \eqref{k234b3e14}, we derive:
\begin{align}
& 4 P_3 P_1^3 + \left(18 P_4 - 6 P_2^2\right) P_1^2 - 36 P_2 P_3 P_1 + 18 P_2^3 - 16 P_3^2 + 18 P_2 P_4 = 0.
\end{align}
By iterating analogous operations, we ultimately arrive at \eqref{k234b3e11}.\\[2mm]
\indent As described in previous examples, we can also directly employ numerical computation methods to obtain  \(\beta_{3, \min}(\beta_4)\). For instance, when \(\beta_4 = 0.89926\), we use Mathematica's \texttt{NSolve} function to solve the following system:
\begin{equation}
\begin{cases}
2 x_3^2 + 1 = 2 y_2^2 + y_3^2 + \beta_4^2, \\
2 x_3^3 + 1 = 2 y_2^3 + y_3^3 + \beta_4^3, \\
2 x_3^4 + 1 = 2 y_2^4 + y_3^4 + \beta_4^4, \\
0 < y_2 < x_3 < y_3 \leq \beta_4.
\end{cases}
\end{equation}
The numerical solution obtained is:
\begin{equation}
\begin{cases}
y_3 = 0.8992450495\cdots, \\
x_3 = 0.6353521303\cdots, \\
y_2 = 0.3082486434\cdots.
\end{cases}
\end{equation}
Thus, we have:
\[
\beta_{3,\min}(0.89926) = y_3 = 0.8992450495\cdots.
\]

\begin{example}
\textbf{Lower Bound of \( \boldsymbol{\beta_{n}} \) for \( \boldsymbol{(k=1,2,3,\cdots,10)} \)} \\ [2mm]
For the given normalized GPTE system:
\begin{flalign}
&\quad \left[ \alpha_{1}, \alpha_{2}, \alpha_{3}, \cdots, \alpha_{11} \right]^{k} = \left[ \beta_{1}, \beta_{2}, \beta_{3}, \cdots,\beta_{11} \right]^{k}, \;\; (k = 1,2,3,\cdots,10).&
\end{flalign}
\noindent
Based on Conjecture \textup{\ref{conjecture_beta_n}}, the minimum value of \(\beta_{10}\) can be derived from the following system of ten equations in ten variables:
\begin{align}
\label{k1to10b10e1}
& (1-\beta _{11}^k-y_{10}^k)/2=y_3^k+y_5^k+y_7^k+y_9^k-x_2^k-x_4^k-x_6^k-x_8^k-x_{10}^k, \nonumber\\
& \qquad\qquad\qquad\qquad\qquad\qquad\qquad\qquad (k = 1,2,3,\cdots,10).
\end{align}
\end{example}
To solve the system of equations \eqref{k1to10b10e1} using the Second Generalization of the Girard-Newton Identities, the main derivation process follows a similar approach to that of previous example. Let
\begin{align}
\label{k1to10b10e11}
& P_k = (1-\beta _{11}^k-y_{10}^k)/2, \quad (k = 1,2,3,\cdots,10),
\end{align}
and
\begin{equation}
\label{k1to10b10e12}
\begin{aligned}
& S_1 = (P_1)/1, \\
& S_2 = (P_2 + S_1 P_1)/2, \\
& S_3 = (P_3 + S_1 P_2 + S_2 P_1)/3, \\
& S_4 = (P_4 + S_1 P_3 + S_2 P_2 + S_3 P_1)/4, \\
& \qquad \vdots \\
& S_{10} = (P_{10} + S_1 P_9 + S_2 P_8 + S_3 P_7 + \cdots + S_8 P_2 + S_9 P_1)/10.
\end{aligned}
\end{equation}
Substituting \eqref{k1to10b10e11} and \eqref{k1to10b10e12} into the following expression:
\begin{equation}
\begin{aligned}
\label{k1to10b10e13}
\begin{vmatrix}
S_2 & S_3 & S_4 & S_5 & S_6 \\
S_3 & S_4 & S_5 & S_6 & S_7 \\
S_4 & S_5 & S_6 & S_7 & S_8 \\
S_5 & S_6 & S_7 & S_8 & S_9 \\
S_6 & S_7 & S_8 & S_9 & S_{10}\\
\end{vmatrix} = 0
\end{aligned}
\end{equation}
we have:
\begin{align}
& y_{10}^{30} -30 \left(\beta _{11}-1\right) y_{10}^{29}+5 \left(87 \beta _{11}^2+62 \beta _{11}-149\right) y_{10}^{28} \nonumber\\
& \quad\; + \cdots -2 \left(\beta _{11}-1\right){}^{25} \left(15 \beta _{11}^4+220 \beta _{11}^3+594 \beta _{11}^2+396 \beta _{11}+55\right) y_{10} \nonumber\\
& \quad\; +\left(\beta _{11}-1\right){}^{25} \left(\beta _{11}^5+55 \beta _{11}^4+330 \beta _{11}^3+462 \beta _{11}^2+165 \beta _{11}+11\right)\nonumber\\
& \quad\; =0
\end{align}
Then we obtain the lookup table of \(\beta_{{10}, \min}(\beta_{11})\) for type \((k=1,2,3,\cdots,10)\) as follows:
\begin{equation}
\begin{aligned}
& \beta_{10,\min}\left(0.97975\right)=0.9797429732\dots \:, \\
& \beta_{10,\min}\left(0.97976\right)=0.9797329670\dots \:, \\
& \beta_{10,\min}\left(0.97977\right)=0.9797229535\dots \:, \\
& \qquad \vdots  \\
& \beta_{10,\min}\left(0.99997\right)=0.7498412730\dots \:, \\
& \beta_{10,\min}\left(0.99998\right)=0.7329238933\dots \:, \\
& \beta_{10,\min}\left(0.99999\right)=0.7039453066\dots \:.
\end{aligned}
\end{equation}

\begin{example}\textbf{Lower Bound of \( \boldsymbol{\beta_{n}} \) for \( \boldsymbol{(k=2,4,6,8,10,12)} \)} \\ [2mm]
For the given normalized GPTE system:
\begin{align}
& \left[ \alpha_{1}, \alpha_{2}, \alpha_{3}, \alpha_{4}, \alpha_{5},\alpha_{6},\alpha_{7} \right]^{k} = \left[ \beta_{1}, \beta_{2}, \beta_{3}, \beta_{4}, \beta_{5},\beta_{6},\beta_{7}  \right]^{k}, \nonumber\\
&\qquad\qquad\qquad\qquad\qquad\qquad\qquad (k = 2,4,6,8,10,12).
\end{align}
\noindent
Based on Conjecture \textup{\ref{conjecture_beta_n}}, the minimum value of \(\beta_5\) can be derived from the following system of six equations in six variables:
\begin{align}
\label{k24681012b6e1}
(1 - y_6^k - \beta_7^k)/2 = y_3^k+y_5^k-x_2^k-x_4^k-x_6^k, \quad (k = 2,4,6,8,10,12).
\end{align}
\end{example}
To solve the system of equations \eqref{k24681012b6e1} using the Second Generalization of the Girard-Newton Identities, the main derivation process follows a similar approach to that of previous example. Let
\begin{align}
\label{k24681012b6e11}
& P_k = (1 - y_6^k - \beta_7^k)/2, \quad (k = 2,4,6,8,10,12).
\end{align}
and
\begin{equation}
\label{k24681012b6e12}
\begin{aligned}
& S_2 = (P_2)/1, \\
& S_4 = (P_4 + S_2 P_2)/2, \\
& S_6 = (P_6 + S_2 P_4 + S_4 P_2)/3, \\
& S_8 = (P_8 + S_2 P_6 + S_4 P_4 + S_6 P_2)/4, \\
& S_{10} = (P_{10} + S_2 P_8 + S_4 P_6 + S_6 P_4 + S_8 P_2)/5, \\
& S_{12} = (P_{12} + S_2 P_{10} + S_4 P_8 + S_6 P_6 + S_8 P_4+ S_{10} P_2)/6 .
\end{aligned}
\end{equation}
Substituting \eqref{k24681012b6e11} and \eqref{k24681012b6e12} into the following expression:
\begin{equation}
\begin{aligned}
\begin{vmatrix}
S_4 & S_6 & S_8 \\
S_6 & S_8 & S_{10} \\
S_8 & S_{10} & S_{12} \\
\end{vmatrix} = 0
\end{aligned}
\end{equation}
we have:
\begin{flalign}
& \qquad y_6^{24}-12 \left(\beta _7^2-1\right) y_6^{22}+2 \left(\beta _7^2-1\right) \left(33 \beta _7^2+59\right) y_6^{20} \nonumber &\\
& \qquad\quad -4 \left(\beta _7^2-1\right) \left(55 \beta _7^4+102 \beta _7^2+91\right) y_6^{18} \nonumber &\\
& \qquad\quad  +\left(\beta _7^2-1\right) \left(495 \beta _7^6+235 \beta _7^4+1197 \beta _7^2+441\right) y_6^{16} \nonumber &\\
& \qquad\quad  -8 \left(\beta _7^2-1\right) \left(99 \beta _7^8-52 \beta _7^6+42 \beta _7^4+252 \beta _7^2-21\right) y_6^{14} \nonumber &\\
& \qquad\quad  +4 \left(\beta _7^2-1\right) \left(231 \beta _7^{10}+25 \beta _7^8-826 \beta _7^6+658 \beta _7^4+483 \beta _7^2-315\right) y_6^{12} \nonumber &\\
& \qquad\quad  -8 \left(\beta _7^2-1\right){}^4 \left(99 \beta _7^6+499 \beta _7^4+585 \beta _7^2+225\right) y_6^{10} \nonumber &\\
& \qquad\quad  +\left(\beta _7^2-1\right){}^4 \left(495 \beta _7^8+3188 \beta _7^6+6378 \beta _7^4+4500 \beta _7^2+1311\right) y_6^8 \nonumber &\\
& \qquad\quad  -4 \left(\beta _7^2-1\right){}^4 \left(55 \beta _7^{10}+285 \beta _7^8+998 \beta _7^6+1018 \beta _7^4+595 \beta _7^2+121\right) y_6^6 \nonumber &\\
& \qquad\quad  +2 \left(\beta _7^2-1\right){}^4 \left(33 \beta _7^{12}+38 \beta _7^{10}+435 \beta _7^8+804 \beta _7^6+375 \beta _7^4+342 \beta _7^2+21\right) y_6^4 \nonumber &\\
& \qquad\quad  -4 \left(\beta _7^2-1\right){}^9 \left(3 \beta _7^4+14 \beta _7^2+7\right) y_6^2 \nonumber &\\
& \qquad\quad  +\left(\beta _7-1\right){}^9 \left(\beta _7^6+21 \beta _7^4+35 \beta _7^2+7\right) =0 &
\end{flalign}
Then we obtain the lookup table of \(\beta_{{6}, \min}(\beta_{7})\) for type \((k=2,4,6,8,10,12)\) as follows:
\begin{equation}
\begin{aligned}
& \beta_{6,\min}\left(0.97493\right)=0.9749258242\dots \:, \\
& \beta_{6,\min}\left(0.97494\right)=0.9749158201\dots \:, \\
& \beta_{6,\min}\left(0.97495\right)=0.9749058102\dots \:, \\
& \qquad \vdots  \\
& \beta_{6,\min}\left(0.99997\right)=0.7001898470\dots \:, \\
& \beta_{6,\min}\left(0.99998\right)=0.6821639061\dots \:, \\
& \beta_{6,\min}\left(0.99999\right)=0.6517336031\dots \:.
\end{aligned}
\end{equation}

\begin{example} \textbf{Lower Bound of \( \boldsymbol{\beta_{n}} \) for \( \boldsymbol{(k=1,3,5,7,9)} \)} \\ [2mm]
For the given normalized GPTE system:
\begin{flalign}
&\quad \left[ \alpha_{1}, \alpha_{2}, \alpha_{3}, \alpha_{4}, \alpha_{5},\alpha_{6} \right]^{k} = \left[ \beta_{1}, \beta_{2}, \beta_{3}, \beta_{4}, \beta_{5},\beta_{6} \right]^{k}, \quad (k = 1,3,5,7,9).&
\end{flalign}
\noindent
Based on Conjecture \textup{\ref{conjecture_beta_n}}, the minimum value of \(\beta_5\) can be derived from the following system of five equations in five variables:
\begin{align}
\label{k13579b5e1}
(1 - y_5^k - \beta_6^k)/2 = y_2^k + y_4^k - x_3^k - x_5^k, \quad (k = 1,3,5,7,9).
\end{align}
\end{example}
To solve the system of equations \eqref{k13579b5e1} using the Third Generalization of the Girard-Newton Identities, the main derivation process is as follows:
Let
\begin{align}
\label{k13579b5e5}
P_k=y_2^k+y_4^k-x_3^k-x_5^k,\quad (k = 1,3,5,7,9).
\end{align}
and define
\begin{equation}
\begin{aligned}
\label{k13579b5e6}
& G_1=-P_1,\\
& G_3=\left(P_1^3-P_3\right)/3,\\
& G_5=\left(P_1^5-P_5\right)/5,\\
& G_7=\left(P_1^7-P_7\right)/7,\\
& G_9=\left(P_1^9-P_9\right)/9-G_3^3/3.
\end{aligned}
\end{equation}
Then, according to Identity \ref{identity_GNI3}, it follows that:
\begin{equation}
\begin{aligned}
\label{k13579b5e7}
\begin{vmatrix}
 G_5 & G_3 & G_1 \\
 G_7 & G_5 & G_1^3+G_3 \\
 G_9 & G_7 & G_1^5+G_3 G_1^2+G_5 \\
\end{vmatrix}=0
\end{aligned}
\end{equation}
Based on \eqref{k13579b5e1}, we replace \eqref{k13579b5e5} with the following expression:
\begin{align}
\label{k13579b5e8}
P_k=(1-y_5^k-\beta _6^k)/2,\quad (k = 1,3,5,7,9).
\end{align}
Substituting \eqref{k13579b5e8} and \eqref{k13579b5e6} into \eqref{k13579b5e7}, then we obtain:
\begin{flalign}
\label{k13579b5e9}
&\qquad y_5^{15}+3 \left(\beta _6-1\right) y_5^{14}-3 \left(\beta _6-1\right)^2 y_5^{13}
-\left(\beta _6-1\right) \left(17 \beta _6^2+18 \beta _6+17\right) y_5^{12} &\nonumber \\
&\qquad\quad -\left(\beta _6-1\right)^2 \left(\beta _6+3\right) \left(3 \beta _6+1\right) y_5^{11}&\nonumber \\
&\qquad\quad  +\left(\beta _6-1\right) \left(39 \beta _6^4+60 \beta _6^3+58 \beta _6^2+60 \beta _6+39\right) y_5^{10} &\nonumber \\
&\qquad\quad +\left(\beta _6^2-1\right)^2 \left(25 \beta _6^2-38 \beta _6+25\right) y_5^9 &\nonumber \\
&\qquad\quad -\left(\beta _6-1\right) \left(\beta _6+1\right)^2 \left(45 \beta _6^4-28 \beta _6^3+78 \beta _6^2-28 \beta _6+45\right) y_5^8 &\nonumber \\
&\qquad\quad -\left(\beta _6^2-1\right)^2 \left(45 \beta _6^4-72 \beta _6^3+70 \beta _6^2-72 \beta _6+45\right) y_5^7 &\nonumber \\
&\qquad\quad +\left(\beta _6-1\right) \left(\beta _6+1\right)^2 \left(25 \beta _6^6-42 \beta _6^5+87 \beta _6^4-76 \beta _6^3+87 \beta _6^2-42 \beta _6+25\right) y_5^6 &\nonumber \\
&\qquad\quad +\left(\beta _6-1\right)^4 \left(\beta _6+1\right)^2 \left(39 \beta _6^4+40 \beta _6^3+114 \beta _6^2+40 \beta _6+39\right) y_5^5  &\nonumber \\
&\qquad\quad -\left(\beta _6-1\right)^5 \left(\beta _6+1\right)^2 \left(3 \beta _6^4-12 \beta _6^3-14 \beta _6^2-12 \beta _6+3\right) y_5^4 &\nonumber \\
&\qquad\quad -\left(\beta _6-1\right)^4 \left(\beta _6+1\right)^2 \left(17 \beta _6^6+38 \beta _6^5+95 \beta _6^4+84 \beta _6^3+95 \beta _6^2+38 \beta _6+17\right) y_5^3 &\nonumber \\
&\qquad\quad -\left(\beta _6-1\right)^5 \left(\beta _6+1\right)^6 \left(3 \beta _6^2-2 \beta _6+3\right) y_5^2 &\nonumber \\
&\qquad\quad +\left(\beta _6-1\right)^4 \left(\beta _6+1\right)^6 \left(\beta _6^2+3\right) \left(3 \beta _6^2+1\right) y_5 &\nonumber \\
&\qquad\quad +\left(\beta _6-1\right)^9 \left(\beta _6+1\right)^6=0 &
\end{flalign}
By solving \eqref{k13579b5e9}, we obtain the lookup table of \(\beta_{{5}, \min}(\beta_{6})\) for type \((k=1,3,5,7,9)\) as follows:
\begin{equation}
\begin{aligned}
& \beta_{5,\min}(0.95950) = 0.9594859464\cdots, \\
& \beta_{5,\min}(0.95951) = 0.9594759421\cdots, \\
& \beta_{5,\min}(0.95952) = 0.9594659344\cdots, \\
& \qquad \vdots \\
& \beta_{5,\min}(0.99997) = 0.5460261545\cdots, \\
& \beta_{5,\min}(0.99998) = 0.5247903911\cdots, \\
& \beta_{5,\min}(0.99999) = 0.4898898843\cdots.
\end{aligned}
\end{equation}

\begin{example}
\textbf{Lower Bound of \( \boldsymbol{\beta_{n}} \) for \( \boldsymbol{(k=-1,0,1,2)} \)} \\ [2mm]
For the given normalized GPTE system:
\begin{align}
\left[ \alpha_{1}, \alpha_{2}, \alpha_{3}, \alpha_{4}, \alpha_{5} \right]^{k} = \left[ \beta_{1}, \beta_{2}, \beta_{3}, \beta_{4}, \beta_{5} \right]^{k}, \quad (k = -1,0,1,2).
\end{align}
\noindent
Based on Conjecture \textup{\ref{conjecture_beta_n}}, when \( \beta_{5} \) takes a specific value, \( \beta_{4} \) achieves its minimum value \( \beta_{4, \min}(\beta_{5}) \) under the following conditions:
\begin{equation}
\begin{cases}
& \beta_{5,min} \leq \beta_{5} < \alpha_{5}=1,\\
& \beta_3 =  \beta_2, \:\quad \beta_1 = \varepsilon_1, \\
& \alpha_4 = \alpha_3, \quad \alpha_2 = \alpha_1.
\end{cases}
\end{equation}
This leads to the following system of four equations in four variables:
\begin{align}
\left[ x_{2}, x_{2}, x_{4}, x_{4}, 1 \right]^{k} = \left[ \varepsilon_1, y_{3}, y_{3}, y_{4},y_{5} \right]^{k}, \quad (k = -1,0,1,2).
\end{align}
where \(\varepsilon_1\) is a positive infinitesimal real number.
Specifically,
\begin{equation}
\label{k012n1b4e1}
\begin{cases}
& \qquad\qquad\:\, x_2^2 \cdot x_4^2=\varepsilon_1 \cdot y_3^2 \cdot y_4 \cdot \beta_5 ,\\
& \quad\: 2 x_2^k+2 x_4^k+1=\varepsilon_1^k+2 y_3^k+ y_4^k+ \beta_5^k.\quad(k=-1,1,2)
\end{cases}
\end{equation}
\end{example}
To solve the system of equations \eqref{k012n1b4e1} using the Second Generalization of the Girard-Newton Identities, we first transform \eqref{k012n1b4e1} to obtain:
\begin{equation}
\label{k012n1b4e4}
\begin{cases}
& \qquad\qquad\quad\; \displaystyle\sqrt{\varepsilon_1 y_4 \beta _5}=\frac{x_2 x_4}{y_3}, \\[1mm]
& \displaystyle\frac{1}{2}\left(1-\varepsilon_1^k-y_4^k-\beta_5^k\right)=y_3^k-x_2^k-x_4^k,\quad (k = -1,1,2).
\end{cases}
\end{equation}
Subsequently, we define:
\begin{equation}
\label{k012n1b4e5}
\begin{cases}
& \displaystyle P_0=-\frac{x_2 x_4}{y_3}\\
& P_k=y_3^k-x_2^k-x_4^k,\quad (k = -1,1,2).
\end{cases}
\end{equation}
and
\begin{equation}
\label{k012n1b4e6}
\begin{cases}
& T_{-1}=P_0 P_{-1}, \\
& T_0=P_0,\\
& T_1=P_1,\\
& T_2=(P_2+ T_1 P_1)/2.
\end{cases}
\end{equation}
and
\begin{equation}
\label{k012n1b4e7}
\begin{cases}
& S_0=1-T_{-1},\\
& S_1=T_1-T_0,\\
& S_2=T_2.
\end{cases}
\end{equation}
Then, according to Identity \ref{identity_GNI2}, it follows that:
\begin{equation}
\begin{aligned}
\label{k012n1b4e8}
\begin{vmatrix}
 S_1 & S_0 \\
 S_2 & S_1 \\
\end{vmatrix}=0
\end{aligned}
\end{equation}
Based on \eqref{k012n1b4e4}, we replace \eqref{k012n1b4e5} with the following expression:
\begin{equation}
\label{k012n1b4e9}
\begin{cases}
& \displaystyle P_0=-\sqrt{\varepsilon_1 y_4 \beta _5}\\
& P_k=\displaystyle\frac{1}{2}\left(1-\varepsilon_1^k-y_4^k-\beta_5^k\right),\quad (k = -1,1,2).
\end{cases}
\end{equation}
Substituting \eqref{k012n1b4e9}, \eqref{k012n1b4e6}, and \eqref{k012n1b4e7} into \eqref{k012n1b4e8}, we obtain:
\begin{align}
\label{k012n1b4e10}
& \varepsilon _1^3 \left(-\beta _5+\beta _5 y_4-y_4\right)\nonumber\\
& \quad +\varepsilon _1^2 \left(2 \beta _5^2-2 \beta _5-2 \beta _5 y_4^2-2 \beta _5^2 y_4-11 \beta _5 y_4+2 y_4^2-2 y_4\right) \nonumber\\
& \quad +\varepsilon _1 \left(-\beta _5^3-2 \beta _5^2+3 \beta _5\right) \nonumber\\
& \quad +\varepsilon _1 \left(\beta _5 y_4^3-2 \beta _5^2 y_4^2-11 \beta _5 y_4^2+\beta _5^3 y_4-11 \beta _5^2 y_4+7 \beta _5 y_4-y_4^3-2 y_4^2+3 y_4\right) \nonumber\\
& \quad +2 \sqrt{\beta _5 y_4 \varepsilon _1} \left(2 \beta _5 \varepsilon _1+3 \beta _5^2-2 \beta _5-1+3 \varepsilon _1^2-2 \epsilon _1\right)\nonumber \\
& \quad +2 \sqrt{\beta _5 y_4 \varepsilon _1} \left(8 \beta _5 y_4 \varepsilon _1+2 \beta _5 y_4+2 y_4 \varepsilon _1+3 y_4^2-2 y_4\right) \nonumber\\
& \quad -y_4 \beta _5 \left(y_4^2+2 y_4-2 \beta _5 y_4+\beta _5^2+2 \beta _5-3\right)=0
\end{align}
Setting \(\varepsilon _1 = 0\), we obtain:
\begin{align}
\label{k012n1b4e11}
y_4^2-2 \left(\beta _5-1\right) y_4+\left(\beta _5-1\right) \left(\beta _5+3\right)=0
\end{align}
Then we obtain the lookup table of \(\beta_{{4}, \min}(\beta_{5})\) for type \((k=-1,0,1,2)\) as follows:
\begin{align}
& \beta_{4,\min}\left(0.75000\right)=0.7500000000\cdots, \nonumber\\
& \beta_{4,\min}\left(0.75001\right)=0.7499899998\cdots, \nonumber\\
& \beta_{4,\min}\left(0.75002\right)=0.7499899992\cdots, \nonumber\\
& \qquad \vdots \nonumber\\
& \beta_{4,\min}\left(0.99997\right)=0.0109244511\cdots,  \nonumber\\
& \beta_{4,\min}\left(0.99998\right)=0.0089242719\cdots,  \nonumber\\
& \beta_{4,\min}\left(0.99999\right)=0.0063145553\cdots.
\end{align}

\begin{example}
\textbf{Lower Bound of \( \boldsymbol{\beta_{n}} \) for \( \boldsymbol{(k=k_1, k_2, \cdots, k_n)} \), \( \boldsymbol{k_1 \notin \mathbb{Z}^+ }\)} \\ [1mm]
Based on Conjecture \textup{\ref{conjecture_beta_n}}, and using methods similar to those in the previous examples, we can obtain the following results:
\end{example}
$\bullet$ For type \( (k=0,1,2,3) \), we have \(\beta_{4,\min}(\beta_5) = y_4\), where \(y_4\) is a real root of the following equation:
\begin{align}
& y_4^4+4 \left(\beta _5-1\right) y_4^3-2 \left(\beta _5-1\right) \left(5 \beta _5+3\right) y_4^2 \nonumber\\
& \quad +4 \left(\beta _5-1\right) \left(\beta _5+1\right)^2 y_4+\left(\beta _5-1\right)^4=0
\end{align}

$\bullet$ For type $ (k=0,1,2,3,4) $,  we have \(\beta_{5,\min}(\beta_6) = y_5\), where \(y_5\) is a real root of the following equation:
\begin{align}
& y_5^6-6 \left(\beta _6-1\right) y_5^5+\left(\beta _6-1\right) \left(15 \beta _6+29\right) y_5^4 \nonumber\\
&\quad -4 \left(\beta _6-1\right) \left(5 \beta _6^2+10 \beta _6+9\right) y_5^3 \nonumber\\
&\quad +\left(\beta _6-1\right) \left(15 \beta _6^3-5 \beta _6^2+45 \beta _6+9\right) y_5^2 \nonumber\\
& \quad -2 \left(\beta _6-1\right)^4 \left(3 \beta _6+5\right) y_5+\left(\beta _6-1\right)^4 \left(\beta _6^2+10 \beta _6+5\right)=0
\end{align}

$\bullet$ For type $ (k=-1,0,1,2,3) $, we have \(\beta_{5,\min}(\beta_6) = y_5\), where \(y_5\) is a real root of the following equation:
\begin{align}
& y_5^4+4 \left(\beta _6-1\right) y_5^3-2 \left(\beta _6-1\right) \left(5 \beta _6+3\right) y_5^2 \nonumber\\
& \quad +4 \left(\beta _6-1\right) \left(\beta _6+1\right){}^2 y_5+\left(\beta _6-1\right){}^4=0
\end{align}

$\bullet$ For type $ (k=-2,-1,0,1,2) $, we have \(\beta_{5,\min}(\beta_6) = y_5\), where \(y_5\) is a real root of the following equation:
\begin{align}
& y_5^2-2 \left(\beta _6-1\right) y_5+  \left(\beta _6-1\right) \left(\beta _6+3\right)=0
\end{align}

\begin{example}
\textbf{Lower Bound of \( \boldsymbol{\beta_{n}} \) for \( \boldsymbol{(k=-1,1,3,5,7)} \)} \\ [2mm]
For the given normalized GPTE system:
\begin{flalign}
&\quad \left[ \alpha_{1}, \alpha_{2}, \alpha_{3}, \alpha_{4}, \alpha_{5},\alpha_{6} \right]^{k} = \left[ \beta_{1}, \beta_{2}, \beta_{3}, \beta_{4}, \beta_{5},\beta_{6} \right]^{k}, \quad (k = -1,1,3,5,7).&
\end{flalign}
\noindent
Based on Conjecture \textup{\ref{conjecture_beta_n}}, the minimum value of \(\beta_5\) can be derived from the following system of five equations in five variables:
\begin{align}
\label{k1357n1b5e1}
(1+\varepsilon_1^k-y_5^k-\beta _6^k)/2=y_2^k+y_4^k-x_3^k-x_5^k,\quad (k = -1,1,3,5,7).
\end{align}
\end{example}
To solve the system of equations \eqref{k1357n1b5e1} using the Third Generalization of the Girard-Newton Identities, the main derivation process is as follows: Let
\begin{align}
\label{k1357n1b5e5}
P_k=(1+\varepsilon_1^k-y_5^k-\beta _6^k)/2,\quad (k =-1,1,3,5,7).
\end{align}
and
\begin{equation}
\begin{aligned}
\label{k1357n1b5e6}
& G_1=-P_1,\\
& G_3=\left(P_1^3-P_3\right)/3,\\
& G_5=\left(P_1^5-P_5\right)/5,\\
& G_7=\left(P_1^7-P_7\right)/7,\\
& G_{-1}=P_{-1}.
\end{aligned}
\end{equation}
Substituting \eqref{k1357n1b5e5} and \eqref{k1357n1b5e6} into the following expression:
\begin{equation}
\begin{aligned}
\label{k1357n1b5e7}
\begin{vmatrix}
 G_{-1} G_1 G_3+G_3 & G_1 & G_{-1} \\
 G_5 & G_3 & G_1 \\
 G_7 & G_5 & G_1^3+G_3 \\
\end{vmatrix}=0
\end{aligned}
\end{equation}
we have:
\begin{align}
\label{k1357n1b5e8}
& y_5^{10} \left(\beta _6 \varepsilon _1+\beta _6-\varepsilon _1\right)+y_5^9 \left(-3 \beta _6^2 \varepsilon _1+3 \beta _6 \varepsilon _1^2-3 \beta _6^2+3 \beta _6-3 \varepsilon _1^2-3 \varepsilon _1\right) \nonumber \\
& \quad +y_5^8 \left(-13 \beta _6^2 \varepsilon _1+13 \beta _6 \varepsilon _1^2+13 \beta _6 \varepsilon _1\right)\nonumber \\
& \quad+ \cdots +\beta _6 \varepsilon _1 \left(1+\cdots+3 \beta _6 \varepsilon _1^8-3 \varepsilon _1^8+\varepsilon _1^9\right)=0
\end{align}
Setting \(\varepsilon _1 = 0\), we obtain:
\begin{align}
\label{k1357n1b5e9}
& y_5^9-3 \left(\beta _6-1\right) y_5^8+4 \left(\beta _6-1\right) \left(2 \beta _6^2+3 \beta _6+2\right) y_5^6 \nonumber\\
& \quad -6 \left(\beta _6^2-1\right)^2 y_5^5-2 \left(\beta _6-1\right) \left(\beta _6+1\right)^2 \left(3 \beta _6^2+4 \beta _6+3\right) y_5^4  \nonumber\\
& \quad +8 \left(\beta _6^2-1\right)^2 \left(\beta _6^2+1\right) y_5^3 +4 \left(\beta _6-1\right) \beta _6 \left(\beta _6+1\right)^4 y_5^2  \nonumber\\
& \quad -\left(\beta _6^2-1\right)^2 \left(\beta _6^2+3\right) \left(3 \beta _6^2+1\right) y_5+\left(\beta _6-1\right)^3 \left(\beta _6+1\right)^6=0
\end{align}
By solving \eqref{k1357n1b5e9}, we obtain the lookup table of \(\beta_{{5}, \min}(\beta_{6})\) for type \((k=-1,1,3,5,7)\) as follows:
\begin{align}
& \beta_{5,\min}\left(0.93970\right)=0.9396852409\cdots, \nonumber\\
& \beta_{5,\min}\left(0.93971\right)=0.9396752381\cdots, \nonumber\\
& \beta_{5,\min}\left(0.93972\right)=0.9396652329\cdots, \nonumber\\
& \qquad \vdots \nonumber\\
& \beta_{5,\min}\left(0.99997\right)=0.3972719027\cdots, \nonumber\\
& \beta_{5,\min}\left(0.99998\right)=0.3761123774\cdots, \nonumber\\
& \beta_{5,\min}\left(0.99999\right)=0.3422527379\cdots.
\end{align}

\begin{example}
\textbf{Lower Bound of \( \boldsymbol{\beta_{n}} \) for \( \boldsymbol{(k=-3,-1,1,3,5)} \) and others} \\ [1mm]
Based on Conjecture \textup{\ref{conjecture_beta_n}}, and using methods similar to those in the previous examples, we can obtain the following results:
\end{example}
$\bullet$ For type $ (k=-3,-1,1,3,5) $,  we have \(\beta_{5,\min}(\beta_6) = y_5\), where \(y_5\) is a real root of the following equation:
\begin{align}
& y_5^6+2 \left(\beta _6-1\right) y_5^5-\left(\beta _6-1\right)^2 y_5^4-4 \left(\beta _6^3-1\right) y_5^3 -\left(\beta _6^2-1\right)^2 y_5^2\nonumber\\
& \quad +2 \left(\beta _6+1\right) \left(\beta _6^4-1\right) y_5+\left(\beta _6-1\right)^4 \left(\beta _6+1\right)^2=0
\end{align}

$\bullet$ For type $ (k=-1,1,3,5) $,  we have \(\beta_{4,\min}(\beta_5) = y_4\), where \(y_4\) is a real root of the following equation:
\begin{align}
& y_4^6+ 2 \left(\beta _5-1\right) y_4^5-\left(\beta _5-1\right)^2 y_4^4-4 \left(\beta _5^3-1\right) y_4^3-\left(\beta _5^2-1\right){}^2 y_4^2 \nonumber\\
& \quad +2 \left(\beta _5+1\right) \left(\beta _5^4-1\right) y_4+\left(\beta _5-1\right)^4 \left(\beta _5+1\right)^2=0
\end{align}

$\bullet$ For type $ (k=-3,-1,1,3) $, we have \(\beta_{4,\min}(\beta_5) = y_4\), where \(y_4\) is a real root of the following equation:
\begin{align}
& y_4^3-\left(\beta _5-1\right) y_4^2-\left(\beta _5-1\right)^2 y_4+\left(\beta _5+1\right) \left(\beta _5^2-1\right)=0
\end{align}

\subsubsection{Upper and Lower Bounds of \(\alpha_{n}\) and Others}

\begin{conjecture}
\label{conjecture_alpha_n}\textnormal{\textbf{(Exact Bounds of \(\boldsymbol{\alpha_n}\))}}\\[1mm]
Let \(\{\alpha_{1}, \alpha_{2}, \dots, \alpha_{n+1}\}\) and \(\{\beta_{1}, \beta_{2}, \dots, \beta_{n+1}\}\) be two sets of non-negative real numbers bounded above by 1, ordered as follows:
\begin{align}
& 0 \leq \alpha_1 \leq \alpha_2 \leq \cdots \leq \alpha_{n}\leq \alpha_{n+1}=1, \\
& 0 \leq \beta_1 \leq \beta_2 \leq \cdots \leq \beta_{n} \leq \beta_{n+1}<1.
\end{align}
For given distinct integers \(\{k_1, k_2, \dots, k_n\}\), if the following system of equations holds:
\begin{align}
\left[ \alpha_{1}, \alpha_{2}, \dots, \alpha_{n+1} \right]^{k} = \left[ \beta_{1}, \beta_{2}, \dots, \beta_{n+1} \right]^{k}, \quad (k = k_1, k_2, \dots, k_n).
\end{align}
1. Let \( \alpha_{n, \min} (\beta_{n+1},\beta_n) \) denote the lower bound of \( \alpha_{n}\), abbreviated as \( \alpha_{n, \min} \). Then, \( \alpha_{n} = \alpha_{n, \min}\) under the following conditions:
\begin{equation}
\begin{cases}
\label{conjecture_alpha_n_min}
& \beta_{n+1,min} \leq \beta_{n+1} < \alpha_{n+1}=1,\\
& \beta_{n,min} \leq \beta_{n} \leq \beta_{n+1},\\
& \beta_{n-1} =  \beta_{n-2},\quad \beta_{n-3} =  \beta_{n-4},\quad \cdots ,\\
& \alpha_{n} = \alpha_{n-1}, \quad \alpha_{n-2} \geq \alpha_{n-3}, \quad \alpha_{n-4}= \alpha_{n-5},\quad \cdots, \\
& \alpha_1 \geq 0 \quad \textup{(if \(n\) is odd)} \quad \textup{or} \quad \beta_1 \geq 0 \quad \textup{(if \(n\) is even)}.
\end{cases}
\end{equation}
As a special case, if the smaller of \(\alpha_1\) and \(\beta_1\) is set to 0, then \(\alpha_{n} = \alpha_{n, \min}\) under the following conditions:
\begin{equation}
\begin{cases}
\label{conjecture_alpha_n_min0}
& \beta_{n+1,min} \leq \beta_{n+1} < \alpha_{n+1}=1,\\
& \beta_{n,min} \leq \beta_{n} \leq \beta_{n+1},\\
& \beta_{n-1} =  \beta_{n-2},\quad \beta_{n-3} =  \beta_{n-4},\quad \cdots ,\\
& \alpha_{n} \geq \alpha_{n-1}, \quad \alpha_{n-2} = \alpha_{n-3}, \quad \alpha_{n-4}= \alpha_{n-5},\quad \cdots, \\
& \beta_2 \geq \beta_1 \quad \textup{(if \(n\) is odd)} \quad \textup{or} \quad \alpha_2 \geq \alpha_1 \quad \textup{(if \(n\) is even)},\\
& \alpha_1 = 0 \;\,\quad \textup{(if \(n\) is odd)} \quad \textup{or} \quad \beta_1 = 0 \quad\;\; \textup{(if \(n\) is even)}.
\end{cases}
\end{equation}
2. Let \( \alpha_{n, \max}  (\beta_{n+1},\beta_n) \) denote the upper bound of \( \alpha_{n}\), abbreviated as \( \alpha_{n, \max} \). Then, \( \alpha_{n} = \alpha_{n, \max}\) under the following conditions:
\begin{equation}
\begin{cases}
\label{conjecture_alpha_n_max}
& \beta_{n+1,min} \leq \beta_{n+1} < \alpha_{n+1}=1,\\
& \beta_{n,min} \leq \beta_{n} < \beta_{n+1},\\
& \beta_{n-1} =  \beta_{n-2},\quad \beta_{n-3} =  \beta_{n-4},\quad \cdots , \\
& \alpha_{n} \geq \alpha_{n-1}, \quad \alpha_{n-2} = \alpha_{n-3}, \quad \alpha_{n-4}= \alpha_{n-5},\quad \cdots , \\
& \alpha_1 = \varepsilon \quad \textup{(if \(n\) is odd)} \quad \textup{or} \quad \beta_1 = \varepsilon \quad \textup{(if \(n\) is even)}.
\end{cases}
\end{equation}
where \(\varepsilon\) is a positive infinitesimal real number.
\end{conjecture}
\begin{example}\textbf{Upper Bound of \( \boldsymbol{\alpha_{n}} \) for \( \boldsymbol{(k=1,2,3,4)} \)} \\ [2mm]
For the given normalized GPTE system:
\begin{align}
&\quad \left[ \alpha_{1}, \alpha_{2}, \alpha_{3}, \alpha_{4}, \alpha_{5} \right]^{k} = \left[ \beta_{1}, \beta_{2}, \beta_{3}, \beta_{4}, \beta_{5} \right]^{k}, \quad (k = 1,2,3,4).
\end{align}
\noindent
Based on \eqref{conjecture_alpha_n_max} of Conjecture \textup{\ref{conjecture_alpha_n}}, when \( \beta_{5},\beta_{4}\) take specific values, \( \alpha_{4} \) achieves its maximum value \( \alpha_{4, \max}(\beta_{5},\beta_{4} ) \) under the following conditions:
\begin{equation}
\begin{cases}
& \beta_{5,min} \leq \beta_{5} < \alpha_{5}=1,\\
& \beta_{4,min} \leq \beta_{4} \leq \beta_{5},\\
& \beta_3 =  \beta_2, \:\quad \beta_1 =  0, \\
& \alpha_4 \geq \alpha_3, \quad \alpha_2 = \alpha_1.
\end{cases}
\end{equation}
This leads to the following system of four equations in four variables:
\begin{align}
\left[ x_{2}, x_{2}, x_{3}, x_{4}, 1 \right]^{k} = \left[ 0, y_{3}, y_{3}, \beta_{4}, \beta_{5} \right]^{k}, \quad (k = 1,2,3,4).
\end{align}
Specifically,
\begin{align}
\label{k1234a4s1}
2 x_2^k+ x_3^k+ x_4^k+1=2 y_3^k+ \beta_4^k+ \beta_5^k,\quad (k = 1,2,3,4).
\end{align}
\end{example}
To solve the system of equations \eqref{k1234a4s1} using the Second Generalization of the Girard-Newton Identities, the main derivation process is as follows: Let
\begin{align}
\label{k1234a4s2}
& u_k=1+x_4^k-\beta _4^k-\beta _5^k,\quad (k=1,2,3,4)
\end{align}
then we have
\begin{align}
\label{k1234a4s4}
& (u_k+x_3^k)/2=y_3^k-x_2^k,\quad (k=1,2,3,4)
\end{align}
Let
\begin{align}
\label{k1234a4s5}
P_k=y_3^k-x_2^k,\quad (k = 1,2,3,4).
\end{align}
and define
\begin{equation}
\begin{aligned}
\label{k1234a4s6}
& S_1=(P_1)/1,\\
& S_2=(P_2 + S_1 P_1)/2,\\
& S_3=(P_3 + S_1 P_2 + S_2 P_1)/3,\\
& S_4=(P_4 + S_1 P_3 + S_2 P_2+ S_3 P_1)/4.\\
\end{aligned}
\end{equation}
Then, according to Identity \ref{identity_GNI2}, it follows that:
\begin{equation}
\begin{aligned}
\label{k1234a4s7}
\begin{vmatrix}
 S_1 & S_2 \\
 S_2 & S_3 \\
\end{vmatrix}=0\\[1mm]
\begin{vmatrix}
 S_2 & S_3 \\
 S_3 & S_4 \\
\end{vmatrix}=0
\end{aligned}
\end{equation}
Based on \eqref{k1234a4s4}, we replace \eqref{k1234a4s5} with the following expression:
\begin{align}
\label{k1234a4s8}
P_k=(u_k+x_3^k)/2,\quad (k = 1,2,3,4).
\end{align}
Substituting \eqref{k1234a4s8} into \eqref{k1234a4s6} and \eqref{k1234a4s7}, then we obtain:
\begin{align}
\label{k1234a4s9}
& 3 x_3^4+12 u_1 x_3^3+\left(-6 u_1^2-24 u_2\right) x_3^2+\left(16 u_3-4 u_1^3\right) x_3 \nonumber\\
& \quad -u_1^4+16 u_3 u_1-12 u_2^2=0
\end{align}
and
\begin{align}
\label{k1234a4s10}
& 45 x_3^6+90 u_1 x_3^5+\left(153 u_1^2+234 u_2\right) x_3^4+\left(-12 u_1^3-360 u_2 u_1-672 u_3\right) x_3^3 \nonumber\\
&\quad +\left(-21 u_1^4-108 u_2 u_1^2-96 u_3 u_1+180 u_2^2+432 u_4\right) x_3^2 \nonumber\\
&\quad +\left(-6 u_1^5-24 u_2 u_1^3+96 u_3 u_1^2-72 u_2^2 u_1+288 u_4 u_1-192 u_2 u_3\right) x_3 \nonumber\\
&\quad -u_1^6-6 u_2 u_1^4+32 u_3 u_1^3-36 u_2^2 u_1^2+144 u_4 u_1^2-192 u_2 u_3 u_1\nonumber\\
&\quad+72 u_2^3-256 u_3^2+288 u_2 u_4=0
\end{align}
From \eqref{k1234a4s9} and \eqref{k1234a4s10}, we have:
\begin{align}
\label{k1234a4s11}
& u_1^{18}+9 u_2 u_1^{16}-64 u_3 u_1^{15}+84 u_2^2 u_1^{14}-144 u_4 u_1^{14}-192 u_2 u_3 u_1^{13}+292 u_2^3 u_1^{12}  \nonumber\\
& \quad +1536 u_3^2 u_1^{12}-1008 u_2 u_4 u_1^{12}-1920 u_2^2 u_3 u_1^{11}+5184 u_3 u_4 u_1^{11}+1158 u_2^4 u_1^{10}  \nonumber\\
& \quad-1152 u_2 u_3^2 u_1^{10}+5112 u_4^2 u_1^{10}-6912 u_2^2 u_4 u_1^{10}-16768 u_3^3 u_1^9+2240 u_2^3 u_3 u_1^9  \nonumber\\
&\quad +8064 u_2 u_3 u_4 u_1^9+486 u_2^5 u_1^8+26784 u_2^2 u_3^2 u_1^8+25560 u_2 u_4^2 u_1^8 +29241 u_2^9 \nonumber\\
&\quad -15552 u_2^3 u_4 u_1^8-55008 u_3^2 u_4 u_1^8+34560 u_2 u_3^3 u_1^7-39168 u_3 u_4^2 u_1^7 \nonumber\\
&\quad -9792 u_2^4 u_3 u_1^7+49536 u_2^2 u_3 u_4 u_1^7+2772 u_2^6 u_1^6+77632 u_3^4 u_1^6+42048 u_4^3 u_1^6 \nonumber\\
&\quad -49152 u_2^3 u_3^2 u_1^6+80496 u_2^2 u_4^2 u_1^6-34128 u_2^4 u_4 u_1^6-63360 u_2 u_3^2 u_4 u_1^6 \nonumber\\
&\quad -109056 u_2^2 u_3^3 u_1^5-389376 u_2 u_3 u_4^2 u_1^5+36288 u_2^5 u_3 u_1^5+165888 u_3^3 u_4 u_1^5 \nonumber\\
&\quad +13824 u_2^3 u_3 u_4 u_1^5-6588 u_2^7 u_1^4-98880 u_2 u_3^4 u_1^4+126144 u_2 u_4^3 u_1^4 \nonumber\\
&\quad +78912 u_2^4 u_3^2 u_1^4+275184 u_2^3 u_4^2 u_1^4-66816 u_3^2 u_4^2 u_1^4-24624 u_2^5 u_4 u_1^4 \nonumber\\
&\quad +310464 u_2^2 u_3^2 u_4 u_1^4-92160 u_3^5 u_1^3-100096 u_2^3 u_3^3 u_1^3+16128 u_3 u_4^3 u_1^3 \nonumber\\
&\quad  -764928 u_2^2 u_3 u_4^2 u_1^3-16128 u_2^6 u_3 u_1^3+681984 u_2 u_3^3 u_4 u_1^3-499392 u_2^4 u_3 u_4 u_1^3 \nonumber\\
&\quad +1377 u_2^8 u_1^2-435264 u_2^2 u_3^4 u_1^2+1296 u_4^4 u_1^2+114048 u_2^2 u_4^3 u_1^2-49152 u_3^6 \nonumber\\
&\quad +321408 u_2^5 u_3^2 u_1^2+424440 u_2^4 u_4^2 u_1^2-230400 u_2 u_3^2 u_4^2 u_1^2+116640 u_2^6 u_4 u_1^2 \nonumber\\
&\quad+87552 u_3^4 u_4 u_1^2+203904 u_2^3 u_3^2 u_4 u_1^2-313344 u_2 u_3^5 u_1+513408 u_2^4 u_3^3 u_1 \nonumber\\
&\quad+13824 u_2 u_3 u_4^3 u_1-13824 u_3^3 u_4^2 u_1-380160 u_2^3 u_3 u_4^2 u_1-174528 u_2^7 u_3 u_1 \nonumber\\
&\quad +709632 u_2^2 u_3^3 u_4 u_1-611712 u_2^5 u_3 u_4 u_1 +165184 u_2^3 u_3^4+1296 u_2 u_4^4 \nonumber\\
&\quad +31104 u_2^3 u_4^3-1152 u_3^2 u_4^3-127584 u_2^6 u_3^2+198936 u_2^5 u_4^2-125568 u_2^2 u_3^2 u_4^2 \nonumber\\
&\quad +147744 u_2^7 u_4+142848 u_2 u_3^4 u_4 -367200 u_2^4 u_3^2 u_4 =0
\end{align}
Substituting \eqref{k1234a4s2} into \eqref{k1234a4s11}, we obtain:
\begin{align}
\label{k1234a4s12}
& \left(\beta _4^2-2 \beta _5 \beta _4+2 \beta _4+\beta _5^2+2 \beta _5-3\right){}^2 x_4^{14} \nonumber\\
& \quad -\left(\beta _4^2-2 \beta _5 \beta _4+2 \beta _4+\beta _5^2+2 \beta _5-3\right) \left(7 \beta _4^3-7 \beta _5 \beta _4^2+\cdots-19 \beta _5+5\right) x_4^{13} \nonumber\\
& \quad +\left(21 \beta _4^6-35 \beta _5 \beta _4^5+\cdots-145 \beta _5+40\right)x_4^{12}+\cdots \nonumber\\
& \quad + \left(\beta _4-1\right){}^4 \left(\beta _4-\beta _5\right){}^2 \left(\beta _5-1\right){}^4 \left(-7 \beta _5^2 \beta _4^5-2 \beta _5 \beta _4^5+\cdots -5 \beta _5^2+4 \beta _5\right)x_4 \nonumber\\
& \quad +\left(\beta _4-1\right){}^7 \left(\beta _4-\beta _5\right){}^4 \left(\beta _5-1\right){}^7=0
\end{align}
By solving \eqref{k1234a4s12}, we obtain \(\alpha_{4,\max}(\beta_5, \beta_4) = x_4\). For instance, when \(\beta_5 = 0.98521, \beta_4=0.70944 \), using Mathematica's \texttt{NSolve} function and incorporating the constraints from Corollary \ref{corollary_Conservative_Bounds}, namely \( ((\beta _4^4+\beta _5^4-1)/2)^{1/4} <x_4< (\beta _4^4+\beta _5^4-1 )^{1/4} \) , we obtain:
\begin{align}
 & \alpha_{4,\max}(0.98521, 0.70944)= x_4=0.61966775125591320827\cdots.
\end{align}

\begin{example}\textbf{Lower Bound of \( \boldsymbol{\alpha_{n}} \) for \( \boldsymbol{(k=1,2,3,4)} \)} \\ [2mm]
For the given normalized GPTE system:
\begin{align}
&\quad \left[ \alpha_{1}, \alpha_{2}, \alpha_{3}, \alpha_{4}, \alpha_{5} \right]^{k} = \left[ \beta_{1}, \beta_{2}, \beta_{3}, \beta_{4}, \beta_{5} \right]^{k}, \quad (k = 1,2,3,4).
\end{align}
\noindent
Based on \eqref{conjecture_alpha_n_min} of Conjecture \textup{\ref{conjecture_alpha_n}}, when \( \beta_{5},\beta_{4} \) take specific values, \( \alpha_{4} \) achieves its minimum value \( \alpha_{4, \min}(\beta_{5},\beta_{4} ) \) under the following conditions:
\begin{equation}
\begin{cases}
& \beta_{5,min} \leq \beta_{5} < \alpha_{5}=1,\\
& \beta_{4,min} \leq \beta_{4} \leq \beta_{5},\\
& \beta_3 =  \beta_2, \:\quad \beta_1 \geq  0, \\
& \alpha_4 = \alpha_3, \quad \alpha_2 = \alpha_1.
\end{cases}
\end{equation}
This leads to the following system of four equations in four variables:
\begin{align}
\left[ x_{2}, x_{2}, x_{4}, x_{4}, 1 \right]^{k} = \left[ y_{1}, y_{3}, y_{3}, \beta_{4}, \beta_{5} \right]^{k}, \quad (k = 1,2,3,4).
\end{align}
Specifically,
\begin{align}
\label{k1234a4e1}
2 x_2^k+ 2 x_4^k+1=y_1^k+2 y_3^k+ \beta_4^k+ \beta_5^k,\quad (k = 1,2,3,4).
\end{align}
\end{example}
To solve the system of equations \eqref{k1234a4e1} using the Second Generalization of the Girard-Newton Identities, the main derivation process is as follows: Let
\begin{align}
\label{k1234a4e2}
& u_k=(1+2 x_4^k-\beta _4^k-\beta _5^k)/2,\quad (k=1,2,3,4)
\end{align}
we have
\begin{align}
\label{k1234a4e4}
& (u_k-y_1^k)/2=y_3^k-x_2^k,\quad (k=1,2,3,4)
\end{align}
Let
\begin{align}
\label{k1234a4e5}
P_k=y_3^k-x_2^k,\quad (k = 1,2,3,4).
\end{align}
and define
\begin{equation}
\begin{aligned}
\label{k1234a4e6}
& S_1=(P_1)/1,\\
& S_2=(P_2 + S_1 P_1)/2,\\
& S_3=(P_3 + S_1 P_2 + S_2 P_1)/3,\\
& S_4=(P_4 + S_1 P_3 + S_2 P_2+ S_3 P_1)/4.\\
\end{aligned}
\end{equation}
Then, according to Identity \ref{identity_GNI2}, it follows that:
\begin{equation}
\begin{aligned}
\label{k1234a4e7}
\begin{vmatrix}
 S_1 & S_2 \\
 S_2 & S_3 \\
\end{vmatrix}=0\\[1mm]
\begin{vmatrix}
 S_2 & S_3 \\
 S_3 & S_4 \\
\end{vmatrix}=0
\end{aligned}
\end{equation}
Based on \eqref{k1234a4e4}, we replace \eqref{k1234a4e5} with the following expression:
\begin{align}
\label{k1234a4e8}
P_k=u_k-y_1^k/2,\quad (k = 1,2,3,4).
\end{align}
Substituting \eqref{k1234a4e8} into \eqref{k1234a4e6} and \eqref{k1234a4e7}, then we obtain:
\begin{align}
\label{k1234a4e9}
& -32 u_1^3 y_1+24 u_1^2 y_1^2+24 u_1 y_1^3-48 u_2 y_1^2+32 u_3 y_1+16 u_1^4-64 u_3 u_1 \nonumber\\
& \quad +48 u_2^2-3 y_1^4=0
\end{align}
and
\begin{align}
\label{k1234a4e10}
& -192 u_1^5 y_1+144 u_1^4 y_1^2+288 u_1^3 y_1^3-384 u_2 u_1^3 y_1+252 u_1^2 y_1^4\nonumber\\
& \quad -288 u_2 u_1^2 y_1^2+768 u_3 u_1^2 y_1-108 u_1 y_1^5-288 u_2 u_1 y_1^3-1152 u_3 u_1 y_1^2 \nonumber\\
& \quad-576 u_2^2 u_1 y_1+1152 u_4 u_1 y_1+396 u_2 y_1^4-576 u_3 y_1^3+1008 u_2^2 y_1^2+288 u_4 y_1^2 \nonumber\\
& \quad-768 u_2 u_3 y_1+64 u_1^6+192 u_2 u_1^4-512 u_3 u_1^3+576 u_2^2 u_1^2-1152 u_4 u_1^2\nonumber\\
& \quad +1536 u_2 u_3 u_1-576 u_2^3+1024 u_3^2-1152 u_2 u_4-9 y_1^6=0
\end{align}
From \eqref{k1234a4e9} and \eqref{k1234a4e10}, we have:
\begin{align}
\label{k1234a4e11}
&
512 u_1^{18}-2304 u_2 u_1^{16}-8192 u_3 u_1^{15}+10752 u_2^2 u_1^{14}+9216 u_4 u_1^{14}-42048 u_4^3 u_1^6\nonumber\\
&\quad -18688 u_2^3 u_1^{12}+49152 u_3^2 u_1^{12}-32256 u_2 u_4 u_1^{12}-61440 u_2^2 u_3 u_1^{11} \nonumber\\
&\quad -82944 u_3 u_4 u_1^{11}+37056 u_2^4 u_1^{10}+18432 u_2 u_3^2 u_1^{10}+40896 u_4^2 u_1^{10}\nonumber\\
&\quad +110592 u_2^2 u_4 u_1^{10}-134144 u_3^3 u_1^9-35840 u_2^3 u_3 u_1^9+64512 u_2 u_3 u_4 u_1^9\nonumber\\
&\quad -7776 u_2^5 u_1^8+214272 u_2^2 u_3^2 u_1^8-102240 u_2 u_4^2 u_1^8-124416 u_2^3 u_4 u_1^8\nonumber\\
&\quad +220032 u_3^2 u_4 u_1^8-138240 u_2 u_3^3 u_1^7-78336 u_3 u_4^2 u_1^7-78336 u_2^4 u_3 u_1^7\nonumber\\
&\quad -198144 u_2^2 u_3 u_4 u_1^7+22176 u_2^6 u_1^6+155264 u_3^4 u_1^6+12288 u_2 u_3 u_1^{13}\nonumber\\
&\quad +160992 u_2^2 u_4^2 u_1^6+136512 u_2^4 u_4 u_1^6-126720 u_2 u_3^2 u_4 u_1^6-218112 u_2^2 u_3^3 u_1^5\nonumber\\
&\quad +389376 u_2 u_3 u_4^2 u_1^5-145152 u_2^5 u_3 u_1^5-165888 u_3^3 u_4 u_1^5+27648 u_2^3 u_3 u_4 u_1^5\nonumber\\
&\quad +26352 u_2^7 u_1^4+98880 u_2 u_3^4 u_1^4+63072 u_2 u_4^3 u_1^4+157824 u_2^4 u_3^2 u_1^4\nonumber\\
&\quad -275184 u_2^3 u_4^2 u_1^4-33408 u_3^2 u_4^2 u_1^4-49248 u_2^5 u_4 u_1^4-310464 u_2^2 u_3^2 u_4 u_1^4\nonumber\\
&\quad -46080 u_3^5 u_1^3+100096 u_2^3 u_3^3 u_1^3-4032 u_3 u_4^3 u_1^3-382464 u_2^2 u_3 u_4^2 u_1^3\nonumber\\
&\quad -32256 u_2^6 u_3 u_1^3+340992 u_2 u_3^3 u_4 u_1^3+499392 u_2^4 u_3 u_4 u_1^3+2754 u_2^8 u_1^2\nonumber\\
&\quad -217632 u_2^2 u_3^4 u_1^2+162 u_4^4 u_1^2-28512 u_2^2 u_4^3 u_1^2-321408 u_2^5 u_3^2 u_1^2-81 u_2 u_4^4\nonumber\\
&\quad +212220 u_2^4 u_4^2 u_1^2+57600 u_2 u_3^2 u_4^2 u_1^2-116640 u_2^6 u_4 u_1^2-21888 u_3^4 u_4 u_1^2\nonumber\\
&\quad +101952 u_2^3 u_3^2 u_4 u_1^2+78336 u_2 u_3^5 u_1+256704 u_2^4 u_3^3 u_1+1728 u_2 u_3 u_4^3 u_1\nonumber\\
&\quad -1728 u_3^3 u_4^2 u_1+95040 u_2^3 u_3 u_4^2 u_1+174528 u_2^7 u_3 u_1-177408 u_2^2 u_3^3 u_4 u_1\nonumber\\
&\quad -305856 u_2^5 u_3 u_4 u_1-29241 u_2^9-6144 u_3^6-41296 u_2^3 u_3^4+3888 u_2^3 u_4^3\nonumber\\
&\quad +72 u_3^2 u_4^3-63792 u_2^6 u_3^2-49734 u_2^5 u_4^2-15696 u_2^2 u_3^2 u_4^2+73872 u_2^7 u_4\nonumber\\
&\quad +196608 u_2^3 u_3^2 u_1^6+17856 u_2 u_3^4 u_4+91800 u_2^4 u_3^2 u_4=0
\end{align}
Substituting \eqref{k1234a4e2} into \eqref{k1234a4e11}, we obtain:
\begin{align}
\label{k1234a4e12}
& 16 x_4^{12}-32 \left(5 \beta _4+5 \beta _5-4\right)x_4^{11}\nonumber\\
& \quad+16 \left(45 \beta _4^2+55 \beta _5 \beta _4-35 \beta _4+45 \beta _5^2-35 \beta _5-9\right)x_4^{10}+\cdots \nonumber\\
& \quad +2 \left(12 \beta _5 \beta _4^{10}+4 \beta _4^{10}+\cdots+31 \beta _5^5-7 \beta _5^4\right) x_4 \nonumber\\
& \quad + (9 \beta _5^2 \beta _4^{10}+\cdots+5 \beta _5^6-\beta _5^5)=0
\end{align}
By solving \eqref{k1234a4e12}, we obtain \(\alpha_{4,\min}(\beta_5, \beta_4) = x_4\). For instance, when \(\beta_5 = 0.98521, \beta_4=0.70944 \), using Mathematica's \texttt{NSolve} function and incorporating the constraints from Corollary \ref{corollary_Conservative_Bounds}, namely \( ((\beta _4^4+\beta _5^4-1)/2)^{1/4} <x_4< (\beta _4^4+\beta _5^4-1 )^{1/4} \) , we obtain:
\begin{align}
 & \alpha_{4,\min}(0.98521, 0.70944)= x_4= 0.57385447758791943219\cdots.
\end{align}

\begin{example}\textbf{Lower Bound of \( \boldsymbol{\alpha_{n}} \) for \( \boldsymbol{(k=1,2,3,4), \beta_1=0} \)} \\ [2mm]
For the given normalized GPTE system:
\begin{align}
&\quad \left[ \alpha_{1}, \alpha_{2}, \alpha_{3}, \alpha_{4}, \alpha_{5} \right]^{k} = \left[ 0, \beta_{2}, \beta_{3}, \beta_{4}, \beta_{5} \right]^{k}, \quad (k = 1,2,3,4).
\end{align}
\noindent
Based on \eqref{conjecture_alpha_n_min0} of Conjecture \textup{\ref{conjecture_alpha_n}}, when \( \beta_{5},\beta_{4} \) take specific values, \( \alpha_{4} \) achieves its minimum value \( \alpha_{4, \min}(\beta_{5},\beta_{4} ) \) under the following conditions:
\begin{equation}
\begin{cases}
& \beta_{5,min} \leq \beta_{5} < \alpha_{5}=1,\\
& \beta_{4,min} \leq \beta_{4} \leq \beta_{5},\\
& \beta_3 =  \beta_2, \:\quad \beta_1 =  0, \\
& \alpha_4 = \alpha_3, \quad \alpha_2 \geq \alpha_1.
\end{cases}
\end{equation}
This leads to the following system of four equations in four variables:
\begin{align}
\left[ x_{1}, x_{2}, x_{4}, x_{4}, 1 \right]^{k} = \left[ 0, y_{3}, y_{3}, \beta_{4}, \beta_{5} \right]^{k}, \quad (k = 1,2,3,4).
\end{align}
Specifically,
\begin{align}
\label{k1234b0a4e1}
x_1^k+ x_2^k+ 2 x_4^k+1=2 y_3^k+ \beta_4^k+ \beta_5^k,\quad (k = 1,2,3,4).
\end{align}
\end{example}
To solve the system of equations \eqref{k1234b0a4e1} using the Second Generalization of the Girard-Newton Identities, the main derivation process is as follows: Let
\begin{align}
\label{k1234b0a4e2}
& u_k=(1+2 x_4^k-\beta _4^k-\beta _5^k)/2,\quad (k=1,2,3,4)
\end{align}
we have
\begin{align}
\label{k1234b0a4e4}
& 2 u_k-2 y_3^k=-x_1^k-x_2^k,\quad (k=1,2,3,4)
\end{align}
Let
\begin{align}
\label{k1234b0a4e5}
P_k=-x_1^k-x_2^k,\quad (k = 1,2,3,4).
\end{align}
and define
\begin{equation}
\begin{aligned}
\label{k1234b0a4e6}
& S_1=(P_1)/1,\\
& S_2=(P_2 + S_1 P_1)/2,\\
& S_3=(P_3 + S_1 P_2 + S_2 P_1)/3,\\
& S_4=(P_4 + S_1 P_3 + S_2 P_2+ S_3 P_1)/4.\\
\end{aligned}
\end{equation}
Then, according to Identity \ref{identity_GNI2}, it follows that:
\begin{equation}
\begin{aligned}
\label{k1234b0a4e7}
 & S_3 =0\\
 & S_4 =0
\end{aligned}
\end{equation}
Based on \eqref{k1234b0a4e4}, we replace \eqref{k1234b0a4e5} with the following expression:
\begin{align}
\label{k1234b0a4e8}
P_k=2 u_k-2 y_3^k,\quad (k = 1,2,3,4).
\end{align}
Substituting \eqref{k1234b0a4e8} into \eqref{k1234b0a4e6} and \eqref{k1234b0a4e7}, then we obtain:
\begin{align}
\label{k1234b0a4e9}
& -6 u_1^2 y_3+3 u_1 y_3^2-3 u_2 y_3+2 u_1^3+3 u_2 u_1+u_3=0
\end{align}
and
\begin{align}
\label{k1234b0a4e10}
& -16 u_1^3 y_3+12 u_1^2 y_3^2-24 u_2 u_1 y_3+6 u_2 y_3^2-8 u_3 y_3+4 u_1^4+12 u_2 u_1^2\nonumber\\
&\quad +8 u_3 u_1+3 u_2^2+3 u_4=0
\end{align}
From \eqref{k1234a4e9} and \eqref{k1234a4e10}, we have:
\begin{align}
\label{k1234b0a4e11}
& 16 u_1^{10}-96 u_3 u_1^7+72 u_2^2 u_1^6-72 u_4 u_1^6+144 u_2 u_3 u_1^5-72 u_2^3 u_1^4+144 u_3^2 u_1^4 \nonumber\\
&\quad -216 u_2^2 u_3 u_1^3+72 u_3 u_4 u_1^3+81 u_2^4 u_1^2-144 u_2 u_3^2 u_1^2-27 u_4^2 u_1^2\nonumber\\
&\quad -54 u_2^2 u_4 u_1^2-64 u_3^3 u_1+180 u_2^3 u_3 u_1+108 u_2 u_3 u_4 u_1-54 u_2^5 \nonumber\\
&\quad +36 u_2^2 u_3^2-54 u_2^3 u_4=0
\end{align}
Substituting \eqref{k1234b0a4e2} into \eqref{k1234b0a4e11}, we obtain:
\begin{align}
\label{k1234b0a4e12}
& 4 x_4^7 +\left(-12 \beta _4-12 \beta _5+8\right) x_4^6 \nonumber\\
& \quad +\left(12 \beta _4^2+36 \beta _5 \beta _4-24 \beta _4+12 \beta _5^2-24 \beta _5+12\right) x_4^5 \nonumber\\
& \quad +\left(-4 \beta _4^3-36 \beta _5 \beta _4^2+24 \beta _4^2+\cdots +24 \beta _5^2-11 \beta _5-9\right) x_4^4 \nonumber\\
& \quad +\left(12 \beta _5 \beta _4^3-8 \beta _4^3+36 \beta _5^2 \beta _4^2 +\cdots -8 \beta _5^3-4 \beta _5^2+12 \beta _5\right) x_4^3 \nonumber\\
& \quad +\left(-12 \beta _5^2 \beta _4^3+8 \beta _5 \beta _4^3+4 \beta _4^3+\cdots -14 \beta _5 \beta _4+4 \beta _5^3-4 \beta _5^2\right) x_4^2 \nonumber\\
& \quad +\left(4 \beta _5^3 \beta _4^3-4 \beta _5 \beta _4^3-4 \beta _5^2 \beta _4^2+4 \beta _5 \beta _4^2-4 \beta _5^3 \beta _4+4 \beta _5^2 \beta _4\right) x_4 \nonumber\\
& \quad -\left(\beta _4-1\right) \beta _4^2 \left(\beta _5-1\right) \beta _5^2=0
\end{align}
By solving \eqref{k1234b0a4e12}, we obtain \(\alpha_{4,\min}(\beta_5, \beta_4) = x_4\). For instance, when \(\beta_5 = 0.98521, \beta_4=0.70944 \), using Mathematica's \texttt{NSolve} function and incorporating the constraints from Corollary \ref{corollary_Conservative_Bounds}, namely \( ((\beta _4^4+\beta _5^4-1)/2)^{1/4} <x_4< (\beta _4^4+\beta _5^4-1 )^{1/4} \) , we obtain:
\begin{align}
\label{k1234b0a4e13}
 & \alpha_{4,\min}(0.98521, 0.70944)= x_4= 0.57438547001705262893\cdots.
\end{align}

\indent As described in previous examples for $\beta_{n+1,min}$ or $\beta_{n,min}(\beta_{n+1})$, we can also directly employ numerical computation methods to obtain  \(\alpha_{4,\min}(\beta_5, \beta_4) = x_4\). For instance, when \(\beta_5 = 0.98521, \beta_4=0.70944 \), by using Mathematica's \texttt{NSolve} function to solve \eqref{k1234b0a4e1}, we obtain:
\begin{equation}
\begin{cases}
& y_3= 0.32739732664248750257\cdots,\\
& x_1= 0.03572427752148547551\cdots,\\
& x_2= 0.16494943572938427176\cdots,\\
& x_4= 0.57438547001705262893\cdots.
\end{cases}
\end{equation}
Then we arrive at the same result in \eqref{k1234b0a4e13}.

\begin{example}\textbf{Upper Bound of \( \boldsymbol{\alpha_{n}} \) for \( \boldsymbol{(k=1,2,4)} \)} \\ [2mm]
For the given normalized GPTE system:
\begin{align}
&\quad \left[ \alpha_{1}, \alpha_{2}, \alpha_{3}, \alpha_{4} \right]^{k} = \left[ \beta_{1}, \beta_{2}, \beta_{3}, \beta_{4} \right]^{k}, \quad (k = 1,2,4).
\end{align}
\noindent
Based on \eqref{conjecture_alpha_n_max} of Conjecture \textup{\ref{conjecture_alpha_n}}, when \( \beta_{4},\beta_{3} \) take specific values, \( \alpha_{3} \) achieves its maximum value \( \alpha_{3, \max}(\beta_{4},\beta_{3} ) \) under the following conditions:
\begin{equation}
\begin{cases}
& \beta_{4,min} \leq \beta_{4} < \alpha_{4}=1,\\
& \beta_{3,min} \leq \beta_{3} \leq \beta_{4},\\
& \beta_2 =  \beta_1, \\
& \alpha_3 \geq \alpha_2, \quad \alpha_1 = 0.
\end{cases}
\end{equation}
This leads to the following system of three equations in three variables:
\begin{align}
\left[ 0, x_{2}, x_{3}, 1 \right]^{k} = \left[ y_{2}, y_{2}, \beta_{3}, \beta_{4} \right]^{k}, \quad (k = 1,2,4).
\end{align}
Specifically,
\begin{align}
\label{k124a3s1}
x_2^k+ x_3^k+1=2 y_2^k+ \beta_3^k+ \beta_4^k,\quad (k = 1,2,4).
\end{align}
\end{example}
To solve the system of equations \eqref{k124a3s1} using the Second Generalization of the Girard-Newton Identities, the main derivation process is as follows: Let
\begin{align}
\label{k124a3s2}
& u_k=1+x_3^k-\beta _3^k-\beta _4^k,\quad (k=1,2,4)
\end{align}
we have
\begin{align}
\label{k124a3s4}
& u_k-2 y_2^k=-x_2^k,\quad (k=1,2,4)
\end{align}
Let
\begin{align}
\label{k124a3s5}
P_k=-x_2^k,\quad (k = 1,2,3,4).
\end{align}
and define
\begin{equation}
\begin{aligned}
\label{k124a3s6}
& S_1=(P_1)/1,\\
& S_2=(P_2 + S_1 P_1)/2,\\
& S_3=(P_3 + S_1 P_2 + S_2 P_1)/3,\\
& S_4=(P_4 + S_1 P_3 + S_2 P_2+ S_3 P_1)/4.\\
\end{aligned}
\end{equation}
Then, according to Identity \ref{identity_GNI2}, it follows that:
\begin{align}
\label{k124a3s7}
 S_2=0 \\
\label{k124a3s8}
 S_3=0 \\
\label{k124a3s9}
 S_4=0 
\end{align}
Substituting \eqref{k124a3s6} into \eqref{k124a3s8}, we obtain:
\begin{align}
\label{k124a3s9b}
& P_1^3+3 P_2 P_1+2 P_3=0
\end{align}
From \eqref{k124a3s9b}, we have:
\begin{align}
\label{k124a3s10}
& P_3=(-P_1^3-3 P_2 P_1)/2
\end{align}
Based on \eqref{k124a3s4} and \eqref{k124a3s10}, we replace \eqref{k124a3s5} with the following expression:
\begin{equation}
\label{k124a3s11}
\begin{cases}
P_k=2 u_k-2 y_3^k,\quad (k = 1,2,4),\\
P_3=(-P_1^3-3 P_2 P_1)/2.
\end{cases}
\end{equation}
Substituting \eqref{k124a3s11} into \eqref{k124a3s6}, \eqref{k124a3s7} and \eqref{k124a3s9}, then we obtain:
\begin{align}
\label{k124a3s12}
& -4 u_1 y_2+u_1^2+u_2+2 y_2^2=0
\end{align}
and
\begin{align}
\label{k124a3s13}
& 8 u_1^3 y_2-20 u_1^2 y_2^2+16 u_1 y_2^3+8 u_2 u_1 y_2-12 u_2 y_2^2-u_1^4-2 u_2 u_1^2 \nonumber \\
& \quad +u_2^2+2 u_4=0
\end{align}
Form \eqref{k124a3s12} and \eqref{k124a3s13}, we have
\begin{align}
\label{k124a3s14}
& y_2=\frac{7 u_1^4+2 u_2 u_1^2-7 u_2^2-2 u_4}{24 u_1 \left(u_1^2-u_2\right)}
\end{align}
Substituting \eqref{k124a3s14} into \eqref{k124a3s7} or \eqref{k124a3s9}, we have
\begin{align}
\label{k124a3s15}
& u_1^8-20 u_2 u_1^6+50 u_2^2 u_1^4+68 u_4 u_1^4-76 u_2^3 u_1^2-104 u_2 u_4 u_1^2+49 u_2^4\nonumber \\
& \quad +4 u_4^2+28 u_2^2 u_4=0
\end{align}
Substituting \eqref{k124a3s2} into \eqref{k124a3s15}, we obtain:
\begin{align}
\label{k124a3s16}
& 4 \left(\beta _3+\beta _4-1\right){}^2 x_3^6 \nonumber \\
& \quad -4 \left(\beta _3+\beta _4-1\right) \left(2 \beta _3^2+\beta _4 \beta _3-\beta _3+2 \beta _4^2-\beta _4-1\right) x_3^5 \nonumber \\
& \quad + \left(\beta _3^2-7 \beta _4 \beta _3+7 \beta _3+\beta _4^2+7 \beta _4-8\right) \left(\beta _3^2+\beta _4 \beta _3-\beta _3+\beta _4^2-\beta _4\right) x_3^4\nonumber \\
& \quad +\cdots + \left(-6 \beta _4 \beta _3^6+6 \beta _3^6+\cdots+12 \beta _4^2-8 \beta _4\right) x_3 \nonumber \\
& \quad +(9 \beta _4^2 \beta _3^6-6 \beta _4 \beta _3^6+\cdots-8 \beta _4^3+4 \beta _4^2)=0
\end{align}
By solving \eqref{k124a3s16}, we obtain \(\alpha_{3,\max}(\beta_4, \beta_3) = x_3\). For instance, when \(\beta_5 = 0.98765, \beta_4=0.87654 \), using Mathematica's \texttt{NSolve} function and incorporating the constraints from Corollary \ref{corollary_Conservative_Bounds}, namely \( ((\beta _3^4+\beta _4^4-1)/2)^{1/4} <x_3< (\beta _3^4+\beta _4^4-1 )^{1/4} \) , we obtain:
\begin{align}
 & \alpha_{3,\max}(0.98765, 0.87654)=x_3= 0.85788487898429062056\cdots.
\end{align}

\begin{example}\textbf{Lower Bound of \( \boldsymbol{\alpha_{n}} \) for \( \boldsymbol{(k=1,2,4)} \)} \\ [2mm]
For the given normalized GPTE system:
\begin{align}
&\quad \left[ \alpha_{1}, \alpha_{2}, \alpha_{3}, \alpha_{4} \right]^{k} = \left[ \beta_{1}, \beta_{2}, \beta_{3}, \beta_{4} \right]^{k}, \quad (k = 1,2,4).
\end{align}
\noindent
Based on \eqref{conjecture_alpha_n_min} of Conjecture \textup{\ref{conjecture_alpha_n}}, when \( \beta_{4},\beta_{3}, \) take specific values, \( \alpha_{3} \) achieves its minimum value \( \alpha_{3, \min}(\beta_{4},\beta_{3} ) \) under the following conditions:
\begin{equation}
\begin{cases}
& \beta_{4,min} \leq \beta_{4} < \alpha_{4}=1,\\
& \beta_{3,min} \leq \beta_{3} \leq \beta_{4},\\
& \beta_2 =  \beta_1, \\
& \alpha_3 = \alpha_2, \quad \alpha_1 \geq 0.
\end{cases}
\end{equation}
This leads to the following system of three equations in three variables:
\begin{align}
\label{k124a3e1}
x_1^k+ 2x_3^k+1=2 y_2^k+ \beta_3^k+ \beta_4^k,\quad (k = 1,2,4).
\end{align}
\end{example}
To solve the system of equations \eqref{k124a3e1} using the Second Generalization of the Girard-Newton Identities, the main derivation process is as follows: Let
\begin{align}
\label{k124a3e2}
& u_k=1+2 x_3^k-\beta _3^k-\beta _4^k,\quad (k=1,2,4)
\end{align}
Then, according to Identity \ref{identity_GNI2}, we have
\begin{align}
\label{k124a3e15}
& u_1^8-20 u_2 u_1^6+50 u_2^2 u_1^4+68 u_4 u_1^4-76 u_2^3 u_1^2-104 u_2 u_4 u_1^2+49 u_2^4\nonumber \\
& \quad +4 u_4^2+28 u_2^2 u_4=0
\end{align}
Substituting \eqref{k124a3e2} into \eqref{k124a3e15}, we obtain:
\begin{align}
\label{k124a3e16}
& 24 x_3^6 -60 \left(\beta _3+\beta _4-1\right) x_3^5 \nonumber \\
& \quad +\left(24 \beta _3^2+47 \beta _4 \beta _3-71 \beta _3+24 \beta _4^2-71 \beta _4+23\right) x_3^4 \nonumber \\
& \quad +\cdots + \left(12 \beta _3^5+60 \beta _4 \beta _3^4+108 \beta _4^2 \beta _3^3+\cdots-20 \beta _4^3-8 \beta _4^2+16 \beta _4\right) x_3 \nonumber \\
& \quad +(-9 \beta _4 \beta _3^5-3 \beta _3^5-18 \beta _4^2 \beta _3^4+\cdots +3 \beta _4^4+4 \beta _4^3-4 \beta _4^2)=0
\end{align}
By solving \eqref{k124a3e16}, we obtain \(\alpha_{3,\min}(\beta_4, \beta_3) = x_3\). For instance, when \(\beta_5 = 0.98765, \beta_4=0.87654 \), using Mathematica's \texttt{NSolve} function and incorporating the constraints from Corollary \ref{corollary_Conservative_Bounds}, namely \( ((\beta _3^4+\beta _4^4-1)/2)^{1/4} <x_3< (\beta _3^4+\beta _4^4-1 )^{1/4} \) , we obtain:
\begin{align}
 & \alpha_{3,\min}(0.98765, 0.87654)=x_3= 0.80376098021626785106\cdots.
\end{align}

The preceding examples in this subsection have focused on the upper and lower bounds of \(\alpha_n\). We now present conjectures related to the upper and lower bounds of \(\alpha_{n-1}\). To illustrate these conjectures, we consider two examples with \( (k = k_1, k_2, \dots, k_6)\), one of which has \(\beta_1\) specified as 0.
\begin{example}
\textbf{Upper and Lower Bound for \( \boldsymbol{(k=k_1,k_2,\cdots,k_6), \beta_1=0} \)} \\ [2mm]
Consider the normalized GPTE system defined by
\begin{align}
\left[ \alpha_{1}, \alpha_{2}, \alpha_{3}, \alpha_{4}, \alpha_{5}, \alpha_{6}, 1 \right]^{k} = \left[ 0, \beta_{2}, \beta_{3}, \beta_{4}, \beta_{5}, \beta_{6}, \beta_{7} \right]^{k}, \nonumber & \\
(k = k_1,k_2,k_3,k_4,k_5,k_6) &
\end{align}
where all k>0.
\begin{enumerate}

\item \textbf{Upper bound of $\boldsymbol{\beta_{n+1}}$} \\[1mm]
By Conjecture \ref{conjecture_beta_n+1}, the maximum value of \( \beta_{7} \) is given by
\begin{align}
\beta_{7, \max} = 1 - \delta,
\end{align}
where \(\delta\) is a positive infinitesimal real number.

\item \textbf{Lower bound of $\boldsymbol{\beta_{n+1}}$} \\[1mm]
By Conjecture \ref{conjecture_beta_n+1}, \( \beta_{7} \) attains its minimum value \( \beta_{7, \min} \) under the following constraints:
\begin{equation}
\begin{cases}
& \beta_7 =  \beta_6, \:\quad\beta_5 =  \beta_4, \:\quad \beta_3 =  \beta_2,  \\
& \alpha_6 = \alpha_5, \quad \alpha_4 = \alpha_3, \quad \alpha_2 = \alpha_1.
\end{cases}
\end{equation}
That is, the minimum value \( \beta_{7, \min}=y_7 \) is determined by solving:
\begin{align}
2 x_2^k+2 x_4^k+ 2 x_6^k+1=2 y_3^k+2 y_5^k+ 2 y_7^k,\nonumber & \\
(k = k_1,k_2,k_3,k_4,k_5,k_6) &
\end{align}
For $(k=1,2,3,4,5,6)$, we have
\begin{align}
\beta_{7, \min} = y_7 = 0.95048443395120956311\cdots = \sin^2\left(\frac{6 \pi}{14}\right)
\end{align}
For $(k=1,3,5,7,9,11)$, we have
\begin{align}
\beta_{7, \min} = y_7 = 0.97094181742605202715\cdots = \cos\left(\frac{\pi}{13}\right)
\end{align}

\item \textbf{Upper bound of $\boldsymbol{\beta_n}$} \\[1mm]
By Conjecture \ref{conjecture_beta_n}, the maximum value of \( \beta_{6} \) is given by:
\begin{align}
 \beta_{6, \max}(\beta_7) = \beta_7.
\end{align}

\item \textbf{Lower bound of $\boldsymbol{\beta_n}$} \\[1mm]
By Conjecture \ref{conjecture_beta_n}, when $ \beta_7 $ takes a specific value, $ \beta_6 $ attains its minimum value $ \beta_{6, \min}(\beta_7 ) $ under the following constraints:
\begin{equation}
\begin{cases}
& \beta_{7,\min} \leq \beta_{7} < \alpha_{7}=1,\\
& \beta_{6} \leq \beta_{7}, \:\quad \beta_5 =  \beta_4, \:\quad \beta_3 =  \beta_2,  \\
& \alpha_6 = \alpha_5, \quad \alpha_4 = \alpha_3, \quad \alpha_2 = \alpha_1.
\end{cases}
\end{equation}
That is, the minimum value $ \beta_{6, \min}(\beta_7)=y_6 $ is determined by solving:
\begin{align}
2 x_2^k+2 x_4^k+ 2 x_6^k+1=2 y_3^k+2 y_5^k+ y_6^k+ \beta_7^k,\nonumber & \\
(k = k_1,k_2,k_3,k_4,k_5,k_6) &
\end{align}
For $(k=1,2,3,4,5,6)$, when $ \beta_7 =0.98765$, we have
\begin{align}
\beta_{6, \min}(0.98765) = y_6 = 0.77308260094468773371\cdots
\end{align}
For $(k=1,3,5,7,9,11)$, when $ \beta_7 =0.98765$, we have
\begin{align}
\beta_{6, \min}(0.98765) = y_6  = 0.94191283060752019969\cdots
\end{align}

\item \textbf{Upper bound of $\boldsymbol{\alpha_n}$} \\[1mm]
By \eqref{conjecture_alpha_n_max} of Conjecture \ref{conjecture_alpha_n}, when \( \beta_{7},\beta_{6}, \) take specific values, \( \alpha_{6} \) attains its maximum value \( \alpha_{6, \max}(\beta_{7},\beta_{6} ) \) under the following constraints:
\begin{equation}
\begin{cases}
& \beta_{7,\min} \leq \beta_{7} < \alpha_{7}=1,\\
& \beta_{6,\min} \leq \beta_{6} \leq \beta_{7},\\
& \beta_5 =  \beta_4, \:\quad \beta_3 =  \beta_2, \:\quad \beta_1 =0 \\
& \alpha_6 \geq \alpha_5, \quad \alpha_4 = \alpha_3, \quad \alpha_2 = \alpha_1.
\end{cases}
\end{equation}
That is, the maximum value \( \alpha_{6, \max}(\beta_{7},\beta_{6})=x_6 \) is determined by solving:
\begin{align}
2 x_2^k+2 x_4^k+ x_5^k+ x_6^k+1=2 y_3^k+2 y_5^k+ \beta_6^k+ \beta_7^k,\nonumber & \\
(k = k_1,k_2,k_3,k_4,k_5,k_6) &
\end{align}
For $(k=1,2,3,4,5,6)$, when $ \beta_7 =0.98765,\beta_6 =0.87654 $, we have
\begin{align}
\alpha_{6, \max}(0.98765,0.87654) = x_6 = 0.82171624363559806540\cdots
\end{align}
For $(k=1,3,5,7,9,11)$, when $ \beta_7 =0.98765,\beta_6 =0.96789 $, we have
\begin{align}
\alpha_{6, \max}(0.98765,0.96789) = x_6 = 0.94606721599636792428\cdots
\end{align}

\item \textbf{Lower bound of $\boldsymbol{\alpha_n}$} \\[1mm]
By \eqref{conjecture_alpha_n_min0} of Conjecture \ref{conjecture_alpha_n}, when $ \beta_{7},\beta_{6} $ take specific values, $ \alpha_{6} $ attains its minimum value $ \alpha_{6, \min}(\beta_{7},\beta_{6} )$ under the following constraints:
\begin{equation}
\begin{cases}
& \beta_{7,\min} \leq \beta_{7} < \alpha_{7}=1,\\
& \beta_{6,\min} \leq \beta_{6} \leq \beta_{7},\\
& \beta_5 =  \beta_4, \:\quad \beta_3 =  \beta_2, \:\quad \beta_1 ,  \\
& \alpha_6 = \alpha_5, \quad \alpha_4 = \alpha_3, \quad \alpha_2 \geq \alpha_1.
\end{cases}
\end{equation}
That is, the minimum value $ \alpha_{6, \min}(\beta_{7},\beta_{6})=x_6 $ is determined by solving:
\begin{align}
x_1^k+ x_2^k+2 x_4^k+ x_5^k+\alpha_6^k+1=2 y_3^k+2 y_5^k+ \beta_6^k+ \beta_7^k,\nonumber & \\
(k = k_1,k_2,k_3,k_4,k_5,k_6) &
\end{align}
For $(k=1,2,3,4,5,6)$, when $ \beta_7 =0.98765,\beta_6 =0.87654 $, we have
\begin{align}
\alpha_{6, \min}(0.98765,0.87654) = x_6 = 0.78695456387212794859\cdots
\end{align}
For $(k=1,3,5,7,9,11)$, when $ \beta_7 =0.98765,\beta_6 =0.96789 $, we have
\begin{align}
\alpha_{6, \min}(0.98765,0.96789) = x_6  = 0.91746546965028476202\cdots
\end{align}

\item \textbf{Upper bound of $\boldsymbol{\alpha_{n-1}}$} \\[1mm]
By our conjecture, the maximum value $ \alpha_{5, \max}(\beta_7,\beta_6,\alpha_6) $ is given by:
\begin{align}
 \alpha_{5, \max}(\beta_7,\beta_6,\alpha_6) =\alpha_6.
\end{align}

\item \textbf{Lower bound of $\boldsymbol{\alpha_{n-1}}$} \\[1mm]
By our conjecture, when $ \beta_{7},\beta_{6},\alpha_6 $ take specific values, $ \alpha_{5} $ attains its minimum value $ \alpha_{5, \min}(\beta_{7},\beta_{6},\alpha_6 )$ under the following constraints:
\begin{equation}
\begin{cases}
& \beta_{7,\min} \leq \beta_{7} < \alpha_{7}=1,\\
& \beta_{6,\min} \leq \beta_{6} \leq \beta_{7},\\
& \alpha_{6,\min} \leq \alpha_6 \leq \alpha_{6,\max},\\
& \beta_5 =  \beta_4, \:\quad \beta_3 =  \beta_2,  \\
& \alpha_6 \geq \alpha_5, \quad \alpha_4 = \alpha_3, \quad \alpha_2 \geq \alpha_1.
\end{cases}
\end{equation}
That is, the minimum value $ \alpha_{5, \min}(\beta_{7},\beta_{6},\alpha_6)=x_5 $ is determined by solving:
\begin{align}
x_1^k+ x_2^k+2 x_4^k+ x_5^k+\alpha_6^k+1=2 y_3^k+2 y_5^k+ \beta_6^k+ \beta_7^k,\nonumber & \\
(k = k_1,k_2,k_3,k_4,k_5,k_6) &
\end{align}
For $(k=1,2,3,4,5,6)$, when $ \beta_7 =0.98765,\beta_6 =0.87654,\alpha_6=0.78910 $, we have
\begin{align}
\alpha_{5, \min}(0.98765,0.87654,0.78910) = x_5 = 0.78474788361411426477\cdots
\end{align}
For $(k=1,3,5,7,9,11)$, when $ \beta_7 =0.98765, \beta_6 =0.96789, \alpha_6=0.92345 $, we have
\begin{align}
\alpha_{5, \min}(0.98765,0.96789,0.92345 ) = x_5  = 0.91029290127374434872\cdots
\end{align}

\end{enumerate}
\end{example}

\begin{example}
\textbf{Upper and Lower Bound for \( \boldsymbol{(k=k_1,k_2,\cdots,k_6)} \) }\\ [2mm]
Consider the normalized GPTE system defined by
\begin{align}
\left[ \alpha_{1}, \alpha_{2}, \alpha_{3}, \alpha_{4}, \alpha_{5}, \alpha_{6}, 1 \right]^{k} = \left[ \beta_{1},, \beta_{2}, \beta_{3}, \beta_{4}, \beta_{5}, \beta_{6}, \beta_{7} \right]^{k}, \nonumber & \\
(k = k_1,k_2,k_3,k_4,k_5,k_6) &
\end{align}
where all k>0.
\begin{enumerate}

\item \textbf{Upper bound of $\boldsymbol{\beta_{n+1}}$} \\[1mm]
By Conjecture \ref{conjecture_beta_n+1}, the maximum value of \( \beta_{7} \) is given by
\begin{align}
\beta_{7, \max} = 1 - \delta,
\end{align}
where \(\delta\) is a positive infinitesimal real number.

\item \textbf{Lower bound of $\boldsymbol{\beta_{n+1}}$} \\[1mm]
By Conjecture \ref{conjecture_beta_n+1}, \( \beta_{7} \) attains its minimum value \( \beta_{7, \min} \) under the following constraints:
\begin{equation}
\begin{cases}
& \beta_7 =  \beta_6, \:\quad\beta_5 =  \beta_4, \:\quad \beta_3 =  \beta_2, \:\quad \beta_1=0 \\
& \alpha_6 = \alpha_5, \quad \alpha_4 = \alpha_3, \quad \alpha_2 = \alpha_1.
\end{cases}
\end{equation}
That is, the minimum value \( \beta_{7, \min}=y_7 \) is determined by solving:
\begin{align}
2 x_2^k+2 x_4^k+ 2 x_6^k+1=2 y_3^k+2 y_5^k+ 2 y_7^k,\nonumber & \\
(k = k_1,k_2,k_3,k_4,k_5,k_6) &
\end{align}
For $(k=1,2,3,4,5,7)$, we have
\begin{align}
\beta_{7, \min} = y_7 = 0.95280865800018751744\cdots
\end{align}
For $(k=1,2,3,5,7,9)$, we have
\begin{align}
\beta_{7, \min} = y_7 = 0.96158977040352043607\cdots
\end{align}

\item \textbf{Upper bound of $\boldsymbol{\beta_n}$} \\[1mm]
By Conjecture \ref{conjecture_beta_n}, the maximum value of \( \beta_{6} \) is given by:
\begin{align}
 \beta_{6, \max}(\beta_7) = \beta_7.
\end{align}

\item \textbf{Lower bound of $\boldsymbol{\beta_n}$} \\[1mm]
By Conjecture \ref{conjecture_beta_n}, when $ \beta_7 $ takes a specific value, $ \beta_6 $ attains its minimum value $ \beta_{6, \max}(\beta_7 ) $ under the following constraints:
\begin{equation}
\begin{cases}
& \beta_{7,\min} \leq \beta_{7} < \alpha_{7}=1,\\
& \beta_{6} \leq \beta_{7}, \:\quad \beta_5 = \beta_4, \:\quad \beta_3 = \beta_2, \:\quad \beta_1=0 \\
& \alpha_6 = \alpha_5, \quad \alpha_4 = \alpha_3, \quad \alpha_2 = \alpha_1.
\end{cases}
\end{equation}
That is, the minimum value $ \beta_{6, \min}(\beta_7)=y_6 $ is determined by solving:
\begin{align}
2 x_2^k+2 x_4^k+ 2 x_6^k+1=2 y_3^k+2 y_5^k+ y_6^k+ \beta_7^k,\nonumber & \\
(k = k_1,k_2,k_3,k_4,k_5,k_6) &
\end{align}
For $(k=1,2,3,4,5,7)$, when $ \beta_7 =0.98765$, we have
\begin{align}
\beta_{6, \min}(0.98765) = y_6 = 0.87572746741822757040\cdots
\end{align}
For $(k=1,2,3,5,7,9)$, when $ \beta_7 =0.98765$, we have
\begin{align}
\beta_{6, \min}(0.98765) = y_6 = 0.90928197497038148564\cdots
\end{align}

\item \textbf{Upper bound of $\boldsymbol{\alpha_n}$} \\[1mm]
By \eqref{conjecture_alpha_n_max} of Conjecture \ref{conjecture_alpha_n}, when \( \beta_{7},\beta_{6}, \) take specific values, \( \alpha_{6} \) attains its maximum value \( \alpha_{6, \max}(\beta_{7},\beta_{6} ) \) under the following constraints:
\begin{equation}
\begin{cases}
& \beta_{7,\min} \leq \beta_{7} < \alpha_{7}=1,\\
& \beta_{6,\min} \leq \beta_{6} \leq \beta_{7},\\
& \beta_5 =  \beta_4, \:\quad \beta_3 =  \beta_2,  \:\quad \beta_1=0, \\
& \alpha_6 \geq \alpha_5, \quad \alpha_4 = \alpha_3, \quad \alpha_2 = \alpha_1.
\end{cases}
\end{equation}
That is, the maximum value \( \alpha_{6, \max}(\beta_{7},\beta_{6})=x_6 \) is determined by solving:
\begin{align}
2 x_2^k+2 x_4^k+ x_5^k+ x_6^k+1=2 y_3^k+2 y_5^k+ \beta_6^k+ \beta_7^k,\nonumber & \\
(k = k_1,k_2,k_3,k_4,k_5,k_6) &
\end{align}
For $(k=1,2,3,4,5,7)$, when $ \beta_7 =0.98765,\beta_6 =0.96789 $, we have
\begin{align}
\alpha_{6, \max}(0.98765,0.96789) = x_6 = 0.94990594792997678570\cdots
\end{align}
For $(k=1,2,3,5,7,9)$, when $ \beta_7 =0.98765,\beta_6 =0.96789 $, we have
\begin{align}
\alpha_{6, \max}(0.98765,0.96789) = x_6 = 0.94990594792997678570\cdots
\end{align}

\item \textbf{Lower bound of $\boldsymbol{\alpha_n}$} \\[1mm]
By \eqref{conjecture_alpha_n_min} of  Conjecture \ref{conjecture_alpha_n}, when $ \beta_{7},\beta_{6} $ take specific values, $ \alpha_{6} $ attains its minimum value $ \alpha_{6, \min}(\beta_{7},\beta_{6} )$ under the following constraints:
\begin{equation}
\begin{cases}
& \beta_{7,\min} \leq \beta_{7} < \alpha_{7}=1,\\
& \beta_{6,\min} \leq \beta_{6} \leq \beta_{7},\\
& \beta_5 =  \beta_4, \:\quad \beta_3 =  \beta_2, \:\quad \beta_1 \geq 0,  \\
& \alpha_6 = \alpha_5, \quad \alpha_4 = \alpha_3, \quad \alpha_2 = \alpha_1.
\end{cases}
\end{equation}
That is, the minimum value $ \alpha_{6, \min}(\beta_{7},\beta_{6})=x_6 $ is determined by solving:
\begin{align}
2 x_2^k+2 x_4^k+ 2 x_6^k+1=y_1^k+2 y_3^k+2 y_5^k+ \beta_6^k+ \beta_7^k,\nonumber & \\
(k = k_1,k_2,k_3,k_4,k_5,k_6) &
\end{align}
For $(k=1,2,3,4,5,7)$, when $ \beta_7 =0.98765,\beta_6 =0.96789 $, we have
\begin{align}
\alpha_{6, \min}(0.98765,0.96789) = x_6 = 0.91892128928650043275\cdots
\end{align}
For $(k=1,3,5,7,9,11)$, when $ \beta_7 =0.98765,\beta_6 =0.96789 $, we have
\begin{align}
\alpha_{6, \min}(0.98765,0.96789) = x_6 = 0.91823834150506495776\cdots
\end{align}

\item \textbf{Upper bound of $\boldsymbol{\alpha_{n-1}}$} \\[1mm]
By our conjecture, the maximum value $ \alpha_{5, \max}(\beta_7,\beta_6,\alpha_6) $ is given by:
\begin{align}
 \alpha_{5, \max}(\beta_7,\beta_6,\alpha_6) =\alpha_6.
\end{align}

\item \textbf{Lower bound of $\boldsymbol{\alpha_{n-1}}$} \\[1mm]
By our conjecture, when $ \beta_{7},\beta_{6},\alpha_6 $ take specific values, $ \alpha_{5} $ attains its minimum value $ \alpha_{5, \min}(\beta_{7},\beta_{6},\alpha_6 )$ under the following constraints:
\begin{equation}
\begin{cases}
& \beta_{7,\min} \leq \beta_{7} < \alpha_{7}=1,\\
& \beta_{6,\min} \leq \beta_{6} \leq \beta_{7},\\
& \alpha_{6,\min} \leq \alpha_6 \leq \alpha_{6,\max},\\
& \beta_5 =  \beta_4, \:\quad \beta_3 =  \beta_2,  \:\quad \beta_3 \geq 0,  \\
& \alpha_6 \geq \alpha_5, \quad \alpha_4 = \alpha_3, \quad \alpha_2 = \alpha_1.
\end{cases}
\end{equation}
That is, the minimum value $ \alpha_{5, \min}(\beta_{7},\beta_{6},\alpha_6)=x_5 $ is determined by solving:
\begin{align}
2 x_2^k+2 x_4^k+x_5^k+\alpha_6^k+1=y_1^k+2 y_3^k+2 y_5^k+ \beta_6^k+ \beta_7^k,\nonumber & \\
(k = k_1,k_2,k_3,k_4,k_5,k_6) &
\end{align}
For $(k=1,2,3,4,5,7)$, when $ \beta_7 =0.98765,\beta_6 =0.96789,\alpha_6=0.92345 $, we have
\begin{align}
\alpha_{5, \min}(0.98765,0.96789,0.92345) = x_5 = 0.91373751403572097319\cdots
\end{align}
For $(k=1,2,3,5,7,9)$, when $ \beta_7 =0.98765, \beta_6 =0.96789, \alpha_6=0.92345 $, we have
\begin{align}
\alpha_{5, \min}(0.98765,0.96789,0.92345) = x_5 = 0.91214832092690240425\cdots
\end{align}
\end{enumerate}
\end{example}

In this section, we have presented conjectures for the exact upper and lower bounds of \(\beta_{n+1}\), \(\beta_{n}\), \(\alpha_{n}\), and \(\alpha_{n-1}\). Our further research indicates that similar methods can be applied to derive the exact bounds for other variables such as \(\beta_{n-1}\), \(\beta_{n-2}\), and \(\alpha_{n-2}\). Given the complexity of these calculations, we suggest precomputing these bounds and storing them in lookup tables for practical use in computer searches for ideal non-negative integer solutions to the GPTE problem. Specifically:

\begin{enumerate}
    \item For \(\beta_{n+1}\), the upper bound is straightforward, while the lower bound is a fixed value that can be precomputed.
    \item For \(\beta_{n}\), the upper bound is also straightforward and does not require computation, while the lower bound can be stored in a one-dimensional lookup table.
    \item For \(\alpha_{n}\), both the upper and lower bounds can be stored in two-dimensional lookup tables, despite the potentially large size of these tables.
    \item For \(\alpha_{n-1}\) and other variables, the bounds would require lookup tables of three or more dimensions. In practice, such high-dimensional lookup tables become complex and inevitably reduce search efficiency, making them impractical for use in computer searches.
\end{enumerate}

In the following section, we introduce two types of conservative bounds for the normalized GPTE system. Although these conservative bounds are less restrictive than exact bounds, thereby expanding the search space,  their simpler computation and lack of need for table lookups make them an effective and practical strategy for the computer search process.

\subsection{Conservative Bounds of  the Normalized GPTE System}
For the two types of conservative bounds presented in this section, we are still unable to provide strict proofs, nor have we found any counterexamples. Therefore, we propose these two types of conservative bounds in the form of conjectures.
\begin{conjecture}\textbf{\textup{(Seesaw Conjecture)}}\\[1mm]
\label{conjecture_seesaw}
Let \(\{\alpha_{1}, \alpha_{2}, \dots, \alpha_{n+1}\}\) and \(\{\beta_{1}, \beta_{2}, \dots, \beta_{n+1}\}\) be two sets of non-negative real numbers bounded above by 1, ordered as follows:
\begin{align}
& 0 \leq \alpha_1 \leq \alpha_2 \leq \cdots \leq \alpha_{n+1}=1, \\
& 0 \leq \beta_1 \leq \beta_2 \leq \cdots \leq \beta_{n+1}<1.
\end{align}
For given distinct integers \(\{k_1, k_2, \dots, k_n\}\), if the following system of equations holds:
\begin{align}
\left[ \alpha_{1}, \alpha_{2}, \dots, \alpha_{n+1} \right]^{k} = \left[ \beta_{1}, \beta_{2}, \dots, \beta_{n+1} \right]^{k}, \quad (k = k_1, k_2, \dots, k_n).
\end{align}
then the elements of the two sets interlace in the following manner:
\begin{align}
\operatorname{sgn}(k) \cdot \left[ \alpha_{n+1} \right]^k & > \operatorname{sgn}(k) \cdot \left[\beta_{n+1} \right]^k, \nonumber\\
\operatorname{sgn}(k) \cdot \left[ \alpha_{n}, \alpha_{n+1} \right]^k & < \operatorname{sgn}(k) \cdot \left[ \beta_{n}, \beta_{n+1} \right]^k, \nonumber\\
\operatorname{sgn}(k) \cdot \left[ \alpha_{n-1},\alpha_{n}, \alpha_{n+1} \right]^k & > \operatorname{sgn}(k) \cdot \left[ \beta_{n-1},\beta_{n}, \beta_{n+1} \right]^k, \nonumber\\
\operatorname{sgn}(k) \cdot \left[\alpha_{n-2}, \alpha_{n-1},\alpha_{n}, \alpha_{n+1} \right]^k & < \operatorname{sgn}(k) \cdot \left[\beta_{n-2}, \beta_{n-1},\beta_{n}, \beta_{n+1} \right]^k, \nonumber\\
& \vdots 
\end{align}
where  \(\operatorname{sgn}(k)\) is the sign function defined by
\begin{align}
\operatorname{sgn}(k) = 
\begin{cases}
1, & \text{if } k \geq 0, \\
-1, & \text{if } k < 0.
\end{cases}
\end{align}
\end{conjecture}

\begin{example}
Let \(\{\alpha_{1}, \alpha_{2}, \dots,  \alpha_{8}\}\) and \(\{\beta_{1}, \beta_{2}, \dots, \beta_{8}\}\) be two sets of non-negative real numbers bounded above by 1, ordered as follows:
\begin{align}
& 0 \leq \alpha_1 \leq \alpha_2 \leq \dots \leq \alpha_8=1, \\
& 0 \leq \beta_1 \leq \beta_2 \leq \dots \leq \beta_8<1.
\end{align}
These sets satisfy:
\begin{align}
& \left[ \alpha_{1}, \alpha_{2}, \alpha_{3}, \alpha_{4}, \alpha_{5}, \alpha_{6}, \alpha_{7}, \alpha_{8} \right]^{k} = \left[ \beta_{1}, \beta_{2}, \beta_{3}, \beta_{4}, \beta_{5}, \beta_{6},\beta_{7}, \beta_{8}  \right]^{k}, \nonumber \\ 
& \qquad\qquad\qquad\qquad\qquad\qquad\qquad\quad (k = -3,-2,-1,0,1,2,3).
\end{align}
Then, according to Conjecture \ref{conjecture_seesaw}, the following inequalities hold:
\begin{align}
& \begin{cases}
\,\alpha_8^k \,<\,  \beta_8^k, \quad (k = -3, -2, -1), \\
\, \alpha_8 \,>\,   \beta_8, \\
\, \alpha_8^k \,>\,  \beta_8^k,  \quad (k = 1, 2, 3).
\end{cases}\\[1mm]
& \begin{cases}
\,\alpha_7^k + \alpha_8^k \,>\, \beta_7^k  + \beta_8^k, \quad (k = -3, -2, -1), \\
\;\; \alpha_7 \cdot \alpha_8 \;<\;  \beta_7 \cdot \beta_8, \\
\,\alpha_7^k + \alpha_8^k \,<\, \beta_7^k  + \beta_8^k,  \quad (k = 1, 2, 3).
\end{cases}\\[1mm]
& \begin{cases}
\,\alpha_6^k +\alpha_7^k + \alpha_8^k \,<\, \beta_6^k+ \beta_7^k  + \beta_8^k, \quad (k = -3, -2, -1), \\
\quad\: \alpha_6 \cdot \alpha_7 \cdot \alpha_8 \;>\; \beta_6 \cdot \beta_7 \cdot \beta_8, \\
\,\alpha_6^k +\alpha_7^k + \alpha_8^k \,>\, \beta_6^k+ \beta_7^k  + \beta_8^k,  \quad (k = 1, 2, 3).
\end{cases}\\[1mm]
& \begin{cases}
\,\alpha_5^k+\alpha_6^k +\alpha_7^k + \alpha_8^k \,>\, \beta_5^k +\beta_6^k+ \beta_7^k  + \beta_8^k, \quad (k = -3, -2, -1), \\
\quad\;\:\, \alpha_5 \cdot \alpha_6 \cdot \alpha_7 \cdot \alpha_8 \;<\; \beta_5 \cdot \beta_6 \cdot \beta_7 \cdot \beta_8, \\
\,\alpha_5^k+\alpha_6^k +\alpha_7^k + \alpha_8^k \,<\, \beta_5^k +\beta_6^k+ \beta_7^k  + \beta_8^k,  \quad (k = 1, 2, 3).
\end{cases}\\[1mm]
& \begin{cases}
\,\alpha_4^k+\alpha_5^k + \cdots + \alpha_8^k \,<\, \beta_4^k +\beta_5^k+ \cdots  + \beta_8^k, \quad (k = -3, -2, -1), \\
\qquad\;\:\: \alpha_4 \cdot \alpha_5 \dots \alpha_8 \;>\; \beta_4 \cdot \beta_5 \dots \beta_8, \\
\,\alpha_4^k+\alpha_5^k + \cdots + \alpha_8^k \,>\, \beta_4^k +\beta_5^k+ \cdots  + \beta_8^k,  \quad (k = 1, 2, 3).
\end{cases}\\[1mm]
& \begin{cases}
\,\alpha_3^k+\alpha_4^k + \cdots + \alpha_8^k \,>\, \beta_3^k +\beta_4^k+ \cdots  + \beta_8^k, \quad (k = -3, -2, -1), \\
\qquad\;\:\: \alpha_3 \cdot \alpha_4 \dots \alpha_8 \;<\; \beta_3 \cdot \beta_4 \dots \beta_8, \\
\,\alpha_3^k+\alpha_4^k + \cdots + \alpha_8^k \,<\, \beta_3^k +\beta_4^k+ \cdots  + \beta_8^k,  \quad (k = 1, 2, 3).
\end{cases}\\[1mm]
& \begin{cases}
\,\alpha_2^k+\alpha_3^k + \cdots + \alpha_8^k \,<\, \beta_2^k +\beta_3^k+ \cdots  + \beta_8^k, \quad (k = -3, -2, -1), \\
\qquad\;\:\: \alpha_2 \cdot \alpha_3 \dots \alpha_8 \;>\; \beta_2 \cdot \beta_3 \dots \beta_8, \\
\,\alpha_2^k+\alpha_3^k + \cdots + \alpha_8^k \,>\, \beta_2^k +\beta_3^k+ \cdots  + \beta_8^k,  \quad (k = 1, 2, 3).
\end{cases}
\end{align}
\end{example}
Assuming Conjecture \ref{conjecture_seesaw} holds, we obtain the following corollary:
\begin{corollary}\textbf{\textup{(Conservative Bounds)}}\\[1mm]
\label{corollary_Conservative_Bounds}
Let \(\{\alpha_{1}, \alpha_{2}, \dots, \alpha_{n+1}\}\) and \(\{\beta_{1}, \beta_{2}, \dots, \beta_{n+1}\}\) be two sets of non-negative real numbers bounded above by 1, ordered as follows:
\begin{align}
& 0 \leq \alpha_1 \leq \alpha_2 \leq \cdots \leq \alpha_{n+1}=1, \\
& 0 \leq \beta_1 \leq \beta_2 \leq \cdots \leq \beta_{n+1}<1.
\end{align}
For given distinct integers \(\{k_1, k_2, \dots, k_n\}\), if the following normalized GPTE system holds:
\begin{align}
\left[ \alpha_{1}, \alpha_{2}, \dots, \alpha_{n+1} \right]^{k} = \left[ \beta_{1}, \beta_{2}, \dots, \beta_{n+1} \right]^{k}, \quad (k = k_1, k_2, \dots, k_n).
\end{align}
then for $k>0$ and $k \in (k_1, k_2, \dots, k_n)$, we have
\begin{flalign}
\frac{1}{2} \alpha_{n+1}^k & < \beta_{n+1}^k < \alpha_{n+1}^k, \nonumber &\\
\alpha_{n+1}^k- \beta_{n+1}^k  & < \beta_{n\phantom{+0}}^k \leq \beta_{n+1}^k , \nonumber &\\
\frac{1}{2} \left( \beta_{n+1}^k+ \beta_{n}^k -\alpha_{n+1}^k \right) & < \alpha_{n\phantom{+0}}^k < \beta_{n+1}^k+ \beta_{n}^k -\alpha_{n+1}^k , \nonumber &\\
\beta_{n+1}^k+ \beta_{n}^k -\alpha_{n+1}^k -\alpha_{n}^k & < \alpha_{n-1}^k \leq \alpha_{n}^k , \nonumber &\\
\frac{1}{2} \left( \alpha_{n+1}^k+\alpha_{n}^k+\alpha_{n-1}^k- \beta_{n+1}^k- \beta_{n}^k  \right) & < \beta_{n-1}^k < \alpha_{n+1}^k+\alpha_{n}^k+\alpha_{n-1}^k- \beta_{n+1}^k- \beta_{n}^k , \nonumber &\\
\alpha_{n+1}^k+\alpha_{n}^k+\alpha_{n-1}^k- \beta_{n+1}^k- \beta_{n}^k- \beta_{n-1}^k  & < \beta_{n-2}^k  \leq \beta_{n-1}^k , \nonumber &\\
& \vdots 
\end{flalign}
\end{corollary}

\begin{example}
For the given normalized GPTE system:
\begin{align}
& \left[ \alpha_{1}, \alpha_{2}, \alpha_{3}, \alpha_{4}, \alpha_{5}, \alpha_{6}, \alpha_{7}, \alpha_{8} \right]^{k} = \left[ \beta_{1}, \beta_{2}, \beta_{3}, \beta_{4}, \beta_{5}, \beta_{6},\beta_{7}, \beta_{8}  \right]^{k}, \nonumber \\ 
& \qquad\qquad\qquad\qquad\qquad\qquad\qquad\quad (k = -3,-2,-1,0,1,2,3),
\end{align}
according to Corollary \ref{corollary_Conservative_Bounds}, the subsequent inequalities hold for $k=1,2,3 :$
\begin{flalign}
\frac{1}{2} \alpha_8^k & < \beta_8^k < \alpha_8^k, \nonumber &\\
\alpha_8^k- \beta_8^k  & < \beta_7^k \leq \beta_8^k , \nonumber &\\
\frac{1}{2} \left( \beta_8^k+ \beta_7^k -\alpha_8^k \right) & < \alpha_7^k < \beta_8^k+ \beta_7^k -\alpha_8^k , \nonumber &\\
\beta_8^k+ \beta_7^k -\alpha_8^k -\alpha_7^k & < \alpha_6^k \leq \alpha_7^k , \nonumber &\\
\frac{1}{2} \left( \alpha_8^k+\alpha_7^k+\alpha_6^k- \beta_8^k- \beta_7^k  \right) & < \beta_6^k < \alpha_8^k+\alpha_7^k+\alpha_6^k- \beta_8^k- \beta_7^k , \nonumber &\\
\alpha_8^k+\alpha_7^k+\alpha_6^k- \beta_8^k- \beta_7^k- \beta_6^k  & < \beta_5^k  \leq \beta_6^k , \nonumber &\\
\frac{1}{2} \left( \sum_{i=5}^8 \beta_i^k -\sum_{i=6}^8 \alpha_i^k\right) & < \alpha_5^k <\sum_{i=5}^8 \beta_i^k -\sum_{i=6}^8 \alpha_i^k , \nonumber &\\
\sum_{i=5}^8 \beta_i^k -\sum_{i=5}^8 \alpha_i^k & < \alpha_4^k \leq \alpha_5^k , \nonumber &\\
\frac{1}{2} \left( \sum_{i=4}^8 \alpha_i^k -\sum_{i=5}^8 \beta_i^k\right) & < \beta_4^k <\sum_{i=4}^8 \alpha_i^k -\sum_{i=5}^8 \beta_i^k , \nonumber &\\
\sum_{i=4}^8 \alpha_i^k -\sum_{i=4}^8 \beta_i^k & < \beta_3^k \leq \beta_4^k , \nonumber &\\
\frac{1}{2} \left( \sum_{i=3}^8 \beta_i^k -\sum_{i=4}^8 \alpha_i^k\right) & < \alpha_3^k <\sum_{i=3}^8 \beta_i^k -\sum_{i=4}^8 \alpha_i^k , \nonumber &\\
\sum_{i=3}^8 \beta_i^k -\sum_{i=3}^8 \alpha_i^k & < \alpha_2^k \leq \alpha_3^k , \nonumber &\\
\frac{1}{2} \left( \sum_{i=2}^8 \alpha_i^k -\sum_{i=3}^8 \beta_i^k\right) & < \beta_2^k <\sum_{i=2}^8 \alpha_i^k -\sum_{i=3}^8 \beta_i^k . &
\end{flalign}
\end{example}

\begin{example}
For the given normalized GPTE system:
\begin{align}
& \left[ \alpha_{1}, \alpha_{2}, \alpha_{3}, \alpha_{4}, \alpha_{5}, \alpha_{6} \right]^{k} = \left[ \beta_{1}, \beta_{2}, \beta_{3}, \beta_{4}, \beta_{5}, \beta_{6}  \right]^{k}, \nonumber \\ 
& \qquad\qquad\qquad\qquad\qquad\qquad\qquad\quad (k = k_1, k_2, k_3, k_4, k_5),
\end{align}
according to Corollary \ref{corollary_Conservative_Bounds}, the subsequent inequalities hold for $k>0$ and $k \in \{k_1, k_2, \dots, k_5\}$:
\begin{flalign}
\quad \frac{1}{2} \alpha_6^k & < \beta_6^k < \alpha_6^k, \qquad\qquad\qquad\qquad\qquad\;
\alpha_6^k- \beta_6^k   < \beta_5^k \leq \beta_6^k , \nonumber &\\
\quad \frac{1}{2} \left( \sum_{i=5}^6 \beta_i^k -\sum_{i=6}^6 \alpha_i^k\right) & < \alpha_5^k <\sum_{i=5}^6 \beta_i^k -\sum_{i=6}^6 \alpha_i^k, \quad\;\; 
\sum_{i=5}^6 \beta_i^k -\sum_{i=5}^6 \alpha_i^k < \alpha_4^k \leq \alpha_5^k , \nonumber &\\
\quad \frac{1}{2} \left( \sum_{i=4}^6 \alpha_i^k -\sum_{i=5}^6 \beta_i^k\right) & < \beta_4^k <\sum_{i=4}^6 \alpha_i^k -\sum_{i=5}^6 \beta_i^k , \quad\;\; 
\sum_{i=4}^6 \alpha_i^k -\sum_{i=4}^6 \beta_i^k < \beta_3^k \leq \beta_4^k , \nonumber &\\
\quad \frac{1}{2} \left( \sum_{i=3}^6 \beta_i^k -\sum_{i=4}^6 \alpha_i^k\right) & < \alpha_3^k <\sum_{i=3}^6 \beta_i^k -\sum_{i=4}^6 \alpha_i^k, \quad\;\; 
\sum_{i=3}^6 \beta_i^k -\sum_{i=3}^6 \alpha_i^k < \alpha_2^k \leq \alpha_3^k , \nonumber &\\
\quad \frac{1}{2} \left( \sum_{i=2}^6 \alpha_i^k -\sum_{i=3}^6 \beta_i^k\right) & < \beta_2^k <\sum_{i=2}^6 \alpha_i^k -\sum_{i=3}^6 \beta_i^k . &
\end{flalign}
\end{example}
\vspace{1mm}
\begin{conjecture}\textbf{\textup{(Threshold Conjecture)}}\\[1mm]
\label{conjecture_threshold}
Let \(\{a_{1}, a_{2}, \dots, a_{n+1}\}\) and \(\{b_{1}, b_{2}, \dots, b_{n+1}\}\) be two sets of non-negative integers, ordered as follows:
\begin{align}
& 0 \leq a_1 \leq a_2 \leq \cdots \leq a_{n+1}, \\
& 0 \leq b_1 \leq b_2 \leq \cdots \leq b_{n+1}.
\end{align}
Without loss of generality, assume that \(a_{n+1} > b_{n+1}\). For given distinct integers \(\{k_1, k_2, \dots, k_n\}\) where \(k_n>0 \), we consider the GPTE system defined by
\begin{align}
\left[ a_{1}, a_{2}, \dots, a_{n+1} \right]^{k} = \left[ b_1, b_2, \dots, b_{n+1} \right]^{k}, \quad (k = k_1, k_2, \dots, k_n),
\end{align}
By solving the following normalized GPTE system of \( n \) equations in \( n \) variables:
\begin{align}
& \left[ \alpha_2, \alpha_2, \alpha_4, \alpha_4, \cdots, \alpha_{n},\alpha_{n},1 \right]^{k} = \left[ 0, \beta_3, \beta_3, \beta_5,\beta_5, \cdots, \beta_{n+1},\beta_{n+1} \right]^{k}, \nonumber \\
& \qquad\qquad  (k = k_1, k_2, \dots, k_n), \quad \text{if } n \text{ is even} \label{threshold_even} \\
& \left[ 0,\alpha_3, \alpha_3, \alpha_5, \alpha_5, \cdots, \alpha_{n},\alpha_{n},1 \right]^{k} = \left[ \beta_2, \beta_2, \beta_4,\beta_4, \cdots, \beta_{n+1},\beta_{n+1} \right]^{k}, \nonumber \\
& \qquad\qquad  (k = k_1, k_2, \dots, k_n), \quad \text{if } n \text{ is odd} \label{threshold_odd} 
\end{align}
we obtain the numerical solutions for \(\{\alpha_n, \alpha_{n-2}, \cdots\}\) and \(\{\beta_{n+1}, \beta_{n-1}, \cdots\}\).\\[1mm]
Then, for those \( k_i > 0 \) among \( k_1, k_2, \dots, k_{n-1} \), we conjecture the following:
\begin{flalign}
& 1 < \frac{\left(b_{n+1}^{k_i}+b_n^{k_i}-a_{n+1}^{k_i}\right){}^{k_n}}{\left(b_{n+1}^{k_n}+b_n^{k_n}-a_{n+1}^{k_n}\right){}^{k_i}}\leq \frac{\left(2 \beta _{n+1}^{k_i}-1\right){}^{k_n}}{\left(2 \beta _{n+1}^{k_n}-1\right){}^{k_i}} & \\[1mm]
& 0 < \frac{\left(b_{n+1}^{k_i}+b_n^{k_i}-a_{n+1}^{k_i}-a_n^{k_i}\right){}^{k_n}}{\left(b_{n+1}^{k_n}+b_n^{k_n}-a_{n+1}^{k_n}-a_n^{k_n}\right){}^{k_i}}\leq \frac{\left(2 \beta _{n+1}^{k_i}-1-\alpha _n^{k_i}\right){}^{k_n}}{\left(2 \beta _{n+1}^{k_n}-1-\alpha _n^{k_n}\right){}^{k_i}} & \\[1mm]
& 1 < \frac{\left(a_{n+1}^{k_i}+a_n^{k_i}+a_{n-1}^{k_i}-b_{n+1}^{k_i}-b_n^{k_i}\right){}^{k_n}}{\left(a_{n+1}^{k_n}+a_n^{k_n}+a_{n-1}^{k_n}-b_{n+1}^{k_n}-b_n^{k_n}\right){}^{k_i}}\leq \frac{\left(2 \alpha _n^{k_i}+1-2\beta _{n+1}^{k_i}\right){}^{k_n}}{\left(2 \alpha_n^{k_n}+1-2\beta_{n+1}^{k_n}\right){}^{k_i}} & \\[1mm]
& 0 < \frac{\left(a_{n+1}^{k_i}+a_n^{k_i}+a_{n-1}^{k_i}-b_{n+1}^{k_i}-b_n^{k_i}-b_{n-1}^{k_i}\right)^{k_n}}{\left(a_{n+1}^{k_n}+a_n^{k_n}+a_{n-1}^{k_n}-b_{n+1}^{k_n}-b_n^{k_n}-b_{n-1}^{k_n}\right)^{k_i}}\leq \frac{\left(1+2 \alpha_n^{k_i}-2\beta _{n+1}^{k_i}-\beta _{n-1}^{k_i}\right)^{k_n}}{\left(1+2 \alpha_n^{k_n}-2\beta_{n+1}^{k_n}-\beta_{n-1}^{k_n}\right)^{k_i}} & \\[1mm]
& \qquad\qquad \vdots \nonumber &
\end{flalign}
\end{conjecture}
The equations \eqref{threshold_even} and \eqref{threshold_odd} correspond to the conditions \eqref{conjecture_beta_n+1_condition} in Conjecture \ref{conjecture_beta_n+1} under which \( \beta_{n+1} \) attains its minimum value \( \beta_{n+1, \min} \). In this conjecture, the values of \(\{\alpha_{n}, \alpha_{n-2}, \cdots\}\) and \(\{\beta_{n+1}, \beta_{n-1}, \cdots\}\) constitute the parameters of the threshold.

\begin{example}\textbf{\textup{Threshold Inequality for $\boldsymbol{(k=1,2,3,4,5,7)}$}}\\[1mm]
Let \(\{a_1, a_2, \dots, a_7\}\) and \(\{b_1, b_2, \dots, b_7\}\) be two sets of non-negative integers, ordered as follows:
\begin{align}
& 0 \leq a_1 \leq a_2 \leq \cdots \leq a_7, \\
& 0 \leq b_1 \leq b_2 \leq \cdots \leq b_7.
\end{align}
and satisfying:
\begin{align}
\left[ a_1, a_2, a_3,a_4,a_5,a_6, a_7 \right]^{k} = \left[ b_1, b_2,b_3,b_4,b_5,b_6,b_7 \right]^{k}, & \nonumber\\
\quad (k = 1,2,3,4,5,7). &
\end{align}
Without loss of generality, assume that \(a_7 > b_7\). By solving the following normalized GPTE system of six equations in six variables:
\begin{align}
2 \alpha_2^k+2 \alpha_4^k+ 2 \alpha_6^k+1=2 \beta_3^k+2 \beta_5^k+ 2 \beta_7^k,\quad
(k = 1,2,3,4,5,7) 
\end{align}
we have
\begin{equation}
\begin{cases}
\alpha_2=0.05164990002386205851\cdots,\\
\alpha_4=0.39985292856797344247\cdots,\\
\alpha_6=0.81918780398271942126\cdots,\\
\beta_3 =0.19527192784719509411\cdots,\\
\beta_5 =0.62261004672717231069\cdots,\\
\beta_7 =0.95280865800018751744\cdots.
\end{cases}
\end{equation}
Then, according to Conjecture \ref{conjecture_threshold}, the following inequalities hold:
\begin{equation}
\begin{aligned}
& 1 < (b_7^1+b_6^1-a_7^1)^7/v_3^1 \leq 1.1731936434\cdots,\\
& 1 < (b_7^2+b_6^2-a_7^2)^7/v_3^2 \leq 1.3249039329\cdots,\\
& 1 < (b_7^3+b_6^3-a_7^3)^7/v_3^3 \leq 1.4306923566\cdots,\\
& 1 < (b_7^4+b_6^4-a_7^4)^7/v_3^4 \leq 1.4647138777\cdots,\\
& 1 < (b_7^5+b_6^5-a_7^5)^7/v_3^5 \leq 1.4059455441\cdots,\\
& \qquad\qquad \textup{where}\;\; v_3  =b_7^7+b_6^7-a_7^7.
\end{aligned}
\end{equation}
\vspace{-2mm}
\begin{equation}
\begin{aligned}
& 0 < (b_7^1+b_6^1-a_7^1-a_6^1)^7/v_4^1 \leq 0.0000002020\cdots,\\
& 0 < (b_7^2+b_6^2-a_7^2-a_6^2)^7/v_4^2 \leq 0.0000416312\cdots,\\
& 0 < (b_7^3+b_6^3-a_7^3-a_6^3)^7/v_4^3 \leq 0.0010920418\cdots,\\
& 0 < (b_7^4+b_6^4-a_7^4-a_6^4)^7/v_4^4 \leq 0.0118236289\cdots,\\
& 0 < (b_7^5+b_6^5-a_7^5-a_6^5)^7/v_4^5 \leq 0.0753342167\cdots,\\
& \qquad\qquad \textup{where}\;\; v_4 =b_7^7+b_6^7-a_7^7-a_6^7.
\end{aligned}
\end{equation}
\vspace{-2mm}
\begin{equation}
\begin{aligned}
& 1 < (a_7^1+a_6^1+a_5^1-b_7^1-b_6^1)^7/v_5^1 \leq 2.2914907855\cdots,\\
& 1 < (a_7^2+a_6^2+a_5^2-b_7^2-b_6^2)^7/v_5^2 \leq 2.8238721881\cdots,\\
& 1 < (a_7^3+a_6^3+a_5^3-b_7^3-b_6^3)^7/v_5^3 \leq 2.8248563348\cdots,\\
& 1 < (a_7^4+a_6^4+a_5^4-b_7^4-b_6^4)^7/v_5^4 \leq 2.3344388138\cdots,\\
& 1 < (a_7^5+a_6^5+a_5^5-b_7^5-b_6^5)^7/v_5^5 \leq 1.6370819280\cdots,\\
& \qquad\qquad \textup{where}\;\; v_5 =a_7^7+a_6^7+a_5^7-b_7^7-b_6^7.
\end{aligned}
\end{equation}
\vspace{-2mm}
\begin{equation}
\begin{aligned}
& 0 < (a_7^1+a_6^1+a_5^1-b_7^1-b_6^1-b_5^1)^7/v_6^1 \leq 0.3003634352\cdots,\\
& 0 < (a_7^2+a_6^2+a_5^2-b_7^2-b_6^2-b_5^2)^7/v_6^2 \leq 0.0968622648\cdots,\\
& 0 < (a_7^3+a_6^3+a_5^3-b_7^3-b_6^3-b_5^3)^7/v_6^3 \leq 0.0157285406\cdots,\\
& 0 < (a_7^4+a_6^4+a_5^4-b_7^4-b_6^4-b_5^4)^7/v_6^4 \leq 0.0009104995\cdots,\\
& 0 < (a_7^5+a_6^5+a_5^5-b_7^5-b_6^5-b_5^5)^7/v_6^5 \leq 0.0000059574\cdots,\\
& \qquad\qquad \textup{where}\;\; v_6 =a_7^7+a_6^7+a_5^7-b_7^7-b_6^7-b_5^7.
\end{aligned}
\end{equation}
\vspace{-2mm}
\begin{equation}
\begin{aligned}
& 1 < (b_7^1+b_6^1+b_5^1+b_4^1-a_7^1-a_6^1-a_5^1)^7/v_7^1 \leq 3.3959670306\cdots,\\
& 1 < (b_7^2+b_6^2+b_5^2+b_4^3-a_7^2-a_6^2-a_5^2)^7/v_7^2 \leq 5.4638952823\cdots,\\
& 1 < (b_7^3+b_6^3+b_5^3+b_4^3-a_7^3-a_6^3-a_5^3)^7/v_7^3 \leq 6.9780041836\cdots,\\
& 1 < (b_7^4+b_6^4+b_5^4+b_4^4-a_7^4-a_6^4-a_5^4)^7/v_7^4 \leq 5.6043427385\cdots,\\
& 1 < (b_7^5+b_6^5+b_5^5+b_4^5-a_7^5-a_6^5-a_5^5)^7/v_7^5 \leq 2.8588079379\cdots,\\
& \qquad\qquad \textup{where}\;\; v_7 =b_7^7+b_6^7+b_5^7+b_4^7-a_7^7-a_6^7-a_5^7.
\end{aligned}
\end{equation}
\vspace{-2mm}
\begin{equation}
\begin{aligned}
& 0 < (b_7^1+b_6^1+b_5^1+b_4^1-a_7^1-a_6^1-a_5^1-a_4^1)^7/v_8^1 \leq 0.7168485697\cdots,\\
& 0 < (b_7^2+b_6^2+b_5^2+b_4^3-a_7^2-a_6^2-a_5^2-a_4^2)^7/v_8^2 \leq 0.4549241768\cdots,\\
& 0 < (b_7^3+b_6^3+b_5^3+b_4^3-a_7^3-a_6^3-a_5^3-a_4^3)^7/v_8^3 \leq 0.1691253278\cdots,\\
& 0 < (b_7^4+b_6^4+b_5^4+b_4^4-a_7^4-a_6^4-a_5^4-a_4^4)^7/v_8^4 \leq 0.0169508467\cdots,\\
& 0 < (b_7^5+b_6^5+b_5^5+b_4^5-a_7^5-a_6^5-a_5^5-a_4^5)^7/v_8^5 \leq 0.0001423962\cdots,\\
& \qquad\qquad \textup{where}\;\; v_8 =b_7^7+b_6^7+b_5^7+b_4^7-a_7^7-a_6^7-a_5^7-a_4^7.
\end{aligned}
\end{equation}
\end{example}

\begin{example}\textbf{\textup{Threshold Inequality for $\boldsymbol{(k=-1,0,1,2,3,5)}$}}\\[1mm]
Let \(\{a_1, a_2, \dots, a_7\}\) and \(\{b_1, b_2, \dots, b_7\}\) be two sets of non-negative integers, ordered as follows:
\begin{align}
& 0 \leq a_1 \leq a_2 \leq \cdots \leq a_7, \\
& 0 \leq b_1 \leq b_2 \leq \cdots \leq b_7.
\end{align}
and satisfying:
\begin{align}
\left[ a_1, a_2, a_3,a_4,a_5,a_6, a_7 \right]^{k} = \left[ b_1, b_2,b_3,b_4,b_5,b_6,b_7 \right]^{k}, & \nonumber\\
\quad (k = -1,0,1,2,3,5). &
\end{align}
Without loss of generality, assume that \(a_7 > b_7\). By solving the following normalized GPTE system with $ \varepsilon_1=0.0000000000000001$:
\begin{equation}
\begin{cases}
2 \alpha_2^k+2 \alpha_4^k+ 2 \alpha_6^k+1=\varepsilon_1^k+2 \beta_3^k+2 \beta_5^k+ 2 \beta_7^k,\quad
(k = -1,1,2,3,5) \\
\alpha_2^2 \alpha_4^2 \alpha_6^2=\varepsilon_1 \beta_3^2 \beta_5^2 \beta_7^2
\end{cases}
\end{equation}
we have
\begin{equation}
\begin{cases}
\alpha_2=0.00000000000000019999\cdots,\\
\alpha_4=0.10366735748124304538\cdots,\\
\alpha_6=0.67710905862459701232\cdots,\\
\beta_3 =0.00000000418342096619\cdots,\\
\beta_5 =0.36741431562534021961\cdots,\\
\beta_7 =0.91336209629707902190\cdots.
\end{cases}
\end{equation}
Then, according to Conjecture \ref{conjecture_threshold}, the following inequalities hold:
\begin{equation}
\begin{aligned}
& 1 < (b_7^1+b_6^1-a_7^1)^5/v_3^1 \leq 1.4235348177\cdots,\\
& 1 < (b_7^2+b_6^2-a_7^2)^5/v_3^2 \leq 1.8134639074\cdots,\\
& 1 < (b_7^3+b_6^3-a_7^3)^5/v_3^3 \leq 1.9768451783\cdots,\\
& \qquad\qquad \textup{where}\;\; v_3  =b_7^5+b_6^5-a_7^5.
\end{aligned}
\end{equation}
\vspace{-2mm}
\begin{equation}
\begin{aligned}
& 0 < (b_7^1+b_6^1-a_7^1-a_6^1)^5/v_4^1 \leq 0.0005813210\cdots,\\
& 0 < (b_7^2+b_6^2-a_7^2-a_6^2)^5/v_4^2 \leq 0.0245475142\cdots,\\
& 0 < (b_7^3+b_6^3-a_7^3-a_6^3)^5/v_4^3 \leq 0.2066805236\cdots,\\
& \qquad\qquad \textup{where}\;\; v_4 =b_7^5+b_6^5-a_7^5-a_6^5.
\end{aligned}
\end{equation}
\vspace{-2mm}
\begin{equation}
\begin{aligned}
& 1 < (a_7^1+a_6^1+a_5^1-b_7^1-b_6^1)^5/v_5^1 \leq 3.0553028897\cdots,\\
& 1 < (a_7^2+a_6^2+a_5^2-b_7^2-b_6^2)^5/v_5^2 \leq 5.3027632896\cdots,\\
& 1 < (a_7^3+a_6^3+a_5^3-b_7^3-b_6^3)^5/v_5^3 \leq 3.5897080250\cdots,\\
& \qquad\qquad \textup{where}\;\; v_5 =a_7^5+a_6^5+a_5^5-b_7^5-b_6^5.
\end{aligned}
\end{equation}
\vspace{-2mm}
\begin{equation}
\begin{aligned}
& 0 < (a_7^1+a_6^1+a_5^1-b_7^1-b_6^1-b_5^1)^5/v_6^1 \leq 0.0157563547\cdots,\\
& 0 < (a_7^2+a_6^2+a_5^2-b_7^2-b_6^2-b_5^2)^5/v_6^2 \leq 0.4231748923\cdots,\\
& 0 < (a_7^3+a_6^3+a_5^3-b_7^3-b_6^3-b_5^3)^5/v_6^3 \leq 0.8032589457\cdots,\\
& \qquad\qquad \textup{where}\;\; v_6 =a_7^5+a_6^5+a_5^5-b_7^5-b_6^5-b_5^5.
\end{aligned}
\end{equation}
\vspace{-2mm}
\begin{equation}
\begin{aligned}
& 1 < (b_7^1+b_6^1+b_5^1+b_4^1-a_7^1-a_6^1-a_5^1)^5/v_7^1 \leq 3.9999999999\cdots,\\
& 1 < (b_7^2+b_6^2+b_5^2+b_4^3-a_7^2-a_6^2-a_5^2)^5/v_7^2 \leq 7.9999999999\cdots,\\
& 1 < (b_7^3+b_6^3+b_5^3+b_4^3-a_7^3-a_6^3-a_5^3)^5/v_7^3 \leq 15.999996771\cdots,\\
& \qquad\qquad \textup{where}\;\; v_7 =b_7^5+b_6^5+b_5^5+b_4^5-a_7^5-a_6^5-a_5^5.
\end{aligned}
\end{equation}
\vspace{-2mm}
\begin{equation}
\begin{aligned}
& 0 < (b_7^1+b_6^1+b_5^1+b_4^1-a_7^1-a_6^1-a_5^1-a_4^1)^5/v_8^1 \leq 0.9999995964\cdots,\\
& 0 < (b_7^2+b_6^2+b_5^2+b_4^3-a_7^2-a_6^2-a_5^2-a_4^2)^5/v_8^2 \leq 0.9999999999\cdots,\\
& 0 < (b_7^3+b_6^3+b_5^3+b_4^3-a_7^3-a_6^3-a_5^3-a_4^3)^5/v_8^3 \leq 0.9999999999\cdots,\\
& \qquad\qquad \textup{where}\;\; v_8 =b_7^5+b_6^5+b_5^5+b_4^5-a_7^5-a_6^5-a_5^5-a_4^5.
\end{aligned}
\end{equation}
\end{example}

In this chapter, we present six conjectures that provide upper and lower bounds for the variables in the GPTE system from different perspectives. Conjectures \ref{conjecture_Interlacing} and \ref{conjecture_seesaw} were discovered by us in 1995, while Conjectures \ref{conjecture_beta_n+1}, \ref{conjecture_beta_n}, \ref{conjecture_alpha_n}, and \ref{conjecture_threshold} were discovered by us in 2021. Among the numerical solutions listed in the Appendix of this paper, those obtained prior to 2021 were mostly found using only Conjectures \ref{conjecture_Interlacing} and \ref{conjecture_seesaw}. In contrast, the numerical solutions obtained after 2021 have generally employed all six conjectures except for Conjecture \ref{conjecture_alpha_n}.
\clearpage

\section{A General Process of Computer Search for GPTE}
For the computer search of the PTE problem, general algorithms have been provided in the literature by P. Borwein et al. in 2003 \cite{Borwein2003}, T. Caley's Ph.D. thesis in 2012 \cite{Caley12}, and the Master's theses by Ming Qiu in 2016 \cite{Qiu2016}, Yingqiong Chen in 2018 \cite{YingqiongChen2018}, and Qian Zhang in 2023 \cite{Zhang2023}. Additionally, D. Coppersmith and M. Mossinghoff et al. published a paper in 2024 \cite{Coppersmith2024}.

In this chapter, for the first time in the literature, we present a general algorithm for the computer search of ideal non-negative integer solutions to the GPTE problem, which is also applicable to the PTE problem. Specifically, we seek non-negative integer solutions to the following system:
\begin{align}
& \left[ a_{1}, a_{2}, \dots, a_{n+1} \right]^{k} = \left[ b_{1}, b_{2}, \dots, b_{n+1} \right]^{k}, \quad (k = k_1, k_2, \dots, k_n).  \tag*{\eqref{GPTEk}}
\end{align}
As an example, we consider the case \( (k=1,2,3,4) \):
\begin{align}
\label{PTEk1234}
& [a_1,a_2,a_3,a_4,a_5]^k=[b_1,b_2,b_3,b_4,b_5]^k,\quad (k=1,2,3,4).
\end{align}
Given that \eqref{PTEk1234} holds, it follows that for any integer \( T \), the following also holds:
\begin{align}
& [a_1+T,a_2+T,a_3+T,a_4+T,a_5+T]^k \nonumber\\
& \qquad =[b_1+T,b_2+T,b_3+T,b_4+T,b_5+T]^k,\quad (k=1,2,3,4).
\end{align}
Therefore, without loss of generality, we can only consider the case where \( b_1 = 0 \):
\begin{align}
\label{k1234b0}
& [a_1,a_2,a_3,a_4,a_5]^k=[0,b_2,b_3,b_4,b_5]^k,\quad (k=1,2,3,4).
\end{align}
\subsection{Mathematical Foundations of the Search Algorithm}
This section provides the various algorithmic foundations related to the computer search of \eqref{k1234b0}.
\begin{algorithm}\textnormal{\textbf{Interlacing Relation}}\\[2mm]
According to Conjecture \ref{conjecture_Interlacing}, we have the following ordering:
\begin{align}
& 0<a_1 \leq a_2 <b_2 \leq b_3 < a_3 \leq a_4 <b_4 \leq b_5 < a_5
\end{align}
\end{algorithm}
\noindent Consequently, we can transform \eqref{k1234b0} into the following form:
\begin{align}
\label{k1234L3R2}
& b_3^k + b_2^k - a_2^k - a_1^k = a_5^k + a_4^k + a_3^k - b_5^k + b_4^k, \quad (k=1,2,3,4)
\end{align}

\begin{algorithm}\textbf{\textup{Integrality Test}}\\[1mm]
\label{alg:integrality-test}
Given five positive integers \( a_5, a_4, a_3, b_5, b_4 \), proceed as follows:\\[1mm]
\noindent\text{Step 1. Compute power sums:}
\begin{align}
\label{P1234L3R2}
P_k &= a_5^k + a_4^k + a_3^k - b_5^k - b_4^k, \quad (k=1,2,3,4).
\end{align}

\noindent\text{Step 2. Construct symmetric functions:}
\begin{equation}
\label{k1234S1234}
\begin{aligned}
S_1 &= P_1, \\
S_2 &= (P_2 + S_1 P_1)/2, \\
S_3 &= (P_3 + S_1 P_2 + S_2 P_1)/3, \\
S_4 &= (P_4 + S_1 P_3 + S_2 P_2 + S_3 P_1)/4.
\end{aligned}
\end{equation}

\noindent\text{Step 3. Evaluate determinants:}
\begin{equation}
\label{k1234W456}
\begin{aligned}
W_{4} &=
\begin{vmatrix}
 S_2 & S_1 \\
 S_3 & S_2 \\
\end{vmatrix},\quad
W_{5} =
\begin{vmatrix}
 S_3 & S_1 \\
 S_4 & S_2 \\
\end{vmatrix},\quad
W_{6} =
\begin{vmatrix}
 S_3 & S_2 \\
 S_4 & S_3 \\
\end{vmatrix}.
\end{aligned}
\end{equation}
\noindent\textbf{Criterion:}  
If \( \{a_5, a_4, a_3, b_5, b_4\} \) constitutes part of a non-negative integer solution to system \eqref{k1234b0}, then  the ratios 
\begin{equation}
\frac{W_5}{W_4} \quad \text{and} \quad \frac{W_6}{W_4}
\end{equation}
must be positive integers.
\end{algorithm}

\begin{proof}
To prove the validity of the Integrality Test presented above, we utilize the Second Generalization of the Girard-Newton Identities, referred to as Identity \ref{identity_GNI2}. We define:
\begin{align}
\label{P1234L2R2}
P_k = b_3^k + b_2^k - a_2^k - a_1^k, \quad (k=1,2,3,4).
\end{align}
Then, we construct the symmetric functions \(S_i\) and determinants \(W_i\) as follows:
\begin{equation}
\begin{aligned}
S_1 &= P_1, \\
S_2 &= (P_2 + S_1 P_1)/2, \\
S_3 &= (P_3 + S_1 P_2 + S_2 P_1)/3, \\
S_4 &= (P_4 + S_1 P_3 + S_2 P_2 + S_3 P_1)/4.
\end{aligned}
\tag*{\eqref{k1234S1234}}
\end{equation}
\vspace{-12pt}
\begin{equation}
\begin{aligned}
W_{4} &=
\begin{vmatrix}
 S_2 & S_1 \\
 S_3 & S_2 \\
\end{vmatrix},\quad
W_{5} =
\begin{vmatrix}
 S_3 & S_1 \\
 S_4 & S_2 \\
\end{vmatrix},\quad
W_{6} =
\begin{vmatrix}
 S_3 & S_2 \\
 S_4 & S_3 \\
\end{vmatrix}.
\end{aligned}
\tag*{\eqref{k1234W456}}
\end{equation}
From these definitions, we derive the following results:
\begin{align}
\label{k1234D1D2}
& D_1 = \frac{W_5}{W_4} = b_2 + b_3, \qquad D_2 = \frac{W_6}{W_4} = b_2 b_3.
\end{align}
Based on \eqref{k1234L3R2}, by substituting \eqref{P1234L2R2} with \eqref{P1234L3R2}, we obtain the criterion for the Integrality Test. This criterion ensures that the ratios ${W_5}/{W_4}$ and ${W_6}/{W_4}$ must be positive integers if \(\{a_5, a_4, a_3, b_5, b_4\}\) constitutes part of a non-negative integer solution to system \eqref{k1234b0}. 
\end{proof}
This Integrality Test significantly reduces the number of necessary loops in a full computer search, thereby improving search efficiency. In Example \ref{example_k1234567}, we provide a similar example for type \( (k=1,2,3,4,5,6,7) \).

After passing the Integrality Test, the original Diophantine system of four equations with nine variables \eqref{k1234b0} is reduced to a system of four equations with four variables \eqref{P1234L2R2}. We can then use the following algorithm to calculate the remaining four variables \(\{b_3, b_2, a_2, a_1\}\).

\begin{algorithm}\textnormal{\textbf{Calculating Remaining Variables}}\\[2mm]
From \eqref{P1234L2R2} and \eqref{k1234D1D2}, we derive:
\begin{align}
D_3 &= D_1^2 - 4 D_2 = (b_3 - b_2)^2, \\
C_3 &= 2 D_1 P_1 + D_1^2 - 4 D_2 - P_1^2 - 2 P_2 = (a_2 - a_1)^2.
\end{align}
Then we have:
\begin{align}
D_0 &= \sqrt{D_3} = b_3 - b_2, \\
C_0 &= \sqrt{C_3} = a_2 - a_1.
\end{align}
Finally, we obtain:
\begin{align}
& b_3=(D_1+D_0)/2,\\
& b_2=(D_1-D_0)/2,\\
& a_2=(D_1-P_1+C_0)/2,\\
& a_1=(D_1-P_1-C_0)/2.
\end{align}
\end{algorithm}
We then need to verify whether \(b_3\), \(b_2\), \(a_2\), and \(a_1\) are all positive integers. If they are, we obtain an ideal non-negative integer solution to \eqref{k1234b0}.

\begin{algorithm}
\textbf{\textup{Exact Bounds}}\\[2mm]
According to Conjecture \ref{conjecture_beta_n+1}, we have:
\begin{align}
\label{k1234b5L}
\beta_{5, \min} &=\sin ^2\left(\frac{4 \pi }{10}\right)= 0.90450849718747371205\cdots. 
\end{align}
According to Conjecture \ref{conjecture_beta_n}, we have:
\begin{align}
\label{k1234b4L}
& y_4^6 - 6 (\beta_5 - 1) y_4^5 + (15 \beta_5^2 + 14 \beta_5 - 29) y_4^4 \nonumber \\
& \quad - 4 (5 \beta_5^3 + 5 \beta_5^2 - \beta_5 - 9) y_4^3 + (15 \beta_5^4 - 20 \beta_5^3 + 50 \beta_5^2 - 36 \beta_5 - 9) y_4^2 \nonumber \\
& \quad - 2 (\beta_5 - 1)^4 (3 \beta_5 + 5) y_4 + (\beta_5 - 1)^4 (\beta_5^2 + 10 \beta_5 + 5) = 0
\end{align}
\end{algorithm}

By solving \eqref{k1234b4L}, with the additional constraints \((1 - \beta_5^4)^{1/4}< y_4 \leq \beta_5 \), we obtain the lookup table for \(\beta_{4, \min}(\beta_5)\) corresponding to the type \((k=1,2,3,4)\) as follows:
\begin{align*}
& \beta_{4,\min}\left(0.9046\right) = 0.9044169355\cdots, \\
& \beta_{4,\min}\left(0.9047\right) = 0.9043167364\cdots, \\
& \beta_{4,\min}\left(0.9048\right) = 0.9042163962\cdots, \\
& \beta_{4,\min}\left(0.9049\right) = 0.9041159147\cdots, \\
& \beta_{4,\min}\left(0.9050\right) = 0.9040152914\cdots, \\
& \beta_{4,\min}\left(0.9051\right) = 0.9039145263\cdots, \\
& \qquad \vdots \\
& \beta_{4,\min}\left(0.9994\right) = 0.3770975986\cdots, \\
& \beta_{4,\min}\left(0.9995\right) = 0.3627999836\cdots, \\
& \beta_{4,\min}\left(0.9996\right) = 0.3458993832\cdots. \\
& \beta_{4,\min}\left(0.9997\right) = 0.3250684022\cdots, \\
& \beta_{4,\min}\left(0.9998\right) = 0.2974966651\cdots, \\
& \beta_{4,\min}\left(0.9999\right) = 0.2549905796\cdots. 
\end{align*}
\noindent When conducting a computer search, the above lookup table can be organized into the following text file format:

\begin{lstlisting}[language=VB]
    9046    904416
    9047    904316
    9048    904216
    9049    904115
    9050    904015
    9051    903914
\end{lstlisting}
$  \qquad \qquad \vdots \qquad\qquad\quad \vdots$
\begin{lstlisting}[language=VB,, firstnumber=949]
    9994    377097
    9995    362799
    9996    345899
    9997    325068
    9998    297496
    9999    254990
\end{lstlisting}
For \(\alpha_4\), although we can also derive its exact bounds based on Conjecture \ref{conjecture_alpha_n}, and subsequently obtain the lookup tables for \(\alpha_{4, \max}(\beta_5, \beta_4)\) and \(\alpha_{4, \min}(\beta_5, \beta_4)\), such two-dimensional tables tend to become complex as the precision requirements increase, thereby potentially reducing search efficiency. Moreover, the exact bounds lookup table for \(\alpha_3\) even becomes three-dimensional. Therefore, for \(\alpha_4\) and \(\alpha_3\), we choose not to consider exact bounds but instead focus on the following Conservative Bounds and Threshold Inequalities as simplified search limits, which can significantly enhance search efficiency.

\begin{algorithm}\textbf{\textup{Conservative Bounds}}\\[2mm]
According to Conjecture \ref{conjecture_seesaw}, we have:
\begin{align}
\frac{1}{2}a_5^k & <b_5^k <a_5^k, \\
a_5^k-b_5^k  & <b_4^k \leq b_5^k , \\
\label{Seesaw_a4}
\frac{1}{2} \left(b_5^k+b_4^k -a_5^k \right) & <a_4^k <b_5^k+b_4^k -a_5^k , \\
\label{Seesaw_a3}
b_5^k+b_4^k -a_5^k -a_4^k & <a_3^k \leq a_4^k 
\end{align}
\end{algorithm}
The inequalities \eqref{Seesaw_a4} and \eqref{Seesaw_a3} can be used as the Conservative Bounds for \(a_4\) and \(a_3\), respectively.

\begin{algorithm}\textbf{\textup{Threshold Inequality}}\\[2mm]
According to Conjecture \ref{conjecture_threshold}, we first solve the following normalized GPTE system:
\begin{align}
2 \alpha_2^k + 2 \alpha_4^k + 1 = 2 \beta_3^k + 2 \beta_5^k, \quad (k = 1,2,3,4)
\end{align}
The numerical solutions are as follows:
\begin{equation}
\begin{cases}
\alpha_2=0.09549150281252628794\cdots,\\
\alpha_4=0.65450849718747371205\cdots,\\
\beta_3 =0.34549150281252628794\cdots,\\
\beta_5 =0.90450849718747371205\cdots.\\
\end{cases}
\end{equation}
Subsequently, we derive the following series of Threshold inequalities:
\begin{equation}
\begin{aligned}
& 0 < (b_5^1+b_4^1-a_5^1-a_4^1)^4/v_4^1 \leq 0.4250992787\cdots,\\
& 0 < (b_5^2+b_4^2-a_5^2-a_4^2)^4/v_4^2 \leq 0.0775630998\cdots,\\
& 0 < (b_5^3+b_4^3-a_5^3-a_4^3)^4/v_4^3 \leq 0.0036725726\cdots,\\
& \qquad\qquad \textup{where}\;\; v_4 =b_5^3+b_4^3-a_5^3-a_4^3.
\end{aligned}
\end{equation}
and
\begin{equation}
\begin{aligned}
& 1 < (a_5^1+a_4^1+a_3^1-b_5^1-b_4^1)^4/v_5^1 \leq 1.8688879619\cdots,\\
& 1 < (a_5^2+a_4^2+a_3^2-b_5^2-b_4^2)^4/v_5^2 \leq 2.9450477444\cdots,\\
& 1 < (a_5^3+a_4^3+a_3^3-b_5^3-b_4^3)^4/v_5^3 \leq 2.2061877674\cdots,\\
& \qquad\qquad \textup{where}\;\; v_5 =a_5^4+a_4^4+a_3^4-b_5^4-b_4^4.
\end{aligned}
\end{equation}
\end{algorithm}

\begin{algorithm}\textbf{\textup{Constant C}}\\[2mm]
According to Identity \ref{identityC_PTE}, we have:
\begin{equation}
\begin{aligned}
C &= \prod_{j=1}^5 (a_i - b_j) = -\prod_{j=1}^5 (b_i - a_j) \quad \text{for } i=1,...,5
\end{aligned}
\end{equation}
When \( b_1 = 0 \), the above fomula for Constant C simplifies to:
\begin{equation}
\begin{aligned}
C & = a_1 \left(b_2-a_1\right) \left(b_3-a_1\right) \left(b_4-a_1\right) \left(b_5-a_1\right)\\
& = a_2 \left(b_2-a_2\right) \left(b_3-a_2\right) \left(b_4-a_2\right) \left(b_5-a_2\right)\\
& = a_3 \left(a_3-b_2\right) \left(a_3-b_3\right) \left(b_4-a_3\right) \left(b_5-a_3\right)\\
& = a_4 \left(a_4-b_2\right) \left(a_4-b_3\right) \left(b_4-a_4\right) \left(b_5-a_4\right)\\
& = a_5 \left(a_5-b_2\right) \left(a_5-b_3\right) \left(a_5-b_4\right) \left(a_5-b_5\right)
\end{aligned}
\end{equation}
\end{algorithm}
When conducting a computer search, we precompute a table of prime factors to enhance the efficiency of the search process. For example, the following table lists the largest prime factor for each integer from 1 to 10000, consisting of 10000 entries.
\begin{lstlisting}[language=VB]
    1 
    2
    3 
    2 
    5 
    3 
    7
    2
    3
    5
\end{lstlisting}
$  \qquad \quad \vdots $
\begin{lstlisting}[language=VB, firstnumber=9991]
    103
    1249
    3331
    263
    1999   
    17
    769    
    4999  
    101
    5 
\end{lstlisting}

\subsection{Reference Code for the Search Program}
In this section, we present a reference code for the computer search of non-negative integer solutions for the case \( (k=1,2,3,4) \). The mathematical algorithm underlying this code is entirely based on the previous section. Here, we use Visual Basic 6.0 (VB6) as the programming language. One reason for choosing VB6 is its high code readability, which makes it easy to understand. Additionally, VB6 is relatively efficient for searching numerical solutions to GPTE problems. The majority of the numerical solutions presented in this paper, which were first found by us, were obtained using VB6. Only in a few cases where high computational speed is required, do we resort to using Visual C++ 6.0 (VC6), which can be 2 to 4 times faster than VB6.
\begin{lstlisting}[language=VB, numbers=left, firstnumber=1, basicstyle=\small\ttfamily, numberstyle=\tiny\color{gray}]
Option Explicit

Dim Power1(20000) As Double
Dim Power2(20000) As Double
Dim Power3(20000) As Double
Dim Power4(20000) As Double

Dim Factor(20000) As Integer

Dim a5 As Double, a4 As Double, a3 As Double, a2 As Double
Dim a1 As Double
Dim b5 As Double, b4 As Double, b3 As Double, b2 As Double

Dim a5_Start As Double
Dim a5_End As Double

Dim mini As Double

Dim i As Double

' Mathematical variables
Dim P1 As Double, P2 As Double, P3 As Double, P4 As Double
Dim S1 As Double, S2 As Double, S3 As Double, S4 As Double
Dim W4 As Double, W5 As Double, W6 As Double
Dim D3 As Double, C3 As Double, D0 As Double, C0 As Double
Dim D1 As Double, D2 As Double

' Sum arrays
Dim SumB4(4) As Double, SumB5(4) As Double
Dim SumA3(4) As Double, SumA4(4) As Double

' Search bounds
Dim b4min As Double, b5min As Double, b5max As Double
Dim a4min As Double, a4max As Double, a3min As Double

Dim Fmax As Integer

Dim q3 As Double
Dim q2 As Double
Dim q1 As Double

Dim TableIndex As Integer

Private Sub Form_load()

mini = 0.000001

Factor(0) = 10000
Open "Factor.txt" For Input As #10
For i = 1 To 10000
  Input #10, Factor(i)                          'Constant C
Next i
Close #10

For i = 1 To 15000
    Power1(i) = i
    Power2(i) = i * i
    Power3(i) = i * i * i
    Power4(i) = i * i * i * i
Next i


ReDim b4table(10000) As Long

For i = 1 To 10000
  b4table(i) = 10000
Next i

Open "Table b4min.txt" For Input As #1
For i = 1 To 954       '=10000-9046
  Input #1, TableIndex, b4table(TableIndex)     'Exact Bounds
Next i
Close #1

Open "Progress.ini" For Input As #3
    Input #3, a5_Start, a5_End
Close #3

For a5 = a5_Start To a5_End

    Fmax = Int(a5 / 2)
   
    If Factor(a5) > Fmax Then GoTo NEXTA5       'Constant C
    
    b5min = Int(0.904508 * a5) + 1
    b5max = a5 - 1

For b5 = b5min To b5max                         'Exact Bounds

    If Factor(a5 - b5) > Fmax Then GoTo NEXTB5  'Constant C
    
    SumB5(1) = Power1(a5) - Power1(b5)          '=a5-b5
    If SumB5(1) < 0 Then GoTo NEXTB5            'Seesaw Conj.
    
    SumB5(2) = Power2(a5) - Power2(b5)
    If SumB5(2) < 0 Then GoTo NEXTB5            'Seesaw Conj.
        
    SumB5(3) = Power3(a5) - Power3(b5)
    If SumB5(3) < 0 Then GoTo NEXTB5            'Seesaw Conj.
    
    SumB5(4) = Power4(a5) - Power4(b5)
    If SumB5(4) < 0 Then GoTo NEXTB5            'Seesaw Conj.
    
    b4min = Int(a5 * b4table(Int(10000 * b5 / a5)) / 1000000)
            
For b4 = b4min To b5                            'Exact Bounds

    If Factor(a5 - b4) > Fmax Then GoTo NEXTB4  'Constant C
    
    SumB4(1) = Power1(b4) - SumB5(1)            '=b4+b5-a5
    If SumB4(1) < 0 Then GoTo NEXTB4            'Seesaw Conj.
    
    SumB4(2) = Power2(b4) - SumB5(2)
    If SumB4(2) < 0 Then GoTo NEXTB4            'Seesaw Conj.
      
    SumB4(3) = Power3(b4) - SumB5(3)
    If SumB4(3) < 0 Then GoTo NEXTB4            'Seesaw Conj.
    
    SumB4(4) = Power4(b4) - SumB5(4)
    If SumB4(4) < 0 Then GoTo NEXTB4            'Seesaw Conj.
    
    q1 = SumB4(1) ^ 4 / SumB4(4) ^ 1
    If q1 < 1 Then GoTo NEXTB4                  'Threshold Ineq.
    If q1 > 1.2648118 Then GoTo NEXTB4          'Threshold Ineq.
    
    q2 = SumB4(2) ^ 4 / SumB4(4) ^ 2
    If q2 < 1 Then GoTo NEXTB4                  'Threshold Ineq.
    If q2 > 1.428761837 Then GoTo NEXTB4        'Threshold Ineq.
    
    q3 = SumB4(3) ^ 4 / SumB4(4) ^ 3
    If q3 < 1 Then GoTo NEXTB4                  'Threshold Ineq.
    If q3 > 1.366557215 Then GoTo NEXTB4        'Threshold Ineq.

    a4min = Int(Sqr(Sqr(SumB4(4) / 2)) - mini) + 1
    a4max = Int(Sqr(Sqr(SumB4(4))) + mini)

For a4 = a4min To a4max                     'Conservative Bounds

    If Factor(a4) > Fmax Then GoTo NEXTA4       'Constant C
    If Factor(b5 - a4) > Fmax Then GoTo NEXTA4  'Constant C
    If Factor(b4 - a4) > Fmax Then GoTo NEXTA4  'Constant C
    
    SumA4(1) = SumB4(1) - Power1(a4)            'b4+b5-a4-a5
    If SumA4(1) < 0 Then GoTo NEXTA4            'Seesaw Conj.
   
    SumA4(2) = SumB4(2) - Power2(a4)
    If SumA4(2) < 0 Then GoTo NEXTA4            'Seesaw Conj.
    
    SumA4(3) = SumB4(3) - Power3(a4)
    If SumA4(3) < 0 Then GoTo NEXTA4            'Seesaw Conj.
    
    SumA4(4) = SumB4(4) - Power4(a4)
    If SumA4(4) < 0 Then GoTo NEXTA4            'Seesaw Conj.
    
    q1 = SumA4(1) ^ 4 / SumA4(4) ^ 1
    If q1 < 0 Then GoTo NEXTA4              'Threshold Ineq.
    If q1 > 0.00367257266 Then GoTo NEXTA4  'Threshold Ineq.
    
    q2 = SumA4(2) ^ 4 / SumA4(4) ^ 2
    If q2 < 0 Then GoTo NEXTA4              'Threshold Ineq.
    If q2 > 0.0775630998 Then GoTo NEXTA4   'Threshold Ineq.
    
    q3 = SumA4(3) ^ 4 / SumA4(4) ^ 3
    If q3 < 0 Then GoTo NEXTA4              'Threshold Ineq.
    If q3 > 0.425099279 Then GoTo NEXTA4    'Threshold Ineq.

    a3min = Int(Sqr(Sqr(SumA4(4))) - mini) + 1

For a3 = a3min To a4                    'Conservative Bounds
    
    If Factor(a3) > Fmax Then GoTo NEXTA3       'Constant C
    If Factor(b5 - a3) > Fmax Then GoTo NEXTA3  'Constant C
    If Factor(b4 - a3) > Fmax Then GoTo NEXTA3  'Constant C
    
    P1 = Power1(a3) - SumA4(1)   '=a3+a4+a5-b4-b5=b2+b3-a2-a1
    If P1 <= 0 Then GoTo NEXTA3                 'Seesaw Conj.
    
    P2 = Power2(a3) - SumA4(2)
    If P2 < 0 Then GoTo NEXTA3                  'Seesaw Conj.
    
    P3 = Power3(a3) - SumA4(3)
    If P3 < 0 Then GoTo NEXTA3                  'Seesaw Conj.
    
    P4 = Power4(a3) - SumA4(4)
    If P4 <= 0 Then GoTo NEXTA3                 'Seesaw Conj.
    
    S1 = P1
    S2 = (P2 + S1 * P1) / 2
    S3 = (P3 + S1 * P2 + S2 * P1) / 3
    S4 = (P4 + S1 * P3 + S2 * P2 + S3 * P1) / 4

    W4 = S2 ^ 2 - S3 * S1      '=(a1-b2)(a1-b3)(a2-b2)(a2-b3)
    If W4 = 0 Then GoTo NEXTA3              'Integrality Test
       
    W5 = S3 * S2 - S4 * S1
    D1 = W5 / W4                    'Integrality Test, =b2+b3
    If Abs(Int(D1 + mini) - D1) > mini Then GoTo NEXTA3
    
    W6 = S3 ^ 2 - S4 * S2
    D2 = W6 / W4                    'Integrality Test, =b2*b3
    If Abs(Int(D2 + mini) - D2) > mini Then GoTo NEXTA3
    
    D3 = D1 ^ 2 - 4 * D2                    '=(b3-b2)^2
    If D3 < 0 Then GoTo NEXTA3              'Var. Calculation
    
    C3 = 2 * (D1 ^ 2 - 2 * D2 - P2) - (D1 - P1) ^ 2  '=(a2-a1)^2
    If C3 < 0 Then GoTo NEXTA3              'Var. Calculation
    
    D0 = Sqr(D3)                            '=(b3-b2)
    If Abs(Int(D0 + mini) - D0) > mini Then GoTo NEXTA3
    b3 = (D1 + D0) / 2
    b2 = (D1 - D0) / 2
    If b3 >= a3 Then GoTo NEXTA3
    If b3 <= 0 Then GoTo NEXTA3
    If b2 <= 0 Then GoTo NEXTA3
    
    C0 = Sqr(C3)                            '=(a2-a1)
    If Abs(Int(C0 + mini) - C0) > mini Then GoTo NEXTA3
    a2 = (D1 - P1 + C0) / 2
    a1 = (D1 - P1 - C0) / 2
    If a2 > a3 Then GoTo NEXTA3
    If a2 <= 0 Then GoTo NEXTA3
    If a1 <= 0 Then GoTo NEXTA3
            
    Open "k1234.dat" For Append As #4
        Print #4, a1; a2; a3; a4; a5, 0; b2; b3; b4; b5
    Close #4

NEXTA3:
Next a3

NEXTA4:
Next a4

NEXTB4:
Next b4

NEXTB5:
Next b5

    Open "Progress.ini" For Output As #3
        Print #3, a5, a5_End
    Close #3
    
    Open "nvtime.ini" For Append As #5
        Print #5, a5, Now()
    Close #5

NEXTA5:
Next a5

End

End Sub

Private Sub Form_Unload(Cancel As Integer)
Unload Me
End Sub
\end{lstlisting}

In above reference code, we employ a variable \texttt{mini}, which enables the entire search process to be conducted using double-precision floating-point numbers, thereby eliminating the need for additional handling of large integers and ensuring search efficiency. It is important to note that the value of \texttt{mini} is related to the search range. In the reference code, \texttt{mini} is set to \(0.000001\), which is appropriate for \(a_5\) values up to 3000.\\

Another important variable in the code is \texttt{Fmax}, which represents the maximum factor of the Constant \( C \). It can be proven that for the case \( (k=1,2,3,4) \), the maximum factor of \( C \) will not exceed \( a_5 / 2 \). Therefore, in line 81 of the code, \texttt{Fmax} is set to \(\text{Int}(a_5 / 2)\). Under this setting, the code can obtain all non-negative integer solutions for \( (k=1,2,3,4) \) within a certain range of \( a_5 \). For example, when \( a_5 \leq 500 \), the program yields 7213 non-negative integer solutions. Among these 7213 solutions, we observe that the vast majority of the solutions have a maximum factor within 100. Only 55 solutions have a maximum factor exceeding 100, among which only 6 solutions (all symmetric) have a maximum factor exceeding 150, including the unique case where the maximum factor exceeds 200. The following are the six solutions with a maximum factor exceeding 150:
\begin{align}
& [11, 13, 182, 334, 360]^k = [0, 26, 178, 347, 349]^k, \quad (k=1,2,3,4),\nonumber \\
& \qquad \text{where} \quad C = 2^5 \cdot 3^2 \cdot 5 \cdot 7 \cdot 11 \cdot 13^2 \cdot 167 \nonumber \\[1mm]
& [21, 39, 208, 302, 380]^k = [0, 78, 172, 341, 359]^k, \quad (k=1,2,3,4),\nonumber \\
& \qquad \text{where} \quad C = 2^7 \cdot 3^2 \cdot 5 \cdot 7 \cdot 13^2 \cdot 19 \cdot 151 \nonumber \\[1mm]
& [9, 10, 210, 398, 418]^k = [0, 20, 208, 408, 409]^k, \quad\phantom{0} (k=1,2,3,4),\nonumber \\
& \qquad \text{where} \quad C = 2^4 \cdot 3^3 \cdot 5^2 \cdot 7 \cdot 11 \cdot 19 \cdot 199 \nonumber \\[1mm]
& [21, 70, 270, 302, 442]^k = [0, 140, 172, 372, 421]^k, \quad (k=1,2,3,4),\nonumber \\
& \qquad \text{where} \quad C = 2^4 \cdot 3^4 \cdot 5^2 \cdot 7^2 \cdot 13 \cdot 17 \cdot 151 \nonumber \\[1mm]
& [21, 30, 230, 382, 442]^k = [0, 60, 212, 412, 421]^k, \quad (k=1,2,3,4),\nonumber \\
& \qquad \text{where} \quad C = 2^4 \cdot 3^2 \cdot 5^2 \cdot 7 \cdot 13 \cdot 17 \cdot 23 \cdot 191, \nonumber \\[1mm]
& [13, 15, 240, 446, 476]^k = [0, 30, 236, 461, 463]^k, \quad (k=1,2,3,4),\nonumber \\
& \qquad \text{where} \quad C = 2^7 \cdot 3^2 \cdot 5^2 \cdot 7 \cdot 13 \cdot 17 \cdot 223 \nonumber
\end{align}
Thus, we can derive an efficient Selective Search strategy: by setting \texttt{Fmax} to a relatively small value, we may potentially miss a few solutions during the search process, but the speed of finding the vast majority of solutions will be significantly increased. The following table compares the search times and the number of solutions obtained for different \texttt{Fmax} settings within the range \( a_5 \leq 500 \). The setting \texttt{Fmax} as \texttt{None} indicates that all lines of code involving \texttt{Fmax}, including lines 81, 83, 90, 108, etc., are removed. The \texttt{Seconds Occupied} were measured on the same i7 laptop.\\
\begin{table}[H]
\centering
\begin{tabular}{|c|c|c|c|c|c|c|c|}
\hline
\text{Fmax Setting} & 100 & 150 & 200 & \(a_5/2\) & 250 & None & 500 \\
\hline
\text{Seconds Occupied} & 172 & 242 & 299 & 305 & 327 & 449 & 480 \\
\hline
\text{Solutions Found} & 7158 & 7207 & 7212 & 7213 & 7213 & 7213 & 7213 \\
\hline
\end{tabular}
\caption{Search Time Occupied for \( a_5 \leq 500 \) }
\label{tab:comparison}
\end{table}
Among the 7213 solutions within the range \( a_5 \leq 500 \), there are 4790 coprime sets of solutions. Of these 4790 coprime sets, 4354 sets are non-symmetric, while the remaining 436 sets are symmetric. In fact, for higher-degree PTE problems, including types \( (k=1,2,3,4,5,6) \) and \( (k=1,2,3,4,5,6,7) \), non-symmetric solutions are the majority.\\

In summary, this chapter presents a comprehensive set of novel computer search algorithms exemplified by the case \( (k=1,2,3,4) \). These algorithms have been iteratively refined over the past four decades of research. As a general algorithm for PTE-type problems, the methods discussed herein are significantly more efficient than those reported in the existing literature, including \cite{Borwein2003}, \cite{Caley12}, \cite{Coppersmith2024}, \cite{Qiu2016}, \cite{YingqiongChen2018} and \cite{Zhang2023}. More importantly, the algorithms presented in this chapter are readily applicable to the broader class of GPTE problems. The following are examples of some classic numerical solutions obtained using the algorithms introduced in this chapter.
\begin{align}
& [ 498, 3773, 3783, 4567, 4787 ]^{k}=[ 517, 3598, 4017, 4463, 4827 ]^{k}, \nonumber \\
& \qquad\qquad\qquad\qquad\qquad\qquad\qquad\qquad\qquad (k=2,4,6,8)
\tag*{\eqref{k2468s4827}}\\
\label{k14}
& [34, 133, 165, 299, 332, 366]^k = [35, 124, 177, 286, 353, 354]^k, \nonumber \\
& \qquad\qquad\qquad\qquad\qquad\qquad\qquad\qquad\qquad (k = 1, 2, 3, 4, 7) \tag*{\eqref{k12347s366}} \\
& [269, 397, 409, 683, 743, 901, 923]^k \nonumber \\ 
& \quad = [299, 313, 493, 613, 827, 839, 941]^k, \quad (k = 1, 2, 3, 5, 7, 9) \tag*{\eqref{k123579s941}} \\
& [0,50,111,233,236,307,438,469]^k \nonumber \\
& \quad =[1,46,119,203,282,285,440,468]^k, \;\; (k = 1, 2, 3, 4, 5, 6, 7)
\tag*{\eqref{k1234567s469}}\\
& [77, 159, 169, 283, 321, 443, 447, 501]^k \nonumber \\ 
& \quad = [79, 137, 213, 237, 363, 399, 481, 491]^k, \;\; (k = 1, 2, 3, 4, 5, 6, 8) \tag*{\eqref{k1234568s501}}\\
& [387, 388, 416, 447, 494, 536, 573, 589, 610]^k \nonumber \\
& \quad = [382, 402, 403, 456, 485, 549, 559, 596, 608]^k, \nonumber \\
& \qquad\qquad\qquad\qquad\qquad (k = 0, 1, 2, 3, 4, 5, 6, 7) \tag*{\eqref{k01234567s610}}\\
& [85, 286, 702, 858]^k = [81, 374, 585, 891]^k, \quad (k = -1, 1, 5) \tag*{\eqref{k15n1s891}} \\
& [266, 494, 494, 1463, 1547]^k = [287, 374, 611, 1394, 1598]^k, \nonumber \\
& \qquad\qquad\qquad\qquad\qquad\qquad\qquad\qquad\quad (k = -1, 1, 2, 3) \tag*{\eqref{k123n1s1598}} \\
& [56, 77, 99, 152, 174, 228, 261 ] ^k = [ 57, 72, 116, 126, 203, 209, 264] ^k, \nonumber \\
& \qquad\qquad\qquad\qquad\qquad\qquad\qquad\qquad\qquad (k = -1, 0, 1, 2, 3, 4)
\tag*{\eqref{k01234n1s264}} \\
& [7980, 8060, 8151, 8360, 8463, 8680, 8778, 8866] ^k  \nonumber \\
&\quad =[8008, 8008, 8246, 8246, 8580, 8580, 8835, 8835] ^k, \nonumber \\
& \qquad\qquad\qquad\qquad\qquad\qquad\qquad (k = -3, -2, -1, 0, 1, 2, 3)
\tag*{\eqref{k0123n123s8866}}
\end{align}

\clearpage

\section{Some New Parametric Solutions for GPTE}
For the various types of the GPTE problem, constructing parametric solutions remains an effective approach, in addition to relying on computer searches. Over the past century, notable contributions have been made by Albert Gloden \cite{Gloden44}, Alfred Moessner \cite{Moessner39}, Ajai Choudhry \cite{Choudhry11, Choudhry13}, Jaros{\l}aw Wr\'oblewski \cite{JW955,JW09}, and others. For details, see Appendices A and B, as well as the references. This chapter lists some of the new parametric solutions we have found, including methods that yield prime solutions for the GPTE problem.
\subsection{Parametric Method for Ideal Integer Solutions}
\begin{example}
For arbitrary $m,n,p$, and $q$, we define the set $\{ a_1, a_2, a_3, a_4,\\ b_1, b_2, b_3, b_4\}$ as follows:
\begin{align} 
& a_1 = m(m-n)(p^2+q^2), \nonumber\\    
& a_2 = n(m+n)(p^2+q^2), \nonumber\\    
& a_3 = (n-m)n(p^2+q^2), \nonumber\\   
& a_4 = m(m+n)(p^2+q^2), \nonumber\\    
& b_1 = p(p-q)(m^2+n^2), \nonumber\\    
& b_2 = q(p+q)(m^2+n^2), \nonumber\\    
& b_3 = (q-p)q(m^2+n^2), \nonumber\\    
& b_4 = p(p+q)(m^2+n^2), \nonumber 
\end{align}
Then, we have the following relationships:
\begin{align}
\label{parah123n1}
& [a_1,a_2,a_3,a_4]^h=[b_1,b_2,b_3,b_4]^h , & (h=-1,1,2,3)\\ 
\label{parah013n1a}
& [a_1,a_2,-b_3,-b_4]^h=[b_1,b_2,-a_3,-a_4]^h , & (h=-1,0,1,3)\\ 
\label{parah013n1b}
& [a_1,a_2,-b_1,-b_2]^h=[b_3,b_4,-a_3,-a_4]^h , & (h=-1,0,1,3)
\end{align}
\end{example}
In the parameter solutions listed above, $\{ a_1, a_2, a_3, a_4, b_1, b_2, b_3, b_4\}$ were \\reorganized by the author in 2017 based on the parameter solution \eqref{parah123n1} for $(h=-1,1,2,3)$ obtained by Ajai Choudhry in 2011 \cite{Choudhry11}. The parameter solutions \eqref{parah013n1a} and \eqref{parah013n1b} for $(h=-1,0,1,3)$ were derived by the author in 2017 \cite{Chenhminus23}. When $\{m,n,p,q\}=\{1,2,3,4\}$, we have
\begin{align}
& [-5, 10, 15, 30]^h = [-3, 4, 21, 28]^h, &(h=-1,1,2,3)  \tag*{\eqref{h123n1s30}} \\
& [-30, 4, 5, 21]^h = [-28, 3, 10, 15]^h, &(h=-1,0,1,3)  \tag*{\eqref{h013n1s30}} \\
& [-28,-5, 3, 30]^h = [ -15, -10, 4, 21]^h,&(h=-1,0,1,3)  \tag{\ref{h013n1s28}}
\end{align}

\begin{example}\label{Parah12345n1}
For arbitrary $m,n,p$, and $q$, we define the set $\{ a_1, a_2, \dots, a_6,\\ b_1, b_2, \dots, b_6\}$ as follows:
\begin{align} 
& a_1 = m(m-n)(p^2+pq+q^2), \nonumber\\ 
& a_2 = n(m+2n)(p^2+pq+q^2), \nonumber\\
& a_3 = (m+n)(2m+n)(p^2+pq+q^2), \nonumber\\ 
& a_4 = n(n-m)(p^2+pq+q^2), \nonumber\\ 
& a_5 = m(2m+n)(p^2+pq+q^2), \nonumber\\ 
& a_6 = (m+n)(m+2n)(p^2+pq+q^2), \nonumber\\ 
& b_1 = p(p-q)(m^2+mn+n^2), \nonumber\\ 
& b_2 = q(p+2q)(m^2+mn+n^2), \nonumber\\ 
& b_3 = (p+q)(2p+q)(m^2+mn+n^2), \nonumber\\ 
& b_4 = q(q-p)(m^2+mn+n^2), \nonumber\\ 
& b_5 = p(2p+q)(m^2+mn+n^2), \nonumber\\ 
& b_6 = (p+q)(p+2q)(m^2+mn+n^2), \nonumber 
\end{align}
Then, we have the following relationships:
\begin{align}
\label{parah12345n1}
& [a_1,a_2,a_3,a_4,a_5,a_6]^h=[b_1,b_2,b_3,b_4,b_5,b_6]^h,\nonumber\\ 
& \qquad\qquad\qquad\qquad\qquad\qquad (h=-1,1,2,3,4,5) \\ 
\label{parah01235n1a}
& [a_1,a_2,a_3,-b_4,-b_5,-b_6]^h=[b_1,b_2,b_3,-a_4,-a_5,-a_6]^h,\nonumber\\
& \qquad\qquad\qquad\qquad\qquad\qquad  (h=-1,0,1,2,3,5) \\
\label{parah01235n1b}
& [a_1,a_2,a_3,-b_1,-b_2,-b_3]^h=[b_4,b_5,b_6,-a_4,-a_5,-a_6]^h,\nonumber\\ 
& \qquad\qquad\qquad\qquad\qquad\qquad   (h=-1,0,1,2,3,5) 
\end{align}
\end{example}
In the parameter solutions listed above, $\{ a_1, a_2, \dots, a_6, b_1, b_2, \dots, b_6\}$ were \\reorganized by the author in 2017 based on the parameter solution \eqref{parah12345n1} for $(h=-1,1,2,3,4,5)$ obtained by Ajai Choudhry in 2011 \cite{Choudhry11}. The parameter solutions \eqref{parah01235n1a} and \eqref{parah01235n1b} for $(h=-1,0,1,2,3,5)$ were derived by the author in 2017 \cite{Chenhminus23}. When $\{m,n,p,q\}=\{1,2,3,5\}$, we have
\begin{align}
& [-7, 14, 28, 70, 84, 105 ]^h = [-6, 10, 33, 65, 88, 104]^h,\nonumber\\
& \qquad\qquad\qquad\qquad\qquad\qquad (h=-1,1,2,3,4,5)  \tag*{\eqref{h12345n1s105}}\\
& [ -88, -65, 6, 14, 28, 105 ]^h = [-84, -70, 7, 10, 33, 104 ]^h,\nonumber\\ 
& \qquad\qquad\qquad\qquad\qquad\qquad (h=-1,0,1,2,3,5)  \tag*{\eqref{h01235n1s104}}\\
& [-104, -33, -10, 14, 28, 105 ]^h = [-84, -70, -6, 7, 65, 88 ]^h,\nonumber\\ 
& \qquad\qquad\qquad\qquad\qquad\qquad (h=-1,0,1,2,3,5)  \tag*{\eqref{h01235n1s88}}
\end{align}

\begin{example}
For arbitrary $p$, and $q$, we define
\begin{align}
& a_1=(57 p - 53 q) (21 p - 17 q), \nonumber\\   
& a_2= (43 p - 39 q) (27 p - 23 q), \nonumber\\   
& a_3=2 (29 p - 25 q) (11 p - 9 q), \nonumber\\   
& a_4= (37 p - 33 q) (15 p - 11 q), \nonumber\\   
& a_5= 2 (17 p -   13 q) (11 p - 9 q), \nonumber\\  
& b_1= (37 p - 33 q) (29 p - 25 q), \nonumber\\   
& b_2= 2 (57 p - 53 q) (11 p - 9 q), \nonumber\\   
& b_3= (43 p - 39 q) (17 p - 13 q), \nonumber\\   
& b_4=2 (21 p - 17 q) (11 p - 9 q), \nonumber\\  
& b_5=(27 p -  23 q) (15 p - 11 q). \nonumber
\end{align}
then we have
\begin{align}
\label{parak0123}
[a_1,a_2,a_3,a_4,a_5]^k=[b_1,b_2,b_3,b_4,b_5]^k,\quad (k=0,1,2,3) 
\end{align}
\end{example}
The above parameter solution \eqref{parak0123} were derived by the author in 2023 \cite{Chen2125}. When  $p=1,q=2$, we have
\begin{align}
[18, 29, 42, 91, 95]^k=[19, 26, 45, 87, 98]^k , \quad (k=0,1,2,3) \tag*{\eqref{k0123s98}}
\end{align}

\begin{example}
Let the sequences of rational numbers $a_n, b_n, c_n, d_n, e_n, p_n, q_n, r_n,\\s_n, t_n$ defined by
\begin{align}
\frac{287419 x^2-559208 x+269325}{2 (x-1) (261 x-275) (275 x-261)}=:& \sum _{n=0}^{\infty} a_n x^n,\\
\frac{286143 x^2-547144 x+261225}{2 (x-1) (261 x-275) (275 x-261)}=:& \sum _{n=0}^{\infty} b_n x^n,\\
\frac{1}{x-1}=:& \sum _{n=0}^{\infty} c_n x^n,\\
\frac{115797 x^2-242800 x+124875}{2 (x-1) (261 x-275) (275 x-261)}=:& \sum _{n=0}^{\infty} d_n x^n,\\
\frac{143 x-153}{(x-1) (275 x-261)}=:& \sum _{n=0}^{\infty} e_n x^n,\\
\frac{957 x-925}{2 (x-1) (261 x-275)}=:& \sum _{n=0}^{\infty} p_n x^n,\\
\frac{583 x-513}{(x-1) (275 x-261)}=:& \sum _{n=0}^{\infty} q_n x^n,\\
\frac{161733 x^2-328784 x+164475}{2 (x-1) (261 x-275) (275 x-261)}=:& \sum _{n=0}^{\infty} r_n x^n,\\
\frac{187 x-189}{(x-1) (275 x-261)}=:& \sum _{n=0}^{\infty} s_n x^n,\\
\frac{80707 x^2-172280 x+91125}{2 (x-1) (261 x-275) (275 x-261)}=:& \sum _{n=0}^{\infty} t_n x^n.
\end{align}
Then for each $n\geq 0$,
\begin{align}
\label{parak0123Ramanujan}
& [\ a_n, b_n, c_n, d_n, e_n \ ]^k=[\ p_n, q_n, r_n, s_n, t_n\ ]^k, \quad (k=0,1,2,3).
\end{align}
\end{example}
\noindent
The above result \eqref{parak0123Ramanujan} was derived by the author in 2023 \cite{Chen2125}, inspired by \mbox{Ramanujan's} remarkable identity \eqref{parak3Ramanujan} for \(a^3 + b^3 = c^3 + 1\), and by referring to the specific methods of M. D. Hirschhorn \cite{Hirschhorn1995} \cite{Hirschhorn1996} \cite{HanHirschhorn2006}, J. M. Laughlin \cite{McLaughlin2010}, and Kwang-Wu Chen \cite{ChenKW2020}, starting from the parametric solution \eqref{parak0123}. 
In \eqref{parak0123Ramanujan}, when \( n = 0, 1\), we obtain the following numerical results for type $(k=0,1,2,3)$:
\begin{align}
& [374, 555, 638, 1161, 1197]^k=[405, 462, 731, 1073, 1254]^k,\\
& [6007650, 8432191, 9158490, 15136189, 15918257]^k \nonumber \\
& \quad =[6465737, 7057930, 10532751, 14085909, 16510450]^k.
\end{align}
The identity given by Ramanujan \cite{Ramanujan1988} is as follows: If the sequences \(\{a_n\}\), \(\{b_n\}\), and \(\{c_n\}\) are defined by
\begin{equation}
\begin{aligned}  
\frac{1 + 53x + 9x^2}{1 - 82x - 82x^2 + x^3} &= \sum_{n=0}^\infty a_n x^n, \\
\frac{2 - 26x - 12x^2}{1 - 82x - 82x^2 + x^3} &= \sum_{n=0}^\infty b_n x^n, \\  
\frac{2 + 8x - 10x^2}{1 - 82x - 82x^2 + x^3} &= \sum_{n=0}^\infty c_n x^n,   
\end{aligned}
\end{equation}
then 
\begin{align}
\label{parak3Ramanujan}
a_n^3 + b_n^3 = c_n^3 + (-1)^n, \quad \text{for all } n \geq 0.
\end{align}
For instance, when \( 0 \leq n \leq 6 \), the specific values are:
\begin{align}
& a_{n} = 1, 135, 11161, 926271, 76869289, 6379224759, 529398785665, \nonumber \\
& b_{n} = 2, 138, 11468, 951690, 78978818, 6554290188, 543927106802, \nonumber \\
& c_{n} = 2, 172, 14258, 1183258, 98196140, 8149096378, 676276803218. \nonumber 
\end{align}
It follows that
\begin{align}
& 135^3 + 138^3 = 172^3 + (-1)^1, \nonumber \\
& 11161^3 + 11468^3 = 14258^3 + (-1)^2, \nonumber \\
& 926271^3 + 951690^3 = 1183258^3 + (-1)^3, \quad \text{etc.} \nonumber 
\end{align}
\begin{example}\label{exampleh134para}
For arbitrary $p$ and $q$, we denote $\{ u, v, s_1, s_2, s_3\}$ as follows:
\begin{align}
& u = p^2 q^2 (p+q)^2, \\
& v = \left(p^2 - pq - q^2\right) \left(p^2 + pq - q^2\right) \left(p^2 + pq + q^2\right) \left(p^2 + 3pq + q^2\right), \\
& s_1 = v + 3p(p+q)u, \\
& s_2 = v + 3q(p+q)u, \\
& s_3 = v - 3pqu.
\end{align}
We also denote $\{a_1, b_1, c_1, d_1, e_1, f_1\}$ as follows:
\begin{align}
& a_1 = p(p+q)\left(s_1 - 3pq^5(p+q)(2p+q)\right), \\
& b_1 = q(p+q)\left(s_2 + 3p^5q(p+q)(p+2q)\right), \\
& c_1 = -pq\left(s_3 + 3pq(p-q)(p+q)^5\right), \\
& d_1 = p(p+q)\left(s_1 + 3pq^5(p+q)(2p+q)\right), \\
& e_1 = q(p+q)\left(s_2 - 3p^5q(p+q)(p+2q)\right), \\
& f_1 = -pq\left(s_3 - 3pq(p-q)(p+q)^5\right).
\end{align}
Then we have
\begin{align}
\label{h134para1}
[a_1, b_1, c_1]^h = [d_1, e_1, f_1]^h, \quad (h=1,3,4).
\end{align}
Furthermore, we denote $\{r_1, r_2, r_3, w, t_1, t_2, t_3\}$ as follows:
\begin{align}
& r_1 = 3(q-p)(p+2q)u, \\
& r_2 = 3(p-q)(2p+q)u, \\
& r_3 = 3(2p+q)(p+2q)u, \\
& w = 54u^4 + 12\left(p^2 + pq + q^2\right)^2u^2v + 6\left(p^2 + pq + q^2\right)uv^2 + v^3, \\
& t_1 = ws_1 - r_1\left(w - 2s_1u(2p+q)^2\left(s_1 + r_1\right)\right), \\
& t_2 = ws_2 - r_2\left(w - 2s_2u(p+2q)^2\left(s_2 + r_2\right)\right), \\
& t_3 = ws_3 - r_3\left(w - 2s_3u(p-q)^2\left(s_3 + r_3\right)\right).
\end{align}
We also denote $\{a_2, b_2, c_2, d_2, e_2, f_2\}$ as follows:
\begin{align}
& a_2 = -p(p+q)\left(t_1 + 6pq^5(p+q)(2p+q)w\right), \\
& b_2 = -q(p+q)\left(t_2 - 6p^5q(p+q)(p+2q)w\right), \\
& c_2 = pq\left(t_3 - 6pq(p-q)(p+q)^5w\right), \\
& d_2 = -p(p+q)\left(t_1 - 6pq^5(p+q)(2p+q)w\right), \\
& e_2 = -q(p+q)\left(t_2 + 6p^5q(p+q)(p+2q)w\right), \\
& f_2 = pq\left(t_3 + 6pq(p-q)(p+q)^5w\right).
\end{align}
Then we have
\begin{align}
\label{h134para2}
[a_2, b_2, c_2]^h = [d_2, e_2, f_2]^h, \quad (h=1,3,4).
\end{align}
\end{example}
In the parameter solutions listed above, $\{a_1, b_1, c_1, d_1, e_1, f_1\}$ were \mbox{reorganized} by the author in 2023 based on the parameter solutions obtained by Ajai Choudhry in 1991 \cite{Choudhry91}, while $\{a_2, b_2, c_2, d_2, e_2, f_2\}$ were derived by the author in 2023 using a new method \cite{Chen2125}. When $p=1$ and $q=2$, by \eqref{h134para1}, we have
\begin{align}
& [-3254, 5583, 5658]^h = [-1329, 2578, 6738]^h, \quad (h=1,3,4). \tag*{\eqref{h134s6738}}
\end{align}
By \eqref{h134para2}, we have
\begin{align} 
 & [-2608425958605, 3669662071230, 6406539709126]^h  \nonumber \\
 & \quad = [-1353176207690, 2232677642190, 6588274387251]^h, \nonumber \\
 & \qquad\qquad\qquad\qquad\qquad\qquad (h=1,3,4) \tag*{\eqref{h134s6588274387251}}
\end{align}

\begin{example}
We take the initial value \(u_0\) as follows:
\begin{align} 
& u_0 = 67.
\end{align}
Next, we use the following recurrence formulas for \(u_i\) and \(v_i\):
\begin{align} 
& u_i=7 u_{i-1}+4 \sqrt{3 u_{i-1}^2-11},\\
& v_i=\frac{1}{2} \sqrt{3 u_i^2-11}.
\end{align}
and define the sets \(\{A_{0,i}, A_{1,i}, \dots, A_{5,i}\}\) and \(\{B_{0,i}, B_{1,i}, \dots, B_{5,i}\}\) as
\begin{align} 
& \left\{A_{0,i},A_{1,i},A_{2,i},A_{3,i},A_{4,i},A_{5,i}\right\}  \nonumber\\
& \quad=\left\{0,\frac{u_i-1}{2},\frac{u_i+1}{2},\frac{3 u_i-1}{2} ,\frac{3 u_i+1}{2},2 u_i\right\},\\
& \left\{B_{0,i},B_{1,i},B_{2,i},B_{3,i},B_{4,i},B_{5,i}\right\}  \nonumber\\
& \quad=\left\{u_i-v_i-1,u_i-v_i+1,u_i-2,u_i+2,u_i+v_i-1,u_i+v_i+1\right\}
\end{align}
Then, we have
\begin{align}
& [A_{0,i},A_{1,i},A_{2,i},A_{3,i},A_{4,i},A_{5,i}]^k=[B_{0,i},B_{1,i},B_{2,i},B_{3,i},B_{4,i},B_{5,i}]^k,\nonumber\\
& \qquad\qquad\qquad\qquad\qquad\qquad\qquad\qquad (k=1,2,3,4,5)
\end{align}
Furthermore,
\begin{align}
& \lim_{i \to \infty} \frac{A_{t,i}}{A_{5,i}}=\lim_{i \to \infty} \frac{A_{t-1,i}}{A_{5,i}}=\sin ^2\left(\frac{t \pi }{12}\right),\quad (t=2,4)\\
& \lim_{i \to \infty} \frac{B_{t,i}}{A_{5,i}}=\lim_{i \to \infty} \frac{B_{t-1,i}}{A_{5,i}}=\sin ^2\left(\frac{t \pi }{12}\right),\quad (t=1,3,5)
\end{align}
\end{example}
For instance, when we take \( i \) to be 1, 2, and 3, respectively, then we obtain the following results:
\begin{align} 
& [0,466,467,1399,1400,1866]^k=[124,126,931,935,1740,1742]^k,\\
& [0,6497,6498,19492,19493,25990]^k \nonumber\\
& \quad =[1740,1742,12993,12997,24248,24250]^k,\\
& [0,90498,90499,271495,271496,361994]^k \nonumber\\
& \quad= [24248,24250,180995,180999,337744,337746]^k,\\
& \qquad\qquad\qquad\qquad\qquad (k=1,2,3,4,5) \nonumber
\end{align}
When we take \( i=20 \), we obtain the following results:
\begin{align} 
& \{A_{0,20},A_{1,20},A_{2,20},A_{3,20},A_{4,20},A_{5,20}\}\nonumber\\
& \quad = \{0, 2528471117073590555281217, 2528471117073590555281218, \nonumber\\
& \qquad 7585413351220771665843652, 7585413351220771665843653, \nonumber\\
& \qquad 10113884468294362221124870\}, \nonumber\\
& \{B_{0,20},B_{1,20},B_{2,20},B_{3,20},B_{4,20},B_{5,20}\}\nonumber\\
& \quad = \{677501793905287305056880, 677501793905287305056882,\nonumber\\
& \qquad 5056942234147181110562433, 5056942234147181110562437, \nonumber\\
& \qquad 9436382674389074916067988, 9436382674389074916067990\}.
\end{align}
These sets satisfy
\begin{align} 
& [A_{0,20},A_{1,20},A_{2,20},A_{3,20},A_{4,20},A_{5,20}]^k\nonumber\\
& \quad=[B_{0,20},B_{1,20},B_{2,20},B_{3,20},B_{4,20},B_{5,20}]^k,\quad (k=1,2,3,4,5)
\end{align}
and
\begin{align}
\frac{A_{1,20}}{A_{5,20}}=\sin ^2(\frac{2\pi }{12}) \times 0.99999999999999999999999980225\ldots\nonumber\\[1mm]
\frac{A_{2,20}}{A_{5,20}}=\sin ^2(\frac{2\pi }{12}) \times 1.00000000000000000000000019775\ldots\nonumber\\[1mm]
\frac{A_{3,20}}{A_{5,20}}=\sin ^2(\frac{4\pi }{12}) \times 0.99999999999999999999999993408\ldots\nonumber\\[1mm]
\frac{A_{4,20}}{A_{5,20}}=\sin ^2(\frac{4\pi }{12}) \times 1.00000000000000000000000006592\ldots \nonumber\\[1mm]
\frac{B_{0,20}}{A_{5,20}}=\sin ^2(\frac{1\pi }{12}) \times 0.99999999999999999999999852398\ldots \nonumber\\[1mm]
\frac{B_{1,20}}{A_{5,20}}=\sin ^2(\frac{1\pi }{12}) \times 1.00000000000000000000000147601\ldots \nonumber\\[1mm]
\frac{B_{2,20}}{A_{5,20}}=\sin ^2(\frac{3\pi }{12}) \times 0.99999999999999999999999960450\ldots \nonumber\\[1mm]
\frac{B_{3,20}}{A_{5,20}}=\sin ^2(\frac{3\pi }{12}) \times 1.00000000000000000000000039550\ldots \nonumber\\[1mm]
\frac{B_{4,20}}{A_{5,20}}=\sin ^2(\frac{5\pi }{12}) \times 0.99999999999999999999999989402\ldots \nonumber\\[1mm]
\frac{B_{5,20}}{A_{5,20}}=\sin ^2(\frac{5\pi }{12}) \times 1.00000000000000000000000010597\ldots
\end{align}
If we change the initial value $u_0$ to 37, when we take $i$ to be $1, 2$, and $3$, we have
\begin{align} 
& [0,257,258,772,773,1030]^k=[68,70,513,517,960,962]^k,\\
& [0,3586,3587,10759,10760,14346]^k \nonumber\\
& \quad =[960,962,7171,7175,13384,13386]^k,\\
& [0,49953,49954,149860,149861,199814]^k \nonumber\\
& \quad= [13384,13386,99905,99909,186428,186430]^k,\\
& \qquad\qquad\qquad\qquad\qquad (k=1,2,3,4,5) \nonumber
\end{align}

\begin{example}
We take the initial value \(u_0\) as follows:
\begin{align} 
& u_0 = 6.
\end{align}
Next, we use the following recurrence formulas for \(u_i\) and \(v_i\):
\begin{align} 
& u_i=2 u_{i-1}+4 \sqrt{3 u_{i-1}^2-8},\\
& v_i=\frac{1}{2} \sqrt{3 u_i^2-8}.
\end{align}
and define the sets \(\{A_{0,i}, A_{1,i}, \dots, A_{5,i}\}\) and \(\{B_{0,i}, B_{1,i}, \dots, B_{5,i}\}\) as
\begin{align} 
& \left\{A_{0,i},A_{1,i},A_{2,i},A_{3,i},A_{4,i},A_{5,i}\right\}  \nonumber\\
& \quad=\left\{0,\frac{u_i-2}{4},\frac{u_i+2}{4},\frac{3 u_i-2}{4},\frac{3 u_i+2}{4},u_i\right\},\\
& \left\{B_{0,i},B_{1,i},B_{2,i},B_{3,i},B_{4,i},B_{5,i}\right\}  \nonumber\\
& \quad=\left\{\frac{u_i-v_i-1}{2},\frac{u_i-v_i+1}{2},\frac{u_i-2}{2},\frac{u_i+2}{2},\frac{u_i+v_i-1}{2},\frac{u_i+v_i+1}{2}\right\}  
\end{align}
Then, we have
\begin{align} 
\label{Parak12345}
& [A_{0,i},A_{1,i},A_{2,i},A_{3,i},A_{4,i},A_{5,i}]^k=[B_{0,i},B_{1,i},B_{2,i},B_{3,i},B_{4,i},B_{5,i}]^k,\nonumber\\
& \qquad\qquad\qquad\qquad\qquad\qquad\qquad\qquad (k=1,2,3,4,5)
\end{align}
Furthermore,
\begin{align}
& \lim_{i \to \infty} \frac{A_{t,i}}{A_{5,i}}=\lim_{i \to \infty} \frac{A_{t-1,i}}{A_{5,i}}=\sin ^2\left(\frac{t \pi }{12}\right),\quad (t=2,4)\\
& \lim_{i \to \infty} \frac{B_{t,i}}{A_{5,i}}=\lim_{i \to \infty} \frac{B_{t-1,i}}{A_{5,i}}=\sin ^2\left(\frac{t \pi }{12}\right),\quad (t=1,3,5)
\end{align}
\end{example}
For instance, when we take $i=78$, we obtain  \eqref{equationk12345}.

\subsection{Parametric Method for Ideal Prime Solutions}
For Example \ref{exampleparap1} to Example \ref{exampleparap4} below, the parametric solutions and prime solutions listed were all obtained by us in 2023 using a special method \cite{Chen2125}. The greatest challenge in constructing these parametric solution formulas is that, generally, parametric solutions found using traditional methods often fail to ensure that the numerical solutions are positive and odd, regardless of the parameter values.
\begin{example}\label{exampleparap1}
For arbitrary integers $m$ and $n$, we define the set $\{a_1, a_2, a_3, b_1, b_2, b_3\}$ as follows:
\begin{equation}
\begin{aligned} 
\label{para_prime23}
& a_1=668607 m^2 - 606430 m n +   135971 n^2, \\
& a_2=3 (331215 m^2 - 356278 m n + 95579 n^2), \\
& a_3=-140793 m^2 +   157762 m n - 43109 n^2, \\
& b_1=971067 m^2 - 1050038 m n +   282799 n^2, \\ 
& b_2=3 (237495 m^2 - 218822 m n + 50083 n^2), \\
& b_3=-59853 m^2 +   39050 m n - 3817 n^2.  
\end{aligned}
\end{equation}
Then, we have
\begin{align} 
[a_1,a_2,a_3]^k=[b_1,b_2,b_3]^k,\quad (k=2,3) \nonumber
\end{align}
\end{example}
It is easy to prove that when \( 0.50146073 < {m}/{n} < 0.50764144 \), $\{a_1, a_2, a_3\}$ and $\{b_1, b_2, b_3\}$ are all positive integers. For instance, when we take $m = 6256$ and $n = 12413$, we obtain the following ideal prime solution:
\begin{align}
&[2847598979, 7656462769, 10942931963]^k \nonumber\\
& \quad =[4233420431, 6368234377, 11313785591]^k, \quad (k=2,3)   \tag*{\eqref{k23s11313785591}}
\end{align}
In the range \( n < 250000 \), we found a total of 66 prime solutions. The smallest prime solution was obtained when \( m = 39505 \) and \( n = 78668 \):
\begin{align}
&[8417237, 104616559, 111462317]^k \nonumber\\
& \quad =[47946583, 69380393, 127776401]^k, \quad (k=2,3)   \tag*{\eqref{k23s127776401}}
\end{align}
\begin{example}\label{exampleparap2}
For arbitrary integers $p$ and $q$, we define the set $\{a_1, a_2, a_3, a_4, b_1,\\ b_2, b_3, b_4\}$ as follows:
\begin{equation}
\begin{aligned}
\label{para_prime125}
& a_1=-1217519 p^2 + 2953935 p q - 1788558 q^2, \\ 
& a_2= -1667247 p^2 + 3701547 p q - 2046074 q^2, \\ 
& a_3= 2133059 p^2 - 4771675 p q + 2667210 q^2, \\ 
& a_4=2389219 p^2 - 5555551 p q + 3235294 q^2, \\ 
& b_1=147749 p^2 - 464365 p q + 345210 q^2, \\ 
& b_2=-2349881 p^2 + 5410697 p q -  3112958 q^2, \\
& b_3= 2815693 p^2 - 6480825 p q + 3734094 q^2, \\
& b_4=1023951 p^2 - 2137251 p q + 1101526 q^2. 
\end{aligned}
\end{equation}
Then, we have
\begin{align} 
[a_1,a_2,a_3,a_4]^k=[b_1,b_2,b_3,b_4]^k,\quad (k=1,2,5) \nonumber
\end{align}
\end{example}
For instance, when we take $p=66587, q=56582$, we obtain the following ideal prime solution:
\begin{align}
& [8445859, 296931097, 393276643, 601843321]^k \nonumber\\ 
& \quad =[88587133, 136648549,521702047, 553559191]^k \nonumber\\ 
& \qquad\qquad\qquad\qquad\qquad (k=1,2,5)   \tag*{\eqref{k125s601843321}}
\end{align}

\begin{example}\label{exampleparap3}
For arbitrary integers $m$ and $n$, we define the set $\{a_1, a_2, a_3, a_4, b_1,\\ b_2, b_3, b_4\}$ as follows:
\begin{equation}
\begin{aligned}
\label{para_prime134} 
& a_1=-164556921941 m^2 + 51631570270 m n + 116267371171 n^2, \\ 
& a_2=22723347425 m^2 - 329613007462 m n +   323989597577 n^2, \\ 
& a_3=-117747885637 m^2 - 38028685666 m n +   179766670163 n^2,\\ 
& a_4=-25578019343 m^2 - 242877589094 m n +   265122550777 n^2, \\ 
& b_1=-117739209481 m^2 - 48889535050 m n +   163085083631 n^2, \\ 
& b_2=-24094365035 m^2 - 229091902142 m n +   277171885117 n^2, \\ 
& b_3=21242832959 m^2 - 336452079898 m n +   318757388759 n^2, \\ 
& b_4=-164568737939 m^2 + 55545805138 m n +   126131832181 n^2.
\end{aligned}
\end{equation}
Then, we have
\begin{align} 
[a_1,a_2,a_3,a_4]^k=[b_1,b_2,b_3,b_4]^k,\quad (k=1,3,4) \nonumber
\end{align}
\end{example}
For instance, when we take $n=33474,m=183515$, we obtain the following ideal prime solution:
\begin{align}
& [896501990958793919143, 2056330598071774290263,  \nonumber\\
& \qquad\qquad 3997663854855273138397, 5094457378727364512429]^k \nonumber \\
& \quad =[994237422911295892921, 1908177853245929320403, \nonumber\\
& \qquad\qquad 4082781516440229111169, 5059757030015751535739]^k \nonumber\\
& \qquad\qquad\qquad\qquad\qquad (k=1,3,4)   \tag*{\eqref{k134s5094457378727364512429}}
\end{align}

\begin{example}\label{exampleparap4}
For arbitrary integers $r,s$ and $t$, we define the set $\{p,q,m,n,u,v\}$ as follows:
\begin{align} 
& p = (r + s) (r^2 + s^2) (s + t) (s^2 + t^2) (r^2 + r t + t^2),  \nonumber\\
& q = (r^4 - t^4) (r^2 s^2 + r^2 s t + r s^2 t + r^2 t^2 + r s t^2 + s^2 t^2),  \nonumber\\
& m = (p + q )/(p - q),  \nonumber\\ 
& n = (r^4 m - s^4 m + s^4 - t^4 )/(r^4 - t^4),  \nonumber\\
& u = s^2 - t^2 + r^2 m^3 - s^2 m^3 - r^2 n^3 + t^2 n^3,  \nonumber\\
& v = s - t + r m^4 - s m^4 - r n^4 + t n^4,  \nonumber
\end{align}
Then, we have
\begin{align} 
\label{k15rst}
&  [-2 m u + r v, 2 n u - r v, -2 n u + t v]^k=[-2 m u + s v, -2 u + t v, 2 u - s v]^k, \nonumber\\
& \qquad\qquad\qquad\qquad\qquad\qquad\qquad\qquad (k=1,5) 
\end{align}
\end{example}
For instance, when we take $\{r, s, t\} =\{-562, -176, -9\}$, we obtain the \\following ideal prime solution:
\begin{align}
&[68049074651809716616587682328308420187688753216224749729, \nonumber \\
& \qquad\quad 148818734733829951795084131100190065917882950867486851411, \nonumber\\
& \qquad\quad 159627898794439357258113507241692664029601335467631785279]^k \nonumber \\
& \quad =[68691516188504321631164213329991813284950709035020363731, \nonumber \\
& \qquad\quad 146607580632312953765823649961087271271521143632892464049, \nonumber\\
& \qquad\quad 161196611359261750272797457379112065578701186883430558639]^k \nonumber\\
& \qquad\qquad\qquad\qquad\qquad\qquad\qquad\qquad (k=1,5)   \tag*{\eqref{k15p2}}
\end{align}

\clearpage

\section{Open Problems and Discussion}

\subsection{Open Problem}
In this section, we present several open problems arising from previous research, including the PTE problem, which has remained unsolved for over 270 years, and its generalization, the GPTE problem, introduced in this paper.

\begin{enumerate}[label=P\arabic*.]

    \item \textbf{Ideal Solution of PTE:} \textit{How to find ideal solutions of the PTE problem of degrees \( n = 10 \) and \( n \geq 12 \)?}
    \item \textbf{Ideal Non-Negative Integer Solution of GPTE:} \textit{For any given type of \( (k = k_1, k_2, \dots, k_n) \), does there always exist non-negative integer solutions that satisfy the Diophantine system:} 
\begin{align}
& \left[ a_{1}, a_{2}, \dots, a_{n+1} \right]^{k} = \left[ b_{1}, b_{2}, \dots, b_{n+1} \right]^{k}, \quad (k = k_1, k_2, \dots, k_n). \tag*{\eqref{GPTEk}}
\end{align}
While ideal non-negative integer solutions have been identified for 164 distinct types of \( (k = k_1, k_2, \dots, k_n) \) for system \eqref{GPTEk}, numerous other types \mbox{remain} unsolved. However, it has not yet been proven that any specific type of \( (k = k_1, k_2, \dots, k_n) \) lacks an ideal non-negative integer solution. For instance, the existence of ideal non-negative integer solutions for the types \( (k=5) \), \( (k=6) \), and \( (k=3,4) \) remains unknown. See \eqref{typek5} and \eqref{typek34}.

    \item \textbf{Ideal Integer Solution of GPTE:} \textit{When \( m < n \), for certain types of \( (h = h_1, h_2, \dots, h_n) \), does the following system have integer solutions?}
\begin{align}
\left[ a_{1}, a_{2}, \dots, a_{m} \right]^{h} = \left[ b_{1}, b_{2}, \dots, b_{m} \right]^{h}, \quad (h = h_1, h_2, \dots, h_n).\tag*{\eqref{GPTE}}
\end{align}
To date, when \( m = n \), a total of 65 distinct types of ideal integer solutions have been discovered for system \eqref{GPTE}. Meanwhile, it has been proven that certain types, such as $(h = 1, 2, \dots, n)$ and $(h = 1, 2, \dots, n-1, n+2)$, do not admit ideal integer solutions when \( m = n \). See \eqref{h123n} and \eqref{h1251236}.\\[1mm]
When \( m < n \), no integer solutions for system \eqref{GPTE} have been found so far. However, we have discovered that trigonometric solutions exist for \eqref{GPTE} when \( m < n \). For instance, the trigonometric solution \eqref{h123468Tri} provides an example with \( m = 5 \) and \( n = 6 \), while the trigonometric solution chain \eqref{h1to20Tri} provides an example with \( m = 11 \) and \( n = 15 \).
    \item \textbf{Ideal Prime Solution of GPTE:} \textit{ When all \( k > 0 \), if there exist ideal non-negative integer solutions for \eqref{GPTEk}, whether there must also exist ideal prime solutions remains an open question.}\\[1mm]
To date, ideal non-negative integer solutions have been found for 42 types of system \eqref{GPTEk} with all \( k > 0 \). Among these 42 types, ideal prime solutions have been identified for 27 types.

    \item \textbf{Ideal Non-Negative Integer Chains of GPTE:} \textit{If there exist ideal non-negative integer solutions for \eqref{GPTEk}, whether there must also exist ideal non-negative integer chains, and whether the length of these chains can be \mbox{arbitrarily} long, remains an open question.}\\[1mm]
To date, ideal non-negative integer solutions have been found for 164 types of system \eqref{GPTEk}. Among these 164 types, ideal non-negative integer chains have been identified for 30 types, and for 15 of these types, it has been proven that ideal non-negative integer chains of arbitrary length exist. \mbox{Additionally}, by applying \refIdentity{identityT4}, we can obtain ideal trigonometric chains of PTE of any degree and any length.

    \item \textbf{The Generalized Girard-Newton Identities:} \textit{How to provide a complete proof of the Equivalent Form of the Girard-Newton Identities (\refIdentity{identity2}), as well as complete proofs of the three generalizations of the Girard-Newton Identities (\refIdentity{identity_GNI1}, \refIdentity{identity_GNI2}, \refIdentity{identity_GNI3})?}\\[1mm]
At present, the author can only prove the correctness of these identities within a limited scope.

\end{enumerate}

\subsection{Discussion}
To address the aforementioned open problems, particularly P1, P2, and P5, we propose the following three approaches:

\begin{enumerate}[label=D\arabic*.]

\item \textbf{Discovering Universal Novel Identities:} For the open problem P1, namely the ideal solutions of the PTE problem, the most ideal approach would be to discover identities that are applicable to any degree and any length, and capable of generating all possible solutions. Similarly, this approach is also the most effective for completely solving P2 and P5, particularly for the three series of GPTE types in Identity \ref{identityC2} to Identity  \ref{identityC4}, which all have a constant \( C \) analogous to that of the PTE problem. \\[1mm]
Finding such identities is highly challenging, yet not impossible. In fact, we have already made substantial progress in this direction and have achieved encouraging results. For instance, by utilizing Identity \ref{identityT4}, we can construct ideal trigonometric chains of PTE for any degree and any length. Meanwhile, all identities presented in Chapter 2 to 6 (i.e., Identity \ref{identity1} to Identity \ref{identityT5}) are universally applicable to any degree. These identities are derived through constructive methods and exhibit a remarkably concise and elegant form.

\item \textbf{Developing Comprehensive Computer Algorithm:} The second approach for solving P1 and P2 is to develop and refine a comprehensive, computer-based exhaustive search algorithm. Such an algorithm should be capable of finding ideal solutions for the PTE and GPTE problems of various types, imposing the most effective constraints on the search range, and achieving the minimum number of search iterations. \\[1mm]
In fact, we have essentially developed such a computer search algorithm. For example, based on Identity \ref{identity_GNI2}, Identity \ref{identity_GNI3} and Identity \ref{identityT1}, we can achieve the minimum number of search iterations (as referenced in Chapter 2 and Chapter 6). Further, applying Conjecture \ref{conjecture_Interlacing} to Conjecture \ref{conjecture_threshold}, we can impose the most effective constraints on the search range. Moreover, by leveraging the corollaries relative to the constant \( C \) in Identity \ref{identityC2} to Identity \ref{identityC4}, we can effectively accelerate the search speed. Utilizing these algorithms, we have found a large number of new solutions for the PTE and GPTE types. All these numerical solutions were obtained using a normal personal computer. Moving forward, by enhancing computational capability, including utilizing faster GPUs and employing distributed computing, it is anticipated that more new results can be achieved.

\item \textbf{Constructing Specialized Methods:} The third approach for solving P1 is to construct specialized methods for a specific degree. This approach has been successful in finding ideal symmetric solutions for degrees 9 and 11 by employing methods that involve elliptic curves. However, such specialized methods encounter significant challenges when applied to degrees 10, 12, and higher.\\[1mm]
For the various types of the GPTE problem, constructing specialized methods remains an effective approach, as illustrated by the examples listed in Chapter 7.
\end{enumerate}
Regardless of whether approach D1, D2, or D3 is pursued, leveraging artificial intelligence (AI) is expected to play a pivotal role in achieving a breakthrough in solving the PTE and GPTE problems. For instance, AI can be utilized to discover new identities or optimize search algorithms. This provides significant opportunities for tackling a mathematical problem that has persisted for over 270 years.
\clearpage

\section*{Acknowledgments}
\addcontentsline{toc}{section}{Acknowledgments}

I first encountered the Prouhet-Tarry-Escott problem in July 1985, right after I had completed my first year of high school. It was through two articles written by Professor Tan Xiangbai, a renowned Chinese science communicator, that I was introduced to this fascinating mathematical problem. I am deeply grateful for his enlightening guidance.\\[2mm]
As an independent mathematics researcher, I would like to express my profound gratitude to Professor Peter Borwein, Professor T. N. Sinha, Professor Trevor D. Wooley, and Professor Chris J. Smyth for their generous sharing of copies of their publications and for providing me with invaluable guidance through correspondence during the period from 1995 to 2001.\\[2mm]
My sincere thanks also go to Professor Jarosław Wróblewski, Mr. Ajai Choudhry, Mr. Jean-Charles Meyrignac, Mr. Carlos Rivera, Academician Xi Nanhua, and Professor Michael J. Mossinghoff for their encouragement and valuable feedback on my preliminary results through email exchanges since 1999.\\[2mm]
I am particularly grateful to my colleague Mr. Ye Xiaofeng at Seekway Innovations Technology Co., Ltd. for his invaluable assistance in computational optimization.\\[2mm]
To my family and friends, I extend my heartfelt appreciation for their steadfast support throughout my forty-year-long purely amateur research endeavor.\\[2mm]
Finally, I have chosen June 12, 2025 as the arXiv submission date of this paper in honor of Professor Hua Loo-Keng, the pioneering Chinese mathematician who first investigated the PTE problem in China. His classic work \textit{Introduction to Number Theory} served as my only professional reference on this subject during my formative years, and his fundamental contributions have inspired my four decades of persistent exploration in this field.

\clearpage
\phantomsection
\addcontentsline{toc}{section}{Appendix\quad Collection of ideal numerical solutions}
\appendix
\begin{appendices}
\label{appendixA}
\section{Ideal non-negative integer solution of GPTE}
So far, a total of 164 distinct types of ideal non-negative integer solutions have been discovered for various combinations of \( (k = k_1, k_2, \dots, k_n) \). These solution types can be categorized into five groups with counts of 42, 42, 24, 23, and 33, respectively, summing to the total of 164. 
\label{AppendixA1}
\subsection{GPTE with all k>0}
Ideal non-negative integer solutions of the GPTE problem have been identified for 42 types with all \( k > 0 \), including 10 known types of PTE.
\subsubsection{\;\;(k$\;=\;$1)}
\label{k1}
\begin{notation}
\end{notation}
$\bullet$ Type: $(k=1)$ \qquad $\bullet$ Abbreviation: $[$k1$]$
\begin{idealsolu}
\end{idealsolu}
$\bullet$ Smallest solutions:
\begin{align}
\label{k1s2}
& [ 0, 2 ]^k = [ 1, 1 ]^k\\
\label{k1s3}
& [ 1,3 ]^k = [2,2 ]^k
\end{align}
\begin{idealchain}
\end{idealchain}
$\bullet$ It is obvious that there are solution chains of any length. 
\begin{align}
\label{k1s9}
& [ 0, 9 ] ^k= [ 1, 8 ] ^k= [ 2, 7 ]^k = [ 3, 6 ]^k = [4,5 ]^k
\end{align}
\begin{primesolu}
\end{primesolu}
$\bullet$ Prime solutions:
\begin{align}
& [ 3, 7 ]^k = [ 5, 5 ]^k\\
& [ 3, 13 ]^k = [ 5, 11 ]^k
\end{align}
\indent
$\bullet$ Prime solution chains:
\begin{align}
\label{k1s53}
& [ 7, 53 ]^k = [ 13, 47 ]^k = [ 17, 43 ]^k = [ 19, 41 ] ^k= [ 23, 37 ]^k = [ 29, 31 ]^k
\end{align}
\begin{mirrortype}
\end{mirrortype}
$\bullet$ \ref{kn1}\quad $(k=-1)$ 
\begin{relatedtype}
\end{relatedtype}
$\bullet$ \ref{r1}\quad $(r=1)$
\\
\subsubsection{\;\;(k$\;=\;$2)}
\label{k2}
\begin{notation}
\end{notation}
$\bullet$ Type: $(k=2)$ \qquad $\bullet$ Abbreviation: $[$k2$]$
\begin{idealsolu}
\end{idealsolu}
$\bullet$ Smallest solutions:
\begin{align}
\label{k2s5}
& [ 0, 5 ]^k = [ 3, 4 ]^k \\
\label{k2s11}
& [ 2, 11 ]^k = [ 5, 10 ]^k
\end{align}
\begin{idealchain}
\end{idealchain}
$\bullet$ It has been proved that there are solution chains for any length \cite{Hardy38}. A.Gloden gave method on how to obtained solution chains in 1940's \cite{Gloden44}. \mbox{Numerical} examples by A.Gloden's method:
\begin{align}
\label{k2s91}
& [ 13, 91 ]^k = [ 23, 89 ]^k = [ 35, 85 ]^k = [ 47, 79 ]^k = [ 65, 65 ]^k
\end{align}
\begin{primesolu}
\end{primesolu}
$\bullet$ First known prime solution, by Jean-Charles Meyrignac in 2000 \cite{JCM}:
\begin{align}
& [ 7, 17 ]^k = [ 13, 13 ]^k
\end{align}
\indent
$\bullet$ Prime solution chains, by Chen Shuwen in 2016 \cite{Chen23}:
\begin{align}
& [1, 47 ]^k = [ 19, 43 ]^k = [ 23, 41 ]^k = [ 29, 37 ]^k \\
\label{k2s281}
& [ 53, 281 ]^k = [ 71, 277 ]^k = [ 97, 269 ] ^k \nonumber \\
& \quad = [ 137, 251 ]^k = [ 157, 239 ]^k = [ 193, 211]^k
\end{align}

\begin{mirrortype}
\end{mirrortype}
$\bullet$ \ref{kn2}\quad $(k=-2)$ 
\begin{relatedtype}
\end{relatedtype}
$\bullet$ \ref{r2}\quad $(r=2)$
\\
\subsubsection{\;\;(k$\;=\;$3)}
\label{k3}
\begin{notation}
\end{notation}
$\bullet$ Type: $(k=3)$ \qquad $\bullet$ Abbreviation: $[$k3$]$
\begin{idealsolu}
\end{idealsolu}
$\bullet$ Smallest solutions:
\begin{align}
\label{k3s12}
& [ 1, 12 ]^k = [ 9, 10 ]^k\\
& [ 2, 16 ]^k = [ 9, 15 ]^k\\
& [ 10, 27 ]^k = [19, 24 ]^k\\
& [ 2, 34 ]^k = [ 15, 33 ]^k
\end{align}
\indent
$\bullet$ Several parametric solutions have been found by Euler, N.Elkies, Ramanujan, Titus Piezas, and others \cite[p.370-375]{Piezas09}.\\[1mm]
\indent
$\bullet$ Jarosław Wróblewski obtained 4330853 solutions below 1000000 by exhaustive computer search in 2006 \cite{JW955}. 
\begin{idealchain}
\end{idealchain}
$\bullet$ First known solution chain, by J.Leech in 1957 \cite{Leech1957} \cite{Silverman1993}:
\begin{align}
& [167, 436 ]^k = [ 228, 423 ]^k = [ 255, 414 ]^k
\end{align}
\indent
$\bullet$  Solution chains of length 4, by E.Rosenstiel, J.A.Dardis and C.R.Rosenstiel in 1991 \cite{Rosenstiel1991}:
\begin{align}
\label{k3s19083}
& [ 2421, 19083 ]^k = [ 5436, 18948 ]^k \nonumber \\
& \quad = [ 10200, 18072 ]^k = [ 13322, 16630 ]^k  \\
\label{k3s23237}
& [ 4275, 23237 ]^k = [ 7068, 23066 ]^k  \nonumber \\
& \quad = [10362, 22580 ]^k = [ 12939, 21869 ]^k
\end{align}
\indent
$\bullet$  Solution chains of length 12, by Christian Boyer and Jarosław Wróblewski	in 2008 \cite{Boyer08} \cite{Boyer11} :
\begin{align}
& [ 21721970100126962640,18038772102965878680]^k \nonumber \\
& \quad=[ 21987454277272575000,17640345906751691760]^k \nonumber \\
& \quad= [ 22267760149437058620,17187779557568021220]^k \nonumber \\
& \quad= [ 22351892062348779840,17044847163758657880]^k \nonumber \\
& \quad= [ 24297225342129820080,12108179966187254040]^k \nonumber \\
& \quad= [ 24332594138439574260,11963837040706905900]^k \nonumber \\
& \quad= [ 24403184190268788996,11663515767522623004]^k \nonumber \\
& \quad= [ 25000720604144916120,7898709837472488720]^k \nonumber \\
& \quad= [ 25060249943031303000,7248918034130441760]^k \nonumber \\
& \quad= [ 25258892211726742296,1544100565125094704]^k \nonumber \\
& \quad= [ 25260575914339118080,771180546485662040]^k \nonumber \\
& \quad= [ 25260802402509788751,292667080168803249]^k 
\end{align}
\begin{primesolu}
\end{primesolu}
$\bullet$ First known prime solution, smallest prime solution, by Jean-Charles Meyrignac in 2000 \cite{JCM}:
\begin{align}
& [ 61, 1823 ]^k = [ 1049, 1699 ]^k
\end{align}
\indent
$\bullet$ 
Among all 4330853 solutions below 1000000 obtained by Jarosław Wróblewski, Chen Shuwen confirmed in 2023 that there are 816 prime solutions. \mbox{Numerical} examples:
\begin{align}
& [31,1867]^k = [397,1861]^k\\
& [593,2333]^k = [1787,1931]^k\\
& [282019,998839]^k = [686131,886183]^k\\
& [509441,999979]^k = [602513,970267]^k
\end{align}
\indent
$\bullet$ In 2023, Chen Shuwen found a quick parametric method to generate \mbox{thousands} of large prime solutions. Numerical examples:
\begin{align}
& [381390944953747901,2219726926842666805981]^k \nonumber \\
& \quad = [1749088019600545177919,1774328041319104499827]^k \\
& [6795554777350361993,28855935673607956539433]^k \nonumber \\
& \quad = [22738395816294430043267,23065215016959104078311]^k
\end{align}
\begin{mirrortype}
\end{mirrortype}
$\bullet$ \ref{kn3}\quad $(k=-3)$
\begin{relatedtype}
\end{relatedtype}
$\bullet$ \ref{k03}\quad $(k=0,3)$
\\
\subsubsection{\;\;(k$\;=\;$4)}
\label{k4}
\begin{notation}
\end{notation}
$\bullet$ Type: $(k=4)$ \qquad $\bullet$ Abbreviation: $[$k4$]$
\begin{idealsolu}
\end{idealsolu}
$\bullet$ Euler gave a two-parameter solution in 1772 \cite{Dickson52}, Numerical example:
\begin{align}
\label{k4s158}
& [ 59, 158 ]^k = [ 133, 134]^k
\end{align}
\indent
$\bullet$ L.J.Lander, T.R.Parkin, and J.L.Selfridge found 46 solutions in 1966-1967 \cite{Lander66} \cite{Lander67} by computer search. Numerical examples:
\begin{align}
& [ 193, 292 ]^k = [ 256, 257 ]^k\\
& [ 271, 502 ]^k = [ 298, 497 ]^k\\
& [ 103, 542 ]^k = [ 359, 514 ]^k\\
& [ 222, 631 ]^k = [ 503, 558 ]^k
\end{align} 
\indent
$\bullet$ 
In 1982, A.J. Zajta discussed the more important solution methods for this equation and presented a list of 218 numerical solutions \cite{Zajta83}. This list contained all known primitive nontrivial solutions within the range up to $10^6$.\\[1mm]
\indent
$\bullet$
Daniel J. Bernstein found all 516 solutions below $10^{6}$ by a fast computer search method in 2001 \cite{Bernstein01}.\vspace{1ex} \\
\indent
$\bullet$ Jarosław Wróblewski obtained 11089 solutions below $10^{14}$ by exhaustive computer search in 2006 \cite{JW06}. The largest two solutions are:
\begin{align}
& [15345135599667, 99829023021259]^k \nonumber \\
&\quad = [79112592557673, 88084675559689]^k\\
& [4630943537570, 99918878943347]^k\nonumber \\
&\quad =  [51913486805005, 98046841752722]^k
\end{align} 
\begin{primesolu}
\end{primesolu}
$\bullet$ First known prime solution, noticed by Jean-Charles Meyrignac in 2000 \cite{JCM}:
\begin{align}
\label{k4s239}
& [ 7, 239 ]^k = [ 157, 227 ]^k
\end{align}
\indent
$\bullet$ Among all 11089 solutions below $10^{14}$ obtained by Jarosław Wróblewski  \cite{JW06}, Chen Shuwen confirmed in 2023 that there are only two prime solutions. The second known prime solution is:
\begin{align}
& [40351,62047]^k = [46747,59693]^k
\end{align}
\begin{mirrortype}
\end{mirrortype}
$\bullet$ \ref{kn4}\quad $(k=-4)$
\begin{relatedtype}
\end{relatedtype}
$\bullet$ \ref{k04}\quad $(k=0,4)$
\\
\subsubsection{\;\;(k$\;=\;$1, 2)}
\label{k12}
\begin{notation}
\end{notation}
$\bullet$ Type: $(k=1,2)$ \qquad $\bullet$ Abbreviation: $[$k12$]$
\begin{idealsolu}
\end{idealsolu}
$\bullet$ During 1750-1751, Goldbach and Euler had studied this system \cite{Dorwart37,Dickson52}. \vspace{1ex}\\
\indent
$\bullet$ The complete ideal solution has been given by L.E.Dickson in 1910's \cite[pp.52]{Dickson57}. Smallest solutions:
\begin{align}
\label{k12s4}
& [ 0, 3, 3 ]^k = [ 1, 1, 4 ]^k\\
\label{k12s6}
& [ 0, 4, 5 ]^k = [ 1, 2, 6 ]^k\\
& [ 0, 7, 7 ]^k = [ 1, 4, 9 ]^k
\end{align}
\begin{idealchain}
\end{idealchain}
$\bullet$ It has been proved that there are solution chains for any length \cite{Hardy38}. A.Gloden gave method on how to obtained symmetric solution chains in 1940's \cite{Gloden44}.\mbox{Numerical} example:
\begin{align}
\label{k12s22}
& [0, 16, 17 ] ^k = [ 1, 12, 20 ]^k = [ 2, 10, 21 ]^k= [ 5, 6, 22 ]^k\\
& [ 0, 71, 73 ]^k = [ 1, 63, 80 ]^k = [ 3, 56, 85 ]^k = [ 5, 51, 88]^k \nonumber \\
& \quad = [ 8, 45, 91]^k = [ 11, 40, 93 ]^k = [ 16, 33, 95 ]^k = [23, 25, 96 ]^k
\end{align}
\begin{primesolu}
\end{primesolu}
$\bullet$ First known prime solutions, by Albert H. Beiler \cite{Beiler1964,Rivera65}:
\begin{align}
&[43, 61, 67 ]^k = [ 47, 53, 71]^k
\end{align}
\indent
$\bullet$ Non-symmetric prime solution chains, by Chen Shuwen in 2016 \cite{Chen23}:
\begin{align}
\label{k12s137}
&[11, 107, 113 ]^k = [ 17, 83, 131 ]^k = [ 23, 71, 137]^k
\end{align}
\indent
$\bullet$ Prime solution chains of length 6, by Chen Shuwen in 2023:
\begin{align}
&[916879, 1569937, 1588333]^k = [926077, 1496353, 1652719]^k \nonumber \\
& \quad = [944473, 1431967, 1698709]^k = [962869, 1385977, 1726303]^k \nonumber \\
& \quad = [990463, 1330789, 1753897]^k = [1128433, 1146829, 1799887]^k
\end{align}
\begin{mirrortype}
\end{mirrortype}
$\bullet$ \ref{kn12}\quad $(k=-2,-1)$
\begin{relatedtype}
\end{relatedtype}
$\bullet$ \ref{h124}\quad $(h=1,2,4)$
\\
\subsubsection{\;\;(k$\;=\;$1, 3)}
\label{k13}
\begin{notation}
\end{notation}
$\bullet$ Type: $(k=1,3)$ \qquad $\bullet$ Abbreviation: $[$k13$]$
\begin{idealsolu}
\end{idealsolu}
$\bullet$ A.Moessner gave parameter solutions of this type in 1939 \cite{Moessner39}. Numerical exampls are:
\begin{align}
\label{k13s9}
& [0, 7, 8 ]^k = [1, 5, 9]^k\\
\label{k13s10}
& [0, 7, 9 ]^k = [2, 4, 10]^k\\
& [12, 23, 28 ]^k = [13, 21, 29]^k
\end{align}
\indent
$\bullet$ Smallest solution, by computer search:
\begin{align}
\label{k13s6}
& [1, 5, 5 ]^k = [ 2, 3, 6]^k
\end{align}
\indent
$\bullet$ Ideal solution with additional condition $a_1 b_1=a_2 b_2$, which leads to ideal solution of $(k=0,1,3)$, by Chen Shuwen in 2001:
\begin{align}
\label{k13s12}
& [5,10,11]^k = [6,8,12]^k
\end{align}
\begin{idealchain}
\end{idealchain}
$\bullet$ First known solution chains, by Chen Shuwen in 1995 \cite{Chen01,Chen23}:
\begin{align}
\label{k13s68}
& [ 2, 52, 58 ] ^k = [ 4, 46, 62 ] ^k = [ 13, 32, 67 ] ^k = [ 22, 22, 68 ] ^k\\
& [ 5, 58, 70 ] ^k = [ 7, 53, 73 ]  ^k= [ 13, 43, 77 ] ^k = [ 21, 33, 79] ^k
\end{align}
\indent
$\bullet$ Solution chains of length 5 to length 65, by Jarosław Wróblewski in 2001:
\begin{align}
& [ 10, 214, 215 ]^k = [ 19, 179, 241 ]^k = [ 43, 139, 257 ]^k \nonumber \\
& \quad = [ 49, 131, 259 ]^k = [87, 88, 264]^k \\
& [ 89, 361, 367 ] ^k= [ 97, 323, 397 ]^k = [ 115, 285, 417 ] ^k \nonumber \\
& \quad = [ 133, 257, 427 ] ^k= [ 152, 232, 433 ] ^k= [ 187, 193, 437 ]^k \\
& [ 16, 624, 699 ] ^k= [ 19, 611, 709 ]^k = [ 39, 555, 745 ] ^k \nonumber \\
& \quad = [ 79, 481, 779 ] ^k= [ 107, 439, 793 ] ^k= [ 169, 359, 811 ] ^k \nonumber \\
& \quad = [ 187, 338, 814 ]^k = [ 259, 261, 819 ]^k
\end{align}
\begin{primesolu}
\end{primesolu}
$\bullet$ First known prime solution, by Chen Shuwen in 2016:
\begin{align}
&[19, 179, 241 ]^k = [ 43, 139, 257]^k
\end{align}
\indent
$\bullet$ Smallest prime solutions, by Chen Shuwen in 2023:
\begin{align}
&[13,43,47]^k = [19,31,53]^k \\
&[19,47,53]^k = [29,31,59]^k 
\end{align}
\indent
$\bullet$ Prime solution chains of length 4, based on \eqref{identity13}, by Chen Shuwen in 2023:
\begin{align}
\label{k13s967}
&[83,757,827]^k = [107,677,883]^k \nonumber \\
& \quad = [197,523,947]^k = [281,419,967]^k\\
&[37,1451,1783]^k = [47,1423,1801]^k \nonumber \\
& \quad = [151,1213,1907]^k = [233,1087,1951]^k\\
&[229,1531,1709]^k = [239,1489,1741]^k \nonumber \\
& \quad = [499,1021,1949]^k = [619,877,1973]^k\\
&[463,1627,1733]^k = [523,1429,1871]^k \nonumber \\
& \quad = [743,1087,1993]^k = [773,1051,1999]^k
\end{align}
\begin{mirrortype}
\end{mirrortype}
$\bullet$ \ref{kn13}\quad $(k=-3,-1)$
\begin{relatedtype}
\end{relatedtype}
$\bullet$ \ref{r13}\quad $(r=1,3)$
\\
\subsubsection{\;\;(k$\;=\;$1, 4)}
\label{k14}
\begin{notation}
\end{notation}
$\bullet$ Type: $(k=1,4)$ \qquad $\bullet$ Abbreviation: $[$k14$]$
\begin{idealsolu}
\end{idealsolu}
$\bullet$ First known solutions, smallest solutions, based on computer search, by Chen Shuwen in 1995 \cite{Chen01},\cite{Chen23}:
\begin{align}
\label{k14s39}
&[ 3, 25, 38 ]^k = [ 7, 20, 39 ]^k\\
&[ 15, 36, 39 ]^k = [ 22, 25, 43 ]^k\\
&[ 4, 36, 43 ]^k = [ 9, 28, 46 ]^k\\
\label{k14s48}
&[ 4, 41, 42 ]^k = [ 11, 28, 48 ]^k
\end{align}
\indent
$\bullet$ Chen Shuwen obtained a parametric solution in 1997. Numerical examples:
\begin{align}
& [ 31, 199, 244 ]^k = [ 73, 139, 262 ]^k\\
& [ 63, 163, 290 ]^k = [ 75, 149, 292 ]^k\\
& [ 32, 294, 463 ]^k = [ 116, 196, 477 ]^k
\end{align}
\begin{idealchain}
\end{idealchain}
$\bullet$ First known solution chains, based on computer search, by Chen Shuwen in 1997:
\begin{align}
\label{k14s244}
&[24, 201, 216 ] ^k= [ 66, 132, 243 ]^k = [ 73, 124, 244]^k
\end{align}
\begin{primesolu}
\end{primesolu}
$\bullet$ First known prime solutions, based on computer search, by Chen Shuwen in 2016:
\begin{align}
&[89, 811, 997 ]^k = [ 251, 577, 1069]^k\\
&[337, 1609, 1637 ]^k= [ 613, 1097, 1873]^k
\end{align}
\indent
$\bullet$ Large prime solutions based on a quick parametric method, by Chen Shuwen in 2023:
\begin{align}
& [2173207,7556539,11161921]^k=[6086767,11470099,3334801] \\
& [2637394687, 18249058669, 19043961451]^k\nonumber\\
& \quad =[6000536077, 12317678671, 21612200059]^k \\
& [11146737079, 143354028799, 162246373489]^k \nonumber\\
& \quad =[21719669113, 122208164731, 172819305523]
\end{align}
\begin{mirrortype}
\end{mirrortype}
$\bullet$ \ref{kn14}\quad $(k=-4,-1)$
\begin{relatedtype}
\end{relatedtype}
$\bullet$ \ref{r14}\quad $(r=1,4)$
\\
\subsubsection{\;\;(k$\;=\;$1, 5)}
\label{k15}
\begin{notation}
\end{notation}
$\bullet$ Type: $(k=1,5)$ \qquad $\bullet$ Abbreviation: $[$k15$]$
\begin{idealsolu}
\end{idealsolu}
$\bullet$ First known solution, by A.Moessner in 1939 \cite{Moessner39}:
\begin{align}
&[39, 92, 100 ] = [ 49, 75, 107]^k
\end{align}
\indent
$\bullet$ Parameter solutions were obtained by Moessner, by H.Swinnerton-Dyer and by L.J.Lander in 1968 \cite{Lander1968}.\vspace{1ex}\\
\indent
$\bullet$ Smallest solutions, based on computer search of $a_1^5+a_2^5+a_3^5=b_1^5+b_2^5+b_3^5$, by L.J.Lander, T.R.Parkin and J.L.Selfridge in 1967 \cite{Lander67}:
\begin{align}
\label{k15s66}
&[ 13, 51, 64 ]^k = [ 18, 44, 66 ]^k\\
&[ 3, 54, 62 ]^k = [ 24, 28, 67 ]^k\\
&[ 8, 62, 68 ] ^k= [ 21, 43, 74 ]^k\\
&[ 53, 72, 81 ]^k = [ 56, 67, 83 ]^k
\end{align}
\indent
$\bullet$ Seiji Tomita provided a collection of all 1824 solutions less than 10000 in 2010 \cite[No.70]{Tomita21}. \\[1mm]
\indent
$\bullet$ Duncan Moore obtained some new results by computer search in range of 17700. \cite[pp.492-493]{Piezas09}.\\[1mm]
\indent
$\bullet$ Solution satisfies $a_1/a_2=b_1/b_3$, which leads to the ideal solution \eqref{k015s10584} of $(k=0,1,5)$, discovered by Chen Shuwen in 2023:
\begin{align}
\label{k15s756}
&[145,406,751]^k = [270,276,756]^k
\end{align}
\indent
$\bullet$ Chen Shuwen found a three-parameter solution \eqref{k15rst} in 2023. Numerical example:
\begin{align}
&[600387, 6945255, 7079961]^k = [1312605, 5514371, 7798627]^k
\end{align}
\begin{primesolu}
\end{primesolu}
$\bullet$ First known prime solution, based on computer search, by Jarosław Wróblewski in 2002 \cite{JCM}:
\begin{align}
\label{k15s6379}
&[1777,5003,6089]^k=[2657,3833,6379]^k
\end{align}
\indent
$\bullet$ Second known prime solution, based on \eqref{k15rst}, by Chen Shuwen in 2023 \cite{Chen2125}:
\begin{align}
\label{k15p2}
&[68049074651809716616587682328308420187688753216224749729, \nonumber \\
& \qquad\quad 148818734733829951795084131100190065917882950867486851411, \nonumber\\
& \qquad\quad 159627898794439357258113507241692664029601335467631785279]^k \nonumber \\
& \quad =[68691516188504321631164213329991813284950709035020363731, \nonumber \\
& \qquad\quad 146607580632312953765823649961087271271521143632892464049, \nonumber\\
& \qquad\quad 161196611359261750272797457379112065578701186883430558639]^k
\end{align}
\begin{mirrortype}
\end{mirrortype}
$\bullet$ \ref{kn15}\quad $(k=-5,-1)$
\begin{relatedtype}
\end{relatedtype}
$\bullet$ \ref{k015}\quad $(k=0,1,5)$
\\
\subsubsection{\;\;(k$\;=\;$2, 3)}
\label{k23}
\begin{notation}
\end{notation}
$\bullet$ Type: $(k=2,3)$ \qquad $\bullet$ Abbreviation: $[$k23$]$
\begin{idealsolu}
\end{idealsolu}
$\bullet$ Christian Goldbach gave a two-parametric solution 300$\small{+}$ years ago \cite[p.296,395]{Piezas09}. However, numerical solutions by this method always include negative integers, which are not ideal solutions. Numerical example: [-4,-3,11]=[-1,8,9].\vspace{1ex} \\
\indent
$\bullet$ A. Gloden obtained a two-parameter solution in 1943 \cite{Gloden49}. There are three numerical solutions less than 1000 by his method:
\begin{align}
&[1,336,385 ]^k = [193,193,432]^k\\
&[31,540,571]^k = [271,331,660]^k\\
&[73,792,793]^k = [361,505,936]^k
\end{align}
\indent
$\bullet$ A.Moessner found two parametric solutions in 1948 \cite{Gloden49}. The first one in fact is same as Goldbach's solution. The second one is a three-parameter solution, which gives 18 numerical solutions less than 1000. Smallest example found by Moessner's method:
\begin{align}
&[16,270,285]^k = [141,160,330]^k
\end{align}
\indent
$\bullet$ A.Gloden found another two parameter solution in 1949 \cite{Gloden49}. Numerical example:
\begin{align}
&[2251,35478,37243]^k=[19747,19747,43254]^k
\end{align}
\indent
$\bullet$ T.N.Sinha gave a three-parameter solution in 1984 \cite[pp.92-93]{Sinha84}. Numerical example:
\begin{align}
&[ 1, 432, 487 ]^k = [ 163, 325, 540 ]^k
\end{align}
\indent
$\bullet$ Ajai Choudhry obtained a six-parameter solution in 1991 \cite{Choudhry91}. Numerical example:
\begin{align}
&[2812,10154,17499]^k=[6700,7131,17930]^k
\end{align}
\indent
$\bullet$ Smallest solutions, based on computer search, by Chen Shuwen in 1995:
\begin{align}
\label{k23s64}
&[0, 37, 62 ]^k = [ 21, 26, 64]^k\\
&[2, 45, 62 ]^k = [ 26, 29, 66]^k\\
\label{k23s73}
&[1, 54, 69 ]^k = [ 18, 45, 73]^k\\
&[3, 58, 69 ]^k = [ 22, 45, 75]^k
\end{align}
\indent
$\bullet$ Ajai Choudhry and Jarosław Wróblewski found several numerical integer solution chains of length 3 by computer search and obtained a parametric solution chains of length 3 in 2011 \cite{Choudhry14}. However, numerical solution chains found by this method always have negative integers.
\begin{primesolu}
\end{primesolu}
$\bullet$ First known prime solutions, based on new parametric method \eqref{para_prime23}, by Chen Shuwen in 2023 \cite{Chen2125}:
\begin{align}
\label{k23s127776401}
&[8417237, 104616559, 111462317]^k \nonumber\\
& \quad =[47946583, 69380393, 127776401]^k \\
\label{k23s11313785591}
&[2847598979, 7656462769, 10942931963]^k \nonumber\\
& \quad =[4233420431, 6368234377, 11313785591]^k \\
\label{k23s35159223563}
&[4021933631,27430626697,31555873031]^k \nonumber\\
& \quad =[14913724849,17487085619,35159223563]^k
\end{align}
\begin{mirrortype}
\end{mirrortype}
$\bullet$ \ref{kn23}\quad $(k=-3,-2)$ 
\begin{relatedtype}
\end{relatedtype}
$\bullet$ \ref{r23}\quad $(r=2,3)$ 
\\
\subsubsection{\;\;(k$\;=\;$2, 4)}
\label{k24}
\begin{notation}
\end{notation}
$\bullet$ Type: $(k=2,4)$ \qquad $\bullet$ Abbreviation: $[$k24$]$
\begin{idealsolu}
\end{idealsolu}
$\bullet$ A.Gloden gave parameter solutions of this system in 1940's \cite{Gloden44} \cite{Gloden1948b}. \mbox{Numerical} examples are:
\begin{align}
\label{k24s8}
& [0, 7, 7 ] = [ 3, 5, 8]^k\\
\label{k24s11}
& [1, 9, 10 ] = [ 5, 6, 11]^k\\
& [1, 11, 12 ] = [ 4, 9, 13]^k\\
& [2, 11, 13 ] = [ 7, 7, 14]^k
\end{align}
\indent
$\bullet$ Ideal solution with additional condition $a_1 b_1=a_2 b_2$, which leads to ideal solution \eqref{k024s207} of $(k=0,2,4)$, by Chen Shuwen in 2001:
\begin{align}
\label{k24s69}
& [19,57,67] = [23,53,69]^k
\end{align}
\begin{idealchain}
\end{idealchain}
$\bullet$ It has been proved that there are solution chains for any length \cite{Hardy38}. A.Gloden gave method on how to obtained solution chains in 1940's \cite{Gloden44}. Numerical \mbox{examples} by A.Gloden's method:
\begin{align}
\label{k24s48}
& [ 23, 25, 48 ]^k = [ 15, 32, 47 ]^k = [ 8, 37, 45 ]^k = [ 3, 40, 43 ]^k \\
& [ 28, 175, 203 ]^k = [ 77, 140, 217 ]^k = [ 107, 113, 220 ]^k = [ 5, 188, 193 ] \nonumber \\
& \quad  = [ 67, 148, 215 ]^k = [ 55, 157, 212 ] ^k
\end{align}
\begin{primesolu}
\end{primesolu}
$\bullet$ First known prime solution, by Chen Shuwen in 2016 \cite{Chen01,Chen23}:
\begin{align}
\label{k24s353}
& [  89, 277, 337 ]^k = [ 139, 233, 353 ]^k
\end{align}
\indent
$\bullet$ Smallest prime solutions, based on computer search, by Chen Shuwen in 2016:
\begin{align}
& [ 11,79,97 ] ^k= [41,59,103 ]^k\\
& [ 11,107,113 ] ^k= [61,67,127 ]^k\\
& [ 41,151,163 ] ^k= [	79,113,179]^k
\end{align}
\indent
$\bullet$ Prime solutions based on parametric method, by Chen Shuwen in 2023:
\begin{align}
& [2671,9697,12277 ] ^k= [5897,7079,12923]^k\\
& [108079,200513,201827 ] ^k= [145207,146521,223757]^k
\end{align}
\begin{mirrortype}
\end{mirrortype}
$\bullet$ \ref{kn24}\quad $(k=-4,-2)$ 
\begin{relatedtype}
\end{relatedtype}
$\bullet$ \ref{k12345}\quad $(k=1,2,3,4,5)$ \vspace{1ex}\\
\indent
$\bullet$ \ref{k024}\quad $(k=0,2,4)$ \vspace{1ex}\\
\indent
$\bullet$ \ref{h124}\quad\: $(h=1,2,4)$ \vspace{1ex}\\
\indent
$\bullet$ \ref{r24}\quad\: $(r=2,4)$ 
\\
\subsubsection{\;\;(k$\;=\;$2, 6)}
\label{k26}
\begin{notation}
\end{notation}
$\bullet$ Type: $(k=2,6)$ \qquad $\bullet$ Abbreviation: $[$k26$]$
\begin{idealsolu}
\end{idealsolu}
$\bullet$ The first known solution, which is also the smallest solution, was found by K. Subba Rao in 1934 \cite{Subba34} \cite{Hardy38}:
\begin{align}
\label{k26s23}
& [3, 19, 22]^k = [10, 15, 23]^k
\end{align}
\indent
$\bullet$ Several smaller solutions were found by L. J. Lander, T. R. Parkin, and J. L. Selfridge in 1967 \cite{Lander67}:
\begin{align}
& [15, 52, 65]^k = [36, 37, 67]^k \\
& [23, 54, 73]^k = [33, 47, 74]^k \\
& [11, 65, 78]^k = [37, 50, 81]^k \\
& [3, 55, 80]^k = [32, 43, 81]^k
\end{align}
\indent
$\bullet$ Parametric solutions of this system were obtained by Simcha Brundo in 1968–1976 \cite{Brundo1969} \cite{Brundo1970} \cite{Brundo1976}, Simcha Brundo and Irving Kaplansky in 1974 \cite{Brundo1974}, Andrew Bremner in 1979 \cite{Bremner1981}, and J. Delorme in 1990 \cite{Delorme1992}.
\begin{primesolu}
\end{primesolu}
$\bullet$ The first known prime solution, based on the result found by Aloril and Jean–Charles Meyrignac \cite{JCM}, was confirmed by Chen Shuwen in 2016 \cite{Chen23}:
\begin{align}
& [379, 2837, 4253]^k = [1237, 2531, 4283]^k
\end{align}
\begin{mirrortype}
\end{mirrortype}
$\bullet$ \ref{kn26}\quad $(k=-6,-2)$
\begin{relatedtype}
\end{relatedtype}
$\bullet$ \ref{h126}\quad\: $(h=1,2,6)$
\\
\subsubsection{\;\;(k$\;=\;$1, 2, 3)}
\label{k123}
\begin{notation}
\end{notation}
$\bullet$ Type: $(k=1,2,3)$ \qquad $\bullet$ Abbreviation: $[$k123$]$
\begin{idealsolu}
\end{idealsolu}
$\bullet$ The complete ideal solution has been given by L.E.Dickson in 1910's \cite[p.55-58]{Dickson57} \cite{Chernick37} \cite{Dorwart37}. Numerical examples:
\begin{align}
\label{k123s11}
& [ 0, 4, 7, 11 ]^k = [ 1, 2, 9, 10 ]^k\\
\label{k123s22}
& [ 0, 9, 11, 22 ]^k = [ 2, 4, 15, 21 ]^k
\end{align}
\begin{idealchain}
\end{idealchain}
$\bullet$ It has been proved that there are solution chains for any length \cite{Hardy38}. A.Gloden gave method on how to obtained symmetric solution chains in 1940's \cite{Gloden44}.\mbox{Numerical} example:
\begin{align}
\label{k123s57}
& [0, 28, 29, 57]^k = [1, 21, 36, 56]^k = [2, 18, 39, 55]^k = [6, 11, 46, 51]^k
\end{align}
\indent
$\bullet$ Non-symmetric solution chains of length 3, which may lead to ideal solution chains of $(h=1,2,3,5)$, found by Chen Shuwen in 1997 based on computer search:
\begin{align}
\label{k123s214}
& [0, 87, 93, 214 ]^k = [ 9, 52, 123, 210 ]^k = [ 24, 30, 133, 207]^k \\
\label{k123s234}
& [0, 103, 116, 234]^k = [ 4, 80, 138, 231 ]^k = [26, 39, 168, 220]^k 
\end{align}
\indent
$\bullet$ Numerical non-symmetric solution chains of length 4 and length 5 can be obtained by \eqref{h1235s7313} to \eqref{h1235s421257161} of $(h=1,2,3,5)$.
\begin{primesolu}
\end{primesolu}
$\bullet$ First known prime solution, by Albert H. Beiler in 1964 \cite{Rivera65}:
\begin{align}
& [281, 281, 1181, 1181 ]^k = [ 101, 641, 821, 1361]^k
\end{align}
\indent
$\bullet$ Non-symmetric prime solution, by Chen Shuwen in 2016:
\begin{align}
& [47, 101, 113, 179 ]^k = [ 59, 71, 137, 173]^k
\end{align}
\indent
$\bullet$ Prime solution chains of length 3, by Carlos Rivera \cite{Rivera65} in 1999:
\begin{align}
\label{k123s241}
& [59, 137, 163, 241 ]^k = [ 61, 127, 173, 239 ]^k = [ 71, 103, 197, 229]^k
\end{align}
\indent
$\bullet$ Prime solution chains of length 4 and length 5, based on \eqref{k2s281}, by Chen Shuwen in 2023:
\begin{align}
& [199, 1567, 2203, 3571]^k = [223, 1459, 2311, 3547]^k \nonumber \\
& \quad = [271, 1303, 2467, 3499]^k= [379, 1063, 2707, 3391]^k \\
& [2492111, 2990063, 3221567, 3719519]^k \nonumber \\
& \quad =[2500847, 2950751, 3260879, 3710783]^k \nonumber \\
& \quad =[2518319, 2893967, 3317663, 3693311]^k \nonumber \\
& \quad =[2557631, 2806607, 3405023, 3653999]^k \nonumber \\
& \quad =[2583839, 2762927, 3448703, 3627791]^k 
\end{align}
\begin{mirrortype}
\end{mirrortype}
$\bullet$ \ref{kn123}\quad $(k=-3,-2,-1)$
\begin{relatedtype}
\end{relatedtype}
$\bullet$ \ref{k2}\quad\: $(k=2)$ \vspace{1ex}\\
\indent
$\bullet$ \ref{h1235}\quad\: $(h=1,2,3,5)$ \vspace{1ex}\\
\indent
$\bullet$ \ref{r123}\quad\: $(r=1,2,3)$
\\
\subsubsection{\;\;(k$\;=\;$1, 2, 4)}
\label{k124}
\begin{notation}
\end{notation}
$\bullet$ Type: $(k=1,2,4)$ \qquad $\bullet$ Abbreviation: $[$k124$]$
\begin{idealsolu}
\end{idealsolu}
$\bullet$ First known solutions, smallest solutions, by Chen Shuwen in 1995 \cite{Chen01,Chen23}:
\begin{align}
& [ 2, 7, 11, 15 ]^k = [ 3, 5, 13, 14]^k\\
\label{k124s19}
& [ 0, 7, 14, 19 ]^k = [ 1, 5, 16, 18]^k\\
& [ 3, 10, 13, 19 ]^k = [ 5, 6, 17, 17]^k\\
\label{k124s21}
& [ 5, 14, 14, 21 ]^k = [ 6, 10, 19, 19]^k
\end{align}
\indent
$\bullet$ Ajai Choudhry obtained two parametric solutions in 2001 \cite{Choudhry01}.
\begin{idealchain}
\end{idealchain}
$\bullet$ Ideal solution chains, based on computer search, by Chen Shuwen in 2001:
\begin{align}
\label{k124s64}
& [14, 37, 39, 64 ]^k = [ 16, 29, 46, 63 ]^k = [ 19, 24, 49, 62]^k\\
& [19, 68, 83, 129 ]^k = [ 27, 47, 101, 124 ]^k = [ 29, 44, 103 , 123]^k\\
& [3, 89, 97, 166 ]^k = [ 11, 58, 127, 159 ]^k = [ 13, 54, 131, 157]^k
\end{align}
\begin{primesolu}
\end{primesolu}
$\bullet$ First known prime solutions, based on computer search, by Chen Shuwen in 2017:
\begin{align}
& [41, 43, 101, 103 ]^k = [ 29, 67, 79, 113]^k\\
& [37,41,97,109 ]^k = [ 29,53,89,113]^k\\
& [59,83,173,173 ]^k = [ 47,113,137,191]^k
\end{align}
\indent
$\bullet$ Large prime solutions, based on parametric method, by Chen Shuwen in 2023:
\begin{align}
& [3923,24247,14851,32051]^k = [3917,24229,14869,32057]^k\\
& [30181,68543,134363,161761]^k = [30169,68567,134339,161773]^k
\end{align}
\begin{mirrortype}
\end{mirrortype}
$\bullet$ \ref{kn124}\quad $(k=-4,-2,-1)$
\begin{relatedtype}
\end{relatedtype}
$\bullet$ \ref{r124}\quad\: $(r=1,2,4)$
\\
\subsubsection{\;\;(k$\;=\;$1, 2, 5)}
\label{k125}
\begin{notation}
\end{notation}
$\bullet$ Type: $(k=1,2,5)$ \qquad $\bullet$ Abbreviation: $[$k125$]$
\begin{idealsolu}
\end{idealsolu}
$\bullet$ Chen Shuwen found a parametric solution in 1995, with the additional condition of $a_1+a_2=b_1+b_2$ and $a_3+a_4=b_3+b_4$. Numerical solutions \cite{Chen01,Chen23}:
\begin{align}
& [ 26,53,87,104 ]^k= [28,49,91,102]^k\\
& [ 53, 113, 156, 204 ]^k= [ 74, 78, 183, 191]^k\\
& [ 36, 329, 407, 608 ]^k = [ 88, 199, 537, 556]^k
\end{align}
\indent
$\bullet$ Smallest solutions, based on computer search, by Chen Shuwen in 2001:
\begin{align}
&[ 1, 28, 39, 58 ]^k= [ 8, 14, 51, 53 ]^k\\
\label{k125s58}
&[ 21, 37, 46, 58 ]^k = [ 23, 32, 51, 56 ]^k\\
&[ 29, 42, 83, 92 ]^k = [ 32, 38, 87, 89 ]^k\\
&[ 3, 52, 54, 97 ]^k= [ 12, 27, 73, 94  ]^k\\
&[ 16, 57, 71, 98 ]^k= [ 21, 43, 86, 92 ]^k
\end{align}
\indent
$\bullet$ Chen Shuwen found a parametric solution with $a_1=0$ in 2023 (See \ref{r125}). Numerical examples:
\begin{align}
&[0, 223, 642, 770 ]^k= [42, 160, 698, 735]^k\\
&[0, 376, 693, 985]^k= [126, 175, 810, 943]^k
\end{align}
\begin{primesolu}
\end{primesolu}
$\bullet$ First known prime solutions, based on parametric method with \eqref{para_prime125}, by Chen Shuwen in 2023 \cite{Chen2125}:
\begin{align}
\label{k125s601843321}
& [8445859, 296931097, 393276643, 601843321]^k \nonumber\\ 
& \quad =[88587133, 136648549,521702047, 553559191]^k \\
\label{k125s14483445169}
& [683832403, 6126532303, 10130822101, 14483445169]^k \nonumber\\ 
& \quad =[2387077453, 3265080619, 11902196953, 13870276951]^k\\
& [4840567903,9249934063,27287131717,33156848821]^k\nonumber\\ 
& \quad =[24914061433,2558769553,33978296227,13083355291]^k
\end{align}
\begin{mirrortype}
\end{mirrortype}
$\bullet$ \ref{kn125}\quad $(k=-5,-2,-1)$
\begin{relatedtype}
\end{relatedtype}
$\bullet$ \ref{r125}\quad $(r=1,2,5)$
\\
\subsubsection{\;\;(k$\;=\;$1, 2, 6)}
\label{k126}
\begin{notation}
\end{notation}
$\bullet$ Type: $(k=1,2,6)$ \qquad $\bullet$ Abbreviation: $[$k126$]$
\begin{idealsolu}
\end{idealsolu}
$\bullet$ First known solution, by Chen Shuwen in 1995 \cite{Chen01,Chen23}, based on a solution of $a_1^6+a_2^6+a_3^6+a_4^6=b_1^6+b_2^6+b_3^6+b_4^6$ found by Lander, Parkin and Selfridge in 1967 \cite{Lander67}:
\begin{align}
& [ 7, 43, 69, 110 ]^k = [ 18, 25, 77, 109]^k
\end{align}
\indent
$\bullet$ Chen Shuwen found a parametric solution in 1997, with additional condition of $a_1+a_2=b_1+b_2$ and $a_3+a_4=b_3+b_4$. Numerical solutions:
\begin{align}
&[ 31, 62, 107, 126 ]^k = [ 38, 51, 118, 119 ]^k\\
&[ 1, 80, 111, 148 ]^k = [ 5, 67, 124, 144 ]^k\\
&[ 67, 129, 138, 179 ] ^k= [ 69, 118, 149, 177 ]^k
\end{align}
\indent
$\bullet$ Smallest solutions, based on computer search, by Chen Shuwen in 2001:
\begin{align}
&[ 7, 16, 25, 30 ]^k = [ 8, 14, 27, 29 ]^k\\
&[ 15, 23, 27, 34 ]^k = [ 17, 19, 30, 33 ]^k\\
\label{k126s45}
&[ 4, 26, 33, 45 ]^k = [ 6, 20, 39, 43 ]^k
\end{align}
\indent
$\bullet$ By parametric method, Chen Shuwen found one integer solution with $a_1=0$ in 2023: $[0,-1105,1597,1713]=[-1323,545,1168,1815]$. It is not an ideal non-negative integer solution.\vspace{1ex} \\
\indent
$\bullet$ By computer search in 2023, Chen Shuwen confirmed that there is no prime solution in range of 8000.
\begin{mirrortype}
\end{mirrortype}
$\bullet$ \ref{kn126}\quad $(k=-6,-2,-1)$
\\
\subsubsection{\;\;(k$\;=\;$1, 3, 4)}
\label{k134}
\begin{notation}
\end{notation}
$\bullet$ Type: $(k=1,3,4)$ \qquad $\bullet$ Abbreviation: $[$k134$]$
\begin{idealsolu}
\end{idealsolu}
$\bullet$ T.N.Sinha found a two-parametric solution in 1984 \cite[pp.90-92]{Sinha84}. However, numerical solutions by this method always include negative
integers. Numerical example: [-1,0,8,21]=[-10,3,17,18].\vspace{1ex}\\
\indent
$\bullet$ Chen Shuwen found a parametric solution in 1995 \cite{Chen01,Chen23}, with additional condition of $a_1+a_2=b_1+b_2$ and $a_3+a_4=b_3+b_4$. Numerical solutions:
\begin{align}
\label{k134s252}
& [3, 140, 149, 252 ]^k = [ 50, 54, 201, 239 ]^k\\
& [127, 324, 1740, 2023 ]^k = [ 24, 439, 1711, 2040]^k\\
& [1059, 1444, 2476, 2763 ]^k = [ 1179, 1288, 2632, 2643]^k\\
& [2763, 5744, 6800, 9291]^k = [3776, 3915, 8139, 8768]^k
\end{align}
\indent
$\bullet$ Ajai Choudhry found a four-parameter solution in 2001 \cite{Choudhry01}. Numerical solutions:
\begin{align}
& [3876,23666,29083,42943]^k = [8551,14316,36563,40138]^k\\
& [9806,96904,188571,241468]^k = [28048,69998,206356,232347]^k\\
& [5318357,13600563,14484592,20533406]^k \nonumber \\
& \quad = [ 6580709,9969256,18251462,19135491]^k
\end{align}
\indent
$\bullet$ By computer search in 2022, Chen Shuwen confirmed that \eqref{k134s252} is the smallest solution and obtained other 9 solutions in range of 1000:
\begin{align}
& [20,201,364,479]^k =[75,122,413,454]^k\\
& [59,315,442,608]^k=[119,206,540,559]^k\\
& [102,361,459,618]^k=[126,297,523,594]^k\\
& [83,402,462,733]^k=[156,247,569,708]^k\\
& [12,353,572,783]^k=[117,194,678,731]^k\\
& [45,434,586,787]^k= [83,342,685,742]^k\\
& [5,422,438,798]^k=[138,170,581,774]^k\\
& [127,468,568,837]^k=[226,282,699,793]^k\\
& [2,277,858,995]^k=[95,169,912,956]^k 
\end{align}
\indent
$\bullet$ Chen Shuwen found a parametric solution with $a_1=0$ in 2023 (See \ref{r134}). Numerical examples:
\begin{align}
&[0,409,838,1161]^k= [109,261,900,1138]^k \\
&[0,2176,2203,4077]^k= [394, 1206, 2871, 3985]^k
\end{align}
\begin{primesolu}
\end{primesolu}
$\bullet$ First known prime solution, based on parametric method with \eqref{para_prime134} , by Chen Shuwen in 2023 \cite{Chen2125}:
\begin{align}
\label{k134s5094457378727364512429}
& [896501990958793919143, 2056330598071774290263,  \nonumber\\
& \qquad\qquad 3997663854855273138397, 5094457378727364512429]^k \nonumber \\
& \quad =[994237422911295892921, 1908177853245929320403, \nonumber\\
& \qquad\qquad 4082781516440229111169, 5059757030015751535739]^k
\end{align}
\begin{mirrortype}
\end{mirrortype}
$\bullet$ \ref{kn134}\quad $(k=-4,-3,-1)$
\begin{relatedtype}
\end{relatedtype}
$\bullet$ \ref{r134}\quad $(r=1,3,4)$
\\
\subsubsection{\;\;(k$\;=\;$1, 3, 5)}
\label{k135}
\begin{notation}
\end{notation}
$\bullet$ Type: $(k=1,3,5)$ \qquad $\bullet$ Abbreviation: $[$k135$]$
\begin{idealsolu}
\end{idealsolu}
$\bullet$ Parameter solutions were obtained by A.Gloden in 1949, G.Xeroudakes and A.Moessner in 1958 \cite{Xeroudakes58}, Lander in 1968 \cite{Lander1968}, and Ajai Choudhry in 1991\cite{Choudhry91} \cite{Choudhry1991b} . Numerical examples are:
\begin{align}
\label{k135s23}
& [ 1, 13, 17, 23 ]^k = [ 3, 9, 21, 21 ]^k\\
\label{k135s51}
& [ 0, 24, 33, 51 ]^k= [ 7, 13, 38, 50]^k\\
& [ 53, 151, 187, 233]^k = [ 51, 165, 173, 235]^k
\end{align}
\indent
$\bullet$ Smallest new solutions, based on computer search, by Chen Shuwen in 1996 \cite{Chen01,Chen23}: 
\begin{align}
\label{k135s24}
& [6, 16, 18, 24 ]^k = [ 7, 13, 21, 23]^k \\
& [7,21,23,33 ]^k = [ 11,13,29,31]^k \\
& [5,20,22,34 ]^k = [ 9,12,27,33]^k
\end{align}
\begin{idealchain}
\end{idealchain}
$\bullet$ Ideal solution chains of length 3, based on computer search, by Chen Shuwen in 2022:
\begin{align}
\label{k135s273}
& [21,169,183,273 ]^k = [ 43,113,229,261 ]^k = [ 53,99,241,253]^k \\
& [67,167,213,273 ]^k = [ 77,143,235,265 ]^k = [ 83,133,247,257]^k \\
& [47,183,217,303 ]^k = [ 73,127,263,287 ]^k = [ 79,119,271,281]^k \\
\label{k135s313}
& [77,193,227,313 ]^k = [ 89,161,253,307 ]^k = [ 117,123,273,297]^k \\
& [125,253,275,371 ]^k = [ 151,191,329,353 ]^k = [ 158,182,338,346]^k
\end{align}
\begin{primesolu}
\end{primesolu}
$\bullet$ First known prime solutions, by Chen Shuwen in 2016, based on the \mbox{result} of Torbjörn Alm and Jean-Charles Meyrignac \cite{JCM}.
\begin{align}
& [ 13, 59, 67, 101 ]^k = [ 29, 31, 83, 97 ]^k\\
& [ 17, 71, 73, 109 ]^k = [ 29, 43, 97, 101]^k
\end{align}
\indent
$\bullet$ By computer search in 2023, Chen Shuwen found 731 prime solutions in range of 10000. Numerical examples:
\begin{align}
& [23,107,109,157]^k = [31,79,137,149]^k\\
& [23,107,131,181]^k = [41,71,163,167]^k\\
& [569,5647,5693,9907]^k = [1619,3229,7247,9721]^k\\
& [3761,6563,7369,9973]^k = [4099,5653,8081,9833]^k
\end{align}
\begin{mirrortype}
\end{mirrortype}
$\bullet$ \ref{kn135}\quad $(k=-5,-3,-1)$ 
\begin{relatedtype}
\end{relatedtype}
$\bullet$ \ref{h1235}\quad\: $(h=1,2,3,5)$ \vspace{1ex}\\
\indent
$\bullet$ \ref{r135}\quad $(r=1,3,5)$
\\
\subsubsection{\;\;(k$\;=\;$1, 3, 7)}
\label{k137}
\begin{notation}
\end{notation}
$\bullet$ Type: $(k=1,3,7)$ \qquad $\bullet$ Abbreviation: $[$k137$]$
\begin{idealsolu}
\end{idealsolu}
$\bullet$ The first known three solutions, with the additional condition \( a_1 + a_4 = b_2 + b_3 \), were obtained by Ajai Choudhry in 1999 \cite{Choudhry00b}:
\begin{align}
& [ 344, 902, 1112, 1555 ]^k = [ 479, 662, 1237, 1535 ]^k \\
\label{k137s3476}
& [ 1741, 2435, 3004, 3476 ]^k = [ 1937, 2111, 3280, 3328 ]^k \\
& [ 1523, 4175, 4492, 5956 ]^k = [ 1951, 3107, 5528, 5560 ]^k 
\end{align}
\indent
$\bullet$ The smallest solution was obtained by Nuutti Kuosa and Jean-Charles Meyrignac through computer search for the seventh power in 1999 \cite{JCM2000}, and independently by Chen Shuwen at the same time \cite{Chen01, Chen23} based on \eqref{identity137}.
\begin{align}
\label{k137s698}
& [ 184, 443, 556, 698 ]^k = [ 230, 353, 625, 673 ]^k
\end{align}
\indent
$\bullet$
Jaros{\l}aw Wr\'oblewski found 14 new solutions in 2002--2006 \cite{JW955}, all satisfying the additional condition \( a_1 + a_4 = b_2 + b_3 \). The two new solutions within 10,000 and the two largest solutions are as follows:
\begin{align}
& [773,2494,3635,5129]^k=[1025,2078,3824,5104]^k \\
& [2365,5042,5785	,7544]^k=[3167,3583,6742,7244]^k \\
& [281415,1008590	,1201371,1767028]^k \nonumber \\
& \quad =[464521,678890,1369553,1745440]^k \\
& [95498	,786937,928400,1813213]^k  \nonumber \\
& \quad=[224830,543860,1044523,1810835]^k
\end{align}
\begin{mirrortype}
\end{mirrortype}
$\bullet$ \ref{kn137}\quad $(k=-7,-3,-1)$
\\
\subsubsection{\;\;(k$\;=\;$2, 3, 4)}
\label{k234}
\begin{notation}
\end{notation}
$\bullet$ Type: $(k=2,3,4)$ \qquad $\bullet$ Abbreviation: $[$k234$]$
\begin{idealsolu}
\end{idealsolu}
$\bullet$ Chen Shuwen obtained a parametric solution in 1995\cite{Chen01,Chen23}, with additional condition of $a_1-a_2=b_1-b_2$ and $a_3-a_4=b_3-b_4$. Numerical solutions:
\begin{align}
\label{k234s366448}
& [ 975, 224368, 300495, 366448 ]^k = [ 37648, 202575, 337168, 344655 ]^k \\
& [7001616, 10868299, 31439172, 34940503 ]^k \nonumber  \\
&\quad = [ 7527024, 10393591, 31599228, 34831147 ]^k\\
& [2756106, 17971525, 31568076, 35616295 ]^k \nonumber\\
&\quad = [ 3727405, 17323956, 32539375, 34968726 ]^k\\
& [2678193, 39734032, 62727021, 76546444 ]^k \nonumber\\
&\quad = [11112628, 34105617, 70918029, 71161456]^k\\
& [33801840, 3033353281, 4414180500, 5723026141 ] ^k \nonumber\\
&\quad= [ 1004104381, 2384931600, 5074604460, 5384483041]^k
\end{align}
\indent
$\bullet$  Ajai Choudhry found another parametric solution in 2001 \cite{Choudhry01} and obtained three numerical solutions. These three solutions also satisfy $a_1-a_2=b_1-b_2$ and $a_3-a_4=b_3-b_4$.
\begin{align}
\label{k234s1058}
& [ 43, 486, 815, 1058 ]^k = [ 242, 335, 907, 1014 ]^k\\
& [239, 2324, 3855, 4564 ]^k = [ 604, 2079, 4220, 4319 ]^k\\
& [813, 14008, 19605, 22096]^k = [ 1096, 13845, 19888, 21933]^k
\end{align}
\indent
$\bullet$ By computer search in 2022, Chen Shuwen confirmed that \eqref{k234s1058} is the smallest solution and found the second smallest one:
\begin{align}
[15,863,1286,1788]^k=[422,564,1455,1727]^k
\end{align}
\begin{mirrortype}
\end{mirrortype}
$\bullet$ \ref{kn234}\quad $(k=-4,-3,-2)$
\\
\subsubsection{\;\;(k$\;=\;$2, 4, 6)}
\label{k246}
\begin{notation}
\end{notation}
$\bullet$ Type: $(k=2,4,6)$ \qquad $\bullet$ Abbreviation: $[$k246$]$
\begin{idealsolu}
\end{idealsolu}
$\bullet$ Parametric solutions of this system were obtained by Crussol \cite{Dickson52} and J.Chernick \cite{Chernick37} independently in 1913. Numerical examples:
\begin{align}
\label{k246s25}
&[ 2, 16, 21, 25 ]^k= [ 5, 14, 23, 24 ]^k\\
&[ 7, 24, 25, 34 ]^k= [ 14, 15, 31, 32 ]^k\\
&[ 7, 31, 36, 50 ]^k= [ 18, 20, 41, 49 ]^k\\
&[ 9, 47, 49, 67 ]^k= [ 23, 31, 61, 63 ]^k
\end{align}
\indent
$\bullet$ Ajai Choudhry obtained four parameter symmetric solutions in 2022 \cite{Choudhry22}. 
\begin{primesolu}
\end{primesolu}
$\bullet$ First known prime solution, by Chen Shuwen in 2017 \cite{Chen23}:
\begin{align}
&[317, 541, 953, 1049 ]^k=[ 139, 719, 827, 1087]^k
\end{align}
\indent
$\bullet$ Second and third known prime solutions, based on computer search in range of 15000, by Chen Shuwen in 2023:
\begin{align}
&[47,653,743,1429]^k=[313,463,821,1427 ]^k\\
&[211,2113,2539,3301]^k=[1069,1429,2971,3137]^k
\end{align}
\begin{mirrortype}
\end{mirrortype}
$\bullet$ \ref{kn246}\quad $(k=-6,-4,-2)$
\begin{relatedtype}
\end{relatedtype}
$\bullet$ \ref{k1234567}\quad $(k=1,2,3,4,5,6,7)$ \vspace{1ex}\\
\indent
$\bullet$ \ref{h1246}\quad\: $(h=1,2,4,6)$ 
\\
\subsubsection{\;\;(k$\;=\;$1, 2, 3, 4)}
\label{k1234}
\begin{notation}
\end{notation}
$\bullet$ Type: $(k=1,2,3,4)$ \qquad $\bullet$ Abbreviation: $[$k1234$]$
\begin{idealsolu}
\end{idealsolu}
$\bullet$ This type is PTE of degree 4. J.Chernick gave a two-parameter solution of this system in 1937 \cite{Chernick37}. All solutions by his method are symmetric:
\begin{align}
\label{k1234s18}
&[ 0, 4, 8, 16, 17 ]^k = [ 1, 2, 10, 14, 18 ]^k\\
\label{k1234s20}
&[ 0, 6, 8, 17, 19 ]^k = [ 1, 3, 12, 14, 20 ]^k
\end{align}
\indent
$\bullet$ J.L.Burchnall and T.W.Chaundy gave a parameter solution for non-symmetric solution in 1937 \cite{Burchnall1937}:
\begin{align}
&[ 0, 9, 13, 26, 32 ]^k = [ 2, 4, 20, 21, 33 ]^k\\
&[ 0, 31, 49, 87, 113]^k = [ 3, 21, 64, 77, 115 ]^k
\end{align}
\indent
$\bullet$ Ajai Choudhry obtained the complete ideal symmetric solution and a parametric ideal non-symmetric solution in 2000 \cite{Choudhry2000}.
\begin{primesolu}
\end{primesolu}
$\bullet$ First known symmetric prime solution, by Chen Shuwen in 2016 \cite{Chen23}:
\begin{align}
&[ 401, 521, 641, 881, 911 ]^k = [ 431, 461, 701, 821, 941]^k
\end{align}
\indent
$\bullet$ First known non-symmetric prime solution, by Chen Shuwen in 2016:
\begin{align}
&[ 337, 607, 727, 1117, 1297 ] ^k= [ 397, 457, 937, 967, 1327]^k
\end{align}
\begin{mirrortype}
\end{mirrortype}
$\bullet$ \ref{kn1234}\quad$(k=-4,-3,-2,-1)$ 
\begin{relatedtype}
\end{relatedtype}
$\bullet$ \ref{h12346}\quad\: $(h=1,2,3,4,6)$ \vspace{1ex}\\
\indent
$\bullet$ \ref{r13}\quad\: $(r=1,3)$ \vspace{1ex}\\
\indent
$\bullet$ \ref{r1234}\quad $(r=1,2,3,4)$
\\
\subsubsection{\;\;(k$\;=\;$1, 2, 3, 5)}
\label{k1235}
\begin{notation}
\end{notation}
$\bullet$ Type: $(k=1,2,3,5)$ \qquad $\bullet$ Abbreviation: $[$k1235$]$
\begin{idealsolu}
\end{idealsolu}
$\bullet$ First known solutions, based on parametric method, by Chen Shuwen in 1995 \cite{Chen01,Chen23}:
\begin{align}
&[719,879,1039,1249,1409]^k=[739,829,1109,1199,1419]^k\\
&[ 15, 25, 55, 55, 73]^k = [ 13, 31, 43, 67, 69 ]^k
\end{align}
\indent
$\bullet$ Smallest solutions, based on computer search, by Chen Shuwen in 1995-1999:
\begin{align}
\label{k1235s28}
&[ 1, 8, 13, 24, 27 ]^k = [ 3, 4, 17, 21, 28 ]^k\\
&[ 1, 17, 20, 42, 42 ] ^k= [ 2, 12, 26, 37, 45 ]^k\\
&[ 11, 24, 36, 52, 53 ]^k = [ 12, 21, 43, 44, 56 ]^k\\
&[ 5, 17, 35, 55, 55 ]^k = [ 7, 13, 41, 47, 59 ]^k\\
&[ 3, 14, 34, 51, 57 ] ^k= [ 6, 9, 42, 43, 59 ]^k
\end{align}
\begin{primesolu}
\end{primesolu}
$\bullet$ First known prime solution, based on computer search, by Chen Shuwen in 2017:
\begin{align}
& [ 17, 79, 113, 191, 199 ]^k =[23, 59, 149, 157, 211 ]^k 
\end{align}
\indent
$\bullet$ Prime solutions based on computer search, by Chen Shuwen in 2023:
\begin{align}
& [13,79,97,191,227]^k =[31,37,139,167,233 ]^k \\
& [61,127,137,227,251]^k =[67,101,167,211,257 ]^k \\
& [43,109,157,251,263]^k =[53,83,199,211,277 ]^k \\
& [13,53,157,239,271]^k =[29,31,173,223,277 ]^k \\
& [13,101,113,269,277]^k =[29,53,157,241,293 ]^k \\
& [59,83,223,277,281]^k =[67,73,239,251,293 ]^k 
\end{align}
\indent
$\bullet$ Large prime solutions, based on parametric method, by Chen Shuwen in 2023:
\begin{align}
& [23293,23909,23957,24677,24841]^k \nonumber\\
& \quad =[23333,23677,24329,24421,24917]^k \\
& [141773,174767,185531,238649,242393]^k \nonumber\\
& \quad =[150197,154409,210803,217121,250583 ]^k \\
& [8387161381, 18349843201, 24874292041, 38143200421, 39230608561]^k \nonumber\\
& \quad =[9474569521, 15043616641, 31486745161, 31530747301, 41449426981]^k
\end{align}
\begin{mirrortype}
\end{mirrortype}
$\bullet$ \ref{kn1235}\quad $(k=-5,-3,-2,-1)$
\begin{relatedtype}
\end{relatedtype}
$\bullet$ \ref{r1235}\quad $(r=1,2,3,5)$
\\
\subsubsection{\;\;(k$\;=\;$1, 2, 3, 6)}
\label{k1236}
\begin{notation}
\end{notation}
$\bullet$ Type: $(k=1,2,3,6)$ \qquad $\bullet$ Abbreviation: $[$k1236$]$
\begin{idealsolu}
\end{idealsolu}
$\bullet$ First known solutions, based on computer search, by Chen Shuwen in 1999 \cite{Chen01,Chen23}:
\begin{align}
&[ 7, 18, 55, 69, 81 ]^k = [ 9, 15, 61, 63, 82 ]^k\\
\label{k1236s107}
&[7, 27, 53, 90, 106 ]^k = [ 10, 21, 58, 87, 107]^k
\end{align}
\indent 
$\bullet$ By selective search and parametric method, Chen Shuwen found 99 solutions in 2022. Numerical examples:
\begin{align}
&[31,63,114,165,169]^k=[39,49,135,141,178]^k\\
&[33,64,118,172,195]^k=[40,52,129,163,198]^k\\	
&[44,91,113,182,196]^k=[56,64,142,161,203]^k\\
&[45,92,102,193,194]^k=[58,62,125,177,204]^k\\	
&[62,109,111,203,207]^k=[75,77,135,190,215]^k\\
&[29,81,123,203,205]^k=[43,53,161,165,219]^k\\	
&[3,93,115,174,221]^k=[5,78,141,159,223]^k\\
&[23,91,130,225,257]^k=[35,61,162,205,263]^k\\
&[10517,10842,11297,12177,12597]^k \nonumber \\
& \quad =[10617,10677,11397,12122,12617]^k\\
&[4985,11127,17103,21004,22498]^k \nonumber \\
& \quad =[5068,10795,18099,19842,22913]^k\\
&[8386,13861,27961,37561,38011]^k \nonumber \\
& \quad =[8761,13186,28861,35761,39211]^k	
\end{align}
\begin{mirrortype}
\end{mirrortype}
$\bullet$ \ref{kn1236}\quad $(k=-6,-3,-2,-1)$
\\
\subsubsection{\;\;(k$\;=\;$1, 2, 3, 7)}
\label{k1237}
\begin{notation}
\end{notation}
$\bullet$ Type: $(k=1,2,3,7)$ \qquad $\bullet$ Abbreviation: $[$k1237$]$
\begin{idealsolu}
\end{idealsolu}
$\bullet$ First known solutions, based on computer search, by Chen Shuwen in 2022 \cite{Chen23}:
\begin{align}
\label{k1237s503}
&[41, 148, 248, 389, 502 ] ^k = [ 46, 134, 262, 383, 503]^k\\
&[519, 710, 781, 937, 954 ]^k = [ 521, 694, 807, 909, 970]^k\\
&[261, 816, 821,1601,1756 ]^k= [ 271, 711, 926, 1581, 1766]^k
\end{align}
\begin{mirrortype}
\end{mirrortype}
$\bullet$ \ref{kn1237}\quad $(k=-7,-3,-2,-1)$
\\
\subsubsection{\;\;(k$\;=\;$1, 2, 4, 6)}
\label{k1246}
\begin{notation}
\end{notation}
$\bullet$ Type: $(k=1,2,4,6)$ \qquad $\bullet$ Abbreviation: $[$k1246$]$
\begin{idealsolu}
\end{idealsolu}
$\bullet$ First known solution, by G.Palama in 1953 \cite{Palama1953}:
\begin{align}
\label{k1246s17}
&[ 3, 7, 10, 16, 16 ]^k = [ 4, 5, 12, 14, 17 ]^k
\end{align}
\indent
$\bullet$ Ideal solutions, by Chen Shuwen in 1994-1997 \cite{Chen01,Chen23}:
\begin{align}
&[ 1, 19, 22, 37, 42 ]^k = [ 2, 14, 29, 33, 43 ]^k\\
&[ 17, 26, 31, 42, 45 ]^k = [ 18, 23, 35, 39, 46 ]^k\\
&[ 13, 25, 32, 48, 50 ]^k = [ 15, 20, 38, 43, 52 ]^k\\
&[ 7, 44, 48, 79, 90 ]^k = [ 9, 33, 64, 70, 92 ]^k\\
&[ 10, 46, 48, 83, 91 ]^k = [ 13, 32, 69, 70, 94 ]^k\\
&[ 13, 50, 56, 95, 102 ]^k = [ 17, 35, 78, 80, 106 ]^k
\end{align}
\begin{primesolu}
\end{primesolu}
$\bullet$ First known prime solution, by Chen Shuwen in 2017:
\begin{align}
\label{k1246s503}
&[ 83, 149, 337, 439, 503 ]^k = [ 71, 173, 313, 463, 491 ]^k
\end{align}
\begin{mirrortype}
\end{mirrortype}
$\bullet$ \ref{kn1246}\quad $(k=-6,-4,-2,-1)$
\\
\subsubsection{\;\;(k$\;=\;$1, 3, 5, 7)}
\label{k1357}
\begin{notation}
\end{notation}
$\bullet$ Type: $(k=1,3,5,7)$ \qquad $\bullet$ Abbreviation: $[$k1357$]$
\begin{idealsolu}
\end{idealsolu}
$\bullet$ First known solutions, by A.Gloden in 1940's \cite{Gloden44} \cite{Gloden1948} \cite{Gloden1949c}:
\begin{align}
\label{k1357s55}
&[ 3, 19, 37, 51, 53 ]^k = [ 9, 11, 43, 45, 55 ]^k\\
&[13, 31, 47, 65, 67 ]^k = [ 15, 27, 51, 61, 69 ]^k\\
&[ 1, 33, 39, 65, 71 ]^k = [ 11, 15, 53, 57, 73 ]^k
\end{align}
\indent
$\bullet$ Only two solutions with zero term are known so far, which lead to the ideal solution of $(k=1,2,3,4,5,6,7,8)$, found by A.Letac in 1942 \cite{Letac42} \cite{Gloden44}:
\begin{align}
\label{k1357s99}
&[ 0, 34, 58, 82, 98 ]^k = [ 13, 16, 69, 75, 99 ]^k\\
\label{k1357s174}
&[ 0, 63, 119, 161, 169 ]^k = [ 8, 50, 132, 148, 174 ]^k
\end{align}
\indent
$\bullet$ 
A.Moessner gave a two-parameter solution based on Gloden's method in 1947-1952 \cite{Moessner1947} \cite{Moessner1952}.\\[1mm]
\indent
$\bullet$ T.N.Sinha gave a two-parameter solution in 1966 \cite{Sinha66} \cite{Sinha1966} \cite[p.116-117]{Piezas09}.
\begin{primesolu}
\end{primesolu}
\indent
$\bullet$ For the first known prime solution, Laurent Lucas found an equality of sums of 7th powers in 2000 using Jean-Charles Meyrignac's program \cite{JCM}. Chen Shuwen confirmed that it is a multigrade prime solution in 2016:
\begin{align}
\label{k1357s257}
&[19, 101, 157, 239, 251]^k = [31, 79, 173, 227, 257]^k
\end{align}
\indent
$\bullet$ Prime solutions, based on a computer search within the range of 5000, were found by Chen Shuwen in 2016 and 2023 \cite{Chen23}:
\begin{align}
&[29, 179, 311, 433, 503]^k = [83, 103, 359, 401, 509]^k \\
&[227, 349, 353, 521, 571]^k = [251, 281, 409, 503, 577]^k \\
&[101, 359, 389, 601, 691]^k = [139, 251, 509, 541, 701]^k 
\end{align}
\begin{mirrortype}
\end{mirrortype}
$\bullet$ \ref{kn1357}\quad $(k=-7,-5,-3,-1)$
\begin{relatedtype}
\end{relatedtype}
$\bullet$ \ref{k12345678}\quad $(k=1,2,3,4,5,6,7,8)$ \vspace{1ex}\\
\indent
$\bullet$ \ref{h12357}\quad\: $(h=1,2,3,5,7)$ \vspace{1ex}\\
\indent
$\bullet$ \ref{r1357}\quad $(r=1,3,5,7)$
\\
\subsubsection{\;\;(k$\;=\;$2, 4, 6, 8)}
\label{k2468}
\begin{notation}
\end{notation}
$\bullet$ Type: $(k=2,4,6,8)$ \qquad $\bullet$ Abbreviation: $[$k2468$]$
\begin{idealsolu}
\end{idealsolu}
$\bullet$ First known solution, by A.Letac in 1942  \cite{Letac42} \cite{Gloden44}:
\begin{align}
\label{k2468s23750}
& [12, 11881, 20231, 20885, 23738 ]^k \nonumber\\
& \quad = [ 436, 11857, 20449, 20667 , 23750  ]^k
\end{align}
\indent
$\bullet$ Second known solution, independently by G.Palama \cite{Palama1950} in 1950 and by C.J.Smyth in 1990 \cite{Smyth91}, both based on Letac's method:
\begin{align}
\label{k2468s529393533005}
& [ 87647378809, 243086774390, \nonumber\\
& \qquad\qquad 308520455907,441746154196, 527907819623 ] ^k \nonumber\\
& \quad = [ 133225698289, 189880696822, \nonumber \\
& \qquad\qquad 338027122801, 432967471212, 529393533005 ]^k
\end{align}
\indent
$\bullet$
In 1980, T.N. Sinha demonstrated that there exist an infinite number of distinct solutions using Letac's method \cite{Sinha84}. Subsequently, in 1990, C.J. Smyth employed elliptic curve theory to prove that Letac's method indeed produces infinitely many genuinely different solutions \cite{Smyth91}.\\
\indent
$\bullet$ Smallest two solutions, by Peter Borwein, Petr Lisonek and Colin Percival in 2000, based on computer search \cite{Borwein2003}:
\begin{align}
\label{k2468s313}
& [71, 131, 180, 307, 308 ]^k = [ 99, 100, 188, 301 , 313]^k\\
\label{k2468s515}
& [18, 245, 331, 471, 508 ]^k = [103, 189, 366, 452, 515]^k
\end{align}
\indent
$\bullet$ The third smallest solution was found by Chen Shuwen in 2023, by applying \eqref{n2m2k2468} based on a computer search \cite{Chen23}:
\begin{align}
\label{k2468s4827}
& [ 498, 3773, 3783, 4567, 4787 ]^k= [ 517, 3598, 4017, 4463, 4827]^k
\end{align}
\begin{mirrortype}
\end{mirrortype}
$\bullet$ \ref{kn2468}\quad $(k=-8,-6,-4,-2)$
\begin{relatedtype}
\end{relatedtype}
$\bullet$ \ref{k123456789}\quad $(k=1,2,3,4,5,6,7,8,9)$\vspace{1ex}\\
\indent
$\bullet$ \ref{h12468}\quad\: $(h=1,2,4,6,8)$
\\
\subsubsection{\;\;(k$\;=\;$1, 2, 3, 4, 5)}
\label{k12345}
\begin{notation}
\end{notation}
$\bullet$ Type: $(k=1,2,3,4,5)$ \qquad $\bullet$ Abbreviation: $[$k12345$]$
\begin{idealsolu}
\end{idealsolu}
$\bullet$ 
This type is PTE of degree 5. G.Tarry gave a two-parameter solution in 1912 \cite{Tarry1912}. Numerical examples are:
\begin{align}
\label{k12345s16}
& [ 0, 3, 5, 11, 13, 16 ]^k = [ 1, 1, 8, 8, 15, 15 ]^k \\
\label{k12345s22}
& [ 0, 5, 6, 16, 17, 22 ]^k = [ 1, 2, 10, 12, 20, 21 ] ^k 
\end{align}
\indent
$\bullet$ First known non-symmetric solution, by A.Gloden in 1944 \cite{Gloden44}:
\begin{align}
\label{k12345s86}
& [ 0, 19, 25, 57, 62, 86 ]^k  = [ 2, 11, 40, 42, 69, 85 ]^k 
\end{align}
\begin{idealchain}
\end{idealchain}
$\bullet$ It has been proved that there are solution chains for any length \cite{Hardy38}. A.Gloden gave method on how to obtained symmetric solution chains in 1940's \cite{Gloden44}.\mbox{Numerical} examples by A.Gloden's method:
\begin{align}
\label{k12345s96}
& [0, 23, 25, 71, 73, 96]^k = [1, 16, 33, 63, 80, 95]^k \nonumber\\
& \quad = [3, 11, 40, 56, 85, 93]^k = [5, 8, 45, 51, 88, 91]^k\\
& [0, 567, 644, 1778, 1855, 2422 ]^k = [ 2, 535, 678, 1744, 1887, 2420 ]^k \nonumber \\
& \quad= [ 7, 490, 728, 1694, 1932, 2415 ]^k = [ 15, 444, 782, 1640, 1978, 2407 ]^k \nonumber\\
& \quad= [ 28, 392, 847, 1575, 2030, 2394 ]^k= [ 42, 350, 903, 1519, 2072, 2380 ]^k \nonumber\\
& \quad= [ 62, 303, 970, 1452, 2119, 2360 ]^k=[ 70, 287, 994, 1428, 2135, 2352 ]^k \nonumber\\
& \quad= [ 95, 244, 1062, 1360, 2178, 2327 ]^k=[ 103, 232, 1082, 1340, 2190, 2319 ]^k \nonumber\\
& \quad= [ 119, 210, 1120, 1302, 2212, 2303 ]^k=[ 144, 180, 1175, 1247, 2242, 2278 ]^k
\end{align}
\begin{primesolu}
\end{primesolu}
$\bullet$ First known symmetric prime solution, by T.W.A. Baumann in 1999 \cite{Rivera65}:
\begin{align}
& [ 277, 937, 1069, 2389, 2521, 3181]^k \nonumber\\
& \quad= [409, 541, 1597, 1861, 2917, 3049 ]^k
\end{align}
\indent
$\bullet$ Non-symmetric prime solution, by Chen Shuwen in 2016 \cite{Chen23}:
\begin{align}
& [31, 3541, 6661, 13291, 14071, 17971]^k \nonumber\\
& \quad= [421, 2371, 9391, 9781, 16411, 17191]^k
\end{align}
\indent
$\bullet$ First known prime solution chains, by Chen Shuwen in 2023:
\begin{align}
\label{k12345s11827}
& [1459, 3943, 4159, 9127, 9343, 11827 ]^k  \nonumber\\
&\quad =[ 1567, 3187, 5023, 8263, 10099, 11719 ]^k \nonumber\\ 
&\quad = [ 1783, 2647, 5779, 7507, 10639, 11503 ]^k\\
& [ 31338661, 34136197, 34379461, 39974533, 40217797, 43015333]^k \nonumber\\
& \quad=[ 31460293, 33284773, 35352517, 39001477, 41069221, 42893701]^k \nonumber\\
& \quad=[ 31703557, 32676613, 36203941, 38150053, 41677381, 42650437]^k 
\end{align}
\begin{mirrortype}
\end{mirrortype}
$\bullet$ \ref{kn12345}\quad $(k=-5,-4,-3,-2,-1)$
\begin{relatedtype}
\end{relatedtype}
$\bullet$ \ref{k24}\quad $(k=2,4)$\vspace{1ex}\\
\indent
$\bullet$ \ref{h123457}\quad $(h=1,2,3,4,5,7)$
\\
\subsubsection{\;\;(k$\;=\;$1, 2, 3, 4, 6)}
\label{k12346}
\begin{notation}
\end{notation}
$\bullet$ Type: $(k=1,2,3,4,6)$ \qquad $\bullet$ Abbreviation: $[$k12346$]$
\begin{idealsolu}
\end{idealsolu}
$\bullet$ First known solution, based on parametric method, by Chen Shuwen in 1995 \cite{Chen01,Chen23}:
\begin{align}
& [116, 166, 206, 331, 336, 411]^k = [ 131, 136, 236, 291, 366, 406 ]^k \\
& [ 11, 23, 24, 47, 64, 70 ]^k = [ 14, 15, 31, 44, 67, 68 ]^k \\
& [ 42, 48, 59, 74, 76, 85 ]^k = [ 43, 46, 62, 69, 80, 84 ]^k \\
& [ 23, 31, 60, 80, 91, 103 ] ^k= [ 25, 28, 65, 73, 96, 101 ]^k \\
& [ 19 , 29 , 89 , 123 , 127 , 152 ] ^k= [ 23 , 24 , 97 , 107 , 139 , 149 ]^k \\
& [ 73 , 85 , 116 , 146 , 147 , 164 ]^k = [ 74 , 83 , 120 , 136 , 157 , 161 ]
\end{align}
\indent
$\bullet$ Smallest solutions, based on computer search, by Chen Shuwen in 2017:
\begin{align}
\label{k12346s60}
& [1, 14, 17, 46, 48, 60 ] ^k = [ 4, 6, 25, 38, 56, 57 ]^k \\
& [ 10, 21, 25, 44, 62, 68 ] ^k = [ 13, 14, 32, 40, 65, 66 ]^k
\end{align}
\begin{mirrortype}
\end{mirrortype}
$\bullet$ \ref{kn12346}\quad $(k=-6,-4,-3,-2,-1)$
\begin{relatedtype}
\end{relatedtype}
$\bullet$ \ref{r12346}\quad $(r=1,2,3,4,6)$
\\
\subsubsection{\;\;(k$\;=\;$1, 2, 3, 4, 7)}
\label{k12347}
\begin{notation}
\end{notation}
$\bullet$ Type: $(k=1,2,3,4,7)$ \qquad $\bullet$ Abbreviation: $[$k12347$]$
\begin{idealsolu}
\end{idealsolu}
$\bullet$ First known solution, based on computer search, by Chen Shuwen in 2022 \cite{Chen23}:
\begin{align}
\label{k12347s366}
[ 34, 133, 165, 299, 332, 366 ]^k = [ 35, 124, 177, 286, 353, 354 ]^k
\end{align}
\begin{mirrortype}
\end{mirrortype}
$\bullet$ \ref{kn12347}\quad $(k=-7,-4,-3,-2,-1)$
\\
\subsubsection{\;\;(k$\;=\;$1, 2, 3, 5, 7)}
\label{k12357}
\begin{notation}
\end{notation}
$\bullet$ Type: $(k=1,2,3,5,7)$ \qquad $\bullet$ Abbreviation: $[$k12357$]$
\begin{idealsolu}
\end{idealsolu}
$\bullet$ First known solution, based on parametric method, by Chen Shuwen in 1999 \cite{Chen01,Chen23}:
\begin{align}
[ 87, 233, 264, 396, 496, 540 ]^k = [ 90, 206, 309, 366, 522, 523 ]^k 
\end{align}
\indent
$\bullet$ Solutions based on selective search, by Chen Shuwen in 2023:
\begin{align}
\label{k12357s257}
&[ 9, 73, 79, 207, 211, 257 ]^k = [ 19, 39, 113, 177, 241, 247 ]^k \\
\label{k12357s381}
&[ 11, 69, 165, 295, 299, 381 ]^k = [ 29, 39, 201, 241, 335, 375 ]^k \\
&[31, 97, 197, 347, 413, 469 ] ^k= [ 53, 61, 227, 317, 439, 457 ]^k \\
&[ 141, 211, 249, 363, 407, 503 ] ^k= [ 155, 175, 305, 315, 423, 501 ]^k \\
&[ 13, 141, 187, 443, 459, 561 ]^k = [ 43, 67, 261, 369, 531, 533 ]^k 
\end{align}
\begin{mirrortype}
\end{mirrortype}
$\bullet$ \ref{kn12357}\quad $(k=-7,-5,-3,-2,-1)$
\\
\subsubsection{\;\;(k$\;=\;$1, 2, 4, 6, 8)}
\label{k12468}
\begin{notation}
\end{notation}
$\bullet$ Type: $(k=1,2,4,6,8)$ \qquad $\bullet$ Abbreviation: $[$k12468$]$
\begin{idealsolu}
\end{idealsolu}
$\bullet$ First known solution, by Chen Shuwen in 1995 \cite{Chen01,Chen23}:
\begin{align}
\label{k12468s36}
& [1, 7, 17, 30, 31, 36 ]^k = [ 3, 4, 19, 27, 34, 35]^k 
\end{align}
\indent
$\bullet$ Ideal solutions, by Chen Shuwen in 1997:
\begin{align}
& [ 64, 169, 184, 277, 347, 417 ]^k = [ 69, 139, 233, 248, 353 , 416 ]^k \\
& [ 111, 242, 243, 445, 446, 596 ]^k = [ 148, 149, 353, 354, 485 , 594 ]^k
\end{align}
\indent
$\bullet$ Ideal solutions, by Chen Shuwen in 2022: 
\begin{align}
& [28,55,58,117,139,151]^k=[36,37,70,113,143,149]^k\\	
& [51,281,445,769,872,992]^k=[89,192,547,680,941,961]^k\\
& [29,291,1037,1803,1850,2132]^k \nonumber \\
& \quad =[122,169,1095,1681,1972,2103]^k\\
& [169,1137,2075,3491,4004,4488]^k \nonumber \\
& \quad =[312,825,2413,3179,4316,4319]^k\\
& [797,2799,3329,4685,6332,7192]^k \nonumber \\
& \quad=[841,2488,3844,4393,6395,7173]^k	\\	
& [893,2993,3219,5591,6606,7772]^k \nonumber \\
& \quad=[1102,2117,4489,4779,6879,7708]^k
\end{align}
\begin{mirrortype}
\end{mirrortype}
$\bullet$ \ref{kn12468}\quad $(k=-8,-6,-4,-2,-1)$
\\ 
\subsubsection{\;\;(k$\;=\;$1, 3, 5, 7, 9)}
\label{k13579}
\begin{notation}
\end{notation}
$\bullet$ Type: $(k=1,3,5,7,9)$ \qquad $\bullet$ Abbreviation: $[$k13579$]$
\begin{idealsolu}
\end{idealsolu}
$\bullet$ First known solution, based on computer search with Identity \eqref{identity13579}, by Chen Shuwen in 2000 \cite{Chen01,Chen23}:
\begin{align}
\label{k13579s323}
& [7, 91, 173, 269, 289, 323 ]^k = [ 29, 59, 193, 247, 311, 313]^k 
\end{align}
\indent
$\bullet$ Second known solutions, by Jarosław Wróblewski in 2009 \cite[p.21-23]{JW09}:
\begin{align}
\label{k13579s407}
& [ 23, 163, 181, 341, 347, 407 ]^k = [ 37, 119, 221, 311, 371, 403 ]^k \\
\label{k13579s463}
& [ 43, 161, 217, 335, 391, 463 ]^k = [ 85, 91, 283, 287, 403, 461 ]^k \\
\label{k13579s1293}
& [ 57, 399, 679, 995, 1167, 1293 ]^k = [ 115, 299, 767, 925, 1205, 1279 ]^k
\end{align}
\indent
$\bullet$ Only the above four ideal non-negative solutions are known so far. In addition, Jarosław Wróblewski found one integer solution in 2009 \cite[p.24]{JW09}: $[ -13, 365, 689, 1111, 1115, 1325 ] = [ 23, 305, 731, 1037, 1177, 1319 ]$.
\begin{mirrortype}
\end{mirrortype}
$\bullet$ \ref{kn13579}\quad $(k=-9,-7,-5,-3,-1)$
\begin{relatedtype}
\end{relatedtype}
$\bullet$ \ref{h123579}\quad $(h=1,2,3,5,7,9)$\vspace{1ex}
\\
\subsubsection{\;\;(k$\;=\;$2, 4, 6, 8, 10)}
\label{k246810}
\begin{notation}
\end{notation}
$\bullet$ Type: $(k=2,4,6,8,10)$ \qquad $\bullet$ Abbreviation: $[$k246810$]$
\begin{idealsolu}
\end{idealsolu}
$\bullet$ First known solution, by Nuutti Kuosa, Jean-Charles Meyrignac and Chen Shuwen in 1999 \cite{Chen22,Chen23}: 
\begin{align}
\label{k246810s151}
&[ 22, 61, 86, 127, 140, 151 ]^k = [ 35, 47, 94, 121, 146, 148 ]^k
\end{align}
\indent
$\bullet$ Second known solution, by David Broadhurst in 2007 \cite{Broadhurst2007} \cite{Broadhurst2007b}:
\begin{align}
\label{k246810s2058}
& [ 257, 891, 1109, 1618, 1896, 2058 ]^k \nonumber\\
& \quad = [ 472, 639, 1294, 1514, 1947, 2037]^k
\end{align}
\indent
$\bullet$ In 2008, A.Choudhry and Jarosław Wróblewski proved that this system has infinitely many solutions and give a method to produce new numerical solutions \cite{CW2008}. The smallest new solutions found by their method are:
\begin{align}
\label{k246810s1511}
&[ 107, 622, 700, 1075, 1138, 1511 ]^k \nonumber \\
& \quad = [ 293, 413, 886, 953, 1180, 1510 ]^k \\
&[929, 7043, 7115, 11632, 12638, 14770 ]^k \nonumber \\
& \quad = [3455, 4054, 9718, 10112, 13165, 14693]^k \\
&[349, 9611, 12734, 18372, 18687, 23742 ]^k \nonumber \\
& \quad = [3309, 7713, 14714, 16426, 19653, 23708 ]^k 
\end{align}
\begin{mirrortype}
\end{mirrortype}
$\bullet$ \ref{kn246810}\quad $(k=-10,-8,-6,-4,-2)$
\begin{relatedtype}
\end{relatedtype}
$\bullet$ \ref{k1234567891011}\quad $(k=1,2,3,4,5,6,7,8,9,10,11)$\vspace{1ex}\\
\indent
$\bullet$ \ref{h1246810}\quad $(h=1,2,4,6,8,10)$
\\
\subsubsection{\;\;(k$\;=\;$1, 2, 3, 4, 5, 6)}
\label{k123456}
\begin{notation}
\end{notation}
$\bullet$ Type: $(k=1,2,3,4,5,6)$ \qquad $\bullet$ Abbreviation: $[$k123456$]$
\begin{idealsolu}
\end{idealsolu}
$\bullet$ This type is PTE of degree 6. The first known symmetric solution was discovered by E.B.Escott in 1910 \cite [p.616] {Dorwart37} \cite{Dorwart47}:
\begin{align}
\label{k123456s102}
&[ 0, 18, 27, 58, 64, 89, 101 ]^k = [ 1, 13, 38, 44, 75, 84, 102 ]^k
\end{align}
\indent
$\bullet$ J.Chernick gave a two-parameter symmetric solution in 1937 \cite{Chernick37}. Numerical example:
\begin{align}
& [ 0, 59, 68, 142, 181, 221, 267 ]^k = [ 1, 47, 87, 126, 200 , 209, 268 ]^k
\end{align}
\indent
$\bullet$ First known non-symmetric solutions, based on parametric method, by Chen Shuwen in 1997 \cite{Chen01,Chen23}:
\begin{align}
\label{k123456s84}
& [ 0, 18, 19, 50, 56, 79, 81 ]^k = [ 1, 11, 30, 39, 68, 70 , 84 ] ^k\\
\label{k123456s204}
& [0,14,43,141,156,193,199]^k=[3,9,46,133,175,176,204]^k 
\end{align}
\indent
$\bullet$ Non-symmetric solutions, based on computer search, by Chen Shuwen in 2001-2019:
\begin{align}
& [0,19,50,86,93,131,133]^k=[1,16,63,65,110,119,138]^k \\
& [0,24,31,93,98,139,146]^k=[3,13,44,76,119,126,150]^k \\
& [0,58,63,137,144,197,226]^k=[1,44,93,102,175,182,228]^k \\
& [0,14,87,149,158,211,232]^k=[2,11,92,133,175,204,234]^k \\
& [0,41,59,157,167,226,240]^k=[2,31,72,136,195,209,245]^k \\
\label{k123456s264}
& [0,57,64,140,194,247,257]^k=[2,40,92,119,217,225,264]^k \\
& [0,49,93,176,188,265,274]^k=[1,44,104,153,210,253,280]^k \\
& [0,21,69,166,194,235,284]^k=[4,14,75,154,216,221,285]^k \\
& [0,41,43,150,178,237,287]^k=[7,17,66,133,195,230,288]^k \\
& [0,43,70,167,197,268,321]^k=[13,15,112,121,230,252,323]^k \\
\label{k123456s391}
& [0,40,87,174,281,331,389]^k=[9,21,110,155,304,312,391]^k \\
& [0,79,100,201,229,392,393]^k=[1,68,117,184,240,385,399]^k \\
\label{k123456s399}
& [0,54,59,183,256,317,398]^k=[13,18,92,164,275,306,399]^k 
\end{align}
\begin{primesolu}
\end{primesolu}
\indent
$\bullet$ First known prime solution, based on non-symmetric solution \eqref{k123456s84}, by Qiu Min in 2016 \cite{Qiu2016}:
\begin{equation}
\begin{aligned}
& [83, 191, 197, 383, 419, 557, 569 ]^k \\
& \quad = [ 89, 149, 263, 317, 491, 503, 587]^k 
\end{aligned}
\end{equation}
\indent
$\bullet$ Symmetric prime solution, based on \eqref{k123456s102}, by Chen Shuwen in 2016:
\begin{equation}
\begin{aligned}
& [ 18443, 90263, 126173, 249863, 273803, 373553, 421433 ]^k\\
& \quad = [ 22433, 70313, 170063, 194003, 317693, 353603, 425423]^k 
\end{aligned}
\end{equation}
\begin{mirrortype}
\end{mirrortype}
$\bullet$ \ref{kn123456}\quad $(k=-6,-5,-4,-3,-2,-1)$
\begin{relatedtype}
\end{relatedtype}
$\bullet$ \ref{h1234568}\quad $(h=1,2,3,4,5,6,8)$\vspace{1ex}\\
\indent
$\bullet$ \ref{r123456}\quad $(r=1,2,3,4,5,6)$\vspace{1ex}\\
\indent
$\bullet$ \ref{r135}\quad $(r=1,3,5)$
\\
\subsubsection{\;\;(k$\;=\;$1, 2, 3, 4, 5, 7)}
\label{k123457}
\begin{notation}
\end{notation}
$\bullet$ Type: $(k=1,2,3,4,5,7)$ \qquad $\bullet$ Abbreviation: $[$k123457$]$
\begin{idealsolu}
\end{idealsolu}
$\bullet$ The first known solutions were obtained by Chen Shuwen in 1995, based on parametric method \cite{Chen01,Chen23}.
\begin{align}
& [4727, 4972, 5267, 5857, 5972, 6557, 6667 ]^k \nonumber \\
&\quad = [ 4772, 4867 ,5477, 5567, 6172, 6457, 6707]^k \\
& [1091, 4226, 4397, 10553, 11579, 17279, 20414 ]^k \nonumber\\
&\quad = [ 1889, 2003 , 7589, 7646, 12947, 16994, 20471]^k 
\end{align}
\indent
$\bullet$ The second known solutions were obtained by Chen Shuwen in 1999, still utilizing the parametric method.
\begin{align}
& [43, 169, 295, 607, 667, 1105, 1189 ]^k \nonumber\\
&\quad = [ 79, 97, 379, 505, 727, 1093, 1195]^k \\
& [535, 656, 809, 1099, 1168, 1451, 1513 ]^k \nonumber\\
&\quad = [560, 601, 919, 953, 1264, 1403, 1531]^k 
\end{align}
\indent
$\bullet$ The smallest solutions were obtained by Chen Shuwen in 2022, through a computer search.
\begin{align}
\label{k123457s152}
& [2,31,33,84,89,136,151]^k=[4,19,51,66,101,133,152]^k\\
& [33,59,62,93,130,151,155]^k=[34,51,75,85,137,143,158]^k\\
& [79,97,94,135,137,162,173]^k=[82,85,107,123,149,157,174]^k\\
& [67,85,104,152,171,205,219]^k=[72,75,115,139,184,197,221]^k\\
& [55,83,102,145,177,213,218]^k=[57,75,122,125,193,198,223]^k\\
& [21,51,80,139,176,229,235]^k=[25,40,99,119,191,216,241]^k\\
& [17,48,71,156,159,218,249]^k=[26,29,93,123,192,204,251]^k\\
& [43,90,89,177,179,239,262]^k=[47,67,129,134,218,219,265]^k
\end{align}
\begin{mirrortype}
\end{mirrortype}
$\bullet$ \ref{kn123457}\quad $(k=-7,-5,-4,-3,-2,-1)$
\\
\subsubsection{\;\;(k$\;=\;$1, 2, 3, 5, 7, 9)}
\label{k123579}
\begin{notation}
\end{notation}
$\bullet$ Type: $(k=1,2,3,5,7,9)$ \qquad $\bullet$ Abbreviation: $[$k123579$]$
\begin{idealsolu}
\end{idealsolu}
$\bullet$ The first known solutions, based on a selective search with Identity \eqref{identity13579}, were found by Chen Shuwen in 2023 \cite{Chen23}.
\begin{align}
\label{k123579s377}
& [7, 89, 91, 251, 253, 341, 373 ]^k = [ 29, 31, 151, 193, 311, 313, 377]^k \\
\label{k123579s941}
&[269, 397, 409, 683, 743, 901, 923 ] ^k \nonumber \\ 
&\quad = [ 299, 313, 493, 613, 827, 839, 941]^k 
\end{align}
\begin{mirrortype}
\end{mirrortype}
$\bullet$ \ref{kn123579}\quad $(k=-9,-7,-5,-3,-2,-1)$
\\
\subsubsection{\;\;(k$\;=\;$1, 2, 3, 4, 5, 6, 7)}
\label{k1234567}
\begin{notation}
\end{notation}
$\bullet$ Type: $(k=1,2,3,4,5,6,7)$ \qquad $\bullet$ Abbreviation: $[$k1234567$]$
\begin{idealsolu}
\end{idealsolu}
$\bullet$ This type is PTE of degree 7. The first known symmetric solution was found by G.Tarry in 1912 \cite{Tarry1912, Gloden44}:
\begin{align}
\label{k1234567s50}
& [0, 4, 9, 23, 27, 41, 46, 50 ]^k = [ 1, 2, 11, 20, 30, 39 , 48, 49]^k 
\end{align}
\indent
$\bullet$ 
J.Chernick gave a two-parameter symmetric solution of this type in 1937 \cite{Chernick37}. Numerical examples by his method are:
\begin{align}
\label{k1234567s68}
& [ 0, 9, 10, 27, 41, 58, 59, 68 ]^k = [ 2, 3, 19, 20, 48, 49, 65, 66 ]^k\\
\label{k1234567s100}
& [ 0, 14, 19, 43, 57, 81, 86, 100 ]^k = [ 1, 9, 30, 32, 68, 70, 91, 99 ]^k
\end{align}
\indent
$\bullet$ Ajai Choudhry obtained four one-parameter symmetric solutions in 2022 \cite{Choudhry22}. \vspace{+1ex}\\
\indent
$\bullet$ First known non-symmetric solutions, based on parametric method \eqref{n3m4k1to7}, by Chen Shuwen in 1997 \cite{Chen01,Chen23}:
\begin{align}
\label{k1234567s96}
& [0,14,15,43,46,73,89,96] ^k= [1,8,24,33,54,70,91,95] ^k\\
\label{k1234567s321}
& [ 0, 21, 82, 149, 155, 262, 278, 321 ]^k \nonumber \\
&\quad = [ 2, 17, 91, 126 , 174, 253, 285, 320 ] ^k
\end{align}
\indent
$\bullet$ It should be noted that each non-symmetric solution has an equivalent solution which can be obtained through simple transformation (subtracted by the largest item). For example, $[ 0, 7, 23, 50, 53, 81, 82, 96 ] ^k= [ 1, 5, 26, 42, 63, 72, 88, 95 ] ^k$ and \eqref{k1234567s96} are equivalent.\vspace{1ex}\\
\indent
$\bullet$ New non-symmetric solutions, based on computer search, by Chen Shuwen in 2019 and 2023:
\begin{align}
\label{k1234567s249}
&[0,16,63,121,127,202,214,249]^k \nonumber \\
&\quad =[6,7,81,88,154,184,225,247]^k\\
\label{k1234567s277}
&[0,28,44,137,149,211,246,277]^k \nonumber\\
&\quad =[4,16,57,116,182,189,253,275]^k \\
\label{k1234567s303}
&[0,41,46,115,157,224,282,303]^k \nonumber \\
&\quad =[3,24,72,95,170,217,286,301]^k \\	
\label{k1234567s320}
&[0,21,75,149,158,254,258,321]^k \nonumber \\
&\quad =[6,11,93,114,189,230,273,320 ]^k \\
&[0,75,87,212,233,336,411,458]^k \nonumber \\
&\quad =[3,51,138,152,296,297,420,455]^k \\
\label{k1234567s469}
&[0,50,111,233,236,307,438,469]^k \nonumber \\
&\quad =[1,46,119,203,282,285,440,468]^k											
\end{align}
\begin{primesolu}
\end{primesolu}
$\bullet$ First known prime solution, based on symmetric solution \eqref{k1234567s100}, by T.W.A. Baumann in 1999 \cite{Rivera65}:
\begin{align}		
\label{k1234567s171251}														
&[12251, 34511, 42461, 80621, 102881, 141041, 148991, 171251]^k \nonumber \\
&\quad =[13841, 26561, 59951, 63131, 120371, 123551, 156941, 169661]^k 	
\end{align}
\indent
$\bullet$ Non-symmetric prime solution, based on \eqref{k1234567s321}, by Chen Shuwen in 2016:
\begin{align}																
&[10289, 14699, 27509, 41579, 42839, 65309, 68669, 77699 ]^k \nonumber \\
&\quad =[10709, 13859, 29399, 36749, 46829, 63419, 70139, 77489]^k								
\end{align}
\begin{mirrortype}
\end{mirrortype}
$\bullet$ \ref{kn1234567}\quad $(k=-7,-6,-5,-4,-3,-2,-1)$
\begin{relatedtype}
\end{relatedtype}
$\bullet$ \ref{k246}\quad $(k=2,4,6)$ \vspace{1ex}\\
\indent
$\bullet$ \ref{h12345679}\quad $(h=1,2,3,4,5,6,7,9)$ \vspace{1ex}\\
\indent
$\bullet$ \ref{r1234567}\quad $(r=1,2,3,4,5,6,7)$ 
\\
\subsubsection{\;\;(k$\;=\;$1, 2, 3, 4, 5, 6, 8)}
\label{k1234568}
\begin{notation}
\end{notation}
$\bullet$ Type: $(k=1,2,3,4,5,6,8)$ \qquad $\bullet$ Abbreviation: $[$k1234568$]$
\begin{idealsolu}
\end{idealsolu}
$\bullet$ The first known solution, utilizing a parametric method, was discovered by Chen Shuwen in 1999 \cite{Chen01,Chen23}.
\begin{align}
\label{k1234568s501}
& [77, 159, 169, 283, 321, 443, 447, 501 ] ^k \nonumber \\
& \quad = [ 79, 137, 213, 237, 363, 399, 481, 491]^k 
\end{align}
\indent
$\bullet$ Using a computer search with \eqref{n3m4k1234568} in 2023, Chen Shuwen confirmed that \eqref{k1234568s501} is the smallest solution.
\begin{mirrortype}
\end{mirrortype}
$\bullet$  \ref{kn1234568}\quad $(k=-8,-6,-5,-4,-3,-2,-1)$
\\
\subsubsection{\;\;(k$\;=\;$1, 2, 3, 4, 5, 6, 7, 8)}
\label{k12345678}
\begin{notation}
\end{notation}
$\bullet$ Type: $(k=1,2,3,4,5,6,7,8)$ \qquad $\bullet$ Abbreviation: $[$k12345678$]$
\begin{idealsolu}
\end{idealsolu}
$\bullet$ This type is PTE of degree 8. So far only two ideal symmetric solutions, based on \eqref{k1357s99} and \eqref{k1357s174}, were found by A.Latec in 1942 \cite{Letac42} \cite[p.48]{Gloden44} \cite{Dorwart47}:
\begin{align}
\label{k12345678s198}
&[0, 24, 30, 83, 86, 133, 157, 181, 197]^k\nonumber \\
&\quad = [ 1, 17, 41, 65, 112, 115, 168, 174, 198]^k\\
\label{k12345678s348}
& [ 0, 26, 42, 124, 166, 237, 293, 335, 343 ]^k\nonumber \\
&\quad = [ 5, 13, 55, 111, 182, 224, 306, 322, 348 ]^k
\end{align}
\begin{primesolu}
\end{primesolu}
\indent
$\bullet$ First known prime solution, based on \eqref{k12345678s198}, by Chen Shuwen in 2016 \cite{Chen23}:
\begin{align}
\label{k12345678s15820583}	
& [ 3522263, 4441103, 5006543, 7904423, 9388703,\nonumber \\
&\qquad \qquad 11897843, 13876883, 15361163, 15643883 ] ^k\nonumber \\
&\quad = [ 3698963, 3981683, 5465963, 7445003, 9954143,\nonumber \\
&\qquad \qquad 11438423, 14336303, 14901743, 15820583 ]^k
\end{align}
\begin{relatedtype}
\end{relatedtype}
$\bullet$ \ref{kn12345678}\quad $(k=-8,-7,-6,-5,-4,-3,-2,-1)$
\\
\subsubsection{\;\;(k$\;=\;$1, 2, 3, 4, 5, 6, 7, 8, 9)}
\label{k123456789}
\begin{notation}
\end{notation}
$\bullet$ Type: $(k=1,2,3,4,5,6,7,8,9)$ \qquad $\bullet$ Abbreviation: $[$k123456789$]$
\begin{idealsolu}
\end{idealsolu}
$\bullet$ This type is PTE of degree 9. So far, only ideal symmetric solutions have been discovered. \\[1mm]
\indent
$\bullet$ The first known solution was obtained by A. Letac in 1942 \cite{Letac42} \cite{Gloden44}, and it was based on \eqref{k2468s23750}.
\begin{align}
\label{k123456789s47500}
& [ 0, 3083, 3301, 11893, 23314, 24186, 35607, 44199, 44417, 47500 ]^k \nonumber \\
& \quad = [ 12, 2865, 3519, 11869, 23738, 23762, 35631, 43981, 44635, 47488 ]^k 
\end{align}
\indent
$\bullet$ The second known solution, which is based on \eqref{k2468s529393533005}, was independently discovered by G. Palama in 1950 \cite{Palama1950} and C.J. Smyth in 1990 \cite{Smyth91}.
\begin{align}
\label{k123456789s1058787066010}
& [0, 96426061793, 191366410204, 339512836183, \nonumber \\
& \qquad\qquad 396167834716, 662619231294, 719274229827,  \nonumber \\
& \qquad\qquad 867420655806, 962361004217, 1058787066010]^k \nonumber \\
& \quad = [1485713382, 87647378809, 220873077098, 286306758615, \nonumber \\
& \qquad\qquad 441746154196, 617040911814, 772480307395, \nonumber \\
& \qquad\qquad  837913988912, 971139687201, 1057301352628]^k
\end{align}
\indent
$\bullet$ The two smallest solutions, based on \eqref{k2468s313} and \eqref{k2468s515}, were discovered by Peter Borwein, Petr Lisonek and Colin Percival in 2000 \cite{Borwein2003}. 
\begin{align}
\label{k123456789s626}
& [0, 12, 125, 213, 214, 412, 413, 501, 614, 626 ]^k \nonumber \\
& \quad =  [ 5, 6, 133, 182, 242, 384, 444, 493, 620, 621 ]^k \\
\label{k123456789s1030}
& [ 0, 63, 149, 326, 412, 618, 704, 881, 967, 1030 ]^k\nonumber \\
& \quad = [ 7, 44, 184, 270, 497, 533, 760, 846, 986, 1023 ]^k 
\end{align}
\indent
$\bullet$ The third smallest known solution, based on \eqref{k2468s4827}, was discovered by Chen Shuwen in 2023. \cite{Chen23}
\begin{align}
\label{k123456789s9654}
& [ 0, 364, 810, 1229, 4310, 5344, 8425, 8844, 9290, 9654  ]^k \nonumber \\
& \quad = [ 40, 260, 1044, 1054, 4329, 5325, 8600, 8610, 9394, 9614 ]^k 
\end{align}
\begin{primesolu}
\end{primesolu}
$\bullet$ 
The first known prime solution, based on \eqref{k123456789s626}, was obtained by Chen Shuwen in 2016.
\begin{align}
&[ 2589701, 2972741, 6579701, 9388661, 9420581,  \nonumber \\ 
&\qquad\qquad 15740741, 15772661, 18581621, 22188581, 22571621 ]^k  \nonumber \\
&\quad= [ 2749301, 2781221, 6835061, 8399141, 10314341, \nonumber \\
&\qquad\qquad  14846981, 16762181, 18326261, 22380101, 22412021 ]^k
\end{align}
\indent
$\bullet$ The second known prime solution, based on \eqref{k123456789s9654}, was obtained by Chen Shuwen in 2023.
\begin{align}
&[21270551 , 22323239 , 23613071 , 24824819, 33735071,  \nonumber \\
&\qquad\qquad 36725399, 45635651, 46847399, 48137231 , 49189919]^{k} \nonumber \\
&\quad = [ 21386231, 22022471, 24289799, 24318719, 33790019, \nonumber \\
&\qquad\qquad 36670451, 46141751, 46170671, 48437999, 49074239]^{k}
\end{align}
\begin{mirrortype}
\end{mirrortype}
$\bullet$ \ref{kn123456789}\quad $(k=-9,-8,-7,-6,-5,-4,-3,-2,-1)$
\begin{relatedtype}
\end{relatedtype}
$\bullet$ \ref{k2468}\quad $(k=2,4,6,8)$
\\
\subsubsection{\;\;(k$\;=\;$1, 2, 3, 4, 5, 6, 7, 8, 9, 10, 11)}
\label{k1234567891011}
\begin{notation}
\end{notation}
$\bullet$ Type: $(k=1, 2, 3, 4, 5, 6, 7, 8, 9, 10, 11)$ \vspace{1ex}\\
\indent
$\bullet$ Abbreviation: $[$k1234567891011$]$
\begin{idealsolu}
\end{idealsolu}
$\bullet$ This type is PTE of degree 11. At present only ideal symmetric solutions are known. \\[1mm]
\indent
$\bullet$ First known solution, based on \eqref{k246810s151}, by Nuutti Kuosa, Jean-Charles Meyrignac and Chen Shuwen in 1999 \cite{Chen22,Chen23}:
\begin{align}
\label{k1234567891011s302}
& [ 0, 11, 24, 65, 90, 129, 173, 212, 237, 278, 291, 302 ]^k \nonumber \\
& \quad = [ 3, 5, 30, 57, 104, 116, 186, 198, 245, 272, 297, 299]^k
\end{align}
\indent
$\bullet$  Second known solution, based on \eqref{k246810s2058}, by David Broadhurst in 2007 \cite{{Broadhurst2007},{Broadhurst2007b}}:
\begin{align}
\label{k1234567891011s4116}
& [0, 162, 440, 949, 1167, 1801, 2315, 2949, 3167, 3676, 3954, 4116 ]^k \nonumber \\
& \quad = [ 21, 111, 544, 764, 1419, 1586, 2530, 2697, 3352, 3572, 4005, 4095]^k
\end{align}
\indent
$\bullet$ In 2008, A.Choudhry and Jarosław Wróblewski proved that this system has infinitely many solutions \cite{CW2008}. The smallest new solution found by their method is based on \eqref{k246810s1511}:
\begin{align}
\label{k1234567891011s30226}
& [0, 373, 436, 811, 889, 1404, 1618, 2133, 2211, 2586, 2649, 3022]^k\nonumber\\
& \quad = [ 1, 331, 558, 625, 1098, 1218, 1804, 1924, 2397, 2464, 2691, 3021]^k
\end{align}
\begin{primesolu}
\end{primesolu}
$\bullet$ 10 prime solutions were found by Jarosław Wróblewski in 2023, all based on \eqref{k1234567891011s302}. The smallest one is:
\begin{align}
\label{k1234567891011s40538248081}	
&[32058169621, 32367046651, 32732083141, 33883352071, \nonumber \\ 
&\qquad\qquad 34585345321, 35680454791, 36915962911, 38011072381,\nonumber \\ &\qquad\qquad 38713065631, 39864334561, 40229371051, 40538248081]^k\nonumber\\
&\quad=[ 32142408811, 32198568271, 32900561521, 33658714231,\nonumber\\ &\qquad\qquad 34978461541, 35315418301, 37280999401, 37617956161,\nonumber\\ &\qquad\qquad 38937703471, 39695856181, 40397849431, 40454008891]^k
\end{align}
\begin{mirrortype}
\end{mirrortype}
$\bullet$ \ref{kn1234567891011}\quad $(k=-11,-10,-9,-8,-7,-6,-5,-4,-3,-2,-1)$
\begin{relatedtype}
\end{relatedtype}
$\bullet$ \ref{k246810}\quad $(k=2,4,6,8,10)$\vspace{1ex}\\
\indent
$\bullet$ \ref{r1234567891011}\quad $(r=1,2,3,4,5,6,7,8,9,10,11)$
\\
\subsection{GPTE with all k<0}
Ideal non-negative integer solutions of the GPTE problem have been identified for 42 types with all \( k < 0 \).
\subsubsection{\;\;(k$\;=\;$-1)}
\label{kn1}
\begin{notation}
\end{notation}
$\bullet$ Type: $(k=-1)$ \qquad $\bullet$ Abbreviation: $[$kn1$]$
\begin{idealsolu}
\end{idealsolu}
$\bullet$ Smallest known solution, based on \eqref{k1s3}:
\begin{align}
&[ 2, 6 ] = [ 3,3 ]^k
\end{align}
\begin{idealchain}
\end{idealchain}
$\bullet$ Based on the solution chains of $(k=1)$, there are solution chains of $(k=-1)$ for any length. For example:
\begin{align}
& [60,420]^k=[70,210]^k=[84,140]^k=[105,105]^k
\end{align}
\begin{mirrortype}
\end{mirrortype}
$\bullet$ \ref{k1}\quad $(k=1)$
\\
\subsubsection{\;\;(k$\;=\;$-2)}
\label{kn2}
\begin{notation}
\end{notation}
$\bullet$ Type: $(k=-2)$ \qquad $\bullet$ Abbreviation: $[$kn2$]$
\begin{idealsolu}
\end{idealsolu}
$\bullet$ Smallest known solution, based on \eqref{k2s11}:
\begin{align}
&[ 10, 55 ] ^k= [ 11, 22 ] ^k
\end{align}
\begin{idealchain}
\end{idealchain}
$\bullet$ Based on the solution chains of $(k=2)$, there are solution chains of $(k=-2)$ for any length. For example:
\begin{align}
& [646043435, 4522304045]^k=[660561265, 2556084895]^k \nonumber\\
&\quad =[691646501, 1679712931]^k =[744176615, 1250850055]^k \nonumber\\
&\quad =[904460809, 904460809] ^k
\end{align}
\begin{mirrortype}
\end{mirrortype}
$\bullet$ \ref{k2}\quad $(k=2)$ 
\\
\subsubsection{\;\;(k$\;=\;$-3)}
\label{kn3}
\begin{notation}
\end{notation}
$\bullet$ Type: $(k=-3)$ \qquad $\bullet$ Abbreviation: $[$kn3$]$
\begin{idealsolu}
\end{idealsolu}
$\bullet$ Smallest known solution, based on \eqref{k3s12}:
\begin{align}
&[  15, 180 ] ^k = [ 18, 20  ] ^k
\end{align}
\begin{idealchain}
\end{idealchain}
$\bullet$ Ideal solution chain, based on \eqref{k3s23237} :
\begin{align}
& [8627408354744952300, 46894757412680340724]^k \nonumber\\
&\quad =[8691367724755417350, 28363764564121173825]^k \nonumber\\
&\quad =[8878436135483102595, 19347142244664008550]^k \nonumber\\
&\quad =[9167089850437077900, 15493862581282050900]^k
\end{align}
\begin{mirrortype}
\end{mirrortype}
$\bullet$ \ref{k3}\quad $(k=3)$
\\
\subsubsection{\;\;(k$\;=\;$-4)}
\label{kn4}
\begin{notation}
\end{notation}
$\bullet$ Type: $(k=-4)$ \qquad $\bullet$ Abbreviation: $[$kn4$]$
\begin{idealsolu}
\end{idealsolu}
$\bullet$ Smallest known solution, based on \eqref{k4s158}:
\begin{align}
&[525749, 1407938 ]^k = [ 619913, 624574  ] ^k
\end{align}
\begin{mirrortype}
\end{mirrortype}
$\bullet$ \ref{k4}\quad $(k=4)$
\\
\subsubsection{\;\;(k$\;=\;$-2, -1)}
\label{kn12}
\begin{notation}
\end{notation}
$\bullet$ Type: $(k=-2,-1)$ \qquad $\bullet$ Abbreviation: $[$kn12$]$
\begin{idealsolu}
\end{idealsolu}
$\bullet$ Smallest known solution, based on \eqref{k12s4}:
\begin{align}
&[5,10,10] = [6,6,15] ^k
\end{align}
\begin{idealchain}
\end{idealchain}
$\bullet$ Based on the solution chains of $(k=1,2)$, there are solution chains of $(k=-2,-1)$ for any length. For example:
\begin{align}
& [12600, 32760, 36400]^k=[13104, 23400, 54600]^k \nonumber\\
&\quad =[13650, 20475, 65520]^k=[15600, 16380, 81900]^k
\end{align}
\begin{mirrortype}
\end{mirrortype}
$\bullet$ \ref{k12}\quad $(k=1,2)$
\\
\subsubsection{\;\;(k$\;=\;$-3, -1)}
\label{kn13}
\begin{notation}
\end{notation}
$\bullet$ Type: $(k=-3,-1)$ \qquad $\bullet$ Abbreviation: $[$kn13$]$
\begin{idealsolu}
\end{idealsolu}
$\bullet$ Smallest known solution, based on \eqref{k13s6}:
\begin{align}
&[5,10,15] = [6,6,30] ^k
\end{align}
\begin{idealchain}
\end{idealchain}
$\bullet$ Ideal solution chain, based on \eqref{k13s68} :
\begin{align}
& [1584850696, 4898629424, 4898629424]^k \nonumber \\
&\quad =[1608505184, 3367807729, 8289988256]^k \nonumber \\
&\quad =[1738223344, 2342822768, 26942461832]^k \nonumber \\
&\quad =[1858100816, 2072497064, 53884923664]^k
\end{align}
\begin{mirrortype}
\end{mirrortype}
$\bullet$ \ref{k13}\quad $(k=1,3)$
\\
\subsubsection{\;\;(k$\;=\;$-4, -1)}
\label{kn14}
\begin{notation}
\end{notation}
$\bullet$ Type: $(k=-4,-1)$ \qquad $\bullet$ Abbreviation: $[$kn14$]$
\begin{idealsolu}
\end{idealsolu}
$\bullet$ Smallest known solution, based on \eqref{k14s48}:
\begin{align}
&[3157, 5412, 13776 ] = [3608, 3696, 37884] ^k
\end{align}
\begin{idealchain}
\end{idealchain}
$\bullet$ Ideal solution chain, based on \eqref{k14s244} :
\begin{align}
& [810565866, 1594984446, 2709288648]^k \nonumber \\
&\quad =[813901528, 1498318722, 2996637444]^k \nonumber \\
&\quad =[915639219, 983970504, 8240752971]^k
\end{align}
\begin{mirrortype}
\end{mirrortype}
$\bullet$ \ref{k14}\quad $(k=1,4)$
\\
\subsubsection{\;\;(k$\;=\;$-5, -1)}
\label{kn15}
\begin{notation}
\end{notation}
$\bullet$ Type: $(k=-5,-1)$ \qquad $\bullet$ Abbreviation: $[$kn15$]$
\begin{idealsolu}
\end{idealsolu}
$\bullet$ Smallest known solution, based on \eqref{k15s66}:
\begin{align}
&[21216, 31824, 77792 ]^k = [ 21879, 27456, 107712 ] ^k
\end{align}
\begin{mirrortype}
\end{mirrortype}
$\bullet$ \ref{k15}\quad $(k=1,5)$
\\
\subsubsection{\;\;(k$\;=\;$-3, -2)}
\label{kn23}
\begin{notation}
\end{notation}
$\bullet$ Type: $(k=-3,-2)$ \qquad $\bullet$ Abbreviation: $[$kn23$]$
\begin{idealsolu}
\end{idealsolu}
$\bullet$ Smallest known solution, based on \eqref{k23s73}:
\begin{align}
&[6210, 10074, 25185 ]^k = [ 6570, 8395, 453330 ] ^k
\end{align}
\begin{mirrortype}
\end{mirrortype}
$\bullet$ \ref{k23}\quad $(k=2,3)$
\\
\subsubsection{\;\;(k$\;=\;$-4, -2)}
\label{kn24}
\begin{notation}
\end{notation}
$\bullet$ Type: $(k=-4,-2)$ \qquad $\bullet$ Abbreviation: $[$kn24$]$
\begin{idealsolu}
\end{idealsolu}
$\bullet$ Smallest known solution, based on \eqref{k24s11}:
\begin{align}
&[90, 165, 198 ]^k = [ 99, 110, 990 ] ^k
\end{align}
\begin{idealchain}
\end{idealchain}
$\bullet$ Based on the solution chains of $(k=2,4)$, there are solution chains of $(k=-4,-2)$ for any length. For example:
\begin{align}
& [257980650, 495322848, 538394400]^k \nonumber \\
&\quad =[263469600, 386970975, 825538080]^k \nonumber \\
&\quad =[275179360, 334677600, 1547883900]^k \nonumber \\
&\quad =[287978400, 309576780, 4127690400]^k
\end{align}
\begin{mirrortype}
\end{mirrortype}
$\bullet$ \ref{k24}\quad $(k=2,4)$
\\
\subsubsection{\;\;(k$\;=\;$-6, -2)}
\label{kn26}
\begin{notation}
\end{notation}
$\bullet$ Type: $(k=-6,-2)$ \qquad $\bullet$ Abbreviation: $[$kn26$]$
\begin{idealsolu}
\end{idealsolu}
$\bullet$ Smallest known solution, based on \eqref{k26s23}:
\begin{align}
&[6270, 9614, 14421 ]^k = [ 6555, 7590, 48070  ] ^k
\end{align}
\begin{mirrortype}
\end{mirrortype}
$\bullet$ \ref{k26}\quad $(k=2,6)$
\\
\subsubsection{\;\;(k$\;=\;$-3, -2, -1)}
\label{kn123}
\begin{notation}
\end{notation}
$\bullet$ Type: $(k=-3,-2,-1)$ \qquad $\bullet$ Abbreviation: $[$kn123$]$
\begin{idealsolu}
\end{idealsolu}
$\bullet$ Smallest known solution, based on \eqref{k123s11}:
\begin{align}
&[110, 165, 264, 1320]^k=[120, 132, 440, 660] ^k
\end{align}
\begin{idealchain}
\end{idealchain}
$\bullet$ Based on the solution chains of $(k=1,2,3)$, there are solution chains of $(k=-3,-2,-1)$ for any length. For example:
\begin{align}
& [13160160, 22857120, 23474880, 96507840]^k \nonumber\\
&\quad =[13362624, 19301568, 28952352, 86857056]^k \nonumber\\
&\quad =[13571415, 18095220, 32169280, 78960960]^k \nonumber\\
&\quad =[14476176, 15792192, 43428528, 57904704]^k
\end{align}
\begin{mirrortype}
\end{mirrortype}
$\bullet$ \ref{k123}\quad $(k=1,2,3)$
\\
\subsubsection{\;\;(k$\;=\;$-4, -2, -1)}
\label{kn124}
\begin{notation}
\end{notation}
$\bullet$ Type: $(k=-4,-2,-1)$ \qquad $\bullet$ Abbreviation: $[$kn124$]$
\begin{idealsolu}
\end{idealsolu}
$\bullet$ Smallest known solution, based on \eqref{k124s21}:
\begin{align}
&[190, 285, 285, 798]^k=[210, 210, 399, 665] ^k
\end{align}
\begin{idealchain}
\end{idealchain}
$\bullet$ Ideal solution chain, based on \eqref{k124s64}:
\begin{align}
&[83334492423, 136754038848, 144146149056, 380957679648]^k \nonumber\\
&\quad=[84657262144, 115943641632, 183910603968, 333337969692] ^k \nonumber\\
&\quad=[86022701856, 108845051328, 222225313128, 280705658688] ^k
\end{align}
\begin{mirrortype}
\end{mirrortype}
$\bullet$ \ref{k124}\quad $(k=1,2,4)$
\\
\subsubsection{\;\;(k$\;=\;$-5, -2, -1)}
\label{kn125}
\begin{notation}
\end{notation}
$\bullet$ Type: $(k=-5,-2,-1)$ \qquad $\bullet$ Abbreviation: $[$kn125$]$
\begin{mirrortype}
\end{mirrortype}
$\bullet$ Smallest known solution, based on \eqref{k125s58}:
\begin{align}
& [4860912, 6128976, 7619808, 13425376]^k \\
&\quad= [ 5034516, 5528096, 8810403, 12257952]^k
\end{align}
\begin{mirrortype}
\end{mirrortype}
$\bullet$ \ref{k125}\quad $(k=1,2,5)$
\\
\subsubsection{\;\;(k$\;=\;$-6, -2, -1)}
\label{kn126}
\begin{notation}
\end{notation}
$\bullet$ Type: $(k=-6,-2,-1)$ \qquad $\bullet$ Abbreviation: $[$kn126$]$
\begin{idealsolu}
\end{idealsolu}
$\bullet$ Smallest known solution, based on \eqref{k126s45}:
\begin{align}
&[24596, 33540, 42570, 276705]^k=[25740, 28380, 55341, 184470] ^k
\end{align}
\begin{mirrortype}
\end{mirrortype}
$\bullet$ \ref{k126}\quad $(k=1,2,6)$
\\
\subsubsection{\;\;(k$\;=\;$-4, -3, -1)}
\label{kn134}
\begin{notation}
\end{notation}
$\bullet$ Type: $(k=-4,-3,-1)$ \qquad $\bullet$ Abbreviation: $[$kn134$]$
\begin{idealsolu}
\end{idealsolu}
$\bullet$ Smallest known solution, based on \eqref{k134s252}:
\begin{align}
\label{kn134s15031403100}
& [ 178945275, 302645700, 322101495, 15031403100]^k \nonumber \\
&\quad= [ 188678700, 224349300, 835077950, 901884186]^k
\end{align}
\begin{mirrortype}
\end{mirrortype}
$\bullet$ \ref{k134}\quad $(k=1,3,4)$
\\
\subsubsection{\;\;(k$\;=\;$-5, -3, -1)}
\label{kn135}
\begin{notation}
\end{notation}
$\bullet$ Type: $(k=-5,-3,-1)$ \qquad $\bullet$ Abbreviation: $[$kn135$]$
\begin{idealsolu}
\end{idealsolu}
$\bullet$ Smallest known solution, based on \eqref{k135s24}:
\begin{align}
\label{kn135s50232}
&[12558, 16744, 18837, 50232]^k=[13104, 14352, 23184, 43056] ^k
\end{align}
\begin{idealchain}
\end{idealchain}
$\bullet$ Smallest known solution chains, based on \eqref{k135s313}, by Chen Shuwen in 2022:
\begin{align}
& [30508515969264333, 42066808362906327,\nonumber \\
& \qquad\qquad 49477541442382053, 124015136342593977]^k \nonumber \\
& \quad =[31104773610357447, 37743737147745993,\nonumber \\
& \qquad\qquad 59311586946457989, 107293994363817261]^k \nonumber \\
& \quad =[32152072385116957, 34978628199193173,\nonumber \\
& \qquad\qquad 77635491856745823, 81616799131450737]^k 
\end{align}
\begin{mirrortype}
\end{mirrortype}
$\bullet$ \ref{k135}\quad $(k=1,3,5)$
\\
\subsubsection{\;\;(k$\;=\;$-7, -3, -1)}
\label{kn137}
\begin{notation}
\end{notation}
$\bullet$ Type: $(k=-7,-3,-1)$ \qquad $\bullet$ Abbreviation: $[$kn137$]$
\begin{idealsolu}
\end{idealsolu}
$\bullet$ Smallest known solution, based on \eqref{k137s698}:
\begin{align}
& [ 841155212997500, 1055982623511250, \nonumber \\
&\qquad\qquad 1325341622285000, 3190904014523125]^k \nonumber \\
&\quad= [ 872401691935000, 939402141875608,\nonumber \\
&\qquad\qquad 1663247418335000, 2552723211618500]^k
\end{align}
\begin{mirrortype}
\end{mirrortype}
$\bullet$ \ref{k137}\quad $(k=1,3,7)$
\\
\subsubsection{\;\;(k$\;=\;$-4, -3, -2)}
\label{kn234}
\begin{notation}
\end{notation}
$\bullet$ Type: $(k=-4,-3,-2)$ \qquad $\bullet$ Abbreviation: $[$kn234$]$
\begin{idealsolu}
\end{idealsolu}
$\bullet$ Smallest known solution, based on \eqref{k234s1058}:
\begin{align}
& [10582456759752735, 13737716873396802, \nonumber \\
&\qquad\qquad 23037529324729205, 260377657019032410]^k \nonumber\\
&\quad= [11041656066882045, 12344254963416090,\nonumber\\
&\qquad\qquad 33421609706920578, 46265451453795015]^k
\end{align}
\begin{mirrortype}
\end{mirrortype}
$\bullet$ \ref{k234}\quad $(k=2,3,4)$
\\
\subsubsection{\;\;(k$\;=\;$-6, -4, -2)}
\label{kn246}
\begin{notation}
\end{notation}
$\bullet$ Type: $(k=-6,-4,-2)$ \qquad $\bullet$ Abbreviation: $[$kn246$]$
\begin{idealsolu}
\end{idealsolu}
$\bullet$ Smallest known solution, based on \eqref{k246s25}:
\begin{align}
&[ 7728, 9200, 12075, 96600]^k=[8050, 8400, 13800, 38640 ] ^k
\end{align}
\begin{mirrortype}
\end{mirrortype}
$\bullet$ \ref{k246}\quad $(k=2,4,6)$
\\
\subsubsection{\;\;(k$\;=\;$-4, -3, -2, -1)}
\label{kn1234}
\begin{notation}
\end{notation}
$\bullet$ Type: $(k=-4,-3,-2,-1)$ \qquad $\bullet$ Abbreviation: $[$kn1234$]$
\begin{idealsolu}
\end{idealsolu}
$\bullet$ Smallest known solution, based on \eqref{k1234s18}:
\begin{align}
& [684, 855, 1140, 3420, 4560]^k=[720, 760, 1368, 2280, 6840]^k
\end{align}
\begin{mirrortype}
\end{mirrortype}
$\bullet$ \ref{k1234}\quad $(k=1,2,3,4)$
\\
\subsubsection{\;\;(k$\;=\;$-5, -3, -2, -1)}
\label{kn1235}
\begin{notation}
\end{notation}
$\bullet$ Type: $(k=-5,-3,-2,-1)$ \qquad $\bullet$ Abbreviation: $[$kn1235$]$
\begin{idealsolu}
\end{idealsolu}
$\bullet$ Smallest known solution, based on \eqref{k1235s28}:
\begin{align}
& [11934,15912,19656,83538,111384]^k \nonumber \\
&\quad= [12376,13923,25704,41769,334152]^k
\end{align}
\begin{mirrortype}
\end{mirrortype}
$\bullet$ \ref{k1235}\quad $(k=1,2,3,5)$
\\
\subsubsection{\;\;(k$\;=\;$-6, -3, -2, -1)}
\label{kn1236}
\begin{notation}
\end{notation}
$\bullet$ Type: $(k=-6,-3,-2,-1)$ \qquad $\bullet$ Abbreviation: $[$kn1236$]$
\begin{idealsolu}
\end{idealsolu}
$\bullet$ Smallest known solution, based on \eqref{k1236s107}:
\begin{align}
& [2904930, 3572730, 5359095, 14801310, 31082751 ]^k \nonumber\\
&\qquad\qquad= [2932335, 3453639, 5864670, 11512130, 44403930]^k
\end{align}
\begin{mirrortype}
\end{mirrortype}
$\bullet$ \ref{k1236}\quad $(k=1,2,3,6)$
\\
\subsubsection{\;\;(k$\;=\;$-7, -3, -2, -1)}
\label{kn1237}
\begin{notation}
\end{notation}
$\bullet$ Type: $(k=-7, -3, -2, -1)$ \qquad $\bullet$ Abbreviation: $[$kn1237$]$
\begin{idealsolu}
\end{idealsolu}
$\bullet$ Smallest known solution, based on \eqref{k1237s503}:
\begin{align}
& [2840097782598520232, 3729945651819988712, 5452554139874258308,\nonumber\\
&\qquad\qquad 10660964064530266244, 31055851840153384276]^k \nonumber\\
&\quad =[2845755347902501348, 3672414356419166264, 5760359615512321277, \nonumber\\
&\qquad\qquad 9652494490858484302, 34843150845050138456]^k
\end{align}
\begin{mirrortype}
\end{mirrortype}
$\bullet$ \ref{k1237}\quad $(k=1,2,3,7)$
\\
\subsubsection{\;\;(k$\;=\;$-6, -4, -2, -1)}
\label{kn1246}
\begin{notation}
\end{notation}
$\bullet$ Type: $(k=-6,-4,-2,-1)$ \qquad $\bullet$ Abbreviation: $[$kn1246$]$
\begin{idealsolu}
\end{idealsolu}
$\bullet$ Smallest known solution, based on \eqref{k1246s17}:
\begin{align}
& [ 1680, 2040, 2380, 5712, 7140 ]^k = [ 1785, 1785, 2856, 4080, 9520 ]^k
\end{align}
\begin{mirrortype}
\end{mirrortype}
$\bullet$ \ref{k1246}\quad $(k=1,2,4,6)$
\\
\subsubsection{\;\;(k$\;=\;$-7, -5, -3, -1)}
\label{kn1357}
\begin{notation}
\end{notation}
$\bullet$ Type: $(k=-7,-5,-3,-1)$ \qquad $\bullet$ Abbreviation: $[$kn1357$]$
\begin{idealsolu}
\end{idealsolu}
$\bullet$ Smallest known solution, based on \eqref{k1357s55}:
\begin{align}
\label{kn1357s4493994285}
& [245126961, 299599619, 313534485, 1225634805, 1497998095]^k \nonumber\\
&\quad  = [254377035, 264352605, 364377915, 709578045, 4493994285]^k
\end{align}
\begin{mirrortype}
\end{mirrortype}
$\bullet$ \ref{k1357}\quad $(k=1,3,5,7)$
\\
\subsubsection{\;\;(k$\;=\;$-8, -6, -4, -2)}
\label{kn2468}
\begin{notation}
\end{notation}
$\bullet$ Type: $(k=-8, -6, -4, -2)$ \qquad $\bullet$ Abbreviation: $[$kn2468$]$
\begin{idealsolu}
\end{idealsolu}
$\bullet$ Smallest known solution, based on \eqref{k2468s313}:
\begin{align}
& [399914884007100, 415858334532300, 665815737735225, \nonumber \\
&\qquad\qquad 1251733586942223, 1264377360547700 ]^k \nonumber \\
&\quad =[406407008747475, 407730810078900, 695407548301235, \nonumber \\
&\qquad\qquad 955521822093300, 1763005052031300 ]^k
\end{align}
\begin{mirrortype}
\end{mirrortype}
$\bullet$ \ref{k2468}\quad $(k=2,4,6,8)$
\\
\subsubsection{\;\;(k$\;=\;$-5, -4, -3, -2, -1)}
\label{kn12345}
\begin{notation}
\end{notation}
$\bullet$ Type: $(k=-5,-4,-3,-2,-1)$ \qquad $\bullet$ Abbreviation: $[$kn12345$]$
\begin{idealsolu}
\end{idealsolu}
$\bullet$ Smallest known solutions, based on \eqref{k12345s16} and \eqref{k12345s22}:
\begin{align}
\label{kn12345s171360}
& [1008, 1224, 1428, 2856, 4284, 17136]^k \nonumber \\
&\quad  = [1071, 1071, 1904, 1904, 8568, 8568]^k \\
\label{kn12345s81900}
& [12600, 15600, 16380, 32760, 36400, 81900]^k \nonumber \\
&\quad  = [13104, 13650, 20475, 23400, 54600, 65520]^k
\end{align}
\begin{idealchain}
\end{idealchain}
$\bullet$ Based on the solution chains of $(k=1,2,3,4,5)$, there are solution chains of $(k=-5,-4,-3,-2,-1)$ for any length. For example:
\begin{align}
& [2801341015200, 3627377468400, 3722834770200, \nonumber\\
&\qquad\qquad 9431181417840, 10104837233400, 56587088507040  ]^k\nonumber\\
&\quad =[2829354425352, 3328652265120, 4160815331400, \nonumber\\
&\qquad\qquad 7445669540400, 13473116311200, 47155907089200 ]^k\nonumber\\
&\quad =[2887096352400, 3143727139280, 4638285943200, \nonumber\\
&\qquad\qquad 6287454278560, 17683465158450, 35366930316900 ]^k\nonumber\\
&\quad =[2947244193075, 3042316586400, 5052418616700,\nonumber \\
&\qquad\qquad 5658708850704,21764264810400, 28293544253520 ]^k
\end{align}
\begin{mirrortype}
\end{mirrortype}
$\bullet$ \ref{k12345}\quad $(k=1,2,3,4,5)$
\\
\subsubsection{\;\;(k$\;=\;$-6, -4, -3, -2, -1)}
\label{kn12346}
\begin{notation}
\end{notation}
$\bullet$ Type: $(k=-6,-4,-3,-2,-1)$ \qquad $\bullet$ Abbreviation: $[$kn12346$]$
\begin{idealsolu}
\end{idealsolu}
$\bullet$ Smallest known solution, based on \eqref{k12346s60}:
\begin{align}
& [1040060, 1300075, 1356600, 3670800, 4457400, 62403600]^k \nonumber\\
&\quad  = [1094800, 1114350, 1642200, 2496144, 10400600, 15600900]^k
\end{align}
\begin{mirrortype}
\end{mirrortype}
$\bullet$ \ref{k12346}\quad $(k=1,2,3,4,6)$
\\
\subsubsection{\;\;(k$\;=\;$-7, -4, -3, -2, -1)}
\label{kn12347}
\begin{notation}
\end{notation}
$\bullet$ Type: $(k=-7,-4,-3,-2,-1)$ \qquad $\bullet$ Abbreviation: $[$kn12347$]$
\begin{idealsolu}
\end{idealsolu}
$\bullet$ First known solution, based on \eqref{k12347s366}:
\begin{align}
& [3985023979526590, 4393128844899795, 4877989219086060, \nonumber \\
&\qquad\qquad 8839507736404436, 10966306590276180, 42897611073727410]^k \nonumber\\
&\quad =[4120109538154610, 4131781236562980, 5099716001771790, \nonumber \\
&\qquad\qquad 8240219076309220, 11762248197634935, 41671965043049484]^k
\end{align}
\begin{mirrortype}
\end{mirrortype}
$\bullet$ \ref{k12347}\quad $(k=1,2,3,4,7)$
\\
\subsubsection{\;\;(k$\;=\;$-7, -5, -3, -2, -1)}
\label{kn12357}
\begin{notation}
\end{notation}
$\bullet$ Type: $(k=-7,-5,-3,-2,-1)$ \qquad $\bullet$ Abbreviation: $[$kn12357$]$
\begin{idealsolu}
\end{idealsolu}
$\bullet$ Smallest known solution, based on \eqref{k12357s381}:
\begin{align}
& [11358362926625, 14473365468375, 14669614491675, \nonumber \\
&\qquad\qquad 26227492576025,62717917029625, 393412388640375]^k \nonumber \\
&\quad =[11540096733451, 12918018731475, 17956582054125, \nonumber \\
&\qquad\qquad 21530031219125, 110962468590875, 149225388794625 ]^k
\end{align}
\begin{mirrortype}
\end{mirrortype}
$\bullet$ \ref{k12357}\quad $(k=1,2,3,5,7)$
\\
\subsubsection{\;\;(k$\;=\;$-8, -6, -4, -2, -1)}
\label{kn12468}
\begin{notation}
\end{notation}
$\bullet$ Type: $(k=-8,-6,-4,-2,-1)$ \qquad $\bullet$ Abbreviation: $[$kn12468$]$
\begin{idealsolu}
\end{idealsolu}
$\bullet$ Smallest known solution, based on \eqref{k12468s36}:
\begin{align}
\label{kn12468s37849140}
& [1051365, 1220940, 1261638,2226420, 5407020, 37849140]^k \nonumber \\
& \quad  = [1081404, 1113210, 1401820,1992060, 9462285, 12616380]^k
\end{align}
\begin{mirrortype}
\end{mirrortype}
$\bullet$ \ref{k12468}\quad $(k=1,2,4,6,8)$
\\
\subsubsection{\;\;(k$\;=\;$-9, -7, -5, -3, -1)}
\label{kn13579}
\begin{notation}
\end{notation}
$\bullet$ Type: $(k=-9,-7,-5,-3,-1)$ \qquad $\bullet$ Abbreviation: $[$kn13579$]$
\begin{idealsolu}
\end{idealsolu}
$\bullet$ Smallest known solution, based on \eqref{k13579s463}:
\begin{align}
& [84989313999946265, 100639520158504145, 117462842931269017,\nonumber \\
&\qquad\quad 181336646921544335, 244410263242081495, 915117497255235365]^k\nonumber \\
&\quad  = [85358031197342995, 97642809880831565, 137108196452874985, \nonumber \\
&\qquad\quad 139046121491078165, 432418158043682645, 462941792729119067]^k
\end{align}
\begin{mirrortype}
\end{mirrortype}
$\bullet$ \ref{k13579}\quad $(k=1,3,5,7,9)$
\\
\subsubsection{\;\;(k$\;=\;$-10, -8, -6, -4, -2)}
\label{kn246810}
\begin{notation}
\end{notation}
$\bullet$ Type: $(k=-10,-8,-6,-4,-2)$ \qquad $\bullet$ Abbreviation: $[$solution chains$]$
\begin{idealsolu}
\end{idealsolu}
$\bullet$ Smallest known solution, based on \eqref{k246810s151}:
\begin{align}
& [730891803492235, 740904019978430, 893983362949180,\nonumber \\
&\qquad\qquad 1150765818264370, 2301531636528740, 3090628197624308]^k\nonumber \\
&\quad  = [716370774283780, 772657049406077, 851747928479140,\nonumber \\
&\qquad\qquad 1257813801358730, 1773311260931980, 4916908496220490]^k 
\end{align}
\begin{mirrortype}
\end{mirrortype}
$\bullet$ \ref{k246810}\quad $(k=2,4,6,8,10)$
\\
\subsubsection{\;\;(k$\;=\;$-6, -5, -4, -3, -2, -1)}
\label{kn123456}
\begin{notation}
\end{notation}
$\bullet$ Type: $(k=-6,-5,-4,-3,-2,-1)$ \qquad $\bullet$ Abbreviation: $[$kn123456$]$
\begin{idealsolu}
\end{idealsolu}
$\bullet$ Smallest known solution, based on \eqref{k123456s84}:
\begin{align}
& [37182145, 45762640, 46356960, 77597520, \nonumber \\
&\qquad\qquad 89237148, 209969760, 237965728]^k \nonumber \\
&\quad  = [37573536, 41993952, 54083120, 62622560, \nonumber \\
&\qquad\qquad 127481640, 137287920, 297457160]^k
\end{align}
\begin{mirrortype}
\end{mirrortype}
$\bullet$ \ref{k123456}\quad $(k=1,2,3,4,5,6)$
\\
\subsubsection{\;\;(k$\;=\;$-7, -5, -4, -3, -2, -1)}
\label{kn123457}
\begin{notation}
\end{notation}
$\bullet$ Type: $(k=-7,-5,-4,-3,-2,-1)$ \qquad $\bullet$ Abbreviation: $[$kn123457$]$
\begin{idealsolu}
\end{idealsolu}
$\bullet$ Smallest known solution, based on \eqref{k123457s152}:
\begin{align}
& [165238377843, 188843860392, 248675578536, 380548991396,\nonumber \\
&\qquad\qquad 492475165336, 1321907022744, 6279058358034]^k \nonumber \\
&\quad  = [166332671736, 184678187001, 282204870024, 299002778954,\nonumber \\
&\qquad\qquad 761097982792, 810201078456, 12558116716068]^k
\end{align}
\begin{mirrortype}
\end{mirrortype}
$\bullet$ \ref{k123457}\quad $(k=1,2,3,4,5,7)$
\\
\subsubsection{\;\;(k$\;=\;$-9, -7, -5, -3, -2, -1)}
\label{kn123579}
\begin{notation}
\end{notation}
$\bullet$ Type: $(k=-9,-7,-5,-3,-2,-1)$ \qquad $\bullet$ Abbreviation: $[$kn123579$]$
\begin{idealsolu}
\end{idealsolu}
$\bullet$ Smallest known solution, based on \eqref{k123579s377}:
\begin{align}
&       [1297752271862287613603, 1563107368984288914787,\nonumber \\
&\qquad\qquad 1573159506405409743821, 2534987598404572177867,\nonumber \\
&\qquad\qquad 3240083486702532651181, 15782342144905884849301,\nonumber \\
&\qquad\qquad 16870779534209738976839]^k \nonumber \\
&\quad =[1311669186305851019647, 1434758376809625895391,\nonumber \\
&\qquad\qquad 1933804768743408815527, 1949213571681603308081,\nonumber \\
&\qquad\qquad 5376402269143762970641, 5497220297663847531779,\nonumber \\
&\qquad\qquad 69893229498868918618333]^k
\end{align}
\begin{mirrortype}
\end{mirrortype}
$\bullet$ \ref{k123579}\quad $(k=1,2,3,5,7,9)$
\\
\subsubsection{\;\;(k$\;=\;$-7, -6, -5, -4, -3, -2, -1)}
\label{kn1234567}
\begin{notation}
\end{notation}
$\bullet$ Type: $(k=-7,-6,-5,-4,-3,-2,-1)$ \qquad $\bullet$ Abbreviation: $[$kn1234567$]$
\begin{idealsolu}
\end{idealsolu}
$\bullet$ Smallest known solution, based on \eqref{k1234567s68}:
\begin{align}
& [17193995, 19650280, 19967220, 27510392, \nonumber \\
&\qquad\qquad 39934440, 88426260, 95228280,309491910]^k\nonumber \\
&\quad  = [17685252, 17941560, 23357880, 23807070,\nonumber \\
&\qquad\qquad 51581985, 53824680,176852520, 206327940]^k
\end{align}
\begin{mirrortype}
\end{mirrortype}
$\bullet$ \ref{k1234567}\quad $(k=1,2,3,4,5,6,7)$
\\
\subsubsection{\;\;(k$\;=\;$-8, -6, -5, -4, -3, -2, -1)}
\label{kn1234568}
\begin{notation}
\end{notation}
$\bullet$ Type: $(k=-8,-6,-5,-4,-3,-2,-1)$ \qquad $\bullet$ Abbreviation: $[$kn1234568$]$
\begin{idealsolu}
\end{idealsolu}
$\bullet$ First known solution, based on \eqref{k1234568s501}:
\begin{align}
&       [4103980139754883042177329807, 4189301972182219487960642277,\nonumber \\ 
&\qquad 5050261274735958831351049963, 5551113632561012599749501199,\nonumber \\
&\qquad 8502338601770664867970755001, 9460348585068767951685769649,\nonumber \\
&\qquad 14708425172406186669409262301, 25507015805311994603912265003]^k\nonumber \\
&\quad =[4022064368502290566285566737, 4507951339193842446776440571,\nonumber \\
&\qquad 4548655188757669466611893759, 6277427565793294622146632197, \nonumber \\
&\qquad 7120333033991687539608017439, 11923397920826317004195674173,\nonumber \\
&\qquad 12673297161129858954145087643, 26169535696359059398819077081]^k
\end{align}
\begin{mirrortype}
\end{mirrortype}
$\bullet$ \ref{k1234568}\quad $(k=1,2,3,4,5,6,8)$
\\
\subsubsection{\;\;(k$\;=\;$-8, -7, -6, -5, -4, -3, -2, -1)}
\label{kn12345678}
\begin{notation}
\end{notation}
$\bullet$ Type: $(k=-8,-7,-6,-5,-4,-3,-2,-1)$ \vspace{1ex}\\
\indent
$\bullet$ Abbreviation: $[$kn12345678$]$
\begin{idealsolu}
\end{idealsolu}
$\bullet$ Smallest known solution, based on \eqref{k12345678s198}:
\begin{align}
&       [50997530167920, 57759633615600,  59739963910992,\nonumber \\ 
&\qquad\quad  85692571183800, 87852888104400, 145201301172550, \nonumber \\
&\qquad\quad 217801951758825,435603903517650,1306811710552950]^k\nonumber \\
&\quad =[51247518060900, 55609008959700, 63746912709900,\nonumber \\
&\qquad\quad  74674954888740,112413910585200, 116161040938040, \nonumber \\
&\qquad\quad 282553883362800,337241731755600,1493499097774800]^k
\end{align}
\begin{mirrortype}
\end{mirrortype}
$\bullet$ \ref{k12345678}\quad $(k=1,2,3,4,5,6,7,8)$
\\
\subsubsection{\;\;(k$\;=\;$-9, -8, -7, -6, -5, -4, -3, -2, -1)}
\label{kn123456789}
\begin{notation}
\end{notation}
$\bullet$ Type: $(k=-9,-8,-7,-6,-5,-4,-3,-2,-1)$ \vspace{1ex}\\
\indent
$\bullet$ Abbreviation: $[$kn123456789$]$
\begin{idealsolu}
\end{idealsolu}
$\bullet$ Smallest known solution, based on \eqref{k123456789s626}:
\begin{align}
&         [9964438965675492480, 10158236806174853760,  12435777994860882480,\nonumber \\
&\qquad   15066423339927607620, 15102727974481650048,  28883097278386565760,\nonumber \\
&\qquad   29016815321342059120, 48965875854764724765, 417842140627325651328, \nonumber \\
&\qquad 2089210703136628256640]^k\nonumber \\
&\quad = [10044282226618405080, 10060404669999815040,  12636355059294122520,\nonumber \\
&\qquad   14021548343198847360, 16195431807260684160,  25582171875142386816,\nonumber \\
&\qquad   33879092483296674432, 46085530216249152720, 696403567712209418880,\nonumber \\
&\qquad  783454013676235596240]^k
\end{align}
\begin{mirrortype}
\end{mirrortype}
$\bullet$ \ref{k123456789}\quad $(k=1,2,3,4,5,6,7,8,9)$
\\
\subsubsection{\;\;(k$\;=\;$-11, -10, -9, -8, -7, -6, -5, -4, -3, -2, -1)}
\label{kn1234567891011}
\begin{notation}
\end{notation}
$\bullet$ Type: $(k=-11,-10,-9,-8,-7,-6,-5,-4,-3,-2,-1)$ \vspace{1ex}\\
\indent
$\bullet$ Abbreviation: $[$kn1234567891011$]$
\begin{idealsolu}
\end{idealsolu}
$\bullet$ Smallest known solution, based on \eqref{k1234567891011s302}:
\begin{align}
& [2697089735898166980, 2797659183677420664, 2926629287889500340,\nonumber \\
&\qquad 3424520577530452680, 3820877125855736555, 4662765306129034440, \nonumber \\
&\qquad 6205334279585256360, 8779887863668501020, 11961006654852740520,\nonumber \\
&\qquad 29475337828029967710, 55020630612322606392, 206327364796209773970]^k \nonumber \\
&\quad = [ 2723793594669435960, 2741891890979531880, 2990251663713185130,\nonumber \\
&\qquad 3314495820019434120, 4085690392004153940, 4343733995709679452, \nonumber \\
&\qquad 6877578826540325799, 7641754251711473110, 13529663265325231080, \nonumber \\
&\qquad 24273807623083502820, 91701051020537677320, 117901351312119870840 ]^k
\end{align}
\begin{mirrortype}
\end{mirrortype}
$\bullet$ \ref{k1234567891011}\quad $(k=1,2,3,4,5,6,7,8,9,10,11)$
\\
\subsection{GPTE with k$_{1}$=0 and all others k>0}
To date, ideal non-negative integer solutions of the GPTE problem have been identified for 24 distinct types with \( k_1 = 0 \) and all others \( k > 0 \).
\subsubsection{\;\;(k$\;=\;$0)}
\label{k0}
\begin{notation}
\end{notation}
$\bullet$ Type: $(k=0)$ \qquad $\bullet$ Abbreviation: $[$k0$]$
\begin{idealsolu}
\end{idealsolu}
$\bullet$ Smallest solutions:
\begin{align}
& [ 1, 4 ]^k = [ 2, 2 ]^k\\
& [ 1, 6 ]^k = [ 2, 3 ]^k
\end{align}
\begin{idealchain}
\end{idealchain}
$\bullet$ It is obvious that there are solution chains for any length. 
\begin{align}
\label{k0s24}
& [ 1, 24 ]^k = [ 2, 12 ]^k = [ 3, 8 ]^k = [ 4, 6 ]^k
\end{align}
\begin{mirrortype}
\end{mirrortype}
$\bullet$ None.
\\
\subsubsection{\;\;(k$\;=\;$0, 1)}
\label{k01}
\begin{notation}
\end{notation}
$\bullet$ Type: $(k=0,1)$ \qquad $\bullet$ Abbreviation: $[$k01$]$
\begin{idealsolu}
\end{idealsolu}
$\bullet$ Smallest solutions:
\begin{align}
\label{k01s8}
 & [2,6,6]^k=[3,3,8]^k\\
\label{k01s9}
 & [1,6,6]^k=[2,2,9]^k
\end{align}
\indent
$\bullet$ Ideal solution, by A.Moessner in 1939 \cite{Moessner39}:
\begin{align}
& [4, 15, 15 ]^k = [ 5, 9, 20 ]^k
\end{align}
\indent
$\bullet$ Ideal solution, by A.Gloden in 1944 \cite{Gloden44}:
\begin{align}
& [2, 12, 15 ]^k = [ 3, 6, 20]^k
\end{align}
\begin{idealchain}
\end{idealchain}
$\bullet$ Solution chain of length 3, by T.N.Sinha in 1984 \cite{Sinha84}:
\begin{align}
&[ 15, 42, 48 ]^k = [ 16, 35, 54 ]^k = [ 21, 24, 60 ]^k
\end{align}
\indent
$\bullet$ Solution chains of length 4, by J.G.Mauldon in 1981 \cite [p.271]{Guy04}:
\begin{align}
\label{k01s72}
&[14,50,54]^k = [15, 40, 63]^k = [18, 30, 70]^k = [21, 25, 72]^k\\
\label{k01s105}
&[6,56,75]^k = [7,40,90]^k = [9,28,100]^k = [12,20,105]^k
\end{align}
\begin{mirrortype}
\end{mirrortype}
$\bullet$ \ref{kn01}\quad $(k=-1,0)$
\\
\subsubsection{\;\;(k$\;=\;$0, 2)}
\label{k02}
\begin{notation}
\end{notation}
$\bullet$ Type: $(k=0,2)$ \qquad $\bullet$ Abbreviation: $[$k02$]$
\begin{idealsolu}
\end{idealsolu}
$\bullet$ This system was posed as a problem by U.~Bini in 1909. Partial solutions were offered by Dubouis and Mathieu \cite {Kelly1991} \cite [p.214]{Guy04}. \vspace{1ex}\\
\indent
$\bullet$ A.Gloden gave a six-parameter solution in 1944 \cite [p.36-39]{Gloden44}. Numerical example:
\begin{align}
& [3,16,20]^k = [5,8,24]^k
\end{align}
\indent
$\bullet$ Ideal solution, by G.Xeroudakes and A.Moessner in 1958 \cite{Xeroudakes58}:
\begin{align}
& [ 13, 38, 51]^k = [ 17, 26, 57 ]^k
\end{align}
\indent
$\bullet$ John B.~Kelly gave a general solution in 1991 \cite {Kelly1991}.\\[1mm]
\indent
$\bullet$ Smallest solutions, based on computer search, by Chen Shuwen:
\begin{align}
\label{k02s15}
 & [1,10,12]^k=[2,4,15]^k\\
\label{k02s22}
 & [2,11,20]^k=[4,5,22]^k
\end{align}
\begin{idealchain}
\end{idealchain}
$\bullet$ Ajai Choudhry and Jarosław Wróblewski found several parametric solutions for chains of length 3 in 2016 \cite{Choudhry17k02}, and obtained all numerical solution chains of length 3 and length 4 with terms below 30000 by exhaustive search, two numerical solution chains of length 5 by selective search. Numerical examples:
\begin{align}
\label{k02s143}
&[17,104,110 ]^k = [22,65,136]^k = [ 34,40,143]^k\\
\label{k02s196}
&[7,126,156]^k = [9,84,182]^k = [18,39,196]^k\\
\label{k02s250}
&[5,140,210]^k = [10,60,245]^k = [21,28,250]^k\\
\label{k02s4455}
&[567,3116,3630]^k = [594,2772,3895]^k \nonumber \\
& \quad = [855,1694,4428]^k = [902,1596,4455]^k \\
\label{k02s6545}
& [527,4246,5320]^k = [616,3230,5983]^k \nonumber \\
& \quad = [952,1930,6479]^k = [1178,1544,6545]^k \\
\label{k02s250582040}
& [ 40299792,177048625,203908744]^k \nonumber \\
& \quad =[42043040,158526615,218290540]^k \nonumber \\
& \quad =[42949660,152187360,222583465]^k \nonumber \\
& \quad =[54655952,108953000,244317489]^k \nonumber \\
& \quad =[70402865,82469040,250582040]^k
\end{align}
\begin{mirrortype}
\end{mirrortype}
$\bullet$ \ref{kn02}\quad $(k=-2,0)$
\begin{relatedtype}
\end{relatedtype}
$\bullet$ \ref{h1235}\quad $(h=1,2,3,5)$ \vspace{1ex}\\
\indent
$\bullet$ \ref{h02n1}\quad $(h=-1,0,2)$
\\
\subsubsection{\;\;(k$\;=\;$0, 3)}
\label{k03}
\begin{notation}
\end{notation}
$\bullet$ Type: $(k=0,3)$ \qquad $\bullet$ Abbreviation: $[$k03$]$
\begin{idealsolu}
\end{idealsolu}
$\bullet$ Stephane Vandemergel found 62 solutions before 1994. \cite{Guy04} \vspace{1ex}\\
\indent
$\bullet$ Smallest solutions:
\begin{align}
\label{k03s78}
 & [4,60,65]^k=[8,25,78]^k\\
\label{k03s80}
 & [6,48,75]^k=[10,27,80]^k\\
\label{k03s90}
 & [4,63,80]^k=[7,32,90]^k
\end{align}
\indent
$\bullet$ Ideal solutions with $a_1=1$, based on computer search, by Chen Shuwen in 2023:
\begin{align}
 & [1,108,120]^k=[9,10,144]^k\\
 & [1,6592,9682]^k=[64,94,10609]^k\\
 & [1,10950,21600]^k=[73,144,22500]^k
\end{align}
\begin{mirrortype}
\end{mirrortype}
$\bullet$ \ref{kn03}\quad $(k=-3,0)$
\\
\subsubsection{\;\;(k$\;=\;$0, 4)}
\label{k04}
\begin{notation}
\end{notation}
$\bullet$ Type: $(k=0,4)$ \qquad $\bullet$ Abbreviation: $[$k04$]$
\begin{idealsolu}
\end{idealsolu}
$\bullet$ Stephane Vandemergel noted that if $a^4+b^4=c^4+d^4$, then $\{ac,bc,d^2;ad,bd, c^2\}$ is a solution for $(k=0,4)$, which shows that there are infinitely many solutions for $(k=0,4)$. He gave 3 numerical solutions before 1994 \cite [p.214]{Guy04}.  
\begin{align}
\label{k04s124}
& [22,93,116]^k=[29,66,124]^k\\
\label{k04s196}
& [28,122,189]^k=[54,61,196]^k\\
& [9,511,589]^k=[19,217,657]^k
\end{align}
\indent
$\bullet$ Same method as Stephane Vandemergel noted, based on \eqref{k4s158} and \eqref{k4s239}, found independently by Chen Shuwen in 2001:
\begin{align}
& [3481,21014,21172]^k=[7847,7906,24964]^k\\
& [49,37523,54253]^k=[1099,1589,57121]^k
\end{align}
\indent
$\bullet$ Ajai Choudhry obtained a two-parameter solution in 2017 \cite{Choudhry17k04}.
\vspace{1ex}\\
\indent
$\bullet$ By computer search in 2016, Chen Shuwen confirmed that \eqref{k04s124} is the smallest solution, and obtained the remained solutions in the range of 1000:
\begin{align}
\label{k04s171}
 & [7,133,153]^k=[17,49,171]^k\\
\label{k04s316}
 & [79,176,312]^k=[96,143,316]^k\\
\label{k04s352}
 & [77,276,320]^k=[115,168,352]^k\\
 & [67,274,381]^k=[127,137,402]^k\\
 & [19,415,539]^k=[77,95,581]^k\\
 & [7,456,527]^k=[51,56,589]^k\\
 & [133,594,788]^k=[197,378,836]^k\\
 & [31,594,779]^k=[82,209,837]^k\\
 & [408,737,931]^k=[469,616,969]^k
\end{align}
\begin{mirrortype}
\end{mirrortype}
$\bullet$ \ref{kn04}\quad $(k=-4,0)$
\begin{relatedtype}
\end{relatedtype}
$\bullet$ \ref{k4}\quad $(k=4)$
\\
\subsubsection{\;\;(k$\;=\;$0, 1, 2)}
\label{k012}
\begin{notation}
\end{notation}
$\bullet$ Type: $(k=0,1,2)$ \qquad $\bullet$ Abbreviation: $[$k012$]$
\begin{idealsolu}
\end{idealsolu}
$\bullet$ Ideal solution, by A.Gloden in 1944 \cite{Gloden44}:
\begin{align}
\label{k012s22}
 & [2, 2, 11, 21 ]^k = [ 1, 6, 7, 22]^k
\end{align}
\indent
$\bullet$ George Xeroudakes and Alfred Moessner obtained a two-parameter solution in 1958  \cite{Xeroudakes58}. Numerical example:
\begin{align}
 & [ 3, 14, 20, 40 ]^k = [ 4, 8, 30, 35 ]^k
\end{align}
\indent
$\bullet$ Ajai Choudhry obtained two parametric solutions in 2013 \cite[p.771-772]{Choudhry13}, one of which is complete. Numerical examples:
\begin{align}
 & [ 5,10,12,21]^k = [6,7,15,20]^k \\
 & [19,48,50,80]^k = [20,38,64,75]^k 
\end{align}
\indent
$\bullet$ Smallest solutions, based on computer search, by Chen Shuwen in 2023:
\begin{align}
 & [  2,6,7,15]^k = [  3,3,10,14 ]^k\\
 & [  4,7,10,15 ]^k = [ 5,5,12,14 ]^k\\
\label{k012s16}
 & [ 3 , 6, 10 , 16]^k = [ 4, 4 , 12, 15]^k\\
\label{k012s20}
 & [3, 9 , 10, 20]^k = [4, 5 , 15, 18]^k
\end{align}
\begin{idealchain}
\end{idealchain}
$\bullet$ First known solution chains, by Chen Shuwen in 2023 \cite{ChenkProducts23}:
\begin{align}
\label{k012s65}
 & [9,28,30,65]^k = [10,20,39,63]^k= [13,14,45,60]^k\\
\label{k012s66}
 & [9,28,35,66]^k = [10,21,44,63]^k= [11,18,49,60]^k
\end{align}
\begin{mirrortype}
\end{mirrortype}
$\bullet$ \ref{kn012}\quad $(k=-2,-1,0)$
\\
\subsubsection{\;\;(k$\;=\;$0, 1, 3)}
\label{k013}
\begin{notation}
\end{notation}
$\bullet$ Type: $(k=0,1,3)$ \qquad $\bullet$ Abbreviation: $[$k013$]$
\begin{idealsolu}
\end{idealsolu}
$\bullet$ First known solution, based on \eqref{k13s12}, by Chen Shuwen in 2001 \cite{ChenkProducts23}:
\begin{align}
\label{k013s24}
 & [ 5, 11, 16, 24 ]^k = [ 6, 8, 20, 22 ]^k
\end{align}
\indent
$\bullet$ Ajai Choudhry obtained a six-parameter solution in 2013 \cite[p.787-788]{Choudhry13}. Numerical example:
\begin{align}
\label{k013s24}
 & [11,31,39,94]^k = [13,22,47,93]^k
\end{align}
\indent
$\bullet$ Small solutions based on computer search, by Chen Shuwen in 2017:
\begin{align}
 & [ 7, 12, 25, 32 ]^k=[ 8, 10, 28, 30 ]^k\\
\label{k013s35}
 & [3,10,17,35 ]^k= [5,5,21,34 ]^k\\
 & [2,14,21,40 ]^k= [3,7,32,35]^k\\
 & [5,13,30,42 ]^k= [6,10,35,39]^k\\
 & [1,10,24,44]^k= [2,4,33,40]^k 
\end{align}
\indent
$\bullet$ Chen Shuwen obtained a two-parameter solution in 2023. Numerical example:
\begin{align}
 &[140,332,487,988]^k=[166,247,560,974]^k
\end{align}
\begin{mirrortype}
\end{mirrortype}
$\bullet$ \ref{kn013}\quad $(k=-3,-1,0)$
\\
\subsubsection{\;\;(k$\;=\;$0, 1, 4)}
\label{k014}
\begin{notation}
\end{notation}
$\bullet$ Type: $(k=0,1,4)$ \qquad $\bullet$ Abbreviation: $[$k014$]$
\begin{idealsolu}
\end{idealsolu}
$\bullet$ First known solutions, smallest solutions, based on computer search, by Chen Shuwen in 2017 \cite{ChenkProducts23}:
\begin{align}
\label{k014s84}
 & [ 5, 36, 40, 84 ]^k=[ 7, 18, 60, 80 ]^k \\
\label{k014s168}
 & [ 16, 79, 88 , 168 ]^k=[ 21 , 44, 128, 158]^k\\
 & [ 22, 84,95, 184]^k=[ 30 , 46, 133 , 176]^k\\
 &  [21,111,166,248]^k=[24,83,217,222]^k\\
 &  [16,107,147,252]^k=[21,63,214,224]^k\\
 &  [11,86,150,272]^k=[17,44,200,258]^k
\end{align}
\indent
$\bullet$ Chen Shuwen found a parameter solution in 2023. Numerical examples:
\begin{align}
 &  [5849, 25894, 59961, 121476]^k=[10123, 12947, 70188, 119922]^k\\
 &  [64735, 182214, 233960, 419727]^k \nonumber \\
 &  \quad =[90629, 105282, 299805, 404920]^k \\
 &  [458470, 3373646, 3759269, 7895419]^k \nonumber \\
 &  \quad =[641858, 1686823, 5639585, 7518538]^k 
\end{align} 
\begin{mirrortype}
\end{mirrortype}
$\bullet$ \ref{kn014}\quad $(k=-4,-1,0)$
\\
\subsubsection{\;\;(k$\;=\;$0, 1, 5)}
\label{k015}
\begin{notation}
\end{notation}
$\bullet$ Type: $(k=0,1,5)$ \qquad $\bullet$ Abbreviation: $[$k015$]$
\begin{idealsolu}
\end{idealsolu}
$\bullet$ First known solutions, smallest solutions, based on computer search, by Chen Shuwen in 2017 \cite{ChenkProducts23}:
\begin{align}
\label{k015s155}
 & [ 13, 66, 91, 155 ]^k=[ 21, 31, 130, 143 ]^k  \\
\label{k015s301}
 & [ 84, 161, 169, 301]^k=[ 91, 129, 196, 299 ] ^k
\end{align}
\indent
$\bullet$ Third known solution, based on \eqref{k15s756} of $(k=1,5)$, by Chen Shuwen in 2023:
\begin{align}
\label{k015s10584}
 & [725,3755,3864,10584]^k=[1350,1380,5684,10514]^k
\end{align}
\indent
$\bullet$ Chen Shuwen noticed that \eqref{k015s301} also satisfies $a_1 a_4=b_2 b_3 $ and $a_2 a_3=b_1 b_4$, then constructed a parametric method and obtained a new solution in 2023:
\begin{align}
 & [71,189,305,555]^k=[105,111,355,549]^k
\end{align}
\begin{mirrortype}
\end{mirrortype}
$\bullet$ \ref{kn015}\quad $(k=-5,-1,0)$
\begin{relatedtype}
\end{relatedtype}
$\bullet$ \ref{k15}\quad $(k=1,5)$
\\
\subsubsection{\;\;(k$\;=\;$0, 2, 3)}
\label{k023}
\begin{notation}
\end{notation}
$\bullet$ Type: $(k=0,2,3)$ \qquad $\bullet$ Abbreviation: $[$k023$]$
\begin{idealsolu}
\end{idealsolu}
$\bullet$ First known solution, based on $(k=2,3)$, by Chen Shuwen in 2022 \cite{ChenkProducts23}:
\begin{align}
\label{k023s16335}
 & [855, 4338, 10406, 16335 ]^k = [ 1215, 2838, 11495, 15906 ]^k
\end{align}
\indent
$\bullet$ Smallest solution, based on computer search, by Chen Shuwen in 2023:
\begin{align}
\label{k023s2040}
 & [16,306,1617,2040 ]^k = [33,144,1734,1960]^k
\end{align}
\begin{mirrortype}
\end{mirrortype}
$\bullet$ \ref{kn023}\quad $(k=-3,-2,0)$
\begin{relatedtype}
\end{relatedtype}
$\bullet$ \ref{k23}\quad $(k=2,3)$
\\
\subsubsection{\;\;(k$\;=\;$0, 2, 4)}
\label{k024}
\begin{notation}
\end{notation}
$\bullet$ Type: $(k=0,2,4)$ \qquad $\bullet$ Abbreviation: $[$k024$]$
\begin{idealsolu}
\end{idealsolu}
$\bullet$ Alfred Moessner and George Xeroudakes found a parametric method in 1976 \cite{Moessner76}, with additonal condition $b_1=b_2$ and $b_3=b_4$. Numerical example:
\begin{align}
& [3,243,289,529]^k = [23,23,459,459]^k
\end{align}
\indent
$\bullet$ Numerical solution, based on \eqref{k24s69}, by Chen Shuwen in 2001:
\begin{align}
\label{k024s207}
 & [19, 67, 159, 207 ]^k = [ 23, 53, 171, 201 ]^k
\end{align}
\indent
$\bullet$ Numerical solution, based on Ajai Choudhry's solution \eqref{h0124s134} in 2013:
\begin{align}
 & [9,62,67,138 ]^k = [ 18,23,93,134]^k
\end{align}
\indent
$\bullet$ Smallest solutions, based on computer search, by Chen Shuwen in 2023:
\begin{align}
\label{k024s28}
 & [1,9,18,28]^k=[2,4,21,27]^k\\
 & [1,17,18,36]^k=[3,4,27,34]^k\\
\label{k024s36}
 & [2,17,21,36]^k=[3,9,28,34]^k \\
\label{k024s39}
 & [2,18,19,39]^k=[3,9,26,38]^k\\
 & [4,19,27,42]^k=[7,9,36,38]^k 
\end{align}
\indent
$\bullet$ Solutions satisfy $a_1+a_2+a_4=3a_3$ and $(a_2-a_3)(2a_3-a_4) \neq 0$, which lead to ideal solution \eqref{k012468s1233} of $(k=0,1,2,4,6,8)$, based on computer search, by Chen Shuwen in 2023:
\begin{align}
\label{k024s1159}
& [423,828,902,1159]^k=[437,732,1034,1107]^k\\
\label{k024s1233}
& [497,754,828,1233]^k=[548,621,923,1218]^k
\end{align}
\begin{idealchain}
\end{idealchain}
$\bullet$ First known solution chains, based on computer search, by Chen Shuwen in 2023 \cite{ChenkProducts23}:
\begin{align}
\label{k024s1150}
 & [15,328,837,1150]^k = [23,200,930,1107]^k= [62,72,1025,1035]^k \\
\label{k024s2905}
 & [141, 1403 , 1643 , 2905]^k = [155 , 1113, 1909, 2867]^k= [345 , 427 , 2491, 2573]^k
\end{align}
\begin{mirrortype}
\end{mirrortype}
$\bullet$ \ref{kn024}\quad $(k=-4,-2,0)$
\begin{relatedtype}
\end{relatedtype}
$\bullet$ \ref{k012468}\quad $(k=0,1,2,4,6,8)$   \vspace {1ex} \\
\indent
$\bullet$ \ref{h0124}\quad\: $(h=0,1,2,4)$   \vspace {1ex} \\
\indent
$\bullet$ \ref{h024n1}\quad\: $(h=-1,0,2,4)$   
\\
\subsubsection{\;\;(k$\;=\;$0, 2, 6)}
\label{k026}
\begin{notation}
\end{notation}
$\bullet$ Type: $(k=0,2,6)$ \qquad $\bullet$ Abbreviation: $[$k026$]$
\begin{idealsolu}
\end{idealsolu}
$\bullet$ First known solution, smallest solution, based on computer search, by Chen Shuwen in 2017 \cite{ChenkProducts23}:
\begin{align}
\label{k026s1887}
 & [155, 779, 1455, 1887 ]^k = [ 185, 615, 1581, 1843 ]^k
\end{align}
\indent
$\bullet$ Chen Shuwen noticed that \eqref{k026s1887} also satisfies $a_1 a_4=b_1 b_3$ and $a_2 a_3=b_2 b_4$, then constructed a parametric method and obtained a new solution in 2023:
\begin{align}
\label{k026s3597}
 & [850, 2154 , 2825, 3597]^k = [935, 1795 , 3270, 3390]^k
\end{align}
\begin{mirrortype}
\end{mirrortype}
$\bullet$ \ref{kn026}\quad $(k=-6,-2,0)$
\\
\subsubsection{\;\;(k$\;=\;$0, 1, 2, 3)}
\label{k0123}
\begin{notation}
\end{notation}
$\bullet$ Type: $(k=0,1,2,3)$ \qquad $\bullet$ Abbreviation: $[$k0123$]$
\begin{idealsolu}
\end{idealsolu}
$\bullet$ Guo Xianqiang studied this system in 2000, and obtained an non-ideal solution. \vspace{1ex}\\
\indent
$\bullet$ First known ideal solution, by Chen Shuwen in 2000-2001 \cite{ChenkProducts23}:
\begin{align}
\label{k0123s52}
 & [ 4, 13, 17, 40, 50 ] ^k= [ 5, 8, 25, 34, 52 ]^k\\
\label{k0123s81}
 & [ 54, 60, 63, 77, 80 ]^k = [ 56, 56, 66, 75, 81 ]^k
\end{align}
\indent
$\bullet$ Smallest solutions, by Chen Shuwen in 2022:
\begin{align}
 & [4,7,11,20,20]^k=[5,5,14,16,22] ^k\\
 & [1,6,8,22,32]^k = [2,2,16,16,33]^k\\
 & [6,11,13,28,32]^k=[7,8,16,26,33] ^k\\
 & [10,14,21,32,33]^k = [11,12,24,28,35]^k\\
 & [15,20,21,32,34]^k = [16,17,24,30,35]^k
\end{align}
\indent
$\bullet$ Solutions of $(k=0,1,2,3)$, which lead to the smallest solutions of $(k=-3,-2,-1,0)$, by Chen Shuwen in 2022:
\begin{align}
\label{k0123s60}
 & [15,25,30,56,56]^k=[16,21,35,50,60] ^k\\
\label{k0123s48}
 & [8,15,24,42,45]^k = [9,12,30,35,48]^k\\
\label{k0123s88}
 & [7,18,33,72,84]^k=[9,12,42,63,88] ^k
\end{align}
\indent
$\bullet$ Parameter solution \eqref{parak0123} were derived by Chen Shuwen in 2023. Numerical example:
\begin{align}
\label{k0123s98}
[18, 29, 42, 91, 95]^k=[19, 26, 45, 87, 98]^k 
\end{align}
\begin{mirrortype}
\end{mirrortype}
$\bullet$ \ref{kn0123}\quad $(k=-3,-2,-1,0)$
\\
\subsubsection{\;\;(k$\;=\;$0, 1, 2, 4)}
\label{k0124}
\begin{notation}
\end{notation}
$\bullet$ Type: $(k=0,1,2,4)$ \qquad $\bullet$ Abbreviation: $[$k0124$]$
\begin{idealsolu}
\end{idealsolu}
$\bullet$ First known solution, smallest solution, by Chen Shuwen in 2017 \cite{ChenkProducts23}:
\begin{align}
\label{k0124s22}
 &  [ 3, 6, 11, 20, 20 ]^k=[ 4, 4, 15, 15, 22 ]^k \\
\label{k0124s28}
 & [ 7, 12, 14, 24, 27 ]^k=[  8, 9, 18, 21, 28 ]^k 
\end{align}
\begin{mirrortype}
\end{mirrortype}
$\bullet$ \ref{kn0124}\quad $(k=-4,-2,-1,0)$
\\
\subsubsection{\;\;(k$\;=\;$0, 1, 2, 6)}
\label{k0126}
\begin{notation}
\end{notation}
$\bullet$ Type: $(k=0,1,2,6)$ \qquad $\bullet$ Abbreviation: $[$k0126$]$
\begin{idealsolu}
\end{idealsolu}
$\bullet$ First known solution, smallest solution, based on computer search, by Chen Shuwen in 2017 \cite{ChenkProducts23}. This solution satisfies $a_1+a_2=b_1+b_2$.
\begin{align}
\label{k0126s360}
 & [ 205, 248, 255, 344, 351 ]^k = [ 215, 221, 279, 328, 360 ]^k
\end{align}
\indent
$\bullet$ Second known solution, based on computer search, by Chen Shuwen in 2022:
\begin{align}
\label{k0126s410}
 & [20, 41 , 154, 345, 408]^k = [24 , 33 , 161, 340, 410]^k
\end{align}
\begin{mirrortype}
\end{mirrortype}
$\bullet$ \ref{kn0126}\quad $(k=-6,-2,-1,0)$
\\
\subsubsection{\;\;(k$\;=\;$0, 2, 4, 6)}
\label{k0246}
\begin{notation}
\end{notation}
$\bullet$ Type: $(k=0,2,4,6)$ \qquad $\bullet$ Abbreviation: $[$k0246$]$
\begin{idealsolu}
\end{idealsolu}
$\bullet$ First known solutions, smallest solutions, based on computer search, by Chen Shuwen in 2017 \cite{ChenkProducts23}:
\begin{align}
\label{k0246s50}
 & [ 2, 16, 25, 45, 48 ]^k = [ 3, 9, 32, 40, 50 ]^k\\
\label{k0246s88}
 & [ 8, 33, 38, 68, 87 ]^k = [ 12, 17, 57, 58, 88 ]^k\\
\label{k0246s110}
 & [11,42,58,100,105]^k = [14,28,75,87,110]^k\\
 & [6,37,41,87,110]^k = [11,15,58,82,111]^k\\
 & [2,43,52,99,114]^k = [3,22,76,86,117]^k\\
 & [7,56,57,102,116]^k = [8,38,84,87,119]^k\\
 & [14,44,56,93,123]^k = [21,24,77,82,124]^k\\
 & [8,50,51,100,123]^k = [12,24,82,85,125]^k
\end{align}
\begin{mirrortype}
\end{mirrortype}
$\bullet$ \ref{kn0246}\quad $(k=-6,-4,-2,0)$
\begin{relatedtype}
\end{relatedtype}
$\bullet$ \ref{h01246}\quad\: $(h=0,1,2,4,6)$
\\
\subsubsection{\;\;(k$\;=\;$0, 1, 2, 3, 4)}
\label{k01234}
\begin{notation}
\end{notation}
$\bullet$ Type: $(k=0,1,2,3,4)$ \qquad $\bullet$ Abbreviation: $[$k01234$]$
\begin{idealsolu}
\end{idealsolu}
$\bullet$ First known solution, based on parametric method, by Chen Shuwen in 2001 \cite{ChenkProducts23}:
\begin{align}
\label{k01234s228}
 & [169, 175, 192, 215, 216, 228 ]^k = [ 171, 172, 195, 208, 224, 225 ]^k
\end{align}
\indent
$\bullet$ Smallest solutions, based on computer search, by Chen Shuwen in 2022:
\begin{align}
\label{k01234s57}
 & [13,20,22,45,45,57 ]^k = [ 15,15,27,38,52,55]^k \\
 & [23,28,39,54,56,66 ]^k = [ 24,26,44,46,63,63]^k \\
 & [3,11,14,36,58,69 ]^k = [4,6,23,29,63,66]^k \\
\label{k01234s72}
 & [5,13,20,51,54,72 ]^k = [6,9,27,40,65,68]^k \\
 & [26,34,35,60,63,72 ]^k = [ 27,30,39,56,68,70]^k 
\end{align}
\begin{mirrortype}
\end{mirrortype}
$\bullet$ \ref{kn01234}\quad $(k=-4,-3,-2,-1,0)$
\\
\subsubsection{\;\;(k$\;=\;$0, 1, 2, 3, 5)}
\label{k01235}
\begin{notation}
\end{notation}
$\bullet$ Type: $(k=0,1,2,3,5)$ \qquad $\bullet$ Abbreviation: $[$k01235$]$
\begin{idealsolu}
\end{idealsolu}
$\bullet$ First known solution, smallest solution, based on computer search, by Chen Shuwen in 2017 \cite{ChenkProducts23}:
\begin{align}
\label{k01235s245}
 & [ 28, 57, 100, 174, 200, 245 ]^k = [ 30, 49, 125, 140, 228, 232]^k
\end{align}
\indent
$\bullet$ Small solutions, based on computer search, by Chen Shuwen in 2023:
\begin{align}
\label{k01235s252}
 & [9,40,48,126,221,252]^k = [12,21,81,104,238,240]^k \\
\label{k01235s387}
 & [24,54,130,260,315,387]^k = [30,39,162,215,360,364]^k \\
 & [100,208,210,391,407,494]^k = [102,176,259,338,460,475]^k \\
 & [236,301,318,468,483,602]^k = [252,258,371,413,516,598]^k \\
 & [363,430,437,572,574,616]^k = [364,418,451,552,602,605]^k 
\end{align}
\indent
$\bullet$ Largest known solution, based on selective search, by Chen Shuwen in 2023:
\begin{align}
 & [527,610,648,990,1089,1221]^k = [549,561,682,968,1110,1215]^k
\end{align}
\begin{mirrortype}
\end{mirrortype}
$\bullet$ \ref{kn01235}\quad $(k=-5,-3,-2,-1,0)$
\\
\subsubsection{\;\;(k$\;=\;$0, 1, 2, 4, 6)}
\label{k01246}
\begin{notation}
\end{notation}
$\bullet$ Type: $(k=0,1,2,4,6)$ \qquad $\bullet$ Abbreviation: $[$k01246$]$
\begin{idealsolu}
\end{idealsolu}
$\bullet$ First known solution, smallest solution, by Chen Shuwen in 2017 \cite{ChenkProducts23}:
\begin{align}
\label{k01246s46}
 & [5, 12, 15, 32, 36, 46 ]^k = [ 6, 8, 23, 24, 40, 45]^k\\
\label{k01246s66}
 & [8, 16, 24, 45, 55, 66 ]^k = [ 10, 11, 33, 36, 60, 64]^k
\end{align}
\indent $\bullet$ Chen Shuwen obtained 54 solutions in range of 300 by computer search in 2023. Numerical examples:
\begin{align}
 & [16,26,32,55,55,68 ]^k = [17,22,40,44,64,65]^k\\
 & [9,17,30,52,56,70 ]^k = [10,14,39,40,63,68]^k\\
 & [25,37,40,64,66,78 ]^k = [26,32,48,55,74,75]^k\\
 & [41,48,56,72,77,86 ]^k = [43,44,63,64,82,84]^k\\
 & [9,21,28,58,66,88 ]^k = [11,14,42,44,72,87]^k\\
 & [13,21,39,63,85,95 ]^k = [15,17,45,57,91,91]^k\\
 & [66,120,132,235,250,299 ]^k = [69,100,165,200,282,286]^k\\
 & [44,85,128,231,238,300 ]^k = [48,70,168,176,275,289]^k
\end{align}
\begin{mirrortype}
\end{mirrortype}
$\bullet$ \ref{kn01246}\quad $(k=-6,-4,-2,-1,0)$
\\
\subsubsection{\;\;(k$\;=\;$0, 1, 2, 3, 4, 5)}
\label{k012345}
\begin{notation}
\end{notation}
$\bullet$ Type: $(k=0,1,2,3,4,5)$ \qquad $\bullet$ Abbreviation: $[$k012345$]$
\begin{idealsolu}
\end{idealsolu}
$\bullet$ First known solution, based on parametric method, by Chen Shuwen in 2001 \cite{ChenkProducts23}:
\begin{align}
\label{k012345s672}
& [480, 497, 558, 595, 616, 663, 666 ] ^k \nonumber\\
& \quad = [ 481, 495, 568, 578, 630, 651, 672]^k
\end{align}
\indent
$\bullet$ Ajai Choudhry obtained a parametric solution in integers in 2013 \cite [p.778-780]{Choudhry13}, of which numerical solution has negative integers.\vspace{+1ex}\\
\indent
$\bullet$ Smallest known solutions, based on computer search, by Chen Shuwen in 2023:
\begin{align}
& [28, 33, 44, 62, 64, 81, 87] ^k= [29, 31, 48, 54, 72, 77, 88] ^k\\
& [16, 25, 35, 62, 78, 105, 108] ^k= [18, 20, 45, 50, 91, 93, 112] ^k\\
& [55, 64, 65, 87, 92, 108, 114] ^k= [57, 58, 72, 80, 99, 104, 115] ^k\\
& [56, 64, 68, 87, 91, 114, 118] ^k= [58, 59, 76, 78, 96, 112, 119] ^k\\
& [50, 60, 65, 84, 96, 111, 119] ^k= [51, 56, 74, 75, 104, 105, 120] ^k\\
\label{k012345s124}
& [14, 20, 39, 70, 93, 117, 120] ^k= [15, 18 , 42, 65, 104, 105, 124] ^k
\end{align}
\begin{mirrortype}
\end{mirrortype}
$\bullet$ \ref{kn012345}\quad $(k=-5,-4,-3,-2,-1,0)$
\\
\subsubsection{\;\;(k$\;=\;$0, 1, 2, 3, 4, 6)}
\label{k012346}
\begin{notation}
\end{notation}
$\bullet$ Type: $(k=0,1,2,3,4,6)$ \qquad $\bullet$ Abbreviation: $[$k012346$]$
\begin{idealsolu}
\end{idealsolu}
$\bullet$ First known solutions, based on selective search, by Chen Shuwen in 2017 \cite{ChenkProducts23}:
\begin{align}
\label{k012346s92}
 & [32, 40, 41, 69, 72, 88, 90 ]^k = [ 33, 36, 45, 64, 80, 82, 92]^k\\
\label{k012346s222}
 & [14, 33, 37, 108, 112, 192, 221 ] ^k= [ 17, 21, 52, 88, 128, 189, 222]^k
\end{align}
\indent
$\bullet$ Chen Shuwen obtained all 14 solutions in range of 300 by computer search in 2023. Numerical examples:
\begin{align}
 & [31,44,55,80,104,128,142]^k = [32,40,64,71,110,124,143]^k\\
 & [16,29,36,72,94,122,143]^k = [18,22,47,61,104,116,144]^k\\
 & [62,79,83,136,142,176,184]^k = [64,71,92,124,158,166,187]^k\\
 & [51,69,78,124,130,190,192]^k = [52,64,85,114,138,186,195]^k\\
 & [4,20,25,86,104,177,198]^k = [6,9,44,65,118,172,200]^k\\
 & [59,76,93,144,146,187,218]^k = [62,68,109,118,171,176,219]^k
\end{align}
\begin{mirrortype}
\end{mirrortype}
$\bullet$ \ref{kn012346}\quad $(k=-6,-4,-3,-2,-1,0)$
\\
\subsubsection{\;\;(k$\;=\;$0, 1, 2, 4, 6, 8)}
\label{k012468}
\begin{notation}
\end{notation}
$\bullet$ Type: $(k=0,1,2,4,6,8)$ \qquad $\bullet$ Abbreviation: $[$k012468$]$
\begin{idealsolu}
\end{idealsolu}
$\bullet$ First known solution, based on \eqref{k024s1159} or \eqref{k024s1233} of $(k=0,2,4)$, and applying Piezas's Theorem \cite[p.623]{Piezas09}, by Chen Shuwen in 2023 \cite{ChenkProducts23}:
\begin{align}
\label{k012468s1233}
& [437, 497, 732, 754, 1034, 1107, 1233 ] ^k \nonumber \\
& \quad = [423, 548, 621, 902, 923, 1159, 1218]^k
\end{align}
\begin{mirrortype}
\end{mirrortype}
$\bullet$ \ref{kn012468}\quad $(k=-8,-6,-4,-2,-1,0)$
\begin{relatedtype}
\end{relatedtype}
$\bullet$ \ref{k024}\quad $(k=0,2,4)$
\\
\subsubsection{\;\;(k$\;=\;$0, 1, 2, 3, 4, 5, 6)}
\label{k0123456}
\begin{notation}
\end{notation}
$\bullet$ Type: $(k=0,1,2,3,4,5,6)$ \qquad $\bullet$ Abbreviation: $[$k0123456$]$
\begin{idealsolu}
\end{idealsolu}
$\bullet$ First known solution, based on \eqref{k1234568s501}, by Chen Shuwen in 2001 \cite{ChenkProducts23}:
\begin{align}
& [1899, 1953, 1957, 2079, 2117, 2231, 2241, 2323] ^k \nonumber \\
& \quad = [1909, 1919, 2001, 2037, 2163, 2187, 2263, 2321]^k
\end{align}
\indent
$\bullet$ Ajai Choudhry obtained a numerical solution in integers in 2013 \cite [p.780]{Choudhry13}, which has negative integers. \vspace{+1ex}\\
\indent
$\bullet$ Smallest ideal solutions, based on computer search, by Chen Shuwen in 2017:
\begin{align}
\label{k0123456s111}
 & [11, 19, 22, 54, 60, 90, 92, 111 ]^k = [ 12, 15, 27, 46, 74, 76, 99, 110]^k\\
\label{k0123456s110}
 & [7, 15, 21, 50, 53, 88, 96, 110 ]^k = [ 8, 11, 32, 33, 70, 75, 105, 106]^k
\end{align}
\indent
$\bullet$ Solutions based on computer search, by Chen Shuwen in 2023:
\begin{align}
& [43, 61, 68, 99, 132, 170, 174, 205]^k \nonumber \\
& \quad = [44, 55, 82, 86, 145, 153, 183, 204]^k\\
& [48, 66, 68, 111, 142, 187, 205, 225]^k \nonumber \\
& \quad = [51, 55, 82, 100, 153, 176, 213, 222]^k\\
& [11, 22, 36, 84, 129, 185, 210, 235]^k \nonumber \\
& \quad = [14, 15, 45, 74, 141, 172, 220, 231]^k\\
& [230, 236, 267, 268 , 305 , 330 , 361, 363]^k \nonumber \\
& \quad = [253, 228, 242, 285 , 295 , 335 , 356, 366]^k\\
& [30 , 58, 65, 159, 220 , 332, 369, 418]^k \nonumber \\
& \quad = [33, 44 , 82, 145, 234, 318, 380, 415]^k
\end{align}
\indent
$\bullet$ Solution based on parametric method, by Chen Shuwen in 2023:
\begin{align}
& [790, 850, 913, 1111, 1159, 1360, 1369, 1528]^k \nonumber \\
& \quad = [808, 814, 955, 1045, 1258, 1264, 1411, 1525]^k
\end{align}
\begin{mirrortype}
\end{mirrortype}
$\bullet$ \ref{kn0123456}\quad $(k=-6,-5,-4,-3,-2,-1,0)$
\begin{relatedtype}
\end{relatedtype}
$\bullet$ \ref{k1234568}\quad $(k=1,2,3,4,5,6,8)$
\\
\subsubsection{\;\;(k$\;=\;$0, 1, 2, 3, 4, 5, 6, 7)}
\label{k01234567}
\begin{notation}
\end{notation}
$\bullet$ Type: $(k=0,1,2,3,4,5,6,7)$ \qquad $\bullet$ Abbreviation: $[$k01234567$]$
\begin{idealsolu}
\end{idealsolu}
$\bullet$ First known solution, based on computer search with parametric method \eqref{n4m4k0to7}, by Chen Shuwen in 2023 \cite{ChenkProducts23}:
\begin{align}
\label{k01234567s610}
&  [ 382, 402, 403, 456, 485, 549, 559, 596, 608]^k \nonumber \\
& \quad =[387, 388, 416, 447, 494, 536, 573, 589, 610] ^k
\end{align}
\begin{mirrortype}
\end{mirrortype}
$\bullet$ \ref{kn01234567}\quad $(k=-7,-6,-5,-4,-3,-2,-1,0)$
\\
\subsection{GPTE with k$_{n}$=0 and all others k<0}
Ideal non-negative integer solutions of GPTE have been found for 23 types with \( k_n = 0 \) and all others \( k < 0 \).
\subsubsection{\;\;(k$\;=\;$-1, 0)}
\label{kn01}
\begin{notation}
\end{notation}
$\bullet$ Type: $(k=-1,0)$ \qquad $\bullet$ Abbreviation: $[$k0n1$]$
\begin{idealsolu}
\end{idealsolu}
$\bullet$ Smallest known solutions, based on \eqref{k01s8} and \eqref{k01s9}:
\begin{align}
 & [3,8,8]^k = [4,4,12]^k\\
 & [2,9,9]^k = [3,3,18]^k
\end{align}
\begin{idealchain}
\end{idealchain}
$\bullet$ Solution chains of length 4, based on \eqref{k01s105} and \eqref{k01s72}:
\begin{align}
& [120,630,1050]^k = [126,450,1400]^k \nonumber\\
& \quad =[140,315,1800]^k = [168,225,2100]^k\\
& [525, 1512, 1800]^k = [540, 1260, 2100]^k \nonumber \\
& \quad =[600, 945, 2520]^k = [700, 756, 2700]^k
\end{align}
\begin{mirrortype}
\end{mirrortype}
$\bullet$ \ref{k01}\quad $(k=0,1)$
\\
\subsubsection{\;\;(k$\;=\;$-2, 0)}
\label{kn02}
\begin{notation}
\end{notation}
$\bullet$ Type: $(k=-2,0)$ \qquad $\bullet$ Abbreviation: $[$k0n2$]$
\begin{idealsolu}
\end{idealsolu}
$\bullet$ Smallest known solutions, based on \eqref{k02s15} and \eqref{k02s22}:
\begin{align}
 & [4,15,30]^k=[5,6,60]^k\\
 & [10,44,55]^k = [11,20,110]^k
\end{align}
\begin{idealchain}
\end{idealchain}
$\bullet$ Ideal solution chains, based on \eqref{k02s143}, \eqref{k02s196} and \eqref{k02s4455}:
\begin{align}
& [680, 2431, 2860]^k=[ 715, 1496, 4420]^k=[ 884, 935, 5720]^k\\
& [117, 588, 1274]^k=[ 126, 273, 2548]^k=[ 147, 182, 3276]^k \\
& [239932, 669735, 1185030]^k=[241395, 630990, 1250172]^k \nonumber \\
& \quad =[274428, 385605, 1799490]^k =[294462, 343035, 1885180]^k 
\end{align}
\begin{mirrortype}
\end{mirrortype}
$\bullet$ \ref{k02}\quad $(k=0,2)$
\\
\subsubsection{\;\;(k$\;=\;$-3, 0)}
\label{kn03}
\begin{notation}
\end{notation}
$\bullet$ Type: $(k=-3,0)$ \qquad $\bullet$ Abbreviation: $[$k0n3$]$
\begin{idealsolu}
\end{idealsolu}
$\bullet$ Smallest known solutions, based on \eqref{k03s78} and \eqref{k03s80}:
\begin{align}
 & [100, 312, 975]^k=[ 120, 130, 1950]^k\\
 & [135, 400, 1080]^k=[ 144, 225, 1800]^k
\end{align}
\begin{mirrortype}
\end{mirrortype}
$\bullet$ \ref{k03}\quad $(k=0,3)$
\\
\subsubsection{\;\;(k$\;=\;$-4, 0)}
\label{kn04}
\begin{notation}
\end{notation}
$\bullet$ Type: $(k=-4,0)$ \qquad $\bullet$ Abbreviation: $[$k0n4$]$
\begin{idealsolu}
\end{idealsolu}
$\bullet$ Smallest known solutions, based on \eqref{k04s124}, \eqref{k04s196}, \eqref{k04s316}, \eqref{k04s171} and \eqref{k04s352}:
\begin{align}
 & [957, 1798, 4092]^k=[ 1023, 1276, 5394]^k\\
 & [1647,5292,5978]^k=[1708,2646,11529]^k\\
 & [3432,7584,11297]^k=[3476,6162,13728]^k\\
 & [833, 2907, 8379]^k=[ 931,1071, 20349]^k\\
 & [4830,10120,14784]^k=[5313,6160,22080]^k
\end{align}
\begin{mirrortype}
\end{mirrortype}
$\bullet$ \ref{k04}\quad $(k=0,4)$
\\
\subsubsection{\;\;(k$\;=\;$-2, -1, 0)}
\label{kn012}
\begin{notation}
\end{notation}
$\bullet$ Type: $(k=-2,-1,0)$ \qquad $\bullet$ Abbreviation: $[$k0n12$]$
\begin{idealsolu}
\end{idealsolu}
$\bullet$ Smallest known solutions, based on \eqref{k012s16} and \eqref{k012s20}:
\begin{align}
 & [10, 12, 36, 45]^k=[ 9, 18, 20, 60]^k\\
 & [16, 20, 60, 60]^k=[ 15, 24, 40, 80]^k
\end{align}
\begin{idealchain}
\end{idealchain}
$\bullet$ Ideal solution chains, based on \eqref{k012s65} and \eqref{k012s66}:
\begin{align}
& [273, 364, 1170, 1260]^k = [260, 420, 819, 1638]^k  \nonumber\\
& \quad =[252, 546, 585, 1820]^k\\
& [1617, 1980, 5390, 8820]^k = [1540, 2205, 4620, 9702]^k  \nonumber\\
& \quad =[1470, 2772, 3465, 10780]^k
\end{align}
\begin{mirrortype}
\end{mirrortype}
$\bullet$ \ref{k012}\quad $(k=0,1,2)$
\\
\subsubsection{\;\;(k$\;=\;$-3, -1, 0)}
\label{kn013}
\begin{notation}
\end{notation}
$\bullet$ Type: $(k=-3,-1,0)$ \qquad $\bullet$ Abbreviation: $[$k0n13$]$
\begin{idealsolu}
\end{idealsolu}
$\bullet$ Smallest known solutions, based on \eqref{k013s24} and \eqref{k013s35}:
\begin{align}
 & [120, 132, 330, 440]^k = [110, 165, 240, 528]^k\\
 & [105, 170, 714, 714]^k = [102, 210, 357, 1190]^k
\end{align}
\begin{mirrortype}
\end{mirrortype}
$\bullet$ \ref{k013}\quad $(k=0,1,3)$
\\
\subsubsection{\;\;(k$\;=\;$-4, -1, 0)}
\label{kn014}
\begin{notation}
\end{notation}
$\bullet$ Type: $(k=-4,-1,0)$ \qquad $\bullet$ Abbreviation: $[$k0n14$]$
\begin{idealsolu}
\end{idealsolu}
$\bullet$ Smallest known solutions, based on \eqref{k014s84} and \eqref{k014s168}:
\begin{align}
 & [ 60, 126, 140, 1008]^k = [63, 84, 280, 720]^k\\
 & [45885, 60720, 175560, 269192]^k = [43890, 85008, 96140, 367080]^k
\end{align}
\begin{mirrortype}
\end{mirrortype}
$\bullet$ \ref{k014}\quad $(k=0,1,4)$
\\
\subsubsection{\;\;(k$\;=\;$-5, -1, 0)}
\label{kn015}
\begin{notation}
\end{notation}
$\bullet$ Type: $(k=-5,-1,0)$ \qquad $\bullet$ Abbreviation: $[$k0n15$]$
\begin{idealsolu}
\end{idealsolu}
$\bullet$ Smallest known solutions, based on \eqref{k015s155} and \eqref{k015s301}:
\begin{align}
 & [6510, 7161, 30030, 44330^k = [ 6006, 10230, 14105, 71610]^k \\
 & [328692, 501423, 761852, 1079988]^k \nonumber \\
 & \quad = [326508, 581532, 610428, 1169987]^k
\end{align}
\begin{mirrortype}
\end{mirrortype}
$\bullet$ \ref{k015}\quad $(k=0,1,5)$
\\
\subsubsection{\;\;(k$\;=\;$-3, -2, 0)}
\label{kn023}
\begin{notation}
\end{notation}
$\bullet$ Type: $(k=-3,-2,0)$ \qquad $\bullet$ Abbreviation: $[$k0n23$]$
\begin{idealsolu}
\end{idealsolu}
$\bullet$ Smallest known solutions, based on \eqref{k023s16335} and \eqref{k023s2040}:
\begin{align}
 & [54978, 69360, 366520, 7009695]^k \nonumber \\
 & \quad =[57222, 64680, 778855, 3398640]^k \\
 & [3544146, 5563485, 13345695, 67711842]^k \nonumber \\
 & \quad =[3639735, 5036418, 20399445, 47649074]^k
\end{align}
\begin{mirrortype}
\end{mirrortype}
$\bullet$ \ref{k023}\quad $(k=0,2,3)$
\\
\subsubsection{\;\;(k$\;=\;$-4, -2, 0)}
\label{kn024}
\begin{notation}
\end{notation}
$\bullet$ Type: $(k=-4,-2,0)$ \qquad $\bullet$ Abbreviation: $[$k0n24$]$
\begin{idealsolu}
\end{idealsolu}
$\bullet$ Smallest known solutions, based on \eqref{k024s28} and \eqref{k024s36}:
\begin{align}
 & [27, 42, 84, 756]^k=[28, 36, 189, 378]^k\\
 & [119, 204, 252, 2142]^k=[126, 153, 476, 1428]^k
\end{align}
\begin{idealchain}
\end{idealchain}
$\bullet$ Ideal solution chains, based on \eqref{k024s1150}:
\begin{align}
& [137268, 188600, 481275, 10523880]^k  \nonumber\\
& \quad =[142600, 169740, 789291, 6863400]^k  \nonumber\\
& \quad =[152520, 154008, 2192475, 2546100]^k
\end{align}
\begin{mirrortype}
\end{mirrortype}
$\bullet$ \ref{k024}\quad $(k=0,2,4)$
\\
\subsubsection{\;\;(k$\;=\;$-6, -2, 0)}
\label{kn026}
\begin{notation}
\end{notation}
$\bullet$ Type: $(k=-6,-2,0)$ \qquad $\bullet$ Abbreviation: $[$k0n26$]$
\begin{idealsolu}
\end{idealsolu}
$\bullet$ Smallest known solution, based on \eqref{k026s1887}:
\begin{align}
\label{kn026s142587381}
& [11712265, 15189721, 28371045, 142587381]^k  \nonumber\\
& \quad =[11991885, 13979155, 35936657, 119465103]^k
\end{align}
\begin{mirrortype}
\end{mirrortype}
$\bullet$ \ref{k026}\quad $(k=0,2,6)$
\\
\subsubsection{\;\;(k$\;=\;$-3, -2, -1, 0)}
\label{kn0123}
\begin{notation}
\end{notation}
$\bullet$ Type: $(k=-3,-2,-1,0)$ \qquad $\bullet$ Abbreviation: $[$k0n123$]$
\begin{idealsolu}
\end{idealsolu}
$\bullet$ Smallest known solutions, based on \eqref{k0123s60}, \eqref{k0123s48} and \eqref{k0123s88}:
\begin{align}
& [140, 168, 240, 400, 525]^k=[ 150, 150, 280, 336, 560]^k\\
& [105, 144, 168, 420, 560]^k=[112, 120, 210, 336, 630]^k\\
& [63, 88, 132, 462, 616]^k=[66, 77, 168, 308, 792]^k
\end{align}
\begin{mirrortype}
\end{mirrortype}
$\bullet$ \ref{k0123}\quad $(k=0,1,2,3)$
\\
\subsubsection{\;\;(k$\;=\;$-4, -2, -1, 0)}
\label{kn0124}
\begin{notation}
\end{notation}
$\bullet$ Type: $(k=-4,-2,-1,0)$ \qquad $\bullet$ Abbreviation: $[$k0n124$]$
\begin{idealsolu}
\end{idealsolu}
$\bullet$ Smallest known solutions, based on \eqref{k0124s22} and \eqref{k0124s28}:
\begin{align}
& [30, 44, 44, 165, 165]^k=[ 33, 33, 60, 110, 220]^k\\
& [54, 72, 84, 168, 189]^k=[56, 63, 108, 126, 216]^k
\end{align}
\begin{mirrortype}
\end{mirrortype}
$\bullet$ \ref{k0124}\quad $(k=0,1,2,4)$
\\
\subsubsection{\;\;(k$\;=\;$-6, -2, -1, 0)}
\label{kn0126}
\begin{notation}
\end{notation}
$\bullet$ Type: $(k=-6,-2,-1,0)$ \qquad $\bullet$ Abbreviation: $[$k0n126$]$
\begin{idealsolu}
\end{idealsolu}
$\bullet$ Smallest known solutions, based on \eqref{k0126s410} and \eqref{k0126s360}:
\begin{align}
& [361284, 435666, 920040, 4488680, 6171935]^k  \nonumber\\
& \quad =[363055, 429352, 961860, 3612840, 7406322]^k \\
& [36234939, 39770055, 46754760, 59025240, 60672456]^k  \nonumber\\
& \quad =[37164040, 37920285, 51155208, 52599105, 63632088]^k
\end{align}
\begin{mirrortype}
\end{mirrortype}
$\bullet$ \ref{k0126}\quad $(k=0,1,2,6)$
\\
\subsubsection{\;\;(k$\;=\;$-6, -4, -2, 0)}
\label{kn0246}
\begin{notation}
\end{notation}
$\bullet$ Type: $(k=-6,-4,-2,0)$ \qquad $\bullet$ Abbreviation: $[$k0n246$]$
\begin{idealsolu}
\end{idealsolu}
$\bullet$ Smallest known solutions, based on \eqref{k0246s50} and \eqref{k0246s110}:
\begin{align}
& [144, 180, 225, 800, 2400]^k=[150, 160, 288, 450, 36000]^k \\
& [6380, 6699, 11550, 15950, 60900]^k=[6090, 7700, 8932, 23925, 47850]^k
\end{align}
\begin{mirrortype}
\end{mirrortype}
$\bullet$ \ref{k0246}\quad $(k=0,2,4,6)$
\\
\subsubsection{\;\;(k$\;=\;$-4, -3, -2, -1, 0)}
\label{kn01234}
\begin{notation}
\end{notation}
$\bullet$ Type: $(k=-4,-3,-2,-1,0)$ \qquad $\bullet$ Abbreviation: $[$k0n1234$]$
\begin{idealsolu}
\end{idealsolu}
$\bullet$ Smallest known solution, based on \eqref{k01234s72}:
\begin{align}
\label{k0n1234s47736}
& [3510, 3672, 5967, 8840, 26520, 39780]^k \nonumber \\
& \quad =[3315, 4420, 4680, 11934, 18360, 47736]^k
\end{align}
\begin{mirrortype}
\end{mirrortype}
$\bullet$ \ref{k01234}\quad $(k=0,1,2,3,4)$
\\
\subsubsection{\;\;(k$\;=\;$-5, -3, -2, -1, 0)}
\label{kn01235}
\begin{notation}
\end{notation}
$\bullet$ Type: $(k=-5,-3,-2,-1,0)$ \qquad $\bullet$ Abbreviation: $[$k0n1235$]$
\begin{idealsolu}
\end{idealsolu}
$\bullet$ Smallest known solutions, based on \eqref{k01235s387} and \eqref{k01235s252}:
\begin{align}
\label{k0n1235s528255}
& [32760, 40248, 48762, 97524, 234780, 528255]^k  \nonumber\\
& \quad =[34830, 35217, 58968, 78260, 325080, 422604]^k\\
& [39780, 45360, 79560, 208845, 250614, 1113840]^k  \nonumber\\
& \quad =[41769, 42120, 96390, 123760, 477360, 835380]^k
\end{align}
\begin{mirrortype}
\end{mirrortype}
$\bullet$ \ref{k01235}\quad $(k=0,1,2,3,5)$
\\
\subsubsection{\;\;(k$\;=\;$-6, -4, -2, -1, 0)}
\label{kn01246}
\begin{notation}
\end{notation}
$\bullet$ Type: $(k=-6,-4,-2,-1,0)$ \qquad $\bullet$ Abbreviation: $[$k0n1246$]$
\begin{idealsolu}
\end{idealsolu}
$\bullet$ Smallest known solutions, based on \eqref{k01246s46} and \eqref{k01246s66}:
\begin{align}
& [720, 920, 1035, 2208, 2760, 6624]^k  \nonumber\\
& \quad = [736, 828, 1380, 1440, 4140, 5520]^k\\
& [480, 576, 704, 1320, 1980, 3960]^k  \nonumber\\
& \quad = [495, 528, 880, 960, 2880, 3168]^k
\end{align}
\begin{mirrortype}
\end{mirrortype}
$\bullet$ \ref{k01246}\quad $(k=0,1,2,4,6)$
\\
\subsubsection{\;\;(k$\;=\;$-5, -4, -3, -2, -1, 0)}
\label{kn012345}
\begin{notation}
\end{notation}
$\bullet$ Type: $(k=-5,-4,-3,-2,-1,0)$ \qquad $\bullet$ Abbreviation: $[$k0n12345$]$
\begin{idealsolu}
\end{idealsolu}
$\bullet$ Smallest known solution, based on \eqref{k012345s124}:
\begin{align}
& [8190, 9672, 9765, 15624, 24180, 56420, 67704]^k\\
& \quad = [8463, 8680, 10920, 14508, 26040, 50778, 72540]^k
\end{align}
\begin{mirrortype}
\end{mirrortype}
$\bullet$ \ref{k012345}\quad $(k=0,1,2,3,4,5)$
\\
\subsubsection{\;\;(k$\;=\;$-6, -4, -3, -2, -1, 0)}
\label{kn012346}
\begin{notation}
\end{notation}
$\bullet$ Type: $(k=-6,-4,-3,-2,-1,0)$ \qquad $\bullet$ Abbreviation: $[$k0n12346$]$
\begin{idealsolu}
\end{idealsolu}
$\bullet$ Smallest known solutions, based on \eqref{k012346s92} and \eqref{k012346s222}:
\begin{align}
\label{k0n12346s933570}
& [324720, 364320, 373428, 466785, 663872, 829840, 905280]^k  \nonumber\\
& \quad = [331936, 339480, 414920, 432960, 728640, 746856, 933570]^k\\
& [9801792, 11513216, 16999983, 24727248,  \nonumber\\
& \qquad\qquad 41846112, 103618944, 127999872]^k  \nonumber\\
& \quad = [9846144, 11333322, 19428552, 20148128,  \nonumber\\
& \qquad\qquad 58810752, 65939328, 155428416]^k
\end{align}
\begin{mirrortype}
\end{mirrortype}
$\bullet$ \ref{k012346}\quad $(k=0,1,2,3,4,6)$
\\
\subsubsection{\;\;(k$\;=\;$-8, -6, -4, -2, -1, 0)}
\label{kn012468}
\begin{notation}
\end{notation}
$\bullet$ Type: $(k=-8,-6,-4,-2,-1,0)$ \qquad $\bullet$ Abbreviation: $[$k0n12468$]$
\begin{idealsolu}
\end{idealsolu}
$\bullet$ Smallest known solution, based on \eqref{k012468s1233}:
\begin{align}
& [1270470709119612, 1415077131295828, 1514981029346694,  \nonumber\\
& \qquad\qquad 2077573454037774, 2140014186262953,  \nonumber\\
& \qquad\qquad 3151892121417468, 3584646188431308]^k  \nonumber \\
& \quad = [1286116900118622, 1351587907113444, 1697172680763252,  \nonumber \\
& \qquad\qquad 1736685570226698,2522528799266476,  \nonumber\\
& \qquad\qquad 2858559095519127, 3703286960625252]^k
\end{align}
\begin{mirrortype}
\end{mirrortype}
$\bullet$ \ref{k012468}\quad $(k=0,1,2,4,6,8)$
\\
\subsubsection{\;\;(k$\;=\;$-6, -5, -4, -3, -2, -1, 0)}
\label{kn0123456}
\begin{notation}
\end{notation}
$\bullet$ Type: $(k=-6,-5,-4,-3,-2,-1,0)$ \qquad $\bullet$ Abbreviation: $[$k0n123456$]$
\begin{idealsolu}
\end{idealsolu}
$\bullet$ Smallest known solution, based on \eqref{k0123456s111}:
\begin{align}
& [89040, 102025, 111300, 184800,195888, 466400, 652960, 1399200]^k  \nonumber\\
& \quad = [92400, 93280, 130592, 139920,296800,306075, 890400, 1224300]^k
\end{align}
\begin{mirrortype}
\end{mirrortype}
$\bullet$ \ref{k0123456}\quad $(k=0,1,2,3,4,5,6)$
\\
\subsubsection{\;\;(k$\;=\;$-7, -6, -5, -4, -3, -2, -1, 0)}
\label{kn01234567}
\begin{notation}
\end{notation}
$\bullet$ Type: $(k=-7,-6,-5,-4,-3,-2,-1,0)$ \qquad $\bullet$ Abbreviation: $[$k0n1234567$]$
\begin{idealsolu}
\end{idealsolu}
$\bullet$ Smallest known solution, based on \eqref{k01234567s610}:
\begin{align}
& [8769115037417904, 9081765998004960, 9335358067757280, \nonumber\\
& \qquad 9979776441837540, 10828259459159760, 11966801281487520, \nonumber\\
& \qquad 12858558107752215, 13786495290785880, 13822119309625120]^k  \nonumber\\
& \quad = [8797960810567305, 8975100961115640, 9569159522048160,  \nonumber\\
& \qquad 9743461152686560, 11029196232628704, 11730614414089740, \nonumber\\
& \qquad 13273350304776480, 13306368589116720, 14003037101635920]^k
\end{align}
\begin{mirrortype}
\end{mirrortype}
$\bullet$ \ref{k01234567}\quad $(k=0,1,2,3,4,5,6,7)$
\\
\subsection{GPTE with k$_{1}$<0 and k$_{n}$>0}
To date, 33 distinct types of ideal non-negative integer solutions of GPTE have been identified with \( k_1 < 0 \) and \( k_n > 0 \).
\subsubsection{\;\;(k$\;=\;$-1, 1)}
\label{k1n1}
\begin{notation}
\end{notation}
$\bullet$ Type: $(k=-1,1)$ \qquad $\bullet$ Abbreviation: $[$k1n1$]$
\begin{idealsolu}
\end{idealsolu}
$\bullet$ First known solutions, smallest solutions, by Chen Shuwen in 2001 \cite{Chenkminus23}:
\begin{align}
 & [4, 10, 12 ]^k = [ 5, 6, 15]^k\\
 & [6, 14, 14 ]^k = [ 7, 9, 18]^k
\end{align}
\indent
$\bullet$ Ajai Choudhry gave a three-parameter solutions in 2011 \cite{Choudhry11}.
\begin{idealchain}
\end{idealchain}
$\bullet$ First known solution chains, based on computer search, by Chen Shuwen in 2023: 
\begin{align}
\label{k1n1s273}
 & [45,165,198]^k = [48,120,240]^k= [65,70,273]^k\\
\label{k1n1s429}
 & [44,253,276]^k = [45,198,330]^k= [66,78,429]^k
\end{align}
\indent
$\bullet$ New solution chains, based on \eqref{k1n1s429} and \eqref{k1n1s273}, by Chen Shuwen in 2023:
\begin{align}
 & [1380, 7590, 8970]^k = [1794, 2990, 13156]^k= [2145, 2340, 13455]^k\\
 & [2640, 10296, 11088]^k = [3003, 6006, 15015]^k= [3640, 4368, 16016]^k
\end{align}
\begin{mirrortype}
\end{mirrortype}
$\bullet$ None.
\\
\subsubsection{\;\;(k$\;=\;$-1, 2)}
\label{k2n1}
\begin{notation}
\end{notation}
$\bullet$ Type: $(k=-1,2)$ \qquad $\bullet$ Abbreviation: $[$k2n1$]$
\begin{idealsolu}
\end{idealsolu}
$\bullet$ First known solutions, smallest solutions, by Chen Shuwen in 2001 \cite{Chenkminus23}:
\begin{align}
\label{k2n1s91}
 & [35, 65, 84 ]^k = [ 39, 52, 91]^k\\
\label{k2n1s210}
 & [26, 143, 165 ]^k = [ 35, 55, 210]^k
\end{align}
\begin{mirrortype}
\end{mirrortype}
$\bullet$ \ref{k1n2}\quad $(k=-2,1)$
\\
\subsubsection{\;\;(k$\;=\;$-2, 1)}
\label{k1n2}
\begin{notation}
\end{notation}
$\bullet$ Type: $(k=-2,1)$ \qquad $\bullet$ Abbreviation: $[$k1n2$]$
\begin{idealsolu}
\end{idealsolu}
$\bullet$ First known solutions, based on \eqref{k2n1s91} and \eqref{k2n1s210}, by Chen Shuwen in 2001 \cite{Chenkminus23}:
\begin{align}
 & [ 60, 105, 140 ]^k = [ 65, 84, 156]^k\\
 & [143, 546, 858 ]^k = [ 182, 210, 1155]^k
\end{align}
\begin{mirrortype}
\end{mirrortype}
$\bullet$ \ref{k2n1}\quad $(k=-1,2)$
\\
\subsubsection{\;\;(k$\;=\;$-2, 2)}
\label{k2n2}
\begin{notation}
\end{notation}
$\bullet$ Type: $(k=-2,2)$ \qquad $\bullet$ Abbreviation: $[$k2n2$]$
\begin{idealsolu}
\end{idealsolu}
$\bullet$ Ajai Choudhry obtained two parametric solutions in 2001 \cite{Choudhry11}. There are four solutions less than 10000 by this method:
\begin{align}
\label{k2n2s1780}
 & [260,1157,1424]^k=[325,400,1780]^k \\
\label{k2n2s1963}
 & [77,1057,1661]^k = [91,143,1963 ]^k\\
\label{k2n2s2075}
 & [323,1411,1615]^k=[415,475,2075]^k \\
 & [527,1955,3689]^k = [575,1085,4025]^k
\end{align}
\indent
$\bullet$ New solutions less than 5000, based on computer search, by Chen Shuwen in 2017-2022 \cite{Chenkminus23}:
\begin{align}
\label{k2n2s350}
 & [22,220,275]^k=[28,35,350]^k\\
 & [207,736,744]^k=[276,279,992]^k\\
 & [60,850,1575]^k=[68,126,1785]^k \\
 & [1177,2475,2889]^k=[1375,1605,3375]^k \\
 & [3225,3655,3825]^k=[3311,3465,3927]^k \\
 & [2867,4087,4331]^k=[3149,3337,4757]^k 
\end{align}
\indent
$\bullet$ All solutions listed above satisfy $a_1 b_1=a_2 b_2=a_3 b_3$.
\begin{mirrortype}
\end{mirrortype}
$\bullet$ None.
\begin{relatedtype}
\end{relatedtype}
$\bullet$ \ref{k02n2}\quad $(k=-2,0,2)$
\\
\subsubsection{\;\;(k$\;=\;$-1, 0, 1)}
\label{k01n1}
\begin{notation}
\end{notation}
$\bullet$ Type: $(k=-1,0,1)$ \qquad $\bullet$ Abbreviation: $[$k01n1$]$
\begin{idealsolu}
\end{idealsolu}
$\bullet$ Guo Xianqiang studied this system in 2000, and obtained an non-ideal solution. \vspace{1ex}\\
\indent
$\bullet$ First known solutions, by Chen Shuwen in 2001 \cite{Chenkminus23}:
\begin{align}
 & [ 4, 10, 18, 45 ]^k = [ 5, 6, 30, 36]^k \\
 & [ 5, 15, 21, 63 ]^k = [ 7, 7, 45, 45]^k
\end{align}
\begin{idealchain}
\end{idealchain}
$\bullet$ First known ideal chains, based on a parametric solution, by Chen Shuwen in 2023:
\begin{align}
 & [210, 560, 726, 1936]^k = [220, 440, 924, 1848]^k \nonumber \\
 & \quad =[231, 385, 1056, 1760]^k = [264, 308, 1320, 1540]^k \\
 & [456, 1425, 2904, 9075]^k = [495, 1045, 3960, 8360]^k \nonumber \\
 & \quad =[550, 836, 4950, 7524]^k = [660, 660, 6270, 6270]^k
\end{align}
\begin{mirrortype}
\end{mirrortype}
$\bullet$ None.
\\
\subsubsection{\;\;(k$\;=\;$-1, 0, 2)}
\label{k02n1}
\begin{notation}
\end{notation}
$\bullet$ Type: $(k=-1,0,2)$ \qquad $\bullet$ Abbreviation: $[$k02n1$]$
\begin{idealsolu}
\end{idealsolu}
$\bullet$ First known solution, smallest solution, by Chen Shuwen in 2017 \cite{Chenkminus23}:
\begin{align}
\label{k02n1s78}
 & [6, 6, 52, 65 ]^k = [ 4, 15, 26, 78]^k\\
 & [28,39,117,270]^k = [26,45,108,273]^k\\
 & [13,33,220,234]^k = [12,45,143,286]^k\\
\label{k02n1s297}
 & [42,55,189,270]^k = [35,90,126,297]^k\\
 & [33,63,180,286]^k = [30,99,117,308]^k
\end{align}
\begin{relatedtype}
\end{relatedtype}
$\bullet$ \ref{k01n2}\quad $(k=-2,0,1)$
\\
\subsubsection{\;\;(k$\;=\;$-1, 1, 2)}
\label{k12n1}
\begin{notation}
\end{notation}
$\bullet$ Type: $(k=-1,1,2)$ \qquad $\bullet$ Abbreviation: $[$k12n1$]$
\begin{idealsolu}
\end{idealsolu}
$\bullet$ First known solutions, smallest solutions, by Chen Shuwen in 2017 \cite{Chenkminus23}:
\begin{align}
 & [15, 28, 48, 70 ]^k = [ 16, 24, 55, 66]^k\\
 & [34, 42, 63, 72 ]^k = [ 35, 40, 68, 68]^k\\
\label{k12n1s76}
 & [19,28,57,76 ]^k = [21,24,63,72]^k\\
 & [14,35,39,78 ]^k = [15,26,50,75]^k\\
 & [9,36,44,90 ]^k = [10,22,70,77]^k\\
\label{k12n1s91}
 & [36,52,52,91]^k = [39,42,60,90]^k
\end{align}
\begin{mirrortype}
\end{mirrortype}
$\bullet$ \ref{k1n12}\quad $(k=-2,-1,1)$
\\
\subsubsection{\;\;(k$\;=\;$-1, 1, 3)}
\label{k13n1}
\begin{notation}
\end{notation}
$\bullet$ Type: $(k=-1,1,3)$ \qquad $\bullet$ Abbreviation: $[$k13n1$]$
\begin{idealsolu}
\end{idealsolu}
$\bullet$ First known solutions, smallest solutions, based on computer search, by Chen Shuwen in 2017 \cite{Chenkminus23}:
\begin{align}
\label{k13n1s30}
 & [3, 10, 15, 30 ]^k = [4, 5, 21, 28]^k\\
 & [ 7,15,50,75]^k = [9,10,56,72]^k\\
 & [ 7,21,44,77]^k = [9,12,56,72]^k\\
 & [ 10,21,44,77]^k = [12,15,50,75]^k\\
 & [ 7,26,39,78 ]^k = [9,13,56,72]^k\\
\label{k13n1s78}
 & [ 10,26,39,78]^k = [13,15,50,75]^k \\
 & [ 12,26,39,78 ]^k = [13,21,44,77]^k\\
 & [ 11,33,44,84 ]^k = [14,18,63,77]^k
\end{align}
\indent
$\bullet$ 
It is obvious that a solution of $(h=-1,1,2,3)$ can be transformed to a solution of $(k=-1,1,3)$. For example, \eqref{h123n1s30} may lead to \eqref{k13n1s30}.
\begin{mirrortype}
\end{mirrortype}
$\bullet$ \ref{k1n13}\quad $(k=-3,-1,1)$
\begin{relatedtype}
\end{relatedtype}
$\bullet$ \ref{h013n1}\quad\: $(h=-1,0,1,3)$ \vspace{1ex}\\
\indent
$\bullet$ \ref{h123n1}\quad\: $(h=-1,1,2,3)$ 
\\
\subsubsection{\;\;(k$\;=\;$-1, 1, 5)}
\label{k15n1}
\begin{notation}
\end{notation}
$\bullet$ Type: $(k=-1,1,5)$ \qquad $\bullet$ Abbreviation: $[$k15n1$]$
\begin{idealsolu}
\end{idealsolu}
$\bullet$ First known solution, smallest solution, by Chen Shuwen in 2017 \cite{Chenkminus23}:
\begin{align}
\label{k15n1s891}
 &  [ 81, 374, 585, 891]^k=[85, 286, 702, 858 ]^k 
\end{align}
\begin{mirrortype}
\end{mirrortype}
$\bullet$ \ref{k1n15}\quad $(k=-5,-1,1)$ 
\begin{relatedtype}
\end{relatedtype}
$\bullet$ \ref{h015n1}\quad\: $(h=-1,0,1,5)$ 
\\
\subsubsection{\;\;(k$\;=\;$-2, 0, 1)}
\label{k01n2}
\begin{notation}
\end{notation}
$\bullet$ Type: $(k=-2,0,1)$ \qquad $\bullet$ Abbreviation: $[$k01n2$]$
\begin{idealsolu}
\end{idealsolu}
$\bullet$ First known solutions, smallest known solutions, based on \eqref{k02n1s78} and \eqref{k02n1s297}, by Chen Shuwen in 2017 \cite{Chenkminus23}:
\begin{align}
 &  [ 10, 30, 52, 195 ]^k = [12, 15, 130, 130]^k \\
 &  [ 70,165,231,594 ]^k = [77,110,378,495]^k 
\end{align}
\begin{mirrortype}
\end{mirrortype}
$\bullet$ \ref{k02n1}\quad $(k=-1,0,2)$
\\
\subsubsection{\;\;(k$\;=\;$-2, 0, 2)}
\label{k02n2}
\begin{notation}
\end{notation}
$\bullet$ Type: $(k=-2,0,2)$ \qquad $\bullet$ Abbreviation: $[$k02n2$]$
\begin{idealsolu}
\end{idealsolu}
$\bullet$ First known solutions, all based on \eqref{k2n2s2075}, by Chen Shuwen in 2014 \cite{Chenkminus23}:
\begin{align}
\label{k02n2s10375}
 & [323, 1411, 2375, 10375]^k = [415, 475, 7055, 8075]^k \\
 & [6137,30685, 34445, 172225]^k = [7885, 9025, 117113, 134045]^k \\
 & [26809,45125,117113,197125]^k = [30685,34445,153425,172225]^k 
\end{align}
\indent
$\bullet$ Smallest solutions with additional condition $a_1 a_2=a_3 a_4=b_1 b_2=b_3 b_4$, based on computer search, by Chen Shuwen in 2017:
\begin{align}
\label{k02n2s42}
 & [5,14,15,42]^k =[6, 7, 30, 35 ]^k \\
 & [3,11,12,44 ]^k=[ 4,4,33,33]^k\\
 & [21,42,52,104]^k =[24,28,78,91 ]^k \\
 & [5,20,30,120]^k =[6,8,75,100 ]^k \\
\label{k02n2s120}
 & [7,28,30,120]^k =[8,12,70,105 ]^k \\
\label{k02n2s105}
 & [7,14,60,120]^k =[8,10,84,105 ]^k \\
\label{k02n2s132}
 & [10,33,40,132]^k =[12,15,88,110 ]^k \\
 & [35,70,72,144]^k =[40,45,112,126 ]^k
\end{align}
\begin{mirrortype}
\end{mirrortype}
$\bullet$ None.
\begin{relatedtype}
\end{relatedtype}
$\bullet$ \ref{k2n2}\quad $(k=-2,2)$ \vspace{1ex}\\
\indent
$\bullet$ \ref{h012n2}\quad $(h=-2,0,1,2)$
\\
\subsubsection{\;\;(k$\;=\;$-2, -1, 1)}
\label{k1n12}
\begin{notation}
\end{notation}
$\bullet$ Type: $(k=-2,-1,1)$ \qquad $\bullet$ Abbreviation: $[$k1n12$]$
\begin{idealsolu}
\end{idealsolu}
$\bullet$ Smallest known solutions, based on \eqref{k12n1s76} and \eqref{k12n1s91}, by Chen Shuwen in 2017 \cite{Chenkminus23}:
\begin{align}
 &  [126, 168, 342, 504 ]^k  = [ 133, 152, 399, 456]^k \\
 &  [180, 315,315,455 ]^k  = [182,273,390,420]^k
\end{align}
\begin{mirrortype}
\end{mirrortype}
$\bullet$ \ref{k12n1}\quad $(k=-1,1,2)$
\\
\subsubsection{\;\;(k$\;=\;$-3, 0, 3)}
\label{k03n3}
\begin{notation}
\end{notation}
$\bullet$ Type: $(k=-3,0,3)$ \qquad $\bullet$ Abbreviation: $[$k03n3$]$
\begin{idealsolu}
\end{idealsolu}
$\bullet$ First known solution, smallest solution, by Chen Shuwen in 2017 \cite{Chenkminus23}:
\begin{align}
 & [ 15, 24, 90, 144]^k = [16, 20, 108, 135 ]^k
\end{align}
\indent
$\bullet$ Solutions based on parametric method, by Chen Shuwen in 2023:
\begin{align}
 &  [231, 693, 750, 2250 ]^k = [275, 297, 1750, 1890]^k \\
 &  [290,928,1290,4128]^k = [320,435,2752,3741]^k \\ 
 &  [6000, 9625, 11376, 18249]^k=[6160, 8400, 13035, 17775]^k\\
 &  [1643,3069,12720,23760]^k = [1860,2120,18414,20988]^k\\
 &  [2349,6699,10800,30800]^k = [2772,3080,23490,26100]^k \\
 &  [7650,12903,21600,36432]^k = [8800,9200,30294,31671]^k \\
 & [576, 3816, 10648, 70543]^k = [ 583, 1696, 23958, 69696]^k
\end{align}
\begin{mirrortype}
\end{mirrortype}
$\bullet$ None.
\\
\subsubsection{\;\;(k$\;=\;$-3, -1, 1)}
\label{k1n13}
\begin{notation}
\end{notation}
$\bullet$ Type: $(k=-3,-1,1)$ \qquad $\bullet$ Abbreviation: $[$k1n13$]$
\begin{idealsolu}
\end{idealsolu}
$\bullet$ Smallest known solutions, based on \eqref{k13n1s30} and \eqref{k13n1s78}, by Chen Shuwen in 2017 \cite{Chenkminus23}:
\begin{align}
\label{k1n13s140}
 &  [14, 28, 42, 140 ]^k = [ 15, 20, 84, 105]^k \\
 &  [25, 50, 75, 195 ]^k = [ 26, 39, 130, 150]^k
\end{align}
\begin{mirrortype}
\end{mirrortype}
$\bullet$ \ref{k13n1}\quad $(k=-1,1,3)$
\\
\subsubsection{\;\;(k$\;=\;$-5, -1, 1)}
\label{k1n15}
\begin{notation}
\end{notation}
$\bullet$ Type: $(k=-5,-1,1)$ \qquad $\bullet$ Abbreviation: $[$k1n15$]$
\begin{idealsolu}
\end{idealsolu}
$\bullet$ First known solution, based on \eqref{k15n1s891}, by Chen Shuwen in 2017 \cite{Chenkminus23}:
\begin{align}
 &  [2210, 3366, 5265, 24310 ] ^k= [ 2295, 2805, 6885, 23166]^k
\end{align}
\begin{mirrortype}
\end{mirrortype}
$\bullet$ \ref{k15n1}\quad $(k=-1,1,5)$
\\
\subsubsection{\;\;(k$\;=\;$-1, 0, 1, 2)}
\label{k012n1}
\begin{notation}
\end{notation}
$\bullet$ Type: $(k=-1,0,1,2)$ \qquad $\bullet$ Abbreviation: $[$k012n1$]$
\begin{idealsolu}
\end{idealsolu}
$\bullet$ First known solutions, smallest solutions, by Chen Shuwen in 2017 \cite{Chenkminus23}:
\begin{align}
\label{k012n1s35}
 &  [ 7, 10, 14, 32, 32]^k=[ 8, 8, 16, 28, 35 ]^k\\
\label{k012n1s40}
 &[4, 8, 10, 35, 35]^k=[ 5, 5, 14,28, 40 ] ^k\\
 &[8,14,15,35,36]^k=[9,10,21,28,40]^k\\
 &[10,15,18,33,44]^k=[11,12,22,30,45]^k\\
 &[20,28,28,44,55]^k=[22,22,35,40,56]^k\\
 &[20,27,35,54,56]^k=[21,24,42,45,60]^k
\end{align}
\begin{mirrortype}
\end{mirrortype}
$\bullet$ \ref{k01n12}\quad $(k=-2,-1,0,1)$
\\
\subsubsection{\;\;(k$\;=\;$-1, 0, 1, 3)}
\label{k013n1}
\begin{notation}
\end{notation}
$\bullet$ Type: $(k=-1,0,1,3)$ \qquad $\bullet$ Abbreviation: $[$k013n1$]$
\begin{idealsolu}
\end{idealsolu}
$\bullet$ First known solutions, by Chen Shuwen in 2017 \cite{Chenkminus23}:
\begin{align}
\label{k013n1s75}
 &[  5, 10, 21, 60, 70]^k=[ 6, 7, 28, 50, 75 ]^k\\
\label{k013n1s91}
 &[ 13, 21, 42, 78, 84]^k=[14, 18, 52, 63, 91 ] ^k
\end{align}
\begin{mirrortype}
\end{mirrortype}
$\bullet$ \ref{k01n13}\quad $(k=-3,-1,0,1)$
\\
\subsubsection{\;\;(k$\;=\;$-1, 1, 2, 3)}
\label{k123n1}
\begin{notation}
\end{notation}
$\bullet$ Type: $(k=-1,1,2,3)$ \qquad $\bullet$ Abbreviation: $[$k123n1$]$
\begin{idealsolu}
\end{idealsolu}
$\bullet$ First known solution, by Chen Shuwen in 2022 \cite{Chenkminus23}:
\begin{align}
\label{k123n1s1598}
 & [266, 494, 494, 1463, 1547]^k = [ 287, 374, 611, 1394, 1598] ^k
\end{align}
\begin{mirrortype}
\end{mirrortype}
$\bullet$ \ref{k1n123}\quad $(k=-3,-2,-1,1)$
\\
\subsubsection{\;\;(k$\;=\;$-2, 0, 2, 4)}
\label{k024n2}
\begin{notation}
\end{notation}
$\bullet$ Type: $(k=-2,0,2,4)$ \qquad $\bullet$ Abbreviation: $[$k024n2$]$
\begin{idealsolu}
\end{idealsolu}
$\bullet$ First known solution, based on computer search, by Chen Shuwen in 2017 \cite{Chenkminus23}:
\begin{align}
\label{k024n2s1376}
 & [135, 240, 612, 731, 1376 ]^k = [ 129, 340, 387, 864, 1360] ^k
\end{align}
\begin{mirrortype}
\end{mirrortype}
$\bullet$ \ref{k02n24}\quad $(k=-4,-2,0,2)$
\\
\subsubsection{\;\;(k$\;=\;$-2, -1, 0, 1)}
\label{k01n12}
\begin{notation}
\end{notation}
$\bullet$ Type: $(k=-2,-1,0,1)$ \qquad $\bullet$ Abbreviation: $[$k01n12$]$
\begin{idealsolu}
\end{idealsolu}
$\bullet$ First known solutions, based on \eqref{k012n1s35} and \eqref{k012n1s40}, by Chen Shuwen in 2017 \cite{Chenkminus23}:
\begin{align}
 & [ 7, 10, 20, 56, 56 ] = [ 8, 8, 28, 35, 70 ]^k\\
 & [ 32, 40, 70, 140, 140 ] = [ 35, 35, 80, 112, 160]^k
\end{align}
\begin{mirrortype}
\end{mirrortype}
$\bullet$ \ref{k012n1}\quad $(k=-1,0,1,2)$
\\
\subsubsection{\;\;(k$\;=\;$-2, -1, 1, 2)}
\label{k12n12}
\begin{notation}
\end{notation}
$\bullet$ Type: $(k=-2,-1,1,2)$ \qquad $\bullet$ Abbreviation: $[$k12n12$]$
\begin{idealsolu}
\end{idealsolu}
$\bullet$ First known solution, based on selective search with additional condition $a_1 b_5=a_2 b_4=a_3 b_3=a_4 b_2=a_5 b_1$, by Chen Shuwen in 2017 \cite{Chenkminus23}:
\begin{align}
 & [ 24, 42, 56, 120, 168] ^k=[25, 35, 75, 100, 175 ]^k 
\end{align}
\indent
$\bullet$ Second and third known solutions, same method as above, by Chen Shuwen in 2019 and 2022:
\begin{align}
 & [212, 318, 477, 848, 1166 ]^k = [ 216, 297, 528, 792, 1188] ^k \\
 & [ 3300, 4675, 6732, 11968, 15048 ]^k = [ 3400, 4275, 7600, 10944, 15504 ] ^k
\end{align}
\begin{mirrortype}
\end{mirrortype}
$\bullet$ None.
\\
\subsubsection{\;\;(k$\;=\;$-3, -1, 0, 1)}
\label{k01n13}
\begin{notation}
\end{notation}
$\bullet$ Type: $(k=-3,-1,0,1)$ \qquad $\bullet$ Abbreviation: $[$k01n13$]$
\begin{idealsolu}
\end{idealsolu}
$\bullet$ First known solutions, based on \eqref{k013n1s75} and \eqref{k013n1s91}, by Chen Shuwen in 2017 \cite{Chenkminus23}:
\begin{align}
 & [36, 52, 63, 182, 234 ]^k = [ 39, 42, 78, 156, 252 ]^k\\
 & [28, 42, 75, 300, 350 ]^k = [ 30, 35, 100, 210, 420]^k
\end{align}
\begin{mirrortype}
\end{mirrortype}
$\bullet$ \ref{k013n1}\quad $(k=-1,0,1,3)$
\\
\subsubsection{\;\;(k$\;=\;$-3, -2, -1, 1)}
\label{k1n123}
\begin{notation}
\end{notation}
$\bullet$ Type: $(k=-3,-2,-1,1)$ \qquad $\bullet$ Abbreviation: $[$k1n123$]$
\begin{idealsolu}
\end{idealsolu}
$\bullet$ First known solution, based on \eqref{k123n1s1598}, by Chen Shuwen in 2022 \cite{Chenkminus23}:
\begin{align}
 & [ 779779, 893893, 2039422, 3331783, 4341766 ] \nonumber \\
 &\quad = [ 805486, 851734, 2522443, 2522443, 4684537 ]^k
\end{align}
\begin{mirrortype}
\end{mirrortype}
$\bullet$ \ref{k123n1}\quad $(k=-1,1,2,3)$
\\
\subsubsection{\;\;(k$\;=\;$-4, -2, 0, 2)}
\label{k02n24}
\begin{notation}
\end{notation}
$\bullet$ Type: $(k=-4,-2,-0,2)$ \qquad $\bullet$ Abbreviation: $[$k02n24$]$
\begin{idealsolu}
\end{idealsolu}
$\bullet$ First known solution, based on \eqref{k024n2s1376}, by Chen Shuwen in 2017 \cite{Chenkminus23}:
\begin{align}
 & [2295, 4320, 5160, 13158, 23392 ]=[ 2322, 3655, 8160, 9288, 24480]
\end{align}
\begin{mirrortype}
\end{mirrortype}
$\bullet$ \ref{k024n2}\quad $(k=-2,0,2,4)$
\\
\subsubsection{\;\;(k$\;=\;$-1, 0, 1, 2, 3)}
\label{k0123n1}
\begin{notation}
\end{notation}
$\bullet$ Type: $(k=-1,0,1,2,3)$ \qquad $\bullet$ Abbreviation: $[$k0123n1$]$
\begin{idealsolu}
\end{idealsolu}
$\bullet$ First known solutions, based on computer search, by Chen Shuwen in 2017 \cite{Chenkminus23}:
\begin{align}
 & [50, 55, 70, 88, 91, 104 ]^k = [ 52, 52, 77, 77, 100, 100] ^k\\
\label{k0123n1s132}
 & [11,18,35,84,90,132 ]^k = [12,15,44,63,110,126] ^k\\
 & [40,56,63,114,120,152 ]^k = [42,48,76,95,140,144]^k\\
\label{k0123n1s153}
 & [24,40,50,100,102,153 ]^k = [25,34,68,72,120,150]^k \\
 & [10,20,25,78,117,156 ]^k = [12,13,36,65,130,150] ^k\\
 & [55,66,77,140,144,168 ]^k = [56,63,80,132,154,165] ^k\\
 & [48,60,77,120,154,175 ]^k = [50,55,84,112,165,168] ^k\\
 & [51,68,70,105,160,176 ]^k = [55,56,85,96,168,170] ^k\\
 & [7,17,17,99,99,189 ]^k = [9,9,27,77,119,187] ^k\\
 & [77,91,112,156,168,198 ]^k = [78,88,117,147,176,196] ^k
\end{align}
\begin{mirrortype}
\end{mirrortype}
$\bullet$ \ref{k01n123}\quad $(k=-3,-2,-1,0,1)$
\\
\subsubsection{\;\;(k$\;=\;$-1, 0, 1, 2, 4)}
\label{k0124n1}
\begin{notation}
\end{notation}
$\bullet$ Type: $(k=-1,0,1,2,4)$ \qquad $\bullet$ Abbreviation: $[$k0124n1$]$
\begin{idealsolu}
\end{idealsolu}
$\bullet$ First known solution, based on computer search with \eqref{n2m3k0124n1}, by Chen Shuwen in 2017 \cite{Chenkminus23}:
\begin{align}
\label{k0124n1s132}
 & [10, 14, 24, 65, 117, 132 ]^k = [ 11, 12, 26, 63, 120, 130] ^k
\end{align}
\indent
$\bullet$ Second known solution, based on computer search with \eqref{n2m3k0124n1}, by Chen Shuwen in 2022:
\begin{align}
\label{k0124n1s238}
 & [ 36, 51, 72, 130, 195, 238 ]^k = [ 40, 42, 85, 117, 204, 234 ] ^k
\end{align}
\indent
$\bullet$ Chen Shuwen noticed the above two solutions satisfy $a_1+a_6=b_1+b_6$ and $a_2+a_5=b_2+b_5$, thus obtained a parametric method in 2023. No new idea solution is found, except an integer solution: $[-63,60,27,240,374,520]=[-60,85 ,24,182,432,495]$.
\begin{mirrortype}
\end{mirrortype}
$\bullet$ \ref{k01n124}\quad $(k=-4,-2,-1,0,1)$
\\
\subsubsection{\;\;(k$\;=\;$-2, -1, 0, 1, 2)}
\label{k012n12}
\begin{notation}
\end{notation}
$\bullet$ Type: $(k=-2,-1,0,1,2)$ \qquad $\bullet$ Abbreviation: $[$k012n12$]$
\begin{idealsolu}
\end{idealsolu}
$\bullet$ First known solutions, based on selective search with additional condition $a_1 a_6=a_2 a_5=a_3 a_4=b_1 b_6=b_2 b_5=b_3 b_4$, by Chen Shuwen in 2017 \cite{Chenkminus23}:
\begin{align}
\label{k012n12s42}
& [20,24,24,35,35,42] ^k = [21,21,28,30,40,40] ^k\\
& [14,18,18,35,35,45] ^k = [15,15,21,30,42,42] ^k\\
& [6,11,11,42,42,77] ^k = [7,7,21,22,66,66] ^k \\
& [8,11,18,44,72,99] ^k = [9,9,22,36,88,88] ^k \\
& [9,15,18,55,66,110] ^k = [10,11,30,33,90,99] ^k \\
& [28,42,44,105,110,165] ^k = [30,33,60,77,140,154] ^k \\
& [32,48,48,110,110,165] ^k = [33,40,60,88,132,160] ^k \\
& [8,11,16,99,144,198] ^k = [9,9,18,88,176,176] ^k \\
& [84,99,99,168,168,198] ^k = [88,88,108,154,189,189] ^k 
\end{align}
\begin{mirrortype}
\end{mirrortype}
$\bullet$ None.
\\
\subsubsection{\;\;(k$\;=\;$-3, -1, 0, 1, 3)}
\label{k013n13}
\begin{notation}
\end{notation}
$\bullet$ Type: $(k=-3,-1,0,1,3)$ \qquad $\bullet$ Abbreviation: $[$k013n13$]$
\begin{idealsolu}
\end{idealsolu}
$\bullet$ First known solutions, based on selective search with additional condition $a_1 a_6=a_2 a_5=a_3 a_4=b_1 b_6=b_2 b_5=b_3 b_4$, by Chen Shuwen in 2017 \cite{Chenkminus23}:
\begin{align}
& [ 8, 12, 14, 36, 42, 63 ]^k = [ 9, 9, 21, 24, 56, 56] ^k\\
& [7, 11, 14, 44, 56, 88 ]^k = [ 8, 8, 22, 28, 77, 77] ^k\\
& [20,28,35,72,90,126]^k = [21,24,45,56,105,120] ^k\\
& [34,51,51,120,120,180]^k = [36,40,72,85,153,170] ^k\\
& [8,11,24,66,144,198]^k = [9,9,33,48,176,176] ^k\\
& [60,84,90,182,195,273]^k = [63,70,117,140,234,260] ^k\\
& [66,88,99,184,207,276]^k = [72,72,132,138,253,253] ^k\\
& [55,77,88,175,200,280]^k = [56,70,100,154,220,275] ^k\\
& [10,20,22,130,143,286]^k = [11,13,52,55,220,260] ^k
\end{align}
\begin{mirrortype}
\end{mirrortype}
$\bullet$ None.
\\
\subsubsection{\;\;(k$\;=\;$-3, -2, -1, 0, 1)}
\label{k01n123}
\begin{notation}
\end{notation}
$\bullet$ Type: $(k=-3,-2,-1,0,1)$ \qquad $\bullet$ Abbreviation: $[$k01n123$]$
\begin{idealsolu}
\end{idealsolu}
$\bullet$ Smallest known solutions, based on \eqref{k0123n1s132} and \eqref{k0123n1s153}, by Chen Shuwen in 2017 \cite{Chenkminus23}:
\begin{align}
\label{k01n123s1260}
 & [ 105, 154, 165, 396, 770, 1260] ^k=[110, 126, 220, 315, 924, 1155] ^k\\
 & [ 200, 300, 306, 612, 765, 1275]^k=[204, 255, 425, 450, 900, 1224] ^k
\end{align}
\begin{mirrortype}
\end{mirrortype}
$\bullet$ \ref{k0123n1}\quad $(k=-1,0,1,2,3)$
\\
\subsubsection{\;\;(k$\;=\;$-4, -2, -1, 0, 1)}
\label{k01n124}
\begin{notation}
\end{notation}
$\bullet$ Type: $(k=-4,-2,-1,0,1)$ \qquad $\bullet$ Abbreviation: $[$k01n124$]$
\begin{idealsolu}
\end{idealsolu}
$\bullet$ Smallest known solutions, based on \eqref{k0124n1s132} and \eqref{k0124n1s238}, by Chen Shuwen in 2017-2022 \cite{Chenkminus23}:
\begin{align}
 & [2730, 3080, 5544, 15015, 25740, 36036] ^k \nonumber \\
 &\quad =[ 2772, 3003, 5720, 13860, 30030, 32760] ^k \\
 & [2340, 2856, 4284, 7735, 10920,15470] ^k \nonumber \\
 &\quad =[2380, 2730, 4760, 6552, 13260, 13923] ^k
\end{align}
\begin{mirrortype}
\end{mirrortype}
$\bullet$ \ref{k0124n1}\quad $(k=-1,0,1,2,4)$
\\
\subsubsection{\;\;(k$\;=\;$-1, 0, 1, 2, 3, 4)}
\label{k01234n1}
\begin{notation}
\end{notation}
$\bullet$ Type: $(k=-1,0,1,2,3,4)$ \qquad $\bullet$ Abbreviation: $[$k01234n1$]$
\begin{idealsolu}
\end{idealsolu}
$\bullet$ First known solutions, based on computer search with \eqref{n3m3kn1to4}, by Chen Shuwen in 2017 \cite{Chenkminus23}:
\begin{align}
\label{k01234n1s153}
 & [9, 17, 21, 51, 99, 143, 143 ] ^k = [ 11, 11, 33, 39, 117, 119, 153] ^k\\
\label{k01234n1s264}
 & [56, 77, 99, 152, 174, 228, 261 ] ^k = [ 57, 72, 116, 126, 203, 209, 264] ^k
\end{align}
\begin{mirrortype}
\end{mirrortype}
$\bullet$ \ref{k01n1234}\quad $(k=-4,-3,-2,-1,0,1)$
\\
\subsubsection{\;\;(k$\;=\;$-4, -3, -2, -1, 0, 1)}
\label{k01n1234}
\begin{notation}
\end{notation}
$\bullet$ Type: $(k=-4,-3,-2,-1,0,1)$ \qquad $\bullet$ Abbreviation: $[$k01n1234$]$
\begin{idealsolu}
\end{idealsolu}
$\bullet$ First known solutions, based on \eqref{k01234n1s153} and \eqref{k01234n1s264}, by Chen Shuwen in 2017 \cite{Chenkminus23}:
\begin{align}
 & [1001, 1287, 1309, 3927, 4641, 13923, 13923] ^k \nonumber \\
 &\quad =[1071, 1071, 1547, 3003, 7293, 9009, 17017] ^k\\
 & [11571, 14616, 15048, 24244, 26334, 42427, 53592] ^k  \nonumber \\
 &\quad =[11704, 13398, 17556, 20097, 30856, 39672, 54549] ^k
\end{align}
\begin{mirrortype}
\end{mirrortype}
$\bullet$ \ref{k01234n1}\quad $(k=-1,0,1,2,3,4)$
\\
\subsubsection{\;\;(k$\;=\;$-3, -2, -1, 0, 1, 2, 3)}
\label{k0123n123}
\begin{notation}
\end{notation}
$\bullet$ Type: $(k=-3,-2,-1,0,1,2,3)$ \qquad $\bullet$ Abbreviation: $[$k0123n123$]$
\begin{idealsolu}
\end{idealsolu}
$\bullet$ First known solution, based on selective search with additional condition $a_1 a_8=a_2 a_7=a_3 a_6=a_4 a_5=b_1 b_8=b_2 b_7=b_3 b_6=b_4 b_5$, by Chen Shuwen in 2017 \cite{Chenkminus23}:
\begin{align}
\label{k0123n123s88}
 & [ 21, 24, 28, 42, 44, 66, 77, 88] ^k = [22, 22, 33, 33, 56, 56, 84, 84] ^k
\end{align}
\indent
$\bullet$ Second known solution, same method as above, by Chen Shuwen in 2019:
\begin{align}
\label{k0123n123s8866}
 & [7980, 8060, 8151, 8360, 8463, 8680, 8778, 8866] ^k  \nonumber \\
 &\quad =[8008, 8008, 8246, 8246, 8580, 8580, 8835, 8835] ^k
\end{align}
\begin{mirrortype}
\end{mirrortype}
$\bullet$ None.
\\
\clearpage
\section{Ideal integer solution of GPTE}
Until now, a total of 65 distinct types of ideal integer solutions of GPTE have been discovered, specifically categorized as \(15 + 15 + 6 + 6 + 23\), which correspond to a wide array of configurations for \( (h = h_1, h_2, \dots, h_n) \).
\subsection{GPTE with all h>0}
Ideal integer solutions of GPTE have been found for 15 types with all \( h > 0 \).
\subsubsection{\;\;(h$\;=\;$1, 2, 4)}
\label{h124}
\begin{notation}
\end{notation}
$\bullet$ Type: $(h=1,2,4)$ \qquad $\bullet$ Abbreviation: $[$h124$]$
\begin{idealsolu}
\end{idealsolu}
$\bullet$ A.Gloden \cite{Gloden44} gave solutions of this type by starting with the non-symmetric solutions of $(k=1,2)$. Numerical examples are:
\begin{align}
 & [ -7, 0, 7 ]^h = [ -8, 3, 5 ] ^h\\
 \label{h124s11}
 & [ -10, 1, 9 ]^h = [ -11, 5, 6 ] ^h\\
 & [ -12, 1, 11 ]^h = [ -13, 4, 9 ] ^h\\
 & [ -13, 0, 13 ]^h = [ -15, 7, 8] ^h
\end{align}
\begin{idealchain}
\end{idealchain}
$\bullet$ A.Gloden's method, based on the ideal non-symmetric solution chains of $(k=1,2)$. Numerical example is:
\begin{align}
\label{h124s40}
& [-48, 23, 25 ]^h = [ -47, 15, 32 ]^h = [ -45, 8, 37 ] ^h= [ -43, 3, 40 ] ^h
\end{align}
\begin{mirrortype}
\end{mirrortype}
$\bullet$ \ref{hn124}\quad\: $(h=-4,-2,-1)$ 
\begin{relatedtype}
\end{relatedtype}
$\bullet$ \ref{k12}\quad\: $(k=1,2)$ \vspace{1ex}\\
\indent
$\bullet$ \ref{k24}\quad $(k=2,4)$ \vspace{1ex}\\
\indent
$\bullet$ \ref{s124}\quad\: $(s=1,2,4)$ 
\\
\subsubsection{\;\;(h$\;=\;$1, 2, 6)}
\label{h126}
\begin{notation}
\end{notation}
$\bullet$ Type: $(h=1,2,6)$ \qquad $\bullet$ Abbreviation: $[$h126$]$
\begin{idealsolu}
\end{idealsolu}
$\bullet$ First known solution, based on $(k=2,6)$, by Chen Shuwen in 1997 \cite{Chen01,Chen23}:
\begin{align}
\label{h126s405}
 & [ -372, 43, 371 ]^h = [ -405, 140, 307 ] ^h
\end{align}
\indent
$\bullet$ Ajai Choudhry obtained a method of generating infinitely many integer solutions of this type in 1999 \cite{Choudhry00}. Numerical examples are:
\begin{align}
 & [-300, 83, 211 ]^h = [ -124, -185, 303 ] ^h\\
 & [-479, 23, 432 ]^h = [ -393, -127, 496 ] ^h
\end{align}
\indent $\bullet$ Chen Shuwen found a parametric solution in 2023. Numerical examples:
\begin{align}
 & [-311397019, 30135648, 259460683]^h \nonumber \\
 & \quad = [-186060123, -159828085, 324087520] ^h\\
 & [-194087324, 92153383, 109486381]^h  \nonumber \\
 & \quad = [-142855951, -39752836, 190161227] ^h
\end{align}
\begin{idealchain}
\end{idealchain}
\indent $\bullet$ Based on \eqref{equationP1236}, Chen Shuwen proved that there is no ideal solution chain of length $j \geq 3$ for the type $(h = 1, 2, 6)$.
\begin{mirrortype}
\end{mirrortype}
$\bullet$ \ref{hn126}\quad\: $(h=-6,-2,-1)$
\begin{relatedtype}
\end{relatedtype}
$\bullet$ \ref{k26}\quad $(k=2,6)$ 
\\
\subsubsection{\;\;(h$\;=\;$1, 3, 4)}
\label{h134}
\begin{notation}
\end{notation}
$\bullet$ Type: $(h=1,3,4)$ \qquad $\bullet$ Abbreviation: $[$h134$]$
\begin{idealsolu}
\end{idealsolu}
$\bullet$ Ajai Choudhry obtained a parametric solution \eqref{h134para1} in 1991 \cite{Choudhry91}. Numerical examples:
\begin{align}
\label{h134s6738}
 & [-3254,5583,5658]^h = [ -1329, 2578, 6738 ]^h\\
 & [-92163, 119572, 222204]^h = [2932, 18429, 228252]^h\\
 & [-104436, 125245, 334980]^h = [-4436, 22845, 337380]^h\\
 & [-305090, 473574, 683115]^h = [-201426, 337515, 715510]^h
\end{align}
\indent
$\bullet$ By computer search in 2022, Chen Shuwen confirmed that \eqref{h134s6738} is the smallest solution.\vspace{1ex}\\
\indent
$\bullet$ Chen Shuwen found a new parametric solution \eqref{h134para2} in 2023. Numerical examples:
\begin{align}
 & [-404194933, 672857102, 692925014]^h \nonumber \\
 & \quad = [-113061493, 251994062, 822654614]^h\\
 & [-4007182568, 6000329352, 11306755329]^h  \nonumber \\
 & \quad = [-3653034368, 5576043777, 11376892704]^h \\
 \label{h134s6588274387251}
 & [-2608425958605, 3669662071230, 6406539709126]^h  \nonumber \\
 & \quad = [-1353176207690, 2232677642190, 6588274387251]^h 
\end{align}
\begin{mirrortype}
\end{mirrortype}
$\bullet$ \ref{hn134}\quad $(h=-4,-3,-1)$
\\
\subsubsection{\;\;(h$\;=\;$2, 3, 4)}
\label{h234}
\begin{notation}
\end{notation}
$\bullet$ Type: $(h=2,3,4)$ \qquad $\bullet$ Abbreviation: $[$h234$]$
\begin{idealsolu}
\end{idealsolu}
$\bullet$ Ajai Choudhry obtained the first solution in 2001 \cite{Choudhry01}, by cancelling out the common term of a parametric solution for $(k=2,3,4)$. 
\begin{align}
\label{h234s1233}
 & [ -815, 358, 1224 ]^h = [ -776, -410, 1233 ]^h
\end{align}
\indent
$\bullet$ Ajai Choudhry gave a method of generating infinitely many solutions in 2003 \cite{Choudhry03}, with additional condition of $a_1+b_1=-a_2-b_2$. New numerical solutions by this method:
\begin{align}
 & [-931219912, -156845590, 378382959]^h \nonumber\\
 & \quad = [-932263416, 195748463, 357088490]^h \\
 & [-207749597674213298, -147729270939734015, 220931652694759344]^h  \nonumber\\
 & \quad = [-230043367232999423, 156840985477974094, 190461498615919440]^h\\
 & [-43992346521996843142080660865,-29910196385462462465784243024,\nonumber \\
 & \qquad \qquad 122741732517246503591277139538 ]^h \nonumber \\
 & \: =[-48660767747355170765997625774,22490972972129645890059151440,\nonumber \\
 & \qquad \qquad 122563310654814476373862529663 ]^h  
\end{align}
\indent
$\bullet$ By computer search in 2022, Chen Shuwen confirmed that there is no other solution except \eqref{h234s1233} in the range of 10000.
\begin{mirrortype}
\end{mirrortype}
$\bullet$ \ref{hn234}\quad\: $(h=-4,-3,-2)$
\\
\subsubsection{\;\;(h$\;=\;$1, 2, 3, 5)}
\label{h1235}
\begin{notation}
\end{notation}
$\bullet$ Type: $(h=1,2,3,5)$ \qquad $\bullet$ Abbreviation: $[$h1235$]$
\begin{idealsolu}
\end{idealsolu}
$\bullet$ A.Gloden gave solutions of this type by starting with the non-symmetric solutions of $(k=1,2,3)$ \cite{Gloden44} . Numerical examples are:
\begin{align}
 & [ -38, -13, 0, 51 ] ^h= [ -33, -24, 7, 50 ]^h\\
 & [  -33, -7, 1, 39 ]^h = [ -27, -21, 11, 37 ]^h\\
\label{h1235s23}
 & [ -21, -3, 1, 23 ]^h = [ -17, -13, 9, 21 ]^h\\
 & [-29, -11, 7, 33 ]^h = [ -23, -21, 13, 31 ]^h\\
 & [ -63, -35, 29, 69 ] ^h= [ -51, -51, 37, 65 ]^h
\end{align}
\indent
$\bullet$ G.Xeroudakes and A.Moessner gave a parameter solution of this type in 1958 \cite{Xeroudakes58}. 
\begin{idealchain}
\end{idealchain}
$\bullet$ First known solution chains, based on \eqref{k123s214} and \eqref{k123s234} of $(k=1,2,3)$, by Chen Shuwen in 1997  \cite{Chen01,Chen23}:
\begin{align}
\label{h1235s197}
& [ -231,11,23,197 ]^h = [ -223,-49,93,179 ]^h \nonumber \\
& \quad = [ -217,-69,137,149 ]^h \\
& [-453, -41, 11, 483]^h = [-437, -133, 99, 471]^h \nonumber \\
& \quad = [-349, -297, 219, 427]^h 
\end{align}
\indent
$\bullet$  Solution chains of length 4 and length 5, based on \eqref{k02s4455} to \eqref{k02s250582040} of $(k=0,2)$, by Ajai Choudhry and Jarosław Wróblewski in 2016.
\begin{align}
\label{h1235s7313}
& [-6179,-1081,-53,7313]^h = [-6073,-1717,529,7261]^h \nonumber \\
& \quad = [-5267,-3589,1879,6977]^h = [-5149,-3761,1957,6953]^h \\
\label{h1235s10093}
& [-9039,-1601,547,10093]^h = [-8597,-3369,2137,9829]^h \nonumber \\
& \quad = [-7457,-5501,3597,9361]^h = [-6911,-6179,3823,9267]^h \\
\label{h1235s421257161}
& [-340657577,-67159911,-13439673,421257161 ]^h \nonumber \\
& \quad = [-334774115,-101806965,17720885,418860195 ]^h \nonumber \\
& \quad = [-331821165,-113345765,27446445,417720485 ]^h \nonumber \\
& \quad = [-298614537,-190020441,80708537,407926441 ]^h \nonumber \\
& \quad = [-262648215,-238515865,97710135,403453945 ]^h
\end{align}
\begin{mirrortype}
\end{mirrortype}
$\bullet$ \ref{hn1235}\quad\: $(h=-5,-3,-2,-1)$
\begin{relatedtype}
\end{relatedtype}
$\bullet$ \ref{k123}\quad $(k=1,2,3)$ \vspace{1ex}\\
\indent
$\bullet$ \ref{k135}\quad $(k=1,3,5)$ \vspace{1ex}\\
\indent
$\bullet$ \ref{k02}\quad\: $(k=0,2)$ \vspace{1ex}\\
\indent
$\bullet$ \ref{s1235}\quad\: $(s=1,2,3,5)$
\\
\subsubsection{\;\;(h$\;=\;$1, 2, 4, 6)}
\label{h1246}
\begin{notation}
\end{notation}
$\bullet$ Type: $(h=1,2,4,6)$ \qquad $\bullet$ Abbreviation: $[$h1246$]$
\begin{idealsolu}
\end{idealsolu}
$\bullet$ First known solution, by G.Palama in 1953 \cite{Palama1953}:
\begin{align}
\label{h1246s25}
 & [ -5, -14, 23, 24 ]^h = [ -16, -2, 21, 25]^h
\end{align}
\indent
$\bullet$ Ideal solutions based on $(k=2,4,6)$, by Chen Shuwen in 1997  \cite{Chen01,Chen23}:
\begin{align}
 & [-43, -1, 49, 72 ]^h = [ -27, -23, 56, 71]^h\\
 & [-70, -13, 96, 127 ]^h = [ -47, -42, 104, 125]^h\\
 & [-107, -21, 116, 145 ]^h = [ -131, 31, 93, 140]^h\\
 & [-41, 5, 23, 48 ]^h = [ -43, 15, 16, 47]^h
\end{align}
\begin{mirrortype}
\end{mirrortype}
$\bullet$ \ref{hn1246}\quad\: $(h=-6,-4,-2,-1)$
\begin{relatedtype}
\end{relatedtype}
$\bullet$ \ref{k246}\quad $(k=2,4,6)$ 
\\
\subsubsection{\;\;(h$\;=\;$1, 2, 3, 4, 6)}
\label{h12346}
\begin{notation}
\end{notation}
$\bullet$ Type: $(h=1,2,3,4,6)$ \qquad $\bullet$ Abbreviation: $[$h12346$]$
\begin{idealsolu}
\end{idealsolu}
$\bullet$ A.Gloden gave solutions of this type in 1944 \cite{Gloden44}, by starting with the non-symmetric solutions of $(k=1,2,3,4)$. Numerical examples are:
\begin{align}
 & [ -23, -10, -5, 14, 24 ]^h = [ -21, -16, 2, 10, 25 ]^h\\
\label{h12346s17}
 & [ -17, -5, -4, 12, 14 ]^h = [ -16, -10, 3, 7, 16 ]^h\\
\label{h12346s21}
 & [ -19, -7, -6, 12, 20 ]^h = [ -16, -15, 2, 8, 21 ]^h\\
 & [-72, -32, -7, 53, 58 ]^h = [ -67, -47, 18, 28, 68 ]^h
\end{align}
\indent
$\bullet$ Ajai Choudhry obtained a parametric solution in 2000 \cite{Choudhry2000}, based on non-symmetric solution of $(k=1,2,3,4)$.
\begin{mirrortype}
\end{mirrortype}
$\bullet$ \ref{hn12346}\quad\: $(h=-6,-4,-3,-2,-1)$
\begin{relatedtype}
\end{relatedtype}
$\bullet$ \ref{k1234}\quad $(k=1,2,3,4)$ 
\\
\subsubsection{\;\;(h$\;=\;$1, 2, 3, 5, 7)}
\label{h12357}
\begin{notation}
\end{notation}
$\bullet$ Type: $(h=1,2,3,5,7)$ \qquad $\bullet$ Abbreviation: $[$h12357$]$
\begin{idealsolu}
\end{idealsolu}
$\bullet$ First known solution, by G.Palama in 1953 \cite{Palama1953}:
\begin{align}
\label{h12357s57}
 & [  -59, -5, -1, 33, 57 ]^h = [ -55, -23, 13, 39, 51 ]^h
\end{align}
\indent
$\bullet$ Solutions based on $(k=1,3,5,7)$, by Chen Shuwen in 1995  \cite{Chen01,Chen23}:
\begin{align}
 & [ -55, -11, 3, 37, 51 ]^h = [ -53, -19, 9, 43, 45 ]^h\\
 & [ -71, -33, 11, 15, 73 ]^h = [ -57, -53, 1, 39, 65 ]^h\\
 & [ -99, -13, 0, 34, 98 ]^h = [ -82, -58, 16, 69, 75 ]^h\\
 & [ -169, -31, 103, 17, 163 ]^h = [ -157, 47, -71, 121, 143 ]^h
\end{align}
\indent
$\bullet$ Chen Shuwen obtained a parametric method of generating infinitely many solutions in 2023. Numerical examples:
\begin{align}
 & [-937, -83, 139, 235, 941]^h = [-731, -589, 323, 535, 757]^h\\
 & [-2332, -47, 323, 348, 2343]^h = [-1927, -1297, 723, 1168, 1968]^h\\
 & [-2907, -1097, 923, 1593, 2873]^h \nonumber \\
 & \quad = [-2707, -1527, 1153, 1823, 2643]^h\\
 & [-4485, -1793, 1571, 2413, 4449]^h \nonumber \\
 & \quad = [-4209, -2371, 1873, 2715, 4147]^h\\
 & [-12961, -1343, 3, 4877, 12811]^h \nonumber \\
 & \quad = [-10901, -7277, 2717, 8751, 10097]^h\\
 & [-14645, -1403, 2633, 3739, 14741]^h \nonumber \\
 & \quad = [-11809, -8711, 6233, 7105, 12247]^h\\
 & [-535483, -71833, 43707, 267097, 527977]^h \nonumber \\
 & \quad = [-456023, -292403, 184817, 386867, 408207]^h
\end{align}
\begin{mirrortype}
\end{mirrortype}
$\bullet$ \ref{hn12357}\quad\: $(h=-7,-5,-3,-2,-1)$
\begin{relatedtype}
\end{relatedtype}
$\bullet$ \ref{k1357}\quad $(k=1,3,5,7)$
\\
\subsubsection{\;\;(h$\;=\;$1, 2, 4, 6, 8)}
\label{h12468}
\begin{notation}
\end{notation}
$\bullet$ Type: $(h=1,2,4,6,8)$ \qquad $\bullet$ Abbreviation: $[$h12468$]$
\begin{idealsolu}
\end{idealsolu}
$\bullet$ First known solution, by A.Letac in 1942 \cite{Letac42} \cite{Gloden44}:
\begin{align}
 & [-12, -20231, 11881, 20885, 23738 ] ^h\nonumber \\
 & \quad = [ -20449, 436, 11857, 20667, 23750 ]^h
\end{align}
\indent
$\bullet$ Ideal solutions, based on \eqref{k2468s313} and \eqref{k2468s515} by Peter Borwein et al., were derived by Chen Shuwen in 2000  \cite{Chen01,Chen23}.
\begin{align}
\label{h12468s313}
 & [-180, -131, 71, 307, 308]^h =[-313, 99, 100, 188, 301]^h\\
 & [-307, -180, 71, 131, 308]^h =[-301, -188, 99, 100, 313]^h\\
 & [-471, -18, 245, 331, 508]^h =[-515, 103, 189, 366, 452]^h
\end{align}
\indent
$\bullet$ Ideal solution, based on \eqref{k2468s4827}, by Chen Shuwen in 2023:
\begin{align}
 & [-3783,-3773,-498, 4567,4787] ^h \nonumber\\
 &\quad = [-4463,-3598,517,4017,4827]^h
\end{align}
\begin{mirrortype}
\end{mirrortype}
$\bullet$ \ref{hn12468}\quad\: $(h=-8,-6,-4,-2,-1)$
\begin{relatedtype}
\end{relatedtype}
$\bullet$ \ref{k2468}\quad $(k=2,4,6,8)$
\\
\subsubsection{\;\;(h$\;=\;$1, 2, 3, 4, 5, 7)}
\label{h123457}
\begin{notation}
\end{notation}
$\bullet$ Type: $(h=1,2,3,4,5,7)$ \qquad $\bullet$ Abbreviation: $[$h123457$]$
\begin{idealsolu}
\end{idealsolu}
$\bullet$ A.Gloden gave one solution of this type in 1940's \cite{Gloden44} by starting with the non-symmetric solution of $(k=1,2,3,4,5)$:
\begin{align}
 & [ -89, -41, -31, 33, 45, 83 ]^h = [ -87, -55, -1, 3, 61, 79 ]^h
\end{align}
\indent
$\bullet$ Ideal solution, based on $(k=1,2,3,4,5)$, by Chen Shuwen in 1997  \cite{Chen01,Chen23}:
\begin{align}
\label{h123457s71}
 & [ -71, -44, -20, 31, 37, 67 ]^h= [ -68, -53, 1, 4, 55, 61 ]^h\\
 & [-295, -199, -73, 107, 149, 311]^h= [-283, -229, -1, 17, 191, 305]^h\\
 & [-271, -169, -121, 101, 155, 305]^h \nonumber \\
 & \quad = [-241, -235, -55, 41, 191, 299]^h
\end{align}
\begin{mirrortype}
\end{mirrortype}
$\bullet$ \ref{hn123457}\quad $(h=-7,-5,-4,-3,-2,-1)$
\begin{relatedtype}
\end{relatedtype}
$\bullet$ \ref{k12345}\quad $(k=1,2,3,4,5)$
\\
\subsubsection{\;\;(h$\;=\;$1, 2, 3, 4, 6, 8)}
\label{h123468}
\begin{notation}
\end{notation}
$\bullet$ Type: $(h=1,2,3,4,6,8)$ \qquad $\bullet$ Abbreviation: $[$h123468$]$
\begin{idealsolu}
\end{idealsolu}
$\bullet$ First known solution, based on Jarosław Wróblewski's result for $(k=2,4,6,8)$ of size 6 \cite[p.6]{JW09}, by Chen Shuwen in 2019  \cite{Chen23}:
\begin{align}
\label{h123468s62}
 & [  -42, -37, -1, 30, 57, 61 ]^h = [ -47, -27, -9, 35, 54, 62]^h
\end{align}
\begin{mirrortype}
\end{mirrortype}
$\bullet$ \ref{hn123468}\quad $(h=-8,-6,-4,-3,-2,-1)$
\\
\subsubsection{\;\;(h$\;=\;$1, 2, 3, 5, 7, 9)}
\label{h123579}
\begin{notation}
\end{notation}
$\bullet$ Type: $(h=1,2,3,5,7,9)$ \qquad $\bullet$ Abbreviation: $[$h123579$]$
\begin{idealsolu}
\end{idealsolu}
$\bullet$ First known solution, based on \eqref{k13579s323}, by Chen Shuwen in 2000  \cite{Chen01,Chen23}:
\begin{align}
\label{h123579s323}
 & [ -247, -193, -59, 91, 289, 323 ]^h = [ -269, -173, -7, 29, 311, 313 ]^h
\end{align}
\indent
$\bullet$ Second and third known solutions, based on Jarosław Wróblewski's solutions \eqref{k13579s407} and \eqref{k13579s1293}, noticed by Chen Shuwen in 2023:
\begin{align}
 & [-371, -119, -37, 163, 341, 407]^h \nonumber \\
 & \quad = [-347, -181, -23, 221, 311, 403]^h\\
 & [-1279, -925, -115, 57, 679, 1293]^h  \nonumber \\
 & \quad =  [-1167, -995, -399, 299, 767, 1205]^h
\end{align}
\indent
$\bullet$ Fourth known solution, based on Jarosław Wróblewski's result \cite[p.24]{JW09}, noticed by Chen Shuwen in 2023:
\begin{align}
 & [-1111, -689, 13, 23, 1177, 1319]^h \nonumber \\
 & \quad = [-1037, -731, -305, 365, 1115, 1325]^h
\end{align}
\begin{mirrortype}
\end{mirrortype}
$\bullet$ \ref{hn123579}\quad $(h=-9,-7,-5,-3,-2,-1)$
\begin{relatedtype}
\end{relatedtype}
$\bullet$ \ref{k13579}\quad $(k=1,3,5,7,9)$ 
\\
\subsubsection{\;\;(h$\;=\;$1, 2, 4, 6, 8, 10)}
\label{h1246810}
\begin{notation}
\end{notation}
$\bullet$ Type: $(h=1,2,4,6,8,10)$ \qquad $\bullet$ Abbreviation: $[$h1246810$]$
\begin{idealsolu}
\end{idealsolu}
$\bullet$ First known solutions, based on \eqref{k246810s151}, by Chen Shuwen in 1999  \cite{Chen01,Chen23}:
\begin{align}
 & [ -22, 61, -86, 127, -140, 151 ]^h  = [ -35, 47, -94, -121, 146, 148 ]^h \\
 & [ 22, -61, -86, 127, -140, 151 ]^h  = [ 35, -47, -94, 121, 146, -148 ]^h \\
 & [ 22, -61, 86, 127, -140, 151 ]^h  = [ -35, -47, 94, -121, 146, 148 ]^h \\
\label{h1246810s151}
 & [ 22, 61, -86, -127, 140, 151 ]^h  = [ 35, 47, -94, -121, 146, 148 ]^h \\
 & [ 22, 61, -86, 127, -140, 151 ]^h  = [ -35, -47, 94, 121, -146, 148 ]^h \\
 & [ 22, 61, 86, -127, -140, 151 ]^h  = [ 35, 47, 94, -121, 146, -148 ]^h \\
 & [ 22, 61, 86, -127, 140, 151 ]^h  = [ -35, 47, -94, 121, 146, 148 ]^h 
\end{align}
\indent
$\bullet$ Ideal solution, based on \eqref{k246810s1511}, by Chen Shuwen in 2023:
\begin{align}
 & [107,622,700,1075,-1138,1511]^h=[-293,413,-886,953,1180,1510]^h
\end{align}
\begin{mirrortype}
\end{mirrortype}
$\bullet$ \ref{hn1246810}\quad $(h=-10,-8,-6,-4,-2,-1)$
\begin{relatedtype}
\end{relatedtype}
$\bullet$ \ref{k246810}\quad $(k=2,4,6,8,10)$ 
\\
\subsubsection{\;\;(h$\;=\;$1, 2, 3, 4, 5, 6, 8)}
\label{h1234568}
\begin{notation}
\end{notation}
$\bullet$ Type: $(h=1,2,3,4,5,6,8)$ \qquad $\bullet$ Abbreviation: $[$h1234568$]$
\begin{idealsolu}
\end{idealsolu}
$\bullet$ First known solutions, based on \eqref{k123456s84} and \eqref{k123456s204} , by Chen Shuwen in 1997 \cite{Chen01,Chen23}:
\begin{align}
\label{h1234568s303}
 & [-285, -187, -173, 30, 93, 226, 296]^h \nonumber \\
 &\quad = [-264, -250, -89, -47, 170, 177, 303]^h \\
 & [-725, -683, -424, 185, 479, 486, 682]^h \nonumber \\
 &\quad = [-746, -648, -445 , 241, 346, 605, 647]^h
\end{align}
\indent
$\bullet$ Smallest known solutions, based on \eqref{k123456s264}, \eqref{k123456s391} and \eqref{k123456s399}, by Chen Shuwen in 2023:
\begin{align}
& [-127,-88,-80,18,45,97,135]^h  \nonumber \\
&\quad = [-120,-110,-57,-3,73,80,137]^h\\
\label{h1234568s205}
& [-186,-146,-99,-12,95,145,203]^h  \nonumber \\
&\quad =[-177,-165,-76,-31,118,126,205]^h \\
& [-181,-127,-122,2,75,136,217]^h  \nonumber \\
&\quad =[-168,-163,-89,-17,94,125,218]^h 
\end{align}
\begin{mirrortype}
\end{mirrortype}
$\bullet$ \ref{hn1234568}\quad $(h=-8,-6,-5,-4,-3,-2,-1)$
\begin{relatedtype}
\end{relatedtype}
$\bullet$ \ref{k123456}\quad $(k=1,2,3,4,5,6)$
\\
\subsubsection{\;\;(h$\;=\;$1, 2, 3, 4, 5, 6, 7, 9)}
\label{h12345679}
\begin{notation}
\end{notation}
$\bullet$ Type: $(h=1,2,3,4,5,6,7,9)$ \qquad $\bullet$ Abbreviation: $[$h12345679$]$
\begin{idealsolu}
\end{idealsolu}
$\bullet$ First known solutions, based on \eqref{k1234567s96} and \eqref{k1234567s321}, by Chen Shuwen in 1997 \cite{Chen01,Chen23}:
\begin{align}
\label{h12345679s47}
 & [-49, -42, -26, 1, 4, 32, 33, 47]^h \nonumber \\
 &\quad = [-48, -44, -23, -7, 14, 23, 39, 46]^h\\
\label{h12345679s317}
 & [-325, -239 , -207, 7, 19, 153, 275, 317]^h\nonumber \\
 &\quad =[-323, -253, -189, -31, 65, 135, 283, 313]^h 
\end{align}
\indent
$\bullet$ Solutions based on \eqref{k1234567s249} and \eqref{k1234567s277}, by Chen Shuwen in 2019:
\begin{align}
 & [-124,-108,-61,-3,3,78,90,125]^h \nonumber \\
 &\quad = [-118,-117,-43,-36,30,60,101,123]^h\\
 & [-273,-217,-185,1,25,149,219,281]^h\nonumber \\
 &\quad =[-265,-241,-159,-41,91,105,233,277]^h 
\end{align}
\begin{mirrortype}
\end{mirrortype}
$\bullet$ \ref{hn12345679}\quad $(h=-9,-7,-6,-5,-4,-3,-2,-1)$
\begin{relatedtype}
\end{relatedtype}
$\bullet$ \ref{k1234567}\quad $(k=1,2,3,4,5,6,7)$ 
\\
\subsection{GPTE with all h<0}
So far, ideal integer solutions of GPTE have been found for 15 types with all \( h < 0 \). 
\subsubsection{\;\;(h$\;=\;$-4,  -2,  -1)}
\label{hn124}
\begin{notation}
\end{notation}
$\bullet$ Type: $(h=-4,-2,-1)$ \qquad $\bullet$ Abbreviation: $[$hn124$]$
\begin{idealsolu}
\end{idealsolu}
$\bullet$ Smallest known solution, based on \eqref{h124s11}:
\begin{align}
 & [ -99, 110, 990 ]^h = [ -90, 165, 198 ]^h
\end{align}
\begin{idealchain}
\end{idealchain}
$\bullet$ Smallest known solution, based on \eqref{h124s40}:
\begin{align}
 & [-257980650,495322848,538394400]^h \nonumber \\
 &\quad = [-263469600,386970975,825538080]^h \nonumber \\
 &\quad = [-275179360,334677600,1547883900]^h \nonumber \\
 &\quad = [-287978400,309576780,4127690400]^h
\end{align}
\begin{mirrortype}
\end{mirrortype}
$\bullet$ \ref{h124}\quad\: $(h=1,2,4)$
\\
\subsubsection{\;\;(h$\;=\;$-6,  -2,  -1)}
\label{hn126}
\begin{notation}
\end{notation}
$\bullet$ Type: $(h=-6,-2,-1)$ \qquad $\bullet$ Abbreviation: $[$hn126$]$
\begin{idealsolu}
\end{idealsolu}
$\bullet$ Smallest known solution, based on \eqref{h126s405}:
\begin{align}
\label{hn126s5719907340}
 & [-661172085, 662954220, 5719907340]^h \nonumber\\
 &\quad = [-607298804, 801159660, 1756828683]^h
\end{align}
\begin{mirrortype}
\end{mirrortype}
$\bullet$ \ref{h126}\quad\: $(h=1,2,6)$
\\
\subsubsection{\;\;(h$\;=\;$-4,  -3,  -1)}
\label{hn134}
\begin{notation}
\end{notation}
$\bullet$ Type: $(h=-4,-3,-1)$ \qquad $\bullet$ Abbreviation: $[$hn134$]$
\begin{idealsolu}
\end{idealsolu}
$\bullet$ First known solution, based on \eqref{h134s6738}:
\begin{align}
 & [-3376103253237849, 1941647222699887,1967730608281562]^h \nonumber\\
 &\quad = [-8266245286708774, 1630430392703467, 4261380910021707]^h
\end{align}
\begin{mirrortype}
\end{mirrortype}
$\bullet$ \ref{h134}\quad $(h=1,3,4)$
\\
\subsubsection{\;\;(h$\;=\;$-4,  -3,  -2)}
\label{hn234}
\begin{notation}
\end{notation}
$\bullet$ Type: $(h=-4,-3,-2)$ \qquad $\bullet$ Abbreviation: $[$hn234$]$
\begin{idealsolu}
\end{idealsolu}
$\bullet$ First known solution, based on \eqref{h234s1233}:
\begin{align}
 & [-119374236504,79485296365,271759784220]^h \nonumber\\
 &\quad = [ -237292689636,-125373714885,  78905111720]^h
\end{align}
\begin{mirrortype}
\end{mirrortype}
$\bullet$ \ref{h234}\quad $(h=2,3,4)$
\\
\subsubsection{\;\;(h$\;=\;$-5,  -3,  -2,  -1)}
\label{hn1235}
\begin{notation}
\end{notation}
$\bullet$ Type: $(h=-5,-3,-2,-1)$ \qquad $\bullet$ Abbreviation: $[$hn1235$]$
\begin{idealsolu}
\end{idealsolu}
$\bullet$ Smallest known solution, based on \eqref{h1235s23}:
\begin{align}
 & [-106743,-15249,  13923, 320229]^h \nonumber\\
 &\quad = [-24633,-18837, 15249,  35581]^h
\end{align}
\begin{idealchain}
\end{idealchain}
$\bullet$ First known solution chain, based on \eqref{h1235s197}:
\begin{align}
 & [-801158649181667, 939429685081041,\nonumber\\
 &\qquad \qquad 8046419476563699, 16824331632815007]^h \nonumber\\
 &\quad =  [-3776890774713573, -829899766641099, \nonumber\\
 &\qquad \qquad 1033897474642263,1989974709257689]^h \nonumber\\
 &\quad = [-2682139825521233, -852846303967581,\nonumber \\
 &\qquad \qquad  1242064751415873,1350858744240621]^h
\end{align}
\begin{mirrortype}
\end{mirrortype}
$\bullet$ \ref{h1246}\quad $(h=1,2,3,5)$
\\
\subsubsection{\;\;(h$\;=\;$-6,  -4,  -2,  -1)}
\label{hn1246}
\begin{notation}
\end{notation}
$\bullet$ Type: $(h=-6,-4,-2,-1)$ \qquad $\bullet$ Abbreviation: $[$hn1246$]$
\begin{idealsolu}
\end{idealsolu}
$\bullet$ Smallest known solution, based on \eqref{h1246s25}:
\begin{align}
 & [-8400, -8050, 13800, 38640]^h= [-9200, -7728, 12075, 96600]^h 
\end{align}
\begin{mirrortype}
\end{mirrortype}
$\bullet$ \ref{h1246}\quad $(h=1,2,4,6)$
\\
\subsubsection{\;\;(h$\;=\;$-6,  -4,  -3,  -2,  -1)}
\label{hn12346}
\begin{notation}
\end{notation}
$\bullet$ Type: $(h=-6,-4,-3,-2,-1)$ \qquad $\bullet$ Abbreviation: $[$hn12346$]$
\begin{idealsolu}
\end{idealsolu}
$\bullet$ Smallest known solution, based on \eqref{h12346s17}:
\begin{align}
 & [-7140, -5712, -1680, 2040, 2380]^h \nonumber \\
 &\quad = [-2856, -1785, 1785, 4080, 9520]^h
\end{align}
\begin{mirrortype}
\end{mirrortype}
$\bullet$ \ref{h12346}\quad $(h=1,2,3,4,6)$
\\
\subsubsection{\;\;(h$\;=\;$-7,  -5,  -3,  -2,  -1)}
\label{hn12357}
\begin{notation}
\end{notation}
$\bullet$ Type: $(h=-7,-5,-3,-2,-1)$ \qquad $\bullet$ Abbreviation: $[$hn12357$]$
\begin{idealsolu}
\end{idealsolu}
$\bullet$ Smallest known solution, based on \eqref{h12357s57}:
\begin{align}
 \label{hn12357s940177095}
 & [-72321315, -24107105, -18434845, 17094129, 40877265]^h \nonumber \\
 &\quad = [-28490215, -16494335, 15935205, 188035419, 940177095]^h
\end{align}
\begin{mirrortype}
\end{mirrortype}
$\bullet$ \ref{h12357}\quad $(h=1,2,3,5,7)$
\\
\subsubsection{\;\;(h$\;=\;$-8,  -6,  -4,  -2,  -1)}
\label{hn12468}
\begin{notation}
\end{notation}
$\bullet$ Type: $(h=-8,-6,-4,-2,-1)$ \qquad $\bullet$ Abbreviation: $[$hn12468$]$
\begin{idealsolu}
\end{idealsolu}
$\bullet$ Smallest known solution, based on \eqref{h12468s313}:
\begin{align}
 & [-955521822093300, -695407548301235, 406407008747475,\nonumber\\
 &\qquad\qquad 407730810078900,1763005052031300]^h\nonumber\\
 &\quad = [-399914884007100, 415858334532300, 665815737735225,\nonumber\\
 &\qquad\qquad 1251733586942223,1264377360547700]^h
\end{align}
\begin{mirrortype}
\end{mirrortype}
$\bullet$ \ref{h12468}\quad $(h=1,2,4,6,8)$
\\
\subsubsection{\;\;(h$\;=\;$-7,  -5,  -4,  -3,  -2,  -1)}
\label{hn123457}
\begin{notation}
\end{notation}
$\bullet$ Type: $(h=-7,-5,-4,-3,-2,-1)$ \qquad $\bullet$ Abbreviation: $[$hn123457$]$
\begin{idealsolu}
\end{idealsolu}
$\bullet$ Smallest known solution, based on \eqref{h123457s71}:
\begin{align}
 & [-1244795491060, -970208250385, 1081543623380,\nonumber\\
 &\qquad\qquad 1199530200476, 16493540256545, 65974161026180]^h\nonumber\\
 &\quad = [-3298708051309, -1499412750595, -929213535580,\nonumber\\
 &\qquad\qquad 984688970540, 1783085433140, 2128198742780]^h
\end{align}
\begin{mirrortype}
\end{mirrortype}
$\bullet$ \ref{h123457}\quad $(h=1,2,3,4,5,7)$
\\
\subsubsection{\;\;(h$\;=\;$-8,  -6,  -4,  -3,  -2,  -1)}
\label{hn123468}
\begin{notation}
\end{notation}
$\bullet$ Type: $(h=-8,-6,-4,-3,-2,-1)$ \qquad $\bullet$ Abbreviation: $[$hn123468$]$
\begin{idealsolu}
\end{idealsolu}
$\bullet$ First known solution, based on \eqref{h123468s62}:
\begin{align}
 & [-13120911510, -4373637170, -2512514970,\nonumber\\
 &\qquad\qquad 1904648445, 2186818585, 3373948674]^h \nonumber\\
 &\quad = [-118088203590, -3191573070, -2811623895,\nonumber\\
 &\qquad\qquad 1935872190, 2071722870, 3936273453]^h
\end{align}
\begin{mirrortype}
\end{mirrortype}
$\bullet$ \ref{h123468}\quad $(h=1,2,3,4,6,8)$
\\
\subsubsection{\;\;(h$\;=\;$-9,  -7,  -5,  -3,  -2,  -1)}
\label{hn123579}
\begin{notation}
\end{notation}
$\bullet$ Type: $(h=-9,-7,-5,-3,-2,-1)$ \qquad $\bullet$ Abbreviation: $[$hn123579$]$
\begin{idealsolu}
\end{idealsolu}
$\bullet$ First known solution, based on \eqref{h123579s323}:
\begin{align}
 & [-25775391257919598477, -2403493075497325903, \nonumber\\
 &\qquad\qquad -2388135292267311041,2778759652340774557,\nonumber\\
 &\qquad\qquad 4320730326472071421,106783763782809765119]^h\nonumber\\
 &\quad = [-8214135675600751163, -2586457946296430297,\nonumber\\
 &\qquad\qquad  -2314199215107332371,3026260512063434639,\nonumber\\
 &\qquad\qquad  3872986251189991481, 12669260109824887387]^h
\end{align}
\begin{mirrortype}
\end{mirrortype}
$\bullet$ \ref{h123579}\quad $(h=1,2,3,5,7,9)$
\\
\subsubsection{\;\;(h$\;=\;$-10,  -8,  -6,  -4,  -2,  -1)}
\label{hn1246810}
\begin{notation}
\end{notation}
$\bullet$ Type: $(h=-10,-8,-6,-4,-2,-1)$ \qquad $\bullet$ Abbreviation: $[$hn1246810$]$
\begin{idealsolu}
\end{idealsolu}
$\bullet$ Smallest known solution, based on \eqref{h1246810s151}:
\begin{align}
 & [-3090628197624308, -1150765818264370, 730891803492235, \nonumber\\
 &\qquad\qquad 740904019978430, 893983362949180, 2301531636528740]^h\nonumber\\
 &\quad = [-851747928479140, 716370774283780, 772657049406077,\nonumber\\
 &\qquad\qquad  1257813801358730, 1773311260931980, 4916908496220490]^h
\end{align}
\begin{mirrortype}
\end{mirrortype}
$\bullet$ \ref{h1246810}\quad $(h=1,2,4,6,8,10)$
\\
\subsubsection{\;\;(h$\;=\;$-8,  -6,  -5,  -4,  -3,  -2,  -1)}
\label{hn1234568}
\begin{notation}
\end{notation}
$\bullet$ Type: $(h=-8,-6,-5,-4,-3,-2,-1)$ \qquad $\bullet$ Abbreviation: $[$hn1234568$]$
\begin{idealsolu}
\end{idealsolu}
$\bullet$ Smallest known solution, based on \eqref{h1234568s205}:
\begin{align}
 & [-354285390690, -331791080170, -203930127324, 236190260460,  \nonumber\\
 &\qquad\qquad 253367733948, 550074685545, 1348570196820]^h\nonumber\\
 &\quad = [-440059748436, -288315007596, -205939291140, 224761699470, \nonumber\\
 &\qquad\qquad 286340247270, 422279556580, 3483806341785]^h 
\end{align}
\begin{mirrortype}
\end{mirrortype}
$\bullet$ \ref{h1234568}\quad $(h=1,2,3,4,5,6,8)$
\\
\subsubsection{\;\;(h$\;=\;$-9,  -7,  -6,  -5,  -4,  -3,  -3,  -2,  -1)}
\label{hn12345679}
\begin{notation}
\end{notation}
$\bullet$ Type: $(h=-9,-7,-6,-5,-4,-3,-2,-1)$ \vspace{1ex}\\
\indent
$\bullet$ Abbreviation: $[$hn12345679$]$
\begin{idealsolu}
\end{idealsolu}
$\bullet$ Smallest known solution, based on \eqref{h12345679s47}:
\begin{align}
 & [-103879776, -31615584, -16526328, -15149134,\nonumber\\
 &\qquad\qquad 15807792, 18645088, 31615584, 51939888]^h\nonumber\\
 &\quad = [-27967632, -17313296, -14839968, 15471456,\nonumber\\
 &\qquad\qquad 22035104, 22723701, 181789608, 727158432]^h
\end{align}
\begin{mirrortype}
\end{mirrortype}
$\bullet$ \ref{h12345679}\quad $(h=1,2,3,4,5,6,7,9)$
\\
\subsection{GPTE with h$_{1}$=0 and all others h>0}
Until now, ideal integer solutions of GPTE have been discovered for 6 types with \( h_1 = 0 \) and all others \( h > 0 \). 
\subsubsection{\;\;(h$\;=\;$0, 1, 3)}
\label{h013}
\begin{notation}
\end{notation}
$\bullet$ Type: $(h=0,1,3)$ \qquad $\bullet$ Abbreviation: $[$h013$]$
\begin{idealsolu}
\end{idealsolu}
$\bullet$ First known solution, by A.Moessner in 1939:
\begin{align}
\label{h013s15}
 & [ -15, 1, 14]^h = [ -10, 3, 7 ]^h
\end{align}
\indent
$\bullet$ Ajai Choudhry gave a parametric solution in 1991 \cite{Choudhry91}.\\[1mm]
\indent
$\bullet$ Titus Piezas gave a three-parameter solution in 2009 \cite[p.395]{Piezas09}. Numerical examples:
\begin{align}
 & [ -16,1,15]^h = [-12,2,10]^h\\
 & [ -44,5,39]^h = [-33,13,20]^h
\end{align}
\begin{idealchain}
\end{idealchain}
$\bullet$ First known solution chains, by Ajai Choudhty in 2010 \cite{Choudhry10}. Numerical example:
\begin{align}
 & [-517561088, 12130338, 505430750]^h \nonumber \\
 & \quad =[-281011375, 48582831, 232428544]^h \nonumber \\
 & \quad =[-242606760, 80868920, 161737840]^h 
\end{align}
\indent
$\bullet$ Smallest solution chains of length 6 and length 9, based on computer search of $(k=1,3)$, by Chen Shuwen in 2013 \cite{ChenkProducts23}:
\begin{align}
\label{h013s231}
 & [-240,9,231]^h=[-210,12,198]^h=[-165,21,144]^h \nonumber \\
 & \quad = [-162,22,140]^h=[-132,42,90]^h=[-126,60,66]^h \\
\label{h013s1610}
 & [-1632,22,1610]^h=[-1564,24,1540]^h=[-1012,60,952]^h \nonumber \\
 & \quad =[-840,92,748]^h=[-782,110,672]^h=[-759,119,640]^h \nonumber \\
 & \quad = [-660,184,476]^h=[-644,204,440]^h=[-616,276,340]^h
\end{align}
\begin{mirrortype}
\end{mirrortype}
$\bullet$ \ref{h0n13} \quad $(h=-3,-1,0)$
\begin{relatedtype}
\end{relatedtype}
$\bullet$ \ref{k13} \quad $(k=1,3)$ 
\\
\subsubsection{\;\;(h$\;=\;$0, 1, 2, 4)}
\label{h0124}
\begin{notation}
\end{notation}
$\bullet$ Type: $(h=0,1,2,4)$ \qquad $\bullet$ Abbreviation: $[$h0124$]$
\begin{idealsolu}
\end{idealsolu}
$\bullet$ Ajai Choudhry obtained a three-parameter solution in 2013 \cite [p.781-782]{Choudhry13}. Numerical example:
\begin{align}
\label{h0124s134}
 & [ -138, 9, 62, 67 ]^h = [ -93, -23, -18, 134 ]^h
\end{align}
\indent
$\bullet$ Smallest solutions, based on \eqref{k024s28} to \eqref{k024s39}, by Chen Shuwen in 2023 \cite{ChenkProducts23}:
\begin{align}
\label{h0124s28}
 & [-28,1,9,18]^h = [-27,2,4,21]^h \\
 & [-36,1,17,18]^h = [-34,3,4,27]^h \\
 & [-36,-2,17,21]^h = [-34,-3,9,28]^h \\
 & [-39,2,18,19]^h = [-38,3,9,26]^h 
\end{align}
\begin{idealchain}
\end{idealchain}
$\bullet$ First known solution chains, based on \eqref{k024s1150} and \eqref{k024s2905}, by Chen Shuwen in 2023:
\begin{align}
\label{h0124s1150}
 & [-1025, -72, 62, 1035]^h = [-930, -200, 23, 1107]^h \nonumber \\
 & \quad = [-837, -328, 15, 1150]^h\\
\label{h0124s2905}
& [-2491,-427,345,2573]^h = [-1909,-1113,155,2867]^h  \nonumber \\
& \quad =  [-1643,-1403,141,2905]^h
\end{align}
\begin{mirrortype}
\end{mirrortype}
$\bullet$ \ref{h0n124} \quad\: $(h=-4,-2,-1,0)$
\begin{relatedtype}
\end{relatedtype}
$\bullet$ \ref{k024} \quad $(k=0,2,4)$
\\
\subsubsection{\;\;(h$\;=\;$0, 1, 3, 5)}
\label{h0135}
\begin{notation}
\end{notation}
$\bullet$ Type: $(h=0,1,3,5)$ \qquad $\bullet$ Abbreviation: $[$h0135$]$
\begin{idealsolu}
\end{idealsolu}
$\bullet$ First known solutions, based on computer search of $(k=1,3,5)$, by Chen Shuwen in 2017 \cite{ChenkProducts23}:
\begin{align}
\label{h0135s323}
 & [ -260, 32, 189, 323 ]^h = [ -210, 68, 114, 312]^h\\
\label{h0135s703}
 & [ -555, -21, 646, 650 ]^h = [ -255, -78, 350, 703]^h
\end{align}
\indent
$\bullet$ Third known solution, same method as above, by Chen Shuwen in 2023:
\begin{align}
\label{h0135s1581}
 & [-1518, 14, 1363, 1581 ] ^h= [ -561, 138, 406, 1457]^h
\end{align}
\indent
$\bullet$ Chen Shuwen noted that both \eqref{h0135s703} and \eqref{h0135s1581} also satisfy the condition $a_1+a_3=b_1+b_3$ and $a_2+a_4=b_2+b_4$, then constructed a parametric method, and obtained a new numerical solution in 2023.
\begin{align}
 & [-5292, 820, 3320, 6235] ^h= [-5375, 783, 3403, 6272]^h
\end{align}
\begin{mirrortype}
\end{mirrortype}
$\bullet$ \ref{h0n135} \quad\: $(h=-5,-3,-1,0)$
\begin{relatedtype}
\end{relatedtype}
$\bullet$ \ref{k135} \quad $(k=1,3,5)$ 
\\
\subsubsection{\;\;(h$\;=\;$0, 1, 2, 3, 5)}
\label{h01235}
\begin{notation}
\end{notation}
$\bullet$ Type: $(h=0,1,2,3,5)$ \qquad $\bullet$ Abbreviation: $[$h01235$]$
\begin{idealsolu}
\end{idealsolu}
$\bullet$ Ajai Choudhry obtained two parametric solutions in 2013 \cite [p.782-785]{Choudhry13}. Numerical examples:
\begin{align}
\label{h01235s28}
 & [ -21, -20, 2, 13, 26 ]^h = [ -26, -13, 5, 6, 28]^h \\
 & [-1314,-209,116,136,1271]^h = [-1189,-584,11,646,1116]^h\\
 & [-2068, -369, 104, 742, 1591]^h = [-2021, -533, 56, 1166, 1332]^h \\
 & [-8308, -3128, 1157, 2322, 7957]^h = [-7303, -4968, 442, 5332, 6497]^h
\end{align}
\indent
$\bullet$ Chen Shuwen found a new parametric method in 2023 \cite{ChenkProducts23}. Numerical examples:
\begin{align}
\label{h01235s15}
 & [ -14,-7,2,4,15 ]^h = [-12,-10,1,7,14]^h \\
 & [ -69,-17,4,30,52 ]^h = [-68,-20,3,39,46]^h \\
 & [ -77,-40,19,20,78 ]^h = [-65,-57,8,44,70]^h \\
 & [ -72,-36,7,14,87 ]^h = [-58,-56,3,27,84]^h \\
 & [-108,-41,17,18,114]^h = [-82,-81,4,57,102]^h \\
 & [-104,-66,27,38,105]^h = [-95,-78,18,56,99]^h 
\end{align}
\begin{mirrortype}
\end{mirrortype}
$\bullet$ \ref{h0n1235} \quad $(h=-5,-3,-2,-1,0)$
\\
\subsubsection{\;\;(h$\;=\;$0, 1, 2, 4, 6)}
\label{h01246}
\begin{notation}
\end{notation}
$\bullet$ Type: $(h=0,1,2,4,6)$ \qquad $\bullet$ Abbreviation: $[$h01246$]$
\begin{idealsolu}
\end{idealsolu}
$\bullet$ First known solutions, based on \eqref{k0246s50} to \eqref{k0246s110}, by Chen Shuwen in 2017 \cite{ChenkProducts23}:
\begin{align}
\label{h01246s50}
 & [-48, -2, 16, 25, 45 ]^h = [-40, -9, 3, 32, 50]^h\\
 & [-68, -8, 33, 38, 87]^h = [-58, -17, 12, 57, 88]^h\\
\label{h01246s110}
 & [-105, -11, 42, 58, 100]^h = [-87, -28, 14, 75, 110]^h
\end{align}
\begin{mirrortype}
\end{mirrortype}
$\bullet$ \ref{h0n1246} \quad\: $(h=-6,-4,-2,-1,0)$
\begin{relatedtype}
\end{relatedtype}
$\bullet$ \ref{k0246} \quad $(k=0,2,4,6)$ 
\\
\subsubsection{\;\;(h$\;=\;$0, 1, 2, 3, 4, 6)}
\label{h012346}
\begin{notation}
\end{notation}
$\bullet$ Type: $(h=0,1,2,3,4,6)$ \qquad $\bullet$ Abbreviation: $[$h012346$]$
\begin{idealsolu}
\end{idealsolu}
$\bullet$ Ajai Choudhry obtained a method for infinitely many solutions in 2013 \cite [p.785-787]{Choudhry13}. Numerical solution:
\begin{align}
\label{h012346s11716}
 & [-11716, -6437, -1460, 1175, 7897, 10541 ]^h \nonumber \\
 & \quad = [ -10865, -8383, -596, 3683, 4700, 11461]^h
\end{align}
\indent
$\bullet$ Chen Shuwen found a new parametric method in 2023 \cite{ChenkProducts23}. Numerical examples:
\begin{align}
\label{h012346s50}
 & [-45, -25, -16, 2, 36, 48]^h = [-40, -36, -9, 3, 32, 50]^h\\
\label{h012346s65}
 & [-52, -23, -20, 5, 26, 64]^h = [-46, -40, -8, 13, 16, 65]^h\\
 & [-88, -58, 12, 17, 57, 60]^h = [-87, -60, 8, 33, 38, 68]^h\\
 & [-68, -33, -32, 8, 38, 87]^h = [-58, -57, -17, 12, 32, 88]^h\\
 & [-77, -33, -31, 5, 45, 91]^h = [-65, -63, -11, 11, 35, 93]^h\\
 & [-110, -94, 14, 28, 75, 87]^h = [-105, -100, 11, 42, 58, 94]^h\\
 & [-113, -82, 11, 15, 58, 111]^h = [-110, -87, 6, 37, 41, 113]^h\\
 & [-102, -56, -33, 7, 68, 116]^h = [-88, -84, -17, 12, 58, 119]^h
\end{align}
\begin{mirrortype}
\end{mirrortype}
$\bullet$ \ref{h0n12346} \quad $(h=-6,-4,-3,-2,-1,0)$
\\
\subsection{GPTE with h$_{n}$=0 and all others h<0}
Ideal integer solutions of GPTE have been found for 6 types with \( h_n = 0 \) and all others \( h < 0 \).
\subsubsection{\;\;(h$\;=\;$-3, -1, 0)}
\label{h0n13}
\begin{notation}
\end{notation}
$\bullet$ Type: $(h=-3,-1,0)$ \qquad $\bullet$ Abbreviation: $[$h0n13$]$
\begin{idealsolu}
\end{idealsolu}
$\bullet$ Smallest known solution, based on \eqref{h013s15}:
\begin{align}
& [-21,30,70]^h= [-14,15,210]^h
\end{align}
\begin{idealchain}
\end{idealchain}
$\bullet$ Smallest known solution chains of length 6 and length 9, based on \eqref{h013s231} and \eqref{h013s1610}, by Chen Shuwen in 2023:
\begin{align}
 & [-3960,7560,8316]^h=[-3780,5544,11880]^h \nonumber \\
 & \quad =[-3080,3564,22680]^h = [-3024,3465,23760]^h \nonumber \\
 & \quad =[-2376,2520,41580]^h=[-2079,2160,55440]^h \\
 & [-93840,170016,209440]^h=[-89760,131376,283360]^h \nonumber \\
 & \quad =[-87584,121440,314160]^h =[-76160,90321,485760]^h \nonumber \\
 & \quad =[-73920,86020,525504]^h=[-68816,77280,628320]^h \nonumber \\
 & \quad = [-57120,60720,963424]^h=[-36960,37536,2408560]^h \nonumber \\
 & \quad =[-35420,35904,2627520]^h
\end{align}
\begin{mirrortype}
\end{mirrortype}
$\bullet$ \ref{h013} \quad $(h=0,1,3)$
\\
\subsubsection{\;\;(h$\;=\;$-4, -2, -1, 0)}
\label{h0n124}
\begin{notation}
\end{notation}
$\bullet$ Type: $(h=-4,-2,-1,0)$ \qquad $\bullet$ Abbreviation: $[$h0n124$]$
\begin{idealsolu}
\end{idealsolu}
$\bullet$ Smallest known solution, based on \eqref{h0124s28}:
\begin{align}
& [-28, 36, 189, 378]^h= [-27, 42, 84, 756]^h
\end{align}
\begin{idealchain}
\end{idealchain}
$\bullet$ Ideal solution chains of length 3, based on \eqref{h0124s1150}:
\begin{align}
& [-2192475, -154008, 152520, 2546100]^h  \nonumber\\
& \quad =[-789291, -169740, 142600, 6863400]^h  \nonumber\\
& \quad =[-481275, -188600, 137268, 10523880]^h
\end{align}
\begin{mirrortype}
\end{mirrortype}
$\bullet$ \ref{h0124} \quad $(h=0,1,2,4)$
\\
\subsubsection{\;\;(h$\;=\;$-5, -3, -1, 0)}
\label{h0n135}
\begin{notation}
\end{notation}
$\bullet$ Type: $(h=-5,-3,-1,0)$ \qquad $\bullet$ Abbreviation: $[$h0n135$]$
\begin{idealsolu}
\end{idealsolu}
$\bullet$ Smallest known solution, based on \eqref{h0135s323}:
\begin{align}
\label{h0n135s3968055}
 & [-604656,406980,1113840,1867320]^h \nonumber \\
 & \quad = [-488376,393120,671840,3968055]^h
\end{align}
\begin{mirrortype}
\end{mirrortype}
$\bullet$ \ref{h0135} \quad $(h=0,1,3,5)$
\\
\subsubsection{\;\;(h$\;=\;$-5, -3, -2, -1, 0)}
\label{h0n1235}
\begin{notation}
\end{notation}
$\bullet$ Type: $(h=-5,-3,-2,-1,0)$ \qquad $\bullet$ Abbreviation: $[$h0n1235$]$
\begin{idealsolu}
\end{idealsolu}
$\bullet$ Smallest known solutions, based on \eqref{h01235s15} and \eqref{h01235s28}:
\begin{align}
 & [-60, -30, 28, 105, 210]^h=[-42, -35, 30, 60, 420]^h \\
 & [ -420, -210, 195, 910, 1092]^h=[-273, -260, 210, 420, 2730]^h
\end{align}
\begin{mirrortype}
\end{mirrortype}
$\bullet$ \ref{h01235} \quad $(h=0,1,2,3,5)$
\\
\subsubsection{\;\;(h$\;=\;$-6, -4, -2, -1, 0)}
\label{h0n1246}
\begin{notation}
\end{notation}
$\bullet$ Type: $(h=-6,-4,-2,-1,0)$ \qquad $\bullet$ Abbreviation: $[$h0n1246$]$
\begin{idealsolu}
\end{idealsolu}
$\bullet$ Smallest known solutions, based on \eqref{h01246s50} and \eqref{h01246s110}:
\begin{align}
\label{h0n1246s2400}
 & [-3600, -150, 160, 288, 450]^h=[-800, -180, 144, 225, 2400]^h\\
 & [-60900, -6380, 6699, 11550, 15950]^h \nonumber \\
 & \quad =[-23925, -7700, 6090, 8932, 47850]^h
\end{align}
\begin{mirrortype}
\end{mirrortype}
$\bullet$ \ref{h01246} \quad $(h=0,1,2,4,6)$
\\
\subsubsection{\;\;(h$\;=\;$-6, -4, -3, -2, -1, 0)}
\label{h0n12346}
\begin{notation}
\end{notation}
$\bullet$ Type: $(h=-6,-4,-3,-2,-1,0)$ \qquad $\bullet$ Abbreviation: $[$h0n12346$]$
\begin{idealsolu}
\end{idealsolu}
$\bullet$ First known solution, based on \eqref{h012346s11716}:
\begin{align}
 & [-368820196108882435, -83653485524152300, -45960864315377975,\nonumber\\
 & \qquad \quad 51084098882361100, 68187601154738300, 458278711760824132]^h \nonumber\\
 & \quad = [-49560744253931740, -64234460970889700, 46983464472469100,\nonumber\\
 & \qquad \quad 114569677940206033, 146206214042619700, 903485715300282475]^h 
\end{align}
\indent
$\bullet$ Smallest known solutions, based on \eqref{h012346s50} and \eqref{h012346s65}, by Chen Shuwen in 2023:
\begin{align}
 & [-800, -200, -180, 144, 225, 2400]^h \nonumber\\
 & \quad  =[-450, -288, -160, 150, 200, 3600]^h \\
 & [-11960, -2392, -2080, 1472, 5980, 7360]^h \nonumber\\
 & \quad  =[-4784, -4160, -1840, 1495, 3680, 19136]^h 
\end{align}
\begin{mirrortype}
\end{mirrortype}
$\bullet$ \ref{h012346} \quad $(h=0,1,2,3,4,6)$
\\
\subsection{GPTE with h$_{1}$<0 and h$_{n}$>0}
Until now, ideal integer solutions of GPTE have been found for 23 types with \( h_1 < 0 \) and \( h_n > 0 \). 
\subsubsection{\;\;(h$\;=\;$-1, 0, 2)}
\label{h02n1}
\begin{notation}
\end{notation}
$\bullet$ Type: $(h=-1,0,2)$ \qquad $\bullet$ Abbreviation: $[$h02n1$]$
\begin{idealsolu}
\end{idealsolu}
$\bullet$ First known solutions, smallest solutions, based on $(k=0,2)$, by Chen Shuwen in 2017 \cite{Chenhminus23}:
\begin{align}
\label{h02n1s15}
& [-12,-10,1]^h = [2,4,15]^h\\
\label{h02n1s28}
& [-21,-20,2]^h = [5,6,28]^h\\
& [-42,9,20]^h = [-30,7,36]^h\\
& [-45,14,20]^h = [-35,10,36]^h\\
& [-35,-24,2]^h = [3,14,40]^h\\
& [-55,-42,4]^h = [7,20,66]^h\\
& [-70,-44,3]^h = [4,30,77]^h\\
& [-78,-56,5]^h = [8,30,91]^h\\
& [-88,-65,6]^h = [10,33,104]^h
\end{align}
\begin{idealchain}
\end{idealchain}
$\bullet$ Based on \eqref{m3T012n1}, Chen Shuwen proved that there is no solution chains for length $j\ge 3$.
\begin{mirrortype}
\end{mirrortype}
$\bullet$ \ref{h01n2} \quad $(h=-2,0,1)$
\begin{relatedtype}
\end{relatedtype}
$\bullet$ \ref{k02} \quad $(k=0,2)$
\\
\subsubsection{\;\;(h$\;=\;$-1, 1, 2)}
\label{h12n1}
\begin{notation}
\end{notation}
$\bullet$ Type: $(h=-1,1,2)$ \qquad $\bullet$ Abbreviation: $[$h12n1$]$
\begin{idealsolu}
\end{idealsolu}
$\bullet$ Ajai Choudhry found a parametric solution in 2011 \cite{Choudhry11}. Numerical examples:
\begin{align}
\label{h12n1s18}
& [-7,14,14]^h=[-6,9,18]^h\\
& [-15, 24, 40]^h=[-14, 21, 42]^h\\
& [-56, 105, 120]^h =[-39, 52, 156]^h
\end{align}
\begin{idealchain}
\end{idealchain}
$\bullet$ Ajai Choudhry's above parametric solution can provide multigrade chains of arbitrary length. Numerical examples: 
\begin{align}
\label{h12n1s20}
& [-7, 14, 14]^h=[-6, 9, 18]^h =[-4, 5, 20]^h \\
& [-13, 26, 26]^h=[-10, 14, 35]^h =[-9, 12, 36]^h\\
\label{h12n1s90}
& [-30, 55, 66]^h=[-21, 28, 84]^h =[-26, 39, 78]^h =[-9, 10, 90]^h\\
& [-91, 182, 182]^h=[-90, 165, 198]^h =[-88, 152, 209]^h \nonumber \\
& \quad = [-78, 117, 234]^h=[-70, 98, 245]^h=[-63, 84, 252]^h \nonumber \\
& \quad =[-52, 65, 260]^h =[-27, 30, 270]^h=[-16, 17, 272]^h 
\end{align}
\begin{mirrortype}
\end{mirrortype}
$\bullet$ \ref{h1n12} \quad $(h=-2,-1,1)$
\\
\subsubsection{\;\;(h$\;=\;$-2, 0, 1)}
\label{h01n2}
\begin{notation}
\end{notation}
$\bullet$ Type: $(h=-2,0,1)$ \qquad $\bullet$ Abbreviation: $[$h01n2$]$
\begin{idealsolu}
\end{idealsolu}
$\bullet$ Smallest known solutions, based on \eqref{h02n1s15} and \eqref{h02n1s28}, by Chen Shuwen in 2017:
\begin{align}
& [ -6, -5, 60 ] = [ 4, 15, 30]^h\\
& [-21,-20,210]^h = [15,70,84]^h
\end{align}
\begin{mirrortype}
\end{mirrortype}
$\bullet$ \ref{h02n1} \quad $(h=-1,0,2)$
\\
\subsubsection{\;\;(h$\;=\;$-2, -1, 1)}
\label{h1n12}
\begin{notation}
\end{notation}
$\bullet$ Type: $(h=-2,-1,1)$ \qquad $\bullet$ Abbreviation: $[$h1n12$]$
\begin{idealsolu}
\end{idealsolu}
$\bullet$ Smallest solution, based on \eqref{h12n1s18}:
\begin{align}
& [-21, 7, 14]^h=[-18, 9, 9]^h 
\end{align}
\begin{idealchain}
\end{idealchain}
$\bullet$ There are ideal solution chains for arbitrary length. Numerical examples, based on \eqref{h12n1s20} and \eqref{h12n1s90}:
\begin{align}
& [-315, 63, 252]^h=[-210, 70, 140]^h=[-180, 90, 90]^h \\
& [-20020, 2002, 18018]^h =[-8580, 2145, 6435]^h \nonumber \\
&\quad =[-6930, 2310, 4620]^h  =[-6006, 2730, 3276]^h 
\end{align}
\begin{mirrortype}
\end{mirrortype}
$\bullet$ \ref{h12n1} \quad $(h=-1,1,2)$
\\
\subsubsection{\;\;(h$\;=\;$-1, 0, 1, 3)}
\label{h013n1}
\begin{notation}
\end{notation}
$\bullet$ Type: $(h=-1,0,1,3)$ \qquad $\bullet$ Abbreviation: $[$h013n1$]$
\begin{idealsolu}
\end{idealsolu}
$\bullet$ First known solution, by Chen Shuwen in 2017 \cite{Chenhminus23}:
\begin{align}
\label{h013n1s30}
& [-30, 4, 5, 21 ]^h = [ -28, 3, 10, 15]^h\\
\label{h013n1s28}
& [-28, -5, 3, 30 ]^h = [ -15, -10, 4, 21]^h
\end{align}
\indent
$\bullet$ Parametric solution \eqref{parah013n1a} of $(h=-1,0,1,3)$ was obtained by Chen Shuwen in 2017, based on Ajai Choudhry's parametric solution \eqref{parah123n1} for $(h=-1,1,2,3)$.
\begin{mirrortype}
\end{mirrortype}
$\bullet$ \ref{h01n13} \quad $(h=-3,-1,0,1)$
\\
\subsubsection{\;\;(h$\;=\;$-1, 0, 1, 5)}
\label{h015n1}
\begin{notation}
\end{notation}
$\bullet$ Type: $(h=-1,0,1,5)$ \qquad $\bullet$ Abbreviation: $[$h015n1$]$
\begin{idealsolu}
\end{idealsolu}
$\bullet$ First known solution, based on \eqref{k15n1s891}, by Chen Shuwen in 2017 \cite{Chenhminus23}:
\begin{align}
\label{h015n1s702}
& [ -891, 85, 286, 702 ]^h = [ -858, 81, 374, 585 ]^h
\end{align}
\indent
$\bullet$ Only the above solution is known so far.
\begin{mirrortype}
\end{mirrortype}
$\bullet$ \ref{h01n15}\quad $(k=-5,-1,0,1)$
\begin{relatedtype}
\end{relatedtype}
$\bullet$ \ref{k15n1}\quad\: $(k=-1,1,5)$ 
\\
\subsubsection{\;\;(h$\;=\;$-1, 0, 2, 4)}
\label{h024n1}
\begin{notation}
\end{notation}
$\bullet$ Type: $(h=-1,0,2,4)$ \qquad $\bullet$ Abbreviation: $[$h024n1$]$
\begin{idealsolu}
\end{idealsolu}
$\bullet$ First known solution, based on the solutions of $(k=0,2,4)$, by Chen Shuwen in 2019 \cite{Chenhminus23}:
\begin{align}
\label{h024n1s315}
& [ -55, -77, 180, 312 ]^h = [ -40, -132, 143, 315 ]^h
\end{align}
\indent
$\bullet$ Second known solution, smallest solution, based on computer search of $(k=0,2,4)$, by Chen Shuwen in 2023:
\begin{align}
\label{h024n1s99}
& [-91, -90, -22, 21]^h = [-14, 35, 78, 99]^h
\end{align}
\indent
$\bullet$ Only the above two solutions are found so far.
\begin{mirrortype}
\end{mirrortype}
$\bullet$ \ref{h01n24} \quad $(h=-4,-2,0,1)$
\begin{relatedtype}
\end{relatedtype}
$\bullet$ \ref{k024} \quad $(k=0,2,4)$
\\
\subsubsection{\;\;(h$\;=\;$-1, 1, 2, 3)}
\label{h123n1}
\begin{notation}
\end{notation}
$\bullet$ Type: $(h=-1,1,2,3)$ \qquad $\bullet$ Abbreviation: $[$h123n1$]$
\begin{idealsolu}
\end{idealsolu}
$\bullet$ Ajai Choudhry found a parametric solution \eqref{parah123n1} in 2011 \cite{Choudhry11}. Numerical examples:
\begin{align}
\label{h123n1s30}
& [-5, 10, 15, 30]^h = [-3, 4, 21, 28]^h\\
& [-35, 84, 85, 204]^h=[-26, 39, 130, 195]^h
\end{align}
\begin{idealchain}
\end{idealchain}
$\bullet$ Ajai Choudhry's above parametric solution can provide multigrade chains
of arbitrary length. Numerical examples:
\begin{align}
\label{h123n1s78}
& [-7, 9, 56, 72 ]^h = [ -10, 15, 50, 75 ]^h \nonumber \\
& \quad = [ -12, 21, 44, 77 ]^h = [ -13, 26, 39, 78]^h\\
& [-228, 516, 589, 1333]^h = [-225, 475, 630, 1330]^h \nonumber \\
& \quad = [-221, 442, 663, 1326]^h = [-207, 369, 736, 1312]^h \nonumber\\
& \quad = [-204, 357, 748, 1309]^h = [-195, 325, 780, 1300]^h \nonumber\\
& \quad = [-182, 286, 819, 1287]^h = [-170, 255, 850, 1275]^h \nonumber\\
& \quad = [-119, 153, 952, 1224]^h = [-116, 148, 957, 1221]^h \nonumber\\
& \quad = [-78, 91, 1014, 1183]^h = [-50, 55, 1050, 1155]^h \nonumber\\
& \quad = [-23, 24, 1081, 1128]^h
\end{align}
\begin{mirrortype}
\end{mirrortype}
$\bullet$ \ref{h1n123} \quad $(h=-3,-2,-1,1)$
\begin{relatedtype}
\end{relatedtype}
$\bullet$ \ref{k13n1} \quad\: $(k=-1,1,3)$
\\
\subsubsection{\;\;(h$\;=\;$-2, 0, 1, 2)}
\label{h012n2}
\begin{notation}
\end{notation}
$\bullet$ Type: $(h=-2,0,1,2)$ \qquad $\bullet$ Abbreviation: $[$h012n2$]$
\begin{idealsolu}
\end{idealsolu}
$\bullet$ First known solutions, based on \eqref{k02n2s42} and \eqref{k02n2s120} to \eqref{k02n2s132}, by Chen Shuwen in 2017 \cite{Chenhminus23}:
\begin{align}
\label{h012n2s4}
& [-6,7,30,35]^h = [-5,14,15,42]^h \\
& [-30,-7,-6,35]^h = [-42,5,14,15]^h \\
& [-7,28,30,120]^h = [-12,8,70,105]^h \\
& [-120,7,28,30]^h = [-105,-12,-8,70]^h \\
& [-7,14,60,120]^h = [-10,8,84,105]^h \\
& [-120,7,14,60]^h = [-105,-10,-8,84]^h \\
& [-10,33,40,132]^h = [-15,12,88,110]^h \\
\label{h012n2s88}
& [-132,10,33,40]^h = [-110,-15,-12,88]^h 
\end{align}
\begin{mirrortype}
\end{mirrortype}
$\bullet$ \ref{h02n12} \quad $(h=-2,-1,0,2)$
\begin{relatedtype}
\end{relatedtype}
$\bullet$ \ref{k02n2} \quad $(k=-2,0,2)$
\\
\subsubsection{\;\;(h$\;=\;$-2, -1, 0, 2)}
\label{h02n12}
\begin{notation}
\end{notation}
$\bullet$ Type: $(h=-2,-1,0,2)$ \qquad $\bullet$ Abbreviation: $[$h02n12$]$
\begin{idealsolu}
\end{idealsolu}
$\bullet$ First known solutions, based on \eqref{h012n2s4} to \eqref{h012n2s88}, by Chen Shuwen in 2017:
\begin{align}
& [ -42,5, 14, 15 ]^h=[-35,6, 7, 30 ]^h \\
& [-42,-15,-14,5]^h=[-6,7,30,35 ]^h \\
& [-120,-30,-28,7 ]^h=[-12,8,70,105 ]^h \\
& [-120,7,28,30]^h=[-70,8,12,105]^h \\
& [-120,-60,-14,7]^h=[-10,8,84,105]^h \\
& [-120,7,14,60]^h=[-84,8,10,105]^h \\
& [-132,-40,-33,10]^h=[-15,12,88,110]^h \\
& [-132,10,33,40]^h=[-88,12,15,110]^h 
\end{align}
\begin{mirrortype}
\end{mirrortype}
$\bullet$ \ref{h012n2} \quad $(h=-2,0,1,2)$
\begin{relatedtype}
\end{relatedtype}
$\bullet$ \ref{k02n2} \quad $(k=-2,0,2)$
\\
\subsubsection{\;\;(h$\;=\;$-2, -1, 1, 2)}
\label{h12n12}
\begin{notation}
\end{notation}
$\bullet$ Type: $(h=-2,-1,1,2)$ \qquad $\bullet$ Abbreviation: $[$h12n12$]$
\begin{idealsolu}
\end{idealsolu}
$\bullet$ Ajai Choudhry obtained a parametric solution in 2011 \cite{Choudhry11}, with additional condition $a_1 b_1=a_2 b_2=a_3 b_3=a_4 b_4$. Numerical solution:
\begin{align}
& [-55,-22,110,220]^h = [-46,-23,92,230]^h
\end{align}
\indent
$\bullet$ Chen Shuwen found 15 numerical solutions by full computer search in 2022 \cite{Chenhminus23}, all these solutions satisfy $a_1 b_1=a_2 b_2=a_3 b_3=a_4 b_4$. Numerical solutions:
\begin{align}
\label{h12n12s55}
& [-15,-9,45,45]^h = [-11,-11,33,55]^h\\
& [-20,-5,40,60]^h = [-9,-6,18,72]^h\\
& [-54,-18,108,189]^h = [-35,-20,70,210]^h\\
& [-66,-33,176,198]^h = [-45,-40,120,240]^h\\
& [-88,-77,308,308]^h = [-82,-82,287,328]^h\\
& [-272,-85,680,680]^h = [-118,-118,295,944]^h
\end{align}
\begin{mirrortype}
\end{mirrortype}
$\bullet$ None.
\\
\subsubsection{\;\;(h$\;=\;$-2, -1, 1, 3)}
\label{h13n12}
\begin{notation}
\end{notation}
$\bullet$ Type: $(h=-2,-1,1,3)$ \qquad $\bullet$ Abbreviation: $[$h13n12$]$
\begin{idealsolu}
\end{idealsolu}
$\bullet$ Ajai Choudhry found a parametric method for $(h=-2,-1,1,3)$ of size 5 in 2011 \cite{Choudhry11}, based on solutions of $(r=1,3)$. \vspace{+1ex}\\
\indent
$\bullet$ Chen Shuwen improved Ajai Choudhry's method and obtained the first ideal solution for $(h=-2,-1,1,3)$ in 2023 \cite{Chenhminus23}, based on \eqref{r13s221} and \eqref{r13s342} together.
\begin{align}
\label{h13n12s36594}
& [-34775, 2247, 21828, 36594 ]^h = [ -15246, 2299, 12100, 26741]^h
\end{align}
\begin{mirrortype}
\end{mirrortype}
$\bullet$ \ref{h12n13} \quad $(h=-3,-1,1,2)$
\begin{relatedtype}
\end{relatedtype}
$\bullet$ \ref{r13} \quad\: $(r=1,3)$
\\
\subsubsection{\;\;(h$\;=\;$-3, -1, 0, 1)}
\label{h01n13}
\begin{notation}
\end{notation}
$\bullet$ Type: $(h=-3,-1,0,1)$ \qquad $\bullet$ Abbreviation: $[$h01n13$]$
\begin{idealsolu}
\end{idealsolu}
$\bullet$ First known solution, based on \eqref{h013n1s30}, by Chen Shuwen in 2017:
\begin{align}
& [ -15, 28, 42, 140 ]^h = [ -14, 20, 84, 105]^h
\end{align}
\begin{mirrortype}
\end{mirrortype}
$\bullet$ \ref{h013n1} \quad $(h=-1,0,1,3)$
\\
\subsubsection{\;\;(h$\;=\;$-3, -1, 1, 2)}
\label{h12n13}
\begin{notation}
\end{notation}
$\bullet$ Type: $(h=-3,-1,1,2)$ \qquad $\bullet$ Abbreviation: $[$h12n13$]$
\begin{idealsolu}
\end{idealsolu}
$\bullet$ First known solution, based on \eqref{h13n12s36594}, by Chen Shuwen in 2023:
\begin{align}
& [-22464650, 12807900, 28305459, 148976100]^h \nonumber \\
&\quad =[-9848916, 9359350, 15690675, 152423700]^h
\end{align}
\begin{mirrortype}
\end{mirrortype}
$\bullet$ \ref{h13n12} \quad $(h=-2,-1,1,3)$
\\
\subsubsection{\;\;(h$\;=\;$-3, -2, -1, 1)}
\label{h1n123}
\begin{notation}
\end{notation}
$\bullet$ Type: $(h=-3,-2,-1,1)$ \qquad $\bullet$ Abbreviation: $[$h1n123$]$
\begin{idealsolu}
\end{idealsolu}
$\bullet$ Smallest solution, based on \eqref{h123n1s30}:
\begin{align}
& [-140, 15, 20, 105]^h=[-84, 14, 28, 42]^h
\end{align}
\begin{idealchain}
\end{idealchain}
$\bullet$ There are ideal solution chains for arbitrary length. Numerical examples, based on \eqref{h123n1s78}:
\begin{align}
& [-257400, 25025, 32175, 200200]^h=[-180180, 24024, 36036, 120120]^h \nonumber \\
&\quad =[-150150, 23400, 40950, 85800]^h =[-138600, 23100, 46200, 69300]^h
\end{align}
\begin{mirrortype}
\end{mirrortype}
$\bullet$ \ref{h123n1} \quad\: $(h=-1,1,2,3)$
\\
\subsubsection{\;\;(h$\;=\;$-4, -2, 0, 1)}
\label{h01n24}
\begin{notation}
\end{notation}
$\bullet$ Type: $(h=-4,-2,0,1)$ \qquad $\bullet$ Abbreviation: $[$h01n24$]$
\begin{idealsolu}
\end{idealsolu}
$\bullet$ First known solution, based on \eqref{h024n1s315}, by Chen Shuwen in 2019:
\begin{align}
& [ -1155, -2002, 4680, 6552 ]^h = [ -1144, -2520, 2730, 9009 ]^h
\end{align}
\indent
$\bullet$ Smallest known solution, based on \eqref{h024n1s99}, by Chen Shuwen in 2023:
\begin{align}
& [-4290, 990, 1001, 4095]^h = [-2574, -1155, -910, 6435]^h
\end{align}
\begin{mirrortype}
\end{mirrortype}
$\bullet$ \ref{h024n1} \quad $(h=-1,0,2,4)$
\\
\subsubsection{\;\;(h$\;=\;$-5, -1, 0, 1)}
\label{h01n15}
\begin{notation}
\end{notation}
$\bullet$ Type: $(h=-5,-1,0,1)$ \qquad $\bullet$ Abbreviation: $[$h01n15$]$
\begin{idealsolu}
\end{idealsolu}
$\bullet$ First known solution, based on \eqref{h015n1s702}, by Chen Shuwen in 2017:
\begin{align}
& [ -2210, 2805, 6885, 23166 ]^h = [ -2295, 3366, 5265, 24310]^h
\end{align}
\begin{mirrortype}
\end{mirrortype}
$\bullet$ \ref{h015n1} \quad $(h=-1,0,1,5)$
\\
\subsubsection{\;\;(h$\;=\;$-1, 0, 1, 2, 4)}
\label{h0124n1}
\begin{notation}
\end{notation}
$\bullet$ Type: $(h=-1,0,1,2,4)$ \qquad $\bullet$ Abbreviation: $[$h0124n1$]$
\begin{idealsolu}
\end{idealsolu}
$\bullet$ First known solutions, based on computer search, by Chen Shuwen in 2019 \cite{Chenhminus23}:
\begin{align}
\label{h0124n1s55}
& [  -66,-10,5,35,36 ]^h = [-60,-14,9,10,55]^h\\
\label{h0124n1s60}
& [  -70,-6,3,28,45 ]^h = [-63,-10,6,7,60]^h\\
\label{h0124n1s66}
& [  -77,-18,6,33,56 ]^h = [-72,-22,7,21,66]^h\\
\label{h0124n1s70}
& [  -63,-14,2,15,60 ]^h = [-42,-36,3,5,70]^h\\
& [  -84,-9,3,40,50 ]^h = [-75,-15,4,14,72]^h\\
& [  -91,-4,2,45,48 ]^h = [-80,-9,4,7,78]^h\\
\label{h0124n1s84}
& [  -72,-36,7,21,80 ]^h = [-63,-45,8,16,84]^h\\
& [  -120,-9,7,52,70 ]^h = [-117,-10,8,35,84]^h\\
& [  -99,-35,14,42,78 ]^h = [-91,-42,21,22,90]^h\\
& [  -117,-54,20,60,91 ]^h = [-105,-65,26,36,108]^h
\end{align}
\indent
$\bullet$ All numerical solutions given above satisfy $-a_1-a_2=a_3+a_4+a_5=-b_1-b_2=b_3+b_4+b_5$.
\begin{mirrortype}
\end{mirrortype}
$\bullet$ \ref{h01n124} \quad $(h=-4,-2,-1,0,1)$
\\
\subsubsection{\;\;(h$\;=\;$-4, -2, -1, 0, 1)}
\label{h01n124}
\begin{notation}
\end{notation}
$\bullet$ Type: $(h=-4,-2,-1,0,1)$ \qquad $\bullet$ Abbreviation: $[$h01n124$]$
\begin{idealsolu}
\end{idealsolu}
$\bullet$ Smallest known solutions, based on \eqref{h0124n1s60}, \eqref{h0124n1s70}, \eqref{h0124n1s84} and \eqref{h0124n1s66}, by Chen Shuwen in 2019:
\begin{align}
& [-210,-18,28,45,420]^h = [-126,-20,21,180,210]^h\\
& [-90,-20,21,84,630]^h = [-35,-30,18,252,420]^h\\
& [-140,-70,63,240,720]^h = [-112,-80,60,315,630]^h\\
& [-308,-72,99,168,924]^h = [-252,-77,84,264,792]^h
\end{align}
\begin{mirrortype}
\end{mirrortype}
$\bullet$ \ref{h0124n1} \quad $(h=-1,0,1,2,4)$
\\
\subsubsection{\;\;(h$\;=\;$-1, 0, 1, 2, 3, 5)}\label{Typeh01235n1}
\label{h01235n1}
\begin{notation}
\end{notation}
$\bullet$ Type: $(h=-1,0,1,2,3,5)$ \qquad $\bullet$ Abbreviation: $[$h01235n1$]$
\begin{idealsolu}
\end{idealsolu}
$\bullet$ First known solutions, based on \eqref{fomulah01235n1zero}, by Chen Shuwen in 2017 \cite{Chenhminus23}:
\begin{align}
\label{h01235n1s196}
& [ -156, -130, 13, 35, 42, 196 ]^h = [ -147, -140, 14, 26, 52, 195 ]^h\\
\label{h01235n1s104}
& [ -88, -65, 6, 14, 28, 105 ]^h = [-84, -70, 7, 10, 33, 104 ]^h\\
\label{h01235n1s88}
& [-104, -33, -10, 14, 28, 105 ]^h = [-84, -70, -6, 7, 65, 88 ]^h
\end{align}
\indent
$\bullet$ Parametric solution \eqref{parah01235n1a} of $(h=-1,0,1,2,3,5)$ was obtained by Chen Shuwen in 2017, based on Ajai Choudhry's parametric solution \eqref{parah12345n1} for $(h=-1,1,2,3,4,5)$.
\begin{mirrortype}
\end{mirrortype}
$\bullet$ \ref{h01n1235} \quad $(h=-5,-3,-2,-1,0,1)$
\\
\subsubsection{\;\;(h$\;=\;$-1, 1, 2, 3, 4, 5)}\label{Typeh12345n1}
\label{h12345n1}
\begin{notation}
\end{notation}
$\bullet$ Type: $(h=-1,1,2,3,4,5)$ \qquad $\bullet$ Abbreviation: $[$h12345n1$]$
\begin{idealsolu}
\end{idealsolu}
$\bullet$ Ajai Choudhry obtained a parametric solution \eqref{parah12345n1} in 2011 \cite{Choudhry11}. All solutions by this method have only one negative integer on each side. Numerical examples:
\begin{align}
\label{h12345n1s105}
& [-7, 14, 28, 70, 84, 105 ]^h = [-6, 10, 33, 65, 88, 104]^h\\
\label{h12345n1s364}
& [-26, 65, 78, 260, 273, 364]^h = [-7, 8, 154, 184, 330, 345]^h
\end{align}
\indent
$\bullet$ By utilizing \eqref{fomulah12345n1zero}, Chen Shuwen found a new parametric method for all solutions in 2023 \cite{Chenhminus23}. This method is based on the ideal symmetric solutions of $(k=1,2,3,4,5)$. New numerical examples:
\begin{align}
& [-45, -22 , 26, 72, 120 , 143]^h = [-42, -28, 35, 63 , 126, 140]^h\\
& [-78, -52 , 65, 117 , 234, 260]^h = [-70, -63 , 84, 98 , 245 , 252]^h
\end{align}
\begin{idealchain}
\end{idealchain}
$\bullet$ Ajai Choudhry's parametric solution can provide multigrade chains
of arbitrary length. Numerical examples
:
\begin{align}
\label{h12345n1s196}
& [-14, 35, 42, 140, 147, 196]^h = [-13, 26, 52, 130, 156, 195]^h \nonumber \\
& \quad = [-8, 11, 72, 110, 171, 190]^h = [-5, 6, 80, 102, 176, 187]^h \\
\label{h12345n1s286}
& [-20, 45, 68, 198, 221, 286]^h = [-19, 38, 76, 190, 228, 285]^h \nonumber \\
& \quad = [-14, 21, 98, 168, 245, 280]^h = [-10, 13, 110, 156, 253, 276]^h\\
&[-266,665,798,2660,2793,3724]^h=[-247,494,988,2470,2964,3705]^h\nonumber\\
& =[-260,585,884,2574,2873,3718]^h=[-255,544,930,2528,2914,3713]^h\nonumber\\
& =[-232,424,1073,2385,3034,3690]^h=[-231,420,1078,2380,3038,3689]^h\nonumber\\
& =[-196,308,1225,2233,3150,3654]^h=[-182,273,1274,2184,3185,3640]^h\nonumber\\
& =[-152,209,1368,2090,3249,3610]^h=[-130,169,1430,2028,3289,3588]^h\nonumber\\
& =[-111,138,1480,1978,3320,3569]^h=[-95,114,1520,1938,3344,3553]^h\nonumber\\
& =[-46,50,1633,1825,3408,3504]^h
\end{align}
\indent
$\bullet$ New ideal solution chains by Chen Shuwen's method:
\begin{align}
&  [-90, -16 , 17, 165, 198, 272]^h = [-88, -27 , 30, 152 , 209, 270]^h \nonumber \\
& \quad =[-78, -52, 65, 117, 234, 260]^h =[-70, -63, 84, 98, 245, 252]^h \\
& [-130, -33, 36, 230, 299, 396]^h = [-126, -49 , 56, 210 , 315, 392]^h \nonumber \\
& \quad = [-114, -76 , 95, 171, 342, 380]^h = [-108, -85 , 110, 156 , 351, 374]^h 
\end{align}
\begin{mirrortype}
\end{mirrortype}
$\bullet$ \ref{h1n12345} \quad $(h=-5,-4,-3,-2,-1,1)$
\begin{relatedtype}
\end{relatedtype}
$\bullet$ \ref{k12345} \quad $(k=1,2,3,4,5)$
\\
\subsubsection{\;\;(h$\;=\;$-5, -3, -2, -1, 0, 1)}
\label{h01n1235}
\begin{notation}
\end{notation}
$\bullet$ Type: $(h=-5,-3,-2,-1,0,1)$ \qquad $\bullet$ Abbreviation: $[$h01n1235$]$
\begin{idealsolu}
\end{idealsolu}
$\bullet$ First known solution, based on \eqref{h01235n1s196}, by Chen Shuwen in 2017:
\begin{align}
&[ -294, -245, 195, 910, 1092, 2940 ]^h \nonumber \\
& \quad = [ -273, -260, 196, 735, 1470, 2730 ]^h
\end{align}
\begin{mirrortype}
\end{mirrortype}
$\bullet$ \ref{h01235n1} \quad $(h=-1,0,1,2,3,5)$
\\
\subsubsection{\;\;(h$\;=\;$-5, -4, -3, -2, -1, 1)}
\label{h1n12345}
\begin{notation}
\end{notation}
$\bullet$ Type: $(h=-5,-4,-3,-2,-1,1)$ \qquad $\bullet$ Abbreviation: $[$h1n12345$]$
\begin{idealsolu}
\end{idealsolu}
$\bullet$ Smallest known solution, based on \eqref{h12345n1s196}:
\begin{align}
&[-2940, 196, 245, 294, 735, 1470]^h \nonumber \\
& \quad = [-2730, 195, 260, 273, 910, 1092]^h
\end{align}
\begin{idealchain}
\end{idealchain}
$\bullet$ Numerical solution chains, based on \eqref{h12345n1s196}:
\begin{align}
& [-325909584, 8714160, 9258795, 15975960, 20369349, 271591320]^h \nonumber \\
&\quad =[-203693490, 8576568, 9529520, 14814072, 22632610, 148140720]^h \nonumber \\
&\quad =[-125349840, 8356656, 10445820, 12534984, 31337460, 62674920]^h \nonumber \\
&\quad =[-116396280, 8314020, 11085360, 11639628, 38798760, 46558512]^h
\end{align}
\begin{mirrortype}
\end{mirrortype}
$\bullet$ \ref{h12345n1} \quad $(h=-1,1,2,3,4,5)$
\clearpage
\section{Ideal positive integer solution of FPTE}
So far, ideal positive integer solutions of FPTE have been found for 54 types of \( (r = r_1, r_2, \dots, r_n) \). These 54 types can be grouped into five categories with counts of 22, 22, 5, 3, and 2, respectively.
\subsection{FPTE with all r>0}
Ideal positive integer solutions of FPTE have been discovered for 22 types with all \( r > 0 \), including 10 known types of PTE. 
\subsubsection{\;\;(r$\;=\;$1)}
\label{r1}
\begin{notation}
\end{notation}
$\bullet$ Type: $(r=1)$ \qquad $\bullet$ Abbreviation: $[$r1$]$
\begin{idealsolu}
\end{idealsolu}
$\bullet$ Smallest solution, equivalent to \eqref{k1s2}:
\begin{align}
\label{r1s2}
& [2]^r = [1,1]^r 
\end{align}
\indent
$\bullet$ Prime solutions:
\begin{align}
& [5]^r = [2,3]^r \\
& [7]^r = [2,5]^r 
\end{align}
\begin{mirrortype}
\end{mirrortype}
$\bullet$ \ref{rn1} \quad $(r=-1)$
\begin{relatedtype}
\end{relatedtype}
$\bullet$ \ref{k1} \quad $(k=1)$ 
\\
\subsubsection{\;\;(r$\;=\;$2)}
\label{r2}
\begin{notation}
\end{notation}
$\bullet$ Type: $(r=2)$ \qquad $\bullet$ Abbreviation: $[$r2$]$
\begin{idealsolu}
\end{idealsolu}
$\bullet$ Smallest solution, equivalent to \eqref{k2s5}:
\begin{align}
\label{r2s5}
& [5]^r = [3,4]^r
\end{align}
\begin{mirrortype}
\end{mirrortype}
$\bullet$ \ref{rn2} \quad $(r=-2)$
\begin{relatedtype}
\end{relatedtype}
$\bullet$ \ref{k2} \quad $(k=2)$ 
\\
\subsubsection{\;\;(r$\;=\;$1, 2)}
\label{r12}
\begin{notation}
\end{notation}
$\bullet$ Type: $(r=1,2)$ \qquad $\bullet$ Abbreviation: $[$r12$]$
\begin{idealsolu}
\end{idealsolu}
$\bullet$ Smallest solutions, equivalent to \eqref{k12s4} and \eqref{k12s6}:
\begin{align}
\label{r12s4}
& [3,3]^r = [1,1,4]^r\\
\label{r12s6}
& [4,5]^r = [1,2,6]^r
\end{align}
\begin{mirrortype}
\end{mirrortype}
$\bullet$ \ref{rn12} \quad $(r=-2,-1)$
\begin{relatedtype}
\end{relatedtype}
$\bullet$ \ref{k12} \quad $(k=1,2)$ 
\\
\subsubsection{\;\;(r$\;=\;$1, 3)}
\label{r13}
\begin{notation}
\end{notation}
$\bullet$ Type: $(r=1,3)$ \qquad $\bullet$ Abbreviation: $[$r13$]$
\begin{idealsolu}
\end{idealsolu}
$\bullet$ Smallest solutions, equivalent to \eqref{k13s9} and \eqref{k13s10}:
\begin{align}
\label{r13s9}
& [7, 8 ]^r = [1, 5, 9]^r\\
\label{r13s10}
& [7, 9 ]^r = [2, 4, 10]^r
\end{align}
\indent
$\bullet$ Ajai Choudhry obtained complete solution in integers in 2000 \cite{Choudhry2000}. \vspace{+1ex}\\
\indent
$\bullet$ Ideal solutions (chains), based on computer search, by Chen Shuwen in 2023:
\begin{align}
& [217,225]^r = [8,187,247]^r = [17,169,256]^r = [ 39,136,267]^r \\
& [2244 , 2436]^r = [60 , 2028 , 2592]^r = [104, 1926, 2650]^r \nonumber \\
& \quad = [156 , 1824 , 2700]^r = [243, 1677, 2760]^r = [872, 910 , 2898]^r
\end{align}
\indent
$\bullet$ Chen Shuwen obtained all solutions for $(r=1,3)$ within the range up to 20000 through computer search in 2023, and found that the following two solutions together satisfy an additional condition, which led to the first ideal solution \eqref{h13n12s36594} for $(h=-2,-1,1,3)$:
\begin{align}
\label{r13s221}
& [126,214]^r = [19,100,221]^r \\
\label{r13s342}
& [242,325]^r = [21,204,342]^r
\end{align}
\begin{mirrortype}
\end{mirrortype}
$\bullet$ \ref{rn13} \quad\: $(r=-3,-1)$
\begin{relatedtype}
\end{relatedtype}
$\bullet$ \ref{k13} \quad\: $(k=1,3)$ \vspace{1ex} \\
\indent
$\bullet$ \ref{k1234} \quad $(k=1,2,3,4)$ \vspace{1ex} \\
\indent
$\bullet$ \ref{h13n12} \quad $(h=-2,-1,1,3)$
\\
\subsubsection{\;\;(r$\;=\;$1, 4)}
\label{r14}
\begin{notation}
\end{notation}
$\bullet$ Type: $(r=1,4)$ \qquad $\bullet$ Abbreviation: $[$r14$]$
\begin{idealsolu}
\end{idealsolu}
$\bullet$ First known solutions, smallest solutions, based on computer search, by Chen Shuwen in 2023:
\begin{align}
\label{r14s573}
&[349,557]^r = [95,238,573]^r\\
&[698,1114]^r = [190,476,1146]^r\\
&[989,1386]^r = [15,966,1394]^r\\
&[1047,1671]^r = [285,714,1719]^r
\end{align}
\begin{mirrortype}
\end{mirrortype}
$\bullet$ \ref{rn14} \quad $(r=-4,-1)$
\begin{relatedtype}
\end{relatedtype}
$\bullet$ \ref{k14} \quad $(k=1,4)$ 
\\
\subsubsection{\;\;(r$\;=\;$2, 3)}
\label{r23}
\begin{notation}
\end{notation}
$\bullet$ Type: $(r=2,3)$ \qquad $\bullet$ Abbreviation: $[$r23$]$
\begin{idealsolu}
\end{idealsolu}
$\bullet$ A.Gloden gave a two-parametric solution in 1949 \cite{Gloden49}, of which numerical solutions always include negative integers.\\[1mm]
\indent
$\bullet$ First known solution, smallest solution, equivalent to \eqref{k23s64}:
\begin{align}
\label{r23s64}
&[37,62]^r =[21,26,64]^r
\end{align}
\indent 
$\bullet$ Ideal solutions based on computer search, by Chen Shuwen in 2023:
\begin{align}
&[189,431]^r =[89,156,435]^r\\
&[377,630]^r =[54,368,633]^r\\
&[686,751]^r =[218,543,832]^r\\
&[721,906]^r =[145,666,936]^r
\end{align}
\indent 
$\bullet$ Ideal solutions based on parametric method, by Chen Shuwen in 2023:
\begin{align}
&[1743, 2078]^r =[368, 1578, 2175]^r\\
&[6930, 15787]^r =[2850, 5995, 15912]^r\\
&[20601, 41011]^r =[3571, 20016, 41145]^r
\end{align}
\begin{mirrortype}
\end{mirrortype}
$\bullet$ \ref{rn23} \quad $(r=-3,-2)$
\begin{relatedtype}
\end{relatedtype}
$\bullet$ \ref{k23} \quad $(k=2,3)$ 
\\
\subsubsection{\;\;(r$\;=\;$2, 4)}
\label{r24}
\begin{notation}
\end{notation}
$\bullet$ Type: $(r=2,4)$ \qquad $\bullet$ Abbreviation: $[$r24$]$
\begin{idealsolu}
\end{idealsolu}
$\bullet$ Smallest solutions, the first one is equivalent to \eqref{k24s8}:
\begin{align}
\label{r24s8}
&[7,7]^r =[3,5,8]^r\\
\label{r24s15}
&[13,13]^r =[7,8,15]^r
\end{align}
\indent
$\bullet$ Alfred Moessner and George Xeroudakes found a parametric method in 1976 \cite{Moessner76}. Numerical examples:
\begin{align}
\label{r24s266}
&[218,241]^r =[120,143,266]^r\\
&[2620682, 5927569]^r =[865007, 2441880, 5940794]^r
\end{align}
\begin{mirrortype}
\end{mirrortype}
$\bullet$ \ref{rn24}\quad\: $(r=-4,-2)$ 
\begin{relatedtype}
\end{relatedtype}
$\bullet$ \ref{k24}\quad $(k=2,4)$ \vspace{1ex}\\
\indent
$\bullet$ \ref{s124}\quad\: $(s=1,2,4)$ 
\\
\subsubsection{\;\;(r$\;=\;$1, 2, 3)}
\label{r123}
\begin{notation}
\end{notation}
$\bullet$ Type: $(r=1,2,3)$ \qquad $\bullet$ Abbreviation: $[$r123$]$
\begin{idealsolu}
\end{idealsolu}
$\bullet$ Ideal solutions, equivalent to \eqref{k123s11} and \eqref{k123s22}:
\begin{align}
\label{r123s11}
& [ 4, 7, 11 ]^r = [ 1, 2, 9, 10 ]^r\\
\label{r123s22}
& [ 9, 11, 22 ]^r = [ 2, 4, 15, 21 ]^r
\end{align}
\begin{mirrortype}
\end{mirrortype}
$\bullet$ \ref{rn123} \quad\: $(r=-3,-2,-1)$
\begin{relatedtype}
\end{relatedtype}
$\bullet$ \ref{k123} \quad $(k=1,2,3)$ 
\\
\subsubsection{\;\;(r$\;=\;$1, 2, 4)}
\label{r124}
\begin{notation}
\end{notation}
$\bullet$ Type: $(r=1,2,4)$ \qquad $\bullet$ Abbreviation: $[$r124$]$
\begin{idealsolu}
\end{idealsolu}
$\bullet$ Smallest solutions, based on computer search, by Chen Shuwen in 2023:
\begin{align}
\label{r124s19}
& [7,14,19 ]^k = [1,5,16,18]^k\\
& [14,21,31]^k = [1,11,24,30]^k\\
\label{r124s35}
& [14,19,35]^k = [5,5,24,34]^k\\
& [17,23,37]^k = [5,7,32,33]^k\\
& [17,25,38]^k = [3,10,32,35]^k
\end{align}
\begin{mirrortype}
\end{mirrortype}
$\bullet$ \ref{rn124} \quad\: $(r=-4,-2,-1)$
\begin{relatedtype}
\end{relatedtype}
$\bullet$ \ref{k124} \quad $(k=1,2,4)$ 
\\
\subsubsection{\;\;(r$\;=\;$1, 2, 5)}
\label{r125}
\begin{notation}
\end{notation}
$\bullet$ Type: $(r=1,2,5)$ \qquad $\bullet$ Abbreviation: $[$r125$]$
\begin{idealsolu}
\end{idealsolu}
$\bullet$ First known solution, smallest solution, based on computer search, by Chen Shuwen in 2023:
\begin{align}
\label{r125s770}
&[223, 642, 770 ]^r= [42, 160, 698, 735]^r
\end{align}
\indent
$\bullet$ New solutions, based on a parametric solution with additional
condition of $a_1+a_2=b_1+b_2$ and $a_3=b_3+b_4$, by Chen Shuwen in 2023:
\begin{align}
\label{r125s985}
&[376, 693, 985]^r= [126, 175, 810, 943]^r\\
&[637, 1584, 2065]^r= [234, 310, 1755, 1987]^r\\
&[1765, 2079, 3514]^r= [495, 720, 2794, 3349]^r
\end{align}
\begin{mirrortype}
\end{mirrortype}
$\bullet$ \ref{rn125} \quad $(r=-5,-2,-1)$
\begin{relatedtype}
\end{relatedtype}
$\bullet$ \ref{k125} \quad $(k=1,2,5)$ 
\\
\subsubsection{\;\;(r$\;=\;$1, 3, 4)}
\label{r134}
\begin{notation}
\end{notation}
$\bullet$ Type: $(r=1,3,4)$ \qquad $\bullet$ Abbreviation: $[$r134$]$
\begin{idealsolu}
\end{idealsolu}
$\bullet$ Chen Shuwen found a parametric solution in 2023, with additional
condition of $a_1+a_2=b_1+b_2$ and $a_3=b_3+b_4$. Numerical solutions:
\begin{align}
\label{r134s1161}
&[409,838,1161]^r= [109,261,900,1138]^r \\
&[2176,2203,4077]^r= [394, 1206, 2871, 3985]^r
\end{align}
\indent
$\bullet$ By computer search in 2023, Chen Shuwen confirmed that \eqref{r134s1161} is the smallest solution.
\begin{mirrortype}
\end{mirrortype}
$\bullet$ \ref{rn134} \quad $(r=-4,-3,-1)$
\begin{relatedtype}
\end{relatedtype}
$\bullet$ \ref{k134} \quad $(k=1,3,4)$ 
\\
\subsubsection{\;\;(r$\;=\;$1, 3, 5)}
\label{r135}
\begin{notation}
\end{notation}
$\bullet$ Type: $(r=1,3,5)$ \qquad $\bullet$ Abbreviation: $[$r135$]$
\begin{idealsolu}
\end{idealsolu}
$\bullet$ Smallest solution, equivalent to \eqref{k135s51}:
\begin{align}
\label{r135s51}
& [ 24, 33, 51 ]^r= [ 7, 13, 38, 50]^r
\end{align}
\indent
$\bullet$ George Xeroudakes and Alfred Moessner obtained a two-parameter solution in 1958  \cite{Xeroudakes58}. Numerical examples:
\begin{align}
& [66, 75, 134]^r= [8, 47, 87, 133]^r\\
& [221, 231, 339]^r= [23, 152, 300, 316]^r\\
& [310, 323, 545]^r= [105, 123, 422, 528]^r\\
& [ 5893, 6001, 11907]^r= [ 121, 5200, 6586, 11894]^r
\end{align}
\begin{mirrortype}
\end{mirrortype}
$\bullet$ \ref{rn135} \quad $(r=-5,-3,-1)$
\begin{relatedtype}
\end{relatedtype}
$\bullet$ \ref{k135} \quad $(k=1,3,5)$ \vspace{1ex} \\
\indent
$\bullet$ \ref{s1235} \quad\: $(s=1,2,3,5)$
\\
\subsubsection{\;\;(r$\;=\;$1, 2, 3, 4)}
\label{r1234}
\begin{notation}
\end{notation}
$\bullet$ Type: $(r=1,2,3,4)$ \qquad $\bullet$ Abbreviation: $[$r1234$]$
\begin{idealsolu}
\end{idealsolu}
$\bullet$ Ideal solutions, equivalent to \eqref{k1234s18} and \eqref{k1234s20}:
\begin{align}
\label{r1234s18}
&[4, 8, 16, 17 ]^r = [ 1, 2, 10, 14, 18 ]^r\\
\label{r1234s20}
&[6, 8, 17, 19 ]^r = [ 1, 3, 12, 14, 20 ]^r
\end{align}
\begin{mirrortype}
\end{mirrortype}
$\bullet$ \ref{rn1234} \quad $(r=-4,-3,-2,-1)$
\begin{relatedtype}
\end{relatedtype}
$\bullet$ \ref{k1234} \quad $(k=1,2,3,4)$
\\
\subsubsection{\;\;(r$\;=\;$1, 2, 3, 5)}
\label{r1235}
\begin{notation}
\end{notation}
$\bullet$ Type: $(r=1,2,3,5)$ \qquad $\bullet$ Abbreviation: $[$r1235$]$
\begin{idealsolu}
\end{idealsolu}
$\bullet$ First known solutions, based on parametric method, by Chen Shuwen in 2023:
\begin{align}
\label{r1235s120}
&[32,51,106,115]^r = [7,16,66,95,120]^r\\
\label{r1235s125}
&[35,48,93,124]^r = [3,24,60,88,125]^r\\
\label{r1235s132}
&[23,82,120,125]^r = [5,15,98,100,132]^r\\
&[39,56,110,137]^r = [12,14,81,95,140]^r\\
&[15,95,115,146]^r = [3,11,104,106,147]^r\\
&[33,104,150,169]^r = [5,24,120,133,174]^r\\
&[55,75,145,206]^r = [17,19,111,126,208]^r\\
&[45,115,202,208]^r = [10,27,133,180,220]^r
\end{align}
\begin{mirrortype}
\end{mirrortype}
$\bullet$ \ref{rn1235} \quad $(r=-5,-3,-2,-1)$
\begin{relatedtype}
\end{relatedtype}
$\bullet$ \ref{k1235} \quad $(k=1,2,3,5)$ 
\\
\subsubsection{\;\;(r$\;=\;$1, 3, 5, 7)}
\label{r1357}
\begin{notation}
\end{notation}
$\bullet$ Type: $(r=1,3,5,7)$ \qquad $\bullet$ Abbreviation: $[$r1357$]$
\begin{idealsolu}
\end{idealsolu}
$\bullet$ Smallest solutions, equivalent to \eqref{k1357s99} and \eqref{k1357s174}:
\begin{align}
\label{r1357s99}
&[34, 58, 82, 98 ]^r = [ 13, 16, 69, 75, 99 ]^r\\
\label{r1357s174}
&[63, 119, 161, 169 ]^r = [ 8, 50, 132, 148, 174 ]^r
\end{align}
\indent
$\bullet$ So far, only the two solutions presented above are known.
\begin{mirrortype}
\end{mirrortype}
$\bullet$ \ref{rn1357}\quad $(r=-7,-5,-3,-1)$ 
\begin{relatedtype}
\end{relatedtype}
$\bullet$ \ref{k1357}\quad $(k=1,3,5,7)$ \vspace{1ex}\\
\indent
$\bullet$ \ref{s12357}\quad\: $(s=1,2,3,5,7)$ 
\\
\subsubsection{\;\;(r$\;=\;$1, 2, 3, 4, 5)}
\label{r12345}
\begin{notation}
\end{notation}
$\bullet$ Type: $(r=1,2,3,4,5)$ \qquad $\bullet$ Abbreviation: $[$r12345$]$
\begin{idealsolu}
\end{idealsolu}
$\bullet$ Smallest solution, equivalent to \eqref{k12345s16}:
\begin{align}
\label{r12345s16}
& [3, 5, 11, 13, 16 ]^r = [ 1, 1, 8, 8, 15, 15 ]^r 
\end{align}
\begin{mirrortype}
\end{mirrortype}
$\bullet$ \ref{rn12345} \quad $(r=-5,-4,-3,-2,-1)$
\begin{relatedtype}
\end{relatedtype}
$\bullet$ \ref{k12345} \quad $(k=1,2,3,4,5)$ 
\\
\subsubsection{\;\;(r$\;=\;$1, 2, 3, 4, 6)}
\label{r12346}
\begin{notation}
\end{notation}
$\bullet$ Type: $(r=1,2,3,4,6)$ \qquad $\bullet$ Abbreviation: $[$r12346$]$
\begin{idealsolu}
\end{idealsolu}
$\bullet$ First known solution, based on computer search, by Chen Shuwen in 2023:
\begin{align}
\label{r12346s136}
& [19, 44, 95, 102, 136 ] = [ 4, 11, 52, 84, 110, 135]^r 
\end{align}
\indent
$\bullet$ Solutions based on parametric method, by Chen Shuwen in 2023:
\begin{align}
\label{r12346s299}
& [43, 101, 146, 282, 299]^r = [3, 35, 117, 134, 286, 296]^r \\
& [111, 117, 233, 311, 389]^r = [3, 81, 161, 209, 320, 387]^r \\
& [61, 169, 278, 440, 482]^r = [22, 26, 209, 245, 460, 468]^r
\end{align}
\begin{mirrortype}
\end{mirrortype}
$\bullet$ \ref{rn12346} \quad $(r=-6,-4,-3,-2,-1)$
\begin{relatedtype}
\end{relatedtype}
$\bullet$ \ref{k12346} \quad $(k=1,2,3,4,6)$ 
\\
\subsubsection{\;\;(r$\;=\;$1, 2, 3, 4, 5, 6)}
\label{r123456}
\begin{notation}
\end{notation}
$\bullet$ Type: $(r=1,2,3,4,5,6)$ \qquad $\bullet$ Abbreviation: $[$r123456$]$
\begin{idealsolu}
\end{idealsolu}
$\bullet$ Smallest solution, equivalent to \eqref{k123456s84}:
\begin{align}
\label{r123456s84}
& [18,19,50,56,79,81]^r=[1,11,30,39, 68, 70 , 84 ] ^r
\end{align}
\begin{mirrortype}
\end{mirrortype}
$\bullet$ \ref{rn123456} \quad $(r=-6,-5,-4,-3,-2,-1)$
\begin{relatedtype}
\end{relatedtype}
$\bullet$ \ref{k123456} \quad $(k=1,2,3,4,5,6)$ 
\\
\subsubsection{\;\;(r$\;=\;$1, 2, 3, 4, 5, 6, 7)}
\label{r1234567}
\begin{notation}
\end{notation}
$\bullet$ Type: $(r=1,2,3,4,5,6,7)$ \qquad $\bullet$ Abbreviation: $[$r1234567$]$
\begin{idealsolu}
\end{idealsolu}
$\bullet$ Smallest solution, equivalent to \eqref{k1234567s50}:
\begin{align}
\label{r1234567s50}
& [4, 9, 23, 27, 41, 46, 50 ]^r = [ 1, 2, 11, 20, 30, 39 , 48, 49]^r
\end{align}
\begin{mirrortype}
\end{mirrortype}
$\bullet$ \ref{rn1234567} \quad $(r=-7,-6,-5,-4,-3,-2,-1)$
\begin{relatedtype}
\end{relatedtype}
$\bullet$ \ref{k1234567} \quad $(k=1,2,3,4,5,6,7)$ 
\\
\subsubsection{\;\;(r$\;=\;$1, 2, 3, 4, 5, 6, 7, 8)}
\label{r12345678}
\begin{notation}
\end{notation}
$\bullet$ Type: $(r=1,2,3,4,5,6,7,8)$ \qquad $\bullet$ Abbreviation: $[$r12345678$]$
\begin{idealsolu}
\end{idealsolu}
$\bullet$ Smallest solution, equivalent to \eqref{k12345678s198}:
\begin{align}
\label{r12345678s198}
&[24, 30, 83, 86, 133, 157, 181, 197]^r  \nonumber \\
&\quad = [ 1, 17, 41, 65, 112, 115, 168, 174, 198]^r
\end{align}
\begin{mirrortype}
\end{mirrortype}
$\bullet$ \ref{rn12345678} \quad $(r=-8,-7,-6,-5,-4,-3,-2,-1)$
\begin{relatedtype}
\end{relatedtype}
$\bullet$ \ref{k12345678} \quad $(k=1,2,3,4,5,6,7,8)$ 
\\
\subsubsection{\;\;(r$\;=\;$1, 2, 3, 4, 5, 6, 7, 8, 9)}
\label{r123456789}
\begin{notation}
\end{notation}
$\bullet$ Type: $(r=1,2,3,4,5,6,7,8,9)$ \qquad $\bullet$ Abbreviation: $[$r123456789$]$
\begin{idealsolu}
\end{idealsolu}
$\bullet$ Smallest solutions, equivalent to \eqref{k123456789s626} and \eqref{k123456789s1030}:
\begin{align}
\label{r123456789s626}
& [12, 125, 213, 214, 412, 413, 501, 614, 626  ]^r \nonumber \\
& \quad =  [ 5, 6, 133, 182, 242, 384, 444, 493, 620, 621 ]^r \\
\label{r123456789s1030}
& [63, 149, 326, 412, 618, 704, 881, 967, 1030]^r \nonumber \\
& \quad =  [7, 44, 184, 270, 497, 533, 760, 846, 986, 1023]^r
\end{align}
\begin{mirrortype}
\end{mirrortype}
$\bullet$ \ref{rn123456789} \quad $(r=-9,-8,-7,-6,-5,-4,-3,-2,-1)$
\begin{relatedtype}
\end{relatedtype}
$\bullet$ \ref{k123456789} \quad $(k=1,2,3,4,5,6,7,8,9)$ 
\\
\subsubsection{\;\;(r$\;=\;$1, 2, 3, 4, 5, 6, 7, 8, 9, 10, 11)}
\label{r1234567891011}
\begin{notation}
\end{notation}
$\bullet$ Type: $(r=1, 2, 3, 4, 5, 6, 7, 8, 9, 10, 11)$ \vspace{1ex}\\
\indent
$\bullet$ Abbreviation: $[$r1234567891011$]$
\begin{idealsolu}
\end{idealsolu}
$\bullet$ Smallest solution, equivalent to \eqref{k1234567891011s302}:
\begin{align}
\label{r1234567891011s302}
& [11, 24, 65, 90, 129, 173, 212, 237, 278, 291, 302 ]^r \nonumber \\
& \quad = [ 3, 5, 30, 57, 104, 116, 186, 198, 245, 272, 297, 299]^r
\end{align}
\begin{mirrortype}
\end{mirrortype}
$\bullet$ \ref{rn1234567891011}\quad $(r=-11,-10,-9,-8,-7,-6,-5,-4,-3,-2,-1)$
\begin{relatedtype}
\end{relatedtype}
$\bullet$ \ref{k1234567891011}\quad $(k=1,2,3,4,5,6,7,8,9,10,11)$
\\
\subsection{FPTE with all r<0}
So far, ideal positive integer solutions of FPTE have been found for 22 types with all \( r < 0 \).
\subsubsection{\;\;(r$\;=\;$-1)}
\label{rn1}
\begin{notation}
\end{notation}
$\bullet$ Type: $(r=-1)$ \qquad $\bullet$ Abbreviation: $[$rn1$]$
\begin{idealsolu}
\end{idealsolu}
$\bullet$ Smallest solution, based on \eqref{r1s2}:
\begin{align}
&[ 1 ]^r = [ 2,2 ]^r
\end{align}
\begin{mirrortype}
\end{mirrortype}
$\bullet$ \ref{r1}\quad $(r=1)$
\\
\subsubsection{\;\;(r$\;=\;$-2)}
\label{rn2}
\begin{notation}
\end{notation}
$\bullet$ Type: $(r=-2)$ \qquad $\bullet$ Abbreviation: $[$rn2$]$
\begin{idealsolu}
\end{idealsolu}
$\bullet$ Smallest solution, based on \eqref{r2s5}:
\begin{align}
&[ 12 ]^r = [15,20]^r
\end{align}
\begin{mirrortype}
\end{mirrortype}
$\bullet$ \ref{r2}\quad $(r=2)$
\\
\subsubsection{\;\;(r$\;=\;$-2, -1)}
\label{rn12}
\begin{notation}
\end{notation}
$\bullet$ Type: $(r=-2,-1)$ \qquad $\bullet$ Abbreviation: $[$rn12$]$
\begin{idealsolu}
\end{idealsolu}
$\bullet$ Smallest solution, based on \eqref{r12s4}:
\begin{align}
&[4,4]^r = [3,12,12]^r
\end{align}
\begin{mirrortype}
\end{mirrortype}
$\bullet$ \ref{r12}\quad $(r=1,2)$
\\
\subsubsection{\;\;(r$\;=\;$-3, -1)}
\label{rn13}
\begin{notation}
\end{notation}
$\bullet$ Type: $(r=-3,-1)$ \qquad $\bullet$ Abbreviation: $[$rn13$]$
\begin{idealsolu}
\end{idealsolu}
$\bullet$ Smallest solution, based on \eqref{r13s10}:
\begin{align}
&[140,180]^r = [126,315,630]^r
\end{align}
\begin{mirrortype}
\end{mirrortype}
$\bullet$ \ref{r13}\quad $(r=1,3)$
\\
\subsubsection{\;\;(r$\;=\;$-4, -1)}
\label{rn14}
\begin{notation}
\end{notation}
$\bullet$ Type: $(r=-4,-1)$ \qquad $\bullet$ Abbreviation: $[$rn14$]$
\begin{idealsolu}
\end{idealsolu}
$\bullet$ Smallest known solution, based on \eqref{r14s573}:
\begin{align}
\label{rn14s26510150982}
& [4521479970, 7216230210]^r \nonumber \\
& \quad = [4395225730, 10581782955, 26510150982]^r
\end{align}
\begin{mirrortype}
\end{mirrortype}
$\bullet$ \ref{r14}\quad $(r=1,4)$
\\
\subsubsection{\;\;(r$\;=\;$-3, -2)}
\label{rn23}
\begin{notation}
\end{notation}
$\bullet$ Type: $(r=-3,-2)$ \qquad $\bullet$ Abbreviation: $[$rn23$]$
\begin{idealsolu}
\end{idealsolu}
$\bullet$ Smallest known solution, based on \eqref{r23s64}:
\begin{align}
& [323232, 541632]^r= [313131, 770784, 954304]^r
\end{align}
\begin{mirrortype}
\end{mirrortype}
$\bullet$ \ref{r23}\quad $(r=2,3)$
\\
\subsubsection{\;\;(r$\;=\;$-4, -2)}
\label{rn24}
\begin{notation}
\end{notation}
$\bullet$ Type: $(r=-4,-2)$ \qquad $\bullet$ Abbreviation: $[$rn24$]$
\begin{idealsolu}
\end{idealsolu}
$\bullet$ Smallest known solution, based on \eqref{r24s8}:
\begin{align}
& [120, 120]^r= [105, 168, 280]^r
\end{align}
\begin{mirrortype}
\end{mirrortype}
$\bullet$ \ref{r24}\quad $(r=2,4)$
\\
\subsubsection{\;\;(r$\;=\;$-3, -2, -1)}
\label{rn123}
\begin{notation}
\end{notation}
$\bullet$ Type: $(r=-3,-2,-1)$ \qquad $\bullet$ Abbreviation: $[$rn123$]$
\begin{idealsolu}
\end{idealsolu}
$\bullet$ Smallest known solution, based on \eqref{r123s22}:
\begin{align}
& [630, 1260, 1540]^r= [660, 924, 3465, 6930]^r
\end{align}
\begin{mirrortype}
\end{mirrortype}
$\bullet$ \ref{r123}\quad $(r=1,2,3)$
\\
\subsubsection{\;\;(r$\;=\;$-4, -2, -1)}
\label{rn124}
\begin{notation}
\end{notation}
$\bullet$ Type: $(r=-4,-2,-1)$ \qquad $\bullet$ Abbreviation: $[$rn124$]$
\begin{idealsolu}
\end{idealsolu}
$\bullet$ Smallest known solution, based on  \eqref{r124s35}, by Chen Shuwen in 2023:
\begin{align}
& [7752, 14280, 19380]^r=[7980, 11305, 54264, 54264]^r
\end{align}
\indent
$\bullet$ Small solutions in the range of 100000, based on $(r=1,2,4)$:
\begin{align}
& [7280, 11856, 14820]^r = [7980, 8645, 27664, 82992]^r\\
& [5040, 6840, 13680]^r = [5320, 5985, 19152, 95760]^r
\end{align}
\begin{mirrortype}
\end{mirrortype}
$\bullet$ \ref{r124}\quad $(r=1,2,4)$
\\
\subsubsection{\;\;(r$\;=\;$-5, -2, -1)}
\label{rn125}
\begin{notation}
\end{notation}
$\bullet$ Type: $(r=-5,-2,-1)$ \qquad $\bullet$ Abbreviation: $[$rn125$]$
\begin{idealsolu}
\end{idealsolu}
$\bullet$ Smallest known solutions, based on \eqref{r125s770} and \eqref{r125s985}, by Chen Shuwen in 2023:
\begin{align}
& [2798036304, 3355900240, 9661380960]^r \nonumber \\
& \quad = [2931276128, 3086658960, 13465549713, 51297332240]^r\\
& [11057203080, 15716226600, 28966343175]^r \nonumber \\
& \quad = [11549676600, 13446104980, 62236257336, 86439246300]^r
\end{align}
\begin{mirrortype}
\end{mirrortype}
$\bullet$ \ref{r125}\quad $(r=1,2,5)$
\\
\subsubsection{\;\;(r$\;=\;$-4, -3, -1)}
\label{rn134}
\begin{notation}
\end{notation}
$\bullet$ Type: $(r=-4,-3,-1)$ \qquad $\bullet$ Abbreviation: $[$rn134$]$
\begin{idealsolu}
\end{idealsolu}
$\bullet$ Smallest known solution, based on \eqref{r134s1161}:
\begin{align}
& [30822942293900, 42703384252050, 87494953553100]^r \nonumber \\
& \quad = [31445901584550,39761595559131,\nonumber \\
& \qquad \qquad 137108950203900,328306752323100]^r
\end{align}
\begin{mirrortype}
\end{mirrortype}
$\bullet$ \ref{r134}\quad $(r=1,3,4)$
\\
\subsubsection{\;\;(r$\;=\;$-5, -3, -1)}
\label{rn135}
\begin{notation}
\end{notation}
$\bullet$ Type: $(r=-5,-3,-1)$ \qquad $\bullet$ Abbreviation: $[$rn135$]$
\begin{idealsolu}
\end{idealsolu}
$\bullet$ Smallest known solution, based on \eqref{r135s51}:
\begin{align}
& [3803800, 5878600, 8083075]^r \nonumber \\
& \quad = [3879876, 5105100, 14922600, 27713400]^r
\end{align}
\begin{mirrortype}
\end{mirrortype}
$\bullet$ \ref{r135}\quad $(r=1,3,5)$
\\
\subsubsection{\;\;(r$\;=\;$-4, -3, -2, -1)}
\label{rn1234}
\begin{notation}
\end{notation}
$\bullet$ Type: $(r=-4,-3,-2,-1)$ \qquad $\bullet$ Abbreviation: $[$rn1234$]$
\begin{idealsolu}
\end{idealsolu}
$\bullet$ Smallest known solution, based on \eqref{r1234s18}:
\begin{align}
& [5040, 5355, 10710, 21420]^r \nonumber \\
& \quad = [4760, 6120, 8568, 42840, 85680]^r
\end{align}
\begin{mirrortype}
\end{mirrortype}
$\bullet$ \ref{r1234}\quad $(r=1,2,3,4)$
\\
\subsubsection{\;\;(r$\;=\;$-5, -3, -2, -1)}
\label{rn1235}
\begin{notation}
\end{notation}
$\bullet$ Type: $(r=-5,-3,-2,-1)$ \qquad $\bullet$ Abbreviation: $[$rn1235$]$
\begin{idealsolu}
\end{idealsolu}
$\bullet$ Smallest known solutions, based on \eqref{r1235s125} and \eqref{r1235s132}:
\begin{align}
& [115500, 154000, 298375, 409200]^r \nonumber \\
& \quad = [114576, 162750, 238700, 596750, 4774000]^r \\
& [12198648, 12706925, 18595500, 66297000]^r \nonumber \\
& \quad = [11551750, 15248310, 15559500, 101655400, 304966200]^r
\end{align}
\begin{mirrortype}
\end{mirrortype}
$\bullet$ \ref{r1235}\quad $(r=1,2,3,5)$
\\
\subsubsection{\;\;(r$\;=\;$-7, -5, -3, -1)}
\label{rn1357}
\begin{notation}
\end{notation}
$\bullet$ Type: $(r=-7,-5,-3,-1)$ \qquad $\bullet$ Abbreviation: $[$rn1357$]$
\begin{idealsolu}
\end{idealsolu}
$\bullet$ Smallest known solution, based on \eqref{r1357s99}:
\begin{align}
& [119665002600, 143014271400, 202192590600, 344916772200]^r \nonumber\\
& \quad = [118456265200, 156362270064, 169958989200, 732948140925,\nonumber\\
& \qquad \qquad  902090019600]^r
\end{align}
\begin{mirrortype}
\end{mirrortype}
$\bullet$ \ref{r1357}\quad $(r=1,3,5,7)$
\\
\subsubsection{\;\;(r$\;=\;$-5, -4, -3, -2, -1)}
\label{rn12345}
\begin{notation}
\end{notation}
$\bullet$ Type: $(r=-5,-4,-3,-2,-1)$ \qquad $\bullet$ Abbreviation: $[$rn12345$]$
\begin{idealsolu}
\end{idealsolu}
$\bullet$ Smallest known solution, based on \eqref{r12345s16}:
\begin{align}
& [2145, 2640, 3120, 6864, 11440]^r \\
& \quad = [2288, 2288, 4290, 4290, 34320, 34320]^r
\end{align}
\begin{mirrortype}
\end{mirrortype}
$\bullet$ \ref{r12345}\quad $(r=1,2,3,4,5)$
\\
\subsubsection{\;\;(r$\;=\;$-6, -4, -3, -2, -1)}
\label{rn12346}
\begin{notation}
\end{notation}
$\bullet$ Type: $(r=-6,-4,-3,-2,-1)$ \qquad $\bullet$ Abbreviation: $[$rn12346$]$
\begin{idealsolu}
\end{idealsolu}
$\bullet$ First known solution, based on \eqref{r12346s136}:
\begin{align}
& [2567565, 3423420, 3675672, 7936110, 18378360]^r \nonumber \\
& \quad = [2586584, 3174444, 4157010, 6715170, 31744440, 87297210]^r
\end{align}
\begin{mirrortype}
\end{mirrortype}
$\bullet$ \ref{r12346}\quad $(r=1,2,3,4,6)$
\\
\subsubsection{\;\;(r$\;=\;$-6, -5, -4, -3, -2, -1)}
\label{rn123456}
\begin{notation}
\end{notation}
$\bullet$ Type: $(r=-6,-5,-4,-3,-2,-1)$ \qquad $\bullet$ Abbreviation: $[$rn123456$]$
\begin{idealsolu}
\end{idealsolu}
$\bullet$ Smallest known solution, based on \eqref{r123456s84}:
\begin{align}
& [5108503400, 5237832600, 7389085275, 8275775508, \nonumber \\
& \qquad \qquad 21778356600,22988265300]^r \nonumber \\
& \quad = [4926056850, 5911268220, 6085129050, 10609968600, \nonumber \\
& \qquad \qquad 13792959180, 37617161400, 413788775400]^r
\end{align}
\begin{mirrortype}
\end{mirrortype}
$\bullet$ \ref{r123456}\quad $(r=1,2,3,4,5,6)$
\\
\subsubsection{\;\;(r$\;=\;$-7, -6, -5, -4, -3, -2, -1)}
\label{rn1234567}
\begin{notation}
\end{notation}
$\bullet$ Type: $(r=-7,-6,-5,-4,-3,-2,-1)$ \qquad $\bullet$ Abbreviation: $[$rn1234567$]$
\begin{idealsolu}
\end{idealsolu}
$\bullet$ Smallest known solution, based on \eqref{r1234567s50}:
\begin{align}
& [1427241816, 1551349800, 1740538800, 2643040400, \nonumber \\
& \qquad \qquad 3102699600, 7929121200, 17840522700]^r \nonumber \\
& \quad = [1456369200, 1486710225, 1829797200, 2378736360, \nonumber \\
& \qquad \qquad 3568104540, 6487462800, 35681045400, 71362090800]^r
\end{align}
\begin{mirrortype}
\end{mirrortype}
$\bullet$ \ref{r1234567}\quad $(r=1,2,3,4,5,6,7)$
\\
\subsubsection{\;\;(r$\;=\;$-8, -7, -6, -5, -4, -3, -2, -1)}
\label{rn12345678}
\begin{notation}
\end{notation}
$\bullet$ Type: $(r=-8,-7,-6,-5,-4,-3,-2,-1)$ \vspace{1ex}\\
\indent
$\bullet$ Abbreviation: $[$rn12345678$]$
\begin{idealsolu}
\end{idealsolu}
$\bullet$ Smallest known solution, based on \eqref{r12345678s198}:
\begin{align}
& [645659524755046161360, 702734399871514330320,\nonumber \\
& \qquad\qquad 810158766730854100560, 956352829900331532240,\nonumber \\
& \qquad\qquad 1479010771822605741720, 1532468992490892696240,\nonumber \\
& \qquad\qquad 4239830879224803126264, 5299788599031003907830]^r \nonumber \\
& \quad = [642398618064364110040, 731005324004276401080, \nonumber \\
& \qquad\qquad 757112657004429129690, 1106042838058644293808, \nonumber \\
& \qquad\qquad 1135668985506643694535,1956845021180678365968, \nonumber \\
& \qquad\qquad 3102315277481563263120, 7482054492749652575760,\nonumber \\
& \qquad\qquad 127194926376744093787920]^r 
\end{align}
\begin{mirrortype}
\end{mirrortype}
$\bullet$ \ref{r12345678}\quad $(r=1,2,3,4,5,6,7,8)$
\\
\subsubsection{\;\;(r$\;=\;$-9, -8, -7, -6, -5, -4, -3, -2, -1)}
\label{rn123456789}
\begin{notation}
\end{notation}
$\bullet$ Type: $(r=-9,-8,-7,-6,-5,-4,-3,-2,-1)$ \vspace{1ex}\\
\indent
$\bullet$ Abbreviation: $[$rn123456789$]$
\begin{idealsolu}
\end{idealsolu}
$\bullet$ Smallest known solution, based on \eqref{r123456789s1030}:
\begin{align}
& [16351371612556032801640443552, 17416662627644998744249903680,\nonumber \\
& \quad 19116813576541105318603469760, 23923171535415786627400080765,\nonumber \\
& \quad 27252286020926721336067405920, 40878429031390082004101108880,\nonumber \\
& \quad 51662309082615686459170726560, 113032971549883985138856757440,\nonumber \\
& \quad 267331948586233552153804077120]^r \nonumber \\
& = [16463257830823767141436614720, 17081047424881048464188292960,\nonumber \\
& \quad 19907698298974838990176899360, 22160411527543044454854811656,\nonumber \\
& \quad 31598335386365316671087536320, 33887148412339464357524460480,\nonumber \\
& \quad 62377454670121162169220951328, 91532134570286487965704656840,\nonumber \\
& \quad 382770744566652586038401292240, 2405987537276101969384236694080]^r 
\end{align}
\begin{mirrortype}
\end{mirrortype}
$\bullet$ \ref{r123456789}\quad $(r=1,2,3,4,5,6,7,8,9)$
\\
\subsubsection{\;\;(r$\;=\;$-11, -10, -9, -8, -7, -6, -5, -4, -3, -2, -1)}
\label{rn1234567891011}
\begin{notation}
\end{notation}
$\bullet$ Type: $(r=-11,-10,-9,-8,-7,-6,-5,-4,-3,-2,-1)$ \vspace{1ex}\\
\indent
$\bullet$ Abbreviation: $[$rn1234567891011$]$
\begin{idealsolu}
\end{idealsolu}
$\bullet$ Smallest known solution, based on \eqref{r1234567891011s302}:
\begin{align}
& [21225111344567893589266440, 22027435141097951422537680,\nonumber\\
& \qquad\quad 23057495057767999510641960, 27046344413753180860584240,\nonumber\\
& \qquad\quad 30235771821035395584709740, 37051928474332392277216560,\nonumber\\
& \qquad\quad 49689795550848867162468720, 71222040289550042932871832,\nonumber\\
& \qquad\quad 98615132708607751753207152, 267082651085812660998269370,\nonumber\\
& \qquad\quad 582725784187227623996224080]^r \nonumber\\
& \quad= [21438072327958206902871120, 21582436451378800888749040,\nonumber\\
& \qquad\quad 23566116272277587735141415, 26163198473712260669218224,\nonumber\\
& \qquad\quad 32373654677068201333123560, 34462277559459698193325080,\nonumber\\
& \qquad\quad 55258479534995722965159180, 61634457942879844845754470,\nonumber\\
& \qquad\quad 112455853088763225683481840, 213666120868650128798615496,\nonumber\\
& \qquad\quad 1281996725211900772791692976, 2136661208686501287986154960]^r 
\end{align}
\begin{mirrortype}
\end{mirrortype}
$\bullet$ \ref{r1234567891011}\quad $(r=1,2,3,4,5,6,7,8,9,10,11)$
\\
\subsection{FPTE with r$_{1}$=0 and all others r>0}
Ideal positive integer solutions of FPTE have been identified for 5 types with \( r_1 = 0 \) and all others \( r > 0 \).
\subsubsection{\;\;(r$\;=\;$0)}
\label{r0}
\begin{notation}
\end{notation}
$\bullet$ Type: $(r=0)$ \qquad $\bullet$ Abbreviation: $[$r0$]$
\begin{idealsolu}
\end{idealsolu}
$\bullet$ Ideal solutions:
\begin{align}
\label{r0s6}
& [6]^r = [2,3]^r\\
& [8]^r = [2,4]^r\\
& [12]^r = [2,6]^r=[3,4]^r
\end{align}
\begin{mirrortype}
\end{mirrortype}
$\bullet$ None.
\\
\subsubsection{\;\;(r$\;=\;$0, 1)}
\label{r01}
\begin{notation}
\end{notation}
$\bullet$ Type: $(r=0,1)$ \qquad $\bullet$ Abbreviation: $[$r01$]$
\begin{idealsolu}
\end{idealsolu}
$\bullet$ First known solutions, by Chen Shuwen in 2023:
\begin{align}
\label{r01s6}
& [2,6]^r = [1,3,4]^r\\
& [3,8]^r = [1,4,6]^r\\
& [4,9]^r = [1,6,6]^r\\
& [4,10]^r = [1,5,8]^r\\
& [6,12]^r = [1,8,9]^r\\
& [5,12]^r = [1,6,10]^r\\
& [8,9]^r = [2,3,12]^r\\
& [16,27]^r = [1,18,24]^r= [3,4,36]^r
\end{align}
\begin{mirrortype}
\end{mirrortype}
$\bullet$ None.
\\
\subsubsection{\;\;(r$\;=\;$0, 2)}
\label{r02}
\begin{notation}
\end{notation}
$\bullet$ Type: $(r=0,2)$ \qquad $\bullet$ Abbreviation: $[$r02$]$
\begin{idealsolu}
\end{idealsolu}
$\bullet$ First known solutions, by Chen Shuwen in 2023:
\begin{align}
\label{r02s780}
& [432, 650 ]^r = [ 18, 20, 780]^r\\
& [702, 1328 ]^r = [ 4, 156, 1494]^r\\
& [407, 1692 ]^r = [ 6, 66, 1739 ]^r\\
& [572, 1827 ]^r = [ 14, 39, 1914]^r\\
& [1380,2835 ]^r = [9,138,3150]^r\\
& [2385,3380 ]^r = [13,150,4134]^r\\
& [2800,5778]^r = [28,90,6420]^r
\end{align}
\begin{mirrortype}
\end{mirrortype}
$\bullet$ None.
\\
\subsubsection{\;\;(r$\;=\;$0, 1, 2)}
\label{r012}
\begin{notation}
\end{notation}
$\bullet$ Type: $(r=0,1,2)$ \qquad $\bullet$ Abbreviation: $[$r012$]$
\begin{idealsolu}
\end{idealsolu}
$\bullet$ First known solutions, by Chen Shuwen in 2023:
\begin{align}
& [4, 5, 9 ]^r = [ 2, 3, 3, 10]^r\\
& [9,9,20 ]^r = [2,3,15,18]^r\\
\label{r012s30}
& [7,20,27 ]^r = [ 1,9,14,30]^r\\
& [5, 18, 28 ]^r = [ 1, 6, 14 , 30]^r\\
& [7,27,32 ]^r = [ 1, 8, 21 , 36]^r\\
& [17,28,63]^r = [2,6,49,51]^r
\end{align}
\begin{mirrortype}
\end{mirrortype}
$\bullet$ None.
\\
\subsubsection{\;\;(r$\;=\;$0, 1, 2, 3)}
\label{r0123}
\begin{notation}
\end{notation}
$\bullet$ Type: $(r=0,1,2,3)$ \qquad $\bullet$ Abbreviation: $[$r0123$]$
\begin{idealsolu}
\end{idealsolu}
$\bullet$ First known solution, by Chen Shuwen in 2023:
\begin{align}
\label{r0123s4300}
& [2401, 3216, 3690, 4300]^r = [1, 2460, 3010, 4020, 4116]^r
\end{align}
\begin{mirrortype}
\end{mirrortype}
$\bullet$ None.
\\
\subsection{FPTE with r$_{n}$=0 and all others r<0}
Ideal positive integer solutions of FPTE have been identified for 3 types with \( r_n = 0 \) and all others \( r < 0 \). 
\subsubsection{\;\;(r$\;=\;$-1, 0)}
\label{r0n1}
\begin{notation}
\end{notation}
$\bullet$ Type: $(r=-1,0)$ \qquad $\bullet$ Abbreviation: $[$r0n1$]$
\begin{idealsolu}
\end{idealsolu}
$\bullet$ First known solutions, by Chen Shuwen in 2023:
\begin{align}
\label{r0n1s30}
& [1,30]^r = [2,3,5]^r\\
& [2,105]^r = [5,6,7]^r\\
& [2,132]^r = [4,6,11]^r\\
& [2,180]^r = [4,5,18]^r\\
& [2,210]^r = [3,10,14]^r\\
& [3,252]^r = [7,9,12]^r
\end{align}
\begin{mirrortype}
\end{mirrortype}
$\bullet$ None.
\\
\subsubsection{\;\;(r$\;=\;$-2, 0)}
\label{r0n2}
\begin{notation}
\end{notation}
$\bullet$ Type: $(r=-2,0)$ \qquad $\bullet$ Abbreviation: $[$r0n2$]$
\begin{idealsolu}
\end{idealsolu}
$\bullet$ First known solutions, by Chen Shuwen in 2023:
\begin{align}
\label{r0n2s120}
& [4,120]^r = [5,8,12]^r\\
& [12,1008]^r = [16,21,36]^r\\
& [10,1320]^r = [11,30,40]^r\\
& [24,3960]^r = [33,40,72]^r\\
& [24,5280]^r = [30,44,96]^r\\
& [18,6840]^r = [19,72,90]^r
\end{align}
\begin{mirrortype}
\end{mirrortype}
$\bullet$ None.
\begin{relatedtype}
\end{relatedtype}
$\bullet$ \ref{s0n12} \quad $(s=-2,-1,0)$
\\
\subsubsection{\;\;(r$\;=\;$-2, -1, 0)}
\label{r0n12}
\begin{notation}
\end{notation}
$\bullet$ Type: $(r=-2,-1,0)$ \qquad $\bullet$ Abbreviation: $[$r0n12$]$
\begin{idealsolu}
\end{idealsolu}
$\bullet$ First known solutions, by Chen Shuwen in 2023:
\begin{align}
\label{r0n12s83538}
& [20,27,83538]^r = [17,52,189,270]^r\\
& [30,45,180180]^r = [26,105,180,495]^r\\
& [120,330,7969500]^r=[115,750,924,3960]^r\\
& [287,312,9133488]^r=[221,819,1968,2296]^r\\
& [150,150,18560568]^r=[114,312,775,15150]^r\\
& [420,448,26694720]^r=[322,1170,1984,6720]^r
\end{align}
\begin{mirrortype}
\end{mirrortype}
$\bullet$ None.
\\
\subsection{FPTE with r$_{1}$<0 and r$_{n}$>0}
Ideal positive integer solutions of FPTE have been identified for 2 types with \( r_1 < 0 \) and \( r_n > 0 \).
\subsubsection{\;\;(r$\;=\;$-1, 1)}
\label{r1n1}
\begin{notation}
\end{notation}
$\bullet$ Type: $(r=-1,1)$ \qquad $\bullet$ Abbreviation: $[$r1n1$]$
\begin{idealsolu}
\end{idealsolu}
$\bullet$ First known solutions, by Chen Shuwen in 2023:
\begin{align}
\label{r1n1s45}
& [5,45]^r = [8,18,24]^r\\
& [ 5, 56 ]^r = [ 7, 24, 30 ]^r\\
& [ 8, 72 ]^r = [ 15, 20, 45 ]^r\\
& [ 3, 40 ]^r = [ 4, 15, 24 ]^r = [ 5, 8, 30 ]^r \\
& [ 15, 168 ]^r = [ 21, 72, 90 ]^r = [ 28, 35, 120 ]^r
\end{align}
\begin{mirrortype}
\end{mirrortype}
$\bullet$ None.
\begin{relatedtype}
\end{relatedtype}
$\bullet$ \ref{s01n1} \quad $(s=-1,0,1)$
\\
\subsubsection{\;\;(r$\;=\;$-1, 0, 1)}
\label{r01n1}
\begin{notation}
\end{notation}
$\bullet$ Type: $(r=-1,0,1)$ \qquad $\bullet$ Abbreviation: $[$r01n1$]$
\begin{idealsolu}
\end{idealsolu}
$\bullet$ First known solution, based on computer search, by Chen Shuwen in 2023:
\begin{align}
\label{r01n1s227700}
& [33, 70200, 157872]^r =[ 54, 143, 208, 227700]^r
\end{align}
\begin{mirrortype}
\end{mirrortype}
$\bullet$ None.
\clearpage
\section{Ideal non-zero integer solution of FPTE}
So far, ideal non-zero integer solutions of FPTE have been discovered for 13 distinct types of \( (s = s_1, s_2, \dots, s_n) \). These 13 types can be categorized into five groups, containing 4, 4, 1, 2, and 2 types, respectively.
\subsection{FPTE with all s>0}
Ideal non-zero integer solutions of FPTE have been found for 4 types with all \( s > 0 \).
\subsubsection{\;\;(s$\;=\;$1, 2, 4)}
\label{s124}
\begin{notation}
\end{notation}
$\bullet$ Type: $(s=1,2,4)$ \qquad $\bullet$ Abbreviation: $[$s124$]$
\begin{idealsolu}
\end{idealsolu}
$\bullet$ First known solutions, based on \eqref{r24s8} and \eqref{r24s15}:
\begin{align}
\label{s124s7}
& [ -7, 7 ]^s = [ -8,3,5]^s\\
\label{s124s13}
& [ -13, 13 ]^s = [ -15,7, 8 ]^s
\end{align}
\begin{mirrortype}
\end{mirrortype}
$\bullet$ \ref{sn124}\quad $(s=-4,-2,-1)$
\begin{relatedtype}
\end{relatedtype}
$\bullet$ \ref{h124}\quad $(h=1,2,4)$ \vspace{1ex}\\
\indent
$\bullet$ \ref{r24}\quad $(r=2,4)$ 
\\
\subsubsection{\;\;(s$\;=\;$1, 2, 3, 5)}
\label{s1235}
\begin{notation}
\end{notation}
$\bullet$ Type: $(s=1,2,3,5)$ \qquad $\bullet$ Abbreviation: $[$s1235$]$
\begin{idealsolu}
\end{idealsolu}
$\bullet$ Smallest solution, based on \eqref{r135s51}:
\begin{align}
\label{s1235s51}
& [-51, 13, 38]^s = [-50, -7, 24, 33] ^s
\end{align}
\indent
$\bullet$ George Xeroudakes and Alfred Moessner obtained a two-parameter solution in 1958  \cite{Xeroudakes58}. Numerical examples:
\begin{align}
& [-134, 47, 87]^s = [-133, -8, 66, 75] ^s\\
& [-339, 23, 316]^s = [-300, -152, 221, 231] ^s\\
& [-545, 123, 422]^s = [-528, -105, 310, 323] ^s\\
& [-699, 191, 508]^s = [-689, -75, 308, 456] ^s\\
& [-1383, 538, 845]^s = [-1387, 42, 465, 880] ^s\\
& [-2931, 68, 2863]^s = [-2627, -1281, 1400, 2508] ^s\\
& [-6652, 2613, 4039]^s = [-6697, 600, 1708, 4389] ^s
\end{align}
\begin{mirrortype}
\end{mirrortype}
$\bullet$ \ref{sn1235}\quad\: $(s=-5,-3,-2,-1)$
\begin{relatedtype}
\end{relatedtype}
$\bullet$ \ref{h1235}\quad\: $(h=1,2,3,5)$ \vspace{1ex}\\
\indent
$\bullet$ \ref{r135}\quad $(r=1,3,5)$ 
\\
\subsubsection{\;\;(s$\;=\;$1, 2, 3, 4, 6)}
\label{s12346}
\begin{notation}
\end{notation}
$\bullet$ Type: $(s=1,2,3,4,6)$ \qquad $\bullet$ Abbreviation: $[$s12346$]$
\begin{idealsolu}
\end{idealsolu}
$\bullet$ First known solutions, by Chen Shuwen in 2023:
\begin{align}
\label{s12346s56}
& [ -47, -46, 37, 56]^s =[-54, -35, -7, 44, 52]^s\\
& [-179,-86,128,137]^s =[-171,-110,52,68,161]^s\\
& [-178,-116,101,193]^s =[-171,-130,17,88,196]^s\\
& [-203,-107,151,159]^s =[-192,-135,49,91,187]^s
\end{align}
\begin{mirrortype}
\end{mirrortype}
$\bullet$ \ref{sn12346}\quad\: $(s=-6,-4,-3,-2,-1)$
\begin{relatedtype}
\end{relatedtype}
$\bullet$ \ref{k1234}\quad $(k=1,2,3,4)$ \vspace{1ex}\\
\indent
$\bullet$ \ref{h12346}\quad\: $(h=1,2,3,4,6)$ 
\\
\subsubsection{\;\;(s$\;=\;$1, 2, 3, 5, 7)}
\label{s12357}
\begin{notation}
\end{notation}
$\bullet$ Type: $(s=1,2,3,5,7)$ \qquad $\bullet$ Abbreviation: $[$s12357$]$
\begin{idealsolu}
\end{idealsolu}
$\bullet$ First known solution, based on \eqref{r1357s99}:
\begin{align}
\label{s12357s98}
& [ -99, -13, 34, 98 ]^s = [ -82, -58, 16, 69, 75 ]^s
\end{align}
\begin{mirrortype}
\end{mirrortype}
$\bullet$ \ref{sn12357}\quad\: $(s=-7,-5,-3,-2,-1)$
\begin{relatedtype}
\end{relatedtype}
$\bullet$ \ref{h12357}\quad\: $(h=1,2,3,5,7)$ \vspace{1ex}\\
\indent
$\bullet$ \ref{r1357}\quad $(r=1,3,5,7)$ 
\\
\subsection{FPTE with all s<0}
So far, ideal non-zero integer solutions of FPTE have been found for 4 types with all \( s < 0 \).
\subsubsection{\;\;(s$=\;$-4, -2, -1)}
\label{sn124}
\begin{notation}
\end{notation}
$\bullet$ Type: $(s=-4,-2,-1)$ \qquad $\bullet$ Abbreviation: $[$sn124$]$
\begin{idealsolu}
\end{idealsolu}
$\bullet$ First known solutions, based on \eqref{s124s7} and \eqref{s124s13}:
\begin{align}
\label{sn124s280}
& [-120, 120 ]^s = [-105,168,280]^s\\
& [-840, 840 ]^s = [-728,1365,1560]^s
\end{align}
\begin{mirrortype}
\end{mirrortype}
$\bullet$ \ref{s124}\quad $(s=1,2,4)$
\\
\subsubsection{\;\;(s$=\;$-5, -3, -2, -1)}
\label{sn1235}
\begin{notation}
\end{notation}
$\bullet$ Type: $(s=-5,-3,-2,-1)$ \qquad $\bullet$ Abbreviation: $[$sn1235$]$
\begin{idealsolu}
\end{idealsolu}
$\bullet$ First known solution, based on \eqref{s1235s51}:
\begin{align}
\label{sn1235s14922600}
& [-3803800, 5105100, 14922600 ]^s \nonumber \\
& \quad = [-27713400, -3879876, 5878600, 8083075]^s
\end{align}
\begin{mirrortype}
\end{mirrortype}
$\bullet$ \ref{s1235}\quad $(s=1,2,3,5)$
\\
\subsubsection{\;\;(s$=\;$-6, -4, -3, -2, -1)}
\label{sn12346}
\begin{notation}
\end{notation}
$\bullet$ Type: $(s=-6,-4,-3,-2,-1)$ \qquad $\bullet$ Abbreviation: $[$sn12346$]$
\begin{idealsolu}
\end{idealsolu}
$\bullet$ First known solution, based on \eqref{s12346s56}:
\begin{align}
& [-1168647480, -772142085, 919999080, 939999060]^s \nonumber \\
& \quad = [-982726290, -831537630, 800739940, 1235427336, 6177136680]^s
\end{align}
\begin{mirrortype}
\end{mirrortype}
$\bullet$ \ref{s12346}\quad $(s=1,2,3,4,6)$
\\
\subsubsection{\;\;(s$=\;$-7, -5, -3, -2, -1)}
\label{sn12357}
\begin{notation}
\end{notation}
$\bullet$ Type: $(s=-7,-5,-3,-2,-1)$ \qquad $\bullet$ Abbreviation: $[$sn12357$]$
\begin{idealsolu}
\end{idealsolu}
$\bullet$ First known solution, based on \eqref{s12357s98}:
\begin{align}
& [-902090019600, -118456265200, 119665002600, 344916772200]^s \nonumber \\
& \quad = [-202192590600, -143014271400, 156362270064, \nonumber \\
& \qquad\qquad   169958989200, 732948140925]^s
\end{align}
\begin{mirrortype}
\end{mirrortype}
$\bullet$ \ref{s12357}\quad $(s=1,2,3,5,7)$
\\
\subsection{FPTE with s$_1$=0 and all other s>0}
So far, for the types where \( s_1 = 0 \) and all other \( s > 0 \), only one ideal non-zero integer solution has been found, and it corresponds to one particular type.
\subsubsection{\;\;(s$\;=\;$0, 1)}
\label{s01}
\begin{notation}
\end{notation}
$\bullet$ Type: $(s=0,1)$ \qquad $\bullet$ Abbreviation: $[$s01$]$
\begin{idealsolu}
\end{idealsolu}
$\bullet$ Only one ideal solution:
\begin{align}
\label{s01s4}
& [4]^s = [2,2]^s
\end{align}
\begin{mirrortype}
\end{mirrortype}
$\bullet$ None.
\begin{relatedtype}
\end{relatedtype}
$\bullet$ \ref{r1}\quad $(r=1)$ \vspace{1ex}\\
\indent
$\bullet$ \ref{r0}\quad $(r=0)$ 
\\
\subsection{FPTE with s$_n$=0 and all others s<0}
Ideal non-zero integer solutions of FPTE have been found for 2 types where \( s_n = 0 \) and all other \( s < 0 \).
\subsubsection{\;\;(s$=\;$-1, 0)}
\label{s0n1}
\begin{notation}
\end{notation}
$\bullet$ Type: $(s=-1,0)$ \qquad $\bullet$ Abbreviation: $[$s0n1$]$
\begin{idealsolu}
\end{idealsolu}
$\bullet$ There are infinite many solutions of this type, as $[p-p^2 ]^s = [ 1-p,p]^s$. Numerical example:
\begin{align}
\label{s0n1s3}
& [-6]^s = [-2,3]^s\\
& [-20]^s = [-4,5]^s
\end{align}
\begin{mirrortype}
\end{mirrortype}
$\bullet$ None.
\begin{relatedtype}
\end{relatedtype}
$\bullet$ \ref{rn1}\quad $(r=-1)$ \vspace{1ex}\\
\indent
$\bullet$ \ref{r0}\quad $(r=0)$ 
\\
\subsubsection{\;\;(s$=\;$-2, -1, 0)}
\label{s0n12}
\begin{notation}
\end{notation}
$\bullet$ Type: $(s=-2,-1,0)$ \qquad $\bullet$ Abbreviation: $[$s0n12$]$
\begin{idealsolu}
\end{idealsolu}
$\bullet$ First known solutions, by Chen Shuwen in 2023:
\begin{align}
\label{s0n12s8}
& [-120,4]^s = [-12,5,8]^s\\
& [-1008,12]^s = [-36,16,21]^s\\
& [-1320,10]^s = [-40,11,30]^s\\
& [-3960,24]^s = [-72,33,40]^s\\
& [-6840,18]^s = [-90,19,72]^s\\
& [-10920	,40	]^s = [-120,	56,	65]^s\\
& [-24360	,28	]^s = [-168,	29,	140]^s\\
& [-24480	,60	]^s = [-180,	85,	96]^s\\
& [-47880,84]^s = [-252,	120,133]^s
\end{align}
\begin{mirrortype}
\end{mirrortype}
$\bullet$ None.
\begin{relatedtype}
\end{relatedtype}
$\bullet$ \ref{r0n2}\quad $(r=-2,0)$
\\
\subsection{FPTE with s$_1$<0 and s$_n$>0}
Ideal non-zero integer solutions of FPTE have been found for 2 types with \( s_1 < 0 \) and \( s_n > 0 \). 
\subsubsection{\;\;(s$=\;$-1, 0, 1)}
\label{s01n1}
\begin{notation}
\end{notation}
$\bullet$ Type: $(s=-1,0,1)$ \qquad $\bullet$ Abbreviation: $[$s01n1$]$
\begin{idealsolu}
\end{idealsolu}
$\bullet$ First known solutions, by Chen Shuwen in 2023:
\begin{align}
\label{s01n1s112}
&[-60, 112]^s = [ -10, 14, 48]^s \\
&[-300, 400 ]^s = [-30, 50, 80]^s
\end{align}
\begin{mirrortype}
\end{mirrortype}
$\bullet$ None.
\begin{relatedtype}
\end{relatedtype}
$\bullet$ \ref{r1n1} \quad $(r=-1,1)$
\\
\subsubsection{\;\;(s$=\;$-2, -1, 0, 1)}
\label{s01n12}
\begin{notation}
\end{notation}
$\bullet$ Type: $(s=-2,-1,0,1)$ \qquad $\bullet$ Abbreviation: $[$s01n12$]$
\begin{idealsolu}
\end{idealsolu}
$\bullet$ First known solution, by Chen Shuwen in 2023:
\begin{align}
\label{s01n12s112}
&[-84,8,112]^s=[-16,14,14,24]^s 
\end{align}
\begin{mirrortype}
\end{mirrortype}
$\bullet$ None.
\\

\end{appendices}

\clearpage

\bibliographystyle{plain}
\providecommand{\noopsort}[1]{}

\section*{Author Information}
\addcontentsline{toc}{section}{Author Information}

\textbf{Chen Shuwen}\footnote{Name in Chinese convention (surname first).} \\
\textit{Seekway Innovations Technology Co., Ltd.} \\
\textit{Room 1001, Tower 3, Qunhua Hi-Tech Park, No.15 Qunhua Road} \\
\textit{Jiangmen City, Guangdong Province, P.R. China} \\
Email: \texttt{Chen.Shuwen@seekway.com}\\

\textbf{Education:}
\begin{itemize}
    \setlength{\itemsep}{-0.2em} 
    \vspace{-0.5\baselineskip}
    \item B.Sc. in Physics, Sun Yat-sen University, 1991
    \item M.Eng. in Communication and Information Systems,\\ Sun Yat-sen University, 2003
    \item EMBA, Peking University, 2019
\end{itemize}

\end{document}